\documentclass{amsart}%
\usepackage{amsfonts}
\usepackage{graphicx}
\usepackage{hyperref}
\usepackage{amsmath}
\usepackage{amssymb}%
\providecommand{\U}[1]{\protect\rule{.1in}{.1in}}
\newtheorem{theorem}{Theorem}[section]
\theoremstyle{plain}

\newtheorem{claim}[theorem]{Claim}
\newtheorem{conclusion}[theorem]{Conclusion}
\newtheorem{condition}[theorem]{Condition}

\newtheorem{corollary}[theorem]{Corollary}

\newtheorem{definition}[theorem]{Definition}
\newtheorem{example}[theorem]{Example}

\newtheorem{lemma}[theorem]{Lemma}
\newtheorem{notation}[theorem]{Notation}
\newtheorem{Observation}[theorem]{Observation}
\newtheorem{Assertion}[theorem]{Assertion}
\newtheorem{Deformation}[theorem]{Deformation}

\newtheorem{remark}[theorem]{Remark}
\newtheorem{solution}[theorem]{Solution}

\numberwithin{equation}{section}
\numberwithin{figure}{section}
\ifx\pdfoutput\relax\let\pdfoutput=\undefined\fi
\newcount\msipdfoutput
\ifx\pdfoutput\undefined\else
\ifcase\pdfoutput\else
\ifx\paperwidth\undefined\else
\ifdim\paperheight=0pt\relax\else\pdfpageheight\paperheight\fi
\ifdim\paperwidth=0pt\relax\else\pdfpagewidth\paperwidth\fi
\fi\fi\fi
\begin{document}
\title[\textbf{The precise form of Ahlfors' second fundamental theorem}]{\textbf{The precise form of Ahlfors' second fundamental theorem}}
\author{Guang Yuan\medskip\ Zhang }
\address{Department of Mathematical Sciences, Tsinghua University, Beijing 100084,
China. \textit{Email:} \textit{gyzhang@mail.tsinghua.edu.cn}}
\thanks{{Project 12171264 \& 10971112 supported by NSFC}}

\begin{abstract}
Let $S=\overline{\mathbb{C}}$ be the unit Riemann sphere. A simply connected
covering surface over $S$ is a pair $\Sigma=\left(  f,\overline{U}\right)  ,$
where $U$ is a Jordan domain in $\mathbb{C}$ and $f:\overline{U}\rightarrow S$
is an orientation-preserving, continuous, open and finite-to-one mapping
(OPCOFOM). Let $\mathbf{F}$ be the space of all simply connected covering
surfaces over $S$, for each $\Sigma=\left(  f,U\right)  \in\mathbf{F}$ let
$A(\Sigma)$ and $L(\partial\Sigma)$ be the area and boundary length of
$\Sigma,$ weighted according to multiplicity, and for each $\mathfrak{a}\in S$
let $\overline{n}\left(  \Sigma,\mathfrak{a}\right)  =\#f^{-1}\left(
\mathfrak{a}\right)  \cap U$ be the cardinality of the set $f^{-1}\left(
\mathfrak{a}\right)  \cap U$. The Second Fundamental Theorem (SFT) of Ahlfors'
covering surface theory is that, for any set $E_{q}=\left\{  \mathfrak{a}%
_{1},\mathfrak{a}_{2},\dots,\mathfrak{a}_{q}\right\}  $ of distinct $q(\geq3)$
points on $S,$ there exists a positive constant $h,$ which depends only on
$E_{q},$ such that for any $\Sigma\in\mathbf{F}$,
\[
(q-2)A(\Sigma)\leq4\pi\overline{n}(\Sigma,E_{q})+hL(\partial\Sigma),
\]
where $\overline{n}(\Sigma,E_{q})=\sum_{v=1}^{q}\overline{n}(\Sigma
,\mathfrak{a}_{v}).$

The goal of this paper is to develop a new method to give the precise bound
$H_{0}$ of $h.$ We write $R(\Sigma)=R(\Sigma,E_{q})$ the error term of
Ahlfors' SFT, say
\[
R(\Sigma)=\left(  q-2\right)  A(\Sigma)-4\pi\overline{n}\left(  \Sigma
,E_{q}\right)  .
\]

Our first main result is that for a.e. $L\in(0,+\infty),$ there exists an
\emph{extremal surface} $\Sigma_{0}$ in $\mathbf{F}\left(  L\right)
=\{\Sigma\in\mathbf{F}:L(\partial\Sigma)\leq L\},$ say%
\[
H_{L}=\sup\left\{  \frac{R(\Sigma)}{L(\partial\Sigma)}:\Sigma\in
\mathbf{F}\left(  L\right)  \right\}  =\frac{R(\Sigma_{0})}{L(\partial
\Sigma_{0})}.
\]

Our second main result is that there exists a subspace $\mathcal{S}_{0}$ of
$\mathbf{F}$, consisted of very simple surfaces, and there exists a $4\pi
$\emph{-extremal} surface $\Sigma_{1}\in\mathcal{S}_{0}$, say,
\[
H_{0}=\lim_{L\rightarrow\infty}H_{L}=\sup\left\{  \frac{R(\Sigma)}%
{L(\partial\Sigma)}:\Sigma\in\mathbf{F}\right\}  =\frac{R(\Sigma_{1})+4\pi
}{L(\partial\Sigma_{1})}.
\]

Our third main result is that among all $4\pi$\emph{-extremal} surfaces, there
exists one which is the simplest (such surfaces may not be unique). Simplest
$4\pi$\emph{-extremal} surfaces can be used to give the precise bound $H_{0}$
as simple as possible.

\end{abstract}
\maketitle
\tableofcontents

\section{Introduction\label{Sect1}}

\label{2023-03-28}We first recall some notations used in \cite{Zh1}. The
Riemann sphere\label{RS} $S$\ is the unit sphere in $\mathbb{R}^{3}$ centered
at the origin which is identified with the extended plane $\overline
{\mathbb{C}}$ via stereographic projection\label{SP} $P:S\rightarrow
\overline{\mathbb{C}}$ as in \cite{Ah0}. Length and area on $S$ have natural
interpretations using the spherical (chordal) metric on $\mathbb{C}$:
\[
ds=\rho(z)|dz|=\frac{2}{1+|z|^{2}}|dz|,z\in\mathbb{C}.
\]
For a set $V$ on $S,$ $\partial V\label{boundary}$ denotes its boundary and
$\overline{V}\label{cl}$ its closure. We write $\Delta=\{z\in\mathbb{C}%
:|z|<1\}\label{Delta}.$ Then for the upper and lower open hemispheres $S^{+}$
and $S^{-}$ on $S,$ we have $P\left(  S^{+}\right)  =\left(  \mathbb{C}%
\backslash\overline{\Delta}\right)  \cup\left\{  \infty\right\}  $ and
$P(S^{-})=\Delta.$

\begin{definition}
\label{triangle}\label{used-in-intro}%
\label{230331-9:50::2023-0403-three hours}A \emph{topological} \emph{triangle
}$\gamma$\emph{\ }on $S$ is a Jordan curve equipped with three vertices lying
on $\gamma.$ The three vertices define three (compact) edges of $\gamma$, the
two components of $S\backslash\gamma$ are called \emph{topological triangular
domains}, and the vertices and edges of $\gamma$ are also called vertices and
edges of the \emph{topological} \emph{triangular} \emph{domains. }If each of
the three edges of $\gamma$ is on a great circle on $S,$ then $\gamma$ is
called a triangle and the two components of $S\backslash\gamma$ are called
triangular domains.
\end{definition}

\begin{definition}
\label{AhlforsS}\label{2023-03-28 copy(1)}A (covering) surface $\Sigma$ (over
the sphere $S$) is defined as in Ahlfors' paper \cite{Ah}: $\Sigma$ is sewn
from a finite number of compact topological triangular domains on $S$ (see
Remark \ref{tri} A(3) for details). Equivalently speaking, $\Sigma$ is defined
to be a pair $\left(  f,\overline{U}\right)  ,$ where $U$ is a
domain\footnote{A domain means a connected open subset of $\overline
{\mathbb{C}}.$} in $\overline{\mathbb{C}}$ so that $\overline{U}$ has a finite
topological triangulation $\cup_{j=1}^{m_{0}}\overline{U_{j}}$ and for each
$j=1,\dots,m_{0},f|_{\overline{U_{j}}}:\overline{U_{j}}\rightarrow
f(\overline{U_{j}})\subset S$ is a homeomorphism. Moreover, $f$ is locally
homeomorphic on $\overline{U}$, except at a finite number of points, which are
some vertices of the triangulation $\{\overline{U_{j}}\}.$
\end{definition}

\label{2023-03-28 copy(2)}A mapping from a compact subset $K$ of $\mathbb{C}$
into $S$ is called an \emph{orientation-preserving, continuous, open and
finite-to-one mapping (OPCOFOM)\label{OPCOFOM}} if it can be extended to be an
\emph{OPCOFOM} from a neighborhood of $K$ in $\mathbb{C}$ into $S.$%
\emph{\ }Then a surface $\Sigma=\left(  f,\overline{U}\right)  $ in the
definition can be regarded as an \emph{OPCOFOM }from\emph{\ }$\overline{U}$
into $S,$ and vice versa. The term \emph{orientation-preserving }means that
$P^{-1}\circ f$ is orientation-preserving. In this setting, we call the pair
$\partial\Sigma=(f,\partial U)$ the boundary of the surface $\Sigma,$ and
define $A(\Sigma)=A(f,U)$ and $L(\partial\Sigma)=L(f,\partial U),$
respectively, area and length, weighted according to multiplicity. For
instance, $A(g,\Delta)=L(g,\partial\Delta)=6\pi$ when $g(z)=z^{3}%
,z\in\overline{\Delta}.$

\label{2023-03-28 copy(3)}Here we list some elementary results on the sphere
$S$ as in \cite{Zh1}: any great circle on $S$ has length $2\pi,$
$A(S)=2A(S^{+})=2A(S^{-})=4\pi,$ $L([0,+\infty])=2L([0,1])=2L([1,+\infty
])=\pi,$ the circle on $S$ with spherical diameter $[0,1]$ or $[1,+\infty],$
has length $\sqrt{2}\pi,$ and a disk in a hemisphere on $S$ with perimeter $L$
has area $2\pi-\sqrt{4\pi^{2}-L^{2}}$.

All surfaces in this paper are covering surfaces defined above.

\begin{definition}
\label{family F,D}\label{2023-03-28 copy(4)} Let $\Sigma=(f,\overline{U})$ be
a surface over $S$.

(1) $\Sigma$ is called a \emph{closed surface,} if $U=\overline{\mathbb{C}}=S
$. For a closed surface $\Sigma,$ we have $\partial\Sigma=\emptyset,$ and then
$L(\partial\Sigma)=0.$

(2) $\Sigma$ is called a \emph{simply-connected} \emph{surface,} if $U$ is a
simply connected domain.

(3) $\mathbf{F}$ denotes all surfaces such that for each $\Sigma=\left(
f,\overline{U}\right)  \in\mathbf{F},$ $U$ is a Jordan domain.
\end{definition}

\begin{remark}
\label{2023-03-28 copy(5)}In \cite{Zh1}, it is assumed that $U$ is a Jordan
domain, when $\left(  f,\overline{U}\right)  $ is a (covering) surface. But in
this paper, there is no such restriction. It is permitted that even for a
simply connected surface $\left(  f,\overline{U}\right)  $, $U$ is not a
Jordan domain, it may be the domain between two tangent circles, for example.
\end{remark}

\begin{remark}
\label{2023-03-28 copy(6)}(A) Let $K_{1}$ and $K_{2}$ be two domains or two
closed domains on $S,$ such that $\partial K_{1}$ and $\partial K_{2}$ are
both consisted of a finite number of disjoint Jordan curves. A mapping
$f:K_{1}\rightarrow K_{2}$ is called a \label{CCM} \emph{complete covering
mapping} (CCM), if (a) for each $p\in K_{2}$ there exists a neighborhood $V$
of $p$ in $K_{2}$ such that $f^{-1}(V)\ $can be expressed as a union
$\cup_{j\in\mathcal{A}}U_{j}$ of disjoint (relative) open sets of $K_{1}$, and
(b) $f|_{U_{j}}:U_{j}\rightarrow V$ is a homeomorphism for each $j\in
\mathcal{A}$.

(B) \label{2023-03-28 copy(7)}We call $f$ a \label{BCCM} \emph{branched
complete covering mapping} (BCCM), if all conditions of (A) hold, except that
(b) is replaced with (b1) or (b2): (b1) If both $K_{1}$ and $K_{2}$ are
domains, then for each $j\in\mathcal{A},$ $U_{j}\cap f^{-1}(p)$ contains only
one point $a_{j}$ of $f^{-1}(p),$ and there exist two homeomorphisms
$\varphi_{j}:U_{j}\rightarrow\Delta,\psi_{j}:V\rightarrow\Delta$ with
$\varphi_{j}\left(  a_{j}\right)  =\psi_{j}\left(  p\right)  =0,$ such that
$\psi_{j}\circ f|_{U_{j}}\circ\varphi_{j}^{-1}(\zeta)=\zeta^{k_{j}},\zeta
\in\Delta,\ $where $k_{j}$ is a positive integer; or (b2) if both $K_{1}$ and
$K_{2}$ are closed domains, then $f|_{K_{1}^{\circ}}:K_{1}^{\circ}\rightarrow
K_{2}^{\circ}$ satisfies (b1) and moreover, $f$ restricted to a neighborhood
of $\partial K_{1}$ in $K_{1}$ is a CCM onto a neighborhood of $\partial
K_{2}$ in $K_{2}.$

(C) \label{2023-03-28 copy(8)}For a surface $\Sigma=\left(  f,\overline
{U}\right)  \ $over $S,$ $f$ is in general not a CCM or BCCM. When $f\left(
z\right)  =z^{2},$ both $f:\overline{\Delta}\rightarrow\overline{\Delta}\ $and
$f:\Delta\rightarrow\Delta$ are BCCMs, but when $f\left(  z\right)
=z^{-1}\left(  \frac{z-1/2}{1-z/2}\right)  ^{2},$ $f:\Delta\rightarrow
f(\Delta)$ is neither a CCM nor a BCCM.
\end{remark}

\label{2023-03-28 copy(9)}Nevanlinna theory of value distribution of
meromorphic functions (\cite{Ha}, \cite{N}, \cite{Ne}, \cite{Y}) and Ahlfors
theory of covering surfaces (\cite{Ah}, \cite{Dr}, \cite{Ere}, \cite{Ha}) are
two major events in the history of the development of function theory. The
most striking result of Nevanlinna theory is the Second Fundamental Theorem.
Ahlfors' theory is a geometric version of Nevanlinna's, in which Nevanlinna's
Second Fundamental Theorem is reinterpreted as follows.

\noindent\textbf{Theorem A. }\label{2023-03-28 copy(10)}\emph{(Ahlfors' Second
Fundamental Theorem (SFT) \cite{Ah}). For an arbitrarily given set}
$E_{q}=\left\{  \mathfrak{a}_{1},\mathfrak{a}_{2},\dots,\mathfrak{a}%
_{q}\right\}  \label{Eq}$ \emph{of} \emph{distinct }$q\left(  \geq3\right)
$\emph{\ points on }$S$\emph{, there exists a positive constant }%
$h$\emph{\ such that for any surface }$\Sigma=(f,\overline{U})\in\mathbf{F},$%
\begin{equation}
(q-2)A(\Sigma)\leq4\pi\overline{n}(\Sigma)+hL(\partial\Sigma), \label{a6}%
\end{equation}
\emph{where }%
\begin{equation}
\overline{n}(\Sigma)=\overline{n}(\Sigma,E_{q})=\sum_{v=1}^{q}\overline
{n}(\Sigma,\mathfrak{a}_{v}),\mathrm{\ }\overline{n}(\Sigma,\mathfrak{a}%
_{v})=\#f^{-1}(\mathfrak{a}_{v})\cap U\emph{,} \label{a70}%
\end{equation}
\emph{\ and }$\#$\emph{\ is the cardinality\label{card}.}\bigskip

\label{2023-03-28 copy(11)}The goal of this paper is to present a method to
identify the precise bound for that $h$ in Ahlfors' SFT for $\mathbf{F}$. This
problem can be traced back to the early 1940s, when J. Dufresnoy first gave a
numerical estimate of $h$ in \cite{Du} as follows.\bigskip

\noindent\textbf{Theorem B. }\label{2023-03-28 copy(12)}(Dufresnoy
\cite{Du})\emph{\ For any surface }$\Sigma=(f,\overline{U})\in\mathbf{F}%
,$\emph{\ }%
\begin{equation}
(q-2)A(\Sigma)\leq4\pi\overline{n}(\Sigma,E_{q})+\left(  q-2\right)
\frac{6\pi}{\delta_{E_{q}}}L(\partial\Sigma), \label{a6-1}%
\end{equation}
\emph{where }$\delta_{E_{q}}=\min_{1\leq i<j\leq q}d\left(  \mathfrak{a}%
_{i},\mathfrak{a}_{j}\right)  .$ \bigskip

Here $d\left(  \mathfrak{a}_{i},\mathfrak{a}_{j}\right)  $ is the spherical
distance on $S,$ which is the minimum of length of all paths on $S$ from
$\mathfrak{a}_{i}\ $to $\mathfrak{a}_{j}.$

In 2011, the author have identified the precise bound for $h$ in a special
case as follows.\bigskip

\noindent\textbf{Theorem C.\label{2023-03-28 copy(13)} }(Zhang \cite{Zh1})
\emph{For any surface }$\Sigma=(f,\overline{U})\in\mathbf{F}$ \emph{with }%
\begin{equation}
\overline{n}\left(  \Sigma,\{0,1,\infty\}\right)  =\emptyset, \label{no3}%
\end{equation}
\emph{\ we have}%
\begin{equation}
A(\Sigma)\leq h_{0}L(\partial\Sigma), \label{a6-2}%
\end{equation}
\emph{where }%
\begin{equation}
h_{0}=\max_{\theta\in\left[  0,\frac{\pi}{2}\right]  }h(\theta), \label{=0.}%
\end{equation}%
\[
h(\theta)=\frac{A\left(  \mathfrak{D}\left(  \overline{0,1},\theta
,\theta\right)  \right)  +4\pi}{L(\partial\mathfrak{D}\left(  \overline
{0,1},\theta,\theta\right)  )}=\frac{\left(  \pi+\theta\right)  \sqrt
{1+\sin^{2}\theta}}{\arctan\dfrac{\sqrt{1+\sin^{2}\theta}}{\cos\theta}}%
-\sin\theta,
\]
\emph{and }$\mathfrak{D}\left(  \overline{0,1},\theta,\theta\right)  $
\emph{is the lens} \emph{on the sphere} $S$ \emph{enclosed by two symmetric
circular arcs with endpoints }$0$\emph{\ and }$1$\emph{\ and with interior
angle }$2\theta$\emph{\ at the cusps (see Definition \ref{lune-lens}).}
\emph{Moreover, the bound }$h_{0}$ \emph{is precise: there exists a sequence
}$\Sigma_{n}\in\mathbf{F}$\emph{\ with }$\overline{n}(\Sigma_{n}%
,\{0,1,\infty\})=\emptyset$\emph{\ such that }$A(\Sigma_{n})/L(\partial
\Sigma_{n})\rightarrow h_{0}$\emph{\ as }$n\rightarrow\infty.$\bigskip

\textbf{\label{2023-03-28 copy(14)}}The proof of (\ref{a6-2}) occupied almost
all space of the long paper \cite{Zh1}, in which it is pointed out that
$h_{0}$ can be easily found and (\ref{a6-2}) can be easily proved, provided
that for every given $L>0,$ the extremal surface $\Sigma_{L}=\left(
f_{L},\overline{\Delta}\right)  $ with $L(\partial\Sigma_{L})\leq L$ so that
$f_{L}(z)\neq0,1,\infty$ and that $A(\Sigma_{L})$ assumes the maximal value
exists. In fact, such extremal surfaces have a lot of good properties, for
example, the boundary $\partial\Sigma_{L}$ is consisted of circulars arcs with
the same curvature and each of these arcs contains two endpoints $\{0,1\}$ or
$\{1,\infty\}.$ From this property, after a little argument as in the
introduction section of \cite{Zh1} one can compute the precise bound. However,
at that time we couldn't prove the existence of the extremal surface
$\Sigma_{L}$ and left it as a conjecture. We were lucky to find a substitute
for the extremal surface, and finally we have successfully proved the
optimality of $h_{0}$ in \cite{Zh1}.

\textbf{\label{2023-03-28 copy(15)}}But the method in \cite{Zh1} is difficult
to be applied to the general case, that is, in (\ref{a6}), $q\geq3, $ the
position of $\mathfrak{a}_{v}$s in $E_{q}=\{\mathfrak{a}_{1},\mathfrak{a}%
_{2},\dots,\mathfrak{a}_{q}\}$ are arbitrarily given and $\overline{n}%
(\Sigma,E_{q})$ is not assumed equal to zero. So, to identify the precise
bound of $h$ in (\ref{a6}) in general case, it seems the simplest way is to
solve the existence of extremal surfaces.

\textbf{\label{2023-03-28 copy(16)}}Consider the general case that $q\geq3$,
$E_{q}=\{\mathfrak{a}_{1},\mathfrak{a}_{2},\dots,\mathfrak{a}_{q}\}$ and $L>0
$ are given arbitrarily. We introduce the following notations%
\begin{equation}
R(\Sigma)=R(\Sigma,E_{q})=(q-2)A(\Sigma)-4\pi\overline{n}(\Sigma,E_{q}),
\label{Ahero}%
\end{equation}
which is the error term of Ahlfors' SFT,%
\[
\mathbf{F}(L)=\{\Sigma\in\mathbf{F}:L(\partial\Sigma)\leq L\},
\]%
\begin{equation}
H(\Sigma)=H(\Sigma,E_{q})=\frac{R(\Sigma,E_{q})}{L(\partial\Sigma)},
\label{DH}%
\end{equation}%
\begin{equation}
H_{0}=H_{0}(E_{q})=\sup_{\Sigma\in\mathbf{F}}H(\Sigma), \label{H_0}%
\end{equation}
and%
\[
H_{L}=H_{L}(E_{q})=\sup_{\Sigma\in\mathbf{F(L)}}H(\Sigma).\label{HL(E)}%
\]
Then it is clear that $H_{L}$ increase with respect to $L$,
\[
H_{0}=\lim_{L\rightarrow\infty}H_{L},
\]
and Ahlfors' SFT can be restated as
\[
H_{0}=H_{0}(E_{q})<+\infty.
\]

Then our goal becomes to prove the existence of extremal surfaces in
$\mathbf{F}\left(  L\right)  $ and present an achievable method that gives the
precise value of $H_{0}.$ For this purpose, we introduce some more terminology
and notations.

\begin{definition}
\label{L}\textbf{\label{2023-03-28 copy(17)}}Let $\mathcal{L}$ be the set of
continuous points of $H_{L}=H_{L}(E_{q}),$ with respect to $L.$
\end{definition}

\begin{remark}
\label{RL}Since $H_{L}$ increase with respect to $L,$ it is clear that
$\left(  0,+\infty\right)  \backslash\mathcal{L}$ is just a countable set.
\end{remark}

\begin{definition}
\label{geodesic}\textbf{\label{2023-03-28 copy(18)}}For any two non-antipodal
points $p$ and $q$ on $S,$ $\overline{pq}$ is the geodesic on $S$ from $p$ to
$q:$ the shorter of the two arcs with endpoints $p$ and $q$ of the great
circle on $S$ passing through $p$ and $q.$ Thus $d(p,q)<\pi$ and
$\overline{pq}$ is uniquely determined by $p$ and $q$. An arc of a great
circle on $S$ is called a \emph{line segment} on $S,$ and to emphasize this,
we also refer to it as a \emph{straight line segment. }For the notation
$\overline{pq},$ when $p$ and $q$ are explicit complex numbers we write
$\overline{p,q},$ to avoid ambiguity such as $\overline{123}=\overline{12,3}$
or $\overline{1,23}.$ When $p$ and $q$ are two antipodal points on $S,$
$\overline{pq}$ is not unique and $d\left(  p,q\right)  =\pi.$ To avoid
confusions, when we write $\overline{pq},$ or say $\overline{pq}$ is well
defined, we always assume $d\left(  p,q\right)  <\pi.$
\end{definition}

All paths and curves considered in this paper are oriented and any subarc of a
path or closed curve inherits this orientation. Sometimes paths and curves
will be regarded as sets, but only when we use specific set operations and set
relations. For an oriented circular arc $c,$ the circle $C$ containing $c $
and oriented by $c$ is called the circle determined\label{determine} by $c. $

\begin{definition}
\label{in}\textbf{\label{2023-03-28 copy(19)}}(1) For a Jordan domain $D$ in
$\overline{\mathbb{C}},$ let $h$ be a M\"{o}bius transformation with
$h(D)\subset\Delta.$ Then $\partial D$ is oriented by $h$ and the
anticlockwise orientation of $\partial h(D).$ The boundary of every Jordan
domain on $S$ is oriented in the same way, via stereographic projection.

(2) For a Jordan curve $C$ on $\overline{\mathbb{C}}$ or $S,$ the domain
$T_{C}$ bounded by $C$ is called \emph{enclosed} by $C$ if the boundary
orientation of $T_{C}$ agrees with the orientation of $C.$

(3) A domain $D$ on $S$ is called \emph{convex} if for any two points $q_{1}$
and $q_{2}$ in $D$ with $d(q_{1},q_{2})<\pi$, $\overline{q_{1}q_{2}}\subset D
$; a Jordan curve on $S$ is called \emph{convex} if it encloses a \emph{convex
domain }on $S$; a path on $S$ is called \emph{convex} if it is locally an arc
of a \emph{convex} Jordan curve.

(4) Let $\gamma:[a,b]\rightarrow S$ be a path on $S$ and $p_{0}\in(a,b).$
$\gamma$ is called \emph{convex at} $p_{0},$ if $\gamma$ restricted to a
neighborhood $I_{\delta}=(p_{0}-\delta,p_{0}+\delta)$\label{6868-5} of $p_{0}
$ in $(a,b)$ is a \emph{convex} Jordan path, with respect to the orientation
of $\gamma$ when $t\in I_{\delta}$ increases; and $\gamma$ is called
\emph{strictly convex} at $p_{0}$ if for some $\delta>0$ the restriction
$\gamma|_{I_{\delta}}$ is convex and $\gamma|_{I_{\delta}}\cap S_{1}%
=\gamma|_{I_{\delta}}\backslash\{\gamma(p_{0})\}$ for some open hemisphere
$S_{1}$ on $S.$
\end{definition}

\textbf{\label{2023-03-28 copy(20)}}By definition, the disk $\{z\in
\overline{\mathbb{C}}:|z|>2\}$ is viewed as a convex domain on $S$ and its
boundary orientation is clockwise, and thus the circle $|z|=2$ oriented
clockwise is also convex on $S.$ Also by definition, the disk $\{z\in
\mathbb{C}:|z|<2\}$ is not convex on $S$ and its boundary orientation is
anticlockwise, and thus the circle $|z|=2$ oriented anticlockwise is not
convex on $S$. By convention, an arc of a curve inherits the orientation of
this curve. A Jordan curve on $S$ is convex and lies in a hemisphere on $S$ if
it is locally convex (see Lemma 4.1 in \cite{Zh1} for polygonal Jordan curves
on $S$). A locally convex curve on $S$ is locally simple and always goes
straight or turns left when viewing $S$ from its interior, such as with planar
convex polygons oriented anticlockwise.

\begin{definition}
\label{interior}\textbf{\label{2023-03-28 copy(21)}}For a curve $\beta,$ in
$\mathbb{C}$ or $S,$ given by $z=z(t),t\in\lbrack a,b],$ $\beta^{\circ}$
denotes the interior of $\beta,$ which is the restriction of $\beta$ to the
open interval $(a,b),$ and $\partial\beta$ denotes the set of the endpoints of
$\beta,$ which is either the singleton $\{z(a)\}$ if $\beta$ is closed, or the
two points set $\{z(a),z(b)\}\label{boundaryarc}.$ For a surface
$\Sigma=\left(  f,\overline{U}\right)  ,$ the interior of the surface $\Sigma$
is the restricted open surface $\Sigma^{\circ}=\left(  f,U\right)  .$ The
notation $q\in\Sigma^{\circ}$ means a pair $\left(  f,p\right)  $ with
$q=f(p)$ for some $p\in U.$
\end{definition}

When $\beta$ is defined by $z=e^{it},t\in\lbrack0,4\pi],$ for example,
$\beta^{\circ}$ is still well defined as $z=e^{it},t\in(0,4\pi),$ and thus as
sets, $\beta$ and $\beta^{\circ}$ may coincide.

\begin{definition}
\label{partition}\textbf{\label{2023-03-28 copy(22)}}For a Jordan curve
$\alpha$ in $\mathbb{C}$, its partition is a collection $\{\alpha_{j}%
\}_{j=1}^{n}$ of its subarcs such that $\alpha=\cup_{j=1}^{n}\alpha_{j}$ and
$\alpha_{j}^{\circ}$ are disjoint and arranged anticlockwise. In this setting
we write $\alpha=\alpha_{1}+\alpha_{2}+\cdots+\alpha_{n}.$ Here $\alpha
_{j}^{\circ}$ is the interior of $\alpha_{j}$, which is now $\alpha_{j}$
without endpoints (since $\alpha$ is simple). A partition
\begin{equation}
\partial\Sigma=\gamma_{1}+\gamma_{2}+\cdots+\gamma_{n} \label{6662}%
\end{equation}
of $\partial\Sigma$ for a surface $\Sigma=(f,\overline{U})\in\mathbf{F}$ is
equivalent to a partition
\[
\partial U=\alpha_{1}+\alpha_{2}+\cdots+\alpha_{n}%
\]
of $\partial U$ such that $\gamma_{j}=(f,\alpha_{j})$ for $j=1,\dots,n.$
\end{definition}

Now we can introduce the subspace $\mathcal{F}$ of $\mathbf{F:}$

\begin{definition}
\label{F}\label{F1}\textbf{\label{2023-03-28 copy(23)}}We denote by
$\mathcal{F}$ the subspace of $\mathbf{F}$ such that for each $\Sigma=\left(
f,\overline{U}\right)  \in\mathcal{F}$, $\partial\Sigma$ has a partition
\begin{equation}
\partial\Sigma=c_{1}+c_{2}+\cdots+c_{n} \label{2022-610-1}%
\end{equation}
of simple convex circular (SCC) arcs. This means that $\partial U$ has a
partition%
\begin{equation}
\partial U=\alpha_{1}+\alpha_{2}+\cdots+\alpha_{n} \label{2021-520-1}%
\end{equation}
such that $f$ restricted to each $\alpha_{j}$ is a homeomorphism onto the SCC
arc $c_{j}.$

$\mathcal{F}\left(  L\right)  $ is the subspace of $\mathcal{F}$ such that for
each $\Sigma\in\mathcal{F}\left(  L\right)  ,$ $L(\partial\Sigma)\leq L.$
\end{definition}

Let $T$ be a Jordan domain on $S$ which is a union of a finite number of disks
$D_{j},j=1,\dots,m.$ Then $\overline{T}$ can be viewed as a surface in
$\mathcal{F},$ if all disks $D_{j}$ are convex on $S,$ say, all disks are of
diameter $\leq\pi.$

\begin{remark}
\label{FF}\textbf{\label{2023-03-28 copy(24)}}For any surface $\Sigma$ in
$\mathbf{F}$ and any $\varepsilon>0\mathbf{,}$ to estimate $H(\Sigma)$ we may
assume $L(\partial\Sigma)<+\infty,$ for otherwise we have $H(\Sigma)=0.$ Then
for any $\varepsilon>0,$ there are standard ways to show that there exists
another surface $\Sigma_{\varepsilon}$ in $\mathcal{F}$ such that
$\partial\Sigma_{\varepsilon}$ is consisted of a finite number of line
segments, $L(\partial\Sigma_{\varepsilon})\leq L(\partial\Sigma)+\varepsilon,
$ $A(\Sigma_{\varepsilon})\geq A(\Sigma)-\varepsilon$ and $\overline{n}\left(
\Sigma_{\varepsilon}\right)  \leq\overline{n}\left(  \Sigma\right)  .$ Then we
have
\[
H(\Sigma_{\varepsilon})\geq\frac{\left(  q-2\right)  \left(  A(\Sigma
)-\varepsilon\right)  -4\pi\overline{n}\left(  \Sigma\right)  }{L(\partial
\Sigma)+\varepsilon}\rightarrow H(\Sigma)
\]
as $\varepsilon\rightarrow0.$ Thus, we have%
\begin{equation}
H_{0}=H_{0}(E_{q})=\sup_{\Sigma\in\mathbf{F}}H(\Sigma)=\sup_{\Sigma
\in\mathcal{F}}H(\Sigma). \label{ag58}%
\end{equation}

\end{remark}

\begin{definition}
\label{HL}\textbf{\label{2023-03-28 copy(25)}}Define
\[
H_{L}=H_{L}(E_{q})=\sup_{\Sigma\in\mathcal{F}(L)}H(\Sigma).
\]
We call $\Sigma_{0}\in\mathcal{F}(L)$ an \emph{extremal} surface of
$\mathcal{F}(L),$ if
\[
H(\Sigma_{0})=H_{L}.
\]
If $\Sigma_{0}$ is an extremal surface of $\mathcal{F}(L)$ with minimal
perimeter among all the extremal surfaces of $\mathcal{F}(L)$, we call it a
\emph{precise extremal surface }of\emph{\ }$\mathcal{F}(L).$
\end{definition}

By Remark \ref{FF}, a \emph{precise extremal surface }of\emph{\ }%
$\mathcal{F}(L)$ is also a \emph{precise extremal surface }of\emph{\ }%
$\mathbf{F}(L),$ if $L\in\mathcal{L}$. Our first main result is the following.

\begin{theorem}
\label{main1}\textbf{\label{2023-03-28 copy(26)}}For each $L\in\mathcal{L}$
with $L>0,$ there exists a precise extremal surface of $\mathbf{F}(L),$ and
for every precise extremal surface $\Sigma$ of $\mathbf{F}\left(  L\right)  $
we have $\Sigma\in\mathcal{F}\left(  L\right)  $, and $\partial\Sigma$ has a
partition%
\[
\partial\Sigma=C_{1}+C_{2}+\cdots+C_{N},
\]
such that the following holds.

(i) All $C_{j},j=1,\dots,N,$ are SCC arcs and have the same curvature.

(ii) At most one of $C_{j},j=1,\dots,N,$ is a major circular arc.

(iii) Either $\partial\Sigma$ is a simple circle containing at most one point
of $E_{q},$ or $N>1,$ $\#\partial C_{j}=2$, $\partial C_{j}\subset E_{q}$ and
$C_{j}^{\circ}\cap E_{q}=\emptyset,$ for $j=1,\dots,N.$

(iv) For every $j=1,\dots,N,C_{j}$ is contained in an open hemisphere $S_{j}$
on $S.$
\end{theorem}

Recall that $\partial C_{j}$ denotes the endpoints of $C_{j}.$ A major
circular arc means more than half of the circle that contains it, and a closed
circular path is regarded major.

From the above theorem we will find the simplest models to compute
$H_{0}=H_{0}(E_{q}).$ But we need introduce some more notations and terminologies.

\begin{definition}
\label{S-surface}\label{230329}$\mathcal{S}_{0}=\mathcal{S}_{0}(E_{q})$ is the
space of all surfaces $\Sigma=\left(  f,\overline{\Delta}\right)  $ in
$\mathcal{F}$\label{Need-in-F_r?} such that the following hold.

(1) There exists a positive integer $q^{\prime}=Q(\Sigma)\in\{2,\dots,q\}$ and
$\partial\Sigma$ has a partition
\begin{equation}
\partial\Sigma=C_{1}\left(  p_{1},p_{2}\right)  +\cdots+C_{q^{\prime}}\left(
p_{q^{\prime}},p_{1}\right)  , \label{ko12}%
\end{equation}
such that for each $j=1,2,\dots,q^{\prime},$ $C_{j}=C_{j}\left(  p_{j}%
,p_{j+1}\right)  $ is an SCC arc on $S$ $(C_{q^{\prime}+1}=C_{1}).$

(2) For each $j=1,2,\dots,q^{\prime},$ the initial point $p_{j}$ and the
terminal point $p_{j+1}$ are two distinct points in $E_{q},$ and $d\left(
p_{j},p_{j+1}\right)  <\pi,$ say, $\overline{p_{j}p_{j+1}}$ is well defined.

(3) $p_{1},\dots,p_{q^{\prime}}$ are distinct each other.

(4)
\begin{equation}
\deg_{\max}\Sigma=\max\left\{  \#f^{-1}(w)\cap\Delta:w\in S\backslash
\partial\Sigma\right\}  \leq2q^{\prime}-2. \label{degm}%
\end{equation}

(5) For each $j=1,2,\dots,q^{\prime},$ $C_{j}^{\circ}\subset S\backslash
E_{q}.$

(6) For each $j=1,2,\dots,q^{\prime},$ $C_{j}$ is contained in an open
hemisphere $S_{j}$ on $S.$

(7) All $C_{j},j=1,2,\dots,q,$ have the same curvature.

(8) For each $j=1,2,\dots,q^{\prime}$, $C_{j}$ is either a minor arc or half
of a circle.$\label{degmax-only-defined-inner}$
\end{definition}

For each surface $\Sigma$ in $\mathcal{S}_{0}=\mathcal{S}_{0}\left(
E_{q}\right)  ,$ all the circular arcs of $\partial\Sigma$ have the same
curvature $k=k\left(  \Sigma,E_{q}\right)  .$

Now we can state our second main result:

\begin{theorem}
\label{main2}\label{230329-8:20}(i) There exists a surface $\Sigma_{0}$ in
$\mathcal{S}_{0}=\mathcal{S}_{0}(E_{q})$ such that%
\begin{equation}
H_{0}=\sup_{F\in\mathbf{F}}\frac{R(F)}{L(\partial F)}=H_{\mathcal{S}_{0}%
}^{4\pi}=\max_{F\in\mathcal{S}_{0}}\frac{R(F)+4\pi}{L(\partial F)}%
=\frac{R(\Sigma_{0})+4\pi}{L(\partial\Sigma_{0})}. \label{ag24}%
\end{equation}

(ii) For each surface of $\mathcal{S}_{0}$ satisfying (\ref{ag24}) the
boundary is consisted of strictly convex circular arcs.

(iii) $k=k\left(  \Sigma,E_{q}\right)  $ is the same for all surfaces $\Sigma$
satisfying (\ref{ag24}).
\end{theorem}

\begin{remark}
\label{230329-8:44}For each $\Sigma\in\mathcal{S}_{0},$ the partition
(\ref{ko12}), ignoring a permutation of subscripts like $i_{0}+1,i_{0}%
+2,\dots,i_{0}+q^{\prime}$ $\left(  \operatorname{mod}q^{\prime}\right)  ,$ is
unique and so is the number $q^{\prime}$ of terms. Then the boundary length of
all surfaces in $\mathcal{S}_{0}$ has a common upper bound, say,
$L(\partial\Sigma_{0})\leq q\pi$ for all $\Sigma_{0}\in\mathcal{S}_{0}.$
Though $\mathbf{F}\left(  L\right)  $ contains extremal surface when
$L\in\mathcal{L}$ by Theorem \ref{main1} and (\ref{ag58}), it is proved in
\cite{SZ} that $\mathbf{F}$ contains no extremal surface and, since
$H_{L}=\sup\left\{  H\left(  \Sigma\right)  \in\mathbf{F}:L(\partial
\Sigma)\leq L\right\}  $ increases as a function of $L\in\left(
0,\infty\right)  ,$ for any sequence $F_{n}\in\mathbf{F}$,$\ $with
\[
H(F_{n})=\frac{R(F_{n})}{L(\partial F_{n})}\rightarrow\sup_{F\in\mathbf{F}%
}\frac{R(F_{n})}{L(\partial F_{n})}=H_{0},
\]
we must have $L(\partial F_{n})\rightarrow\infty\left(  n\rightarrow
\infty\right)  .$ Thus (\ref{ag24}) plays the role converting infinite
problems into finite problems.
\end{remark}

To make (\ref{ag24}) easier to use, we will reduce the space $\mathcal{S}_{0}
$ as much as possible to make (\ref{ag24}) simpler:

\begin{definition}
\label{4pi-extr}\label{230329-9:54}A surface $\Sigma$ in $\mathcal{S}%
_{0}=\mathcal{S}_{0}(E_{q})$ satisfying (\ref{ag24}) will be called a $4\pi
$-\emph{extremal surface of }$\mathcal{S}_{0}$. A surface $\Sigma$ of
$\mathcal{S}_{0}$ is called the\emph{\ simplest }$4\pi$\emph{-extremal surface
of }$\mathcal{S}_{0}$ if the following hold.

(1) $Q(\Sigma)=Q(E_{q})=\min\{Q(\Sigma^{\prime})=\#\partial\Sigma^{\prime}\cap
E_{q}:\Sigma^{\prime}$ are $4\pi$-extremal surface of $\mathcal{S}_{0}\}.$

(2) $L(\partial\Sigma^{\prime})\geq L(\partial\Sigma)$ for any $4\pi$-extremal
surface $\Sigma^{\prime}$ of $\mathcal{S}_{0}$ with $Q\left(  \Sigma^{\prime
}\right)  =Q(\Sigma).$

(3) $\deg_{\max}(\Sigma)\leq\deg_{\max}(\Sigma^{\prime})$ for any $4\pi
$-extremal extremal surface $\Sigma^{\prime}$ of $\mathcal{S}_{0}$ with
$Q\left(  \Sigma^{\prime}\right)  =Q(\Sigma)\ $and $L(\partial\Sigma^{\prime
})=L(\partial\Sigma).$
\end{definition}

\begin{definition}
\label{4psim}$\mathcal{S}^{\ast}=\mathcal{S}^{\ast}(E_{q})$ is defined to be
the space of all the \emph{simplest} $4\pi$-\emph{extremal surfaces }of
$\mathcal{S}_{0}$.
\end{definition}

Our third main theorem is the following

\begin{theorem}
\label{4pi-sim}For any set $E_{q}$ of distinct $q$ points on $S,q\geq3,$
$\mathcal{S}^{\ast}=\mathcal{S}^{\ast}(E_{q})\neq\emptyset.$
\end{theorem}

To present some application of the last two main theorems, we introduce some
more notations. For a path $\Gamma$ on $S$ or $\mathbb{C}$ given by
$z=z(t),t\in\lbrack t_{1},t_{2}],$ $-\Gamma$ is the opposite path of $\Gamma$
given by $z=z(t_{2}-t+t_{1}),t\in\lbrack t_{1},t_{2}].$

\begin{definition}
\label{lune-lens}\label{230329-22:45}A convex domain enclosed by a convex
circular arc $c$ and its chord $I$ is called a \emph{lune} and is denoted by
$\mathfrak{D}^{\prime}\left(  I,c\right)  ,\mathfrak{D}^{\prime}\left(
I,\theta(c)\right)  ,$ $\mathfrak{D}^{\prime}\left(  I,L(c)\right)  ,$ or
$\mathfrak{D}^{\prime}\left(  I,k(c)\right)  ,$ where $\theta$ is the interior
angle at the two cusps, $k$ is the curvature of $c$ and $I$ is oriented such
that\footnote{The initial and terminal points of $I$ and $c$ are the same,
respectively, in the notation $\mathfrak{D}^{\prime}(I,\theta),$ in other
words, $\mathfrak{D}^{\prime}(I,\theta)$ is on the right hand side of $I.$}
$\partial\mathfrak{D}^{\prime}\left(  I,\theta\right)  =c-I.$

For two lunes $\mathfrak{D}^{\prime}\left(  I,\theta_{1}\right)  $ and
$\mathfrak{D}^{\prime}\left(  -I,\theta_{2}\right)  $ sharing the common chord
$I$ we write%
\[
\mathfrak{D}\left(  I,\theta_{1},\theta_{2}\right)  =\mathfrak{D}^{\prime
}\left(  I,\theta_{1}\right)  \cup I^{\circ}\cup\mathfrak{D}^{\prime}\left(
-I,\theta_{2}\right)
\]
and call the Jordan domain $\mathfrak{D}\left(  I,\theta_{1},\theta
_{2}\right)  $ a lens. Then the notations $\mathfrak{D}\left(  I,l_{1}%
,l_{2}\right)  $, $\mathfrak{D}\left(  I,c_{1},c_{2}\right)  \ $and
$\mathfrak{D}\left(  I,k_{1},k_{2}\right)  $ are in sense and denote the same
lens, when $l_{j}=L(c_{j})$ and $k_{j}=k\left(  c_{j}\right)  $ are the length
and curvature of $c_{j},$ $j=1,2$, say,%
\begin{align*}
\mathfrak{D}\left(  I,c_{1},c_{2}\right)   &  =\mathfrak{D}\left(
I,l_{1},l_{2}\right)  =\mathfrak{D}\left(  I,k_{1},k_{2}\right)
=\mathfrak{D}^{\prime}\left(  I,l_{1}\right)  \cup I^{\circ}\cup
\mathfrak{D}^{\prime}\left(  -I,l_{2}\right) \\
&  =\mathfrak{D}^{\prime}\left(  I,c_{1}\right)  \cup I^{\circ}\cup
\mathfrak{D}^{\prime}\left(  -I,c_{2}\right)  =\mathfrak{D}^{\prime}\left(
I,k_{1}\right)  \cup I^{\circ}\cup\mathfrak{D}^{\prime}\left(  -I,k_{2}%
\right)  .
\end{align*}
\label{230329-22:45 copy(1)}
\end{definition}

For a lune $\mathfrak{D}^{\prime}\left(  I,\tau\right)  ,$ whether $\tau$
denotes the length $l$, the angle $\theta,$ or the curvature $k$ is always
clear from the context, and so is for the lens $\mathfrak{D}\left(  I,\tau
_{1},\tau_{2}\right)  .$

By definition we have $0<\theta_{j}\leq\pi$ for $j=1,2,$ since $\mathfrak{D}%
^{\prime}\left(  I,\theta\right)  $ is convex, but for the domain
$\mathfrak{D}\left(  I,\theta_{1},\theta_{2}\right)  $ it is permitted that
$\theta_{1}$ or $\theta_{2}$ is zero, say $\mathfrak{D}\left(  I,\theta
_{1},\theta_{2}\right)  $ reduces to $\mathfrak{D}^{\prime}\left(
I,\theta_{1}\right)  $ or $\mathfrak{D}^{\prime}\left(  -I,\theta_{2}\right)
.$ By definition of $\mathfrak{D}(I,\theta,\theta)$ we have
\[
\mathfrak{D}(I,\theta,\theta)=\mathfrak{D}^{\prime}\left(  I,\theta\right)
\cup\mathfrak{D}^{\prime}\left(  -I,\theta\right)  \cup I^{\circ},
\]
and $\theta\in(0,\pi].$ If $I=\overline{1,0,-1}$ and $\theta=\pi/2,$ for
example, $\mathfrak{D}\left(  I,\theta,\theta\right)  =\Delta$ and
$\mathfrak{D}^{\prime}\left(  I,\pi/2\right)  =\Delta^{+}$ is the upper half
disk of $\Delta.$

As a consequence of the last two main results we have our forth main result:

\begin{theorem}
\label{2,3}\label{230329-23:00}Let $\Sigma=\left(  f,\overline{\Delta}\right)
\in\mathcal{S}^{\ast}(E_{q}).$ Then $\partial\Sigma$ has a partition
(\ref{ko12}) satisfying Definition \ref{S-surface} (1)--(8) and one of the
following holds.

(i) If $q^{\prime}=Q(\Sigma)=2,$ then $\Sigma$ is a closed lens $\overline
{\mathfrak{D}\left(  \overline{p_{1}p_{2}},\theta_{0},\theta_{0}\right)  }$
(see Definition \ref{lune-lens}) with $\{p_{1},p_{2}\}\subset E_{q},$
$\theta_{0}\in(0,\pi/2)$ and
\[
d\left(  p_{1},p_{2}\right)  =\delta_{E_{q}}=\min\left\{  d\left(  a,b\right)
:a\in E_{q},b\in E_{q},a\neq b\right\}  .
\]

(ii) If $q=3$ and $E_{q}$ is on a great circle on $S,$ then $q^{\prime
}=Q(\Sigma)=2.$

(iii) If $q=q^{\prime}=Q(\Sigma)=3,$ and the three point set $\left(
\partial\Sigma\right)  \cap E_{q}=\{p_{1},p_{2},p_{3}\}$ is not in a great
circle on $S$, then $\partial\Delta$ has a partition $\partial\Delta
=\alpha_{1}\left(  a_{1},a_{2}\right)  +\alpha_{2}\left(  a_{2},a_{3}\right)
+\alpha_{3}\left(  a_{3},a_{1}\right)  $ and there is a Jordan curve
$\alpha^{\prime}=\alpha_{1}^{\prime}\left(  a_{1},a_{2}\right)  +\alpha
_{2}^{\prime}\left(  a_{2},a_{3}\right)  +\alpha_{3}^{\prime}\left(
a_{3},a_{1}\right)  $ in $\overline{\Delta}$ such that $\alpha^{\prime}%
\cap\partial\Delta=\{a_{1},a_{2},a_{3}\},$ $f$ restricted to the domain
enclosed by $\alpha^{\prime}$ is a homeomorphism onto the triangular domain
(see Definition \ref{triangle}) $T$ enclosed by $\overline{p_{1}p_{2}%
p_{3}p_{1}}$, and $f$ restricted to the closed Jordan domain enclosed by
$\alpha_{j}-\alpha_{j}^{\prime}$ is a homeomorphism onto the closed lune
$\overline{\mathfrak{D}^{\prime}\left(  \overline{p_{j}p_{j+1}},\theta
_{j}\right)  }$ enclosed by $C_{j}\left(  p_{j},p_{j+1}\right)  +\overline
{p_{j+1}p_{j}},$ for $j=1,2,3$ (see Definition \ref{lune-lens}).

(iv) If $q^{\prime}=q=3,$ or if $q^{\prime}=2$ and $q\geq3,$ then
$\overline{n}\left(  \Sigma\right)  =\overline{n}\left(  \Sigma,E_{3}\right)
=\emptyset.$
\end{theorem}

\begin{remark}
\label{finite}\label{230329-23:00 copy(1)}(1) We call two surfaces $\Sigma
_{1}=\left(  f_{1},\overline{U_{1}}\right)  $ and $\Sigma_{2}=\left(
f_{2},\overline{U_{2}}\right)  $ \emph{equivalent} and write
\[
\left(  f_{1},\overline{U_{1}}\right)  \sim\left(  f_{2},\overline{U_{2}%
}\right)  ,
\]
if there exists an orientation preserving homeomorphism (OPH) $\varphi
:\overline{U_{1}}\rightarrow\overline{U_{2}}$ such that $f_{1}=f_{2}%
\circ\varphi.$ If $\Sigma_{1}$ and $\Sigma_{2}$ are equivalent, then it is
clear that $H(\Sigma_{1})=H(\Sigma_{2}),$ and thus, it is very useful to
regard $\Sigma_{1}$ and $\Sigma_{2}$ as the same surface in our paper. Then we
can show that $\mathcal{S}^{\ast}$ contains at most a finite number surfaces,
say, equivalent classes.

(2) \label{230329-23:00 copy(2)}\label{samecurve}Similar to (1), we call two
curves $(\alpha_{1},[a_{1},b_{1}])$ and $(\alpha_{2},[a_{2},b_{2}])$ on
$S\ $\emph{equivalent} and write
\[
(\alpha_{1},[a_{1},b_{1}])\sim(\alpha_{2},[a_{2},b_{2}])
\]
if there is an increasing homeomorphism $\tau:[a_{1},b_{1}]\rightarrow\lbrack
a_{2},b_{2}]$ such that $\alpha_{2}\circ\tau=\alpha_{1}.$

(3) \label{230329-23:00 copy(3)}$\mathcal{S}^{\ast}$ may contain more than one
equivalent classes. When $q=3,$ $E_{3}=\{0,1,\infty\},$ $\mathcal{S}^{\ast}$
is consisted of two equivalent classes, and when $E_{3}=\{0,r,\infty\},$
$\mathcal{S}^{\ast}$ contains only one equivalent class when $d\left(
0,r\right)  <\pi/2,\ $by Theorem \ref{2,3} and Definition of $\mathcal{S}%
^{\ast}.$
\end{remark}

\label{230329-23:00 copy(4)}Let $\mathbf{F}^{\ast}=\mathbf{F}^{\ast}\left(
E_{q}\right)  \label{F*}$ be the subspace of $\mathbf{F}$ such that for each
$\Sigma\in\mathbf{F}^{\ast}\left(  E_{q}\right)  ,$ $\overline{n}(\Sigma
,E_{q})=0.$ Define%
\[
h_{0}=h_{0}(E_{q}):=\sup_{\Sigma\in\mathbf{F}^{\ast}\left(  E_{q}\right)
}H(\Sigma).\label{h0}%
\]
then we have%
\[
h_{0}\leq H_{0}.
\]
As a consequence of Theorem \ref{main2}--\ref{2,3}, we can easily prove the
following at the end of this paper.

\begin{theorem}
\label{q3}\label{230329-23:00 copy(5)}If $q=3,$ then
\begin{equation}
H_{0}=H_{0}(E_{3})=h_{0}(E_{3})=h_{0}; \label{q=3}%
\end{equation}
and if $q>3,$ then there exists a set $E_{q}^{\prime}$ of $q$ distinct points
such that
\begin{equation}
H_{0}(E_{q}^{\prime})>h_{0}(E_{q}^{\prime}). \label{000}%
\end{equation}

\end{theorem}

\label{230329-23:00 copy(6)}If $E_{3}=\{0,1,\infty\}.$ Then by Theorem
\ref{2,3}, $\Sigma_{0}\in\mathcal{S}^{\ast}\left(  E_{3}\right)
=\mathcal{S}^{\ast}\left(  \left\{  0,1,\infty\right\}  \right)  $ is the
closed lens $\overline{\mathfrak{D}(\overline{0,1},\theta_{0},\theta_{0})}$ or
$\overline{\mathfrak{D}(\overline{1,\infty},\theta_{0},\theta_{0})}$ for some
$\theta_{0}\in\lbrack0,\pi/2].$ Then
\[
H_{0}=h_{0}=\frac{A\left(  \mathfrak{D}(\overline{0,1},\theta_{0},\theta
_{0})\right)  +4\pi}{L(\partial\mathfrak{D}(\overline{0,1},\theta_{0}%
,\theta_{0}))}=\max_{\theta\in\lbrack0,\pi/2]}\frac{A\left(  \mathfrak{D}%
(\overline{0,1},\theta,\theta)\right)  +4\pi}{L(\partial\mathfrak{D}%
(\overline{0,1},\theta,\theta))}%
\]
and Theorem C is recovered with more: The inequality (\ref{a6-2}) of
\cite{Zh1} still holds and is sharp without the assumption (\ref{no3}).

\begin{remark}
\label{tri}\label{conv-1}(A) We always assume the triangulation $\{\overline
{U}_{j}\}=\mathbf{\{}\overline{U_{i}}\}_{i=1}^{m_{0}}$ of $\overline{U}$ in
Definition \ref{AhlforsS} satisfy the rules (1)--(3) of triangulation:

\label{230329-23:07}(1) There is a closed triangular domain $K$ on
$\mathbb{C},$ such that for each $j$, there exists a homeomorphism
$\varphi_{j}\ $from $\overline{U_{j}}$ to $K,$ the inverse of vertices and
edges of $K$ are called vertices and edges of $\overline{U_{j}}.$ We will
write $\mathbf{U=\{}\overline{U_{i}},e_{i1},e_{i2},e_{i3}\}_{i=1}^{m_{0}}$ and
also call $\mathbf{U}$ a triangulation of $\overline{U}\mathbf{,}$ the
vertices and edges of $\overline{U_{j}}$ are also called vertices and edges of
$\mathbf{U},$ respectively, and each $\overline{U_{i}}$ is called a
\emph{closed topological triangular domain (CTTD), }or a \emph{face} of
$\mathbf{U}.$

(2) For each pair of edges $e_{i}$ and $e_{j}$ of $\mathbf{U}$ with $e_{i}\neq
e_{j}\mathbf{,}$ $e_{i}\cap e_{j}$ is empty, or a singleton which is a common
vertex of $e_{i}$ and $e_{j}.$

(3) For each pair of faces $\overline{U_{i}}$ and $\overline{U_{j}}$ of
$\mathbf{U\ }$with $\overline{U_{i}}\neq\overline{U_{j}}$ and $\overline
{U_{i}}\cap\overline{U_{j}}\neq\emptyset,$ $\overline{U_{i}}\cap
\overline{U_{j}}$ is either a common vertex, or a common edge, of
$\overline{U_{i}}$ and $\overline{U_{j}}.$

\label{230330}Then the surface $\Sigma=\left(  f,\overline{U}\right)  $ can be
identified as the disjoint collection
\[
\mathbf{T}=\{T_{i},l_{i1},l_{i2},l_{i3}\}_{i=1}^{m_{0}}=\{f(\overline{U_{i}%
}),f(e_{i1}),f(e_{i2}),f(e_{i3})\}_{i=1}^{m_{0}}%
\]
of CTTDs and their edges on $S$ with adjacency condition, the equivalent
relation $\mathbf{R}$ of the collection of edges
\[
\mathbf{E=}\{l_{ij},i=1,\dots,m,j=1,2,3\}
\]
of $\mathbf{T:}l_{ij}\sim l_{i_{1}j_{1}}$ iff the two CTTDs $\overline{U_{i}}$
and $\overline{U_{i_{1}}}$ share an common edge $e_{ij}=e_{i_{1}j_{1}}.$ Thus
$f\left(  e_{ij}\right)  =l_{ij}=l_{i_{1}j_{1}}=f\left(  e_{i_{1}j_{1}%
}\right)  \ $(as sets). Note that this relation depends on the order of
$U_{i}$ and $e_{ij}.$ When such identification is established, we can get rid
of $\overline{U_{j}}$ and $f.$ That is, we can understand the surface $\Sigma$
as the collection $\mathbf{T}=\{T_{j},l_{j1},l_{j2},l_{j3}\}_{j=1}^{m_{0}}$
with a relation $\mathbf{R}=\left\{  \left(  i,j,i_{1},j_{1}\right)  \right\}
\ $such that $l_{ij}\sim l_{i_{1}j_{1}}$ iff $\left(  i,j,i_{1},j_{1}\right)
\in\mathbf{R.}$ Then, for a pair $\left(  i,j\right)  \in\left\{
1,2,\dots,m_{0}\right\}  \times\{1,2,3\},$ $l_{ij}$ is on the boundary of
$\Sigma$ iff there is no pair $\left(  i_{1},j_{1}\right)  $ such that
$\left(  i,j,i_{1},j_{1}\right)  \in\mathbf{R},$ and we can understand
$\Sigma$ as $\Sigma=(\mathbf{T},\mathbf{R}),$ and regard $\left(
\mathbf{T},\mathbf{R}\right)  $ the Riemann surface of $\Sigma,$ with a finite
number of branch points. In this way we can understand the relation
$\Sigma_{1}\sim\Sigma_{2}$ as this: the Riemann surface of $\Sigma_{1}$ and
$\Sigma_{2}$ are the same (when we order $U_{j}$ and its edges properly).

(B) If $\Sigma=\left(  f,\overline{U}\right)  $ is a surface, where $U$ is a
Jordan domain, we should understand the whole boundary $\partial\Sigma\ $as a
simple curve on the surface. In fact, we can define the positive distance
$d_{f}\left(  \cdot,\cdot\right)  $ in Definition \ref{df}. But for
simplicity, sometimes we state that $\partial\Sigma$ contains a proper
\emph{closed} arc $\gamma\left(  p_{1},p_{1}\right)  $. This only means
$\partial U$ has a partition $\partial U=\alpha\left(  a_{1},a_{2}\right)
+\beta\left(  a_{2},a_{1}\right)  $ such that $a_{1}\neq a_{2}$ but $f\left(
a_{1})=f(a_{2}\right)  =p_{1}.$ That is to say when we project the curve
$\partial\Sigma$ to $S,$ the arc $\left(  f,\alpha_{1}\right)  $ is projected
onto the closed path $\gamma\left(  p_{1},p_{1}\right)  .$

(C) For convenience, we make the agreement: For two faces $T_{j}=\left(
f,\overline{U_{j}}\right)  $ of the partition $\mathbf{T}$ of $\Sigma,j=1,2,$
they can be regarded to be closed subdomains of both $\Sigma$ and $S.$
Regarded to be in $\Sigma,$ $T_{1}$ and $T_{2}$ can not intersect when
$\overline{U_{1}}\cap\overline{U_{2}}=\emptyset$. But when they are regarded
sets on $S,$ $T_{1}$ and $T_{2}$ intersect when $f\left(  \overline{U_{1}%
}\right)  \cap f(\overline{U_{2}})\neq\emptyset.$ For a set $K\subset S,$ when
we write $K\subset T_{1}\cap T_{2},\ $we only regard $T_{j}$ as the set
$f(\overline{U_{j}})$ on $S$ for $j=1,2$, say, $K\subset T_{1}\cap T_{2}$ iff
$K\subset f(\overline{U_{1}})\cap f(\overline{U_{2}}).$
\end{remark}

\section{Elementary properties of surfaces of $\mathcal{F}$\label{Sect2}}

\label{230330-11:14}By Stoilow's theorem, every surface $\left(
f,\overline{U}\right)  $ is equivalent to a surface whose defining function is
holomorphic on the domain of Definition.

\begin{theorem}
\label{st}\label{230330-11:14 copy(1)}(i). (Stoilow's Theorem \cite{S}
pp.120--121) Let $U$ be a domain on $\overline{\mathbb{C}}$ and let
$f:U\rightarrow S$ be an open, continuous and discrete mapping. Then there
exist a domain $V$ on $\overline{\mathbb{C}}$ and a homeomorphism
$h:V\rightarrow U,$ such that $f\circ h:V\rightarrow S$ is a holomorphic mapping.

\label{230330-11:14 copy(2)}(ii). Let $\Sigma=(f,\overline{U})$ be a surface
where $U$ is a domain on $\overline{\mathbb{C}}.$ Then there exists a domain
$V$ on $\overline{\mathbb{C}}$ and an OPH $h:\overline{V}\rightarrow
\overline{U}$ such that $f\circ h:\overline{V}\rightarrow S$ is a holomorphic mapping.

\label{230330-11:14 copy(3)}(iii) Let $\Sigma=(f,\overline{U})\in\mathbf{F}.$
Then there exists an OPH $\varphi:\overline{U}\rightarrow\overline{U}$ such
that $f\circ\varphi$ is holomorphic on $U$.
\end{theorem}

What $f$ is discrete means that $f^{-1}(w)\cap K$ is finite for any compact
subset $K$ of $U.$

\begin{proof}
\label{230330-11:14 copy(4)}\label{ST1}Let $\Sigma=(f,\overline{U})$ be a
surface where $U$ is a domain on $\overline{\mathbb{C}}.$ Then $f:\overline
{U}\rightarrow S$ is the restriction of an OPCOFOM $g$ defined in a
neighborhood $U_{1}$ of $\overline{U},$ and thus by Stoilow's theorem, there
exists a domain $V_{1}$ on $\overline{\mathbb{C}}$ and an OPH $h:V_{1}%
\rightarrow U_{1}$ such that $g\circ h$ is holomorphic on $V_{1}$ and then for
$\overline{V}=h^{-1}(\overline{U}),$ $f\circ h$ is holomorphic on
$\overline{V},$ and thus (ii) holds.

\label{230330-11:14 copy(5)}Continue the above discussion and assume $U$ is a
Jordan domain. Then $V$ is also a Jordan domain and by Riemann mapping theorem
there exists a conformal mapping $h_{1}$ from $U$ onto $V$ and by
Caratheodory's extension theorem $h_{1}$ can be extended to be homeomorphic
from $\overline{U}$ onto $\overline{V},$ and thus the extension of $h\circ
h_{1}$ is the desired mapping $\varphi$ in (iii).
\end{proof}

\begin{remark}
\label{holomo}\label{230330-11:14 copy(6)}Since equivalent surfaces have the
same area, boundary length and Ahlfors error term (\ref{Ahero}), when we study
a surface $\Sigma=(f,\overline{U})$ in $\mathbf{F}$ in which $U$ is not
specifically given, we can always assume that $f$ is holomorphic on
$\overline{U}.$
\end{remark}

We shall denote by $D(a,\delta)$ the disk on $S$ with center $a$ and spherical
radius $\delta.$ Then $\Delta\subset S$ is the disk $D\left(  0,\pi/2\right)
.$

\begin{definition}
\label{interiorangle}\label{230330-11:14 copy(7)}Let $\Sigma=\left(
f,\overline{U}\right)  \in\mathcal{F}$ and let $p\in\partial U.$ If $f$ is
injective near $p,$ then $f$ is homeomorphic in a closed Jordan neighborhood
$N_{p}$ of $p$ in $\overline{U}$, and then $f(\overline{N_{p}})$ is a closed
Jordan domain on $S$ whose boundary near $f(p)$ is an SCC arc, or two SCC arcs
joint at $f(p),$ and thus the interior angle of $f(\overline{N_{p}})$ at
$f(p)$ is well defined, called the interior angle of $\Sigma$ at $p$ and
denoted by $\angle\left(  \Sigma,p\right)  .$

In general, we can draw some paths $\{\beta_{j}\}_{j=1}^{k}$ in $\overline{U}$
with $\cup_{j=1}^{k}\beta_{j}\backslash\{p\}\subset U$ and $\beta_{j}\cap
\beta_{i}=\{p\}\ $if $i\neq j,$ such that each $\left(  f,\beta_{j}\right)  $
is a simple line segment on $S$, $\cup_{j=1}^{k}\beta_{j}$ divides a closed
Jordan neighborhood $N_{p}$ of $p$ in $\overline{U}$ into $k+1$ closed Jordan
domains $\overline{U_{j}}\ $with $p\in\overline{U_{j}},j=1,\dots,k+1,$ and
$U_{i}\cap U_{j}=\emptyset\ $if $i\neq j,$ and $f$ restricted to
$\overline{U_{j}}$ is a homeomorphism with $\left(  f,\overline{U_{j}}\right)
\in\mathcal{F}$ for each $j.$ Then the interior angle of $\Sigma$ at $p$ is
defined by
\[
\angle\left(  \Sigma,p\right)  =\sum_{j=1}^{k+1}\angle\left(  \left(
f,U_{j}\right)  ,p\right)  .
\]
The existences of $\left\{  \beta_{j}\right\}  _{j=1}^{k}$ will be given later
in Corollary \ref{cov-2} (v).
\end{definition}

This definition is independent of coordinate transform of $\overline{U},$ and
thus one can understood it with the assumption that $f$ is holomorphic on
$\overline{U}.$ The following result is a consequence of the previous theorem.

\begin{lemma}
\label{cov-1}\label{230330-19:26}Let $(f,\overline{U})\ $be a surface, $U$ be
a domain on $\mathbb{C}$ bounded by a finite number of Jordan curves and
$\left(  f,\partial U\right)  $ is consisted of a finite number of simple
circular arcs and let $q\in f(\overline{U}).$ Then, for sufficiently small
disk $D(q,\delta)$ on $S\ $with $\delta<\pi/2,$ $f^{-1}(\overline{D(q,\delta
)})\cap\overline{U}$ is a finite union of disjoint sets $\{\overline{U_{j}%
}\}_{1}^{n}$ in $\overline{U},$ where each $U_{j}$ is a Jordan domain in $U,$
such that for each $j,$ $\overline{U_{j}}\cap f^{-1}(q)$ contains exactly one
point $x_{j}$ and (A) or (B) holds:

(A) $x_{j}\in U_{j}\subset\overline{U_{j}}\subset U$ and $f:\overline{U_{j}%
}\rightarrow\overline{D(q,\delta)}$ is a BCCM such that $x_{j}$ is the only
possible branch point.

(B) $x_{j}\in\partial U,$ $f$ is locally homeomorphic on $\overline{U_{j}%
}\backslash\{x_{j}\},$ and when $\left(  f,\overline{U}\right)  \in
\mathcal{F}, $ the following conclusions (B1)--(B3) hold:

(B1) The Jordan curve $\partial U_{j}$ has a partition $\alpha_{1}\left(
p_{1},x_{j}\right)  +\alpha_{2}\left(  x_{j},p_{2}\right)  +\alpha_{3}\left(
p_{2},p_{1}\right)  $ such that $\alpha_{1}+\alpha_{2}=\left(  \partial
U\right)  \cap\partial U_{j}$ is an arc of $\partial U,$ $\alpha_{3}^{\circ
}\subset U,$ $c_{j}=\left(  f,\alpha_{j}\right)  $ is an SCC arc for $j=1,2,$
and $c_{3}=\left(  f,\alpha_{3}\right)  $ is a locally SCC\footnote{The
condition $\delta<\pi/2$ makes $\partial D\left(  q,\delta\right)  $ strictly
convex, and it is possible that $\left(  f,\alpha_{3}^{\circ}\right)  $ may
describes $\partial D\left(  q,\delta\right)  $ more than one round, and in
this case $\left(  f,\alpha_{3}^{\circ}\right)  $ is just locally SCC.} arc in
$\partial D\left(  q,\delta\right)  $ from $q_{2}=f\left(  p_{2}\right)  $ to
$q_{1}=f\left(  p_{1}\right)  $. Moreover, $f$ is homeomorphic in a
neighborhood of $\alpha_{j}\backslash\{x_{j}\},$ for $j=1,2,$ in $\overline
{U}$ and
\[
\partial\left(  f,\overline{U_{j}}\right)  =\left(  f,\partial U_{j}\right)
=c_{1}+c_{2}+c_{3}.
\]

(B2) The interior angle of $\left(  f,\overline{U_{j}}\right)  $ at $p_{1}$
and $p_{2}$ are both contained in $[\frac{7\pi}{16},\frac{9\pi}{16}].$

(B3) There exists a rotation $\psi$ of $S$ with $\psi(q)=0$ such that the
following conclusion (B3.1) or (B3.2) holds:

(B3.1) $q_{1}=q_{2},\left(  f,\alpha_{1}\right)  =\overline{q_{1}q}%
=\overline{q_{2}q}=-\left(  f,\alpha_{2}\right)  ,$ say, $\left(  f,\alpha
_{1}+\alpha_{2}\right)  =\overline{q_{1}q}+\overline{qq_{1}},$ and $\left(
\psi\circ f,\overline{U_{j}}\right)  $ is equivalent to the
surface\footnote{Here $\delta z^{\omega_{j}}$ is regarded as the mapping
$z\mapsto\delta z^{\omega_{j}}\in S,z\in\overline{\Delta^{+}},$ via the
stereographic projection $P.$} $\left(  \delta z^{\omega_{j}}:\overline
{\Delta^{+}}\right)  $ on $S$ so that
\[
\left(  \delta z^{\omega_{j}},[-1,1]\right)  =\overline{a_{\delta}%
,0}+\overline{0,a_{\delta}},
\]
where $\omega_{j}$ is an even positive integer and $a_{\delta}\in\left(
0,1\right)  $ with $d\left(  0,a_{\delta}\right)  =\delta.$

(B3.2) $q_{1}\neq q_{2},$ as sets $c_{1}\cap c_{2}=\{q\},$ and $\left(
\psi\circ f,\overline{U_{j}}\right)  $ is equivalent to the the surface
$\left(  F,\overline{\Delta^{+}}\cup\overline{\mathfrak{D}_{1}^{\prime}}%
\cup\overline{\mathfrak{D}_{2}^{\prime}}\right)  $ so that the following holds.

(B3.2.1) $\mathfrak{D}_{1}^{\prime}=\mathfrak{D}^{\prime}\left(
\overline{-1,0},\theta_{1}\right)  \ $and $\mathfrak{D}_{2}^{\prime
}=\mathfrak{D}^{\prime}\left(  \overline{0,1},\theta_{2}\right)  $, such that
for each $j=1,2,\theta_{j}\in\lbrack0,\frac{\pi}{4}].$ Moreover $\theta_{1}=0$
(or $\theta_{2}=0$) when $c_{1}=\overline{q_{1}q}$ (or $c_{2}=\overline
{qq_{2}}$), and in this case $\mathfrak{D}_{1}^{\prime}=\emptyset$ (or
$\mathfrak{D}_{2}^{\prime}=\emptyset$). See Definition \ref{lune-lens} for the
notation $\mathfrak{D}^{\prime}\left(  \cdot,\cdot\right)  .$

(B3.2.2) $\left(  F,\overline{\Delta^{+}}\right)  $ is the surface $T=\left(
\delta z^{\omega_{j}},\overline{\Delta^{+}}\right)  $, where $\omega_{j}\ $is
a positive number which is not an even number and even may not be an integer,
$\left(  F,\overline{\mathfrak{D}_{1}^{\prime}}\right)  $ is the lune
$\psi\left(  \overline{\mathfrak{D}^{\prime}\left(  \overline{q_{1}q}%
,c_{1}\right)  }\right)  $ and $\left(  F,\overline{\mathfrak{D}_{2}^{\prime}%
}\right)  $ is the lune $\psi\left(  \overline{\mathfrak{D}^{\prime}\left(
\overline{qq_{2}},c_{2}\right)  }\right)  .$ That is to say, $\left(
f,\overline{U_{j}}\right)  $ is obtained by sewing\label{sew40} the sector
$\psi^{-1}\left(  T\right)  $ with center angle\footnote{This angle maybe
larger than $2\pi$ as the sector $\left(  z^{3},\overline{\Delta^{+}}\right)
.$} $\omega_{j}\pi,$ and the closed lunes $\overline{\mathfrak{D}^{\prime
}\left(  \overline{q_{1}q},c_{1}\right)  }$ and $\overline{\mathfrak{D}%
^{\prime}\left(  \overline{qq_{2}},c_{2}\right)  }$ along $\overline{q_{1}q}$
and $\overline{qq_{2}}$ respectively.
\end{lemma}

\begin{proof}
(A) follows from Stoilow's theorem directly when $x_{j}\in U$. (B) follows
from (A) and the assumption $\left(  f,\overline{U}\right)  \in\mathcal{F}$,
by considering the extension of $f$ which is an OPCOFOM in a neighborhood of
$x_{j}$ in $\mathbb{C}.$
\end{proof}

\begin{remark}
\label{notation}\label{230330-19:30}We list more elementary conclusions
deduced from the previous lemma directly and more notations. Let
$\Sigma=(f,\overline{U})$ with $\Sigma\in\mathcal{F}$, $q\in f(\overline{U}),$
$\delta,$ $x_{j},$ $U_{j}$ and $\alpha_{1}+\alpha_{2}$ be given as in Lemma
\ref{cov-1}.

(A) If for some $j,$ $x_{j}\in\Delta,$ then by Lemma \ref{cov-1} (A), $f$ is a
BCCM in the neighborhood $U_{j}$ of $x_{j}$ in $\Delta,$ and the order
$v_{f}(x_{j})$ of $f$ at $x_{j}$ is well defined, which is a positive integer,
and $f$ is a $v_{f}(x_{j})$-to-$1$ CCM on $U_{j}\backslash\{x_{j}\}.$

(B) If for some $j,x_{j}\ $is contained in $\partial\Delta,$ then, using
notations in Lemma \ref{cov-1} (B), there are two possibilities:

(B1) $q_{1}=q_{2},$ the interior angle of $\Sigma$ at $x_{j}$ equals
$\omega_{j}\pi,$ and the order $v_{f}\left(  x_{j}\right)  $ is defined to be
$\omega_{j}/2,$ which is a positive integer.

(B2) $q_{1}\neq q_{2},c_{1}+c_{2}$ is a simple arc from $q_{1}$ to $q,$ and
then to $q_{2}$. In this case the interior angle of $\Sigma$ at $x_{j}$ equals
$\omega_{j}\pi+\varphi_{1}+\varphi_{2},$ where $\varphi_{1}$ and $\varphi_{2}$
are the interior angles of $\mathfrak{D}^{\prime}\left(  \overline{q_{1}%
q},c_{1}\right)  $ and $\mathfrak{D}^{\prime}\left(  \overline{qq_{2}}%
,c_{2}\right)  $ at the cusps, and we defined the order of $f$ at $x_{j}$ to
be the least integer $v_{f}\left(  x_{j}\right)  $ with $v_{f}\left(
x_{j}\right)  \geq\left(  \omega_{j}\pi+\varphi_{1}+\varphi_{2}\right)  /2\pi
$. Since $\omega_{j}\pi+\varphi_{1}+\varphi_{2}\geq\omega_{j}\pi>0,$ we have
$v_{f}\left(  x_{j}\right)  \geq1$ and $f$ is injective on $U_{j}%
\backslash\left\{  c_{1}+c_{2}\right\}  $ iff $v_{f}\left(  x_{j}\right)  =1.$
This is also easy to see by Corollary \ref{cov-2} (v).

(C) The number $v_{f}(x_{j})$ can be used to count path lifts with the same
initial point $x_{j}:$ when $x_{j}\in\Delta,$ any sufficiently short line
segment on $S$ starting from $q=f(x_{j})\ $has exactly $v_{f}\left(
x_{j}\right)  $ $f$-lifts starting from $x_{j}$ and disjoint in $\Delta
\backslash\{x_{j}\};$ and when $x_{j}\in\partial\Delta,$ for each arc $\beta$
of the two sufficiently short arcs of $\partial\Delta$ with initial point
$x_{j}$, $(f,\beta)$ is simple and has exactly $v_{f}(x_{j})-1$ $f$-lifts
$\left\{  \beta_{j}\right\}  _{j=1}^{v_{f}\left(  x_{j}\right)  -1}$ with the
same initial point $x_{j},$ $\beta_{j}\backslash\{x_{j}\}\subset\Delta$ for
each $j$ and they are disjoint in $\Delta.$ This is also easy to see by
Corollary \ref{cov-2} (v).

(D) A point $x\in\overline{U}$ is called a \emph{branch point} of $f$ (or
$\Sigma$) if $v_{f}(x)>1,$ or otherwise called a regular point if
$v_{f}\left(  x\right)  =1.$ We denote by $C_{f}$ the set of all branch points
of $f,$ and $CV_{f}$ the set of all branch values of $f.$ For a set
$A\subset\overline{U},$ we denote by $C_{f}\left(  A\right)  =C_{f}\cap A$ the
set of branch points of $f$ located in $A,$ and by $CV_{f}(K)=CV_{f}\cap
K\label{CVF}$ the set of branch values of $f$ located in $K\subset S.$ We will
write%
\[
C_{f}^{\ast}\left(  A\right)  =C_{f}\left(  A\right)  \backslash f^{-1}%
(E_{q})\mathrm{\ and\ }C_{f}^{\ast}=C_{f}\backslash f^{-1}(E_{q})=C_{f}\left(
\overline{U}\right)  \backslash f^{-1}(E_{q}).
\]

(E) For each $x\in\overline{U},$ $b_{f}\left(  x\right)  =v_{f}\left(
x\right)  -1\ $is called the branch number of $f$ at $x,$ and for a set
$A\subset\overline{U}$ we write $B_{f}\left(  A\right)  =\sum_{x\in A}%
b_{f}\left(  x\right)  .$ Then we have $b_{f}\left(  x\right)  \neq0$ iff
$C_{f}\left(  x\right)  =\{x\},$ and $B_{f}\left(  A\right)  =\sum_{x\in
C_{f}\left(  A\right)  }b_{f}(x).$ We also define%
\[
B_{f}^{\ast}\left(  A\right)  =B_{f}\left(  A\backslash f^{-1}(E_{q})\right)
.
\]
Then $B_{f}^{\ast}\left(  A\right)  \geq0,$ equality holding iff $C_{f}^{\ast
}\left(  A\right)  =\emptyset.$ When $A=\overline{U}$ is the domain of
definition of $f,$ we write%
\[
B_{f}=B_{f}\left(  \overline{U}\right)  \mathrm{\ and\ }B_{f}^{\ast}%
=B_{f}^{\ast}\left(  \overline{U}\right)  .\label{BBC}%
\]

\end{remark}

\begin{definition}
\label{nod}\label{230330-22:41}Let $\Sigma=\left(  f,\overline{U}\right)  $ be
a surface in $\mathcal{F},$ let $x\in\overline{U},$ and let $V$ be a
(relatively)\footnote{This means that $V=\overline{U}\cap V^{\ast},$ where
$V^{\ast}$ is an open set on $\mathbb{C}$. Thus when $x\in\partial U,$ $V$
contains the neighborhood $V^{\ast}\cap\partial U$ of $x$ in $\partial U.$}
open subset of $\overline{U}.$ The pair $\left(  x,V\right)  $ is called a
(relatively) disk of $\Sigma$ with center $x$ and radius $\delta,$ if $x\ $and
$\overline{V}$ satisfy all conclusions in Lemma \ref{cov-1} (A) or (B) as
$x_{j}$ and $\overline{U_{j}}\ $and $\delta.$
\end{definition}

\begin{remark}
\label{nod-1}If $x\in U$ and $\left(  x,V\right)  $ is a disk of
$\Sigma=\left(  f,\overline{U}\right)  $ with radius $\delta,$ then $V$ is
open in $U$ and $\overline{V}\subset U.$

If $x\in\partial U$ and $\left(  x,V\right)  $ is a disk of $\Sigma$ with
radius $\delta,$ then $\partial V$ has a partition $\alpha_{1}\left(
x^{\prime},x\right)  +\alpha_{2}\left(  x,x^{\prime\prime}\right)  +\alpha
_{3}\left(  x^{\prime\prime},x^{\prime}\right)  $ such that $\alpha_{1}%
+\alpha_{2}\ $is the \emph{old boundary} of $V$, $\alpha_{3}$ is the \emph{new
boundary} of $V$, $V=V^{\circ}\cup\left(  \alpha_{1}+\alpha_{2}\right)
^{\circ},$ $c_{1}=\left(  f,\alpha_{1}\right)  $ and $c_{2}=\left(
f,\alpha_{2}\right)  $ are SCC arcs, $c_{3}=\left(  f,\alpha_{3}\right)  $ is
a locally SCC arc, which maybe more than a circle and which is
contained\label{bdr} in the circle $d\left(  f(x),w\right)  =\delta,$ and the
interior angles of $\left(  f,\overline{V}\right)  $ at $x^{\prime}$ and
$x^{\prime\prime}$ are contained in $[\frac{7\pi}{16},\frac{9\pi}{16}].$ If
$x$ is fixed and $\delta$ tends to $0,$ then $x^{\prime}$ and $x^{\prime
\prime}$ tend to $x$ and the interior angles of $\left(  f,\overline
{V}\right)  $ at $x^{\prime}$ and $x^{\prime\prime}$ both tend to $\pi/2.$ The
paths $-\alpha_{1}$ and $\alpha_{2}$ are called \emph{boundary radii} of the
disk $V.$
\end{remark}

Now we can state a direct Corollary to Lemma \ref{cov-1}.

\begin{corollary}
\label{cov-2}\label{230330-22:41 copy(1)}Let $\Sigma=\left(  f,\overline
{U}\right)  \in\mathcal{F}$ and let $\left(  x_{1},U_{1}\right)  $ be a disk
of $\Sigma$ with radius $\delta_{1}.$ Then, the following hold.

(i) $f$ is locally homeomorphic on $\overline{U_{1}}\backslash\{x_{1}\}$; and
if $\left(  x_{1},U_{1}^{\prime}\right)  $ is another disk of $\Sigma$ with
radius $\delta_{1}^{\prime}>\delta_{1},$ then $\overline{U_{1}}\subset
U_{1}^{\prime},$ whether $x_{1}\ $is in $\partial U$ or $U$.

(ii) If $f$ is homeomorphic in some neighborhood of $x_{1}$ in $\overline{U}$
(which may be arbitrarily small), or if $f$ locally homeomorphic on
$\overline{U}$, then the disk $\left(  x_{1},\overline{U_{1}}\right)  $ is a
one sheeted closed domain of $\Sigma,$ say, $f$ restricted to $\overline
{U_{1}}$ is a homeomorphism onto $f\left(  \overline{U_{1}}\right)  .$

(iii) For each $x_{2}\in U_{1}\backslash\{x_{1}\},$ any closed disk $\left(
x_{2},\overline{U_{2}}\right)  $ of $\Sigma$ is a one sheeted closed domain of
$\Sigma,$ moreover, $\overline{U_{2}}\subset U_{1}$ when the radius of
$\left(  x_{2},U_{2}\right)  $ is smaller than $\delta-d\left(  f(x_{1}%
),f(x_{2})\right)  .$

(iv) If $x_{1}\in\partial U$, $f$ is regular at $x_{1}$ and $\left(
f,\partial U\right)  $ is circular near $x_{1},$ then $\left(  f,\overline
{U_{1}}\right)  $ is a \emph{convex} and one sheeted closed domain of
$\Sigma,$ which is in fact the closed lens $\overline{\mathfrak{D}\left(
I,c_{1},c_{1}^{\prime}\right)  }$, where $c_{1}$ and $c_{1}^{\prime}$ are
circular subarcs of $\partial\Sigma$ and the circle $\partial D\left(
f(x_{1}),\delta_{1}\right)  ,$ $I$ is the common chord, and the three paths
$c_{1},-c_{1}^{\prime},I$ have the same initial point. Moreover, if $\Sigma$
is regular at $x_{1}$ and $\partial\Sigma$ is straight near $x_{1},$ then
$f(\overline{U_{1}})=\overline{\mathfrak{D}^{\prime}\left(  -I,c_{1}^{\prime
}\right)  }=\overline{\mathfrak{D}^{\prime}\left(  -c_{1},c_{1}^{\prime
}\right)  }\ $is "half" of the disk $\overline{D(f(x_{1}),\delta_{1})}$ on the
left hand side of "diameter" $c_{1}$ (see Definition \ref{lune-lens} for
lenses and lunes).

(v) For any $x\in\overline{U_{1}},$ there exists a path $I\left(
x_{1},x\right)  $ in $\overline{U_{1}}$ from $x_{1}$ to $x$ such that
$I\left(  x_{1},x\right)  $ is the unique $f$-lift of $\overline{f(x_{1}%
)f(x)}.$ That is to say, $\left(  f,\overline{U_{1}}\right)  $ can be foliated
by the family of straight line segments $\{\left(  f,I\left(  x_{1},x\right)
\right)  :x\in\partial U_{1}\}$ and for each pair $\left\{  I\left(
x_{1},x\right)  ,I\left(  x_{1},y\right)  \right\}  $ of the family $\left\{
I\left(  x_{1},x\right)  :x\in\partial U_{j}\right\}  $ with $x\neq y$, one
has $I\left(  x_{1},x\right)  \cap I\left(  x_{1},y\right)  =\{x_{1}\}.$
\end{corollary}

\begin{lemma}
\label{int-ang}\label{int-arg}Let $\Sigma=(f,\overline{U})\in\mathcal{F}$ and
let $x\in\partial U.$ Then $\angle(\Sigma,x)>0.$
\end{lemma}

\begin{proof}
This is clear by Lemma \ref{cov-1} (B) and Remark \ref{notation} (B).
\end{proof}

If the assumption $\Sigma=(f,\overline{U})\in\mathcal{F}$ is not be satisfied,
the conclusion may fail. For example, for the convex closed half disk
$\overline{\Delta^{+}}$ and the disk $B=\{z\in\mathbb{C}:\left\vert z-\frac
{1}{2}\right\vert <1\}$ in $\mathbb{C},$ $T=\overline{\Delta^{+}}\backslash B$
can be regarded as a surface on $S$ (via the sterographic projection), whose
interior angle at the origin equals $0,$ and it is clear that $T$
$\notin\mathcal{F}.$ In fact the part of $\partial T$ lying on $\partial B$ is
not convex on $S$.

Lemma \ref{cov-1} also directly implies the following lemma.

\begin{lemma}
\label{inj}\label{230331-8:45}Let $(f,\overline{U})\in\mathcal{F}.$ Then the
following hold.

(A) For each $p\in\overline{U},f$ restricted to some neighborhood of $p$ in
$\overline{U}$ is a homeomorphism if one of the following alternatives holds.

(A1) $p\in U$ and $p$ is a regular point of $f.$

(A2) $p\in\partial U,$ $p$ is a regular point of $f$ and $(f,\partial U)$ is
simple in a neighborhood of $p$ on $\partial U$.

(B) For any SCC arc $\left(  f,\alpha\right)  $ of $\partial\Sigma=\left(
f,\partial U\right)  ,$ $f$ restricted to a neighborhood of $\alpha^{\circ}$
in $\overline{U}$ is homeomorphic if and only if $h$ has no branch point on
$\alpha^{\circ}.$
\end{lemma}

(A) follows from the definition of $\mathcal{F}.$ (B) follows from (A). The
hypothesis in (A2) that $(f,\partial U)$ is simple is necessary:
$f(z)=z^{2},z\in\overline{\Delta^{+}},$ is regular at $z=0$ but not injective
in any neighborhood of $0.$

\begin{lemma}
\label{Ri}(\cite{Ri} p. 32--35) \label{230331-8:47}Let $\Sigma=(f,\overline
{\Delta})\in\mathcal{F}$ and let $\beta=\beta\left(  q_{1},q_{2}\right)  $ be
a path on $S$ from $q_{1}$ to $q_{2}$. Assume that $\alpha=\alpha\left(
p_{1},p\right)  $ is a path in $\overline{\Delta}$ from $p_{1}$ to $p$ which
is an $f$-lift of a subarc $\beta\left(  q_{1},q\right)  $ of $\beta$ from
$q_{1}$ to $q$, with $\alpha\backslash\{p_{1}\}\subset\Delta$ and $f\left(
p_{1}\right)  =q_{1}.$ Then $\alpha$ can be extended to an $f$-lift
$\alpha^{\prime}=\alpha\left(  p_{1},p^{\prime}\right)  $ of a longer subarc
of $\beta$ with $\alpha^{\prime\circ}\subset\Delta,$ such that either
$p^{\prime}\in\partial\Delta,$ or $p^{\prime}\in\Delta$ and $\alpha^{\prime}$
is an $f$-lift of the whole path $\beta.$\label{6868-16}\label{230331-9:33}
\end{lemma}

The following lemma is obvious.

\begin{lemma}
\label{anti-pod}Any two distinct great circles on $S$ intersect at exactly two
points, which are antipodal points on $S.$
\end{lemma}

The following result follows from Definition \ref{AhlforsS}, which is
essentially Lemma 5.2 of \cite{Z1}.

\begin{lemma}
\label{continue0}Let $(f,\overline{\Delta})\in\mathbf{F}$ and let $D$ be a
Jordan domain on $S$ such that $f^{-1}$ has a univalent
branch\footnote{Univalent branch for the inverse of an OPCOFOM always means an
OPH in this paper.} $g$ defined on $D.$ Then $g$ can be extended to a
univalent branch of $f^{-1}$ defined on $\overline{D}$.
\end{lemma}

The following result follows from the Argument principle.

\begin{lemma}
\label{b-in}Let $D_{1}$ and $D_{2}$ be Jordan domains on $\mathbb{C}$ or $S$
and let $f:\overline{D_{1}}\rightarrow\overline{D_{2}}$ be a mapping such that
$f:\overline{D_{1}}\rightarrow f(\overline{D_{1}})$ is a homeomorphism. If
$f(\partial D_{1})\subset\partial D_{2},$ then $f(\overline{D_{1}}%
)=\overline{D_{2}}$.
\end{lemma}

The following result is a generalization of the existence of lifts of curves
for a CCM.

\begin{lemma}
\label{cov-3}\label{2023-3-24}\label{2023-4-03-21:29}Let $U$ be a domain on
$S$ enclosed by a finite number of Jordan curves, $f:\overline{U}\rightarrow
S\mathcal{\ }$be a finite-to-one mapping which is locally homeomorphic on
$\overline{U}$, $\Gamma_{n}:[0,1]\rightarrow S$ be a sequence of paths on $S$
which converges to a path $\Gamma_{0}:[0,1]\rightarrow S$ uniformly, and let
$a\in\overline{U}$. If for each $n,\Gamma_{n}$ has an $f$-lift $I_{n}%
:[0,1]\rightarrow\overline{U}\ $from $a,$ say, $I_{n}$ is a path with
\[
f\left(  I_{n}\left(  s\right)  \right)  =\Gamma_{n}\left(  s\right)
\mathrm{\ for\ all\ }s\in\lbrack0,1],
\]
and $I_{n}\left(  0\right)  =a.$ Then $I_{n}\left(  s\right)  $ uniformly
converges to a path $I_{0}\left(  s\right)  ,s\in\lbrack0,1],$ in
$\overline{U},$ such that $I_{0}$ is an $f$-lift of $\Gamma_{0}$ with
$I_{0}\left(  a\right)  =a.$

\begin{proof}
We first show the following.

\begin{claim}
\label{liftmid1}For any $s_{0}\in\lbrack0,1],$ if $\lim_{n\rightarrow\infty
}I_{n}\left(  s_{0}\right)  \rightarrow a_{0},$ then $s_{0}$ has a
neighborhood $N_{s_{0}}\left(  \varepsilon\right)  =[s_{0}-\varepsilon
,s_{0}+\varepsilon]\cap\lbrack0,1]$ in $[0,1],$ such that $I_{n}(s)$ converges
to an arc $I_{0}\left(  s\right)  $ uniformly on $N_{s_{0}}$ with $I_{0}%
(s_{0})=a_{0}$ and $f(I_{0}\left(  s\right)  )=\Gamma_{0}(s)$ for all $s$ on
$N_{s_{0}}.$
\end{claim}

Let $w_{0}=\Gamma_{0}\left(  s_{0}\right)  $. Then $w_{0}=f\left(
a_{0}\right)  $ and $f^{-1}(w_{0})=\{a_{j}\}_{j=0}^{m}$ is a finite set. Since
$f$ is locally homeomorphic, each $a_{j}$ has a connected and (relatively)
open neighborhood $U_{j}$ in $\overline{U}$ such that $U_{i}\cap
U_{j}=\emptyset$ when $i\neq j$ and $f:\overline{U_{j}}\rightarrow
f(\overline{U_{j}})$ is a homeomorphism for each $j=0,1,2,\dots,m$. Then we have

\begin{claim}
\label{liftmid2}$w_{0}$ is outside the compact subset $f\left(  \overline
{U}\backslash\cup_{j=0}^{m}U_{j}\right)  \ $of $S,$ say, there
exists\label{2023-4-3:22:55+s} a disk $D\left(  w_{0},\delta\right)  $ on $S$
such that $f^{-1}(D\left(  w_{0},\delta\right)  )\subset\cup_{j=0}^{m}U_{j}.$
\end{claim}

$s_{0}$ has a connected neighborhood $N_{s_{0}}$ in $[0,1]$ so that
$\Gamma_{0}\left(  N_{s_{0}}\right)  \subset D\left(  w_{0},\delta/2\right)  $
and thus $\Gamma_{n}\left(  N_{s_{0}}\right)  \subset D\left(  w_{0}%
,\delta\right)  $ for all $n>n_{0}$ for some $n_{0}>0.$ Then $I_{n}\left(
N_{s_{0}}\right)  \subset\cup_{j=0}^{m}U_{j}$ and, since $I_{n}(N_{s_{0}})$ is
connected, we have $I_{n}\left(  N_{s_{0}}\right)  \subset U_{0}$ for all
$n>n_{0}.$ Therefore, $\Gamma_{n}\left(  N_{s_{0}}\right)  \subset f(U_{0})$
for $n>n_{0}.$ It is clear that $\Gamma_{0}\left(  N_{s_{0}}\right)
\subset\overline{f(U_{0})}$ since $\Gamma_{n}(N_{s_{0}})$ converges to
$\Gamma_{0}\left(  N_{s_{0}}\right)  .$ Since $f:\overline{U_{0}}\rightarrow
f(\overline{U_{0}})$ is homeomorphic, we conclude that $I_{n}\left(  s\right)
$ converges to the path $I_{0}\left(  s\right)  =f^{-1}(\Gamma_{0}\left(
s\right)  )\cap\overline{U_{0}}$ uniformly for $s\in N_{s_{0}}\ $with
$I_{0}(s_{0})=a_{0}.$ It is obvious that $I_{0}\left(  s\right)  ,s\in
N_{s_{0}},$ is an $f$-lift of $\Gamma_{0}\left(  s\right)  ,s\in N_{s_{0}}%
,\ $and Claim \ref{liftmid1} is proved.

Let $A\ $be the set of $t\in\lbrack0,1]$ such that $\lim_{n\rightarrow\infty
}I_{n}\left(  t\right)  $ exists. Then $A$ is an open subset of $[0,1]$ by
Claim \ref{liftmid1}. Ley $B=[0,1]\backslash A.$ We show that $B$ is also an
open subset of $[0,1]$. Let $s_{1}\in B.$ Since $\overline{U}$ is compact,
$I_{n}\left(  s_{1}\right)  $ has two subsequences $I_{n_{k}^{j}}\left(
s_{1}\right)  $ such that $I_{n_{k}^{j}}\left(  s_{1}\right)  \rightarrow
a_{1j}\left(  k\rightarrow\infty\right)  $ with $j=1,2$ and $a_{11}\neq
a_{12}.$ Then Claim \ref{liftmid1} applies to $\Gamma_{n_{k}^{j}}$ and
$\Gamma_{0},$ and $s_{1}$ has a neighborhood $N_{s_{1}}$ in $[0,1]$ such that
$I_{n_{k}^{j}}\left(  s\right)  $ converges uniformly to a path $I_{0j}\left(
s\right)  ,s\in N_{s_{1}},$ and $I_{0j}$ is an $f$-lift of $\Gamma
_{0}|_{N_{S_{1}}},j=1,2$. Then both $I_{0j}\left(  s\right)  $ are continuous
with $I_{01}(s_{1})=a_{11}\neq a_{12}=I_{02}(s_{1}),$ and then $I_{01}\cap
I_{02}=\emptyset$ when $N_{s_{1}}$ is chosen small enough. This implies that
$I_{n}\left(  s\right)  $ cannot converges as $n\rightarrow\infty$ for each
$s\in N_{s_{1}}$ and so $B$ is open.

Since $0\in A$. We have $A=[0,1]$ and $B=\emptyset.$ We have proved that
$I_{n}$ converges to a path $I_{0}$ uniformly on $[0,1],$ and it is clear that
$I_{0}$ is the $f$-lift of $\Gamma_{0}.\label{2023-3-24-1}%
\label{2023-4-04-7:30}$\newpage
\end{proof}
\end{lemma}

\section{Sewing two surfaces along a common boundary arc\label{Sect3}}

We now introduce the method to sew two surfaces sharing a common boundary arc.
We let $H^{+}$ and $H^{-}\label{H+-}$ be the upper and lower open half planes
of $\mathbb{C}$, Then $H^{+}$ and $H^{-1}$ can be regarded as open hemispheres
on $S$ and $\overline{H^{+}}$ and $\overline{H^{-1}}\label{Hbar+-}$ can be
regarded as closure of $H^{+}$ and $H^{-}$ on $S.$ For a closed curve $\gamma$
in $\mathbb{C},$ we write $\gamma^{\pm}=\overline{H^{\pm}}\cap\gamma.$ But
recall that $\Delta^{\pm}\label{Del+-}\ $always denotes $H^{\pm}\cap\Delta,$
not $\overline{H^{\pm}}\cap\Delta.$ Then $\left(  \partial\Delta\right)
^{\pm}=\left(  \partial\Delta\right)  \cap\overline{H^{\pm}}=\left\{  z=e^{\pm
i\theta}:\theta\in\lbrack0,\pi]\right\}  \label{+-arc}.\label{2023-4-04-7:32}$

\begin{example}
\label{exam1}A surface $\Sigma=\left(  f,\overline{\Delta}\right)  $ can be
cut into two subsurfaces $\Sigma_{1}=\left(  f,\overline{\Delta^{+}}\right)  $
and $\Sigma_{2}=\left(  f,\overline{\Delta^{-}}\right)  $ by the diameter
$[-1,1]$ of $\overline{\Delta}.$ Conversely, we can recover $\Sigma$ by
sewing\label{sew39} $\Sigma_{1}$ and $\Sigma_{2}$ along $\left(  f,\left[
-1,1\right]  \right)  .$ The interval $\left[  -1,1\right]  $ in
$\overline{\Delta}$ is called the \emph{suture line }when we sew $\Sigma_{1}$
and $\Sigma_{2}.$ This trivial observation can be generalized in Lemma
\ref{patch-1} and Lemma \ref{patch}.$\label{2023-4-04-7:40}$
\end{example}

\begin{example}
\label{exam-glue-1}Let $\Sigma=\left(  f,\overline{\Delta}\right)  \ $be a
surface and $B=\{z\in\mathbb{C}:|z-1/2|<1/2\}.$ Then $\partial B$ cut the
surface $\Sigma$ into two subsurfaces $\Sigma_{1}=\left(  f,\overline{\Delta
}\backslash B\right)  $ and $\Sigma_{2}=\left(  f,\overline{B}\right)  .$
Conversely, we can glue $\Sigma_{1}$ and $\Sigma_{2}$ to recover the surface
$\Sigma.$ This trivial observation can be generalized in Corollary
\ref{glue-1}.$\label{2023-4-04-7:42}$
\end{example}

\begin{lemma}
\label{patch-1}For $j=1,2,$ let $\Sigma_{j}=(f_{j},\overline{U_{j}})$ be a
surface and let $\alpha_{j}=\alpha_{j}\left(  x_{j1},x_{j2}\right)  $ be a
proper arc of $\partial U_{j}$ such that $\left(  f_{j},\alpha_{j}\right)  $
is a simple arc with distinct endpoints. If
\begin{equation}
\gamma=(f_{1},\alpha_{1})\sim-(f_{2},\alpha_{2}), \label{pp13}%
\end{equation}
then $(f_{1},\overline{U_{1}})$ and $(f_{2},\overline{U_{2}})\ $can be
\emph{sewn along}\label{sew38} $\gamma=\left(  f_{1},\alpha_{1}\right)  $ to
become a surface $\Sigma_{3}=(f_{3},\overline{\Delta})$ with \emph{suture
line} $[-1,1]$, such that the following hold:

(i). There exist \emph{orientation-preserving homeomorphisms} (OPHs)
\label{OPH} $h_{1}:\overline{U_{1}}\rightarrow\overline{\Delta^{+}}$ and
$h_{2}:\overline{U_{2}}\rightarrow\overline{\Delta^{-}},$ called
\emph{identification mappings} (\emph{IM}s), such that
\begin{equation}
(h_{1},\alpha_{1})\sim\lbrack-1,1]\sim-\left(  h_{2},\alpha_{2}\right)
=\left(  h_{2},-\alpha_{2}\right)  , \label{pp5}%
\end{equation}%
\begin{equation}
f_{1}\circ h_{1}^{-1}\left(  x\right)  =f_{2}\circ h_{2}^{-1}(x),\forall
x\in\lbrack-1,1], \label{pp6}%
\end{equation}
and%
\begin{equation}
f_{3}(z)=\left\{
\begin{array}
[c]{l}%
f_{1}\circ h_{1}^{-1}(z),z\in\overline{\Delta^{+}},\\
f_{2}\circ h_{2}^{-1}(z),z\in\overline{\Delta^{-}}\backslash\lbrack-1,1],
\end{array}
\right.  \label{pp7}%
\end{equation}
is a well defined OPCOFOM, and we have the equivalent relations
\[
(f_{3},\overline{\Delta^{+}})\sim(f_{1},\overline{U_{1}}),(f_{3}%
,\overline{\Delta^{-}})\sim(f_{2},\overline{U_{2}}),
\]%
\[
\partial\Sigma_{3}=\left(  f_{3},\left(  \partial\Delta\right)  ^{+}\right)
+\left(  f_{3},\left(  \partial\Delta\right)  ^{-}\right)  \sim\left(
f_{1},\left(  \partial U_{1}\right)  \backslash\alpha_{1}^{\circ}\right)
+\left(  f_{2},\left(  \partial U_{2}\right)  \backslash\alpha_{2}^{\circ
}\right)  ,
\]
and
\[
(f_{3},[-1,1])\sim(f_{1},\alpha_{1})\sim(f_{2},-\alpha_{2}).
\]

(ii)
\[
L(\partial\Sigma_{3})=L(\partial\Sigma_{1})+L(\partial\Sigma_{2})-2L(\gamma),
\]%
\[
A(\Sigma_{3})=A\left(  \Sigma_{1}\right)  +A(\Sigma_{2}),
\]%
\[
\overline{n}\left(  \Sigma_{3}\right)  =\overline{n}\left(  \Sigma_{1}\right)
+\overline{n}\left(  \Sigma_{2}\right)  +\#\left(  \gamma^{\circ}\cap
E_{q}\right)  ,
\]
and%
\begin{equation}
R(\Sigma_{3})=R(\Sigma_{1})+R(\Sigma_{2})-4\pi\#\left(  \gamma^{\circ}\cap
E_{q}\right)  . \label{RRR-1}%
\end{equation}

(iii). $z\in C_{f_{3}}\left(  \overline{\Delta}\backslash\{-1,1\}\right)  $ if
and only if $h_{1}^{-1}(z)\in C_{f_{1}}\left(  \overline{U_{1}}\backslash
\partial\alpha_{1}\right)  \mathrm{\ or\ }h_{2}^{-1}(z)\in C_{f_{2}}\left(
\overline{U_{2}}\backslash\partial\alpha_{2}\right)  $. In particular, if
$f_{1}(\partial\alpha_{1})\subset E_{q},$ then $f_{2}(\partial\alpha
_{2})\subset E_{q}$ and in addition
\begin{equation}
CV_{f_{3}}(S\backslash E_{q})=CV_{f_{1}}(S\backslash E_{q})\cup CV_{f_{2}%
}(S\backslash E_{q}). \label{RRR-2}%
\end{equation}
$\label{2023-4-04-8:01}$
\end{lemma}

Recall that $C_{f}\left(  A\right)  \ $is the set of branch points of
$f\ $located in $A$ and $CV_{f}\left(  T\right)  $ is the branch value of $f$
located in $T$ (see Remark \ref{notation} (D)), $\partial\alpha$ denotes the
endpoints of $\alpha$ and $\alpha^{\circ}$ denotes the interior of $\alpha$
(see Definition \ref{interior}). The condition $(f_{1},\alpha_{1})\sim
(f_{2},-\alpha_{2})$ is crucial (see Remark \ref{finite} for the relation
$\sim$) and note that $(f_{2},-\alpha_{2})=-\left(  f_{2},\alpha_{2}\right)  $
and $\left(  f_{2},\alpha_{2}\right)  $ are the same path with opposite
direction. Two copies of the hemisphere $\overline{H^{+}}$ on $S$ cannot be
sewn along their common boundary section $\overline{\infty,-1,0}\subset S$ to
become a surface, but $\overline{H^{+}}$ and $\overline{H^{-}}$ can be
sewn\label{sew37} along $\overline{\infty,-1,0}$ to become the surface
$\left(  f_{3},\overline{\Delta}\right)  ,$ where $f_{3}|_{\overline
{\Delta^{\pm}}}$ are homeomorphisms from $\overline{\Delta^{\pm}}$ onto
$\overline{H^{\pm}},$ and $f_{3}$ maps $[-1,1]$ onto $\overline{\infty,-1,0}$,
$\left(  \partial\Delta\right)  ^{+}=\left(  \partial\Delta\right)
\cap\overline{H^{+}}$ onto $\overline{0,1,\infty},$ and $\left(
\partial\Delta\right)  ^{-}=\left(  \partial\Delta\right)  \cap\overline
{H^{-}}$ onto $\overline{\infty,1,0}.$

\begin{proof}
The conclusion (i) in fact gives a routine how to sew \label{sew36}$\Sigma
_{1}$ and $\Sigma_{2},$ which is inspired by Example \ref{exam1}. By
(\ref{pp13}), there exists an orientation-preserving homeomorphism (OPH)
\footnote{Note that $-\alpha_{2}$ is the same path with opposite direction,
not the set $\{-y:y\in\alpha_{2}\}.$} $\varphi:\alpha_{1}\rightarrow
-\alpha_{2}$ such that
\[
\left(  f_{1},\alpha_{1}\right)  =\left(  f_{2}\circ\varphi,\alpha_{1}\right)
,
\]
that is
\[
f_{2}\left(  \varphi(x)\right)  \equiv f_{1}(x),\forall x\in\alpha_{1}.
\]

Let $h_{1}:\overline{U_{1}}\rightarrow\overline{\Delta^{+}}$ be any OPH such
that $h_{1}(\alpha_{1})=[-1,1].$ Then let $h_{2}:\overline{U_{2}}%
\rightarrow\overline{\Delta^{-}}$ be an OPH such that
\begin{equation}
h_{2}\left(  y\right)  \equiv h_{1}\left(  \varphi^{-1}(y)\right)  ,\forall
y\in\alpha_{2}. \label{pp14}%
\end{equation}
In fact, $h_{2}|_{\alpha_{2}}$ defined by (\ref{pp14}) is an OPH from
$\alpha_{2}$ onto $[1,-1]$ and can be extended to be an OPH $h_{2}$ from
$\overline{U_{2}}$ onto $\overline{\Delta^{-}}.$ The pair of $h_{1}$ and
$h_{2}$ are the desired mappings satisfying (i). Then (ii) is trivial to verify.

To prove (iii) we may assume that $\Sigma_{1}$ and $\Sigma_{2}$ are the
surfaces $\Sigma_{\pm}=(f_{\pm},\overline{\Delta^{\pm}})$ such that $f_{\pm}$
agree on $[-1,1],$ and then $f_{3}$ defined by $f_{\pm}$ on $\overline
{\Delta^{\pm}}\ $is an OPCOFOM. Then $f_{\pm}$ are the restrictions of $f_{3}$
to $\overline{\Delta^{\pm}},$ and thus $x\in\left(  -1,1\right)  $ is a branch
point of $f_{3},$ say $x\in C_{f_{3}}\cap\left(  -1,1\right)  ,$ iff $x$ is a
branch point of $f_{+}$ or $f_{-},$ say $x\in C_{f_{1}}\left(  -1,1\right)
\cup C_{f_{2}}\left(  -1,1\right)  .$ In consequence we have
\[
C_{f_{3}}\left(  \overline{\Delta}\backslash\{-1,1\}\right)  =C_{f_{1}}\left(
\overline{\Delta^{+}}\backslash\{-1,1\}\right)  \cup C_{f_{2}}\left(
\overline{\Delta^{-}}\backslash\{-1,1\}\right)  ,
\]
and then all conclusions of (iii) follow.$\label{2023-4-04-9:51}$
\end{proof}

The suture line may not be straight: the sewn surface $\Sigma_{3}=\left(
f_{3},\overline{\Delta}\right)  $ can be reparametrized as $\Sigma^{\prime
}=\left(  f_{3}^{\prime},\overline{U}\right)  $ with $f_{3}^{\prime}%
=f_{3}\circ\psi,$ where $\psi$ is a OPH from $\overline{U}$ onto
$\overline{\Delta} $ for some Jordan domain $U,$ and for $\Sigma^{\prime},$
the suture line becomes $h^{-1}([-1,1]).$

\begin{remark}
\label{abs-two}\label{sew35}In the previous lemma, the sewing process can be
understood as an abstract process via equivalent relations. Let $U_{j}%
,\Sigma_{j}=\left(  f_{j},\overline{U_{j}}\right)  ,\alpha_{j}\subset\partial
U_{j}$ satisfy all assumption of the previous lemma. Then we can define an
equivalent relation $\sim$ on the disjoint union $\overline{U}_{1}%
\sqcup\overline{U_{2}}$. For any pair of points $x$ and $y$ in $\overline
{U}_{1}\sqcup\overline{U_{2}},$ $x\sim y$ if and only if one of the three
conditions holds: (1)$\ x=y\in\overline{U_{1}},$ (2) $x=y\in\overline{U_{2}},$
(3) $x\in\alpha_{1},y\in\alpha_{2}$ and $f_{1}(x)=f_{2}(y).$ Since $\left(
f_{1},\alpha_{1}\right)  $ is simple and $\left(  f_{1},\alpha_{1}\right)
\sim-\left(  f_{2},\alpha_{2}\right)  ,$ $\left(  f_{2},\alpha_{2}\right)
\ $is also simple, and thus for each $x\in\alpha_{1}$ that $y\in\alpha_{2}$
with $f_{1}(x)=f_{2}(y)$ is unique, and vice versa for each $y\in\alpha_{2}.$
We write $[x]$ the equivalent class of $x:$%
\[
\lbrack x]=\{y\in\overline{U}_{1}\sqcup\overline{U_{2}}:y\sim x\}.
\]
Then for $x\in\left(  \overline{U_{1}}\backslash\alpha_{1}^{\circ}\right)
\sqcup\left(  \left(  \overline{U_{2}}\backslash\alpha_{2}^{\circ}\right)
\right)  ,$ $[x]$ contains only one point in $\overline{U}_{1}\sqcup
\overline{U_{2}}$, and for $x\in\alpha_{1},[x]$ contains two points in
$\overline{U}_{1}\sqcup\overline{U_{2}},$ say $[x]=\{x,y\}$ with
$f_{1}(x)=f_{2}(y)$ and $y\in\alpha_{2}$. The previous lemma show that the
quotient space $\overline{Q}=\left(  \overline{U}_{1}\sqcup\overline{U_{2}%
}\right)  /\sim,$ with the quotient topology, is topologically equivalent to
the unit disk $\overline{\Delta}.$ Then the sewn surface $\left(
f_{3},\overline{\Delta}\right)  $ can be identified as a representation of the
abstract space $\left(  \tilde{f},\overline{Q}\right)  ,$ where $\tilde
{f}\left(  [x]\right)  =f_{1}(x)$ when $x\in\lbrack x]\cap\overline{U_{1}},$
or $\tilde{f}\left(  [x]\right)  =f_{2}\left(  y\right)  $ when $y\in\lbrack
x]\cap\overline{U_{2}}$ $,\forall\lbrack x]\in\overline{Q},$ is well defined.
In this abstract version, $[\alpha_{1}]=[\alpha_{2}]$ is the suture line$,$
the quotient mapping $x\rightarrow\lbrack x]$ is the \textrm{IM.}

$\overline{U}_{1}$ may intersects $\overline{U_{2}}$, but for the union
$\overline{U_{1}}\sqcup\overline{U_{2}}$ a point in $\overline{U_{1}}$ and a
point in $\overline{U_{2}}$ are always regarded as distinct points. In fact we
may assume $\overline{U}_{1}$ and $\overline{U_{2}}$ are bounded with positive distance.

For the sewn surface $\left(  f_{3},\overline{\Delta}\right)  $ in the lemma
there exists a homeomorphism $h$ from $\overline{Q}=\left(  \overline{U}%
_{1}\sqcup\overline{U_{2}}\right)  /\sim$ onto $\overline{\Delta},$ such that
$h\left(  \left[  \alpha_{1}\right]  \right)  =[-1,1],$ with
\[
h([x])=\left\{
\begin{array}
[c]{c}%
h_{1}(x),x\in\overline{U_{1}},\\
h_{2}\left(  x\right)  ,x\in\overline{U_{2}}.
\end{array}
\right.
\]
The topological equivalence $h:Q\rightarrow\overline{\Delta}$ is also called
the \emph{IM}, when we use $\overline{\Delta}$ to represent $\overline{Q}.$

All surfaces obtained by sewing surfaces along arcs can be interpreted in this
way. Though this way is abstract, it keeps more information of the elder
surfaces than the new sewn \label{sew34}surface as $\Sigma_{3}$ in Lemma
\ref{patch}, and moreover, it is easier to state the process than that version
which involves the concrete \emph{IMs} and \emph{suture lines}%
.$\label{2023-4-04-9:58}$
\end{remark}

\begin{lemma}
\label{-path}Let $\gamma$ be a simple arc on $S$ with\footnote{Recall that
$\#\partial\gamma=2\ $means $\gamma$ contains two distinct endpoints.}
$\#\partial\gamma=2$ (recall that $\partial\gamma$ is the endpoints of
$\gamma$). Let $S_{\gamma}=\left(  f,\overline{U}\right)  $ be a surface in
$\mathbf{F}$ such that $f:U\rightarrow S\backslash\gamma$ is a homeomorphism
and $\partial S_{\gamma}=\left(  f,\partial U\right)  =\gamma-\gamma.$ Then
\begin{equation}
R(S_{\gamma})=4\pi\#\gamma^{\circ}\cap E_{q}+4\pi\#\left(  \partial
\gamma\right)  \cap E_{q}-8\pi\leq4\pi\#\gamma^{\circ}\cap E_{q}, \label{RRE}%
\end{equation}
equality holding if and only if $\partial\gamma\subset E_{q}.$
\end{lemma}

\begin{proof}
Since $\#E_{q}=q,\ $we have
\begin{align*}
\overline{n}\left(  S_{\gamma}\right)   &  =\#\left(  S\backslash
\gamma\right)  \cap E_{q}=q-\#\gamma\cap E_{q}=q-\#\gamma^{\circ}\cap
E_{q}-\#\left(  \partial\gamma\right)  \cap E_{q}\\
&  =q-2-\#\gamma^{\circ}\cap E_{q}+2-\#\left(  \partial\gamma\right)  \cap
E_{q}%
\end{align*}
Since $A(S_{\gamma})=4\pi,$ (\ref{RRE}) follows from
\begin{align*}
R(S_{\gamma})  &  =\left(  q-2\right)  A(S_{\gamma})-4\pi\overline{n}\left(
S_{\gamma}\right) \\
&  =4\pi\#\gamma^{\circ}\cap E_{q}-8\pi+4\pi\#\left(  \partial\gamma\right)
\cap E_{q}\\
&  \leq4\pi\#\gamma^{\circ}\cap E_{q},
\end{align*}
with equality if and only if $\partial\gamma\subset E_{q}%
.\label{2023-4-04-10:27}$
\end{proof}

Here is a useful result directly from Lemma \ref{patch-1}:

\begin{lemma}
\label{cut-off}Let $\Sigma=\left(  f,\overline{\Delta}\right)  $ be a surface
in $\mathbf{F}$ and $\alpha$ be a simple arc in $\overline{\Delta}$ such that
$\alpha^{\circ}\subset\Delta,\partial\alpha\subset\partial\Delta$ and
$\#\partial\alpha=2.$ If $\gamma=\left(  f,\alpha\right)  $ is simple, then
$\alpha$ cuts the disk $\Delta$ into two Jordan domains $\Delta_{1}$ and
$\Delta_{2}$ and for the surfaces $\Sigma_{j}=\left(  f,\overline{\Delta_{j}%
}\right)  ,j=1,2,$ we have%
\[
R(\Sigma)=R(\Sigma_{1})+R(\Sigma_{2})-4\pi\#\gamma^{\circ}\cap E_{q}.
\]
Moreover, if $\partial\alpha\subset E_{q},$ then we have%
\[
CV_{f}\left(  S\backslash E_{q}\right)  =CV_{f_{1}}(S\backslash E_{q})\cup
CV_{f_{2}}\left(  S\backslash E_{q}\right)  .
\]

\end{lemma}

\begin{proof}
In fact $\Sigma$ can be recovered by gluing the surfaces $\left(
f,\overline{\Delta_{1}}\right)  $ and $(f,\overline{\Delta_{2}})\ $along
$\gamma,$ and so the desired equalities follows from (\ref{RRR-1}) and
(\ref{RRR-2}).
\end{proof}

For two equivalent curves $\gamma_{j}:\alpha_{j}\rightarrow S,$ we will write
$\gamma_{1}=\gamma_{2}$ when there is no confusion. Then Lemma \ref{patch-1}
can be restated as simplified but more useful versions (Lemmas \ref{patch} and
\ref{ps-dm}).

\begin{lemma}
\label{patch}Let $\Sigma_{1}$ and $\Sigma_{2}$ be two surfaces in $\mathbf{F}$
such that
\[
\partial\Sigma_{1}=\gamma+\Gamma_{1}\mathrm{\ and\ }\partial\Sigma_{2}%
=-\gamma+\Gamma_{2},
\]
where $\gamma$ is a simple and proper arc of $\partial\Sigma_{1}$ which is not
closed, say, it has distinct endpoints. Then $-\gamma$ is a proper subarc of
$\partial\Sigma_{2},$ and $\Sigma_{1}$ and $\Sigma_{2}$ can be sewn along
$\gamma,$ resulting a surface $\Sigma_{3}=\left(  f_{3},\overline{\Delta
}\right)  $ such that the following hold.

(i).
\[
\partial\Sigma_{3}=\Gamma_{1}+\Gamma_{2},
\]%
\[
L(\partial\Sigma_{3})=L(\partial\Sigma_{1})+L(\partial\Sigma_{2})-2L(\gamma),
\]%
\[
A(\Sigma_{3})=A\left(  \Sigma_{1}\right)  +A(\Sigma_{2}),
\]%
\[
\overline{n}\left(  \Sigma_{3}\right)  =\overline{n}\left(  \Sigma_{1}\right)
+\overline{n}\left(  \Sigma_{2}\right)  +\#\left(  \gamma^{\circ}\cap
E_{q}\right)  ,
\]%
\begin{equation}
R(\Sigma_{3})=R(\Sigma_{1})+R(\Sigma_{2})-4\pi\#\left(  \gamma^{\circ}\cap
E_{q}\right)  . \label{RRR}%
\end{equation}

(ii). If $\partial\gamma\subset E_{q},$ then
\[
CV_{f_{3}}(S\backslash E_{q})=CV_{f_{1}}(S\backslash E_{q})\cup CV_{f_{2}%
}(S\backslash E_{q}).
\]

(iii). If $\partial\Sigma_{2}=-\gamma+\gamma$ and the interior of $\Sigma_{2}$
is the simple domain $S\backslash\gamma,$ then
\[
\partial\Sigma_{3}=\partial\Sigma_{1}=\gamma+\Gamma_{1},
\]
and if in addition\footnote{note that by assumption $\#\partial\gamma=2,$ say,
$\gamma$ contains two distinct endpoints.} $\partial\gamma\subset E_{q},$ then
the following two equalities hold:
\begin{equation}
R\left(  \Sigma_{3}\right)  =R\left(  \Sigma_{1}\right)  , \label{drr}%
\end{equation}
and%
\begin{equation}
CV_{f_{3}}(S\backslash E_{q})=CV_{f_{1}}(S\backslash E_{q}).
\label{2023-07-05-1}%
\end{equation}

\end{lemma}

\begin{proof}
All conclusions, except (\ref{drr}) and (\ref{2023-07-05-1}), follow from
Lemma \ref{patch-1}. Assume $\partial\Sigma_{2}=-\gamma+\gamma$, the interior
of $\Sigma_{2}$ is the simple domain $S\backslash\gamma,$ and $\partial
\gamma\subset E_{q}.$ Then we have $CV_{f_{2}}(S\backslash E_{q})=\emptyset,$
and thus (\ref{2023-07-05-1}) follows from (ii). On the other hand, by Lemma
\ref{-path} we have $R(\Sigma_{2})=4\pi\#\left(  \gamma^{\circ}\cap
E_{q}\right)  ,$ which with (\ref{RRR}) implies (\ref{drr}).
\end{proof}

$\Sigma_{3}$ in (iii) can also be obtained by continuously extending
$\Sigma_{1}:$ let the two endpoints of $\gamma$ be fixed and let $\gamma$
continue to move to the right hand side, and return to the initial position of
$\gamma$ after scanning the whole sphere.

\begin{lemma}
\label{ps-dm}Let $\Sigma=\left(  f,\overline{\Delta}\right)  $ be a surface
such that $\partial\Sigma=\gamma+\Gamma$ where $\gamma$ is a simple arc on $S$
with $\#\partial\gamma=2.$ Let $T$ be a closed Jordan domain on $S$ with
$\partial T=\gamma-\gamma^{\prime}.$ Then $\Sigma$ and $T^{c}=S\backslash
T^{\circ}$ can be sewn along $\gamma,$ resulting a surface $\Sigma^{\prime
}=\left(  f^{\prime},\overline{\Delta}\right)  $ such that $\partial
\Sigma^{\prime}=\Gamma+\gamma^{\prime}.$ Moreover, in the case $\partial
\gamma\subset E_{q}$ we have
\[
R\left(  \Sigma\right)  =R(\Sigma^{\prime})+R(T)-4\pi\#E_{q}\cap\gamma
^{\prime\circ},
\]
and
\begin{equation}
CV_{f^{\prime}}\left(  S\backslash E_{q}\right)  =CV_{f}\left(  S\backslash
E_{q}\right)  . \label{2023-07-05-2}%
\end{equation}

\end{lemma}

\begin{proof}
It is trivial to see that $\partial T^{c}=-\gamma+\gamma^{\prime},$ and then
by Lemma \ref{patch} (i) we have $\partial\Sigma^{\prime}=\Gamma
+\gamma^{\prime}.\ $Assume $\partial\gamma\subset E_{q}.$ Then by Lemma
\ref{patch} (ii) we have $CV_{f^{\prime}}(S\backslash E_{q})=CV_{f}%
(S\backslash E_{q})\cup CV_{id}(\overline{T_{c}}\backslash E_{q}%
)=CV_{f}(S\backslash E_{q}).$ The surface $\Sigma^{\prime}$ can also be
obtained in this way: first sew $\Sigma$ and the surface whose interior is
$S\backslash\gamma$ and boundary is $\gamma-\gamma,$ obtaining a surface
$\Sigma^{\prime\prime}=\left(  f^{\prime\prime},\overline{\Delta}\right)  ,$
and then cut from the new surface $\Sigma^{\prime\prime}$ the the domain
$T^{\circ},$ together with the open boundary $\gamma^{\circ},$ along
$\gamma^{\prime},$ obtaining the surface $\Sigma^{\prime}=\left(  f^{\prime
},\overline{\Delta}\right)  .$ Then we have by Lemma \ref{patch} (iii) that
$R(\Sigma^{\prime\prime})=R(\Sigma).$ And by Lemma \ref{cut-off} we have%
\[
R(\Sigma)=R(\Sigma^{\prime\prime})=R(\Sigma^{\prime})+R(T)-4\pi\#E_{q}%
\cap\gamma^{\prime\circ}.
\]

\end{proof}

Lemma \ref{patch-1} can be extended to the case that $\alpha_{2}$ is the whole
boundary $\partial U_{2}$ but $\alpha_{1}$ is a proper arc of $\partial
U_{1},$ as the reverse process of Example \ref{exam-glue-1}.

\begin{corollary}
\label{glue-1}Let $\Sigma_{1}=(f_{1},\overline{U})$ be a surface in
$\mathbf{F}$ and assume $\partial\Sigma_{1}$ has a partition%
\[
\partial\Sigma_{1}=\gamma+\Gamma
\]
such that $\gamma=\left(  f,\alpha\right)  $ is a Jordan curve on $S,$ where
$\alpha=\alpha\left(  x_{1},x_{2}\right)  $ is a proper arc of $\partial U$.
Let $T_{\gamma}$ be the closed Jordan domain enclosed by $\gamma$ and let
$T^{c}=S\backslash T_{\gamma}^{\circ}.$ Then $\Sigma_{1}$ and $T^{c}$ can be
sewn\label{sew32} along $\gamma$ becoming a surface $\Sigma_{2}=\left(
f_{2},\overline{\Delta}\right)  \ $such that for the disk $B=\{z\in
\mathbb{C}:|z-\frac{1}{2}|<\frac{1}{2}\}$ the following holds:

(i) There exist an OPCOFOM $h_{1}:\overline{U}\rightarrow\overline{\Delta
}\backslash B$ and an OPH $f_{1}^{\prime}:\overline{B}\rightarrow T^{c}$ such
that $h_{1}:\overline{U}\backslash\{x_{1},x_{2}\}\rightarrow\left(
\overline{\Delta}\backslash B\right)  \backslash\{1\}$ is an OPH,
\[
(h_{1},\partial U\backslash\alpha^{\circ})=\partial\Delta,\mathrm{\ }%
(h_{1},\alpha)=-\partial B,\mathrm{\ }h_{1}\left(  x_{1}\right)  =h_{1}\left(
x_{2}\right)  =1,
\]%
\[
f_{1}^{\prime}\left(  y\right)  =f_{1}\circ h_{1}^{-1}\left(  y\right)
,\mathrm{\ }y\in\partial B,
\]
and
\[
f_{2}\left(  z\right)  =\left\{
\begin{array}
[c]{l}%
f_{1}\circ h_{1}^{-1}\left(  z\right)  ,z\in\overline{\Delta}\backslash B,\\
f_{1}^{\prime},z\in B.
\end{array}
\right.  .
\]

(ii) $p\in\overline{\Delta}\backslash\partial B$ is a branch point of $f_{2}$
if and only if
\[
h_{1}^{-1}(p)\in C_{f_{1}}\left(  \overline{U}\backslash\alpha\right)  .
\]

(iii) $p\in\left(  \partial B\right)  \backslash\{1\}$ is a branch point of
$f_{2}$ if and only if $h_{1}^{-1}(p)\in C_{f_{1}}\left(  \alpha^{\circ
}\right)  .\label{2023-4-04-11:22}$

(iv) $\partial\Sigma_{2}=\Gamma,$ when $\Gamma$ is viewed as a closed curve on
$S$, and moreover%
\begin{equation}
L(\partial\Sigma_{1})=L(\partial\Sigma_{2})+L(\gamma), \label{LL}%
\end{equation}%
\begin{equation}
A(\Sigma_{1})=A(\Sigma_{2})+A(T_{\gamma})-4\pi, \label{AA}%
\end{equation}%
\begin{equation}
\overline{n}\left(  \Sigma_{1}\right)  =\overline{n}\left(  \Sigma_{2}\right)
+\overline{n}\left(  T_{\gamma}\right)  -q+\chi_{E_{q}}(f\left(  x_{1}\right)
), \label{NN}%
\end{equation}
where
\[
\chi_{E_{q}}\left(  w\right)  =\left\{
\begin{array}
[c]{c}%
0,w\notin E_{q},\\
1,w\in E_{q}.
\end{array}
\right.
\]

\begin{equation}
R(\Sigma_{1})=R(\Sigma_{2})+R(T_{\gamma})+8\pi-4\pi\chi_{E_{q}}(f_{1}%
(x_{1}))\geq R(\Sigma_{2})+R(T_{\gamma})+4\pi, \label{8pi}%
\end{equation}
equality holding if and only if
\begin{equation}
f\left(  x_{1}\right)  =f\left(  x_{2}\right)  \in E_{q}; \label{EqEq}%
\end{equation}
and when (\ref{EqEq}) holds, we have
\begin{equation}
CV_{f_{1}}\left(  S\backslash E_{q}\right)  =CV_{f_{2}}\left(  S\backslash
E_{q}\right)  . \label{CVCV1}%
\end{equation}

\end{corollary}

\begin{proof}
(i) in fact gives the method how to sew\label{sew31} $\Sigma_{1}$ and $T^{c}$
along $\gamma,$ all of (i) is easy to see. Then (ii) and (iii) follows from (i).

The relation $\partial\Sigma_{2}=\Gamma$ and (\ref{LL})$\ $are trivial to see.
(\ref{AA}) follows from the equalities $A(\Sigma_{1})=A(\Sigma_{2}%
)-A(T^{c})\ $and $A(T^{c})=A(S\backslash T_{\gamma})=4\pi-A(T_{\gamma}).$

It is clear that\footnote{Note that by definition of $\overline{n},$
$\overline{n}\left(  T^{c}\right)  =\overline{n}\left(  S\backslash T_{\gamma
}\right)  =\#\left(  S\backslash T_{\gamma}\right)  \cap E_{q},$ say, any
point of $E_{q}$ on the boundary is not counted for $\overline{n}.$}
$\overline{n}\left(  T^{c}\right)  =q-\#\gamma\cap E_{q}-\overline{n}\left(
T_{\gamma}\right)  ,$ and that%
\[
\overline{n}\left(  \Sigma_{1}\right)  =\overline{n}\left(  \Sigma_{2}\right)
-\overline{n}\left(  T^{c}\right)  -\#\gamma\cap E_{q}+\chi_{E_{q}}(f\left(
x_{1}\right)  ),
\]
where $\chi_{E_{q}}(f\left(  x_{1}\right)  )$ appears because after the
sewing, the endpoints of $\alpha$ are sewn to one boundary point of $\Delta.$
In consequence we have (\ref{NN}).

(\ref{8pi}) follows from (\ref{AA}) and (\ref{NN}). When (\ref{EqEq}) holds,
(\ref{CVCV1}) follows from (ii) and (iii) directly. Therefore all conclusions
in (iv) hold.
\end{proof}

Assume that the arc $\alpha=\alpha\left(  x_{1},x_{2}\right)  $ in Corollary
\ref{glue-1} is the whole boundary $\partial U,$ say, $x_{2}=x_{1},$ and
$\gamma=\left(  f,\partial U\right)  =\partial\Sigma$ is still simple, say,
$\gamma=\partial\Sigma$ is a Jordan curve. Then we can sew $\Sigma_{1}$ and
the closed domain $T_{c}=S\backslash T_{\gamma}^{\circ}$ so that the result
surface $\Sigma_{2}$ is a closed surface, say, $\Sigma_{2}=\left(
f_{2},S\right)  ,$ where $S$ is the sphere. Then above equations in the proof
all hold when we replace $\chi_{E_{q}}(f(x_{1}))$ by $0,$ even if $f(x_{1})\in
E_{q}.$ This can be explained in another way: we may choose the end point
$x_{2}=x_{1}$ of $\alpha$ not contained in $f_{1}^{-1}(E_{q}),$ then
$\chi_{E_{q}}(f(x_{1}))=0$ and thus the above argument works. This means that
we have

\begin{corollary}
\label{loop1}Let $\Sigma_{1}=\left(  f_{1},\overline{\Delta}\right)  $ be a
surface such that $\partial\Sigma$ is a Jordan curve on $S$ and let
$T_{\partial\Sigma}$ be the closed domain enclosed by $\partial\Sigma$ on $S.$
Then for the closed surface $\Sigma_{2}=\left(  f_{2},S\right)  $ over $S$
which is obtained by sewing $\Sigma$ and $T^{c}=S\backslash T_{\partial\Sigma
}^{\circ}$ along $\partial\Sigma,$ we have
\[
R(\Sigma)=R(\Sigma_{2})+R(T_{\partial\Sigma})+8\pi.
\]
Moreover, we have $CV_{f_{2}}=CV_{f_{1}}.$
\end{corollary}

\label{abs-one}The equivalent class argument for two surface in Remark
\ref{abs-two} can be used for the above method. We left this to the reader.

\begin{example}
\label{exam2}For a surface $\Sigma=\left(  f,\overline{\Delta}\right)
\in\mathcal{F},$ we can cut $\overline{\Delta}$ by the radius $[0,1]$ to
obtain a surface $\Sigma_{1}$ whose interior is $\left(  f,\Delta
\backslash\lbrack0,1]\right)  $ and boundary is $\partial\Sigma_{1}+\left(
f,\left[  1,0\right]  \right)  +\left(  f,[0,1]\right)  .$ $\Sigma_{1}$ can be
expressed as $\left(  f_{1},\overline{\Delta^{+}}\right)  $ which split the
arc $\left(  f,[0,1]\right)  $ into two boundary arcs of $\Sigma_{1}$, where
$f_{1}\left(  z\right)  =f\left(  z^{2}\right)  $. Conversely, we can sew
\label{sew14}$\Sigma_{1}$ along $\left(  f_{1},[0,1]\right)  $ to recover the
surface $\Sigma.$ This trivial observation is generalized in Lemma \ref{glue}.
\end{example}

\begin{lemma}
\label{glue}$\label{2023-04-4-21:00 copy(12)}$Let $\Sigma=(f,\overline{U}%
)\ $be a surface in $\mathbf{F}$ such that $\partial U$ has a partition
$\partial U=\alpha_{1}\left(  x_{1},x_{2}\right)  +\alpha_{2}\left(
x_{2},x_{3}\right)  +A\left(  x_{3},x_{1}\right)  ,$ $\partial\Sigma$ has the
corresponding partition%
\[
\partial\Sigma=\gamma-\gamma+\Gamma=\left(  f,\alpha_{1}\right)  +\left(
f,\alpha_{2}\right)  +\left(  f,A\right)  ,
\]
and $\gamma=\left(  f,\alpha_{1}\right)  =-\left(  f,\alpha_{2}\right)  $ is a
simple arc with distinct endpoints. Then $\Sigma$ can be sewn along $\gamma$
becoming a surface $\Sigma_{1}$ such that the following hold:

(i) If $\Gamma$ is a just a point, say $\partial\Sigma=\gamma-\gamma,$ then
$\Sigma_{1}$ is a closed surface $\Sigma_{1}=\left(  f,S\right)  $ such that%
\[
R(\Sigma)=R(\Sigma_{1})+4\pi\#\gamma\cap E_{q},
\]
and
\[
CV_{f}=CV_{f_{1}}.
\]
(ii) If $\Gamma$ is not a point, then $\Sigma_{1}$ is a surface $\Sigma
_{1}=\left(  f_{1},\overline{\Delta}\right)  ,$ such that%
\begin{equation}
\partial\Sigma_{1}=\Gamma\label{LL1}%
\end{equation}%
\begin{equation}
A(\Sigma_{1})=A\left(  \Sigma\right)  ,L(\partial\Sigma_{1})=L(\partial
\Sigma)-2L(\gamma), \label{AL}%
\end{equation}%
\begin{equation}
CV_{f_{1}}\left(  \overline{\Delta}\right)  =CV_{f}\left(  \overline{\Delta
}\right)  \mathrm{\ \ \ or\ \ }CV_{f_{1}}\left(  \overline{\Delta}\right)
=CV_{f}\left(  \overline{\Delta}\right)  \cup\{f\left(  x_{1}\right)  \},
\label{CVCV}%
\end{equation}%
\begin{equation}
R(\Sigma)=R(\Sigma_{1})+4\pi\#\left(  \left[  \gamma\backslash\{f\left(
x_{1}\right)  \}\right]  \cap E_{q}\right)  , \label{RR1}%
\end{equation}
and if $f(x_{1})\in E_{q},$ then the following two equalities hold:%
\begin{equation}
CV_{f_{1}}(S\backslash E_{q})=CV_{f}\left(  S\backslash E_{q}\right)  ,
\label{CVCV2}%
\end{equation}
and%
\begin{equation}
R(\Sigma)=R(\Sigma_{1})+4\pi\#\gamma\cap E_{q}-4\pi. \label{RR2}%
\end{equation}

\end{lemma}

\begin{proof}
$\label{2023-04-4-21:00 copy(13)}$Assume $\Gamma$ is a point. Then we may
understand the surface $\Sigma$ is obtained by a closed surface $\Sigma
_{1}=\left(  f_{1},S\right)  $ so that $f_{1}$ maps $[0,1]$ homeomorphically
onto $\gamma$ and that the interior of $\Sigma$ is the open surface
$f_{1}:S\backslash\lbrack0,1]\rightarrow S$ and the boundary of $\partial
\Sigma$ is $\left(  f_{1},[0,1]\right)  +(f_{1},[1,0])=\gamma-\gamma.$ From
this we have the result (i).

Assume $\Gamma$ is not a point, say, $\alpha_{1}+\alpha_{2}$ is a proper arc
of $\partial U$, then $(f,\overline{U})$ can be sewn\label{sew29} along
$\gamma$ to become a surface $\Sigma_{1}=(f_{1},\overline{\Delta}):$ There
exists an OPCOFOM $h:\overline{U}\rightarrow\overline{\Delta}$ such that
$h|_{\overline{U}\backslash\alpha_{1}+\alpha_{2}}\rightarrow\overline{\Delta
}\backslash\lbrack0,1]$ is an OPH, $\left(  h,\alpha_{1}\right)
=-[0,1]=[1,0],\left(  h,\alpha_{2}\right)  =[0,1],$
\[
f(h^{-1}\left(  y\right)  \cap\alpha_{1})=f(h^{-1}\left(  y\right)  \cap
\alpha_{2}),y\in\lbrack0,1],
\]
and
\[
f_{1}\left(  z\right)  =f\circ h^{-1}\left(  z\right)  ,z\in\overline{\Delta
}.
\]
It is clear that $\Sigma_{1}=\left(  f_{1},\overline{\Delta}\right)  $ is a
well defined surface and (\ref{LL1})--(\ref{CVCV}) hold. Moreover we have%
\[
\overline{n}\left(  \Sigma\right)  =\overline{n}\left(  \Sigma_{1}\right)
-\#E_{q}\cap\left[  \gamma\backslash\{f\left(  x_{1}\right)  \}\right]  ,
\]
which with the first equality in (\ref{AL}) implies (\ref{RR1}). Then, under
the special case $f\left(  x_{1}\right)  \in E_{q},$ (\ref{CVCV2}) and
(\ref{RR2}) are just speecial case (\ref{CVCV}) and (\ref{RR1}).
\end{proof}

\section{The spherical isoperimetric inequalities\label{sg}}

In this section we list some results follows from Bernstein's isoperimetric inequalities.

\begin{lemma}
\label{Rad1}\label{ber}$\label{2023-04-4-21:00 copy(14)}$Let $\Gamma$ be a
closed curve on $S.$

(i) (Bernstein \cite{Ber}) If $L=L(\Gamma)\leq2\pi,$ $\Gamma$ is simple and
contained in some hemisphere $S_{1}$ on $S,$ then the area $A$ of the Jordan
domain of $S_{1}$ bounded by $\Gamma$ satisfies
\[
A\leq2\pi-\sqrt{4\pi^{2}-L^{2}},
\]
with equality if and only if $\Gamma$ is a circle.

(ii) (Rad\'{o} \cite{R}) If $L(\Gamma)<2\pi$ and $\Gamma$ is simple, then
$\Gamma$ lies in some open hemisphere on $S.$

(iii) If $L(\Gamma)<2\pi$ and $\Gamma$ consists of finitely many circular arcs
on $S$, then $\Gamma$ also\label{6868-22} lies in some open hemisphere on $S$.
\end{lemma}

(i) and (ii) are known, and (iii) can be proved by (i) and (ii) as in
\cite{Zh1}.

\begin{lemma}
\label{arg}Let $f:\overline{\Delta}\rightarrow S\ $be an OPCOFOM. If
$f(\Delta)\subset S\backslash E_{q}$ and $f(\partial\Delta)$ lies in some open
disk $D$ in $S\backslash E_{q},$ then $f(\overline{\Delta})\subset D.$
\end{lemma}

\begin{proof}
By the assumption, $f(\partial\Delta)\cap\left(  S\backslash D\right)
=\emptyset.$ If $f(\Delta)\cap\left(  S\backslash D\right)  \neq\emptyset,$
then by the argument principle we have $f(\Delta)\supset S\backslash D\supset
E_{q},$ contradicting the assumption $f(\Delta)\subset S\backslash E_{q}.$
\end{proof}

The following lemma is Lemma 3.5 in \cite{Zh1}.

\begin{lemma}
\label{ber1}For each $k=1,\dots,n,$ let $F_{k}\neq\emptyset\ $be a domain in a
hemisphere $S_{k}$ on $S$ which is enclosed by a finite number of Jordan
curves and let $l_{k}=L(\partial F_{k}).$ If $l=\sum_{k=1}^{n}l_{k}<2\pi\ $and
$D_{l}$ is a disk in some hemisphere on $S$ with $L(\partial D_{l})=l,$ then%
\[
A(F_{1})+\cdots+A(F_{n})\leq A(D_{l})=2\pi-\sqrt{4\pi^{2}-l^{2}},
\]
with equality if and only if $n=1$ and $F_{1}$ is also a disk.
\end{lemma}

The following result is a consequence of Theorem 3.6 in \cite{Zh1}.

\begin{lemma}
\label{good}Let $\Sigma=(f,\overline{U})\in\mathcal{F}$. Assume $f(U)$ is
contained in some open hemisphere and $L(f,\partial U)<2\pi.$ Then
\begin{equation}
A(f,U)\leq A(T)<L(f,\partial U), \label{a4}%
\end{equation}
where $T$ is a disk in some open hemisphere on $S$ with $L(\partial
T)=L(f,\partial U).$
\end{lemma}

But the author of \cite{Zh1} did not proved that the equality in (\ref{a4})
holds only if $\Sigma$ is a simple disk with perimeter $L(f,\partial U).$ We
will improve this and give a self-contained proof after we introduce an area formula.

For a rectifiable Jordan curve $\gamma$ in $\mathbb{C}$ which is oriented
anticlockwise, the spherical area $A(\gamma)$ of the domain $D_{\gamma}$ in
$\mathbb{C}$ enclosed by $\gamma$ can be defined by%
\begin{equation}
A(\gamma)=\iint_{D_{\gamma}}\frac{4dx\wedge dy}{(1+z\overline{z})^{2}}%
=\frac{2}{i}\int_{\gamma}\frac{\overline{z}dz}{1+z\overline{z}}, \label{zza7}%
\end{equation}
since the exterior differential of $\frac{\overline{z}dz}{1+z\overline{z}}$
equals
\[
d\frac{\overline{z}dz}{1+z\overline{z}}=\frac{d\overline{z}\wedge
dz}{1+z\overline{z}}-\frac{z\overline{z}d\overline{z}\wedge dz}{\left(
1+z\overline{z}\right)  ^{2}}=\frac{d\overline{z}\wedge dz}{\left(
1+z\overline{z}\right)  ^{2}}=\frac{2idx\wedge dy}{\left(  1+z\overline
{z}\right)  ^{2}}.
\]

Note that the formula for $A(\gamma)$ with $\gamma\subset\mathbb{C}$ depends
on the orientation of $\gamma,$ which is positive when $\gamma$ is
anticlockwise, or is negative when $\gamma$ is clockwise.

If $\Sigma=(f,\overline{U})$ is a surface over $S$, then the area may not be
determined by the boundary $\partial\Sigma=\left(  f,\partial U\right)  $
only. But if $\infty\notin\partial\Sigma$, then $A(\Sigma)$ is determined by
$\partial\Sigma$ and the covering number $\deg_{f}(\infty)$ of $f,$ where
$\deg_{f}(\infty)$ is defined to be $\lim_{q_{n}\rightarrow\infty}%
\#f^{-1}(q_{n})$ in which $q_{n}$ is a sequence of regular values of $f$
converging to $\infty.$ That is to say, we have the following.

\begin{lemma}
\label{Al}$\label{2023-04-4-21:00 copy(11)}$Let $\Sigma=\left(  f,\overline
{U}\right)  $ be a surface in which $U$ is a Jordan domain enclosed by a
finite number of piecewise smooth Jordan curves and $\partial\Sigma=\left(
f,\partial U\right)  $ is also piecewise smooth. Assume that $\infty
\notin\partial\Sigma.$ Then
\[
A(\Sigma)=4\pi\deg_{f}(\infty)+A(\partial\Sigma),
\]
where
\[
A(\partial\Sigma)=\frac{2}{i}\int_{\partial\Sigma}\frac{\overline{w}%
dw}{1+|w|^{2}}=\frac{2}{i}\int_{\partial U}\frac{\overline{f(z)}%
df(z)}{1+|f(z)|^{2}}.
\]

\end{lemma}

\begin{proof}
This is essentially follows from Argument principle of meromorphic functions.
By Theorem \ref{st} (ii), we may assume $f$ is meromorphic on $\overline{U}.$

Let $p_{1},\dots,p_{k}$ be all poles, which are all distinct points of
$f^{-1}(\infty),$ of $f$ with multiplicities $m_{1},\dots,m_{k}.$ Then
$\{p_{j}\}\subset U$ and $\sum_{j=1}^{k}m_{j}=\deg_{f}(\infty).$ For each
$j=1,\dots,k,$ let $C_{j,\varepsilon}$ be small circles centered at $p_{j}$
with radius $\varepsilon.$ Then we have%
\[
\frac{2}{i}\int_{\partial U-C_{1,\varepsilon}-\cdots-C_{k,\varepsilon}}%
\frac{\overline{f(z)}f^{\prime}(z)dz}{1+|f(z)|^{2}}=%
{\displaystyle\iint\limits_{U\backslash\cup_{j=1}^{k}D\left(  p_{j}%
,\varepsilon\right)  }}
\frac{4\left\vert f^{\prime}(z)\right\vert ^{2}dx\wedge dy}{\left(
1+|f(z)|^{2}\right)  ^{2}}\rightarrow A(\Sigma)\ \mathrm{as\ }\varepsilon
\rightarrow0.
\]

For each $j=1,\dots,k,$ we have by Argument principle that%
\begin{align*}
&  \lim_{\varepsilon\rightarrow0}\frac{2}{i}\int_{-C_{j,\varepsilon}}%
\frac{\overline{f(z)}f^{\prime}(z)dz}{1+|f(z)|^{2}}=\lim_{\varepsilon
\rightarrow0}\frac{2}{i}\int_{-C_{j,\varepsilon}}\frac{\left\vert
f(z)\right\vert ^{2}d\log f(z)}{1+|f(z)|^{2}}\\
&  =\lim_{\varepsilon\rightarrow0}\frac{2}{i}\int_{-C_{j,\varepsilon}}d\log
f(z)=4\pi m_{j}.
\end{align*}
Then the conclusion follows from $\sum m_{j}=\deg_{f}(\infty).$
\end{proof}

\begin{corollary}
\label{A1}$\label{2023-04-4-21:00 copy(10)}$Assume $\Sigma_{j}=\left(
f_{j},\overline{\Delta}\right)  ,j=1,2,$ are two surfaces with $\partial
\Sigma_{1}=\partial\Sigma_{2}$ and $\partial\Sigma_{1}$ is piecewise smooth.
Then $A(\Sigma_{1})-A(\Sigma_{2})=n_{0}4\pi,$ where $n_{0}$ is an integer.
\end{corollary}

\begin{proof}
When $\infty\notin\partial\Sigma_{1},$ this follows from the previous lemma
directly. When $\infty\in\partial\Sigma_{1}$ we can consider the surfaces
$\Sigma_{j}^{\prime}=\left(  \varphi\circ f_{j},\overline{\Delta}\right)
,j=1,2,$ where $\varphi$ is a rotation of $S$ such that $\infty\notin
\partial\Sigma_{j}^{\prime},j=1,2.$ It is clear that $A(\Sigma_{j}%
)=A(\Sigma_{j}^{\prime}),$ and then we can apply the previous lemma to
$\Sigma_{j}^{\prime},j=1,2,$ to obtain the conclusion.
\end{proof}

\begin{lemma}
\label{f_t}$\label{2023-04-4-21:00 copy(9)}$Let $\Gamma:\partial
\Delta\rightarrow S\backslash\{\infty\}$ be a closed curve in $\mathbb{C}$
consisted of a finitely many simple circular arcs. If $\varphi_{t},$
$t\in\lbrack0,1],$ is a family of rotations on $S$ which is continuous with
respect to $t$ and $\varphi_{t}\left(  \Gamma\right)  \cap\{\infty
\}=\emptyset\ $for all $t\in\lbrack0,1].$ Then%
\[
A(\varphi_{t}\circ\Gamma)=A(\Gamma).
\]

\end{lemma}

\begin{proof}
For each $a\in\mathbb{C}\backslash\Gamma$ let%
\[
n\left(  \Gamma,a\right)  =\frac{1}{2\pi i}\int_{\Gamma}\frac{dz}{z-a}.
\]
Then for each component $V$ of $\mathbb{C}\backslash\Gamma,$ $n\left(
\Gamma,a\right)  $ is a constant integer $n_{V}=n\left(  \Gamma,a\right)  $
for all $a\in V,$ and then we can write%
\[
n_{V}=n\left(  \Gamma,V\right)  .
\]
We call $n_{V}=n\left(  \Gamma,V\right)  $ the index of $V$ with respect to
$\Gamma.$ It is clear that the index of the unbounded component is zero. Let
$V_{1},\dots,V_{m}$ be all distinct bounded components of $\mathbb{C}%
\backslash\Gamma.$ Then we have%
\[
A(\Gamma)=\sum_{j=1}^{m}n_{V_{j}}A(V_{j}).
\]
By the assumption it is clear that for each $t\in\lbrack0,1],$ $V_{t,j}%
=\varphi_{t}\left(  V_{j}\right)  ,j=1,\dots,m,$ are all distinct bounded
components of $\mathbb{C}\backslash\varphi_{t}(\Gamma),$ and we have%
\[
n_{V_{t,j}}=n\left(  \varphi_{t}\left(  \Gamma\right)  ,V_{j}\right)
=n_{V_{j}},A(V_{t,j})=A(V_{j}),t\in\lbrack0,1].
\]
Thus%
\[
A(\varphi_{t}\left(  \Gamma\right)  )=\sum_{j=1}^{m}n_{V_{t,j}}A(V_{t,j}%
)=\sum_{j=1}^{m}n_{V_{j}}A(V_{j})=A(\Gamma),t\in\lbrack0,1].
\]

\end{proof}

For any piecewise smooth curve $\Gamma=\left(  f,\partial\Delta\right)  $ with
$\infty\notin\Gamma$ and any rotation $\varphi$ of the sphere $S$ with
$\infty\notin\varphi\left(  \Gamma\right)  ,$ $A(\Gamma)$ need not equals
$A(\left(  \varphi\circ f,\partial\Delta\right)  ).$ For example, for the
family of congruent circles $C_{x}=\partial D\left(  x,\frac{\pi}{4}\right)  $
in $S\backslash\{\infty\}$ oriented anticlockwise, we have $A(C_{x})=A(C_{0})
$ when $d\left(  x,\infty\right)  >\frac{\pi}{4},$ but $A(C_{x})=4\pi
-A(C_{0})$ when $d\left(  x,\infty\right)  <\frac{\pi}{4}.$ In general we have

\begin{lemma}
\label{Jordan-1}Let $\Gamma$ be a piecewise smooth Jordan curve in
$\mathbb{C}$ oriented anticlockwise. Then we have

(i) $0<A(\Gamma)=A(D)<4\pi,$ where $D$ is the Jordan domain in $\mathbb{C}$
bounded by $\Gamma.$

(ii) For any rotation $\varphi$ of $S$ with $\infty\notin\varphi\left(
\Gamma\right)  $,
\[
A(\varphi(\Gamma))=\left\{
\begin{array}
[c]{l}%
A(\Gamma),\mathrm{\ \ \ \ \ \ if\ }\infty\notin\varphi\left(  D\right) \\
A(\Gamma)-4\pi,\mathrm{\ if\ }\infty\in\varphi\left(  D\right)
\end{array}
\right.
\]

\end{lemma}

\begin{proof}
$\label{2023-04-4-21:00 copy(8)}\Gamma$ divides $\overline{\mathbb{C}}$ into
two components $D$ and $D_{\infty},$ where $D_{\infty}$ contains $\infty.$ It
is clear that $A\left(  \varphi(D)\right)  =A(D)$ and $A(\varphi\left(
D_{\infty}\right)  )=A(D_{\infty}).$ If $\varphi\left(  D\right)  $ does not
contains $\infty,$ then
\[
n_{\varphi(D)}=n\left(  \varphi\left(  \Gamma\right)  ,\varphi\left(
D\right)  \right)  =n\left(  \Gamma,D\right)  =n_{D}=1
\]
and
\[
A(\varphi(\Gamma))=A(\varphi(D))=A(D)=A(\Gamma).
\]
If $\infty\in\varphi\left(  D\right)  ,$ then $\infty\notin\varphi\left(
D_{\infty}\right)  $ and $n\left(  \varphi\left(  \Gamma\right)
,\varphi\left(  D_{\infty}\right)  \right)  =-1.$ Thus we have
\[
A(\varphi\left(  \Gamma\right)  )=-A\left(  \varphi\left(  D_{\infty}\right)
\right)  =-A(D_{\infty})=A(D)-4\pi=A(\Gamma)-4\pi.
\]

\end{proof}

Now we prove the following lemma.

\begin{lemma}
\label{A of curve}For $j=0,1,$ let $\Gamma_{j}=\left(  f_{j},\partial
\Delta\right)  $ be a closed curve on $S$ consisted of a finite number of
simple circular arcs, such that $\Gamma_{j}$ is contained in some open
hemisphere $S_{j}$ on $S$ with $S_{j}\subset\mathbb{C}\ $and there exists a
rotation $\varphi$ of $S$ such that $\varphi\circ f_{0}=f_{1}.$ Then
$A(\Gamma_{1})=A(\Gamma_{0}).$
\end{lemma}

It is clear that the circle $\Gamma_{1}=\partial D\left(  0,\frac{\pi}%
{4}\right)  $ and $\Gamma_{2}=-\partial D\left(  \infty,\frac{\pi}{4}\right)
$ are both in $\mathbb{C}$ and oriented anticlockwise, but they do not satisfy
the hypothesis of the lemma, since $\Gamma_{2}$ can not be contained in any
open hemisphere which does not contain $\infty.$

\begin{proof}
By the assumption, there exists a family $\varphi_{t},t\in\lbrack0,1],$ of
rotations of $S$ so that $\varphi_{t}\left(  z\right)  $ is continuous for
$\left(  z,t\right)  \in S\times\lbrack0,1],$ $\varphi_{0}=id$ and
$\varphi_{1}=\varphi,$ and the family $\Gamma_{t}\left(  z\right)
=\varphi_{t}\circ\Gamma_{1}(z),z\in\partial\Delta,$ never meet $\infty.$ This
implies that $A\left(  \Gamma_{t}\right)  $ is locally invariant for all
$t\in\lbrack0,1].$ Thus we have $A(\Gamma_{0})=A(\Gamma_{1}).$
\end{proof}

Now we can enhance Lemma \ref{good} as follows.

\begin{lemma}
\label{good-1}$\label{2023-04-4-21:00 copy(7)}$Let $\Gamma=\left(
f,\partial\Delta\right)  \ $be a closed curve consisted of finitely many
simple circular arcs such that $\partial\Sigma$ is contained in some open
hemisphere $S_{1}$ in $\mathbb{C}$ and $L(\Gamma)<2\pi.$ Then the following hold.

(i) $A(\Gamma)\leq A(T),$ where $T$ is a disk in some hemisphere on $S$ with
perimeter $L,$ with equality holding if and only if $\Gamma$ is a circle
oriented anticlockwise (say, $\Gamma$ is a convex circle).

(ii) If, in addition, $\Gamma$ is the boundary of a surface $\Sigma=\left(
f,\overline{\Delta}\right)  \in\mathcal{F}$ with $f(\overline{\Delta})\subset
S_{1}\subset\mathbb{C},$ then $A(\Sigma)\leq A(T)<L(f,\partial U),$ with
equality holding if and only if $\Sigma$ is a simple disk.
\end{lemma}

Since $\Gamma\subset S_{1}\subset\mathbb{C},$ $\Gamma$ is an anticlockwise
circle if and only it is a convex circle on $S$.

\begin{proof}
$\label{2023-04-4-21:00 copy(6)}$If $\Gamma$ is of the form $\Gamma
=I_{1}-I_{1}.$ Then $A(\Gamma)=0$ and (i) holds. If $\Gamma$ is a Jordan curve
and $D$ is the domain bounded by $\Gamma$ in $S_{1}\subset\mathbb{C},$ then
$D\subset S_{1}$ and $|A(\Gamma)|=A(D)$ and (i) follows from Lemma \ref{ber} (i).

In general, $\partial\Delta$ has a partition $\left\{  \left\{  \alpha
_{ij}\right\}  _{j=1}^{k_{i}}\right\}  _{i=1}^{n}$ of a finite number of
subarcs with $\alpha_{ij}^{\circ}\cap\alpha_{i_{1}j_{1}}^{\circ}=\emptyset$
when $i\neq i_{1}$ or $j\neq j_{1}$ such that $\left(  f,\alpha_{ij}\right)  $
are all simple circular arcs and for each $i,\alpha_{i1}+\alpha_{i2}%
+\dots+\alpha_{ik_{i}}$ may not be a subarc of $\partial\Delta,$ but%
\[
\gamma_{i}=(f,\alpha_{i1})+\left(  f,\alpha_{i2}\right)  +\dots+\left(
f,\alpha_{ik_{i}}\right)
\]
is either a Jordan curve on $S$ or is of the form $\gamma_{i}=I_{i}-I_{i},$
where $I_{i}$ is a simple arc on $S.$ Then we have%
\[
A(\Gamma)=\sum_{i=1}^{n}A(\gamma_{i}),
\]
and by Lemma \ref{ber} (i)
\[
\left\vert A(\gamma_{i})\right\vert =\left\{
\begin{array}
[c]{l}%
0,\mathrm{\ \ \ \ \ if\ }\gamma_{i}\mathrm{\ is\ of\ the\ form\ }I_{i}-I_{i}\\
A(D_{i}),\mathrm{\ if\ }\gamma_{i}%
\mathrm{\ is\ a\ Jordan\ curve\ bounding\ a\ domain\ }D_{i}\mathrm{\ in\ }%
S_{1}%
\end{array}
\right.  \leq A(T_{i}),
\]
where $T_{i}$ is a disk in $\mathbb{C}$ with perimeter $L(\gamma_{i})$, with
equality holding if and only if $D_{j}$ is a disk. Then we have, by Lemma
\ref{ber1},%
\[
A(\Gamma)\leq\sum_{j=1}^{k}\left\vert A(\gamma_{j})\right\vert \leq\sum
_{j=1}^{k}A(T_{j})\leq A(T),
\]
where $T$ is a disk in $S_{1}$ with $L(\partial T)=L(\Gamma),$ and
$A(\Gamma)=A(T)$ if and only if $k=1$ and $\Gamma$ is a circle oriented
anticlockwise. (i) is proved.

Assume $\Gamma=\left(  f,\partial\Delta\right)  $ is the boundary of a surface
$\Sigma=\left(  f,\overline{\Delta}\right)  \ $with $f(\overline{\Delta
})\subset S_{1}\subset\mathbb{C}$. Then by Lemma \ref{Al} and Lemma
\ref{good-1} (i), we have%
\[
A(\Sigma)=A(\Gamma)\leq A(T),
\]
equality holding if and only if $\Sigma$ is a simple disk. By Lemma
\ref{good}, we also have $A(T)<L(f,\partial\Delta)$ and (ii) is proved.
\end{proof}

\begin{lemma}
\label{for-circular}$\label{2023-04-4-21:00 copy(5)}$Let $\partial
\Delta=\alpha\left(  a_{1},a_{2}\right)  +\beta\left(  a_{2},a_{1}\right)  $
be a partition of $\partial\Delta,$ and for each $j=1,2,$ let $\Gamma
_{j}=\left(  f_{j},\partial\Delta\right)  $ be a closed curve consisted of a
finite number of circular arcs such that $\Gamma_{j}$ is contained in an open
hemisphere $S_{j}\subset\mathbb{C},$ and assume
\[
\left(  f_{1},\alpha\right)  =\left(  f_{2},\alpha\right)  ,
\]%
\[
L\left(  f_{1},\beta\right)  =L\left(  f_{2},\beta\right)  ,
\]%
\[
L(f_{1},\beta)+\overline{q_{1}q_{2}}<2\pi,
\]
and $\left(  f_{2},\beta\right)  $ is an SCC arc, where $q_{j}=f(a_{j})$ for
$j=1,2.$ Then
\[
A\left(  \Gamma_{1}\right)  \leq A(\Gamma_{2}),
\]
equality holding if and only if $(f_{1},\beta)=\left(  f_{2},\beta\right)
\ $when $q_{1}\neq q_{2}$ or $(f_{1},\beta)\ $is a convex circle when
$q_{1}=q_{2}.$
\end{lemma}

\begin{remark}
\label{for-circular-1}It is permitted that $q_{1}=q_{2},$ and in this case,
$\left(  f_{j},\alpha\right)  $ and $\left(  f_{j},\beta\right)  $ are both
closed curves for $j=1,2$, and the conclusion follows from Lemma \ref{good-1}
directly. Note that when $q_{1}=q_{2},$ $L\left(  f,\beta_{1}\right)
=L\left(  f,\beta_{2}\right)  <2\pi.$
\end{remark}

\begin{proof}
$\label{2023-04-4-21:00 copy(4)}$Any closed hemisphere on $S$ can not contain
$\infty$ if it contains $0$ in its interior. Thus there exists a rotation
$\varphi$ of $S,$ such that $\varphi\left(  q_{1}\right)  =0\ $and
$\infty\notin\varphi(\overline{S_{j}}),$ and thus $\varphi\left(
\overline{S_{j}}\right)  \subset\mathbb{C}.$ Then $\left(  \varphi\circ
f_{j},\partial\Delta\right)  $ is contained in the open hemisphere
$\varphi\left(  S_{j}\right)  \subset\varphi\left(  \overline{S_{j}}\right)
\subset\mathbb{C},$ and thus by Lemma \ref{A of curve}, for $j=1,2,$ we have
\[
A\left(  \left(  \varphi\circ f_{j},\partial\Delta\right)  \right)  =A\left(
(f_{j},\partial\Delta)\right)  =A(\Gamma_{j}).
\]

So we may assume $q_{1}=0.\ $If $c=\left(  f_{2},\beta\right)  $ is the line
segment $\overline{q_{2}q_{1}}=\overline{q_{2},0}$, then $\left(  f_{1}%
,\beta\right)  $ is also the line segment $\overline{q_{2},0},$ and then
$\partial\Gamma_{1}\sim\partial\Gamma_{2}$ and the conclusion of the lemma
holds with $A\left(  \Gamma_{1}\right)  =A(\Gamma_{2}).$ Thus we may assume
that $c$ is strictly convex.

Let $C$ be the circle on $S$ determined by $c,$ and let $D$ be the disk
enclosed by $C$. Then $C$ is strictly convex and $\overline{D}$ is contained
in an open hemisphere $S_{C}$ on $S.$ Since $q_{1}=f_{1}(a_{1})=f_{2}%
(a_{1})=0,$ we may take $S_{C}$ so that $\overline{S_{C}}\subset\mathbb{C}.$
Define
\[
\Gamma_{1}^{\prime}=\left(  f_{1}^{\prime},\partial\Delta\right)  =\left(
f_{1}^{\prime},\alpha\right)  +\left(  f_{1}^{\prime},\beta\right)
=C\backslash c^{\circ}+\left(  f_{1},\beta\right)  .
\]
Then $L(\Gamma_{1}^{\prime})=L(C\backslash c)+L\left(  f_{1},\beta\right)
=L(C)<2\pi,$ which together with Lemma \ref{ber} (ii), implies that
$\Gamma_{1}^{\prime}$ is contained in some open hemisphere $S_{1}^{\prime}$ on
$S.$ Since $0\in\Gamma_{1}^{\prime},$ we have $\overline{S_{1}^{\prime}%
}\subset S\backslash\{\infty\}=\mathbb{C}.$ Therefore Lemma \ref{good-1} (i)
implies
\[
A(\Gamma_{1}^{\prime})\leq A\left(  D\right)  =A(C),
\]
equality holding if and only if $\Gamma_{1}^{\prime}$ is an anticlockwise
circle in $S_{1}^{\prime},$ say, $\left(  f_{1},\beta\right)  =c$ if $c\neq C$
or $\Gamma_{1}^{\prime}=\left(  f_{1},\beta\right)  $ is an anticlockwise
circle in $S_{1}^{\prime}$ if $c=C.$ Then we have%
\begin{align*}
A(\Gamma_{1})  &  =A(\left(  f_{1},\alpha\right)  +\left(  f_{1},\beta)\right)
\\
&  =A(\left(  f_{1},\alpha\right)  -C\backslash c^{\circ})+A(\Gamma
_{1}^{\prime})\\
&  \leq A(\left(  f_{1},\alpha\right)  -C\backslash c^{\circ})+A(C)\\
&  =A(\left(  f_{1},\alpha\right)  +c)=A\left(  \Gamma_{2}\right)  .
\end{align*}
say $A(\Gamma_{1})\leq A(\Gamma_{2}),$ equality holding if and only if
$(f_{1},\beta)$ is the arc $c\ $or $(f_{1},\beta)$ is a convex circle.
\end{proof}

\begin{lemma}
\label{good2}$\label{2023-04-4-21:00 copy(3)}$Let $\Sigma=(f,\overline{\Delta
})\in\mathcal{F}.$ Assume that the restriction $I=-\left(  f,\left[
-1,1\right]  \right)  $ is a simple line segment $I\ $on $S$, $\Sigma$ is
contained in some open hemisphere on $S$ and%
\begin{equation}
L(f,\partial\Delta)\leq2\pi\sin\frac{L(I)}{2}. \label{4.1}%
\end{equation}
Then the following hold.

(i) There exists $\theta_{1}\in(0,\pi)$ such that%
\begin{equation}
L(f,\partial\Delta)=L(\partial\mathfrak{D}(I,\theta_{1},\theta_{1}%
))\mathrm{\ and\ }A(\Sigma)\leq A(\mathfrak{D}(I,\theta_{1},\theta_{1})).
\label{4.2}%
\end{equation}

(ii) $A(\Sigma)=A(\mathfrak{D}(I,\theta_{1},\theta_{1}))$ if and only if
$\Sigma$ is a simple closed domain congruent to $\overline{\mathfrak{D}%
(I,\theta_{1},\theta_{1})}.$
\end{lemma}

Since $\Sigma$ is contained in some open hemisphere on $S,$ we have
$L(I)<\pi,$ and the condition (\ref{4.1}) implies that $L(f,\partial\Delta)$
is not larger than the perimeter of the circle $\partial D\left(
0,\frac{L(I)}{2}\right)  $ with diameter $L(I),$ and so $L(f,\partial
\Delta)<2\pi.$ See Definition \ref{lune-lens} for the notation $\mathfrak{D}%
(I,\theta,\theta).$ Geometrically, it is clear that $L(\theta)=L(\partial
\left(  \mathfrak{D}(I,\theta,\theta\right)  )$ is strictly increasing on
$[0,\pi/2]$ as a continuous function of the angle $\theta,$ and thus
$\theta_{1}$ is uniquely determined by $L(\theta_{1})=L(f,\partial\Delta).$

(i) is essentially implied in the proof of Theorem 3.8 in \cite{Zh1}, based on
Theorem 3.6 in \cite{Zh1}, though in Theorem 3.8 of \cite{Zh1}, $I$ is
replaced by the special segment $\overline{1,0}$ and the condition (\ref{4.1})
is replaced by
\[
L(f,\partial\Delta)\leq2\pi\sin\frac{d(0,1)}{2}=2\pi\sin\frac{\pi}{4}=\sqrt
{2}\pi.
\]
But if we use the same method of \cite{Zh1} based on Lemma \ref{good-1}, we
can prove (ii) by the way.$\label{2023-04-4-21:00 copy(2)}$

\begin{proof}
$\label{2023-04-4-21:00 copy(1)}$ Recall that for a circular arc $c$,
$k\left(  c\right)  \geq0$ denotes the curvature.

Since $f(\overline{\Delta})$ is contained in a hemisphere on $S$,
$L(f,\partial\Delta)=2L(I)$ implies $f(\overline{\Delta})=I,$ which
contradicts $\Sigma\in\mathcal{F}.$ Thus $2L(I)<L(f,\partial\Delta)<2\pi.$

We now use Lemma \ref{for-circular} to give a new proof. Without loss of
generality, assume $q_{1}=f(1)=0.$ Let $q_{2}=f(-1),$ $\gamma_{1}=\left(
f,\left(  \partial\Delta\right)  ^{+}\right)  ,\gamma_{2}=\left(  f,\left(
\partial\Delta\right)  ^{-}\right)  ,$ $c_{1}$ and $c_{1}^{\prime}$ be the SCC
arcs from $q_{1}$ to $q_{2}$ with $L(c_{1})=L(f,\left(  \partial\Delta\right)
^{+})$ and $L(c_{1}^{\prime})=\frac{1}{2}L(f,\partial\Delta),$ and let $c_{2}$
and $c_{2}^{\prime}$ be the SCC arcs from $q_{2}$ to $q_{1}$ with
$L(c_{2})=L\left(  f,\left(  \partial\Delta\right)  ^{-}\right)  $ and
$L(c_{2}^{\prime})=\frac{1}{2}L\left(  f,\partial\Delta\right)  .$ Then we
have
\[
L(c_{1}+c_{2})=L(c_{1}^{\prime}+c_{2}^{\prime})=L(\gamma_{1}+c_{2})=L\left(
c_{1}+\gamma_{2}\right)  =L\left(  f,\partial\Delta\right)  <2\pi,
\]
and thus $\Sigma$ and the close curves $c_{1}+c_{2},c_{1}^{\prime}%
+c_{2}^{\prime},\gamma_{1}+c_{2},c_{1}+\gamma_{2}$ are contained in five open
hemispheres on $S$ in $\mathbb{C=}S\backslash\mathbb{\{}\infty\}$
respectively. By Lemmas \ref{Al} and \ref{for-circular} we have%
\[
A(\Sigma)=A(\gamma_{1}+\gamma_{2})\leq A(c_{1}+\gamma_{2})\leq A(c_{1}%
+c_{2}),
\]
with the second equality holding iff $\gamma_{1}=c_{1}$ and the last equality
holding iff $\gamma_{2}=c_{2}.$

We will show that
\[
A(c_{1}+c_{2})\leq A(c_{1}^{\prime}+c_{2}^{\prime}),
\]
with equality if and only if $c_{1}=c_{1}^{\prime}$ and $c_{2}=c_{2}^{\prime
}.$ It is clear that $c_{1}+c_{2}$ encloses the lens $D=\mathfrak{D}\left(
I,k_{1},k_{2}\right)  ,$ where $k_{j}$ is the curvature $k\left(
c_{j}\right)  $ of $c_{j},j=1,2;$ and $c_{1}^{\prime}+c_{2}^{\prime}$ encloses
the lens $D^{\prime}=\mathfrak{D}\left(  I,k_{1}^{\prime},k_{1}^{\prime
}\right)  ,$ where $k_{1}^{\prime}=k\left(  c_{1}^{\prime}\right)  =k\left(
c_{2}^{\prime}\right)  .$ $\overline{D}$ and $\overline{D^{\prime}}$ are
contained in two open hemispheres $S_{1}\subset\mathbb{C}$ and $S_{1}^{\prime
}\subset\mathbb{C}\ $of the five. Thus we have $A(D)=A(c_{1}+c_{2})$ and
$A(D^{\prime})=A(c_{1}^{\prime}+c_{2}^{\prime}).$

We show that $A(D)\leq A(D^{\prime}),$ equality holding only if $c_{1}%
=c_{1}^{\prime}$ and $c_{2}=c_{2}^{\prime}.$ Let $C_{1}^{\prime}$ be the
convex circle determined by $c_{1}^{\prime}.$ Then $C_{1}^{\prime}$ is
contained in some open hemisphere $S_{1}^{\prime\prime}\subset\mathbb{C}$ on
$S,$ and $c_{1}^{\prime}$ is the arc of $C_{1}^{\prime}$ from $q_{1}$ to
$q_{2}.$ By (\ref{4.1}) $c_{1}^{\prime}$ is at most half of $C_{1}^{\prime},$
and then $C_{1}^{\prime}\backslash c_{1}^{\prime\circ}$ contains the arc
$\mathfrak{c}_{2}^{\prime}=\mathfrak{c}_{2}^{\prime}\left(  q_{2}%
,q_{3}\right)  $ with $d(q_{2},q_{3})=d\left(  q_{1},q_{2}\right)  $, and
write $\mathfrak{c}_{3}=C_{1}^{\prime}\backslash\left(  c_{1}^{\prime
}+\mathfrak{c}_{2}^{\prime}\right)  ^{\circ}.$ When $c_{1}^{\prime}$ is half
of $C_{1}^{\prime},$ $\mathfrak{c}_{3}=\{q_{3}\}=\{q_{1}\}.$ Then
$\mathfrak{c}_{2}^{\prime}$ is congruent to $c_{2}^{\prime}$ and we have a
partition%
\[
C_{1}^{\prime}=c_{1}^{\prime}+\mathfrak{c}_{2}^{\prime}+\mathfrak{c}_{3}.
\]
When we replace $c_{1}^{\prime}$ by $c_{1},$ $\mathfrak{c}_{2}^{\prime}$ by
the convex arc $\mathfrak{c}_{2}=\mathfrak{c}_{2}\left(  q_{2},q_{3}\right)  $
congruent to $c_{2}=c_{2}\left(  q_{2},q_{1}\right)  ,$ the convex circle
$C_{1}^{\prime}$ becomes the Jordan curve%
\[
C_{1}=c_{1}+\mathfrak{c}_{2}+\mathfrak{c}_{3},
\]
with
\[
L(C_{1})=L(C_{1}^{\prime})=L(f,\partial\Delta)+L(\mathfrak{c}_{3})<2\pi.
\]
Then,$\ $by Lemma \ref{ber} (i), for the Jordan domain $D_{1}$ enclosed by
$C_{1}$ and the disk $D_{1}^{\prime}$ enclosed by $C_{1}^{\prime},$ we have
$A(D_{1})\leq A(D_{1}^{\prime})$, equality holding if and only if $C_{1}$ is a
circle, say $C_{1}=C_{1}^{\prime},$ $c_{1}=c_{1}^{\prime},$ $\mathfrak{c}%
_{2}=\mathfrak{c}_{2}^{\prime},$ which implies $c_{2}=c_{2}^{\prime}.$

We have the disjoint unions%
\[
D_{1}=D_{c_{1}}\cup\overline{q_{1}q_{2}}^{\circ}\cup D_{\mathfrak{c}_{2}}%
\cup\overline{q_{2}q_{3}}^{\circ}\cup D^{\prime\prime},
\]%
\[
D_{1}^{\prime}=D_{c_{1}^{\prime}}\cup\overline{q_{1}q_{2}}^{\circ}\cup
D_{\mathfrak{c}_{2}^{\prime}}\cup\overline{q_{2}q_{3}}^{\circ}\cup
D^{\prime\prime},
\]
where $D_{c_{1}}$ and $D_{\mathfrak{c}_{2}}$ are the disjoint lunes in $D_{1}$
of circular arcs $c_{1}\left(  q_{1},q_{2}\right)  $ and $\mathfrak{c}%
_{2}\left(  q_{2},q_{3}\right)  ,$ and $D_{c_{1}^{\prime}}$ and
$D_{\mathfrak{c}_{2}^{\prime}}$ are the disjoint lunes in $D_{1}^{\prime}$ of
circular arcs $c_{1}^{\prime}\left(  q_{1},q_{2}\right)  $ and $\mathfrak{c}%
_{2}^{\prime}\left(  q_{2},q_{3}\right)  .$ Then we have
\[
A\left(  D_{c_{1}}\cup D_{\mathfrak{c}_{2}}\right)  =A(D)\leq A(D_{c_{1}%
^{\prime}}\cup D_{\mathfrak{c}_{2}^{\prime}})=A(D^{\prime}),
\]
equality holding only if $c_{1}^{\prime}=c_{1}$ and $c_{2}=c_{2}^{\prime
}.\label{2023-04-4-21:00}$
\end{proof}

\begin{corollary}
\label{2-curvature}$\label{2023-04-5-6:00}$Let $l$ be a positive number and
for $j=1,2,$ let $I_{j}=\overline{q_{j1}q_{j2}}$ be a line segment with
$L(I_{j})<\pi\ $and let $\mathfrak{D}_{j}^{\prime}(l_{j})\ $be the lune such
that $\partial\mathfrak{D}_{j}^{\prime}\left(  l_{j}\right)  =c_{j}%
(l_{j})-I_{j},$ where $c_{j}$ is a convex circular arc from $q_{j1}$ to
$q_{j2}$ with $L(c_{j}(l_{j}))=l_{j}.$ Assume $l_{1}+l_{2}=l,I_{1}$ and
$I_{2}$ are fixed, but $l_{1}$ and $l_{2}$ vary, and for $j=1,2,\overline
{\mathfrak{D}_{j}^{\prime}(l_{j})}$ is contained in some open hemisphere
$S_{j}$ on $S.$ Then we have the following.

(A) If the curvature $k\left(  c_{1}(l_{1})\right)  $ of $c_{1}(l_{1})$ is
larger than $k\left(  c_{2}(l_{2})\right)  $ when $l_{1}=l_{1}^{0},l_{2}%
=l_{2}^{0}=l-l_{2}^{0}.$ Then there exists a $\delta>0,$ such that
\[
A\left(  \mathfrak{D}_{1}^{\prime}(l_{1})\right)  +A\left(  \mathfrak{D}%
_{2}^{\prime}(l_{2})\right)  =A\left(  \mathfrak{D}_{1}^{\prime}%
(l_{1})\right)  +A\left(  \mathfrak{D}_{2}^{\prime}(l-l_{1})\right)
\]
is a strictly decreasing function of $l_{1}\in(l_{1}^{0}-\delta,l_{1}%
^{0}+\delta)$.

(B) If $k\left(  c_{1}(l_{1})\right)  =k\left(  c_{2}(l_{2})\right)  $ when
$l_{1}=l_{1}^{0}$ and both $c_{1}(l_{1})$ and $c_{2}(l_{2})$ are major
circular arcs, say, the interior angle of $\mathfrak{D}_{j}^{\prime}(l_{j})$
at the cusps $>\pi/2$, then there exists a $\delta>0$ such that
\begin{equation}
A\left(  \mathfrak{D}_{1}^{\prime}(l_{1})\right)  +A\left(  \mathfrak{D}%
_{2}^{\prime}(l-l_{1})\right)  >A\left(  \mathfrak{D}_{1}^{\prime}(l_{1}%
^{0})\right)  +A\left(  \mathfrak{D}_{2}^{\prime}(l-l_{1}^{0})\right)  ,
\label{ss5}%
\end{equation}
when $0<\left\vert l_{1}-l_{1}^{0}\right\vert <\delta.$

(C)\label{20220908-1} (A) and (B) still hold when $q_{11}=q_{12},$ or
$q_{21}=q_{22},$ or both, holds (when $q_{j1}=q_{j2},$ $\mathfrak{D}%
_{j}^{\prime}(l_{j})$ becomes into a disk, say, $c_{j}(l_{j}$) is a circle).
\end{corollary}

\begin{proof}
Assume
\begin{equation}
k\left(  c_{1}(l_{1}^{0})\right)  >k\left(  c_{2}(l_{2}^{0})\right)  =k\left(
c_{2}(l-l_{1}^{0})\right)  . \label{k>k}%
\end{equation}
For sufficiently small $\varepsilon>0,$ we may take small arcs
$c_{j,\varepsilon}=c_{j}\left(  q_{j1,\varepsilon},q_{j2,\varepsilon}\right)
$ in $c_{j}\left(  l_{j}^{0}\right)  $ so that the center points of
$c_{j,\varepsilon}$ and $c_{j}\left(  l_{j}^{0}\right)  $ are the same and the
chard $\overline{q_{j1,\varepsilon}q_{j2,\varepsilon}}$ has the same length
$\varepsilon$ for $j=1,2,$ and let $\mathfrak{c}_{j,\varepsilon}%
=\mathfrak{c}_{j,\varepsilon}\left(  q_{j1,\varepsilon},q_{j2,\varepsilon
}\right)  $ be the convex circular arc from $q_{j1,\varepsilon}\ $to
$q_{j2,\varepsilon}$ so that
\[
L(\mathfrak{c}_{j,\varepsilon})=\frac{1}{2}\left(  L(c_{1,\varepsilon}\right)
+L(c_{2,\varepsilon})),j=1,2.
\]
When $\varepsilon$ is small enough by (\ref{k>k}) we have $\mathfrak{c}%
_{1,\varepsilon}^{\circ}\subset\mathfrak{D}_{1}^{\prime}\left(  l_{1}%
^{0}\right)  $ and $\mathfrak{c}_{2,\varepsilon}^{\circ}$ is on the right side
of $c_{2,\varepsilon}.$ Let $D_{j}$ be the lune enclosed by $c_{j,\varepsilon
}-\overline{q_{j1,\varepsilon}q_{j2,\varepsilon}}$, and let $D_{j}^{\prime}$
be the lune enclosed by $\mathfrak{c}_{j,\varepsilon}-\overline
{q_{j1,\varepsilon}q_{j2,\varepsilon}}$, $j=1,2.$ Then by Lemma \ref{good2} we
have
\begin{equation}
A(D_{1})+A(D_{2})<A(D_{1}^{\prime})+A(D_{2}^{\prime}), \label{ss1}%
\end{equation}
and when $\varepsilon$ is small enough, the domain $T_{j}=\left(
\mathfrak{D}_{j}^{\prime}\left(  l_{j}^{0}\right)  \backslash D_{j}\right)
\cup D_{j}^{\prime}\ $is a Jordan domain for $j=1,2,$ and we have by
(\ref{ss1})
\begin{equation}
A\left(  \mathfrak{D}_{1}^{\prime}\left(  c_{1}\left(  l_{1}^{0}\right)
\right)  \right)  +A\left(  \mathfrak{D}_{2}^{\prime}\left(  c_{2}\left(
l_{2}^{0}\right)  \right)  \right)  <A\left(  T_{1}\right)  +A(T_{2}).
\label{ss2}%
\end{equation}
It is clear that
\[
\partial T_{j}=\left(  \left[  c_{j}\left(  l_{j}^{0}\right)  \backslash
c_{j,\varepsilon}\right]  \cup\mathfrak{c}_{j,\varepsilon}\right)  -I_{j}.
\]
Now we replace $\left[  c_{j}\left(  l_{j}^{0}\right)  \backslash
c_{j,\varepsilon}\right]  \cup\mathfrak{c}_{j,\varepsilon}$ with the whole SCC
arc $c_{j}\left(  l_{j}\left(  \varepsilon\right)  \right)  $ from $q_{j1}$ to
$q_{j2}$ with
\[
l_{j}\left(  \varepsilon\right)  =L(c_{j}\left(  l_{j}\left(  \varepsilon
\right)  \right)  )=L\left(  \left[  c_{j}\left(  l_{j}^{0}\right)  \backslash
c_{j,\varepsilon}\right]  \cup\mathfrak{c}_{j,\varepsilon}\right)  .
\]
Then $l_{1}\left(  \varepsilon\right)  <l_{1}^{0}\ $and $l_{2}\left(
\varepsilon\right)  =l-l_{1}\left(  \varepsilon\right)  >l_{2}^{0}.$ By Lemma
\ref{for-circular}, we have%
\begin{equation}
A(\mathfrak{D}_{j}^{\prime}\left(  c_{j}\left(  l_{j}\left(  \varepsilon
\right)  \right)  \right)  >A(T_{j}),j=1,2, \label{ss3}%
\end{equation}
which, with (\ref{ss2}) implies%
\begin{equation}
\sum_{j=1}^{2}A(\mathfrak{D}_{j}^{\prime}\left(  c_{j}\left(  l_{j}\left(
\varepsilon\right)  \right)  \right)  >\sum_{j=1}^{2}A(\mathfrak{D}%
_{j}^{\prime}\left(  l_{j}^{0}\right)  ). \label{ss4}%
\end{equation}
Since $l_{1}\left(  \varepsilon\right)  <l_{1}^{0}$ and $l_{1}\left(
\varepsilon\right)  $ depends on $\varepsilon$ continuously when
$\varepsilon>0$ is small enough, (\ref{ss4}) implies that
\[
\sum_{j=1}^{2}A(\mathfrak{D}_{j}^{\prime}\left(  c_{j}\left(  l_{j}\right)
\right)  >\sum_{j=1}^{2}A(\mathfrak{D}_{j}^{\prime}\left(  l_{j}^{0}\right)
),
\]
when $l_{1}<l_{1}^{0}$ and $l_{1}^{0}-l_{1}$ is small enough. Then we have
proved that there exists a small enough $\delta>0$ such that (A) holds for
$l_{1}\in(l_{1}^{0}-\delta,l_{1}^{0})$.

Now, assume $k(c_{1}\left(  l_{1}^{0}\right)  )=k\left(  c_{2}\left(
l_{2}^{0}\right)  \right)  $ and both $c_{1}\left(  l_{1}^{0}\right)  $ and
$c_{2}\left(  l_{2}^{0}\right)  $ are major arcs. Then on $(l_{1}^{0}%
-\delta,l_{1}^{0}]$, $k(c_{1}\left(  l_{1}\right)  )$ strictly decreases$\ $%
but $k\left(  c_{2}\left(  l-l_{1}\right)  \right)  $ strictly increases when
$\delta$ is small enough. Thus by (A) for sufficient small $\delta
>0,\sum_{j=1}^{2}A(\mathfrak{D}_{j}^{\prime}\left(  c_{j}\left(  l_{j}\right)
\right)  ,$ as a function of $l_{1},$ strictly decreases on $(l_{1}^{0}%
-\delta,l_{1}^{0})$ and strictly increase on $(l_{1}^{0},l_{1}^{0}+\delta).$
On the other hand, it is clear that $\sum_{j=1}^{2}A(\mathfrak{D}_{j}^{\prime
}\left(  c_{j}\left(  l_{j}\right)  \right)  $ is a continuous function of
$l_{1}$ when $\left\vert l_{1}-l_{1}^{0}\right\vert $ small enough. Hence (B) holds.

The proof of (C) is the same as (A) and (B).
\end{proof}

\begin{corollary}
\label{2-curvature2}Let $I=\overline{pq}$ be a line segment with $L(I)<\pi$,
$l\in(2\pi\sin\frac{L(I)}{2},2\pi),$ and let $\mathfrak{D}\left(
\overline{pq},x,l-x\right)  $ be the lens enclosed by SCC arcs $c_{x}%
=c_{x}\left(  p,q\right)  $ and $c_{l-x}^{\prime}=c_{l-x}^{\prime}\left(
q,p\right)  $ with $L(c_{x})=x$ and $L(c_{l-x}^{\prime})=l-x.$ Assume $x_{0}$
is the number such that $\mathfrak{D}\left(  \overline{pq},x_{0}%
,l-x_{0}\right)  $ is a disk and $x_{0}>l-x_{0}$. Then the following hold.

(i) If $x>l-x$ and $\angle\left(  \mathfrak{D}\left(  \overline{pq}%
,x,l-x\right)  ,p\right)  >\pi,$ then $x\in(\frac{l}{2},x_{0}).$

(ii) The function $A\left(  \mathfrak{D}\left(  \overline{pq},x,l-x\right)
\right)  $ is strictly increases for $x\in\lbrack\frac{l}{2},x_{0}].$
\end{corollary}

\begin{proof}
If $x>l-x,$ then $x>l/2.$ If $x>x_{0},$ we must have $\angle\left(
\mathfrak{D}\left(  \overline{pq},x,l-x\right)  ,p\right)  <\pi.$ Thus by
assumption of (i) we have $x\in(\frac{l}{2},x_{0}),$ and thus (i) holds true.

By Lemma \ref{ber} and Corollary \ref{2-curvature} (B), $A\left(
\mathfrak{D}\left(  \overline{pq},x,l-x\right)  \right)  $ assumes the minimum
when $x=l/2$ and the maximum when $x=x_{0}.$ Either $x=l/2$ or $x=x_{0}$ is
the condition such that $c_{x}$ and $c_{l-x}^{\prime}$ have the same
curvature, say, when $x\neq l/2,x_{0},$ the two circular arc of $\mathfrak{D}%
\left(  \overline{pq},x,l-x\right)  $ have distinct curvature. On the other
hand the curvature continuously depends on $x\ $and when $x$ increase from
$l/2$ a little, the curvature of $c_{x}$ decrease a little (note that
$c_{l/2}$ is a major circular arc and so is $c_{x}$ for $x>l/2$). Thus when
$x$ increases in $(\frac{l}{2},x_{0}),$ the curvature $k\left(  c_{x}\right)
\ $of $c_{x}$ decreases and the curvature of $c_{l-x}^{\prime}$ increases.
Then (ii) follows from Corollary \ref{2-curvature} (A).
\end{proof}

The following lemma is a direct consequence of Lemma \ref{for-circular}.

\begin{lemma}
\label{good-lune}Let $\Sigma=\left(  f,\overline{\Delta^{+}}\right)
\in\mathcal{F}$ be a surface such that $I=\left(  f,[-1,1]\right)  $ is a line
segment on $S$ and $\Sigma$ is contained in some open hemisphere $S^{\prime}$
on $S.$ Assume $L\left(  \partial\Sigma\right)  <2\pi.$ Then%
\[
A(\Sigma)\leq A(\mathfrak{D}^{\prime}(I,\theta))
\]
for the unique $\theta$ with $L(\partial\mathfrak{D}^{\prime}(I,\theta
))=L(\partial\Sigma),$ with equality if and only if $\Sigma\mathfrak{\ }$is a
simple closed domain congruent to $\overline{\mathfrak{D}^{\prime}(I,\theta
)}.$
\end{lemma}

\begin{lemma}
\label{2circle}Let $L<2\pi$ be a positive number, $x\in\lbrack0,L]$ and
$T_{x}$ and $T_{L-x}$ be two disks in some open hemisphere on $S$ with
$L(\partial T_{x})=x$ and $L(\partial T_{L-x})=L-x.$ Then
\[
2A(T_{L/2})\leq A(T_{x})+A(T_{L-x})\leq T_{L}%
\]
and $A(T_{x})+A(T_{L-x})$ strictly decreases on $[0,L/2].$
\end{lemma}

\begin{proof}
This follows from Corollary \ref{2-curvature} (C). But here we can give a
simpler proof. Let $f(x)=A(T_{x})+A(T_{L-x}).$ Then we have
\[
f(x)=4\pi-\sqrt{4\pi^{2}-x^{2}}-\sqrt{4\pi^{2}-\left(  L-x\right)  ^{2}},
\]
and%
\begin{align*}
f^{\prime}(x)  &  =\frac{x}{\sqrt{4\pi^{2}-x^{2}}}+\frac{\left(  x-L\right)
}{\sqrt{4\pi^{2}-\left(  L-x\right)  ^{2}}}\\
&  =\frac{x\sqrt{4\pi^{2}-\left(  L-x\right)  ^{2}}+\left(  x-L\right)
\sqrt{4\pi^{2}-x^{2}}}{\sqrt{4\pi^{2}-x^{2}}\sqrt{4\pi^{2}-\left(  L-x\right)
^{2}}}.
\end{align*}
Thus we have $f^{\prime}\left(  x\right)  <0$ on $[0,\frac{L}{2}),$ and so
$f(x)$ strictly decreases on $x\in\lbrack0,\frac{L}{2}]$.
\end{proof}

\begin{lemma}
\label{ccircle}Let $c_{n}=c_{n}\left(  q_{n}^{\prime},q_{n}^{\prime\prime
}\right)  $ be a sequence of SCC arcs convergent to an SCC arc $c_{0}%
=c_{0}\left(  q_{0}^{\prime},q_{0}^{\prime\prime}\right)  $ with
$q_{0}^{\prime}\neq q_{0}^{\prime\prime}.$ Assume that either $c_{0}$ is
straight with $L(c_{0})>\pi$ or $c_{0}$ is strictly convex. Then
$c_{0}+\overline{q_{0}^{\prime\prime}q_{0}^{\prime}}$ is a convex Jordan curve
and, for sufficiently large $n$, $c_{n}+\overline{q_{n}^{\prime\prime}%
q_{n}^{\prime}}$ is also a convex Jordan curve and converges to $c_{0}%
+\overline{q_{0}^{\prime\prime}q_{0}^{\prime}}.$ Thus for sufficiently large
$n,$ the lune enclosed by $c_{n}+\overline{q_{n}^{\prime\prime}q_{n}^{\prime}%
}$ is convex and converges to the lune enclosed by $c_{0}+\overline
{q_{0}^{\prime\prime}q_{0}^{\prime}}.$
\end{lemma}

\begin{proof}
This is clear when $c_{0}$ is strictly convex.

Assume $c_{0}$ is straight and $L(c_{0})>\pi.$ Then $d\left(  q_{0}^{\prime
},q_{0}^{\prime\prime}\right)  <\pi$ and thus for sufficiently large $n,$
$d\left(  q_{n}^{\prime},q_{n}^{\prime\prime}\right)  <\pi$ and $\overline
{q_{n}^{\prime}q_{n}^{\prime\prime}}$ converges to $\overline{q_{0}^{\prime
}q_{0}^{\prime\prime}}$. It is clear that $c_{0}+\overline{q_{0}^{\prime
\prime}q_{0}^{\prime}}$ is a great circle and thus, by the assumption of
convergence, $c_{n}+\overline{q_{n}^{\prime\prime}q_{n}^{\prime}}\ $converges
to $c_{0}+\overline{q_{0}^{\prime\prime}q_{0}^{\prime}}.$
\end{proof}

\section{The monotonicity of the function $h_{1}(\delta)$ \label{mono}}

For a line segment $I$ on $S$ with $L(I)=\delta<\pi,$ and a positive number
$\theta\in\lbrack0,\pi]$ define%
\[
A(\delta,\theta,\theta)=A(\mathfrak{D}\left(  I,\theta,\theta\right)  ),
\]
and
\[
L(\delta,\theta,\theta)=L(\partial\mathfrak{D}\left(  I,\theta,\theta\right)
),
\]
where $\mathfrak{D}\left(  I,\theta,\theta\right)  $ is the lens defined in
Definition \ref{lune-lens}. Then for a constant $A_{0}\ $define
\begin{equation}
h(A_{0},\delta,\theta,\theta)=\frac{A_{0}+\left(  q-2\right)  A(\delta
,\theta,\theta)}{L(\delta,\theta,\theta)}=\left(  q-2\right)  \frac
{\frac{A_{0}}{q-2}+A(\delta,\theta,\theta)}{L(\delta,\theta,\theta)}.
\label{ag59}%
\end{equation}

\begin{lemma}
\label{mid}For $\delta\in(0,\pi)$ and $\theta\in\lbrack0,\frac{\pi}{2}]$ we
have%
\begin{equation}
L(\delta,\theta,\theta)=\frac{4\tan\frac{\delta}{2}}{\sqrt{\sin^{2}\theta
+\tan^{2}\frac{\delta}{2}}}\left[  \arctan\frac{\sqrt{\sin^{2}\theta+\tan
^{2}\frac{\delta}{2}}}{\cos\theta}\right]  , \label{1110-1}%
\end{equation}%
\begin{equation}
A(\delta,\theta,\theta)=4\theta-\frac{L(\delta,\theta,\theta)\sin\theta}%
{\tan\frac{\delta}{2}}, \label{1110-2}%
\end{equation}
and%
\begin{equation}
h(A_{0},\delta,\theta,\theta)=\frac{\left(  \frac{A_{0}}{4}+\left(
q-2\right)  \theta\right)  \sqrt{\sin^{2}\theta+\tan^{2}\frac{\delta}{2}}%
}{\tan\frac{\delta}{2}\left[  \arctan\frac{\sqrt{\sin^{2}\theta+\tan^{2}%
\frac{\delta}{2}}}{\cos\theta}\right]  }-\frac{\left(  q-2\right)  \sin\theta
}{\tan\frac{\delta}{2}}. \label{ag60}%
\end{equation}

\end{lemma}

Let $a_{\delta}$ be the positive number such that%
\[
\delta=\int_{0}^{a_{\delta}}\frac{2dx}{1+x^{2}}=2\arctan a_{\delta},
\]
Then $a_{\delta}=\tan\frac{\delta}{2},$ and we have
\[
A(\delta,\theta,\theta)=A(\mathfrak{D}(\overline{0,a_{\delta}},\theta,\theta)
\]
and
\[
L(\delta,\theta,\theta)=L(\partial\mathfrak{D}(\overline{0,a_{\delta}}%
,\theta,\theta).
\]
%

\begin{figure}
[ptb]
\begin{center}
\ifcase\msipdfoutput
\includegraphics[
trim=0.000000in 0.000000in 0.050069in 0.000000in,
height=2.8478in,
width=4.1451in
]%
{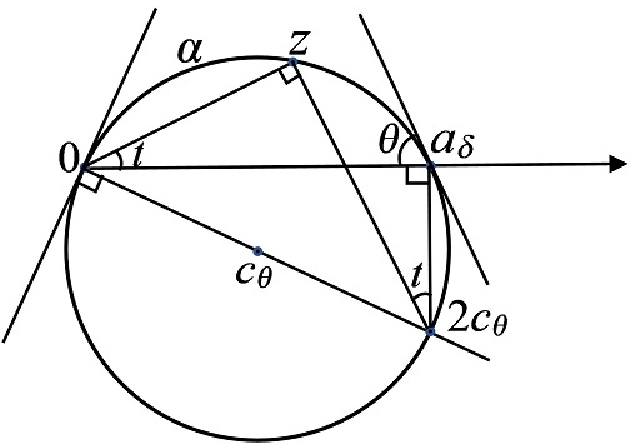}%

\caption{ }%
\label{Ratio-1}%
\end{center}
\end{figure}

\begin{proof}
[\textbf{Proof of Lemma \ref{mid}}]We assume $\theta\in(0,\frac{\pi}{2}]$.
Then, the line segment $\overline{0,a_{\delta}}$ on the sphere $S$ divides
$\mathfrak{D}(\delta,\theta,\theta)$ into two symmetric lunes $\mathfrak{D}%
^{\prime}(\overline{0,a_{\delta}},\theta)$ and $\mathfrak{D}^{\prime
}(\overline{a_{\delta},0},\theta)$, and the upper lune is $\mathfrak{D}%
^{\prime}(\overline{a_{\delta},0},\theta)$, as in Figure \ref{Ratio-1} which
is in the plane $\mathbb{C}$. As shown in the plane figure, let $\alpha$ be
the circular arc on $\mathbb{C}$ of $\partial\mathfrak{D}^{\prime}%
(\overline{a_{\delta},0},\theta)$ from $a_{\delta}$ to $0$, let $c_{\theta}$
with $\operatorname{Im}c_{\theta}\leq0$ be the center of the circle on
$\mathbb{C}$ containing $\alpha,$ and let $c_{\theta}^{\prime}=2c_{\theta}.$
Then
\[
|c_{\theta}^{\prime}-c_{\theta}|=|0-c_{\theta}|=|c_{\theta}-a_{\delta}|,
\]
and thus $c_{\theta}^{\prime}$ is on the circle containing $\alpha$, and the
triangle in $\mathbb{C}$ with vertices $0,a_{\delta}$ and $c_{\theta}^{\prime
}$ is a right triangle whose interior angle at $c_{\theta}^{\prime}$ equals
$\theta$. Thus we have $|c_{\theta}^{\prime}|=a_{\delta}/\sin\theta.$ On the
other hand, for any point $z\in\alpha,$ the triangle with vertices
$0,c_{\theta}^{\prime}$ and $z$ is also a right triangle, as in the figure,
whose angle at $c_{\theta}^{\prime}$ has value $\theta-\arg z$. Thus
\[
|z|=\sin(\theta-t)|c_{\theta}^{\prime}|=a_{\delta}\sin(\theta-t)/\sin\theta,
\]
where $t=\arg z.$ Then, we obtain a parametric expression of the circular path
$\alpha$:%
\[
\alpha(t)=\frac{a_{\delta}\sin(\theta-t)}{\sin\theta}e^{it},t\in
\lbrack0,\theta],
\]
and%
\begin{equation}
|d\alpha(t)|=\frac{a_{\delta}|-e^{it}\cos(\theta-t)+ie^{it}\sin(\theta
-t)|dt}{\sin\theta}=\frac{a_{\delta}dt}{\sin\theta},t\in\lbrack0,\theta].
\label{68}%
\end{equation}

Therefore, we have%
\begin{align}
L(\delta,\theta,\theta)  &  =2L(\alpha)=2\int_{\alpha}\frac{2|dz|}{1+|z|^{2}%
}=\int_{0}^{\theta}\frac{4|d\alpha(t)|}{1+|\alpha(t)|^{2}}\label{Lcom}\\
&  =4\int_{0}^{\theta}\frac{a_{\delta}\sin\theta}{\sin^{2}\theta+a_{\delta
}^{2}\sin^{2}(\theta-t)}dt\nonumber\\
&  =4\int_{0}^{\theta}\frac{a_{\delta}\sin\theta}{\sin^{2}\theta+a_{\delta
}^{2}\sin^{2}x}dx,\nonumber
\end{align}
and so (\ref{1110-1}) follows from
\begin{align*}
L(\delta,\theta,\theta)  &  =-4\frac{a_{\delta}}{\sqrt{\sin^{2}\theta
+a_{\delta}^{2}}}\int_{0}^{\theta}\frac{d\frac{\sin\theta\cot x}{\sqrt
{\sin^{2}\theta+a_{\delta}^{2}}}}{\left[  \frac{\sin\theta\cot x}{\sqrt
{\sin^{2}\theta+a_{\delta}^{2}}}\right]  ^{2}+1}\\
&  =\left.  -4\frac{a_{\delta}}{\sqrt{\sin^{2}\theta+a_{\delta}^{2}}}%
\arctan\frac{\sin\theta\cot x}{\sqrt{\sin^{2}\theta+a_{\delta}^{2}}%
}\right\vert _{0}^{\theta}\\
&  =\frac{2\pi a_{\delta}}{\sqrt{\sin^{2}\theta+a_{\delta}^{2}}}%
-4\frac{a_{\delta}\arctan\frac{\cos\theta}{\sqrt{\sin^{2}\theta+a_{\delta}%
^{2}}}}{\sqrt{\sin^{2}\theta+a_{\delta}^{2}}}\\
&  =\frac{4a_{\delta}\arctan\frac{\sqrt{\sin^{2}\theta+a_{\delta}^{2}}}%
{\cos\theta}}{\sqrt{\sin^{2}\theta+a_{\delta}^{2}}}\\
&  =\frac{4\tan\frac{\delta}{2}}{\sqrt{\sin^{2}\theta+\tan^{2}\frac{\delta}%
{2}}}\arctan\frac{\sqrt{\sin^{2}\theta+\tan^{2}\frac{\delta}{2}}}{\cos\theta}%
\end{align*}

It is then clear that (\ref{1110-2}) will follow from (\ref{68}), the first
line of (\ref{Lcom}) and%
\begin{align*}
A(\delta,\theta,\theta)  &  =2A(\mathfrak{D}^{\prime}(\overline{0,a_{\delta}%
},\theta))=2%
{\displaystyle\iint\limits_{\mathfrak{D}^{\prime}(\overline{0,a_{\delta}%
},\theta)}}
\frac{4dxdy}{\left(  1+|z|^{2}\right)  ^{2}}\\
&  =2\int_{0}^{\theta}dt\int_{0}^{|\alpha_{\theta}(t)|}\frac{4rdr}{\left(
1+r^{2}\right)  ^{2}}=2\int_{0}^{\theta}\left(  2-\frac{2}{1+|\alpha_{\theta
}(t)|^{2}}\right)  dt\\
&  =4\theta-4\int_{0}^{\theta}\frac{dt}{1+|\alpha_{\theta}(t)|^{2}}%
=4\theta-\frac{\sin\theta}{a_{\delta}}\int_{0}^{\theta}\frac{4|d\alpha
_{\theta}(t)|}{1+|\alpha_{\theta}(t)|^{2}}\\
&  =4\theta-\frac{\sin\theta}{\tan\frac{\delta}{2}}L(\delta,\theta,\theta).
\end{align*}

Then
\begin{align*}
h(A_{0},\delta,\theta,\theta)  &  =\frac{A_{0}+\left(  q-2\right)
A(\delta,\theta)}{L(\delta,\theta)}\\
&  =\frac{A_{0}+4\theta\left(  q-2\right)  -\frac{\left(  q-2\right)
\sin\theta}{\tan\frac{\delta}{2}}L(\delta,\theta,\theta)}{L(\delta
,\theta,\theta)}\\
&  =\frac{A_{0}+4\theta\left(  q-2\right)  }{L(\delta,\theta,\theta)}%
-\frac{\left(  q-2\right)  \sin\theta}{\tan\frac{\delta}{2}}\\
&  =\frac{\left(  A_{0}+4\theta\left(  q-2\right)  \right)  \sqrt{\sin
^{2}\theta+\tan^{2}\frac{\delta}{2}}}{4\tan\frac{\delta}{2}\arctan\frac
{\sqrt{\sin^{2}\theta+\tan^{2}\frac{\delta}{2}}}{\cos\theta}}-\frac{\left(
q-2\right)  \sin\theta}{\tan\frac{\delta}{2}}\\
&  =\frac{\left(  \frac{A_{0}}{4}+\theta\left(  q-2\right)  \right)
\sqrt{\sin^{2}\theta+\tan^{2}\frac{\delta}{2}}}{\tan\frac{\delta}{2}%
\arctan\frac{\sqrt{\sin^{2}\theta+\tan^{2}\frac{\delta}{2}}}{\cos\theta}%
}-\frac{\left(  q-2\right)  \sin\theta}{\tan\frac{\delta}{2}}.
\end{align*}

\end{proof}

It is clear that when the boundary length of $\mathfrak{D}\left(
I,\theta,\theta\right)  $ is fixed while $\theta$ increases, $\delta=L(I)$ has
to decrease.

\begin{lemma}
\label{area}If $0<\delta_{2}<\delta_{1}<\pi,$ $\theta_{1}<\theta_{2}\leq
\frac{\pi}{2}$ and $L\left(  \delta_{1},\theta_{1},\theta_{1}\right)
=L(\delta_{2},\theta_{2},\theta_{2}),$ then
\[
A(\delta_{1},\theta_{1},\theta_{1})<A(\delta_{2},\theta_{2},\theta_{2}).
\]
%

\begin{figure}
[ptb]
\begin{center}
\ifcase\msipdfoutput
\includegraphics[
trim=0.000000in 0.000000in 0.022420in 0.000000in,
height=3.2526in,
width=4.1234in
]%
{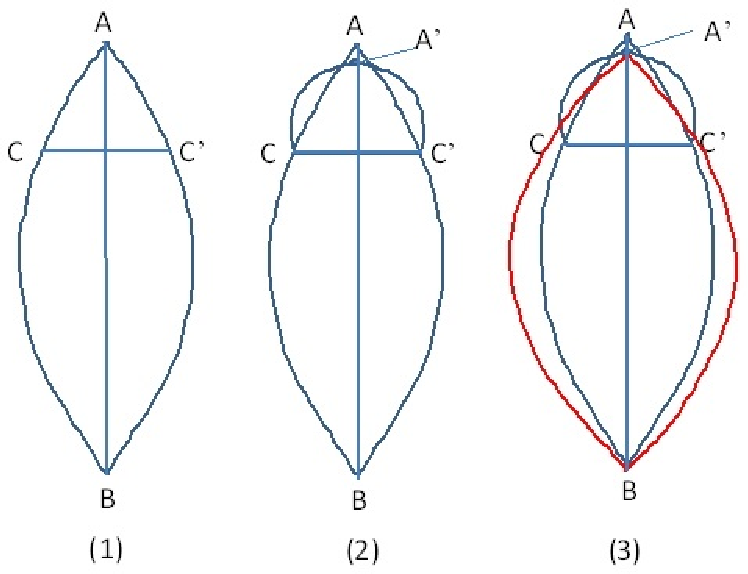}%

\caption{ }%
\label{area-increase}%
\end{center}
\end{figure}

\end{lemma}

\begin{proof}
We fix $\theta_{1}$ and $L(\delta_{1},\theta_{1},\theta_{1})=L(\delta
_{2},\theta_{2},\theta_{2})=L,$ and show that the conclusion holds when
$\theta_{2}-\theta_{1}$ is small enough. Assume $\mathfrak{D}(\overline
{AB},\theta_{1},\theta_{1})$ with $L(\overline{AB})=\delta_{1}$ is given as in
Figure \ref{area-increase}, in which $A$ and $B$ are the two cusps. For a pair
of points $C$ and $C^{\prime}$ on $\partial\mathfrak{D}(\overline{AB}%
,\theta_{1},\theta_{1})$ which are symmetric to $\overline{AB},$ we let
$\gamma_{CC^{\prime}}$ be the circular arc with chord $\overline{CC^{\prime}}$
and length equal to the part $\widehat{C^{\prime}AC}$ of $\partial
\mathfrak{D}(\overline{AB},\theta_{1},\theta_{1}).$ It is clear that for
sufficiently small $\varepsilon,$ there exists a pair of $C$ and $C^{\prime}$
such that $\gamma_{CC^{\prime}}$ intersects $AB$ at $A^{\prime}$ with
$d(A^{\prime},B)=\delta_{1}-\varepsilon$ as in Figure \ref{area-increase} (2).
Let $D_{CBC^{\prime}A^{\prime}C}$ be the domain enclosed by the Jordan curve
consisted of $\gamma_{CC^{\prime}}$ and the part of $\partial\mathfrak{D}%
(\overline{AB},\theta_{1},\theta_{1})$ under the segment $\overline
{CC^{\prime}}.$ Then by Lemma \ref{good-lune} we have
\[
A(\delta_{1},\theta_{1},\theta_{1})\leq A(D_{CBC^{\prime}A^{\prime}C}),
\]
equality holding only if $\widehat{C^{\prime}AC}$ is a circular arc. But by
the hypothesis $\theta_{1}<\theta_{2}\leq\frac{\pi}{2},$ $\widehat{C^{\prime
}AC}$ can not be circular, and so the equality can not hold. It is clear that
the boundary length of the domain $D_{CBC^{\prime}A^{\prime}C}$ equals $L,$
and by Lemma \ref{good-lune} the domain $\mathfrak{D}(\overline{A^{\prime}%
B},\theta_{2},\theta_{2})$ with $L(\mathfrak{D}(\overline{A^{\prime}B}%
,\theta_{2},\theta_{2}))=L$ has larger area than $D_{CBC^{\prime}A^{\prime}%
C}.$ Thus we have%
\[
A(\delta_{1},\theta_{1},\theta_{1})<A(D_{CBC^{\prime}A^{\prime}C}%
)<A(\delta_{1}-\varepsilon,\theta_{2},\theta_{2}).
\]
Since $\varepsilon$ is arbitrary, the result follows.
\end{proof}

\begin{lemma}
\label{disk}(i) For $\delta\in\lbrack\delta_{E_{q}},\pi),$ let $D(\delta)$ be
a closed disk on $S$ with diameter $\delta$ and write $A(\delta)=A(D(\delta
))$, $L(\delta)=L(\partial D(\delta)).$ Then%
\begin{equation}
\left.
\begin{tabular}
[c]{ll}%
$h\left(  \delta\right)  $ & $=\dfrac{R(D(\delta))+4\pi}{L(\delta)}%
=\dfrac{4\pi-4\pi\overline{n}\left(  D(\delta\right)  )+\left(  q-2\right)
A(\delta)}{L(\delta)}$\\
& $=h(4\pi\left(  1-\overline{n}\left(  D(\delta\right)  \right)
),\delta,\frac{\pi}{2},\frac{\pi}{2}),$%
\end{tabular}
\ \ \ \right.  \label{529-1}%
\end{equation}
where $h\left(  A_{0},\delta,\theta,\theta\right)  $ is defined in (\ref{ag59}).

(ii) For any open interval $I^{\circ}$ in $[\delta_{E_{q}},\pi),$ if
$\overline{n}\left(  D\left(  \delta\right)  \right)  $ is a constant for each
$\delta\in I^{\circ}$, then $h\left(  \delta\right)  =\frac{R(D(\delta))+4\pi
}{L(\delta)}$ is real analytic on $I^{\circ}$ with%
\begin{equation}
h^{\prime}\left(  \delta\right)  =\frac{-\left(  q-2\overline{n}\left(
D(\delta\right)  )\right)  \cos\frac{\delta}{2}+q-2}{2\sin^{2}\frac{\delta}%
{2}}, \label{h'(d)}%
\end{equation}
and the following hold.

(ii1) If $\overline{n}\left(  D\left(  \delta\right)  \right)  =0$ on
$I^{\circ}$ and $\delta_{E_{q}}\geq2\arccos\frac{q-2}{q},$ then $h\left(
\delta\right)  $ strictly increases on $I^{\circ}\cap\lbrack\delta_{E_{q}}%
,\pi).$

(ii2) If $\overline{n}\left(  D\left(  \delta\right)  \right)  =0$ on
$I^{\circ}$ and $\delta_{E_{q}}<2\arccos\frac{q-2}{q},$ then $h\left(
\delta\right)  $ strictly decreases on $I^{\circ}\cap\lbrack\delta_{E_{q}%
},2\arccos\frac{q-2}{q}]$ and strictly increases on $I^{\circ}\cap
\lbrack2\arccos\frac{q-2}{q},\pi].$

(ii3) If $\overline{n}\left(  D\left(  \delta\right)  \right)  \geq1$ on
$I^{\circ},$ then $h\left(  \delta\right)  $ strictly increases on $I^{\circ
}.$

(ii4) $h\left(  \delta\right)  $ cannot assume the maximum value for every
$\delta\in I^{\circ}.$
\end{lemma}

\begin{proof}
By definition, (\ref{529-1}) holds trivially. It is clear that%
\begin{align*}
h(\delta)  &  =\frac{R(D(\delta))+4\pi}{L(\delta)}=h\left(  4\pi-4\pi
\overline{n}\left(  D\left(  \delta\right)  \right)  ,\delta,\frac{\pi}%
{2},\frac{\pi}{2}\right) \\
&  =\frac{4\pi-4\pi\overline{n}\left(  D(\delta\right)  )+\left(  q-2\right)
A(\delta)}{L(\delta)}\\
&  =\frac{4\pi-4\pi\overline{n}\left(  D(\delta\right)  )+2\pi\left(
q-2\right)  (1-\cos\frac{\delta}{2})}{2\pi\sin\frac{\delta}{2}}\\
&  =\frac{q-2\overline{n}\left(  D(\delta\right)  )-\left(  q-2\right)
\cos\frac{\delta}{2}}{\sin\frac{\delta}{2}}.
\end{align*}
Thus we have (\ref{h'(d)}). If $h(\delta)$ assume a maximum value at some
point $\delta\in I^{\circ},$ then we must have
\begin{equation}
-\left(  q-2\overline{n}\left(  D(\delta\right)  )\right)  \cos\frac{\delta
}{2}+q-2=0. \label{h'(d)-1}%
\end{equation}
Therefore, (ii1)--(ii3) follow from (\ref{h'(d)}) directly. Now, (ii4) follows
from (ii1)--(ii3).\label{2023-04-05 11:16}
\end{proof}

\begin{lemma}
\label{must-strict-con}\label{used-in-must-strict-in Theoren}For any pair of
constants $\delta\in(0,\pi)$ and $A_{0},$ let $h\left(  A_{0},\delta
,\theta,\theta\right)  $ be given by (\ref{ag59}) with expression
(\ref{ag60}).
Then $h(\theta)=h\left(  A_{0},\delta,\theta,\theta\right)  $ strictly
increases in an interval $[0,\theta_{0}]$ for some $\theta_{0}>0.$
\end{lemma}

\begin{proof}
Since $\delta\in(0,\pi)$, we have $\tan\frac{\delta}{2}>0,$ and thus it is
clear that the functions $h\left(  \theta\right)  $ is real analytic in a
neighborhood $I_{0}$ of $0\ $in $(-1,1).$ It is clear that
\[
\frac{d}{d\theta}\frac{\sqrt{\sin^{2}\theta+\tan^{2}\frac{\delta}{2}}}%
{\tan\frac{\delta}{2}\left[  \arctan\frac{\sqrt{\sin^{2}\theta+\tan^{2}%
\frac{\delta}{2}}}{\cos\theta}\right]  }=0
\]
at $\theta=0.$ Since $\delta\in(0,\pi),$ we have by (\ref{ag60}),%

\[
h^{\prime}\left(  0\right)  =\frac{q-2}{\frac{\delta}{2}}-\frac{q-2}{\tan
\frac{\delta}{2}}>0,
\]
and so the existence of $\theta_{0}$ follows.\label{2023-04-05 12:22}
\end{proof}

\begin{lemma}
\label{area2}For $\delta\in\left(  0,\pi\right)  $ and for a constant $A_{0},$
let $h(\delta)=\max_{\theta\in\lbrack0,\frac{\pi}{2}]}h(A_{0},\delta
,\theta,\theta),$ where $h(A_{0},\delta,\theta,\theta)$ is given by
(\ref{ag59}) with expression (\ref{ag60}). Let $\delta_{0}\in(0,\pi)$ and
assume $h\left(  \delta_{0}\right)  =h\left(  A_{0},\delta_{0},\theta
_{0},\theta_{0}\right)  $ for some $\theta_{0}\in\left(  0,\frac{\pi}%
{2}\right)  $. Then for each $\delta<\delta_{0}$ which is sufficiently close
to $\delta_{0},$ we have%
\[
h\left(  A_{0},\delta_{0},\theta_{0},\theta_{0}\right)  <h\left(  A_{0}%
,\delta,\theta_{\delta},\theta_{\delta}\right)  ,
\]
where $\theta_{\delta}$ is determined by $L(\partial\mathfrak{D}\left(
\delta_{0},\theta_{0},\theta_{0}\right)  )=L(\partial\mathfrak{D}\left(
\delta,\theta_{\delta},\theta_{\delta}\right)  )$, and thus%
\[
h\left(  \delta\right)  =\max_{\theta\in\lbrack0,\frac{\pi}{2}]}h(A_{0}%
,\delta,\theta,\theta)>h\left(  \delta_{0}\right)  .
\]

\end{lemma}

\begin{proof}
It is clear that for sufficiently small $h>0,$ there exists $\theta
_{1}^{\prime}\in(\theta_{0},\frac{\pi}{2})$ such that
\[
L(\delta_{0}-h,\theta_{1}^{\prime},\theta_{1}^{\prime})=L(\delta_{0}%
,\theta_{0},\theta_{0}),
\]
and then by Lemma \ref{area} we have%
\begin{align*}
h(A_{0},\delta_{0},\theta_{0},\theta_{0})  &  =\frac{A_{0}}{L(\delta
_{0},\theta_{0},\theta_{0})}+\frac{\left(  q-2\right)  A(\delta_{0},\theta
_{0},\theta_{0})}{L(\delta_{0},\theta_{0},\theta_{0})}\\
&  <\frac{A_{0}}{L(\delta_{0}-h,\theta_{1}^{\prime},\theta_{1}^{\prime}%
)}+\frac{\left(  q-2\right)  A(\delta_{0}-h,\theta_{1}^{\prime},\theta
_{1}^{\prime})}{L(\delta_{0}-h,\theta_{1}^{\prime},\theta_{1}^{\prime})}\\
&  =\frac{A_{0}+\left(  q-2\right)  A(\delta_{0}-h,\theta_{1}^{\prime}%
,\theta_{1}^{\prime})}{L(\delta_{0}-h,\theta_{1}^{\prime},\theta_{1}^{\prime
})}\\
&  \leq\max_{\theta\in\lbrack0,\pi/2]}\frac{A_{0}+\left(  q-2\right)
A(\delta_{0}-h,\theta,\theta)}{L(\delta_{0}-h,\theta,\theta)}\\
&  =h(\delta_{0}-h).
\end{align*}

\end{proof}

\begin{lemma}
Let $\delta_{E_{q}}=\min\left\{  d(\mathfrak{a}_{i},\mathfrak{a}%
_{j}):\mathfrak{a}_{i}\in E_{q},\mathfrak{a}_{j}\in E_{q},\mathfrak{a}_{i}%
\neq\mathfrak{a}_{j}\right\}  .$ Then $\delta_{E_{q}}\leq\frac{2\pi}{3},$
equality holding only if $q=3$ and $\mathfrak{a}_{1},\mathfrak{a}%
_{2},\mathfrak{a}_{3}$ are on a great circle on $S$ and $d(\mathfrak{a}%
_{i},\mathfrak{a}_{j})=\frac{2\pi}{3}$ for each pair of $i$ and $j$ with
$i\neq j. $
\end{lemma}

\begin{proof}
Assume $\delta_{E_{q}}>\frac{2\pi}{3}.$ Then we may assume $\mathfrak{a}%
_{2},\mathfrak{a}_{3},\dots,\mathfrak{a}_{q}$ are contained in the open disk
$D$ complementary to the disk $d(z,\mathfrak{a}_{1})\leq\frac{2\pi}{3}.$ But
it is clear that the diameter of $D$ is $\frac{2\pi}{3}$ and since $q\geq3$ we
have $\delta_{E_{q}}<\frac{2\pi}{3}.$ This is a contradiction.

Assume $\delta_{E_{q}}=\frac{2\pi}{3}.$ Then we may assume $d\left(
\mathfrak{a}_{1},\mathfrak{a}_{2}\right)  =\frac{2\pi}{3}$ and consider the
great circle $C$ on $S$ passing through $\mathfrak{a}_{1}$ and $\mathfrak{a}%
_{2}.$ If some $\mathfrak{a}_{j}$ of $\mathfrak{a}_{3},\dots,\mathfrak{a}_{q}
$ is contained in $S\backslash C,$ then we have $d\left(  \mathfrak{a}%
_{j},\{\mathfrak{a}_{1},\mathfrak{a}_{2}\}\right)  <\frac{2\pi}{3},$
contradicting to $\delta_{E_{q}}=\frac{2\pi}{3}.$ Thus $\left\{
\mathfrak{a}_{3},\dots,\mathfrak{a}_{q}\right\}  \subset C$, $q=3,$ and
$\mathfrak{a}_{1},\mathfrak{a}_{2},\mathfrak{a}_{3}$ divide $C$ equally.
\end{proof}

\begin{lemma}
\label{dDisk}For the number $\delta_{E_{q}},$ there exists a disk
$T_{0}\subset S$ with perimeter $3\delta_{E_{q}}$ and $T_{0}\cap
E_{q}=\emptyset.$
\end{lemma}

\begin{proof}
Let $\mathfrak{a}_{1}$ and $\mathfrak{a}_{2}$ be two points of $E_{q}$ such
that $d\left(  \mathfrak{a}_{1},\mathfrak{a}_{2}\right)  =\delta_{E_{q}}.$
Then there exists a convex disk $T_{0}$ on $S$ whose boundary contains
$\mathfrak{a}_{1},$ $\mathfrak{a}_{2}$ and some $a\in S$ such that
$\mathfrak{a}_{1},\mathfrak{a}_{2},a$ divide $\partial T_{0}$ equally. It is
clear that $T_{0}$ is the desired disk.
\end{proof}

\begin{lemma}
\label{TD}\label{hd}Let $T$ be a convex disk on $S$ with $T\cap E_{q}%
=\emptyset.$ Then $H\left(  T\right)  $ is strictly increase as a function of
$L=L(\partial T)\in(0,\min\{L_{0},2\pi\})\ $and $H\left(  T\right)
\leq\left(  q-2\right)  ,$ where
\[
L_{0}=\max\left\{  L:%
\begin{array}
[c]{c}%
\mathrm{there\ is\ a\ convex\ open\ disk\ }T\mathrm{\ on\ }S\mathrm{\ with}\\
T\cap E_{q}=\emptyset\text{ and }L(\partial T)=L\}
\end{array}
\right\}  \geq3\delta_{E_{q}}.\mathrm{\ }%
\]
$\mathrm{\ }$
\end{lemma}

\begin{proof}
This follows from Lemma \ref{dDisk}, $\overline{n}\left(  T\right)  =0,$ and
that, as a function of $L=L(\partial T),$
\begin{align*}
H(T)  &  =\frac{\left(  q-2\right)  A(T)}{L}=\frac{\left(  q-2\right)  \left(
2\pi-\sqrt{4\pi^{2}-L^{2}})\right)  }{L}\\
&  =\frac{\left(  q-2\right)  L}{\left(  2\pi+\sqrt{4\pi^{2}-L^{2}})\right)  }%
\end{align*}
strictly increase for $L\in(0,\min\{L_{0},2\pi\})]$ and $H\left(  T\right)
\leq\left(  q-2\right)  $ on $(0,\min\{L_{0},2\pi\}]$.
\end{proof}

\section{The Riemann Hurwitz formula and branch points}

The simplest version of Riemann Hurwitz formula is that for any BCCM
$f:S\rightarrow S$ with degree $d$
\begin{equation}
B_{f}=\sum_{z\in S}b_{f}\left(  z\right)  =\sum_{z\in S}\left(  v_{f}\left(
z\right)  -1\right)  =2d-2. \label{rh}%
\end{equation}
Recall that $B_{f},b_{f}\left(  z\right)  \ $and $v_{f}\left(  z\right)  $ are
defined in Remark \ref{notation} (D) and (E). The formula implies that $f$ has
exactly $2d-2$ branch points on $S.$ This formula implies the following directly.

\begin{lemma}
\label{RH}Let $\Sigma=(f,S)$ be a surface, say, $f:S\rightarrow S$ is a $d$ to
$1$ BCCM. Then
\begin{equation}
\sum_{v=1}^{q}\bar{n}(\Sigma,\mathfrak{a}_{v})=qd-\sum_{z\in E_{q}}%
b_{f}(z)\geq qd-\sum_{z\in S}b_{f}(z)=\left(  q-2\right)  d+2, \label{RH1}%
\end{equation}
and
\begin{equation}
R(\Sigma)=\left(  q-2\right)  A(\Sigma)-4\pi\sum_{v=1}^{q}\overline{n}%
(\Sigma,\mathfrak{a}_{v})\leq-8\pi. \label{RH2}%
\end{equation}
both equality holding iff $C_{f}^{\ast}=C_{f}\left(  S\backslash f^{-1}%
(E_{q})\right)  =\emptyset,$ say, $CV_{f}\subset E_{q}$ (see Remark
\ref{notation} (D) for the notations $C_{f}^{\ast}$ and $CV_{f}$).
\end{lemma}

In fact, the first two parts of (\ref{RH1}) are trivial, and the third part of
(\ref{RH1}) follows from the equation (\ref{rh}). Then (\ref{RH1}), together
with the area formula $A(\Sigma)=4\pi d,$ implies (\ref{RH2}).

\begin{corollary}
\label{loop2}Let $\Sigma=\left(  f,\overline{\Delta}\right)  \in\mathcal{F}$
such that $\partial\Sigma$ is a Jordan curve and let $T$ be the closed domain
enclosed by $\partial\Sigma$ on $S.$ Then for the closed surface $\Sigma
_{0}=\left(  f_{0},S\right)  $ which is obtained by sewing $\Sigma$ and the
domain $S\backslash T$ along $\partial\Sigma,$ we have%
\begin{equation}
R(\Sigma)=R(\Sigma_{0})+R(T)+8\pi\leq R(T), \label{RH-1}%
\end{equation}
with equality holding if ond only if $\Sigma\in\mathcal{F}_{r},$ where%
\begin{equation}
\mathcal{F}_{r}=\{\left(  f,\overline{U}\right)  \in\mathcal{F}:C_{f}^{\ast
}=C_{f}\left(  \overline{U}\backslash f^{-1}(E_{q})\right)  =\emptyset\}.
\label{Fr}%
\end{equation}

\end{corollary}

\begin{proof}
The first equality of (\ref{RH-1}) follows from Corollary \ref{loop1}. By
Lemma \ref{RH}, $R(\Sigma_{0})+8\pi\leq0,$ equality holding iff $CV_{f_{0}%
}\subset E_{q}.$ On the other hand, by Corollary \ref{loop1} we have
$CV_{f_{0}}=CV_{f}.$ Thus $CV_{f_{0}}\subset E_{q}$ iff $CV_{f}\subset E_{q},$
say, $\Sigma\in\mathcal{F}_{r}$.
\end{proof}

\begin{lemma}
\label{hw1}Let $K$ be a Jordan domain on $S$, let $\Sigma=\left(
f,\overline{U}\right)  \ $be a surface in $\mathbf{F}$ such that $f|_{\partial
U}:\partial U\rightarrow\partial K$ is an orientation preserving CCM with
degree $d$ and that $f$ covers $K$ by $d_{0}$ times\footnote{This only means
that each point of $K$ has $d_{0}$ inversers, counted with multiplicity.}.
Then the following hold.

(i) In the case $E_{q}\subset K,$%
\[
\#f^{-1}(E_{q})\geq\left(  q-2\right)  d_{0}+d+1.
\]

(ii) In the case $E_{q}\cap K=q-1$ and $E_{q}\cap\partial K=\emptyset$
\[
\#f^{-1}(E_{q})\geq\left(  q-2\right)  d_{0}+1.
\]

\end{lemma}

By the convention, $\partial K$ is oriented in the way that $K$ is on the left
of $\partial K,$ and so $d\leq d_{0},$ with equality only if $f\left(
U\right)  \subset K.$

\begin{proof}
This is a consequence of Riemann-Hurwitz formula (\ref{rh}). We may extend $f$
to be a BCCM $F:S=\overline{\mathbb{C}}\rightarrow S$ so that $F$ restricted
to $\overline{\mathbb{C}}\backslash\overline{U}$ is a BCCM onto $S\backslash
K$ and that $F$ contains no branch point in $\overline{\mathbb{C}}%
\backslash\overline{U}$ when $d=1$, or contains only one branch point $p$ in
$\overline{\mathbb{C}}\backslash\overline{U}$ with $v_{f}(p)=d>1$.

(i) If $E_{q}\subset K,$ then
\begin{align*}
\#f^{-1}(E_{q})  &  =\#F^{-1}(E_{q})=d_{0}q-\sum_{z\in F^{-1}(E_{q})}\left(
v_{F}(z)-1\right) \\
&  \geq d_{0}q-\sum_{z\in U}\left(  v_{F}(z)-1\right) \\
&  =d_{0}q-\left[  \sum_{z\in\overline{\mathbb{C}}}\left(  v_{F}(z)-1\right)
-\sum_{z\in\overline{\mathbb{C}}\backslash\overline{U}}\left(  v_{F}%
(z)-1\right)  \right] \\
&  =d_{0}q-\left[  \left(  2d_{0}-2\right)  -\left(  d-1\right)  \right]
=\left(  q-2\right)  d_{0}+d+1.
\end{align*}

(ii) If $E_{q}\cap K=q-1,$ then we may construct $F$ so that $F(p)\in E_{q}$
and $v_{F}(p)=d$ for some $p\in\overline{\mathbb{C}}\backslash\overline{U},$
and then we have
\begin{align*}
\#f^{-1}(E_{q})  &  =\#F^{-1}(E_{q})-1=\left[  d_{0}q-\sum_{z\in F^{-1}%
(E_{q})}\left(  v_{F}(z)-1\right)  \right]  -1\\
&  \geq\left[  d_{0}q-\left(  2d_{0}-2\right)  \right]  -1=\left(  q-2\right)
d_{0}+1.
\end{align*}

\end{proof}

Now, we apply the previous lemma to prove the following lemma and theorem.

\begin{lemma}
\label{R<A}Let $\Sigma=\left(  f,\overline{\Delta}\right)  \in\mathbf{F}$ and
assume that there exists a Jordan curve $\Gamma$ in $S\backslash E_{q}$ such
that for the two components $K_{1}$ and $K_{2}$ of $S\backslash\Gamma,$
$\overline{K_{1}}\ $is contained in the lower half sphere $D\left(
0,\frac{\pi}{2}\right)  $ on $S,$ $\#E_{q}\cap K_{2}\geq q-1,$ and
$\partial\Sigma\subset K_{1}.$ Then
\begin{equation}
R(\Sigma)\leq\left(  q-2\right)  A\left(  \partial\Sigma\right)  , \label{mm2}%
\end{equation}
and moreover, if in addition $\overline{n}\left(  \Sigma\right)  \neq0,$ then%
\[
R(\Sigma)\leq\left(  q-2\right)  A\left(  \partial\Sigma\right)  -4\pi.
\]

\end{lemma}

\begin{proof}
Since $f$ has a finite number of branch values, we may assume that $\Gamma$
contains no branch value of $f.$ For otherwise, we can replace $\Gamma$ with
another Jordan curve $\Gamma^{\prime}$ close to $\Gamma$ enough so that
$\Gamma^{\prime}$ contains no branch value of $f$ and satisfies all the hypothesis.

Firstly, consider the case $\infty\notin f(\Delta).$ Then by Lemma \ref{Al}
\[
A(\Sigma)=A\left(  \partial\Sigma\right)  +4\pi\deg_{f}\left(  \infty\right)
=A\left(  \partial\Sigma\right)  ,
\]
and then
\begin{equation}
R(\Sigma)=\left(  q-2\right)  A(\partial\Sigma)-4\pi\overline{n}\left(
\Sigma\right)  , \label{nb1}%
\end{equation}
which implies the desired result when $\infty\notin f(\Delta)$.

Now assume $\infty\in f(\Delta)$ and let $\deg_{f}\left(  \infty\right)
=d_{0}.$ Then $d_{0}\ $is a positive integer and $f^{-1}(\Gamma)$ is not empty
and is consisted of a finite number of disjoint Jordan curves in $\Delta$.
Then $f^{-1}(\Gamma)\cap\partial\Delta=\emptyset,$ and $f^{-1}(\Gamma)$
divides $\Delta$ into a finite number of domains, and we let $V$ be the
component of $\Delta\backslash f^{-1}(\Gamma)$ with $\partial\Delta
\subset\partial V.$ Then $f(V)\cap\Gamma=\emptyset$ and%
\begin{equation}
f(V)\subset K_{1}, \label{aab1}%
\end{equation}
and then $\Delta\backslash\overline{V}$ is consisted of a finite number of
Jordan domains $U_{j},j=1,\dots,k,$ such that for each $j,$ $\overline{U_{j}%
}\subset\Delta,$ $\overline{U_{i}}\cap\overline{U_{j}}=\emptyset$ if $i\neq
j,$ and $f|_{\partial U_{j}}$ is a CCM from $\partial U_{j}$ onto $\Gamma.$ We
assume $f|_{U_{j}}$ covers $K_{2}$ with degree $d_{j}$ and $f|_{\partial
U_{j}}$ covers $\Gamma$ with degree $d_{j}^{\prime},j=1,\dots,k.$ Then we
have
\[
d_{0}=d_{1}+\cdots+d_{k}.
\]
Note that\footnote{It is possible that some Jordan curve of $f^{-1}(\Gamma
)\ $encloses another Jordan curve of $f^{-1}(\Gamma)$ in $\Delta.$ Thus
$U_{1},\dots,U_{k}$ maybe Jordan domains enclosed by just a part of Jordan
curves of $f^{-1}(\Gamma)$. If some $U_{j}$ contains some Jordan curve of
$f^{-1}(\Gamma),$ then $d_{j}>d_{j}^{\prime}$ and $f(U_{j})\supset S.$}
$d_{j}>d_{j}^{\prime}$ when $f|_{\overline{U_{j}}}\supset S$ and $d_{j}%
=d_{j}^{\prime}$ when $f|_{\overline{U_{j}}}=K_{2}.$ We write
\[
d_{0}^{\prime}=d_{1}^{\prime}+\cdots+d_{k}^{\prime}.
\]
By Lemma \ref{Al} we have%
\begin{equation}
A(\Sigma)=A(\partial\Sigma)+4\pi d_{0}, \label{A=A+}%
\end{equation}
where
\[
A(\partial\Sigma)=\frac{2}{i}\int_{\partial\Sigma}\frac{\overline{w}%
dw}{1+|w|^{2}}.
\]

The restriction $(f,\overline{U_{j}}),\partial K_{2},d_{j},d_{j}^{\prime}$
satisfies the hypothesis of Lemma \ref{hw1} (as $(f,\overline{U}),\partial
K,d_{0},d\ $there). Thus we have%
\[
\overline{n}\left(  U_{j}\right)  =\#f^{-1}(E_{q})\cap U_{j}\geq\left(
q-2\right)  d_{j}+1,\mathrm{\ for\ }j=1,\dots,k.
\]
Therefore
\begin{align*}
\overline{n}\left(  \Sigma\right)   &  \geq\#\left(  f^{-1}(E_{q})\cap
\cup_{j=1}^{k}U_{j}\right)  =\sum_{j=1}^{k}\#\left(  f^{-1}(E_{q})\cap
U_{j}\right) \\
&  \geq\sum_{j=1}^{k}\left(  \left(  q-2\right)  d_{j}+1\right)  =\left(
q-2\right)  d_{0}+k>\left(  q-2\right)  d_{0}+1.
\end{align*}
In consequence we have by (\ref{A=A+})
\begin{align*}
R(\Sigma)  &  =\left(  q-2\right)  A(\Sigma)-4\pi\overline{n}\left(
\Sigma,E_{q}\right) \\
&  \leq\left(  q-2\right)  \left[  A\left(  \partial\Sigma\right)  +4\pi
d_{0}\right]  -4\pi\left[  \left(  q-2\right)  d_{0}+1\right] \\
&  =\left(  q-2\right)  A\left(  \partial\Sigma\right)  -4\pi.
\end{align*}
This completes the proof.
\end{proof}

\begin{theorem}
\label{l<2dt}\label{2023-4-5-17:18}Let $L$ be a positive number with
$L\leq2\delta_{Eq}$, let $\Sigma=\left(  f,\overline{\Delta}\right)
\in\mathbf{F}$ with $L\left(  \partial\Sigma\right)  \leq L\ $and let $T_{0}$
be a closed disk in some hemisphere on $S$ with perimeter $L$ and $T_{0}\cap
E_{q}=\emptyset.$ Then
\begin{equation}
H(\Sigma)\leq H(T_{0}) \label{ab2}%
\end{equation}
with equality holding iff $\Sigma$ is a simple disk with $f(\Delta)\cap
E_{q}=\emptyset$.
\end{theorem}

By Lemma \ref{dDisk}, $T_{0}$ must exist.

\begin{proof}
By Rado's theorem, $\partial\Sigma$ is contained in some open hemisphere
$S^{\prime}$ on $S.$ We may assume $S^{\prime}=D(0,\pi/2),$ otherwise we
replace the surface $\Sigma$ and the set $E_{q}$ with $\Sigma^{\prime}=\left(
\varphi\circ f,\overline{\Delta}\right)  $ and $E_{q}^{\prime}=\varphi
(E_{q})\ $so that $\partial\Sigma^{\prime}\subset D(0,\pi/2),$ where $\varphi$
is a rotation of $S$ so that $\varphi\left(  \partial\Sigma\right)  \subset
D(0,\pi/2).$

Let $A$ be the convex hull of $\partial\Sigma$ in $D(0,\pi/2).$ Then by the
assumption we have
\begin{equation}
L(\partial A)\leq L(\partial\Sigma)\leq L \label{ab1}%
\end{equation}
and we can conclude that $A$ contains at most two points of $E_{q}.$

If $A$ contains two points $\mathfrak{a}$ and $\mathfrak{b}$ of $E_{q},$ then
we have $L=2\delta_{Eq}$, $\partial\Sigma=\overline{\mathfrak{ab}}%
+\overline{\mathfrak{ba}}$ and $\left(  \partial\Sigma\right)  \cap
E_{q}=\{\mathfrak{a},\mathfrak{b}\},$ and then we can sew\label{sew13}
$\Sigma$ along $\overline{\mathfrak{ab}}$ to obtain a surface $\Sigma
_{0}=\left(  F,S\right)  $ such that $F$ is a BCCM from $S$ onto $S.$ Then
$\overline{n}(\Sigma_{0})=\overline{n}\left(  \Sigma\right)  +2,$
$A(\Sigma_{0})=A(\Sigma)$ and thus we have by Lemma \ref{RH}
\[
R(\Sigma)=R(\Sigma_{0})+8\pi\leq0<R(T_{0})=\left(  q-2\right)  A(T_{0}),
\]
which implies (\ref{ab2}).

Assume $A$ contains at most one point of $E_{q}.$ Then there exists a Jordan
curve $\Gamma$ in $D\left(  0,\pi/2\right)  \backslash A$ whose interior
domain contains $A$ such that for the doubly connected domain $V$ between
$\Gamma$ and $\partial A$ in $D\left(  0,\pi/2\right)  ,$ $\overline
{V}\backslash\partial A$ contains no point of $E_{q}.$ Let $K_{1}$ be the
component of $S\backslash\Gamma$ which contains $A$ and $K_{2}$ the other
component. Then we have $\infty\in K_{2},$ $K_{1}\cup K_{2}=S\backslash
\Gamma,$ and $K_{2}$ contains at least $q-1$ points of $E_{q}.$ Thus by Lemma
\ref{R<A} and Lemma \ref{good-1} (i) we have
\[
R(\Sigma)\leq\left(  q-2\right)  A\left(  \partial\Sigma\right)  \leq\left(
q-2\right)  A(T_{1}),
\]
where $T_{1}$ is a disk in $S\backslash E_{q}$ with perimeter $L\left(
\partial\Sigma\right)  $, which with Lemma \ref{hd}, implies (\ref{ab2}). This
completes the proof.\label{11-1}
\end{proof}

\section{The spaces $\mathcal{F},\mathcal{F}(L),\mathcal{F}_{r},\mathcal{F}%
_{r}(L),\mathcal{C}\left(  L,m\right)  ,\mathcal{C}^{\ast}\left(  L,m\right)
,\mathcal{F}(L,m),\mathcal{F}_{r}(L,m)$}

Continuing Definition \ref{F}, in which $\mathcal{F}$ and $\mathcal{F}\left(
L\right)  $ have been defined, we introduce some subspaces of $\mathcal{F}$
and $\mathcal{F}\left(  L\right)  .$

\begin{definition}
\label{circu}(a) $\mathcal{C}(L,m)$ denotes the subspace of $\mathcal{F}(L)$
such that $\Sigma=\left(  f,\overline{U}\right)  \in\mathcal{C}(L,m)$ if and
only if $\partial U$ and $\partial\Sigma$ have partitions
\begin{equation}
\partial U=\alpha_{1}+\alpha_{2}+\cdots+\alpha_{m}, \label{part}%
\end{equation}%
\begin{equation}
\partial\Sigma=c_{1}+c_{2}+\cdots+c_{m}, \label{part1}%
\end{equation}
such that $c_{j}=\left(  f,\alpha_{j}\right)  $ is an \emph{SCC} arc on $S$
for $j=1,\dots,m$ (it is permitted that some $c_{j}$ may be a whole circle).
In this case, the partitions (\ref{part}) and (\ref{part1}) are both called
$\mathcal{C}\left(  L,m\right)  $-partition of $\partial\Sigma.$

(b) $\mathcal{C}^{\ast}(L,m)$ denotes the subspace of $\mathcal{C}%
(L,m)\mathcal{\ }$such that $\Sigma=\left(  f,\overline{U}\right)
\in\mathcal{C}^{\ast}(L,m)$ if and only if $\partial U$ and $\partial\Sigma$
have partitions (\ref{part}) and (\ref{part1}) such that $c_{j}=\left(
f,\alpha_{j}\right)  $ is an \emph{SCC} arc on $S$ and $f$ has no branch point
in $\alpha_{j}^{\circ}\cap f^{-1}(E_{q}),$ for every $j=1,\dots,m$. In this
case, the partitions (\ref{part}) and (\ref{part1}) are both called
$\mathcal{C}^{\ast}\left(  L,m\right)  $-partition of $\partial\Sigma$.

(c) $\mathcal{F}(L,m)$ denotes the subspace of $\mathcal{C}^{\ast}(L,m)$ such
that $\Sigma=\left(  f,\overline{U}\right)  \in\mathcal{F}(L,m)\ $if and only
if $\partial U$ and $\partial\Sigma$ have partitions (\ref{part}) and
(\ref{part1}), and that the following (i)---(iii) hold.

(i) Each $c_{j}=\left(  f,\alpha_{j}\right)  $ is an SCC arc.

(ii) For each $j=1,2,\dots,m,$ $f$ has no branch point in $\alpha_{j}^{\circ
}.$

(iii) For each $j=1,2,\dots,m,$ $f$ restricted to a neighborhood $D_{j}$ of
$\alpha_{j}^{\circ}$ in $\overline{\Delta}$ is a homeomorphism onto a one side
neighborhood $T_{j}\cup c_{j}^{\circ}$ of $c_{j}^{\circ},$ where $T_{j}$ is a
Jordan domain enclosed by $c_{j}$ and another circular arc $c_{j}^{\prime}$,
say, $\partial T_{j}=c_{j}-c_{j}^{\prime}$ and $c_{j}^{\prime}$ is a circular arc.

\label{Frpara}The partitions (\ref{part}) and (\ref{part1}) are both called
$\mathcal{F}(L,m)$-partitions of $\partial\Sigma$ if they satisfy (i)--(iii)
in addition.

(d) \label{FrFLMr}$\mathcal{F}_{r},$ as defined in (\ref{Fr}), is the subspace
of $\mathcal{F}$ such that $\Sigma=\left(  f,\overline{U}\right)
\in\mathcal{F}_{r}$ if and only if $f$ has no branch point in $\overline
{U}\backslash f^{-1}(E_{q}),$ and define
\[
\mathcal{F}_{r}(L)=\mathcal{F}_{r}\cap\mathcal{F}(L),
\]%
\[
\mathcal{F}_{r}(L,m)=\mathcal{F}_{r}\cap\mathcal{F}(L,m).
\]

\end{definition}

\begin{remark}
\label{circu-b}Note that $\mathcal{C}^{\ast}\left(  L,m\right)  $-partitons
and $\mathcal{F}\left(  L,m\right)  $-partitions envolve the inerior
information of the corresponding surfaces. But $\mathcal{C}\left(  L,m\right)
$-partitions does not involve the interior information of the surface, and
thus $\mathcal{C}\left(  L,m\right)  $-partitions can be defined for closed
curves which may not be bondary curves of surfaces.
\end{remark}

\begin{remark}
\label{spaces}The conditions (ii) and (iii) in the definition are equivalent
by Lemma \ref{int-arg} (also see Lemma \ref{int-arg1} later in this section),
we list (iii) just to emphasize. On the other hand it is clear that%
\[
\mathcal{C}\left(  L,m\right)  \supset\mathcal{C}^{\ast}\left(  L,m\right)
\supset\mathcal{F}\left(  L,m\right)  \supset\mathcal{F}_{r}\left(
L,m\right)  ,
\]%
\begin{align*}
\mathcal{C}\left(  L,m+1\right)   &  \supset\mathcal{C}\left(  L,m\right)  ,\\
\mathcal{C}^{\ast}\left(  L,m+1\right)   &  \supset\mathcal{C}^{\ast}\left(
L,m\right)  ,\\
\mathcal{F}\left(  L,m+1\right)   &  \supset\mathcal{F}\left(  L,m\right)  ,\\
\mathcal{F}^{\ast}\left(  L,m+1\right)   &  \supset\mathcal{F}^{\ast}\left(
L,m\right)  ,
\end{align*}
and all inclusion relationship are strict. On the hand, we have%
\[
\cup_{m=1}^{\infty}\mathcal{C}\left(  L,m\right)  =\cup_{m=1}^{\infty
}\mathcal{C}^{\ast}\left(  L,m\right)  =\cup_{m=1}^{\infty}\mathcal{F}\left(
L,m\right)  =\mathcal{F}\left(  L\right)  ,
\]
since for every $\Sigma=\left(  f,\overline{U}\right)  \in\mathcal{F}\left(
L\right)  ,$ $\partial\Sigma$ has an $\mathcal{F}(L,m)$-partition for some
$m\in\mathbb{N}.$
\end{remark}

\begin{definition}
\label{bylength}A curve $\left(  f,\partial\Delta\right)  $ is called
parametrized by length, if for each $\theta\in\lbrack0,2\pi]$ and the arc
$\Theta_{\theta}=\{e^{\sqrt{-1}t}:t\in\lbrack0,\theta]\}\subset\partial
\Delta,$
\[
L(f,\Theta_{\theta})=\frac{\theta}{2\pi}L(f,\partial\Delta).
\]

\end{definition}

\begin{remark}
\label{C(L,m)p}When a closed curve $\Gamma=\left(  \varphi,\partial
\Delta\right)  $ is parametrized by length and has $\mathcal{C}(L,m)$%
-partitions (\ref{part}) and (\ref{part1}), the partitions are uniquely
determined by the initial point of $\alpha_{1}$ and the length of
$c_{j},j=1,2,...,m,$ since we have%
\[
L(\alpha_{1}):L(\alpha_{2}):\dots:L(\alpha_{m})=L(c_{1}):L(c_{2}%
):...:L(c_{m}).
\]

\end{remark}

\begin{remark}
\label{parameter}Assume $\Sigma=\left(  f,\overline{\Delta}\right)
\in\mathcal{F}\left(  L,m\right)  $ and $\partial\Sigma$ is parametrized by
length. Then for any $\mathcal{F}\left(  L,m\right)  $-partition
\begin{equation}
\partial\Delta=\alpha_{1}\left(  a_{1},a_{2}\right)  +\cdots+\alpha_{m}\left(
a_{m},a_{1}\right)  \label{fmp1}%
\end{equation}
of $\partial\Sigma$ with $a_{1}=1$ and the corresponding $\mathcal{F}\left(
L,m\right)  $-partition
\begin{equation}
\partial\Sigma=c_{1}\left(  q_{1},q_{2}\right)  +\cdots+c_{m}\left(
q_{m},q_{1}\right)  , \label{fmp2}%
\end{equation}
we have $a_{j}=e^{\sqrt{-1}\theta_{j}}\ $for some $\theta_{j}$ such that
\[
0=\theta_{1}<\theta_{2}<\dots<\theta_{m}<\theta_{m+1}=2\pi,
\]
and%
\[
\theta_{j}=\frac{L(f,\Theta_{\theta_{j}})}{L\left(  \partial\Sigma\right)
}2\pi,j=1,2,\dots,m+1.
\]
Then for the homeomorphism $\varphi:[1,m+1]\rightarrow\lbrack0,2\pi]$ which
maps each interval $[j,j+1]$ linearly onto $[\theta_{j},\theta_{j+1}],$ we
have $a_{j}=e^{\sqrt{-1}\varphi\left(  j\right)  },$ and for $x\in(j,j+1),$
$a_{x}=e^{\sqrt{-1}\varphi\left(  x\right)  }$ is a point in $\alpha_{j}$, and
$q_{x}=f(a_{x})$ is a point in $c_{j}$. For example, $a_{1+\frac{1}{2}}$ is
the middle point of $\alpha_{1}$ and $q_{1+\frac{1}{2}}$ is the middle point
of $c_{1},$ $q_{m+\frac{1}{2}}$ is the middle point of $c_{m}$, while
$q_{m+1}=q_{1}.$
\end{remark}

\begin{lemma}
\label{tangent}\label{tangent-1}Let $\Sigma=\left(  f,\overline{\Delta
}\right)  \in\mathcal{F}\left(  L,m\right)  ,$ $\{a,b\}\subset\partial\Delta,$
$\alpha$ be an arc of $\partial\Delta$ from $a$ to $b,$ oriented by
$\partial\Delta,$ $\beta$ be a simple arc in $\overline{\Delta}$ from $a$ to
$b,$ and assume that the following (a)--(c) hold.

(a) (\ref{fmp1}) and (\ref{fmp2}) are $\mathcal{F}\left(  L,m\right)
$-partitions of $\partial\Sigma\ $with $c_{j}=\left(  f,\alpha_{j}\right)
,j=1,\dots,m.$

(b) $\alpha\cap\beta$ contains a connected component $I=I(a^{\prime}%
,b^{\prime}),$ oriented by $\partial\Delta,$ with $b^{\prime}\neq a,b,$ say,
$b^{\prime}\in\alpha^{\circ}\cap\beta^{\circ}.$

(c) $\left(  f,-\beta\right)  $ is an SCC arc.

Then the following hold.

(i) $b^{\prime}\in\{a_{j}\}_{j=1}^{m}.$

(ii) If $a^{\prime}\neq a,$ then $a^{\prime}\in\{a_{j}\}_{j=1}^{m},$ and $I$
is either a point of $\{a_{j}\}_{j=1}^{m},$ or an arc of $\partial\Delta$
which can be written as
\[
I=\alpha_{j_{1}}\left(  a_{j_{1}},a_{j_{1}+1}\right)  +\alpha_{j_{1}+1}\left(
a_{j_{1}+1},a_{j_{1}+2}\right)  +\cdots+\alpha_{j_{1}+k}\left(  a_{j_{1}%
+k},a_{j_{1}+k+1}\right)
\]
for some $j_{1}\leq m$ and some $0\leq k<m$ (with $a_{j}=a_{j-m}$ for each $j$
with $j>m$).

(iii) If $b^{\prime}$ is not a branch point of $f,$ then $\left(
f,\partial\Delta\right)  $ is not convex at $b^{\prime}$.

(iv) If, in addition to (a)--(c), $\left(  f,-\beta\right)  $ is strictly
convex, then $a^{\prime}=b^{\prime}\in\{a_{j}\}_{j=1}^{m},$ say, $I$ is a
point in $\{a_{j}\}_{j=1}^{m},$ and moreover $\beta^{\circ}\cap\partial
\Delta\subset\{a_{j}\}_{j=1}^{m}$.
\end{lemma}

Let $\gamma$ be a simple arc on $S$. Then for any $x\in\gamma^{\circ}$ there
exist a positive number $\delta_{x,\gamma}\ $and a subarc $\gamma^{\prime}$ of
$\gamma$ in $\overline{D(x,\delta_{x,\gamma})}$ such that $\partial
\gamma^{\prime}\subset\partial D(x,\delta_{x,\gamma})$ and $x\in\gamma
^{\prime\circ}\subset D(x,\delta_{x,\gamma})$. Hence $\gamma^{\prime}$ divides
$D(x,\delta_{x,\gamma})$ into two Jordan domains $D^{l}\left(  x,\delta
_{x,\gamma}\right)  $ and $D^{r}\left(  x,\delta_{x,\gamma}\right)  ,$ on the
left and right hand side of $\gamma^{\prime}$ (by convention, $\gamma^{\prime
}$ inherits the orientation of $\gamma$). From this observation, we can
introduce the following definition.

\begin{definition}
\label{LR}Let $\gamma_{1}$ and $\gamma_{2}$ be two arcs on $S$ with
$\gamma_{1}^{\circ}\cap\gamma_{2}^{\circ}\neq\emptyset\ $and assume
$\gamma_{2}$ is simple. We say that $\gamma_{1}$ is on the left hand side of
$\gamma_{2},$ if for each $x\in\gamma_{1}^{\circ}\cap\gamma_{2}^{\circ},$ $x$
has a neighborhood in $\gamma_{1}$ contained in $\overline{D^{l}\left(
x,\delta_{x,\gamma_{2}}\right)  }\backslash\partial D\left(  x,\delta
_{x,\gamma_{2}}\right)  .$ When we replace $D^{l}\left(  x,\delta
_{x,\gamma_{2}}\right)  $ with $D^{r}\left(  x,\delta_{x,\gamma_{2}}\right)  $
for every $x\in\gamma_{1}^{\circ}\cap\gamma_{2}^{\circ},$ we obtain the
definition of that $\gamma_{1}$ is on the right hand side of $\gamma_{2}.$
\end{definition}

By this definition, $\gamma_{2}$ is on the left hand side of itself, and on
the right hand side as well.

\begin{definition}
\label{arcarc}Let $\gamma_{1}$ and $\gamma_{2}$ be two simple arcs on $S$ and
assume $\gamma_{1}^{\circ}\cap\gamma_{2}^{\circ}\neq\emptyset$. When
$\gamma_{1}$ is on the left (right) hand side of $\gamma_{2},$ and $\gamma
_{2}$ is on the right (left) hand side of $\gamma_{1},$ we say that the
direction of $\gamma_{1}$ is determined by $\gamma_{2}.$
\end{definition}

The proof of the previous lemma is similar as the proof of Lemma 5.4 in
\cite{Zh1}, which is based on an observation stated after that lemma in
\cite{Zh1}. The proof here is based on a similar observation but is a little
more general:

\begin{Observation}
\label{ob1}Let $\gamma_{1}$ and $\gamma_{2}$ be two simple and convex arcs on
$S$ with $\gamma_{1}^{\circ}\cap\gamma_{2}^{\circ}\neq\emptyset.$ Assume that
$\gamma_{1}$ is on the left hand side of $\gamma_{2}$ and the direction of
$-\gamma_{1}$ is determined by $\gamma_{2}$ in the sense of Definition
\ref{arcarc}. Then $\gamma_{1}\cap\gamma_{2}$ is a simple line segment on $S,
$ and each of the two endpoints of $\gamma_{1}\cap\gamma_{2}$ is an endpoint
of $\gamma_{1}$ or $\gamma_{2}$, say, $\gamma_{1}^{\circ}\cap\gamma_{2}%
^{\circ}$ has no compact connected component.
\end{Observation}

\begin{proof}
[Proof of Lemma \ref{tangent}]\label{2023-4-5-19:08}We may write
\[
I=I\left(  a^{\prime},b^{\prime}\right)  =\alpha\left(  a^{\prime},b^{\prime
}\right)  =\beta\left(  a^{\prime},b^{\prime}\right)  ,
\]
which, by convention, means that $I$ is from $a^{\prime}$ to $b^{\prime}$ and
$I$ is a subarc of $\alpha$ (and $\beta$) whose direction is determined by
$\alpha$ (and $\beta$). Since $\Sigma\in\mathcal{F}\left(  L,m\right)  $ and
$\left(  f,-\beta\right)  $ is simple and circular, $\alpha\cap\beta$ is
consisted of a finite number of connected components. Thus $\alpha$ and
$\beta$ have subarcs\footnote{By convention, subarcs inherits the oriention.}
$\alpha_{b^{\prime}}=\alpha\left(  a^{\prime\prime},b^{\prime\prime}\right)  $
and $\beta_{b^{\prime}}=\beta(A^{\prime\prime},B^{\prime\prime})$ which are
neighborhoods of $b^{\prime}$ in $\alpha$ and in $\beta,$ respectively, such that

\begin{condition}
\label{condi1}(d) When $b^{\prime}\neq a^{\prime},$ we have $A^{\prime\prime
}=a^{\prime\prime}\in I^{\circ},$%
\[
\alpha_{b^{\prime}}\cap\beta_{b^{\prime}}=\alpha_{b^{\prime}}\cap
I=\beta_{b^{\prime}}\cap I=\alpha\left(  a^{\prime\prime},b^{\prime}\right)
=\beta\left(  a^{\prime\prime},b^{\prime}\right)  ;
\]
and when $b^{\prime}=a^{\prime},$ we have
\[
\alpha_{b^{\prime}}\cap\beta_{b^{\prime}}=\{b^{\prime}\}\subset\alpha
_{b^{\prime}}^{\circ}\cap\beta_{b^{\prime}}^{\circ}.
\]

(e) $\alpha_{b^{\prime}}\backslash\{b^{\prime}\}$ and $\beta_{b^{\prime}%
}\backslash\{b^{\prime}\}$ contain neither point of $\{a_{j}\}_{j=1}^{m}\ $nor
branch point of $f.$

(f) $\left(  f,-\beta_{b^{\prime}}\right)  $ is an SCC arc.

(g) If $b^{\prime}$ is not a branch point of $f$ and $\left(  f,\partial
\Delta\right)  $ is not folded at $b^{\prime},$ then $f$ is an OPH in a
neighborhood\footnote{When $b^{\prime}$ satisfies the assumption of (g), $f$
is a homeomorphism in a neighborhood of $b^{\prime}$ in $\overline{\Delta},$
and so the conclusion holds when $\alpha_{b^{\prime}}$ and $\beta_{b^{\prime}%
}$ is contained in that neighborhood.} of $\alpha_{b^{\prime}}\cup
\beta_{b^{\prime}}$ in $\overline{\Delta},$ $\gamma_{1}=(f,\alpha_{b^{\prime}%
})$ is a simple arc on the left hand side of $\gamma_{2}=\left(
f,-\beta_{b^{\prime}}\right)  ,$ and the direction of $\gamma_{1}$ is
determined by $-\gamma_{2}.$
\end{condition}

To prove (i), assume its opposite $b^{\prime}\notin\{a_{j}\}_{j=1}^{m}.$ Then
by (e) $\gamma_{1}=(f,\alpha_{b^{\prime}})$ is contained in some $c_{j}$ of
the partition (\ref{fmp2}) and thus is convex. Hence, (f), (g) and Observation
\ref{ob1} imply that $f(b^{\prime})$ is in the interior $\left(  \gamma
_{1}\cap\gamma_{2}\right)  ^{\circ}$, and thus $b^{\prime}$ is an interior
point of $\alpha_{b^{\prime}}\cap\beta_{b^{\prime}},$ which contradicts (d).
(i) is proved and the proof of (ii) is the same.

To prove (iii), assume that $b^{\prime}$ is not a branch point of $f\ $and the
conclusion fails, say, $\left(  f,\partial\Delta\right)  $ is convex at $b.$
Then $\left(  f,\partial\Delta\right)  $ is not folded at $b^{\prime}$ and so
all conclusions of (g) hold, and moreover, $\gamma_{1}=(f,\alpha_{b^{\prime}%
})$ is convex by (d) and (e). Thus (f), (g) and Observation \ref{ob1} again
imply that $b^{\prime}$ is an interior point of $\alpha_{b^{\prime}}\cap
\beta_{b^{\prime}}$, contradicting (d), as in the proof of (i). This proves (iii).

Assume $\left(  f,-\beta\right)  $ is strictly convex, and $I$ is not a point,
say, $a^{\prime}\neq b^{\prime}.$ Then $I$ has a subarc $I^{\prime}$ such that
$I^{\prime}\subset\left(  \partial\Delta\right)  \backslash\{a_{j}\}_{j=1}%
^{m},$ and then by (c) and definition of $\mathcal{F}\left(  L,m\right)  ,$
both $\left(  f,I^{\prime}\right)  $ and $\left(  f,-I^{\prime}\right)
\subset\left(  f,-\beta\right)  $ are convex arcs, which implies that $\left(
f,I^{\prime}\right)  $ is a line segment and thus $\left(  f,-\beta\right)  $
cannot be strictly convex. This proves (iv).
\end{proof}

The following result and the method of proof will be used a lot of times.

\begin{lemma}
\label{split}\label{ratio}Let $\left\{  L_{j}\right\}  _{j=1}^{N}$ and
$\{R_{j}\}_{j=1}^{N}$ be sets of positive numbers and assume $\frac{R_{1}%
}{L_{1}}\geq\frac{R_{j}}{L_{j}},j=2,\dots,N.$ Then%
\[
\frac{R_{1}}{L_{1}}\geq\frac{\sum_{j=1}^{N}R_{j}}{\sum_{j=1}^{N}L_{j}},
\]
with equality holding if and only if
\[
R_{1}/L_{1}=R_{2}/L_{2}=\dots=R_{N}/L_{N}.
\]

\end{lemma}

\begin{proof}
It follows from%
\[
\sum_{j=1}^{N}R_{j}=\sum_{j=1}^{N}\frac{R_{j}}{L_{j}}L_{j}\leq\sum_{j=1}%
^{N}\frac{R_{1}}{L_{1}}L_{j}=\frac{R_{1}}{L_{1}}\sum_{j=1}^{N}L_{j}.
\]

\end{proof}

The following lemma is very useful (recall definitions (\ref{Ahero}) and
(\ref{DH}) of $R(\Sigma)$ and $H(\Sigma)$).

\begin{lemma}
\label{dec}Let $\Sigma_{j},j=0,1,\dots,n,$ be surfaces in $\mathbf{F,}$ and
let $\varepsilon_{0}\ $be a positive number such that%
\[
R(\Sigma_{0})-\varepsilon_{0}\leq\sum_{j=1}^{n}R(\Sigma_{j})\mathrm{\ and\ }%
L(\partial\Sigma_{0})+\varepsilon_{0}\geq\sum_{j=1}^{n}L(\partial\Sigma_{j}).
\]
Then
\begin{equation}
\frac{H(\Sigma_{0})-\varepsilon_{0}/L(\partial\Sigma_{0})}{1+\varepsilon
_{0}/L(\partial\Sigma_{0})}\leq\max_{1\leq j\leq n}H(\Sigma_{j}). \label{ma}%
\end{equation}

\end{lemma}

\begin{proof}
This follows from (\ref{Ahero}), (\ref{DH}) and Lemma \ref{ratio} directly.
\end{proof}

\begin{definition}
\label{hs1}\label{hs}We define%
\[
H_{L,m}=H_{L,m}\left(  E_{q}\right)  =\sup_{\Sigma\in\mathcal{F}(L,m)}%
H(\Sigma).
\]

\end{definition}

Then by Definition \ref{HL} we have%
\begin{equation}
H_{L}=\lim_{n\rightarrow+\infty}H_{L,m}. \label{HH}%
\end{equation}

\begin{lemma}
\label{increase}If $L\leq2\delta_{E_{q}},$ then%
\begin{equation}
H_{L}=H_{L,m}=H_{L,1}=\frac{\left(  q-2\right)  A(T)}{L}=\left(  q-2\right)
\frac{2\pi-\sqrt{4\pi^{2}-L^{2}}}{L}, \label{HLM}%
\end{equation}
$H_{L}$ increases strictly as a function of $L\in(0,2\delta_{E_{q}}]$ for any
given positive integer $m,$ where $T$ is a disk with perimeter $L\ $and
diameter less than $\pi.$
\end{lemma}

\begin{proof}
By Theorem \ref{l<2dt}, for any $L\in(0,2\delta_{E_{q}}],$ we have
(\ref{HLM}). Thus, by Lemma \ref{TD}, $H_{L}=H_{L,m}$ strictly increase on
$(0,2\delta_{E_{q}}]$.
\end{proof}

\begin{lemma}
\label{H>0}For any $L>0,$ let $\mathcal{T}_{L}$ be the set of all closed
convex\footnote{A convex disk on $S$ is contained in some hemisphere on $S,$
and vice versa.} disks $T$ with $T^{\circ}\subset S\backslash E_{q}$ and
$L(\partial T)\leq L, $ and let $L_{0}=\sup_{T\in\mathcal{T}_{L}}L(\partial
T).$ Then for any positive integer $m\ $and any disk $T_{L_{0}}$ in
$\mathcal{T}_{L}\ $with $L(T_{0})=L_{0}$
\[
H_{L}\geq H_{L,m}\geq H_{L,1}\geq\sup_{T\in\mathcal{T}_{L}}H(\overline
{T})=H(T_{L_{0}})\geq H(T_{\min\{L,2\delta_{E_{q}}\}})>0.
\]

\end{lemma}

\begin{proof}
This follows from the relation $\mathcal{T}_{L}\subset\mathcal{F}\left(
L,1\right)  \subset\mathcal{F}\left(  L,m\right)  \subset\mathcal{F}\left(
L\right)  $, $T_{L_{0}}\in\mathcal{T}_{L}$ and Lemma \ref{TD}.
\end{proof}

\begin{lemma}
\label{dl}Let $L\in\mathcal{L}.$ Then there exists a positive number
$\delta_{L}$ such that for each $L^{\prime}\in(L-\delta_{L},L+\delta_{L}),$%
\[
H_{L}-\frac{\pi}{2L}<H_{L^{\prime}}<H_{L}+\frac{\pi}{2L},
\]

\end{lemma}

\begin{proof}
This follows from Definition \ref{L}.
\end{proof}

In \cite{S-Z} it is proved that for any surface $\Sigma$ in $\mathbf{F,}$
there exists a surface $\Sigma^{\prime}$ in $\mathbf{F=}\left(  f,\overline
{\Delta}\right)  $ with piecewise analytic boundary, such that $f$ has no
branch point in $\Delta\backslash f^{-1}(E_{q}),$ $A(\Sigma^{\prime})\geq
A(\Sigma),L(\partial\Sigma)\geq L(\partial\Sigma^{\prime})$, and moreover,
$\partial\Sigma^{\prime}$ is consisted of subarcs of $\partial\Sigma.$ In
\cite{CLZ1}, using the similar method and more analysis, the authors proved
the following theorem, which is a key step for proving the first main theorem.

\begin{theorem}
\label{re}Let $L\in\mathcal{L}$, $m$ be a positive integer, $\Sigma=\left(
f,\overline{\Delta}\right)  \in\mathcal{C}^{\ast}(L,m),$
\begin{equation}
\partial\Delta=\alpha_{1}+\cdots+\alpha_{m}, \label{001}%
\end{equation}
is a $\mathcal{C}^{\ast}\left(  L,m\right)  $-partition of $\partial\Sigma$
and
\begin{equation}
\partial\Sigma=c_{1}+\cdots+c_{m} \label{002}%
\end{equation}
is the corresponding $\mathcal{C}^{\ast}\left(  L,m\right)  $-partition$,$ and
assume that
\begin{equation}
H(\Sigma)>H_{L}-\frac{\pi}{2L(\partial\Sigma)}. \label{210615}%
\end{equation}
Then there exists a surface $\Sigma^{\prime}=\left(  f^{\prime},\overline
{\Delta}\right)  $ such that

(i) $\Sigma^{\prime}\in\mathcal{F}_{r}(L,m)$.

(ii) $H(\Sigma^{\prime})\geq H(\Sigma)$ and $L(\partial\Sigma^{\prime})\leq
L(\partial\Sigma).$ Moreover, at least one of the inequalities is strict if
$\Sigma\notin\mathcal{F}_{r}(L,m)$.

(iii) When $L(\partial\Sigma^{\prime})=L(\partial\Sigma)$ we have
$\partial\Sigma^{\prime}=\partial\Sigma$ and (\ref{001}) and (\ref{002}) are
$\mathcal{F}\left(  L,m\right)  $-partitions of $\partial\Sigma^{\prime}.$
\end{theorem}

The following Theorem is proved in \cite{CLZ2}, which is also a key step for
proving the first main theorem.

\begin{theorem}
\label{sim}There exists an integer $d^{\ast}=d^{\ast}(m,E_{q})$ depending only
on $m$ and $E_{q}$ such that for any $\Sigma=\left(  f,\overline{\Delta
}\right)  \in\mathcal{F}_{r}\left(  L,m\right)  $, there exists a surface
$\Sigma_{1}=\left(  f_{1},\overline{\Delta}\right)  $ in $\mathcal{F}%
_{r}\left(  L,m\right)  $ such that%
\[
d_{\max}\left(  \Sigma_{1}\right)  =d_{\max}\left(  f_{1}\right)  =\max_{w\in
S\backslash\partial\Sigma_{1}}\#f_{1}^{-1}(w)\leq d^{\ast},
\]
\[
\partial\Sigma_{1}=\partial\Sigma,
\]
and
\[
H(\Sigma_{1})=H(\Sigma).
\]

\end{theorem}

\begin{definition}
\label{FR'}$\mathcal{F}_{r}^{\prime}(L,m)$ is defined to be the subspace of
$\mathcal{F}_{r}(L,m)$ such that $\Sigma=\left(  f,\overline{U}\right)
\in\mathcal{F}_{r}^{\prime}(L,m)\ $iff $d_{\max}(\Sigma)\leq d^{\ast}=d^{\ast
}(m,q),$ where $d^{\ast}$ is the integer determined in Theorem \ref{sim}%
.\label{2023-4-5-19:46}
\end{definition}

\begin{corollary}
\label{FF'}\label{sim1}\label{FinF}Let $L\in\mathcal{L}$. Then for
sufficiently large $m,$ we have
\[
H_{L,m}=\sup_{\Sigma\in\mathcal{C}^{\ast}(L,m)}H(\Sigma)=\sup_{\Sigma
\in\mathcal{F}(L,m)}H(\Sigma)=\sup_{\Sigma\in\mathcal{F}_{r}(L,m)}%
H(\Sigma)=\sup_{\Sigma\in\mathcal{F}_{r}^{\prime}(L,m)}H(\Sigma).
\]

\end{corollary}

\begin{proof}
It suffices to show that, for sufficiently large $m$,
\begin{equation}
\sup_{\Sigma\in\mathcal{C}^{\ast}(L,m)}H(\Sigma)\leq\sup_{\Sigma\in
\mathcal{F}_{r}^{\prime}(L,m)}H(\Sigma). \label{C*F'}%
\end{equation}
By Remark \ref{spaces},
\[
H_{L}=\lim_{m\rightarrow\infty}\sup_{\Sigma\in\mathcal{C}^{\ast}(L,m)}%
H(\Sigma)>H_{L}-\frac{\pi}{2L}.
\]

Let $m$ be any large enough integer in $\mathbb{N}$ such that
\[
\sup_{\Sigma\in\mathcal{C}^{\ast}(L,m)}H(\Sigma)>H_{L}-\frac{\pi}{2L}.
\]
Then there exists a sequence $\Sigma_{n}$ in $\mathcal{C}^{\ast}(L,m)$ such
that
\[
\lim_{n\rightarrow\infty}H(\Sigma_{n})=\sup_{\Sigma\in\mathcal{C}^{\ast}%
(L,m)}H(\Sigma)
\]
and
\[
H(\Sigma_{n})>H_{L}-\frac{\pi}{2L}>H_{L}-\frac{\pi}{2L\left(  \partial
\Sigma_{n}\right)  }.
\]
Thus by Theorems \ref{re} and \ref{sim}, there exists a sequence $\Sigma
_{n}^{\prime}\in\mathcal{F}_{r}^{\prime}(L,m)$ such that $H(\Sigma_{n}%
^{\prime})\geq H(\Sigma_{n}).$ Thus (\ref{C*F'}) holds.
\end{proof}

\begin{lemma}
\label{LSZ}Let $\Sigma=\left(  f,\overline{\Delta}\right)  \in\mathbf{F}$ and
assume either (1) $f^{-1}(E_{q}\cap\Delta)\neq\emptyset,$ or (2)
$\partial\Sigma$ contains a simple arc with distinct end points in $E_{q}.$
Then
\begin{equation}
\frac{R(\Sigma)+4\pi}{L(\partial\Sigma)}\leq H_{0}, \label{H<4p}%
\end{equation}
where $H_{0}$ is defined by (\ref{H_0}).
\end{lemma}

This is Theorem 1.7 in \cite{L-S-Z}. In fact when (1) holds, we may assume
$0\in f^{-1}(E_{q})$ and let $\Sigma_{n}=\left(  f_{n},\overline{\Delta
}\right)  $ be the surface in $\mathbf{F}$ with $f_{n}\left(  z\right)
=f\left(  z^{n}\right)  ,z\in\overline{\Delta}.$ Then we have $R(\Sigma
_{n})=\left(  q-2\right)  nA(\Sigma)-4\pi\left(  n\overline{n}\left(
\Sigma\right)  -\left(  n-1\right)  \right)  =nR(\Sigma)+4\pi\left(
n-1\right)  ,$ and $L(\partial\Sigma_{n})=nL(\partial\Sigma),$ and thus
\[
H_{0}\geq\frac{R(\Sigma_{n})}{L(\partial\Sigma_{n})}=\frac{R(\Sigma
)+\frac{4\pi\left(  n-1\right)  }{n}}{L(\partial\Sigma)}\rightarrow
\frac{R(\Sigma)+4\pi}{L(\partial\Sigma)}\mathrm{\ as\ }n\rightarrow\infty,
\]
which implies (\ref{H<4p}). When (2) holds, (\ref{H<4p}) follows from the
following Lemma which is also proved in \cite{L-S-Z}.

\begin{lemma}
\label{LSZ1}Let $\Sigma=\left(  f,\overline{\Delta}\right)  \in\mathbf{F}$ and
assume that $\partial\Sigma$ contains a simple arc $\gamma$ with distinct end
points in $E_{q},$ say, (2) of the previous lemma holds. Then there exists a
surface $\Sigma_{1}=\left(  f_{1},\overline{\Delta}\right)  \in\mathbf{F}$
such that $L(\partial\Sigma)=L(\partial\Sigma_{1})$, $H(\Sigma)=H(\Sigma_{1})$
and $f_{1}^{-1}\left(  E_{q}\right)  \cap\Delta\neq\emptyset.$
\end{lemma}

In fact $\Sigma_{1}$ can be obtained by sew\label{sew26} $\Sigma$ and the
surface whose interior is $S\backslash\gamma$ and boundary is $\gamma-\gamma.
$ Then we have $A(\Sigma_{1})=A(\Sigma)+4\pi,$ $\overline{n}\left(  \Sigma
_{1}\right)  =\overline{n}\left(  \Sigma\right)  +q-2$ and $\partial
\Sigma=\partial\Sigma_{1}$. Therefore $H(\Sigma)=H(\Sigma_{1})$ and
$f_{1}^{-1}\left(  E_{q}\right)  \cap\Delta\neq\emptyset,$ and then
(\ref{H<4p}) holds by the discussion of (1) of the previous lemma.

The above discussion about Lemmas \ref{LSZ} and \ref{LSZ1} implies the following

\begin{corollary}
\label{LSZ2}If $\Sigma\in\mathcal{F}$ satisfies (1) or (2) of Lemma \ref{LSZ},
then there exists a sequence $\Sigma_{n}$ in $\mathcal{F}$ such $H(\Sigma
_{n})=\frac{R(\Sigma_{n})}{L(\partial\Sigma_{n})}\rightarrow\frac
{R(\Sigma)+4\pi}{L(\partial\Sigma)}$ as $n\rightarrow\infty.$
\end{corollary}

\begin{lemma}
\label{int-arg1}\label{7:33-20230406}Let $\Sigma=\left(  f,\overline{\Delta
}\right)  \in\mathcal{F}_{r}(L)$, let $\alpha=\alpha\left(  b_{1}%
,b_{2}\right)  $ be an arc of $\partial\Delta$ such that $c=c\left(
p_{1},p_{2}\right)  =\left(  f,\alpha\right)  \ $is a simple arc (it is
possible that $p_{1}=p_{2}$ even if $b_{1}\neq b_{2}$). If $\alpha^{\circ}$
contains no branch point of $f,$ then there exists a \emph{closed} domain $T$
on $S$ such that

(i) $\partial T=c-c^{\prime},$ where $c^{\prime}$ is a simple arc from $p_{1}
$ to $p_{2}$ which is consisted of a finite number of line segments on $S$
such that $c\cap c^{\prime}=\{p_{1},p_{2}\}.$

(ii) The interior angles of $T$ at $p_{1}$ and $p_{2}$ are positive.

(iii) $f^{-1}$ has a univalent branch $g$ defined on $T\backslash\{p_{1}%
,p_{2}\}$ with $g(c)=\alpha.$

(iv) If $p_{1}\neq p_{2},$ then $T$ is a closed Jordan domain and the
univalent branch $g$ of $f^{-1}$ can be extended to a homeomorphism defined on
the closed domain $T.$
\end{lemma}

\begin{proof}
Since $f$ has no branch point on $\alpha^{\circ}$ and $c$ is simple, the
results follows from Lemmas \ref{int-arg} and \ref{cov-1}.
\end{proof}

\begin{remark}
\label{Riemann}\label{7:43 -20230406}Let $\Sigma=\left(  f,\overline{\Delta
}\right)  \in\mathcal{F}.$

(i) $\Sigma=\left(  f,\overline{\Delta}\right)  $ can be understood as a
branched Riemann surface with boundary $\partial\Sigma=\left(  f,\partial
\Delta\right)  $, such that every point of $\Sigma$ is in fact a pair
$(f,p)=\left(  f(p),p\right)  $ with $p\in\overline{\Delta}:$ $\Sigma$ can be
regarded as a union of a finite number of disks $\left(  x_{j},U_{j}\right)  $
of $\Sigma\ $defined in Definition \ref{nod}. Then $\left(  f,U_{j}\right)  $
plays the role of chards for Riemann surfaces when $\Sigma$ is regarded as the
set of pairs $\left(  f,p\right)  =\left(  f(p),p\right)  ,p\in\overline
{\Delta}.$

(ii) Assume $a$ is a point in $\overline{\Delta}$ and $\overline{D}$ is a
closed domain in $\overline{\Delta}$ which is a neighborhood of $a$ in
$\overline{\Delta}$. If the restriction $f:\overline{D}\rightarrow T\ $with
$T=f(\overline{D})$ is a homeomorphism, we will call the subsurface $K=\left(
f,\overline{D}\right)  $ a simple closed domain of $\Sigma$ determined by $a$
(when $f(D)$ is given), and use the pair $\left(  T,a\right)  $ or $\left(
T,\left(  f,a\right)  \right)  \ $to denote this closed domain. It is clear
that $K$ and $\overline{D}$ are uniquely determined by $T$ and $a.$

(ii1) Then the term "$\left(  W,a\right)  $\emph{\ is a closed simple domain
of }$\Sigma$" means that "$W$\emph{\ is a closed domain} on $S$ \emph{and }%
$f$\emph{\ has a univalent branch }$g$\emph{\ defined on }$W$\emph{\ and
}$g(W)$ \emph{is a neighborhood of }$a$ in $\overline{\Delta}$".

(ii2) When $\overline{D}$ is a closed Jordan domain in $\overline{\Delta}$
such that $\left(  \partial D\right)  \cap\partial\Delta$ contains an arc
$\alpha$ of $\partial\Delta$ and $f:\overline{D}\rightarrow T=f(\overline{D})$
is a homeomorphism. Then $\overline{D},$ as an $f$-lift of $T,$ is uniquely
determined by the pair $(T,\left(  \partial D\right)  \cap\partial\Delta),$ or
$(T,\alpha),$ or $\left(  T,\left(  f,\alpha\right)  \right)  ,$ or $\left(
T,a\right)  ,$ where $a\ $is any interior point of $\alpha,$ say $a\in
\alpha^{\circ}.$ So we will write
\[
\left(  f,\overline{D}\right)  =\left(  T,\left(  \partial\Delta\right)
\cap\partial D\right)  )=\left(  T,\alpha\right)  =\left(  T,\left(
f,\alpha\right)  \right)  =\left(  T,a\right)  ,
\]
call $\left(  f,\left(  \partial\Delta\right)  \cap\partial D\right)  $ the
\emph{old} boundary of $\left(  f,\overline{D}\right)  ,$ and $\left(
f,\Delta\cap\partial D\right)  $ the \emph{new} \emph{boundary} of $\left(
f,\overline{D}\right)  .$

(iii) Assume $\Sigma,\alpha,c,c^{\prime},T,g$ satisfies all conditions of
Lemma \ref{int-arg1}. If $p_{1}\neq p_{2},$ then $g$ can be extended to $T$ so
that for $\overline{D}=g(T),$ $f(\overline{D})$ is the closed Jordan domain
$T\ $and $\left(  T,a\right)  $ is a closed simple Jordan domain of $\Sigma,$
where $a\in\alpha^{\circ}.$ Then $\left(  T,\alpha\right)  \ $and $\left(
T,c\right)  $ with $c=\left(  f,\alpha\right)  $ both denote $\left(
f,\overline{D}\right)  .$ If $p_{1}=p_{2},$ then $D$ is still a Jordan domain,
while $T$ is not. In this case we still use $\left(  T,\alpha\right)  ,$ or
$\left(  T,a\right)  ,$ where $a$ is an interior point of $\alpha,$ to denote
the surface $\left(  f,\overline{D}\right)  ,$ and call it a simple closed
domain as well. In fact, $T$ can be expressed as a union of Jordan curves,
each pair of which only intersect at $f(a)$.

(iv) Now assume $\Sigma=\left(  f,\overline{\Delta}\right)  \in\mathcal{F}%
\left(  L,m\right)  $ and let
\begin{align}
\partial\Sigma &  =c_{1}\left(  q_{1},q_{2}\right)  +\cdots+c_{m}\left(
q_{m},q_{1}\right) \label{LMP}\\
&  =\left(  f,\alpha_{1}\left(  a_{1},a_{2}\right)  \right)  +\cdots+\left(
f,\alpha_{m}\left(  a_{m},a_{1}\right)  \right)  ,\nonumber
\end{align}
be an $\mathcal{F}\left(  L,m\right)  $-partition of $\partial\Sigma.$ For
each $j,$ by definition of $\mathcal{F}\left(  L,m\right)  $ and
$\mathcal{F}\left(  L,m\right)  $-partitions, $f$ has no branch point in
$\alpha_{j}^{\circ}$ and $f$ is homeomorphism in a neighborhood of $\alpha
_{j}^{\circ}$ in $\overline{\Delta}.$ Then Lemma \ref{int-arg1} applies to
each $c_{j}$, and there exist a positive number $\theta>0$ and a closed domain
$T_{j,\theta}$ on $S$ enclosed by $c_{j}$ and $c_{j}^{\prime}$ such that:

(iv1) If $q_{j}\neq q_{j+1},$ $c_{j}^{\prime}$ is a circular arc from $q_{1}$
to $q_{2}$, $\partial T_{j,\theta}=c_{j}-c_{j}^{\prime}$ and
\[
\angle\left(  T_{j,\theta},q_{j}\right)  =\angle\left(  T_{j,\theta}%
,q_{j+1}\right)  =\theta\in(0,\min_{i=j+1}\angle\left(  \Sigma,a_{i}\right)
),
\]
and $\left(  T_{j,\theta},\alpha_{j}\right)  $ is a simple Jordan domain of
$\Sigma.$

(iv2) If $q_{j}=q_{j+1},$ $c_{j}^{\prime}$ is consisted of two convex
circulars arcs such that $\partial T_{j,\theta}=c_{j}-c_{j}^{\prime},$
$c_{j}^{\prime}$ is contained in the disk enclosed by $c_{j}$ and
$c_{j}^{\prime}\cap c_{j}=\{q_{j}\},$ the two interior angles of $T_{\theta}$
at $q_{j}$ are equal to $\theta,$ and there exists a Jordan domain $D$ in
$\Delta$ such that $\partial D=\alpha_{j}-\alpha_{j}^{\prime}$ with
$\alpha_{j}^{\prime\circ}\subset\Delta,$ and $f$ restricted to $\overline
{D}\backslash\{a_{j},a_{j+1}\}$ is a homeomorphism onto $T_{j,\theta
}\backslash\{q_{i}\}$. We then can use $\left(  T_{j,\theta},\alpha
_{j}\right)  $ or $\left(  T_{j,\theta},c_{j}\right)  ,$ in which
$c_{j}=\left(  f,\alpha_{j}\right)  ,$ to denote the subsurface $\left(
f,\overline{D}\right)  ,$ and call it a closed simple domain of $\Sigma$ as in (iii).
\end{remark}

\begin{definition}
\label{simple}Let $\Sigma=\left(  f,\overline{\Delta}\right)  $ be a surface
of $\mathbf{F.}$

(a) A point $a\in\overline{\Delta}$ is called a simple point of $f$ if $f$ is
homeomorphic in a neighborhood of $a$ in $\overline{\Delta}.$

(b) A point $\left(  f,a\right)  $ of $\Sigma$ is called a simple point of
$\Sigma$ if $a$ is a simple point of $f.$

(c) For a subset $A$ of $\overline{\Delta}\mathbf{\ }$and a point $a\in A,$
$a$ is called a simple point of $f$ in $A$ if $f$ is homeomorphic in a
neighborhood of $a$ in $A,$ and $\left(  f,a\right)  $ is called a simple
point of $\Sigma$ in $\left(  f,D\right)  $ if $a$ is a simple point of $f$ in
$D.$
\end{definition}

By definition a point $\left(  f,a\right)  $ of $\Sigma^{\circ},$ say,
$a\in\Delta,$ is a simple point of $\Sigma$ if and only if $a$ is a regular
point of $f.$ A point $\left(  f,a\right)  \in\partial\Sigma$ is a simple
point of $\Sigma,$ if and only if $a$ is a regular point of $f$ and
$\partial\Sigma$ is simple in a neighborhood of $a$ in $\partial\Delta.$ Note
that a simple point of a subsurface of $\Sigma$ needs not be a simple point of
$\Sigma.$

\begin{Deformation}
\label{Deform1}Consider a surface $\Sigma=\left(  f,\overline{\Delta}\right)
\in\mathcal{F}\left(  L,m\right)  $ with $\mathcal{F}\left(  L,m\right)
$-partition (\ref{LMP}). Assume that for some pair $j_{1}$ and $j_{2}$ with
$1\leq j_{1}<j_{2}\leq m,$
\[
L(c_{j_{i}})<\pi,c_{j_{i}}^{\circ}\cap E_{q}=\emptyset,
\]
and one of the following hold

(a) The curvatures $k(c_{j_{i}})$ of $c_{j_{i}}$ are distinct for $i=1,2.$

(b) $k(c_{j_{1}})=k(c_{j_{2}})$ and both $c_{j_{1}}$ and $c_{j_{2}}$ are major
circular arcs.

We will show that we can deform $\Sigma$ by changing the two arcs $c_{j_{1}}$
and $c_{j_{2}}$ to obtain a surface $\Sigma^{\prime}\in\mathcal{F}\left(
L,m\right)  $ such that
\[
H(\Sigma^{\prime})>H(\Sigma),L(\partial\Sigma^{\prime})=L(\partial\Sigma).
\]

By Corollary \ref{2-curvature}, using the notations $\left(  \overline
{T_{j_{i},\theta_{i}}},c_{j_{i}}\right)  $ with old boundary $c_{j_{i}}$ and
new boundary $c_{j_{i}}^{\prime\circ},i=1,2,$ in Remark \ref{Riemann} (iv2),
we can deform the simple domain $\left(  \overline{T_{j_{i},\theta_{i}}%
},c_{j_{i}}\right)  $ of $\Sigma$, $i=1,2$, as follows.

We replace $\left(  \overline{T_{j_{i},\theta_{i}}},c_{j_{i}}\right)  $ with
$\left(  \overline{T_{j_{i},\theta_{i}^{\prime}}^{\prime}},\mathfrak{c}%
_{j_{i}}\right)  ,$ where $T_{j_{i},\theta_{i}^{\prime}}^{\prime}$ is a domain
on $S$ enclosed by $\mathfrak{c}_{j_{i}}-c_{j_{i}}^{\prime}$ and
$\mathfrak{c}_{j_{i}}$ is a convex circular are from $q_{j_{i}}$ to
$q_{j_{i}+1},$ which is a small perturbation of $c_{j_{i}}$ with the same
endpoints and $\mathfrak{c}_{j_{i}}\cap c_{j_{i}}^{\prime}=\{q_{j_{i}%
},q_{j_{i}+1}\}$. Then by (a), or (b), and Corollary \ref{2-curvature}, we may
choose $\mathfrak{c}_{j_{i}}$ such that $\sum_{i=1}^{2}A(T_{j_{i},\theta
})<\sum_{i=1}^{2}A(T_{j_{i},\theta^{\prime}}^{\prime})$ and $\sum_{i=1}%
^{2}L(c_{j_{i}})=\sum_{j=1}^{2}L(\mathfrak{c}_{j_{i}}).$ After this
deformation we obtain a new surface $\Sigma^{\prime}\in\mathcal{F}\left(
L,m\right)  \ $such that the $\mathcal{F}\left(  L,m\right)  $-partition
(\ref{LMP}) changes into $\mathcal{F}\left(  L,m\right)  $-partition
\[
\partial\Sigma^{\prime}=c_{1}+\cdots+c_{j_{1}-1}+\mathfrak{c}_{j_{1}}%
+c_{j_{1}+1}+\cdots+c_{j_{2}-1}+\mathfrak{c}_{j_{2}}+c_{j_{2}+1}+\cdots+c_{m}%
\]
of $\partial\Sigma^{\prime}.$ Thus the surface $\Sigma^{\prime}$ with
$H(\Sigma^{\prime})$ larger and $L(\partial\Sigma^{\prime})$ unchanged exists.
$\Sigma^{\prime}$ is in fact obtained by moving $c_{j_{i}}$ to its left (or
right) hand side a little to the position of $\mathfrak{c}_{j_{i}},i=1,2.$
\end{Deformation}

\section{The distances on surfaces in $\mathcal{F}$}

We first introduce some results for counting terms of partitions.

\begin{lemma}
\label{mm'}Assume that $m\geq3,$ $\Sigma=\left(  f,\overline{\Delta}\right)
\in\mathcal{F}_{r}\left(  L,m\right)  ,$ (\ref{fmp1}) and (\ref{fmp2}) are
$\mathcal{F}\left(  L,m\right)  $-partitions of $\partial\Sigma\ $with
$c_{j}\left(  q_{j},q_{j+1}\right)  =\left(  f,\alpha_{j}\left(  a_{j}%
,a_{j+1}\right)  \right)  ,j=1,\dots,m,$ and that the following (a) and (b) hold.

(a) $a$ and $b$ are two points on $\partial\Delta,a\neq b,$ $\gamma_{1}$ is an
arc on $\partial\Delta$ from $a$ to $b$ and $\gamma_{0}=\left(  \partial
\Delta\right)  \backslash\gamma_{1}^{\circ},$ both oriented by $\partial
\Delta.$

(b) $I$ is a simple arc in $\overline{\Delta}$ from $a$ to $b$ such that
$I^{\circ}\in\Delta,\ I\cap f^{-1}(E_{q})=\emptyset,$ and either (b1) $\left(
f,-I\right)  $ is an SCC arc with $L\left(  f,I\right)  \leq L(f,\gamma_{0})$,
or (b2) $\left(  f,I\right)  $ is straight with $L(f,I)<\pi.$

Then the following hold.

(i) $I$ divides $\Delta$ into two Jordan domains $\Delta_{0}$ and $\Delta_{1}
$ on the left and right hand side of $I,$ respectively, $\partial\Delta
_{0}=\gamma_{0}+I,$ and $\partial\Delta_{1}=\gamma_{1}-I.$

(ii) The surface $\Sigma_{1}=\left(  f,\overline{\Delta_{1}}\right)  $ is
contained in $\mathcal{F}_{r}\left(  L,m_{1}^{\prime}\right)  $ with%
\begin{equation}
m_{1}^{\prime}=m+2-\#\left[  \gamma_{0}\cap\{a_{j}\}_{j=1}^{m}\right]  .
\label{m'}%
\end{equation}

(iii) When (b2) holds $\Sigma_{0}=\left(  f,\overline{\Delta_{0}}\right)  $ is
also contained in $\mathcal{F}_{r}\left(  L,m_{0}^{\prime}\right)  $ with%
\[
m_{0}^{\prime}=m+2-\#\left[  \gamma_{1}\cap\{a_{j}\}_{j=1}^{m}\right]  .
\]

\end{lemma}

\begin{proof}
By the assumption, (i) is trivial to verify.

By (b), $f$ has no branch point in $I^{\circ}$ and thus we have:

(c) $f$ is homeomorphic in a neighborhood of $I^{\circ}$ in $\overline{\Delta
},$ and thus $\Sigma_{1}=\left(  f,\Delta_{1}\right)  \in\mathcal{F}_{r},$ and
when (b2) holds $\Sigma_{0}\in\mathcal{F}_{r}$ as well.

It is clear that by (b)
\[
L\geq L\left(  f,\partial\Delta\right)  =L(\gamma_{1}+\gamma_{0})\geq
L(\gamma_{1})+L(f,I)=L(\partial\Sigma_{1}).
\]
and for the same reason $L\geq L(\partial\Sigma_{0})$ when (b2) holds.
Therefore we have by (b) and (c) that:

(d) $\Sigma_{1}\in\mathcal{F}_{r}\left(  L\right)  ;$ and if (b2) holds, then
$\Sigma_{0}\in\mathcal{F}_{r}\left(  L\right)  $ also holds.

The endpoints $\{a,b\}$ gives a refinement of the $\mathcal{F}\left(
L,m\right)  $-partition (\ref{fmp1}) which contains $m+2$ terms, among which
at most two are just points, and we let $\mathbf{A}$ be the set of all these
terms. It is easy to see that if $\gamma_{0}$ contains $s$ points of
$\{a_{j}\}_{j=1}^{m},$ then $\gamma_{1}$ is a sum of $m+1-s$ terms of
$\mathbf{A,}$ no matter what $\#\{a,b\}\cap\{a_{j}\}_{j=1}^{m}$ is equal to.
Thus by (c) and (d) $\partial\Delta_{1}$ has an $\mathcal{F}_{r}\left(
L,m^{\prime}\right)  $-partition$\mathcal{\ }$consisted of the term $I$ and
$m+1-s$ terms of $\mathbf{A},$ for $\Sigma_{1},$ with $s=\#\gamma_{0}%
\cap\{a_{j}\}_{j=1}^{m}.$ Thus we have $\Sigma_{1}\in\mathcal{F}_{r}\left(
L,m_{1}^{\prime}\right)  $ and for the same reason, $\Sigma_{0}\in
\mathcal{F}_{r}\left(  L,m_{0}^{\prime}\right)  $ in the case (b2).
\end{proof}

\begin{lemma}
\label{m-1}$\label{k>1 every used?}$Assume $m>3$, $\Sigma=\left(
f,\overline{\Delta}\right)  \in\mathcal{F}_{r}\left(  L,m\right)  $,
(\ref{fmp1}) and (\ref{fmp2}) are $\mathcal{F}\left(  L,m\right)  $-partitions
of $\partial\Sigma\ $with $c_{j}\left(  q_{j},q_{j+1}\right)  =\left(
f,\alpha_{j}\left(  a_{j},a_{j+1}\right)  \right)  ,j=1,\dots,m,$ and the
following (a)--(e) hold (see Figure \ref{counting-edges-1} for $k=3$).

(a) $k$ is a positive integers with $2\leq k<m,$ $b_{2}$ and $b_{2k+1}$ are
two points on $\partial\Delta$ with $b_{2}\neq b_{2k+1};$ $\gamma_{0}$ is the
arc of $\partial\Delta$ from $b_{2k+1}$ to $b_{2}\ $and $\gamma_{0}^{c}$ is
the arc $\left(  \partial\Delta\right)  \backslash\gamma_{0}^{\circ}$, both
oriented by $\partial\Delta;$ and $\gamma_{0}^{\prime}=\gamma_{0}^{\prime
}\left(  a_{i_{0}},a_{i_{2}}\right)  $ is the smallest arc on $\partial\Delta$
containing $\gamma_{0}$ such that $\partial\gamma_{0}^{\prime}=\{a_{i_{0}%
},a_{i_{2}}\}\subset\{a_{j}\}_{j=1}^{m},$ that is, $\gamma_{0}^{\prime}$ is
the union of all the terms in (\ref{fmp1}) which intersect $\gamma_{0}^{\circ
}.$

(b) $I=I\left(  b_{2},b_{2k+1}\right)  $ is a simple arc in $\overline{\Delta
}$ from $b_{2}\in\partial\Delta$ to $b_{2k+1}\in\partial\Delta$ which has the
partition%
\[
I\left(  b_{2},b_{2k+1}\right)  =I_{2}\left(  b_{2},b_{3}\right)
+I_{3}\left(  b_{3},b_{4}\right)  +\cdots+I_{2k}\left(  b_{2k},b_{2k+1}%
\right)  ,
\]

such that for each $j=2,\dots,k,I_{2j-1}\subset\partial\Delta,$ while for each
$j=1,\dots,k,\emptyset\neq I_{2j}^{\circ}\subset\Delta.$

(c) $I\cap\gamma_{0}=\{b_{2},b_{2k+1}\},$ say, $I_{3}\left(  b_{3}%
,b_{4}\right)  ,\dots,I_{2k-1}\left(  b_{2k-1},b_{2k}\right)  $ are all
contained in the open arc $\left(  \gamma_{0}^{c}\right)  ^{\circ}\ $and
$b_{2},b_{3},\dots,b_{2k},b_{2k+1}$ are arranged anticlockwise on
$\partial\Delta.$

(d) $\left(  f,-I\right)  \ $is an SCC arc on $S$, $L(f,-I)\leq L(f,\gamma
_{0})$ and $\cup_{j=1}^{k}I_{2j}^{\circ}\subset\Delta\backslash f^{-1}%
(E_{q}).$

(e) One of the conditions (e1)--(e3) holds:

(e1) $\gamma_{0}^{\prime}\cap I^{\circ}=\emptyset,$ say, $I_{3}\cap\gamma
_{0}^{\prime}=I_{2k-1}\cap\gamma_{0}^{\prime}=\emptyset,$ as in Figures
\ref{counting-edges-1} (1) and (2);

(e2) $\gamma_{0}^{\prime}\cap I^{\circ}=\emptyset$ and $\gamma_{0}^{\circ}%
\cap\{a_{j}\}_{j=1}^{m}\neq\emptyset;$

(e3) $\gamma_{0}^{\prime}\cap I^{\circ}=\emptyset$ and $\gamma_{0}^{\prime
}=\gamma_{0},$ as in Figure \ref{counting-edges-1} (1).

Then the following hold:

(i) For each $j=2,\dots,k,$ the two end points of $I_{2j-1}$ are contained in
$\{a_{j}\}_{j=1}^{m},$ say $\{b_{3},\dots,b_{2k}\}\subset\{a_{j}\}_{j=1}^{m}.$

(ii) $I$ divides $\Delta$ into $k+1$ Jordan domains $\left\{  \Delta
_{i}\right\}  _{i=0}^{k}$ such that $\Delta_{0}$ is on the left hand side of
$I$ and $\Delta_{1},\dots,\Delta_{k}$ are on the right hand side of $I.$

(iii) For each $j=1,2,\dots,k$, one of the following holds.

(iii1) $\Sigma_{j}=\left(  f,\overline{\Delta_{j}}\right)  \in\mathcal{F}%
_{r}\left(  L,m\right)  $ if (e1) holds.

(iii2) $\Sigma_{j}=\left(  f,\overline{\Delta_{j}}\right)  \in\mathcal{F}%
_{r}\left(  L,m-1\right)  \ $if (e2) or (e3) holds.

(iv) $\min\left\{  L(\partial\Sigma_{1}),L(\partial\Sigma_{k})\right\}
\geq\min\{L(f,\gamma_{01}),L(f,\gamma_{02})\}$ where $\gamma_{0i},i=1,2,$ are
the two components of $\gamma_{0}^{\prime}\backslash\gamma_{0}^{\circ}.$
\end{lemma}%

\begin{figure}
[ptb]
\begin{center}
\ifcase\msipdfoutput
\includegraphics[
height=2.3393in,
width=4.1658in
]%
{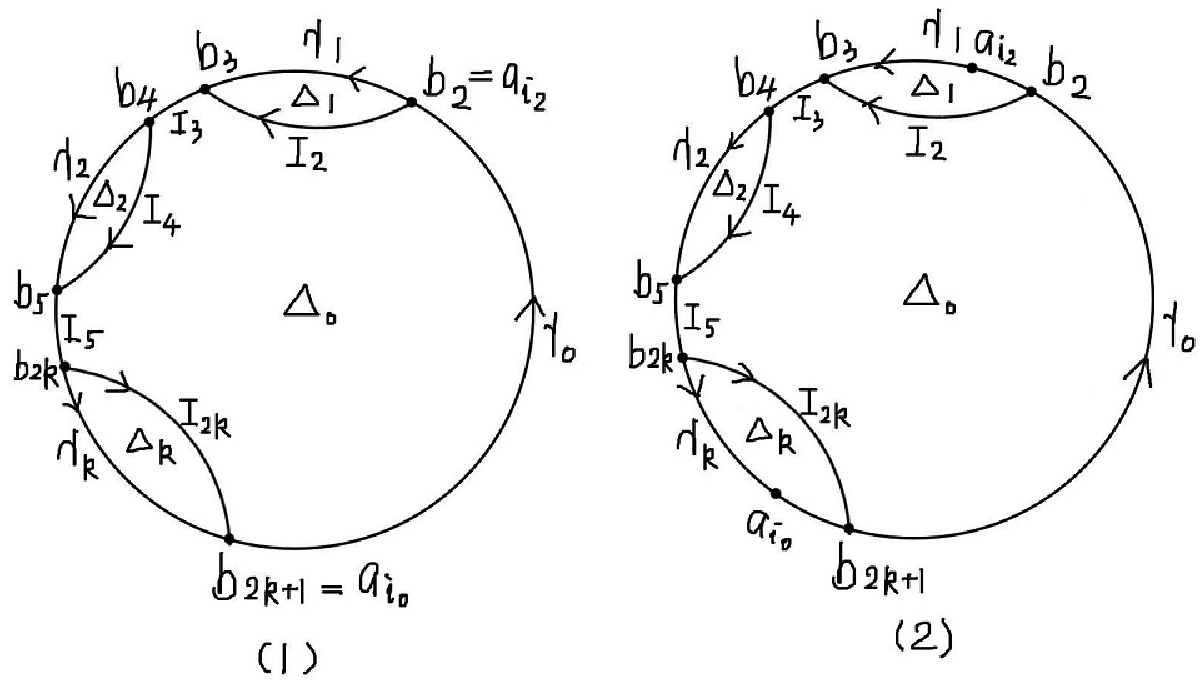}%
\caption{ }%
\end{center}
\end{figure}

\begin{proof}
(i) follows from Lemma \ref{tangent}, and (ii) is trivial.

Let $\gamma_{j}$ be the arc of $\partial\Delta$ from $b_{2j}$ to $b_{2j+1}$
for $j=1,\dots,k.$ Then we may arrange $\Delta_{j}$ so that $\partial
\Delta_{j}=\gamma_{j}-I_{2j},j=1,2,\dots,k.$ It is clear that $\gamma_{j}%
\cap\gamma_{0}$ contains at most one point for $j=1,\dots,k,$ and then by (d)
we have
\[
L(f,\partial\Delta_{j})=L\left(  f,\gamma_{j}-I_{2j}\right)  \leq
L(f,\gamma_{j})+L(f,\gamma_{0})\leq L(f,\partial\Delta)\leq L,
\]
and moreover $f$ restricted to a neighborhood of $I_{2j}^{\circ}$ in
$\overline{\Delta}$ is a homeomorphism, by (d) and the assumption
$\Sigma=\left(  f,\overline{\Delta}\right)  \in\mathcal{F}_{r}(L,m).$

We may assume
\[
\left\{  a_{i_{0}},a_{i_{2}}\right\}  \subset\{a_{j}\}_{j=1}^{m}\ \text{with
}i_{0}<i_{2}<m+i_{0}\left(  a_{m+i}=a_{i}\right)  .
\]

Assume (e1) holds. Then $\{b_{3},b_{2k}\}\ $is contained in $I^{\circ}%
\cap\{a_{j}\}_{j=1}^{m}$ by (i), and is outside $\gamma_{0}^{\prime}$ by (e1),
and then $a_{i_{2}}$ and $b_{3}$ are distinct and both contained in
$\gamma_{1}\cap\{a_{j}\}_{j=1}^{m},$ and for the same reason, $a_{i_{0}}$ and
$b_{2k}$ are distinct and both contained in $\gamma_{k}\cap\{a_{j}\}_{j=1}%
^{m}.$ Then it is easy to see $s_{1}=\#\left[  \left(  \partial\Delta\right)
\backslash\gamma_{1}^{\circ}\right]  \cap\{a_{j}\}_{j=1}^{m}\geq\#\gamma
_{k}\cap\{a_{j}\}_{j=1}^{m}\geq2\ $and $\Sigma_{1}$ is contained in
$\mathcal{F}_{r}\left(  L,m\right)  ,$ by applying Lemma \ref{mm'} to $I_{2}$;
and so is $\Sigma_{k}$ for the same reason. It is trivial to see that when
$k>2,$ for $j=2,3,\dots,k-1,$ $s_{j}=\#\left[  \left(  \partial\Delta\right)
\backslash\gamma_{j}^{\circ}\right]  \cap\{a_{j}\}_{j=1}^{m}\geq\#\{a_{i_{2}%
},b_{3},b_{2k},a_{i_{0}}\}=4,$ and thus by Lemma \ref{mm'} $\Sigma_{j}%
\in\mathcal{F}_{r}\left(  L,m-2\right)  .$ Thus (iii1) holds.

Assume (e2) holds. Then there exists $i_{1}$ with $i_{0}<i_{1}<i_{2}$ and
$a_{i_{1}}\in\gamma_{0}^{\circ}.$ Consider $\Delta_{1}$ and $\Delta_{k}.$ It
is clear that $a_{i_{0}}$ and $a_{i_{1}}$ are both contained in $\left(
\partial\Delta\right)  \backslash\gamma_{1}.$ On the other hand, by (e2) and
(i) $b_{3}\in\left\{  \partial\gamma_{1}\right\}  \cap\{a_{j}\}_{j=1}^{m}$.
Thus $s_{1}=\#\left[  \left(  \partial\Delta\right)  \backslash\gamma
_{1}^{\circ}\right]  \cap\{a_{j}\}_{j=1}^{m}\geq3$ and by Lemma \ref{mm'} we
have $\Sigma_{1}\in\mathcal{F}_{r}\left(  L,m+2-s\right)  \subset
\mathcal{F}_{r}\left(  L,m-1\right)  .$ For the same reason $\Sigma_{k}%
\in\mathcal{F}_{r}\left(  L,m-1\right)  .$ It is trivial to see that when
$k>2,$ for $j=2,3,\dots,k-1,$ $s_{j}=\#\left[  \left(  \partial\Delta\right)
\backslash\gamma_{j}^{\circ}\right]  \cap\{a_{j}\}_{j=1}^{m}\geq\#\{a_{i_{2}%
},b_{3},b_{2k},a_{i_{0}},a_{i_{1}}\}=5,$ and thus we have by Lemma \ref{mm'}
$\Sigma_{j}\in\mathcal{F}_{r}\left(  L,m-3\right)  .$

Assume (e3) holds. Then we still have $s_{1}=\#\left[  \left(  \partial
\Delta\right)  \backslash\gamma_{1}^{\circ}\right]  \cap\{a_{j}\}_{j=1}%
^{m}\geq3$ and thus $\Sigma_{1}\in\mathcal{F}_{r}\left(  L,m-1\right)  .$ For
the same reason, we also have $\Sigma_{k}\in\mathcal{F}_{r}\left(
L,m-1\right)  $. Assume $k>2$ and let $j\in\{2,\dots,k-1\}.$ Then by (i)
$\partial\gamma_{j}$ are contained in $\{a_{j}\}_{j=1}^{m},$ and by the
assumptions, $a_{i_{0}},a_{i_{2}}$ are outside $\gamma_{j}^{\circ}$ and thus
$s_{j}=\#\left[  \left(  \partial\Delta\right)  \backslash\gamma_{j}^{\circ
}\right]  \cap\{a_{j}\}_{j=1}^{m}\geq\#\left\{  a_{i_{2}},b_{3},b_{2k}%
,a_{i_{0}}\right\}  =4,$ and then by Lemma \ref{mm'} we have $\Sigma_{j}%
\in\mathcal{F}_{r}\left(  L,m+2-4\right)  =\mathcal{F}_{r}(L,m-2)$. (iii2) has
been proved.

It is clear that $L\left(  \partial\Sigma_{1}\right)  \geq L\left(
\gamma_{01}\right)  $ and $L\left(  \partial\Sigma_{2}\right)  \geq L\left(
\gamma_{02}\right)  $ this implies that (iv) holds true.
\end{proof}

\begin{definition}
\label{df}Let $\Sigma=\left(  f,\overline{\Delta}\right)  \in\mathcal{F}.$ For
any two points $a$ and $b$ in $\overline{\Delta},$ define their $d_{f}%
$-distance $d_{f}\left(  a,b\right)  $ by%
\[
d_{f}(a,b)=\inf\{L(f,I):I\mathrm{\ is\ a\ curve\ in\ }\overline{\Delta
}\mathrm{\ with\ endpoints\ }a\mathrm{\ and\ }b\}\mathrm{;}%
\]
for any two sets $A$ and $B$ in $\overline{\Delta}$ define their $d_{f}%
$-distance by%
\[
d_{f}\left(  A,B\right)  =\inf\left\{  d_{f}\left(  a,b\right)  :a\in A,b\in
B\right\}  ;
\]
and for any set $A$ in $\overline{\Delta}$ and any $\varepsilon>0$ define the
$d_{f}$-$\varepsilon$-neighborhood of $A$ (in $\overline{\Delta}$) by%
\[
N_{f}(A,\varepsilon)=\left\{  x\in\overline{\Delta}:d_{f}(A,x)<\varepsilon
\right\}  .
\]

The distance $d_{f}\left(  a,b\right)  $ is also called the distance of
$\Sigma$ between the two points $\left(  f,a\right)  $ and $\left(
f,b\right)  $ of $\Sigma.$ Sometimes we will write $d_{\Sigma}\left(  \left(
f,a\right)  ,\left(  f,b\right)  \right)  =d_{f}\left(  a,b\right)  .$ Then
the notation $d_{\Sigma}\left(  \left(  f,A\right)  ,\left(  f,B\right)
\right)  $ between two sets of $\Sigma,$ and $N_{\Sigma}(\left(  f,A\right)
,\varepsilon)$ is well defined.
\end{definition}

\begin{remark}
When $\varepsilon$ is small enough, $(a,N_{f}\left(  a,\varepsilon\right)  )$
is the disk of $\Sigma$ with radius $\varepsilon$ (see Lemma \ref{cov-1},
Corollary \ref{cov-2} and Definition \ref{nod}). On the other hand, Corollary
\ref{cov-2} (iv) directly implies
\end{remark}

\begin{lemma}
\label{short-disk}Let $\Sigma=\left(  f,\overline{\Delta}\right)
\in\mathcal{F}(L,m)$, let (\ref{fmp1}) and (\ref{fmp2}) be $\mathcal{F}\left(
L,m\right)  $-partitions of $\partial\Sigma\ $with $c_{j}\left(  q_{j}%
,q_{j+1}\right)  =\left(  f,\alpha_{j}\left(  a_{j},a_{j+1}\right)  \right)
,j=1,\dots,m,$ for any $j$ let $a\in\alpha_{j}^{\circ},$ and finally let $I$
be a $d_{f}$-shortest path in $\overline{\Delta}.$ Then for any disk $\left(
a,U_{\delta}\right)  $ of $\Sigma$ with small enough radius $\delta<\frac{\pi
}{2}$ the following hold:

(i) $f:\overline{U_{\delta}}\rightarrow f(\overline{U_{\delta}})$ is a
homeomorphism and $f(\overline{U_{\delta}})$ is a convex lens. If $c_{j}$ is
straight, then $f(\overline{U_{\delta}})$ is half of a disk whose diameter is
contained in $c_{j}$ and $f(\overline{U_{\delta}})$ is on the left hand side
of $c_{j}.$

(ii) For any two points $a_{1}$ and $a_{2}$ in $\overline{U_{\delta}},$ the
$d_{f}$-shortest path from $a_{1}$ to $a_{2}$ exists, which is the unique
$f$-lift of the line segment $\overline{f(a_{1})f(a_{2})}$ in $\overline{U}.$

(iii) If $I\cap\overline{U_{\delta}}\neq\emptyset,$ then $I\cap\overline
{U_{\delta}}$ is a subarc of $I.$

(iv) If $a\in I^{\circ}\cap\alpha_{j}^{\circ},$ then $c_{j}$ is straight and
$I^{\circ}\cap\alpha_{j}^{\circ}$ is an open neighborhood of $a$ in
$\alpha_{j}.$ Thus $I^{\circ}\cap\alpha_{j}^{\circ}=\emptyset$ if $c_{j}$ is
strictly convex.
\end{lemma}

\begin{lemma}
\label{dfcon}For any $\Sigma=\left(  f,\overline{\Delta}\right)
\in\mathcal{F},$ $d_{f}(x,y)$ is a continuous function on $\overline{\Delta
}\times\overline{\Delta}.$ In other words, $d_{\Sigma}\left(  \cdot
,\cdot\right)  $ is a continuous function on $\Sigma\times\Sigma.$
\end{lemma}

The proof is simple and standard and left to the reader.

\begin{lemma}
\label{d>d}Let $\Sigma=\left(  f,\overline{\Delta}\right)  \in\mathcal{F}$ and
let $\left(  x,U\right)  $ be a disk of $\Sigma$ with radius $\delta$ (see
Definition \ref{nod}). Then for any $y\in\overline{\Delta}\backslash U,$
$d_{f}(x,y)\geq\delta,$ and $d_{f}\left(  x,y\right)  =\delta$ for every
$y\in\overline{\left(  \partial U\right)  \backslash\partial\Delta}.$
\end{lemma}

\begin{proof}
This is trivial by Definition \ref{nod} and Remark \ref{nod-1}.
\end{proof}

\begin{lemma}
\label{dfshor}Let $\Sigma=\left(  f,\overline{\Delta}\right)  \in\mathcal{F}$,
let $x\in\overline{\Delta}$, and let $\left(  x,U_{x}\right)  $, $\left(
x,U_{x}^{\prime}\right)  $ and $\left(  x,U_{x}^{\prime\prime}\right)  $ be
three disks in $\Sigma$ with radius $\delta/4,\delta/2,$ and $\delta$ (see
Definition \ref{nod}), respectively. Then for any two distinct points $a$ and
$b$ in $U_{x},$ the $d_{f}$-shortest path $I(a,b)$ exists; and more precisely,
putting $A=f(a),B=f(b),X=f(x),$ one of the following holds.

(i) If $a=x,$ then the $f$-lift $I\left(  x,b\right)  $ of $\overline{XB}$ is
the unique shortest path from $x$ to $b$.

(ii) If $A\neq B$ and $\overline{AB}$ has an $f$-lift $I=I(a,b)$ from $a$ to
$b,$ then $I$ is the unique $d_{f}$-shortest path.

(iii) If $A=B,$ or $A\neq B$ but $\overline{AB}$ has no $f$-lift from $a$ to
$b,$ then the $f$-lift $I=I(a,b)$ of $\overline{AXB},$ from $a$ to $x,$ and
then to $b,$ is the unique $d_{f}$-shortest path.
\end{lemma}

\begin{proof}
(i) follows from Corollary \ref{cov-2} (v). (ii) is trivial. We only prove (iii).

By definition there exists a sequence of paths $I_{n}\subset\overline{\Delta}$
from $a$ to $b$ such that for\footnote{If $I_{n}$ is the path $I_{n}%
:[0,1]\rightarrow\overline{\Delta},$ the length $L\left(  f,I_{n}\right)  $
should be understood to be $L(f\circ I_{n},[0,1]).$} $s_{n}=L\left(
f,I_{n}\right)  $
\[
\lim_{n\rightarrow\infty}s_{n}=d_{f}\left(  a,b\right)  .
\]
It is clear that $d_{f}\left(  a,b\right)  \leq L(\overline{AXB})<\delta/2$ by
Corollary \ref{cov-2} (v), since $d\left(  X,A\right)  <\delta/4,d\left(
X,B\right)  <\delta/4$. Thus, for sufficiently large $n,$ we see that
\[
I_{n}\subset U_{x}^{\prime},
\]
for otherwise we have $d_{f}\left(  a,b\right)  \geq\delta/2.$

We may parametrize $I_{n}\ $by length with $L\left(  f,I_{n}|_{[0,s]}\right)
=s,s\in\lbrack0,s_{n}]$ and
\[
s_{n}=L(f,I_{n})\rightarrow d_{f}\left(  a,b\right)  >0.
\]
By Aazela-Ascoli theorem, we may assume $\Gamma_{n}=(f,I_{n}),$ as a mapping
from $[0,s_{n}]$ to $S,$ has a subsequence uniformly converging to a path
$\Gamma_{0}:[0,s_{0}]\rightarrow S$ from $A$ to $B$ and we assume the
subsequence is $\Gamma_{n}$ itself. Then we have
\begin{equation}
L(\Gamma_{0})\leq d_{f}\left(  a,b\right)  \leq L(\overline{AXB}),
\label{L<=df}%
\end{equation}
and $\Gamma_{0}\subset\overline{D\left(  x,\delta/2\right)  }.$

If $X\in\Gamma_{0},$ then we have%
\[
L(\Gamma_{0})=d_{f}\left(  a,b\right)  =L(\overline{AXB})
\]
and by (i) the $f$-lift $I(a,b)$ of $\overline{AXB}$ is a $d_{f}$-shortest
path from $a$ to $b$. Let $I^{\prime}\left(  a,b\right)  $ be another $d_{f}%
$-shortest path. If $x\in I^{\prime}\left(  a,b\right)  ,$ then $x$ gives a
partition $I^{\prime}\left(  a,b\right)  =I^{\prime}\left(  a,x\right)
+I^{\prime}\left(  x,b\right)  ,$ $I^{\prime}\left(  a,x\right)  $ and
$I^{\prime}\left(  x,b\right)  $ have to be the $d_{f}$-shortest paths from
$a$ to $x,$ and $x$ to $b,$ respectively by (i), and thus $I^{\prime}\left(
a,x\right)  +I^{\prime}\left(  x,b\right)  =I\left(  a,b\right)  $ by(i). If
$x\notin I^{\prime}\left(  a,b\right)  ,$ we can show, as the following
discussion for the case $X\notin\Gamma_{0},$ that that $I^{\prime}\left(
a,b\right)  $ is the unique $f$-lift of $\overline{AB},$ which implies
$d_{f}\left(  a,b\right)  =L(I^{\prime}\left(  a,b\right)  )=L\left(
\overline{AB}\right)  =L(\overline{AX})+L(\overline{XB})=d_{f}\left(
a,b\right)  .$ This is a contradiction, since $L\left(  \overline{AB}\right)
<L(\overline{AX})+L(\overline{XB})$ when $X\notin\overline{AB}.$

Assume $X\notin\Gamma_{0}.$ Then there exists a disk $\left(  x,V_{x}\right)
$ of $\Sigma$ in $\left(  x,U_{x}\right)  $ such that $I_{n}\subset
U_{x}\backslash\overline{V_{x}}$, $f$ is locally homeomorphic on
$\overline{U_{x}}\backslash V_{x}$, and $\Gamma_{0}\subset f(\overline{U_{x}%
}\backslash V_{x}).$ Then $\Gamma_{0}$ has an $f$-lift $I_{0}$ such that
$I_{n}$ uniformly converges to $I_{0},$ by Lemma \ref{cov-3}. Then $I_{0}$ is
a $d_{f}$-shortest path from $a$ to $b.$ We will show that $\Gamma
_{0}=\overline{AB}.$

It is clear that $I_{0}$ is simple, for otherwise there is another path
$I_{0}^{\prime}=I_{0}^{\prime}\left(  a,b\right)  $ which is obtained from
$I_{0}$ by omitting a loop of $I_{0}$ so that $L\left(  f,I_{0}^{\prime
}\right)  <d_{f}\left(  a,b\right)  $ contradicting $I_{0}$ being shortest. It
is also clear that any subarc of $I_{0}$ is a $d_{f}$-shortest path.

Let $y\in I_{0}^{\circ}.$ Then by the assumption we have $x\notin I_{0}$ and
by Corollary \ref{cov-2} (i) and (iv), $y$ has a neighborhood $I_{y}%
=I_{y}\left(  y^{\prime},y^{\prime\prime}\right)  $ in $I_{0}$ so that $I_{y}$
is contained in a disk $\left(  y,U_{y}\right)  $ of $\left(  x,U_{x}^{\prime
}\right)  \ $with $x\notin U_{y}.$ Then $f(U_{y})$ is convex and
$f:U_{y}\rightarrow f(U_{y})$ is homeomorphic. Thus both $f\left(
I_{y}\right)  $ and $\overline{f(y^{\prime})f(y^{\prime\prime})}\ $can be lift
into $\overline{U_{y}}\subset U_{x}^{\prime}$ from $y^{\prime}$ to
$y^{\prime\prime}.$ Then $I_{y}$ has to be the lift of $\overline{f(y^{\prime
})f(y^{\prime\prime})}.$ Thus $\left(  f,I_{0}\right)  $ is straight
everywhere and we have $\left(  f,I_{0}\right)  =\Gamma_{0}=\overline{AB}.$
Then in this case, $I_{0}$ is the unique $d_{f}$-shortest from $a$ to $b.$
\end{proof}

\begin{lemma}
\label{ex-short}Let $\Sigma=\left(  f,\overline{U}\right)  \in\mathcal{F}$ and
let $a_{1}$ and $a_{2}$ be two distinct points in $\overline{\Delta}.$ Then
the $d_{f}$-shortest path $I$ from $a_{1}$ to $a_{2}$ exists, and for any such
path $I,$ $\left(  f,I\right)  $ is a polygonal path on $S$ from $f(a_{1}) $
to $f(a_{2}).$
\end{lemma}

\begin{proof}
It is clear that, for some positive integer $\delta_{0},$ There are $3s_{0}$
disks $\left(  p_{s},U_{s}\right)  ,$ $\left(  p_{s},U_{s}^{\prime}\right)  ,$
$\left(  p_{s},U_{s}^{\prime\prime}\right)  $ of $\Sigma,$ with radius
$\delta/4,\delta/2,\delta$ for each $s=1,\dots,s_{0},$ such that
$\mathcal{O}=\{U_{s}\}_{s=1}^{s_{0}}$ is an open covering of $\overline{U}.$
Note that by Lemma \ref{cov-1} and Definition \ref{nod}, for each
$s=1,\dots,s_{0},$ the four relations $U_{s}\cap\partial U\neq\emptyset
,U_{s}^{\prime}\cap\partial U\neq\emptyset,U_{s}^{\prime\prime}\cap\partial
U\neq\emptyset$ and $p_{s}\in\partial U$ are equivalent.

Write $C_{s}=\overline{U\cap\partial U_{s}},$ the closure of the part of
$\partial U_{s}$ located in $U,$ for $s=1,2,\dots,s_{0}.$ We assume that no
$U_{s}^{\prime\prime}$ contains $\overline{U},$ and then $C_{s}\neq\emptyset$
and, as $\alpha_{3}$ in Lemma \ref{cov-1} (B) (B1), $C_{s}$ is connected for
all $s=1,\dots,s_{0}.$

By definition, there exists a sequence of paths $J_{n}$ in $\overline{U}$ from
$a_{1}$ to $a_{2}$ such that
\begin{equation}
\lim_{n\rightarrow\infty}L(f,J_{n})=d_{f}(a_{1},a_{2}).
\end{equation}
It is clear that $a_{1}\in U_{s_{1}}\in\mathcal{O}\ $for some $s_{1}\leq
s_{0}.$ If $C_{s_{1}}\cap J_{n}=\emptyset$ for some $J_{n},$ then
$\{a_{1},a_{2}\}\subset U_{s_{1}}$ and the $d_{f}$-shortest path $I,$ such
that $\left(  f,I\right)  $ is polygonal, exists by Lemma \ref{dfshor}. Thus
we may assume that for each $n,C_{s_{1}}\cap J_{n}\neq\emptyset,$ and let
$a_{n2}$ be the latest point of $J_{n}$ contained in $J_{n}\cap C_{s_{1}}%
.\ $By taking subsequence we may assume that $a_{n2}\rightarrow a_{02}\in
C_{s_{1}}$, and then for the smaller arc $C_{s_{1}}^{n}$ of $C_{s_{1}}$
between $a_{02}$ and $a_{n2}$, the length $L(f,C_{s_{1}}^{n})$ tends to $0,$
since we assumed $f$ is holomorphic on $\overline{U}\ $(note that $C_{s_{1}}$
may be a circle). By the previous lemma, the $d_{f}$-shortest path
$I_{01}=I_{01}\left(  a_{01},a_{02}\right)  =I_{01}\left(  a_{1}%
,a_{02}\right)  ,$ which means by convention that $I_{01}$ is an arc from
$a_{01}=a_{1}$ to $a_{02},$ exists, and $\left(  f,I_{01}\right)  $ is
polygonal. Let $J_{n,1}=C_{s_{1}}^{n}+J_{n}\left(  a_{n2},a_{2}\right)  $ and%
\[
J_{n}^{1}=I_{01}\left(  a_{01},a_{02}\right)  +J_{n,1}.
\]
Then we still have%
\[
L(f,J_{n}^{1})\rightarrow d_{f}(a_{1},a_{2}).
\]

Let $U_{s_{2}}$ be an element of $\mathcal{O}$ such that $a_{02}\in U_{s_{2}%
}.$ Then $s_{2}\neq s_{1}$ and when $n$ is large enough $C_{s_{1}}^{n}\subset
U_{s_{2}}$. Thus we have by definition of $a_{n2},$%
\[
J_{n,1}\subset\cup_{s\in\{1,2,\dots,s_{0}\}\backslash\left\{  s_{1}\right\}
}U_{s}.
\]

Applying the same argument to $J_{n,1},$ we can show that there exist a point
$a_{03}\in C_{s_{2}},$ a $d_{f}$-shortest path $I_{02}\left(  a_{02}%
,a_{03}\right)  ,$ a path $J_{n,2}\ $from $a_{03}$ to $a_{2}$ such that%
\[
J_{n,2}\subset\cup_{s\in\{1,2,\dots,s_{0}\}\backslash\left\{  s_{1}%
,s_{2}\right\}  }U_{s},
\]
and for the path $J_{n}^{2}=I_{01}\left(  a_{01},a_{02}\right)  +I_{02}\left(
a_{02},a_{03}\right)  +J_{n,2}$ from $a_{1}=a_{01}$ to $a_{2}$
\[
L(f,J_{n}^{2})\rightarrow d_{f}(a_{1},a_{2}).
\]

Repeating the above method a finite number of times, we can finally prove that
there exists a path
\begin{equation}
I=I_{01}+I_{02}+\cdots+I_{0s^{\ast}} \label{poly}%
\end{equation}
with $s^{\ast}\leq s_{0}$ such that $L(f,I)=d_{f}(a_{1},a_{2}).$ The existence
of $I$ is proved.

For any $d_{f}$-shortest path $I$ from $a_{1}$ to $a_{2},$ we may apply the
above argument to $I$ to show that $(f,I)$ is a polygonal. This completes the proof.
\end{proof}

\begin{lemma}
\label{shortest}\label{dfd}Let $m\geq3,$ $\Sigma=\left(  f,\overline{\Delta
}\right)  \in\mathcal{F}_{r}(L,m)$ and let $b_{1}$ and $b$ be two distinct
points on $\partial\Delta$ with%
\begin{equation}
d_{f}(b_{1},b)<\pi. \label{k13}%
\end{equation}
Let
\begin{equation}
\partial\Delta=\alpha_{1}\left(  a_{1},a_{2}\right)  +\alpha_{2}\left(
a_{2},a_{3}\right)  +\cdots+\alpha_{m}\left(  a_{m},a_{1}\right)  \label{ma11}%
\end{equation}
be an $\mathcal{F}(L,m)$-partition of $\partial\Sigma$ and
\begin{equation}
\partial\Sigma=c_{1}\left(  q_{1},q_{2}\right)  +c_{2}\left(  q_{2}%
,q_{3}\right)  +\cdots+c_{m}\left(  q_{m},q_{1}\right)  \label{ma12}%
\end{equation}
be the corresponding $\mathcal{F}(L,m)$-partition. Then for any $d_{f}%
$-\emph{shortest} path $I=I\left(  b_{1},b\right)  $ from $b_{1}$ to $b$, say,
$L(f,I)=d_{f}(b_{1},b),$ the following hold:

(i) $I$ is simple.

(ii) For each component $J$ of $I\cap\Delta\backslash f^{-1}(E_{q}),$ $\left(
f,J\right)  $ is a simple straight arc with $L(f,J)<\pi.$

(iii) For each $j\in\{1,\dots,m\},$ if $c_{j}$ is strictly convex, then
$I\cap\alpha_{j}^{\circ}\subset\{b_{1},b\},$ and thus $I^{\circ}\cap\alpha
_{j}^{\circ}=\emptyset$.

(iv) For each $j\in\{1,\dots,m\},$ if $c_{j}$ is straight and $I^{\circ}%
\cap\alpha_{j}^{\circ}\neq\emptyset$, then one of the following holds:

(iv1) $\alpha_{j}\subset I\ $or $I\subset\alpha_{j};$

(iv2) $\alpha_{j}\cap I$ is an subarc of $\alpha_{j}$ with $\alpha_{j}\cap
I=\alpha_{j}\left(  a_{j},b^{\prime}\right)  ,$ or $\alpha_{j}\cap
I=\alpha_{j}\left(  b^{\prime},a_{j+1}\right)  ,$ where $b^{\prime}\in
\{b_{1},b\};$

(iv3) $\alpha_{j}\cap I\ $is consisted of two subarcs $I^{\prime}$ and
$I^{\prime\prime}$ of $\alpha_{j},$ with $I^{\prime}\cap I^{\prime\prime
}=\emptyset,I^{\prime}=\alpha_{j}\left(  a_{j},b_{1}^{\prime}\right)  $,
$I^{\prime\prime}=\alpha_{j}\left(  b_{2}^{\prime},a_{j+1}\right)  ,$ and
$\{b_{1}^{\prime},b_{2}^{\prime}\}=\{b_{1},b\}.$ This occurs only if
$\{b_{1},b\}\subset\alpha_{j}$ and the subarc of $c_{j}$ with endpoints
$f(b_{1})$ and $f(b)$ has length $>\pi;$\newline Thus each compact component
of $I^{\circ}\cap\partial\Delta$ is a point of $\left\{  a_{j}\right\}
_{j=1}^{m},$ or an arc of $\partial\Delta$ with two endpoints in $\left\{
a_{j}\right\}  _{j=1}^{m}.$

(v) $I$ has a partition $I=I_{1}+I_{2}+\cdots+I_{2k+1},k\leq m+1,$ such that
for each $j=1,2,\dots,k,$ $I_{2j}^{\circ}\neq\emptyset$ and $I_{2j}^{\circ
}\subset\Delta;$ and for each $j=1,2,\dots,k+1,$ $I_{2j-1}\subset
\partial\Delta$ and either $I_{2j-1}$ is a point or $\left(  f,I_{2j-1}%
\right)  $ is an arc of $\partial\Sigma$ which is a polygonal path on $S$ (if
$I\cap\Delta=\emptyset,$ $k=0$ and $I=I_{1}$).

(vi) If $k\geq1,$ and if for some $j=1,2,\dots,k,$ $I_{2j}^{\circ}\cap
f^{-1}(E_{q})=\emptyset,$ then $\left(  f,I_{2j}\right)  $ is a simple line
segment with $L\left(  f,I_{2j}\right)  <\pi\ $and $f$ restricted to a
neighborhood of $I_{2j}^{\circ}$ is a homeomorphism.

(vii) If $k\geq1$ and $I_{1}$ is not a point, then the joint point of $I_{1}$
and $I_{2}$ is contained in $\{a_{j}\}_{j=1}^{m},$ and if $I_{2k+1}$ is not a
point, then the joint point of $I_{2k}$ and $I_{2k+1}$ is also contained in
$\{a_{j}\}_{j=1}^{m}.$

(viii) If $k\geq2,$ then the endpoints of all $I_{j}$ for $j=3,5,\dots,2k-1,$
are contained in $\{a_{j}\}_{j=1}^{m}.$

(ix) Assume $k\geq1.$ Then $I_{2}$ divides $\Delta$ into two Jordan domains
$\Delta_{1}\ $and $\Delta_{2}.$ Assume, in addition to $k\geq1,$
\begin{equation}
\Delta\cap I_{2}\cap f^{-1}(E_{q})=\emptyset, \label{ma13}%
\end{equation}
and one of the following (a)--(d) holds:

(a) $\left(  \partial\Delta_{i}\right)  \cap\partial\Delta$ contains at least
two points of $\{a_{j}\}_{j=1}^{m}$ for each $i=1,2;$ or

(b) $k>1$, and the two endpoints of $I_{2}\ $are not simultaneously contained
in any term $\alpha_{j}$ in (\ref{ma11}), for $j=1,\dots,m;$ or

(c) $k=1$, the two endpoints of $I_{2}\ $are not simultaneously contained in
any term $\alpha_{j}$ in (\ref{ma11}), for $j=1,\dots,m,$ and either one of
$I_{1}$ and $I_{3}$ is not a point, or one of $I_{1}$ and $I_{3}$ is a point
contained in $\{a_{j}\}_{j=1}^{m}.$ \newline Then both $\left(  f,\overline
{\Delta_{1}}\right)  $ and $\left(  f,\overline{\Delta_{2}}\right)  \ $are
contained in $\mathcal{F}_{r}(L,m).$

(x) If $b_{1}\in\alpha_{1}^{\circ}\mathrm{\ }$and
\begin{equation}
L(f,I_{1})<d\left(  \left\{  q_{1},q_{2}\right\}  ,f(b_{1})\right)  ,
\label{ina1<d}%
\end{equation}
then $I_{1}$ is a point, and if in addition $k\geq1,$ (\ref{ma13}) holds,
\begin{equation}
b_{2}\not \in \left(  \alpha_{m}+\alpha_{1}+\alpha_{2}\right)  ^{\circ},
\label{noinm12}%
\end{equation}
and
\begin{equation}
L(f,I_{2})<d\left(  \left\{  q_{1},q_{2}\right\}  ,f(b_{1})\right)  ,
\label{<d}%
\end{equation}
then the conclusion of (ix) holds: $I_{2}$ cuts $\Sigma$ into two surfaces
contained in $\mathcal{F}_{r}(L,m).$
\end{lemma}

\begin{proof}
(i) trivially holds.

For each regular point $x\in\Delta$ of $f,$ there exists a disk $\left(
x,U_{x}\right)  $ of $\Sigma$ such that $f$ restricted to $U_{x}$ is a
homeomorphism onto a disk on $S.$ Then for any two points $a$ and $b$ in
$U_{x},$ the path $I(a,b)$ such that $\left(  f,I(a,b)\right)  $ is a simple
straight path is a $d_{f}$-shortest path from $a$ to $b$. On the other hand
all singular points are all contained in $f^{-1}(E_{q}).$ Therefore (ii)
follows from Lemma \ref{dfshor} (ii) and (\ref{k13}).

\label{2023-04-10-7:30}(iii) and (iv) follows from Lemma \ref{short-disk} (iv).

It follows from (iii) and (iv) that each component of $I\cap\partial\Delta$
contains at least one point of $\{a_{j}\}_{j=1}^{m}\cup\{b_{1},b_{2}\}.$ Thus,
by (i), we conclude that $I\cap\partial\Delta$ contains at most $m+2$
components, and then $k\leq m+1.$ (v) is proved.

Now, it is clear that (ii) implies (vi), and that (iii) and (iv) imply (vii)
and (viii).

To prove (ix) assume that $k\geq1$ and (\ref{ma13}) holds. Then by (vi)
$\left(  f,I_{2}\right)  $ is a simple line segment and by the assumption
$\Sigma\in\mathcal{F}_{r}(L,m)$ we have

\begin{condition}
\label{conn2}$f$ restricted to a neighborhood of $I_{2}^{\circ}$ in
$\overline{\Delta}$ is a homeomorphism.
\end{condition}

We first assume that (a) holds and write $I_{2}=I_{2}(B_{1},B_{2}).$ Then
$\Delta\backslash I_{2}$ is consisted of two Jordan domains $\Delta
_{j},j=1,2.$ By (a) $B_{1}$ and $B_{2}$ give a refinement of the partition
(\ref{ma11}) and each of the two arcs on $\partial\Delta$ with endpoints
$B_{1}$ and $B_{2}$ contains at most $m-1$ terms. Then Condition \ref{conn2}
implies that both $\Sigma_{1}=\left(  f,\overline{\Delta_{1}}\right)  $ and
$\Sigma_{2}=\left(  f,\overline{\Delta_{2}}\right)  $ are contained in
$\mathcal{F}_{r}(L,m)$ if both $L(\partial\Sigma_{1})\leq L(\partial\Sigma)$
and $L(\partial\Sigma_{2})\leq L(\partial\Sigma)\ $hold. But these two
inequalities easily follows from that $\left(  f,I_{2}\right)  $ is straight
and less than $\pi.$ We have proved (ix) when (a) holds.

To compete the proof of (ix), it suffices to show that (b) or (c) implies (a).

When $k>1,$ the terminal point $B_{2}$ of $I_{2}$ is contained in
$\{a_{j}\}_{j=1}^{m}$ by (viii), and thus the hypothesis (b) implies (a). When
$k=1,$ by (c) and (vii) at least one endpoints of $I_{2}$ is contained in
$\{a_{j}\}_{j=1}^{m}\ $and thus (c) also implies (a). (ix) is proved.

Assume $b_{1}\in\alpha_{1}^{\circ}$ and (\ref{ina1<d}) holds. Then for
$\delta=d\left(  \left\{  q_{1},q_{2}\right\}  ,f(b_{1})\right)  $
\begin{equation}
L(f,I_{1})<\delta\leq\min\{L(c_{1}\left(  q_{1},f(b_{1})\right)
,L(c_{1}\left(  f(b_{1}),q_{2}\right)  \}, \label{<delta}%
\end{equation}
and then $\left(  f,I_{1}\right)  \subset c_{1}^{\circ},$ which implies
$I_{1}\subset\alpha_{1}^{\circ}$ and thus $\partial I_{1}$ does not intersects
$\{a_{j}\}_{j=1}^{m},$ contradicting (vii) if $I_{1}$ is not a point. Thus
$I_{1}$ is a point and we have
\[
I_{1}+I_{2}=I_{2}\left(  b_{1},B_{2}\right)  =I_{2}\left(  B_{1},B_{2}\right)
.
\]

In addition to (\ref{ina1<d}), assume (\ref{ma13}), (\ref{noinm12}) and
(\ref{<d}) hold. Then Condition \ref{conn2} also holds. We show that the
condition (a) is satisfied.

First consider the case $k>1.$ Then by (viii) we have $B_{2}\in\{a_{j}%
\}_{j=1}^{m}.$ If $B_{2}=a_{1}$ or $a_{2}$, then we have $L(f,I_{2}%
)=L(f,I_{2}\left(  b_{1},B_{2}\right)  )\geq\delta$ since $b_{1}=B_{1}%
\in\alpha_{1}^{\circ}$, contradicting (\ref{<d}), and so $B_{2}\in
\{a_{j}\}_{j=3}^{m}.$ Thus (a) holds.

Assume $k=1.$ If $I_{3}$ is not a point, then by (vii) we have $b_{3}=B_{2}%
\in\{a_{j}\}_{j=1}^{m},$ and if in addition $B_{2}=a_{1}$ or $a_{2}$ we can
obtain a contradiction by (\ref{<d}) again, and thus we have $B_{2}\in
\{a_{j}\}_{j=3}^{m}$ and (a) holds again. If $I_{3}$ is a point, then
$I_{2}=I\left(  b_{1},b_{2}\right)  =I_{2}=I\left(  B_{1},B_{2}\right)  $
(note that $k=1$) and (a) also holds, by (\ref{noinm12}). We have proved that
(\ref{ina1<d}), (\ref{noinm12}), (\ref{<d}) imply (a), and then all conclusion
of (x) hold, by (ix). The lemma is proved completely.\label{2023-04-10-8:06}
\end{proof}

\begin{lemma}
\label{dfset}Let $\Sigma=\left(  f,\overline{\Delta}\right)  \ $be a surface
in $\mathcal{F},\varepsilon$ be a positive number and let $A,B$ be two compact
sets in $\overline{\Delta}.$ Then%
\begin{equation}
d_{f}(N_{f}(A,\varepsilon),N_{f}(B,\varepsilon))\geq d_{f}(A,N_{f}%
(B,\varepsilon))-\varepsilon\geq d_{f}(A,B)-2\varepsilon. \label{cc9}%
\end{equation}

\end{lemma}

\begin{proof}
It suffices to prove the second inequality. Let $a$ and $b^{\prime}$ be any
two points of $A$ and $N_{f}(B,\varepsilon).$ Then there exists $b\in B$ such
that $d_{f}(b,b^{\prime})<\varepsilon$. Then
\[
d_{f}(a,b^{\prime})\geq d_{f}(a,b)-d_{f}(b,b^{\prime})\geq d_{f}%
(a,b)-\varepsilon\geq d_{f}(A,B)-\varepsilon.
\]
This implies $d_{f}(A,N_{f}(B,\varepsilon))\geq d_{f}(A,B)-\varepsilon
$.\label{2023-04-10-8:10}
\end{proof}

\begin{lemma}
\label{Lemma1}\label{nobo}Let $L$ be a positive number in $\mathcal{L}$ (see
Definition \ref{L} for $\mathcal{L}$) with $L\geq2\delta_{E_{q}}.$ Then there
exists a positive number $\delta_{0}$ such that
\begin{equation}
d_{f}(\Delta\cap f^{-1}(E_{q}),\partial\Delta)>\delta_{0} \label{god}%
\end{equation}
holds for all surfaces $\Sigma=\left(  f,\overline{\Delta}\right)  $ in
$\mathcal{F}(L)$ with%
\[
L(\partial\Sigma)\geq\delta_{E_{q}},
\]%
\[
\Delta\cap f^{-1}(E_{q})\neq\emptyset,
\]
and with
\begin{equation}
H(\Sigma)>H_{L}-\frac{\pi}{2L\left(  \partial\Sigma\right)  }. \label{x5}%
\end{equation}

\end{lemma}

\begin{proof}
Since $L\in\mathcal{L},$ by Lemma \ref{dl}, there exists a sequence
$\delta_{L,n}$ with $0<\delta_{L,n}<\frac{1}{n}$ such that
\begin{equation}
H_{L}+\frac{1}{n}>H_{L+\delta_{L,n}}>H_{L},n=1,2,\dots. \label{x11}%
\end{equation}

Assume $\delta_{0}$ does not exists, then there exists a sequence $\Sigma
_{n}=\left(  f_{n},\overline{\Delta}\right)  \in\mathcal{F}(L)$ such that for
every $n,$
\[
L(\partial\Sigma_{n})\geq\delta_{E_{q}},
\]%
\[
\mathrm{\ }\Delta\cap f_{n}^{-1}(E_{q})\neq\emptyset,
\]%
\begin{equation}
H(\Sigma_{n})>H_{L}-\frac{\pi}{2L\left(  \partial\Sigma_{n}\right)  },
\label{x12}%
\end{equation}
and%
\[
d_{f_{n}}(a_{n},\partial\Delta)<\frac{\delta_{L,n}}{2}%
\]
for some $a_{n}\in\Delta\cap f_{n}^{-1}(E_{q}).$ Then for each $n$ there is an
arc $\alpha_{n}=\alpha\left(  a_{n},b_{n}\right)  $ in $\overline{\Delta}$
such that $b_{n}\in\partial\Delta,$ $\alpha_{n}\backslash\{b_{n}%
\}\subset\Delta$ and $f_{n}$ restricted to $\alpha_{n}$ is a homeomorphism
onto a polygonal path $\beta_{n}$ with%
\[
L(\beta_{n})<\frac{\delta_{L,n}}{2},
\]
and then the new surface $\Sigma_{n}^{\prime}$ which is obtained by cutting
$\Sigma_{n}$ along $\beta_{n}$ so that the interior of $\Sigma_{n}^{\prime}$
is equivalent to $\left(  f_{n},\Delta\backslash\alpha_{n}\right)  $ and
\[
\partial\Sigma_{n}^{\prime}=\beta_{n}+\partial\Sigma_{n}-\beta_{n}.
\]
Then for all $n=1,2,\dots,$%
\begin{equation}
R(\Sigma_{n}^{\prime})=R(\Sigma_{n})+4\pi, \label{x6}%
\end{equation}
and
\begin{equation}
2\delta_{E_{q}}<L(\partial\Sigma_{n}^{\prime})=L(\partial\Sigma_{n}%
)+2L(\beta_{n})<L(\partial\Sigma_{n})+\delta_{L,n}\leq L+\delta_{L,n},
\label{x7}%
\end{equation}
Thus by (\ref{x11})--(\ref{x7}), we have a contradiction estimation
\begin{align*}
H_{L}+\frac{1}{n}  &  \geq H_{L+\delta_{L,n}}\geq H(\Sigma_{n}^{\prime}%
)=\frac{R(\Sigma_{n}^{\prime})}{L(\partial\Sigma_{n}^{\prime})}\geq
\frac{R(\Sigma_{n})+4\pi}{L(\partial\Sigma_{n})+\delta_{L,n}}\\
&  =\frac{H(\Sigma_{n})+\frac{4\pi}{L(\partial\Sigma_{n})}}{1+\frac
{\delta_{L,n}}{L(\partial\Sigma_{n})}}\geq\frac{H_{L}-\frac{\pi}{2L\left(
\partial\Sigma_{n}\right)  }+\frac{4\pi}{L(\partial\Sigma_{n})}}%
{1+\frac{\delta_{L,n}}{L(\partial\Sigma_{n})}}\\
&  =\frac{H_{L}+\frac{7\pi}{2L\left(  \partial\Sigma_{n}\right)  }}%
{1+\frac{\delta_{L,n}}{L(\partial\Sigma_{n})}}\geq\frac{H_{L}+\frac{7\pi}{2L}%
}{1+\frac{\delta_{L,n}}{\delta_{E_{q}}}}\\
&  \rightarrow H_{L}+\frac{7\pi}{2L}\ \left(  \mathrm{as\ }n\rightarrow
0\right)  .
\end{align*}

\end{proof}

\section{Undecomposability of precise extremal sequences}

This section is to study precise extremal sequences of $\mathcal{F}%
_{r}^{\prime}(L,m),$ $\mathcal{F}_{r}(L,m),$ or $\mathcal{F}(L,m)$.

\begin{definition}
\label{extr-surf}\label{ext-seq}(1) A surface $\Sigma_{L,m}=\left(
f,\overline{\Delta}\right)  $ is called \emph{extremal in }$\mathcal{F}%
_{r}^{\prime}(L,m),$ if
\begin{equation}
H(\Sigma_{L,m})=H_{L,m}=\sup_{\Sigma\in\mathcal{F}_{r}^{\prime}(L,m)}%
H(\Sigma). \label{6-1}%
\end{equation}

(2) If in addition
\begin{equation}
H(\Sigma_{L,m})>H_{L}-\frac{\pi}{2L\left(  \partial\Sigma_{L,m}\right)  },
\label{6}%
\end{equation}
and for any other extremal surface $\Sigma$ in $\mathcal{F}_{r}^{\prime}%
(L,m)$,
\[
L(\partial\Sigma_{L,m})\leq L(\partial\Sigma),
\]
then $\Sigma_{L,m}$ is called a \emph{precise extremal} surface of
$\mathcal{F}_{r}^{\prime}(L,m).$

\emph{Extremal surfaces} and \emph{precise extremal surfaces }of
$\mathcal{F}_{r}\left(  L,m\right)  $ (or $\mathcal{F}\left(  L,m\right)  $)
are defined when the above $\mathcal{F}_{r}^{\prime}(L,m)$ is replaced by
$\mathcal{F}_{r}\left(  L,m\right)  $ (or $\mathcal{F}\left(  L,m\right)  $).
\end{definition}

Recall in Definition \ref{HL}, where we have defined precise extremal surfaces
of $\mathcal{F}\left(  L\right)  .$ In that definition, we do not require
(\ref{6}). But for a precise extremal surface $\Sigma_{L,m}$ in $\mathcal{F}%
_{r}^{\prime}(L,m),\mathcal{F}_{r}\left(  L,m\right)  $ or $\mathcal{F}\left(
L,m\right)  $, we always assume $H(\Sigma_{L,m})\ $is sufficiently close to
$H_{L}=\sup_{\Sigma\in\mathcal{F}(L,m)}H(\Sigma)$, that is, (\ref{6}) is satisfied.

\begin{remark}
\label{HH1}By Definition \ref{hs}, the equality (\ref{HH}) and Corollary
\ref{sim1}, if $m$ is large enough and $\Sigma_{L,m}$ is extremal in
$\mathcal{F}_{r}^{\prime}(L,m),$ $\mathcal{F}_{r}\left(  L,m\right)  $ or
$\mathcal{F}\left(  L,m\right)  $, then $H(\Sigma_{L,m})>H_{L}-\frac{\pi}%
{2L},$ which implies (\ref{6}).
\end{remark}

\begin{definition}
(1) A sequence $\Sigma_{n}=\left(  f_{n},\overline{\Delta}\right)
\in\mathcal{F}_{r}^{\prime}(L,m)$ is called an \emph{extremal} sequence in
$\mathcal{F}_{r}^{\prime}(L,m),$ if
\[
\lim_{n\rightarrow\infty}H(\Sigma_{n})=H_{L,m}.
\]

(2) If in addition
\begin{equation}
H(\Sigma_{n})\geq H_{L}-\frac{2\pi}{L(\partial\Sigma_{n})},n=1,2,\dots,
\label{k8}%
\end{equation}
$\lim_{n\rightarrow\infty}L(\partial\Sigma_{n})$ exists and for any other
extremal sequence $\Sigma_{n}^{\prime}$ in $\mathcal{F}_{r}^{\prime}(L,m)$,
\[
\lim_{n\rightarrow\infty}\inf L(\partial\Sigma_{n}^{\prime})\geq
\lim_{n\rightarrow\infty}L(\partial\Sigma_{n}),
\]
$\Sigma_{n}$ is called a \emph{precise extremal sequence }in\emph{
}$\mathcal{F}_{r}^{\prime}(L,m)$.

\emph{Extremal sequences} and \emph{precise extremal sequences} of
$\mathcal{F}_{r}\left(  L,m\right)  $ (or $\mathcal{F}\left(  L,m\right)  $)
are defined, if $\mathcal{F}_{r}^{\prime}(L,m)$ is replaced by $\mathcal{F}%
_{r}\left(  L,m\right)  $ (or $\mathcal{F}\left(  L,m\right)  $).
\end{definition}

\begin{lemma}
\label{Fr'FrF}Any precise extremal surface (sequence) of $\mathcal{F}%
_{r}\left(  L,m\right)  $ is a precise extremal surface (sequence) of
$\mathcal{F}(L,m),$ and any precise extremal surface (sequence) of
$\mathcal{F}_{r}^{\prime}\left(  L,m\right)  $ is a precise extremal surface
(sequence) of $\mathcal{F}_{r}(L,m)$ and $\mathcal{F}\left(  L,m\right)  .$
\end{lemma}

\begin{proof}
This follows from Corollary \ref{FF'}.
\end{proof}

Note that surfaces in an extremal sequence need not be extremal, that is, it
is possible that $H(\Sigma_{n})<H_{L,m}$ for some $n$. Lemma \ref{nobo} can be
extended to extremal sequences of $\mathcal{F}\left(  L,m\right)  .$

\begin{lemma}
\label{narrow}Let $L\geq2\delta_{E_{q}}$ be a positive number with
$L\in\mathcal{L}$. Then there exists a positive integer $m_{0}$ and a positive
number $\delta_{0}$ such that: for any $m>m_{0}$ and any sequence $\Sigma
_{n}=\left(  f_{n},\overline{\Delta}\right)  \in\mathcal{F}(L,m)$, if
\begin{equation}
\lim_{n\rightarrow\infty}H(\Sigma_{n})=H_{L,m} \label{k1}%
\end{equation}%
\begin{equation}
f_{n}^{-1}(E_{q})\cap\Delta\neq\emptyset, \label{k3}%
\end{equation}
and%
\begin{equation}
L(\partial\Sigma_{n})\geq\delta_{E_{q}}, \label{k2}%
\end{equation}
then
\[
d_{f_{n}}(f_{n}^{-1}(E_{q})\cap\Delta,\partial\Delta)\geq\delta_{0},
\]
for sufficiently large $n.$
\end{lemma}

\begin{proof}
By (\ref{k1}) $\Sigma_{n}$ is extremal in $\mathcal{F}(L,m).$ Then by Remark
\ref{HH1}, when $m$ is given large enough, for sufficiently large $n,$ we have
(\ref{6}), and then by Lemma \ref{nobo} we have the conclusion.
\end{proof}

\begin{lemma}
\label{idex}(1) For any fixed $L>0$ and any sufficiently large positive
integer $m,$ there exists a precise extremal sequence in $\mathcal{F}%
_{r}^{\prime}(L,m).$

(2) If $L<2\delta_{E_{q}},$ then for any extremal sequence $\Sigma_{n}$ of
$\mathcal{F}_{r}^{\prime}(L,m),$ $\mathcal{F}_{r}(L,m)$ or $\mathcal{F}%
(L,m),$
\[
\lim_{n\rightarrow\infty}\inf L(\partial\Sigma_{n})=L.
\]

(3) If $L\geq2\delta_{E_{q}},$ then for any extremal sequence $\Sigma_{n}$ in
$\mathcal{F}_{r}^{\prime}(L,m),$ $\mathcal{F}_{r}(L,m)$ or $\mathcal{F}%
(L,m),$
\[
\lim_{n\rightarrow\infty}\inf L(\partial\Sigma_{n})\geq2\delta_{E_{q}}.
\]

\end{lemma}

\begin{proof}
(1) By (\ref{HH}), Definition \ref{hs} and Corollary \ref{FF'}, we may assume
that $m$ is sufficiently large such that
\begin{equation}
H_{L,m}>H_{L}-\frac{\pi}{2L}. \label{k10}%
\end{equation}

Then by Definition \ref{hs}, Corollary \ref{FF'} and (\ref{k10}), there exists
an extremal sequence in $\mathcal{F}_{r}^{\prime}(L,m),$ and for any extremal
sequence $F_{n}$ in $\mathcal{F}_{r}^{\prime}(L,m),$ we have
\begin{equation}
H(F_{n})\rightarrow H_{L,m}>H_{L}-\frac{\pi}{2L}\mathrm{\ as\ }n\rightarrow
\infty. \label{k10-1}%
\end{equation}

Let $L_{0}\leq L$ be the infimum of all numbers $L^{\prime}$ such that there
exists an extremal sequence $F_{n}\ $of $\mathcal{F}_{r}^{\prime}(L,m)$ such
that
\[
\lim_{n\rightarrow\infty}\inf L(\partial F_{n})=L^{\prime}.
\]
Then there exists an extremal sequence $\Sigma_{n}$ in $\mathcal{F}%
_{r}^{\prime}(L,m)\ $such that $\lim_{n\rightarrow\infty}L(\partial\Sigma
_{n})=L_{0}.$ By (\ref{k10-1}), for sufficiently large $n_{0}$ and any $n\geq
n_{0},$ $\Sigma_{n}$ satisfies
\begin{equation}
H(\Sigma_{n})>H_{L}-\frac{\pi}{2L}\geq H_{L}-\frac{\pi}{2L\left(
\partial\Sigma_{n}\right)  }, \label{aa}%
\end{equation}
and so $\Sigma_{n}$ with $n\geq n_{0}$ is a precise extremal sequence in
$\mathcal{F}_{r}^{\prime}(L,m).$

(2) Assume $L<2\delta_{E_{q}}\ $and $\Sigma_{n}$ is an extremal sequence of
$\mathcal{F}_{r}^{\prime}(L,m),$ $\mathcal{F}_{r}(L,m)$ or $\mathcal{F}(L,m),$
with $\lim_{n\rightarrow\infty}\inf L\left(  \partial\Sigma_{n}\right)
=l_{0}<L.$ We will deduce a contradiction. Since subsequences of an extremal
sequence are also extremal, we may assume $\lim_{n\rightarrow\infty}L\left(
\partial\Sigma_{n}\right)  =l_{0}.$ Let $L^{\prime},L^{\prime\prime}\in
(l_{0},L),L^{\prime}<L^{\prime\prime}.$ Then it is clear that there exist
disks $T_{L^{\prime}},T_{L^{\prime\prime}}$ and, for each large enough $n,$ a
disk $T_{n}$ on $S$ with $T_{n}\subset T_{L^{\prime}}\subset T_{L^{\prime
\prime}}\subset S\backslash E_{q}$ such that
\[
L(\partial\Sigma_{n})=L(\partial T_{n})<L^{\prime}=L(\partial T_{L^{\prime}%
})<L^{\prime\prime}=L(\partial T_{L^{\prime\prime}}).
\]
Then by Theorem \ref{l<2dt} and Lemma \ref{hd}, we have a contradiction:
\[
H_{L,m}=\lim_{n\rightarrow\infty}H(\Sigma_{n})\leq\lim_{n\rightarrow\infty
}H(\overline{T_{n}})\leq H(\overline{T_{L^{\prime}}})<H(\overline
{T_{L^{\prime\prime}}})\leq H_{L,m},
\]
where the last inequality follows from that $\overline{T_{L^{\prime\prime}}}$
is a surface of $\mathcal{F}_{r}^{\prime}(L,m)$.

(3) Assume $L\geq2\delta_{E_{q}}$ and let $\Sigma_{n}$ be any extremal
sequence in $\mathcal{F}_{r}^{\prime}(L,m),$ or $\mathcal{F}_{r}(L,m),$ or
$\mathcal{F}(L,m)$ with $l_{0}=\lim_{n\rightarrow\infty}\inf L(\partial
\Sigma_{n})<2\delta_{E_{q}}.$ Then for some $L^{\prime}$ and $L^{\prime\prime
}$ in $(l_{0},2\delta_{E_{q}})$ with $L^{\prime}<L^{\prime\prime}$, as
discussed in (2), we have $H_{L,m}<H(\overline{T_{L^{\prime\prime}}})\leq
H_{L^{\prime\prime},m}.$ But this contradicts $L^{\prime\prime}<L$, which
implies $\overline{T_{L^{\prime\prime}}}\in\mathcal{F}_{r}^{\prime}(L,m),$ and
then $H_{L^{\prime\prime},m}\leq H_{L,m}.$
\end{proof}

\begin{lemma}
\label{ideatoo}%
Let $\Sigma_{L,m}$ be a precise extremal surface in $\mathcal{F}_{r}(L,m).$
Then there exists a precise extremal surface $\Sigma_{1}$ of $\mathcal{F}%
_{r}^{\prime}(L,m)$ such that
\[
\partial\Sigma_{1}=\partial\Sigma_{L,m}.
\]

\end{lemma}

\begin{proof}
By Corollary \ref{sim1}, there exists a surface $\Sigma_{1}\ $in
$\mathcal{F}_{r}^{\prime}(L,m)$ such that $H(\Sigma_{1})\geq H(\Sigma)$ and
$L(\partial\Sigma_{1})\leq L(\partial\Sigma).$ Since $\mathcal{F}_{r}^{\prime
}(L,m)\subset\mathcal{F}_{r}(L,m),$ $\Sigma_{1}$ is a precise extremal surface
of $\mathcal{F}_{r}^{\prime}(L,m)$ which is also precise extremal in
$\mathcal{F}_{r}(L,m).$ Thus $L(\partial\Sigma_{1})=L(\partial\Sigma),$ and
then we have $\partial\Sigma_{1}=\partial\Sigma_{L,m}$ by Theorem \ref{sim}.
\end{proof}

\begin{definition}
\label{undec-seqr}\label{decflm}Let $L\in\mathcal{L}$ be a positive number and
let $\Sigma_{n}$ be an extremal sequence of $\mathcal{F}(L,m)$. $\Sigma_{n}$
is called decomposable in $\mathcal{F}(L,m^{\prime})$ if for some positive
integer $j_{0}\geq2,$ there exists a subsequence $\Sigma_{n_{k}}$ of
$\Sigma_{n},$ a sequence $\left\{  \left\{  \Sigma_{n_{k}j}\right\}
_{j=1}^{j_{0}}\right\}  _{k=1}^{\infty}$ with $\left\{  \Sigma_{n_{k}%
j}\right\}  _{j=1}^{j_{0}}\subset\mathcal{F}(L,m^{\prime})$ and a sequence
$\varepsilon_{k}$ of positive numbers such that%
\begin{equation}
\lim_{k\rightarrow\infty}\varepsilon_{k}=0, \label{ma2}%
\end{equation}%
\begin{equation}
\sum_{j=1}^{j_{0}}R(\Sigma_{n_{k}j})\geq R(\Sigma_{n_{k}})-\varepsilon
_{k},\mathrm{\ }k=1,2,\dots, \label{ma3}%
\end{equation}%
\begin{equation}
\sum_{j=1}^{j_{0}}L(\partial\Sigma_{n_{k}j})\leq L(\partial\Sigma_{n_{k}%
})+\varepsilon_{k}, \label{ma1}%
\end{equation}
and one of the following conditions holds:

(a) For each $j\leq j_{0},$ $\lim_{k\rightarrow\infty}\inf L(\partial
\Sigma_{n_{k}j})<\lim_{k\rightarrow\infty}\inf L(\partial\Sigma_{n_{k}}).$

(b) For at least two distinct subscribe $j_{1}\ $and $j_{2}$ in $\{1,2,\dots
,j_{0}\}$%
\[
\lim_{k\rightarrow\infty}\inf L(\partial\Sigma_{n_{k}j_{i}})>0,i=1,2.
\]

(c) For some subscribe $j_{1}\leq j_{0},$
\[
\lim_{k\rightarrow\infty}\inf L(\partial\Sigma_{n_{k}j_{1}})>0\mathrm{\ and\ }%
\lim_{k\rightarrow\infty}\sup H(\Sigma_{n_{k}j_{1}})=0.
\]

\end{definition}

\begin{lemma}
\label{undec-seq}Let $L\in\mathcal{L}$ be a positive number and let
$\Sigma_{n}$ be a precise extremal sequence of $\mathcal{F}(L,m)$. Then
$\Sigma_{n}$ can't be decomposable in $\mathcal{F}(L,m^{\prime})\ $for every
positive integer $m^{\prime}\leq m.$
\end{lemma}

\begin{proof}
By Lemma \ref{idex} we have
\begin{equation}
L_{1}=\lim_{k\rightarrow\infty}\inf L(\partial\Sigma_{n_{k}})>0. \label{ma5}%
\end{equation}

Assume that the sequence $\left\{  \left\{  \Sigma_{n_{k}j}\right\}
_{j=1}^{j_{0}}\right\}  _{k=1}^{\infty}$ satisfying (\ref{ma3}), (\ref{ma1})
and one of (a)--(c) in the definition exists. We may further assume%
\[
H(\Sigma_{n_{k}1})=\max_{1\leq j\leq j_{0}}H(\Sigma_{n_{k}j}).
\]
Then for each $k,$ by (\ref{ma2})--(\ref{ma1}) and Lemma \ref{dec}%
\[
H(\Sigma_{n_{k}1})\geq\frac{H(\Sigma_{n_{k}})-\varepsilon_{k}/L(\partial
\Sigma_{n_{k}})}{1+\varepsilon_{k}/L(\partial\Sigma_{n_{k}})},
\]
which with (\ref{ma5}) implies%
\[
\lim_{k\rightarrow\infty}\inf H(\Sigma_{n_{k}1})\geq\lim_{k\rightarrow\infty
}H(\Sigma_{n_{k}})=H_{L,m}.
\]
Thus $\Sigma_{n_{k}1},k=1,2,\dots,$ is an extremal sequence of $\mathcal{F}%
(L,m)$ and in fact the equality holds (note that $\mathcal{F}\left(
L,m^{\prime}\right)  \subset\mathcal{F}(L,m)$ and we assumed $\Sigma_{n_{k}%
j}\in\mathcal{F}\left(  L,m^{\prime}\right)  $).

Assume (a) or (b) holds. Then, by (\ref{ma2}) and (\ref{ma1}), we have, for
sufficiently large $k,$%
\begin{equation}
L(\partial\Sigma_{n_{k}1})\leq L^{\prime}<L_{1}, \label{ma4}%
\end{equation}
for some $L^{\prime}<L_{1}$. Since $\Sigma_{n_{k}1}$ is extremal in
$\mathcal{F}(L,m),$ (\ref{ma5}) and (\ref{ma4}) show that $\Sigma_{n}$ is not
a precise extremal sequence in $\mathcal{F}(L,m)$. This contradicts the
assumption that $\Sigma_{n}$ is precise extremal in $\mathcal{F}(L,m)$.

Assume (c) holds. Then $1\neq j_{1},$ and then (\ref{ma4}) holds for
sufficiently $k$ and we obtain a contradiction again. We have proved that
$\Sigma_{n}$ is not decomposable in $\mathcal{F}(L,m^{\prime})$ when
$m^{\prime}\leq m.$
\end{proof}

As a direct consequence of the previous lemma we have the following.

\begin{lemma}
\label{undec}Let $L$ be a positive number. Then for any positive integer
$m^{\prime}\leq m$, any precise extremal surface $\Sigma$ of $\mathcal{F}%
(L,m)$ can't be decomposable in $\mathcal{F}(L,m^{\prime})$. That is to say,
for any integer $j_{0}\geq2,$ there does not exist a number of $j_{0}$
surfaces $\Sigma_{j},j=1,\dots,j_{0},\ $in $\mathcal{F}(L,m^{\prime})$ such
that
\[
\sum_{j=1}^{j_{0}}R(\Sigma_{j})\geq R(\Sigma),\mathrm{\ }\sum_{j=1}^{j_{0}%
}L(\partial\Sigma_{j})\leq L(\partial\Sigma),
\]
and that%
\[
L(\partial\Sigma_{j})>0,j=1,2,\dots,j_{0}.
\]
The same conclusion hold when $\mathcal{F}\left(  L,m\right)  $ and
$\mathcal{F}(L,m^{\prime})$ are both replaced by $\mathcal{F}(L).$
\end{lemma}

\section{Precise extremal sequences in $\mathcal{F}_{r}\left(  L,m\right)  $
with converging boundary}

The goal of this section is to prove the following Theorem, which plays a key
role of the proof of the main theorems.

\begin{theorem}
\label{nobi}For fixed $L\in\mathcal{L}$ with $L\geq2\delta_{E_{q}}$ and large
enough $m>3,$ let $\Sigma_{n}=\left(  f_{n},\overline{\Delta}\right)  $ be a
precise extremal sequence in $\mathcal{F}_{r}(L,m)$, and assume that the
following conditions (A)--(D) hold.

(A) For each $n=1,2,\dots,\Gamma_{n}=\partial\Sigma_{n}$ has $\mathcal{F}%
(L,m)$-partitions%
\begin{equation}
\partial\Delta=\alpha_{n1}\left(  a_{n1},a_{n2}\right)  +\alpha_{n2}%
(a_{n2},a_{n3})+\cdots+\alpha_{nm}(a_{nm},a_{n1}) \label{pk1}%
\end{equation}
and%
\begin{equation}
\Gamma_{n}=\partial\Sigma_{n}=c_{n1}\left(  q_{n1},q_{n2}\right)
+c_{n2}\left(  q_{n2},q_{n3}\right)  +\cdots+c_{nm}\left(  q_{nm}%
,q_{n1}\right)  \label{pk2}%
\end{equation}
with $c_{nj}=\left(  f,\alpha_{nj}\right)  ,j=1,\dots,m.$

(B) $\Gamma_{0}=\left(  f_{0},\partial\Delta\right)  $ is a curve on $S$ which
has $\mathcal{C}(L,m)$-partitions%
\begin{equation}
\partial\Delta=\alpha_{01}\left(  a_{01},a_{02}\right)  +\alpha_{02}%
(a_{02},a_{03})+\cdots+\alpha_{0m}(a_{0m},a_{01}), \label{pk3}%
\end{equation}
and%
\begin{equation}
\Gamma_{0}=c_{01}\left(  q_{01},q_{02}\right)  +c_{02}\left(  q_{02}%
,q_{03}\right)  +\cdots+c_{0m}\left(  q_{0m},q_{01}\right)  , \label{pk4}%
\end{equation}
with $c_{0j}=\left(  f_{0},\alpha_{0j}\right)  ,j=1,\dots,m.$

(C) $\Gamma_{n}=\partial\Sigma_{n}=\left(  f_{n},\partial\Delta\right)  $
uniformly converges to $\Gamma_{0}=\left(  f_{0},\partial\Delta\right)  ,$ and
moreover, for each $j=1,\dots,m,\alpha_{nj}$ uniformly converges to
$\alpha_{0j},$ and $c_{nj}$ uniformly converges to $c_{0j}$.

(D) For every $n=0,1,\dots,$ $\Gamma_{n}=\left(  f_{n},\partial\Delta\right)
$ are parametrized by length and $a_{n1}=1.$

Then

(i) For any pair of distinct two points $a$ and $b$ in $\partial\Delta,$
\begin{equation}
\lim_{n\rightarrow\infty}\inf d_{f_{n}}(a,b)>0. \label{mar1}%
\end{equation}

(ii) For any disjoint compact arcs $I$ and $J$ of $\partial\Delta,$
\[
\lim_{n\rightarrow\infty}\inf d_{f_{n}}(I,J)>0.
\]

\end{theorem}

\begin{center}
\label{st1}\textbf{Some conventions for the proof of Theorem \ref{nobi}.}
\end{center}

We first make some\textbf{\ }conventions on the assumptions of this theorem.
By Lemma \ref{idex} and the assumption of Theorem \ref{nobi} we have
\begin{equation}
L(\Gamma_{0})=L\left(  f_{0},\partial\Delta\right)  =\lim_{n\rightarrow\infty
}L(f_{n},\partial\Delta)=\lim_{n\rightarrow\infty}L(\partial\Sigma_{n}%
)\geq2\delta_{E_{q}}>0, \label{ap28-2}%
\end{equation}
and then by Lemmas \ref{narrow}, for the sequence $\Sigma_{n}$ in the theorem
in proof, there exists a positive number $\delta_{0}$ such that%
\begin{equation}
d_{f_{n}}(f_{n}^{-1}(E_{q})\cap\Delta,\partial\Delta)\geq\delta_{0}%
,n=1,2,\dots\label{az1}%
\end{equation}

\begin{remark}
\label{428-1}By the assumption, the restriction $f_{n}|_{\partial\Delta
}:\partial\Delta\rightarrow S$ is linear in length, say,
\begin{equation}
L(f_{n},e^{\sqrt{-1}\left[  0,\theta\right]  })=\frac{\theta}{2\pi}%
L(f_{n},\partial\Delta),\theta\in\lbrack0,2\pi],n=0,1,2,\dots, \label{ap28-1}%
\end{equation}
where $e^{\sqrt{-1}\left[  0,\theta\right]  }$ denotes the arc $\{e^{\sqrt
{-1}t}:t\in\lbrack0,\theta]\}$ on $\partial\Delta.$ It is permitted that some
$\alpha_{0j},$ as the limit of $\alpha_{nj},$ is a point, and by (C), (D) and
(\ref{ap28-2}) it is clear that $\alpha_{0j}$ is a point iff $c_{0j}$ is a
point, for each $j=1,\dots,m.$
\end{remark}

\begin{remark}
\label{ap11}For each $j_{0}=2,\dots,m$ and each $n=0,1,2,\dots,$ let $\phi
_{n}(z)=a_{nj_{0}}z$ be the rotation of $\mathbb{C}$ (note that $a_{nj_{0}%
}=e^{\sqrt{-1}\theta_{nj_{0}}}\in\partial\Delta$ for some $\theta_{nj_{0}}%
\in\lbrack0,2\pi)$). Then the sequence $\Sigma_{n}^{j_{0}}=\left(  f_{n}%
\circ\phi_{n},\overline{\Delta}\right)  $ is also a precise extremal sequence
of $\mathcal{F}_{r}(L,m)$ and
\begin{equation}
\partial\Delta=\alpha_{n1}^{j_{0}}\left(  a_{n1}^{j_{0}},a_{n2}^{j_{0}%
}\right)  +\alpha_{n2}^{j_{0}}\left(  a_{n2}^{j_{0}},a_{n3}^{j_{0}}\right)
+\cdots+\alpha_{nm}^{j_{0}}\left(  a_{nm}^{j_{0}},a_{n1}^{j_{0}}\right)  ,
\label{ap28}%
\end{equation}
with
\[
a_{nj}^{j_{0}}=\frac{a_{n,j_{0}+j-1}}{a_{nj_{0}}}%
\]
for $j=1,\dots,m,$ are $\mathcal{F}(L,m)$-partitions of $\partial\Sigma
_{n}^{j_{0}}$ with $n\geq1,$ or is a $\mathcal{C}(L,m)$-partition for
$\Gamma_{0}^{j_{0}}=\left(  f_{0}\circ\phi_{0},\partial\Delta\right)  .$ Since
$a_{n1}^{j_{0}}=1$ for all $n=0,1,\dots,$ and $a_{nj_{0}}$ converges to
$a_{0j_{0}}$ as $n\rightarrow\infty,$ it is clear that the partition
(\ref{ap28}), the sequence $\Sigma_{n}^{j_{0}}$ and the curve $\Gamma
_{0}^{j_{0}}$ satisfy all hypothesis of the theorem in proof as the partition
(\ref{pk1}), the sequence $\Sigma_{n}$ and the curve $\Gamma_{0}$. On the
other hand we have%
\[
d_{f_{n}\circ\phi_{n}}\left(  \phi_{n}^{-1}(a),\phi_{n}^{-1}(b)\right)
=d_{f_{n}}\left(  a,b\right)  ,n=1,2,\dots,
\]
for any pair of two points $a$ and $b$ in $\overline{\Delta},$ and by (C) at
least one arc $\alpha_{0j}$ is not a point. Therefore, to prove (i) of the
theorem, we may always assume that the following conditions hold (by omitting
at most a finite number of terms of $\left\{  \Sigma_{n}\right\}
_{n=1}^{\infty}$).
\end{remark}

\begin{condition}
\label{fa2}$\alpha_{01}$ is not a point and that point $a$ in (\ref{mar1}) is
fixed and contained in $\alpha_{01}.$
\end{condition}

\begin{condition}
\label{da}If $a\in\alpha_{01}^{\circ},$ then we have: (a) there exists a
compact arc
\[
\alpha_{a}=\alpha_{a}\left(  a^{\prime},a^{\prime\prime}\right)  =\alpha
_{01}\left(  a^{\prime},a^{\prime\prime}\right)  \subset\alpha_{01}^{\circ}%
\]
such that $\alpha_{a}$ is a neighborhood of $a$ in $\alpha_{01}^{\circ},$ that
for each $n\in\mathbb{N}^{0}$ and the two points $q_{n}^{\prime}%
=f_{n}(a^{\prime}),q_{n}^{\prime\prime}=f_{n}(a^{\prime\prime}),$ the arc
\[
c_{n,a}=c_{n1}\left(  q_{n}^{\prime},q_{n}^{\prime\prime}\right)  =\left(
f_{n},\alpha_{a}\right)  \subset c_{n1}^{\circ}%
\]
is a compact neighborhood of $f_{n}(a)$ in $c_{n1}^{\circ},$ and%
\begin{equation}
L(c_{n,a})>\pi\mathrm{\ if}\;L(c_{01})>\pi. \label{ma7-2}%
\end{equation}

(b) there exists a number $\delta_{a}\in\left(  0,\delta_{0}\right)  \ $such
that for each$\mathrm{\ }n=0,1,2,\dots,$%
\begin{equation}
\alpha_{a}\subset\alpha_{n1}\backslash D\left(  \left\{  a_{n1},a_{n2}%
\right\}  ,4\delta_{a}\right)  \mathrm{\ }\text{\textrm{and}}\mathrm{\ }%
c_{n,a}\subset c_{n1}\backslash D\left(  \left\{  q_{n1},q_{n2}\right\}
,4\delta_{a}\right)  , \label{ma7-1}%
\end{equation}%
\begin{equation}
L(c_{n1}\left(  q_{n1},q_{n}^{\prime}\right)  )=L(c_{n1}\left(  q_{n}%
^{\prime\prime},q_{n2}\right)  )>4\delta_{a}. \label{ma7}%
\end{equation}
Here $D\left(  Q,r\right)  =\cup_{x\in Q}D(x,r)$ is the $r$-neighborhood of
$Q$ on the sphere $S$ and $\partial\Delta$ is regarded as a set on $S.$
\end{condition}

In fact, Condition \ref{da} follows from (A)--(D) and Condition \ref{fa2}.

If $a\in\alpha_{01}^{\circ},$ then for every positive number $\varepsilon
\leq4\delta_{a}$ (see Condition \ref{da}) we introduce:

\begin{notation}
\label{note1}$\alpha_{01,\varepsilon}$ is the largest connected and compact
neighborhood of $\alpha_{01}$ in $\partial\Delta$ such that
\[
\left(  f_{0},\alpha_{01,\varepsilon}\backslash\alpha_{01}\right)
\subset\overline{D\left(  \{q_{01},q_{02}\},\varepsilon\right)  }.
\]

\end{notation}

\begin{lemma}
\label{l1}(i) For any arc $\gamma$ on $\partial\Delta$ with distinct
endpoints, we have
\[
L(f_{n},\gamma)=\frac{L(\gamma)}{2\pi}L(f_{n},\partial\Delta)>0,n=0,1,2,\dots
,
\]
and for any two arcs $\gamma_{1}$ and $\gamma_{2}$ on $\partial\Delta$ with
$L(\gamma_{1})\leq L(\gamma_{2})$ we have%
\[
L(f_{n},\gamma_{1})\leq L(f_{n},\gamma_{2}),n=0,1,2,\dots.
\]

(ii) If $\{x_{n}\}$ and $\{y_{n}\}$ are two sequence of points on
$\partial\Delta$ such that $\lim_{n\rightarrow\infty}d\left(  x_{n}%
,y_{n}\right)  =0,$ then $\lim_{n\rightarrow\infty}d_{f_{n}}\left(
x_{n},y_{n}\right)  =0.$

(iii) If the given point $a$ is contained in $\alpha_{01}^{\circ}$, then for
the points $a^{\prime},a^{\prime\prime},q_{n}^{\prime},q_{n}^{\prime\prime
},n=0,1,2,\dots,$ in Condition \ref{da} we have,
\begin{align*}
\lim_{n\rightarrow\infty}L(\alpha_{n1}\left(  a_{n1},a^{\prime}\right)  )  &
=L(\alpha_{01}\left(  a_{01},a^{\prime}\right)  )>4\delta_{a},\mathrm{\ }\\
\lim_{n\rightarrow\infty}L(\alpha_{n1}\left(  a^{\prime\prime},a_{n2}\right)
)  &  =L(\alpha_{01}\left(  a^{\prime\prime},a_{02}\right)  )>4\delta_{a};\\
\lim_{n\rightarrow\infty}L(c_{n1}\left(  q_{n1},q_{n}^{\prime}\right)  )  &
=L(c_{01}\left(  q_{01},q_{0}^{\prime}\right)  )>4\delta_{a},\\
\lim_{n\rightarrow\infty}L(c_{n1}\left(  q_{n}^{\prime\prime},q_{n2}\right)
)  &  =L(c_{01}\left(  q_{0}^{\prime\prime},q_{02}\right)  )>4\delta_{a}.
\end{align*}

(iv) If $a\in\alpha_{01}^{\circ}\ $and $\{b_{n}\}_{n=1}^{\infty}\ $is a
sequence in $(\partial\Delta)\backslash\alpha_{01,3\delta_{a}}^{\circ}$ with
\begin{equation}
\lim_{n\rightarrow\infty}d_{f_{n}}\left(  a,b_{n}\right)  \rightarrow0,
\label{-0}%
\end{equation}
then for sufficiently large $n,$
\begin{equation}
d_{f_{n}}\left(  a,b_{n}\right)  <\min\{\delta_{0},\mathrm{\ }\inf\left\{
L(f_{n},\gamma_{0}\left(  a,b_{n}\right)  \right\}  _{n=0}^{\infty}),
\label{ddd}%
\end{equation}
where $\gamma_{0}(a,b_{n})$ is the shorter\footnote{Since $\left(
f_{n},\partial\Delta\right)  $ are parametrized by length, this also means
$L(f_{n},\gamma_{0})\leq L(f_{n},\left(  \partial\Delta\right)  \backslash
\gamma_{0}).$} arc of $\partial\Delta$ with endpoints $a$ and $b_{n}.$
\end{lemma}

\begin{proof}
In fact, (i)--(iii) follows from (A)--(D), and Conditions \ref{fa2} and
\ref{da}. Note that Condition \ref{fa2} (b) holds for $n=0,1,\dots.$

Assume that the given point $a$ is contained in $\alpha_{01}^{\circ}.$ Then
for the sequence $\left\{  b_{n}\right\}  _{n=1}^{\infty}\subset
(\partial\Delta)\backslash\alpha_{01,3\delta_{a}}^{\circ},$ we have by
(\ref{ma7-1}) that $L(\gamma_{0}\left(  a,b_{n}\right)  )>4\delta_{a},$ and
thus by Remark \ref{428-1} we have
\[
L\left(  f_{n},\gamma_{0}\left(  a,b_{n}\right)  \right)  >\frac{4\delta_{a}%
}{2\pi}L(f_{n},\partial\Delta)>0
\]
for $n=0,1,\dots,n.$ Hence, by (\ref{ap28-2}), we have $\inf\left\{
L(f_{n},\gamma_{0}\left(  a,b_{n}\right)  \right\}  _{n=0}^{\infty}>0,$ which
implies (\ref{ddd}) for sufficiently large $n.$\medskip
\end{proof}

\begin{center}
\textbf{Step 1.} \textbf{Some useful results inspiring Theorem \ref{nobi}}
\end{center}

In this step we will prove (\ref{mar1}) in some special cases, and we will
deduce some contradictions under the condition that (\ref{mar1}) fails. These
simple results inspire Theorem \ref{nobi}, though the complete proof of
Theorem \ref{nobi} is very complicated. $\medskip\medskip$

\begin{lemma}
\label{ccl1}For any two points $x$ and $y$ on $\partial\Delta$ with
$f_{0}(x)\neq f_{0}(y),$ (\ref{mar1}) holds, say,
\begin{equation}
\lim_{n\rightarrow\infty}\inf d_{f_{n}}(x,y)\geq d(f_{0}(x),f_{0}(y))>0.
\label{aug1}%
\end{equation}

\end{lemma}

\begin{proof}
By (B)--(D) we have
\[
d_{f_{n}}(x,y)\geq d(f_{n}(x),f_{n}(y))\rightarrow d(f_{0}(x),f_{0}(y)),
\]
and then we have (\ref{aug1}).$\medskip$
\end{proof}

\begin{lemma}
\label{obs2}For the given point $a$ in $\alpha_{01},$ if there exists a number
$\delta>0$ such that each surface $\Sigma_{n}$ contains a disk\footnote{See
Definition \ref{nod}.} $\left(  a,U_{n}^{\delta}\right)  $ of radius $\delta$,
then (\ref{mar1}) holds for all $b\in\left(  \partial\Delta\right)
\backslash\{a\}.$
\end{lemma}

\begin{proof}
This follows from Lemma \ref{d>d}. In fact we can write
\[
\partial U_{n}^{\delta}=\alpha_{1,n}^{\delta}+\alpha_{2,n}^{\delta}%
+\alpha_{3,n}^{\delta}=\alpha_{1,n}^{\delta}\left(  a_{1,n}^{\delta},a\right)
+\alpha_{2,n}^{\delta}\left(  a,a_{2,n}^{\delta}\right)  +\alpha_{3,n}%
^{\delta}\left(  a_{2,n}^{\delta},a_{1,n}^{\delta}\right)
\]
as in Lemma \ref{cov-1}, so that $-\alpha_{1,n}^{\delta}$ and $\alpha
_{2,n}^{\delta}$ are the boundary radius and $\alpha_{3,n}^{\delta}$ is the
new boundary (see Definition \ref{nod} and Remark \ref{nod-1} for the radius,
the boundary radius, the new and the old boundary of a disk of a surface). For
$j=1,2,$ since the spherical distance of the two endpoints of $c_{j,n}%
^{\delta}=\left(  f_{n},\alpha_{j,n}^{\delta}\right)  $ equals $\delta,$ we
have, by (B)--(D), that $c_{j,n}^{\delta}=\left(  f_{n},\alpha_{j,n}^{\delta
}\right)  $ converges uniformly to an SCC arc $c_{j,0}^{\delta},$ and thus
$\alpha_{j,n}^{\delta}$ converges to an arc $\alpha_{j,0}^{\delta}$ on
$\partial\Delta,$ for $j=1,2.$

It is clear that for any $r\in(0,1],$ $\left(  a,U_{n}^{\delta}\right)  $
contains the disk $\left(  a,U_{n}^{r\delta}\right)  $ of radius $r\delta$ of
$\Sigma_{n},$ the corresponding $\alpha_{j,n}^{r\delta}$ are well defined with
$\partial U_{n}^{r\delta}=\alpha_{1,n}^{r\delta}+\alpha_{2,n}^{r\delta}%
+\alpha_{3,n}^{r\delta}$ for $j=1,2,3$ and $n=0,1,2,\dots,$ and the above
argument applies when $\delta$ is replaced by $r\delta.$ Moreover,
$\alpha_{1,0}^{r\delta}+\alpha_{2,0}^{r\delta}$ uniformly converges to the
point $a$ as $r\rightarrow0.$ Thus, for any $b\in\left(  \partial
\Delta\right)  \backslash\{a\},$ there exists $r\in(0,1)$ such that
$b\notin\alpha_{1,0}^{r\delta}+\alpha_{2,0}^{r\delta}.$ Then for sufficiently
large $n,$ $b\notin\alpha_{1,n}^{r\delta}+\alpha_{2,n}^{r\delta},$ which
implies $b\notin\left(  a,U_{n}^{r\delta}\right)  ,$ and thus by Lemma
\ref{d>d} we have $d_{f_{n}}\left(  a,b\right)  >r\delta>0.$ Therefore
(\ref{mar1}) holds for all $b\in\partial\Delta$ with $b\neq a.$
\end{proof}

Lemma \ref{obs2} can be used to show Lemma \ref{ccl2}, which state that
(\ref{mar1}) holds when $a\in\alpha_{01}^{\circ}$ and $\alpha_{01}$ either is
strictly convex or is straight and $L(c_{01})>\pi.$ Lemma \ref{ccl2} is the
second key ingredient of the proof of Theorem \ref{nobi}. By Lemma \ref{ccl1}
and Condition \ref{da} we have the following lemma.

\begin{lemma}
\label{ccl6}If the given point $a$ is contained in $\alpha_{01}^{\circ},$ then
for any $b\in\alpha_{01,4\delta_{a}}\backslash\{a\}$%
\[
\lim_{n\rightarrow\infty}\inf d_{f_{n}}\left(  a,b\right)  \geq d\left(
f_{0}(a),f_{0}(b)\right)  >0.
\]

\end{lemma}

\begin{lemma}
\label{in-bd-1}Let $a\in\alpha_{01}^{\circ}$ and let $b_{n}$ be a sequence in
$\left(  \partial\Delta\right)  \backslash\alpha_{01,3\delta_{a}}^{\circ}$
which satisfies (\ref{-0}). Then for sufficiently large $n$ and the $d_{f_{n}%
}$-shortest path $I_{n}=I_{n}\left(  a,b_{n}\right)  $ from $a$ to $b_{n}$%
\begin{equation}
I_{n}\cap\Delta\neq\emptyset\mathrm{\ but\ }I_{n}\cap\Delta\cap f_{n}%
^{-1}(E_{q})=\emptyset. \label{neq0=0}%
\end{equation}

\end{lemma}

\begin{proof}
Let $\gamma_{0}\left(  a,b_{n}\right)  $ be the shorter arc of $\partial
\Delta$ with $\partial\gamma_{0}=\{a,b_{n}\}.$ Then for sufficiently large
$n,$ (\ref{ddd}) holds. If $I_{n}\cap\Delta=\emptyset$ for some $n=n_{0}$
which is so large that (\ref{ddd}) holds for this $n_{0},$ then we have
$I_{n_{0}}\subset\partial\Delta$ and thus
\begin{align*}
d_{f_{n_{0}}}\left(  a,b_{n_{0}}\right)   &  =L\left(  f_{n_{0}},I_{n_{0}%
}\right)  \geq L\left(  f_{n_{0}},\gamma_{0}\left(  a,b_{n_{0}}\right)
\right) \\
&  \geq\min\{\delta_{0},\mathrm{\ }\inf\left\{  L(f_{n},\gamma_{0}\left(
a,b_{n}\right)  \right\}  _{n=0}^{\infty}),
\end{align*}
contradicting to (\ref{ddd}). Thus, when $n$ is large enough, $I_{n}\cap
\Delta\neq\emptyset,$ and for any $x\in I_{n}\cap\Delta,$ we have
\[
d_{f_{n}}\left(  x,\partial\Delta\right)  \leq L(f_{n},I_{n})=d_{f_{n}}\left(
a,b_{n}\right)  <\delta_{0},
\]
and so by (\ref{az1}) we have $I_{n}\cap\Delta\cap f_{n}^{-1}(E_{q}%
)=\emptyset.$
\end{proof}

\label{deleted Lemma obs3}

\begin{lemma}
\label{ap7}(\textbf{The first key step for the proof of Theorem \ref{nobi})
}Let $\Sigma_{n}=\left(  f_{n},\overline{\Delta}\right)  $ be the precise
extremal sequence satisfying all conditions (A)--(D) in Theorem \ref{nobi}.
Assume that the following additional condition hold:

(E) $a\in\alpha_{01}^{\circ}$, $\{b_{n}\}$ is a sequences in $\left(
\partial\Delta\right)  \backslash\alpha_{0,3\delta_{a}}^{\circ},$ and for the
$d_{f_{n}}$-shortest path $I_{n}=I_{n}\left(  a,b_{n}\right)  $ from $a$ to
$b_{n}$,
\begin{equation}
\lim_{n\rightarrow\infty}\inf d_{f_{n}}\left(  a,b_{n}\right)  =\lim
_{n\rightarrow\infty}\inf\left(  f_{n},I_{n}\right)  =0. \label{--0}%
\end{equation}

Then $\Sigma_{n}=\left(  f_{n},\overline{\Delta}\right)  $ contains a
subsequence which is still denoted by $\Sigma_{n}=\left(  f_{n},\overline
{\Delta}\right)  $ such that

(i) $I_{n}\cap\Delta\neq\emptyset,$ and $I_{n}$ has a partition
\[
I_{n}=I_{n1}\left(  a,b_{n1}\right)  +I_{n2}\left(  b_{n1},b_{n2}\right)
+\dots+I_{n,2k+1}\left(  b_{n,2k},b_{n}\right)
\]
given by Lemma \ref{shortest} (v): $I_{n,2j-1}\subset\partial\Delta$ for
$j=1,\dots,k+1$, and $I_{n,2j}^{\circ}\subset\Delta\ $for $j=1,\dots,k.$
Moreover, $k\geq1$ is independent of $n.$

(ii) $I_{n,2j}^{\circ}\cap f_{n}^{-1}\left(  E_{q}\right)  =\emptyset,$
$f_{n}$ restricted to a neighborhood of $I_{n,2j}^{\circ}$ is a homeomorphism,
$\left(  f_{n},I_{n,2j}\right)  $ is a straight arc on $S,j=1,\dots,k,$ and
moreover $I_{n2}$ divides $\Delta$ into two Jordan domains $\Delta_{n1}$ and
$\Delta_{n2}$, say, $I_{n2}$ divides $\Sigma_{n}$ into two surfaces
$F_{nj}=\left(  f_{n},\overline{\Delta_{nj}}\right)  ,j=1,2.$

(iii) $I_{n1}$ is a point. Thus $I_{n2}=I_{n2}\left(  b_{n1},b_{n2}\right)
=I_{n2}\left(  a,b_{n2}\right)  =I_{n}\left(  a,b_{n2}\right)  .$

(iv) $b_{n2}\in\{a_{nj}\}_{j=3}^{m}$ and $F_{nj}$ are both contained in
$\mathcal{F}_{r}\left(  L,m\right)  ,$ provided that%
\begin{equation}
\left\{  b_{n},b_{n2}\right\}  \cap\{a_{nj}\}_{j=1}^{m}\neq\emptyset.
\label{inaj}%
\end{equation}

(v) $F_{nj}$ are both contained in $\mathcal{F}_{r}\left(  L,m\right)  ,$
provided that
\begin{equation}
b_{n2}\not \in \left(  \alpha_{nm}+\alpha_{n1}+\alpha_{n2}\right)  ^{\circ}
\label{noin1m}%
\end{equation}

(vi) $\Sigma_{n}$ is decomposable in $\mathcal{F}_{r}(L,m)$, provided that
(\ref{inaj}) or (\ref{noin1m}) holds for each $n=1,2,\dots.$
\end{lemma}

\begin{proof}
By (\ref{--0}), taking subsequence, we may assume that (\ref{-0}) holds. Then
(i) follows from Lemma \ref{shortest} and Lemma \ref{in-bd-1}.

By Lemma \ref{in-bd-1}, we have $I_{n,2j}\cap f^{-1}(E_{q})=\emptyset$ for
sufficiently large $n$ and each $j=1,\dots,2k.$ Thus (ii) follows from Lemma
\ref{shortest}, and $\tau_{n}=\left(  f_{n},I_{n2}\right)  $ is a simple and
straight arc on $S$.

By (\ref{-0}) and (\ref{ma7-1}), omitting a finite number of terms of the
sequence $\{I_{n}\},$ we may assume%
\begin{equation}
L(f_{n},I_{n1}\left(  a,b_{n1}\right)  )<\delta_{a}<\min\{L\left(
c_{n1}(q_{n1},f_{n}(a)\right)  ),L\left(  c_{n1}(f_{n}(a),q_{n2}\right)  )\}.
\label{<da}%
\end{equation}
Then by $I_{n1}\subset\partial\Delta$ we have $\left(  f_{n},I_{n1}\right)
\subset c_{n1}^{\circ}\ $and $I_{n1}\subset\alpha_{n1}^{\circ},$ and so
$I_{n1}\cap\{a_{nj}\}_{j=1}^{m}=\emptyset.$ Then by Lemma \ref{shortest} (vii)
$I_{n1}=\{a\}$ is a singleton and (iii) is proved.

By (\ref{-0}) and (ii) we have%
\begin{equation}
\lim_{n\rightarrow\infty}d_{f_{n}}\left(  b_{n1},b_{n2}\right)  =\lim
_{n\rightarrow\infty}d_{f_{n}}\left(  a,b_{n2}\right)  =\lim_{n\rightarrow
\infty}\overline{f\left(  a\right)  f\left(  b_{n2}\right)  }=0. \label{---0}%
\end{equation}
Assume $b_{n2}\in\{a_{nj}\}_{j=1}^{m}$. For sufficiently large $n,$ by
(\ref{---0}) we have $b_{n2}\neq a_{n1},a_{n2},$ since
\[
d\left(  f_{n}\left(  a\right)  ,\left\{  f_{n}\left(  a_{n1}\right)
,f_{n}\left(  a_{n2}\right)  \right\}  \right)  >4\delta_{a}%
\]
by (\ref{ma7-1}); and thus $b_{n2}\in\{a_{nj}\}_{j=3}^{m}.$ Assume $b_{n}%
\in\{a_{nj}\}_{j=1}^{m}$. Then we have $b_{n2}\in\{a_{nj}\}_{j=1}^{m}$ as well
by Lemma \ref{shortest}, and then for sufficiently large $n$ we also have
$b_{n2}\in\{a_{nj}\}_{j=3}^{m},$ by (\ref{---0}).

Now assume (\ref{inaj}) holds. Then by the above discussion we may assume
$b_{n2}\in\{a_{nj}\}_{j=3}^{m},$ and then both of the two arcs of
$\partial\Delta$ with common endpoints $a$ and $b_{n2}$ contain at least two
points of $\{a_{nj}\}_{j=1}^{m},$ and thus $F_{n1}$ and $F_{n2}$ are both
contained in $\mathcal{F}_{r}(L,m)$ by Lemma \ref{shortest} (ix), and (iv) is
proved completely.

If (\ref{noin1m}) holds, then each of the two arcs of $\partial\Delta$ with
common endpoints $a$ and $b_{n2}$ contains at least two points of
$\{a_{nj}\}_{j=1}^{m}\ $again. Thus we also have $F_{nj}\in\mathcal{F}%
_{r}(L,m),j=1,2,$ by Lemma \ref{shortest} (ix), and (v) is proved.

Now assume (\ref{inaj}) holds. Then $F_{n1}$ and $F_{n2}$ are both contained
in $\mathcal{F}_{r}(L,m).$ On the other hand by (i) and (ii) and (\ref{-0}) we
have%
\[
R(F_{n1})+R(F_{n2})=R(\Sigma_{n}),
\]
and%
\[
L(\partial F_{n1})+L(\partial F_{n2})=L(\partial\Sigma_{n})+2L(\tau_{n}),
\]
with
\[
\lim_{n\rightarrow\infty}L(\tau_{n})=\lim_{n\rightarrow\infty}L(f_{n}%
,I_{n})=0.
\]

Let $\gamma_{nj},j=1,2,$ be the two arcs of $\partial\Delta$ with endpoints
$a$ and $b_{n2}$ such that $\gamma_{nj}=\left(  \partial\Delta_{nj}\right)
\cap\partial\Delta.$ Then by (iv) and (\ref{ma7-1}) we have
\begin{align*}
\lim_{n\rightarrow\infty}L(\partial F_{nj})  &  =\lim_{n\rightarrow\infty
}\left(  L(f_{n},\gamma_{nj})+L(\tau_{n})\right) \\
&  \geq\min\{L\left(  c_{n1}(q_{n1},f_{n}(a)\right)  ),L\left(  c_{n1}%
(f_{n}(a),q_{n2}\right)  )\}\geq4\delta_{a},j=1,2.
\end{align*}
Therefore, $\Sigma_{n}$ is decomposable in $\mathcal{F}_{r}(L,m)\ $by
Definition \ref{undec-seqr} (b).

If (\ref{noin1m}) holds for each large enough $n,$ the above argument can be
used to show that $\Sigma_{n}$ is also decomposable in $\mathcal{F}_{r}(L,m).$
This completes the proof of Lemma \ref{ap7}.
\end{proof}

\begin{corollary}
\label{out-3d}For any $a\in\alpha_{01}^{\circ}$ and $b\in\partial
\Delta\backslash\alpha_{0,3\delta_{a}}^{\circ},$ if the interior $I_{n}%
^{\circ}$ of the $d_{f_{n}}$-shortest path $I_{n}=I_{n}\left(  a,b\right)  $
is contained in $\Delta$ and if $b\in\partial\Delta\backslash\left(
\alpha_{nm}+\alpha_{n1}+\alpha_{n2}\right)  ^{\circ}.$ Then $\lim
_{n\rightarrow\infty}\inf d_{f_{n}}\left(  a,b\right)  >0.$
\end{corollary}

\begin{proof}
If $\lim_{n\rightarrow\infty}\inf d_{f_{n}}\left(  a,b\right)  =0,$ then we
have a contradiction by Lemma \ref{undec-seq} and Lemma \ref{ap7} (vi) with
$b_{n}=b$ for every $n=1,2,\dots$
\end{proof}

\begin{center}
\label{st2}\textbf{Step 2. Proof of Theorem \ref{nobi} (i) in a special case}
\end{center}

This step is to prove the following Lemma, which is the second key to Prove
Theorem \ref{nobi} (i).

\begin{lemma}
\label{ccl2}(\textbf{The second key to the proof of Theorem \ref{nobi} (i))
}For any $j=1,2,\dots,m,$ if $L(c_{0j})>\pi,$ or if $L(c_{0j})>0$ and $c_{j}$
is strictly convex, then%
\[
\lim_{n\rightarrow\infty}\inf d_{f_{n}}(x,b)>0
\]
for every $x\in\alpha_{0j}^{\circ}$ and every $b\in\left(  \partial
\Delta\right)  \backslash\{x\}.$
\end{lemma}

\begin{proof}
By Remark \ref{ap11} and Lemma \ref{ccl6}, it suffices to prove Lemma
\ref{ccl2} for $j=1,$ $x=a\in\alpha_{01}^{\circ}$ (the fixed point) and each
$b\in\left(  \partial\Delta\right)  \backslash\alpha_{01,3\delta_{a}}\ $(see
Condition \ref{da}).

Let $C_{nj}$ be the circle determined by $c_{nj},$ for $j=1,\dots
,m,n=0,1,\dots.$ If $c_{01}$ is strictly convex, then we may assume, by taking
subsequence of $\Sigma_{n}$, that all $c_{n1}$ are strictly convex, and then
$c_{n,a}+\overline{q_{n}^{\prime\prime}q_{n}^{\prime}}$ encloses a closed
domain $T_{n},$ for all $n=0,1,\dots,$ and it is clear that $\overline
{q_{n}^{\prime\prime}q_{n}^{\prime}}$ divides the closed disk enclosed by
$C_{n1}$ into two lunes, and $T_{n}$ is the closure of the lune on the right
hand side of $\overline{q_{n}^{\prime}q_{n}^{\prime\prime}}\ $(see Condition
\ref{da} for the notation $c_{n,a}=c_{n,a}\left(  q_{n}^{\prime},q_{n}%
^{\prime\prime}\right)  $). If $c_{01}$ is straight, then by the assumption
$L(c_{01})>\pi$ and Condition \ref{da}, all $c_{n,a},n=0,1,2,\dots,$ have
length $>\pi$ and thus $d\left(  q_{n}^{\prime},q_{n}^{\prime\prime}\right)
<\pi,$ and then $c_{n,a}+\overline{q_{n}^{\prime\prime}q_{n}^{\prime}}$ also
encloses a closed and convex domain $T_{n},$ and $T_{n}$ is in fact a closed
hemisphere with $\overline{q_{n}^{\prime}q_{n}^{\prime\prime}}\subset C_{n1}$
when $c_{n1}$ and $c_{01}$ are straight.

For $n=0,1,2,\dots,$ let $\theta_{n}$ be the interior angle of $T_{n}$ at the
cusps $q_{n}^{\prime}$ and $q_{n}^{\prime\prime}$ and for each $\theta
\in\lbrack0,\theta_{n}]$ let $c_{n,a,\theta}$ be the circular arc in $T_{n}$
from $q_{n}^{\prime}$ to $q_{n}^{\prime\prime}$ so that the closed domain
$T_{n,\theta}$ enclosed by $c_{n,a}-c_{n,a,\theta}$ has interior angle
$\theta$ at the cusps $q_{n}^{\prime}$ and $q_{n}^{\prime\prime}.$ Then
$T_{n,\theta}\subset T_{n},$ $T_{n,0}$ is just the arc $c_{n,a}=c_{n,a,0}$ and
$T_{n,\theta_{n}}=T_{n},$ for $n=0,1,2,...$ It is clear that $c_{n,a,\theta}$
is strictly convex when $\theta\in(0,\theta_{n}),$ whether $c_{01}$ is
straight or not. On the other hand by Lemma \ref{ccircle} we may assume that%
\begin{equation}
\theta_{n}\geq\frac{\theta_{0}}{2}=\left\{
\begin{array}
[c]{c}%
\pi/2,\mathrm{\ if\ }c_{01}\mathrm{\ is\ straight,}\\
\theta_{0}/2>0,\mathrm{\ if\ }c_{01}\mathrm{\ is\ strictly\ convex},
\end{array}
\right.  \mathrm{\ \ }n=1,2,\dots\label{app1}%
\end{equation}
Note that each $c_{nj}$ is convex and so it is either straight or strictly convex.

Since $\alpha_{a}$ given in Condition \ref{da} is a compact subarc of all
$\alpha_{n1}^{\circ},n=0,1,2\dots,$ by Definition of $\mathcal{F}_{r}(L,m)$
and Lemma \ref{int-arg1}, for each $n=1,2,\dots,$ and sufficiently small
$\theta=\theta(n)\in(0,\theta_{n}],$ $f_{n}^{-1}$ has a univalent branch
$g_{n,\theta}$ defined on the closed domain $T_{n,\theta}$ such that
$g_{n,\theta}\left(  c_{n,a}\right)  =\alpha_{a}$. Let $\theta_{n}^{\ast}$ be
the largest positive number in $(0,\theta_{n}]$ such that $g_{n,\theta}$ is
well defined for every $\theta\in(0,\theta_{n}^{\ast})$, say, $g_{n,\theta}$
is a univalent branch of $f_{n}^{-1}$ defined on $T_{n,\theta}$ with
$g_{n,\theta}\left(  c_{n,a,0}\right)  =\alpha_{a}.$ Then it is clear that
$g_{n,\theta^{\prime}}$ equals $g_{n,\theta^{\prime\prime}}$ on $T_{n,\theta
^{\prime}}\subset T_{n,\theta^{\prime\prime}}$ for every pair of
$\theta^{\prime}$ and $\theta^{\prime\prime}$ with $0<\theta^{\prime}%
<\theta^{\prime\prime}<\theta_{n}^{\ast},$ and then $f_{n}^{-1}$ has a
univalent branch $g_{n,\theta_{n}^{\ast}}$ defined on $T_{n,\theta_{n}^{\ast}%
}\backslash c_{n,a,\theta_{n}^{\ast}}^{\circ}\ $with $g_{n,\theta_{n}^{\ast}%
}\left(  c_{n,a}\right)  =\alpha_{a}.$ This, together with Lemma
\ref{continue0}, implies that $g_{n,\theta_{n}^{\ast}}$ can extend to a
univalent branch of $f_{n}^{-1}$ defined on the closed Jordan domain
$T_{n,\theta_{n}^{\ast}}$. We still use $g_{n,\theta_{n}^{\ast}}$ to denote
the extension and let $\alpha_{a,\theta_{n}^{\ast}}=g_{n,\theta_{n}^{\ast}%
}\left(  c_{n,a,\theta_{n}^{\ast}}\right)  ,$ which is a simple arc in
$\overline{\Delta}$ from $a^{\prime}$ to $a^{\prime\prime}.$ We summarize this
by a claim:

\begin{claim}
\label{ap9}$f_{n}:D_{n}=g_{n}(T_{n,\theta_{n}^{\ast}})\rightarrow
T_{n,\theta_{n}^{\ast}}$ is a homeomorphism and $\partial D_{n}=\alpha
_{a}-\alpha_{a,\theta_{n}^{\ast}}$ is a Jordan curve, and thus $D_{n}^{\circ
}\cup\alpha_{a}=D_{n}\backslash\alpha_{a,\theta_{n}^{\ast}}^{\circ}$ contains
no branch point of $f_{n}.$
\end{claim}

Since $\alpha_{a}=\alpha(a^{\prime},a^{\prime\prime})$ is a compact subarc of
all $\alpha_{n1}^{\circ},n\in\mathbb{N},$ by Lemma \ref{int-arg1} and Claim
\ref{ap9}, we have

\begin{claim}
\label{a'a''}For each $n\in\mathbb{N},\alpha_{n1}^{\circ}$ has a neighborhood
$N_{n}$ in $\overline{\Delta}$ such that $f_{n}:D_{n}\cup N_{n}\rightarrow
T_{n,\theta_{n}^{\ast}}\cup f_{n}\left(  N_{n}\right)  $ is also a
homeomorphism. Thus, if $\theta_{n}^{\ast}\in\left(  0,\theta_{n}\right)  ,$
the part of $\alpha_{a,\theta_{n}^{\ast}}^{\circ}$ near its endpoints
$a^{\prime}$ and $a^{\prime\prime}$ is contained in $N_{n}^{\circ}%
\subset\Delta.$
\end{claim}

We will prove%
\begin{equation}
\underline{\theta}=\lim_{n\rightarrow\infty}\inf\theta_{n}^{\ast}>0.
\label{fa4}%
\end{equation}

We first show that this implies Lemma \ref{ccl2}. Since $\theta_{n}^{\ast}>0$
for each $n\geq1,$ (\ref{fa4}) implies $\theta_{n}^{\ast}>\varphi_{0}$ for
some $\varphi_{0}>0$ and all $n\geq1.$ By Lemma \ref{ccircle} $T_{n,\theta
_{n}}$ and $T_{n,\varphi_{0}}$ converge to $T_{0,\theta_{0}}$ and
$T_{0,\varphi_{0}}.$ Thus it is clear that there is a positive number
\[
\delta_{1}<\left\{  \delta_{0},\delta_{a}\right\}
\]
such that for the $\delta_{1}$-neighborhood $V_{n}=D\left(  f_{n}%
(a),\delta_{1}\right)  \cap T_{n,\varphi_{0}}$ of $f_{n}(a)$ in $T_{n,\varphi
_{0}}\ $and large enough $n,$ $V_{n}$ is the part of the disk $D\left(
f_{n}(a),\delta_{1}\right)  $ on the left hand side of $c_{n,a}\ $with
$D\left(  f_{n}(a),\delta_{1}\right)  \cap c_{n,a}\subset V_{n}$ and $V_{n}$
does not intersects $c_{n,a,\varphi_{0}}.$ Thus, when we choose $\delta_{1}$
small enough, $(a,U_{n})$ with $U_{n}=g_{n,\theta_{n}^{\ast}}(V_{n})$ is a
disk of $\Sigma_{n}$ with radius $\delta_{1},$ $U_{n}\subset D_{n}$ and
$U_{n}\cap\partial\Delta\subset\alpha_{a}$, but $(a,U_{n})$ depends on $n$
(see Definition \ref{nod}). Therefore by Lemma \ref{obs2}, (\ref{mar1}) holds,
and thus (\ref{fa4}) implies Lemma \ref{ccl2}.

Now return to prove (\ref{fa4}), and assume that it fails. Then, by
(\ref{app1}), we may assume
\begin{equation}
\underline{\theta}=\lim_{n\rightarrow\infty}\theta_{n}^{\ast}%
=0\mathrm{\ and\ }\theta_{n}^{\ast}<\inf\{\theta_{n}\}_{n=1}^{\infty}
\label{v1}%
\end{equation}
for all $n$, and then for sufficiently large $n$ and each point $y$ in
$T_{n,\theta_{n}^{\ast}}$, there exists a line segment $I_{n,y}=\overline
{yy^{\ast}}\subset T_{n,\theta_{n}^{\ast}}$ of length $<\theta_{n}^{\ast}\pi$
for some $y^{\ast}\in c_{n,a}.$ Thus we have
\begin{equation}
d_{f_{n}}\left(  x,\partial\Delta\right)  \leq d_{f_{n}}\left(  x,\alpha
_{a}\right)  \leq L(I_{n,f_{n}(x)})<\theta_{n}^{\ast}\pi\rightarrow0
\label{d->0}%
\end{equation}
for all $x\in D_{n}\backslash\alpha_{a}$ and sufficiently large $n,$ and thus
by (\ref{az1}) we may assume that

\begin{claim}
\label{CCL8}$D_{n}\cap\Delta\cap f_{n}^{-1}(E_{q})=\emptyset$ for all
$n=1,2,\dots,$ and so $f_{n}$ has no any branch point in $D_{n}\cap\Delta.$
Thus $\alpha_{a,\theta_{n}^{\ast}}\cap\Delta\cap f_{n}^{-1}(E_{q})=\emptyset$
for all $n=1,2,\dots.$
\end{claim}

Since $0<\theta_{n}^{\ast}<\theta_{n},$ when $\alpha_{a,\theta_{n}^{\ast}%
}^{\circ}\cap\partial\Delta=\emptyset,$ $f_{n}$ is homeomorphic in a
neighborhood of $\alpha_{a,\theta_{n}^{\ast}}$ in $\overline{\Delta}$ by
Claims \ref{a'a''} and \ref{CCL8}, and so $\theta_{n}^{\ast}$ can be enlarge,
which contradicts the definition of $\theta_{n}^{\ast}.$ Hence $\alpha
_{a,\theta_{n}^{\ast}}^{\circ}\cap\partial\Delta$ is not empty, which together
with Lemma \ref{tangent} (iv) implies that $\alpha_{a,\theta_{n}^{\ast}%
}^{\circ}\cap\partial\Delta$ is a finite set contained in $\{a_{j}\}_{j=1}%
^{m},$ and thus we have by Claim \ref{a'a''}

\begin{claim}
\label{CCL10}$\alpha_{a,\theta_{n}^{\ast}}\cap\partial\Delta=\{a^{\prime
},a^{\prime\prime}\}\cup\{a_{ni_{j}}\}_{j=1}^{k-1}$ in which $\{a_{ni_{j}%
}\}_{j=1}^{k-1}=\alpha_{a,\theta_{n}^{\ast}}^{\circ}\cap\partial\Delta$ for
some $k\geq2$ is a subset of $\{a_{nj}\}_{j=1}^{m},$ and $a^{\prime\prime
},a_{ni_{1}},a_{ni_{2}},\dots,a_{ni_{k-1}},a^{\prime}$ are arranged on
$\partial\Delta$ anticlockwise.
\end{claim}

It is clear that $\{a_{ni_{j}}\}_{j=1}^{k-1}\subset D_{n}$ since $D_{n}$ is
closed, and thus by (\ref{d->0}) $d_{f_{n}}\left(  a_{ni_{j}},a_{\alpha
}\right)  $ $<\theta_{n}^{\ast}\pi\rightarrow0$ for every $j=1,\dots,k-1.$ On
the other hand, we have $d_{f_{n}}\left(  \{a_{n1},a_{n2}\},\alpha_{a}\right)
$ $\geq d\left(  \{q_{n1},q_{n2}\},c_{\alpha}\right)  >4\delta_{a}$ by
(\ref{ma7-1}). Thus for large enough $n$ we have
\begin{equation}
a_{ni_{j}}\notin\{a_{n1},a_{n2}\}\mathrm{\ for\ each\ }j=1,\dots,k-1.
\label{2-side}%
\end{equation}
Then $\alpha_{n,\theta_{n}^{\ast}}$ cut $\Delta$ into $k+1$ components
$\left\{  \Delta_{nj}\right\}  _{j=0}^{k}$ of which $\Delta_{n0}$ is on the
left hand side of $-\alpha_{n,\theta_{n}^{\ast}}$ and $\Delta_{nj}%
,j=1,\dots,k,$ are on the right hand side of $-\alpha_{n,\theta_{n}^{\ast}},$
and the intersections $\gamma_{j}=\overline{\partial\Delta_{nj}}\cap
\partial\Delta,j=1,\dots,k,$ are arcs of $\partial\Delta$ arranged on
$\partial\Delta$ anticlockwise. We may assume $k$ is independent of $n.$

Since $c_{a,n}$ and $c_{a,\theta_{n}^{\ast}}$ are both convex and share the
same endpoints $\{q_{n}^{\prime},q_{n}^{\prime\prime}\}$, $c_{a,n}$ is on the
convex circle $C_{n1}\ $and $c_{a,\theta_{n}^{\ast}}^{\circ}$ is in the domain
enclosed by $C_{n1},$ we have
\begin{equation}
L(f_{n},\alpha_{a,\theta_{n}^{\ast}})=L(c_{a,\theta_{n}^{\ast}})<L(c_{a,n}).
\label{LL<L}%
\end{equation}
Then $k,\alpha_{a},\alpha_{n1},-\alpha_{a,\theta_{n}^{\ast}},\Sigma
_{n},a^{\prime\prime},a^{\prime},\left\{  \{a_{i_{j}}\}\right\}  _{j=1}%
^{k-1}\ $satisfy (a) (b) (c) (d) (e1) of Lemma \ref{m-1} as $k,\gamma
_{0},\gamma_{0}^{\prime},I,\Sigma,b_{2},b_{2k+1},\{I_{2j-1}\}_{j=2}^{k}$
there, but here all $I_{2j-1}=\{a_{i_{j}}\}$ are points. Then by Lemma
\ref{m-1}, $\Sigma_{nj}=\left(  f_{n},\overline{\Delta_{nj}}\right)
\in\mathcal{F}_{r}\left(  L,m\right)  $ for $j=1,2,\dots,k.$

It is clear that by Claim \ref{CCL8} and (\ref{LL<L}),
\begin{align*}
\sum_{j=1}^{k}R(\Sigma_{nj})  &  =R(\Sigma_{n})-\left(  q-2\right)
A(T_{n,\theta_{n}^{\ast}}),\\
\sum_{j=1}^{k}L(\partial\Sigma_{nj})  &  \leq L(\partial\Sigma)-L(c_{a,n}%
)+L(c_{a,\theta_{n}^{\ast}})<L(\partial\Sigma),
\end{align*}
with $A(T_{n,\theta_{n}^{\ast}})\rightarrow0$ as $n\rightarrow\infty.$ On the
other hand it is clear that $\alpha_{n1}\left(  a^{\prime\prime}%
,a_{n2}\right)  \subset\partial\Delta_{n1}$ and by (\ref{ma7}),
\[
L(f_{n},\alpha_{n1}\left(  a^{\prime\prime},a_{n2}\right)  )=L(c_{n1}\left(
q_{n}^{\prime\prime},q_{n2}\right)  )\geq4\delta_{a}.
\]
Thus $L(\partial\Sigma_{n1})>4\delta_{a}$ and, for the same reason,
$L(\partial\Sigma_{nk})>4\delta_{a}.$ Hence $\Sigma_{n}$ is decomposable by
Definition \ref{undec-seqr} with (b), which contradicts Lemma \ref{undec-seq}.
We have proved (\ref{fa4}), and then Lemma \ref{ccl2} is proved
completely.\medskip\medskip
\end{proof}

\begin{center}
\label{st3} \textbf{Step 3. Preliminary discussion in the case }$a\in
\alpha_{01}^{\circ}$
\end{center}

The purpose of this and next steps is to prove Assertion \ref{asC} introduced
later, which deduce Theorem (i) in the case $a\in\alpha_{01}^{\circ}.$

By Condition \ref{fa2} and Lemma \ref{ccl1}, to prove Theorem \ref{nobi} (i),
it suffices to prove the following assertion.

\begin{Assertion}
\label{asA}For the fixed $a\in\alpha_{01}$ and any $b\in\left(  \partial
\Delta\right)  \backslash\{a\},$ we have%
\begin{equation}
\lim_{n\rightarrow\infty}\inf d_{f_{n}}(a,b)>0. \label{ma17}%
\end{equation}

\end{Assertion}

We first prove this assertion under the condition
\begin{equation}
a\in\alpha_{01}^{\circ}. \label{zz2}%
\end{equation}
Then by Lemma \ref{ccl6}, to prove Assertion \ref{asA} under (\ref{zz2}), it
suffices to prove

\begin{Assertion}
\label{asB}(\ref{ma17}) holds when $a\in\alpha_{01}^{\circ}$ and $b\in\left(
\partial\Delta\right)  \backslash\alpha_{01,4\delta_{a}}^{\circ}.$
\end{Assertion}

Assume that (\ref{zz2}) holds. Then we may assume that $\alpha_{n1}%
\subset\alpha_{01,\delta_{a}}$ for all $n$ (see Notation \ref{note1} for
$\alpha_{01,\delta_{a}}$)$,$ since this is true for sufficiently large $n.$

By Lemma \ref{dfcon}, for each $n,$ there exists a point $b_{n}$ in $\left(
\partial\Delta\right)  \backslash\alpha_{01,3\delta_{a}}^{\circ}$ such that
\begin{equation}
d_{f_{n}}(a,b_{n})=\min_{x\in\left(  \partial\Delta\right)  \backslash
\alpha_{01,3\delta_{a}}^{\circ}}d_{f_{n}}(a,x). \label{min-1}%
\end{equation}
By Lemma \ref{shortest}, the $d_{f_{n}}$-shortest path $I_{n}=I_{n}\left(
a,b_{n}\right)  $ with $b_{n}\in\left(  \partial\Delta\right)  \backslash
\alpha_{01,3\delta_{a}}^{\circ}$ exists. Then, to prove Assertion \ref{asB},
it suffices to prove the following assertion.

\begin{Assertion}
\label{asC}When the fixed $a$ is contained in $\alpha_{01}^{\circ},$%
\[
\lim_{n\rightarrow\infty}\inf L(f_{n},I_{n}\left(  a,b_{n}\right)
)=\lim_{n\rightarrow\infty}\inf d_{f_{n}}(a,b_{n})>0.
\]

\end{Assertion}

To prove Assertion \ref{asC}, we may assume that
\begin{equation}
\lim_{n\rightarrow\infty}b_{n}=b_{0}\in\left(  \partial\Delta\right)
\backslash\alpha_{01,3\delta_{a}}^{\circ}, \label{lim=b0}%
\end{equation}
exists and thus, by Condition \ref{da} (b), we have
\begin{equation}
a\in\alpha_{n1}^{\circ}\cap\alpha_{01}^{\circ}\text{ \textrm{and} }b_{n}%
\in\left(  \partial\Delta\right)  \backslash\alpha_{01,3\delta_{a}}^{\circ
},\mathrm{\ for\ }n=0,1,2,\dots\label{ma18}%
\end{equation}

Assume Assertion \ref{asC} fails. Then by Lemma \ref{l1} (ii) we may assume,
by taking subsequence, that
\begin{equation}
\lim_{n\rightarrow\infty}d_{f_{n}}\left(  a,b_{0}\right)  =\lim_{n\rightarrow
\infty}d_{f_{n}}(a,b_{n})=\lim_{n\rightarrow\infty}L(f_{n},I_{n}\left(
a,b_{n}\right)  )=0. \label{az4}%
\end{equation}

If $b_{0}\in\alpha_{01,4\delta_{a}},$ then by (\ref{lim=b0}) we have $b_{0}%
\in\alpha_{01,4\delta_{a}}\backslash\{a\}$, and then by Lemma \ref{ccl6} we
have $\lim_{n\rightarrow\infty}\inf d_{f_{n}}\left(  a,b_{0}\right)  >0.$ This
contradicts (\ref{az4}), and so $b_{0}\in\left(  \partial\Delta\right)
\backslash\alpha_{01,4\delta_{a}},$ and so, by taking subsequence,
(\ref{ma18}) can be enhanced to be%
\begin{equation}
a\in\alpha_{n1}^{\circ}\cap\alpha_{01}^{\circ}\text{ \textrm{and} }b_{n}%
\in\left(  \partial\Delta\right)  \backslash\alpha_{01,4\delta_{a}%
},\mathrm{\ for\ }n=0,1,2,\dots\label{ma18-1}%
\end{equation}

By (\ref{ma18-1}) and Lemma \ref{in-bd-1}, we have
\begin{equation}
I_{n}\cap\Delta\neq\emptyset,\mathrm{\ but\ }I_{n}\cap\Delta\cap f_{n}%
^{-1}(E_{q})=\emptyset\mathrm{\ for\ }n\mathrm{\ large\ enough}, \label{ag64}%
\end{equation}
and then by Lemma \ref{ap7}, taking subsequence, we conclude that $I_{n}$ has
a partition
\begin{equation}
I_{n}=I_{n1}\left(  a,b_{n1}\right)  +I_{n2}\left(  b_{n1},b_{n2}\right)
+I_{n3}\left(  b_{n2},b_{n3}\right)  +\cdots+I_{n,2k+1}\left(  b_{n,2k}%
,b_{n,2k+1}\right)  \label{a7}%
\end{equation}
(with $b_{n,2k+1}=b_{n}$) satisfying all conclusions of Lemma \ref{ap7} for
each $n,$ in which $k\geq1$ is independent of $n.$ Then

\begin{claim}
\label{a=b1I2noeq}$I_{n1}\ $is the point $a$ in $\alpha_{01}^{\circ}%
,I_{n2}^{\circ}\subset\Delta$, $\left(  f_{n},I_{n2}\right)  $ is straight,
$\partial I_{n2}=I_{n2}\cap\partial\Delta,$ and
\begin{equation}
I_{n2}^{\circ}\cap f^{-1}(E_{q})=\emptyset,n=1,2,\dots, \label{ma19}%
\end{equation}

\end{claim}

It is clear that $I_{n2}$ divides $\Delta$ into two Jordan domains
$\Delta_{1n}$ and $\Delta_{n2}$ and we let $\Delta_{n1}$ be the one on the
right hand side of $I_{n2},$ and write%
\[
F_{nj}=\left(  f_{n},\overline{\Delta_{nj}}\right)  ,j=1,2.
\]

We may assume
\begin{equation}
\lim_{n\rightarrow\infty}b_{n2}=b_{2}\mathrm{\ for\ some\ }b_{2}\in
\partial\Delta. \label{lim=b2}%
\end{equation}

Now, it is clear that to prove Assertion \ref{asC}, by taking subsequence, we
may assume that one of the following holds:

\noindent\textbf{Case 1. }$k=1\ $and $b_{n2}\in\{a_{nj}\}_{j=1}^{m}$ for each
$n=1,2,\dots,$ or $k\geq2.$

\noindent\textbf{Case 2. }$k=1\ $and $b_{n2}\notin\{a_{nj}\}_{j=1}^{m},$ for
each $n=1,2,\dots,$ but $b_{2}\in\{a_{0j}\}_{j=1}^{m}.$

\noindent\textbf{Case 3. }$k=1\ $and $b_{n2}\notin\{a_{nj}\}_{j=1}^{m},$ for
each $n=1,2,\dots,$ and $b_{2}\notin\{a_{0j}\}_{j=1}^{m}.$\medskip\medskip

\begin{center}
\label{st4}\textbf{Step 4. Case 1, 2, or 3 implies a contradiction}
\end{center}

We will deduce a contradiction in each of Cases 1--3.

\begin{center}
\textbf{Step 4.1 Discussion of Case 1}
\end{center}

\begin{claim}
\label{ccl5}Case 1 cannot occur.
\end{claim}

\begin{proof}
If $k>1$ then $b_{n2}\in\{a_{nj}\}_{j=1}^{m}$ by Lemma \ref{shortest}. Thus
$\{b_{n},b_{n2}\}$ satisfies (\ref{inaj}) in Case 1, and so by Lemma \ref{ap7}
(vi), $\Sigma_{n}$ is decomposable in $\mathcal{F}_{r}\left(  L,m\right)  ,$
contradicting to Lemma \ref{undec-seq}.\medskip
\end{proof}

\begin{center}
\textbf{Step 4.2 A general discussion of Case 2 and Case 3}
\end{center}

Now assume that Case 2 or Case 3 occurs. Then $k=1.$ We may assume $I_{n3}$ is
a point, for otherwise by Lemma \ref{shortest} $b_{n2}\in\{a_{nj}\}_{j=1}^{m}$
and Case 1 occurs. Then $b_{n2}=b_{n3}=b_{n}$ and thus, by Claim
\ref{a=b1I2noeq}, for $n=1,2,\dots,$%
\begin{equation}
I_{n}=I_{n}\left(  a,b_{n}\right)  =I_{n2}=I_{n2}(b_{n1},b_{n2})=I_{n2}%
(a,b_{n}), \label{428-2}%
\end{equation}
and then by (\ref{ma18-1}) we have $b_{n2}=b_{n}\in\left(  \partial
\Delta\right)  \backslash\alpha_{01,4\delta_{a}}.$ If $b_{n2}\notin\left(
\alpha_{nm}+\alpha_{n1}+\alpha_{n2}\right)  ^{\circ}\ $holds for infinitely
many $n,$ then by (\ref{az4}) and Lemma \ref{ap7} (v) $\Sigma_{n}$ is
decomposable in $\mathcal{F}\left(  L,m\right)  ,$ contradicting Lemma
\ref{undec-seq}. We then can assume, for each $n,b_{n}=b_{n2}\ $is in
$\alpha_{n2}^{\circ}\backslash\alpha_{01,4\delta_{a}}$ or $\alpha_{nm}^{\circ
}\backslash\alpha_{01,4\delta_{a}},$ and thus we can further assume
\begin{equation}
b_{n}=b_{n2}\in\alpha_{n2}^{\circ}\backslash\alpha_{01,4\delta_{a}%
},n=1,2,\dots. \label{ap1}%
\end{equation}
Then we have,
\begin{equation}
b_{0}=b_{2}=\lim_{n\rightarrow\infty}b_{n}=\lim_{n\rightarrow\infty}b_{n2}%
\in\alpha_{02}\backslash\alpha_{01,4\delta_{a}}^{\circ},n=1,2,\dots.
\label{ap1-1}%
\end{equation}
By the way we see that $\alpha_{02},$ as the limit of $\alpha_{n2},$ is not a
point, for otherwise $\alpha_{n2}\backslash\alpha_{01,4\delta_{a}}$ is empty
for all large enough $n.$

By definition of $\mathcal{F}_{r}\left(  L,m\right)  ,f_{n}$ is homeomorphic
in a neighborhood of $\alpha_{nj}^{\circ}$ for each $j=1,\dots,m,$ which with
(\ref{ma18-1}) and (\ref{ap1}) implies that $f_{n}$ is homeomorphic in some
neighborhoods of the endpoints of $I_{n}=I_{n}\left(  a,b_{n}\right)  .$ On
the other hand, for large enough $n,$ $f_{n}$ is homeomorphic in a
neighborhood of $I_{n2}^{\circ}\ $by Lemma \ref{ap7} (ii). Therefore by
(\ref{ag64}) we conclude that

\begin{condition}
\label{homI2}$f_{n}$ is homeomorphic in a neighborhood of $I_{n2}\left(
a,b_{n2}\right)  =I_{n}\left(  a,b_{n}\right)  $ in $\overline{\Delta}$ and
$I_{n2}^{\circ}\cap f_{n}^{-1}(E_{q})=\emptyset.$
\end{condition}

\begin{center}
\textbf{Step 4.3 Complete the discussion of Case 2}
\end{center}

\begin{claim}
\label{ccl5-1}Case 2 implies a contradiction.
\end{claim}

\begin{proof}
Now assume Case 2 occurs. Then by the condition of Case 2 and (\ref{ap1-1}),
\begin{equation}
b_{n}=b_{n2}\rightarrow b_{2}=a_{03}\in\alpha_{02}\backslash\alpha
_{01,4\delta_{a}}^{\circ}\mathrm{\ as\ }n\rightarrow\infty, \label{b2=a03}%
\end{equation}
which with the condition $\lim_{n\rightarrow\infty}a_{n3}=a_{03}$ implies, for
sufficiently large $n,$
\begin{equation}
a_{n3}\in\alpha_{n2}\backslash\alpha_{01,3\delta_{a}}^{\circ}. \label{in-3d}%
\end{equation}
By (\ref{b2=a03}) and the condition $\lim_{n\rightarrow\infty}a_{n3}=a_{03}$
we have $\lim_{n\rightarrow\infty}d\left(  b_{n},a_{n3}\right)  =0,$ which
with Lemma \ref{l1} (i) implies $\lim_{n\rightarrow\infty}d_{f_{n}}\left(
b_{n},a_{n3}\right)  =0.$ Thus by (\ref{az4}) we have
\[
\lim_{n\rightarrow\infty}d_{f_{n}}\left(  a,a_{n3}\right)  \leq\lim
_{n\rightarrow\infty}d_{f_{n}}\left(  a,b_{n}\right)  +\lim_{n\rightarrow
\infty}d_{f_{n}}\left(  b_{n},a_{n3}\right)  =0.
\]
Hence by (\ref{in-3d}) and Lemma \ref{ap7} (vi), by taking subsequence, we may
assume that the $d_{f_{n}}$-shortest path $\tilde{I}_{n}=\tilde{I}_{n}\left(
a,a_{n3}\right)  $ from $a$ to $a_{n3}$ has a partition%
\[
\tilde{I}_{n}=\tilde{I}_{n1}\left(  a,\tilde{b}_{n1}\right)  +\tilde{I}%
_{n2}\left(  \tilde{b}_{n1},\tilde{b}_{n2}\right)  +\tilde{I}_{n3}\left(
\tilde{b}_{n2},\tilde{b}_{n3}\right)  +\cdots+\tilde{I}_{n,2\tilde{k}%
+1}\left(  \tilde{b}_{n,2\tilde{k}},a_{n3}\right)  .
\]
satisfying all conclusions of Lemma \ref{ap7} for each $n,$ in which
$1\leq\tilde{k}\leq m+1$ is independent of $n.$ Then $\{a_{n3},\tilde{b}%
_{n2}\}\cap\{a_{nj}\}_{j=1}^{m}\neq\emptyset$, say (\ref{inaj}) holds, and
thus by Lemma \ref{ap7} (vi) $\Sigma_{n}$ is decomposable in $\mathcal{F}%
_{r}\left(  L,m\right)  ,$ contradicting to Lemma \ref{undec-seq}%
.\medskip\medskip
\end{proof}

\begin{center}
\textbf{Step 4.4 Discussion of Case 3: (1) Case 3 implies Condition
\ref{cond3-31}}
\end{center}

Now we assume that Case 3 occurs. Then by Condition of Case 3 and
(\ref{ap1-1}) we have $b_{2}\in\alpha_{02}^{\circ}$ and thus by (\ref{ap1-1})
we have%
\[
b_{2}=\lim_{n\rightarrow\infty}b_{n2}\in\alpha_{02}^{\circ}\backslash
\alpha_{01,4\delta_{a}},
\]
and then by (\ref{ap1}) we may assume
\[
\left\{  b_{2},b_{n2}\right\}  \subset\left[  \alpha_{02}^{\circ}\cap
\alpha_{n2}^{\circ}\right]  \backslash\alpha_{01,4\delta_{a}},n=1,2,\dots.
\]
Then
\begin{equation}
a\neq b_{2},b_{n2}\mathrm{\ for\ }n=1,2,\dots. \label{may1}%
\end{equation}

By (\ref{az4}) and (\ref{ap1-1}), we have
\begin{equation}
\lim_{n\rightarrow\infty}d_{f_{n}}\left(  a,b_{2}\right)  =0. \label{may2}%
\end{equation}
Then by Lemma \ref{ccl1} we have $f_{0}(a)=f_{0}(b_{2}),$ and then
\[
f_{0}\left(  a\right)  =f_{0}\left(  b_{2}\right)  \in c_{01}^{\circ}\cap
c_{02}^{\circ}.
\]
Thus by (\ref{may1}), (\ref{may2}) and Lemma \ref{ccl2}, we may assume
$c_{0j}=\left(  f_{0},\alpha_{0j}\right)  $ are both straight and
$L(c_{0j})\leq\pi,$ for $j=1,2.$ Then Case 3 can be discussed under the
following condition.

\begin{condition}
\label{cond3-31}Both $c_{01}$ and $c_{02}$ are straight, $f_{0}(a)=f_{0}%
\left(  b_{2}\right)  \in c_{01}^{\circ}\cap c_{02}^{\circ},$ $L(c_{01}%
)\leq\pi\ $and $L(c_{02})\leq\pi.\medskip$
\end{condition}

\begin{center}
\textbf{Step 4.5 Discussion of Case 3: (2) The closed Jordan domains }$K_{n1}
$, $K_{n2}$ \textbf{and }$K_{n,\psi_{n}^{\ast}}$
\end{center}

Let $C_{nj}$ be the circle determined by $c_{nj}$ for $n=0,1,\dots,$ and
$j=1,2.$ Then $C_{01}$ and $C_{02}$ are great circles on $S.$ By
(\ref{ma7-1}), we have
\[
f_{n}(a)\in c_{n1}\backslash D\left(  \left\{  q_{n1},q_{n2}\right\}
,3\delta_{a}\right)  \rightarrow c_{01}\backslash D\left(  \left\{
q_{01},q_{02}\right\}  ,3\delta_{a}\right)  \ni f_{0}(a)
\]
as $n\rightarrow\infty,$ which implies that $f_{0}(a)$ and $q_{02}$ are not
antipodal on $S$ by Condition \ref{cond3-31}. Then we have $\{q_{02}%
,f_{0}(a)\}\subset C_{01}\cap C_{02},$ which implies $C_{01}=-C_{02},$ say,
$c_{01}+c_{02}$ is folded at $q_{02}.$ Let $A_{n}=f_{n}(a)\ $and $B_{n}%
=f_{n}(b_{n}).$ Then $\tau_{n}=\left(  f_{n},I_{n}\right)  =\left(
f_{n},I_{n2}\right)  =\overline{A_{n}B_{n}}\rightarrow\{f_{0}(a)\}$ by
(\ref{az4}), and we see that, by (\ref{ap1}), $\alpha_{n2}^{\circ}%
\backslash\alpha_{01,3\delta_{a}}$ is an open arc of $\alpha_{n2}^{\circ}$
containing $b_{n}$ (see Notation \ref{note1} for $\alpha_{01,3\delta_{a}}$).
Then by (\ref{min-1}) and Condition \ref{homI2} we must have $\tau_{n}\perp
c_{n2}$ at $B_{n},$ where $\tau_{n}\perp c_{n2}$ indicates that $\tau_{n}$
intersects $c_{n2}$ at $B_{n}$ perpendicularly.

Since $\tau_{n}\perp c_{n2}\ $at $B_{n}\ $and $\tau_{n}\ $is on the left hand
side of $c_{n1}\ $and $c_{n2},$ we see that
\[
\eta_{n}=c_{n1}\left(  A_{n},q_{n2}\right)  +c_{n2}\left(  q_{n2}%
,B_{n}\right)  -\tau_{n}\left(  A_{n},B_{n}\right)
\]
is a convex topological triangle, enclosed by the two convex circular subarcs
of $c_{n1}$ and $c_{n2}$ and the line segment $\tau_{n}.$ It is clear that
$\eta_{n}$ converges to the folded arcs $\overline{f_{0}\left(  a\right)
f_{0}\left(  a_{02}\right)  }+\overline{f_{0}\left(  a_{02}\right)
f_{0}\left(  a\right)  }$ as $n\rightarrow\infty.$ Then $c_{n1}+c_{n2}$ is
strictly convex at $q_{n2}$, the interior angle of $\eta_{n}$ at $A_{n}$ is
almost $\pi/2,$ $C_{n1}$ and $C_{n2}$ enclose a "thin" and convex closed lens
$K_{n}$ which is on the left hand side of $c_{n1}+c_{n2}$, and $\tau_{n} $
divides $K_{n}$ into two closed Jordan domains $K_{n1}$ and $K_{n2}$ such that
$K_{n1}$ is on the right hand side of $\tau_{n},$ for $n=1,2,\dots.$ Then
$\partial K_{n1}=\eta_{n}$ and both $K_{n1}$ and $K_{n2}$ are contained in
some open hemispheres, respectively. It is clear that one of the two cusps of
$K_{n}$ is $q_{n2},$ and we let $q_{n2}^{\ast}$ be the other cusp.

Assume the interior center of $C_{n2}$, the center on the left hand side of
$C_{n2},$ is $P_{n}\ $and $C_{n2}=\partial D(P_{n},R_{n}).$ Since the
orientation of $C_{n2}$ is given by $c_{n2}$ and all $c_{nj}$ are convex for
all $n=0,1,\dots,$ and all $j=1,2,\dots,m,$ we have $R_{n}\leq\frac{\pi}{2}$.
Let $R_{n,\psi}$ be the radius $\overline{P_{n}B_{n,\psi}}$ of $C_{n2}$ such
that $B_{n,\psi}$ is the point on the arc $\partial K_{n}\cap C_{n2}$ so that
the angle between the radius $\overline{P_{n}q_{n2}}$ and $\overline
{P_{n}B_{n,\psi}}$ at $P_{n}$ equals $\psi$, $A_{n,\psi}$ be the intersection
of $R_{n,\psi}$ and $c_{n1},$ and write $R_{n,\psi_{n}}=\overline{P_{n}%
q_{n2}^{\ast}}\ $and $R_{n,\psi_{n,a}}=\overline{P_{n}B_{n}}.\ $Then
$R_{n,0}=\overline{P_{n}q_{2n}}$ and $0<\psi_{n,a}<\psi_{n}.$ Let
\begin{equation}
\tau_{n,\psi}=\tau_{n,\psi}\left(  A_{n,\psi},B_{n,\psi}\right)
=\overline{A_{n,\psi}B_{n,\psi}}=K_{n}\cap R_{n,\psi},\psi\in\lbrack0,\psi
_{n}], \label{ap4}%
\end{equation}
and%
\begin{equation}
K_{n,\psi}=\cup_{\theta\in\lbrack0,\psi]}\tau_{n,\theta}. \label{Knc}%
\end{equation}
Then $\tau_{n,0}=\{q_{n2}\},$
\[
\tau_{n}=\tau_{n,\psi_{n,a}}=\overline{A_{n,\psi_{n,a}}B_{n,\psi_{n,a}}%
}=\overline{A_{n}B_{n}},
\]
$K_{n,\psi_{n,a}}=K_{n1}$ and $K_{n,\psi_{n}}=K_{n}$. By the way we see that
$I_{n}$ is an $f_{n}$-lift of $\tau_{n,\psi_{n,a}}=\tau_{n},$ and $K_{n,\psi}$
is a closed Jordan domain for each $\psi\in\lbrack0,\psi_{n}],$ with an
exception at $\psi=0$: $K_{n,0}=\{q_{n2}\}.$

Since $C_{n1}$ and $C_{n2}$ tend to the great circles $C_{01}\ $and $C_{02}$
with $C_{02}=-C_{01}$, we have
\begin{equation}
\lim_{n\rightarrow\infty}\max_{\psi\in\lbrack0,\psi_{n}]}L(\tau_{n,\psi
})\rightarrow0. \label{ap2}%
\end{equation}
Then we may assume $K_{n}$ is so thin that
\begin{equation}
\max_{w\in K_{n}}d(w,\partial K_{n})<\delta_{0}. \label{ap10}%
\end{equation}

Now we can prove the following

\begin{claim}
\label{ccl3}For sufficiently large $n,$ the restriction $f_{n1}=f_{n}%
|_{\overline{\Delta_{n1}}}$ is a homeomorphism onto $K_{n,\psi_{a}}=K_{n1}$,
that is to say, $F_{n1}=\left(  f_{n},\overline{\Delta_{n1}}\right)  $ is the
simple closed domain $K_{n1}.$
\end{claim}

\begin{proof}
By Condition \ref{homI2} and Definition of $\mathcal{F}_{r}(L,m)$, $f_{n}$
restricted to $\partial\Delta_{n1}=-I_{n2}\left(  a,b_{n}\right)  +\alpha
_{n1}\left(  a,a_{n2}\right)  +\alpha_{2}\left(  a_{n2},b_{n}\right)  ,$ with
$b_{n}=b_{n2},$ is a homeomorphism onto $\partial K_{n1}$ and $f_{n}$ has no
branch point on $\left(  \partial\Delta_{n1}\right)  \backslash\{a_{n2}\}$.
Recall that we are still in Case 3, and so $I_{n}\left(  a,b_{n}\right)
=I_{2n}\left(  a,b_{n2}\right)  .$ On the other hand, $f_{n}$ is locally
homeomorphic on $\left(  \Delta\backslash f^{-1}\left(  E_{q}\right)  \right)
\cup\left(  \partial\Delta\right)  \backslash\{a_{nj}\}_{j=1}^{m}.$ Then
$f_{n1}=f_{n}|_{\overline{\Delta_{n1}}}$ is homeomorphic in a neighborhood of
$\left(  \partial\Delta_{n1}\right)  \backslash\{a_{n2}\}$ in $\overline
{\Delta_{n1}}$ and $\left(  f_{n1},\overline{\Delta_{n1}}\right)
\in\mathcal{F}_{r}\left(  L,3\right)  $, and then $f_{n1}(\overline
{\Delta_{n1}})\supset K_{n1}$.

The set of branch values of $f_{n1}$ in $K_{n1}\backslash\{c_{n2}\}$ is
contained in $E^{\prime}=K_{n1}^{\circ}\cap E_{q}.$ If $E^{\prime}%
\neq\emptyset,$ let $\psi\in(0,\psi_{n,a}]$ decrease continuously from
$\psi_{n,a}$ to $0$ so that $\tau_{n,\psi}$ first meets a value $Q\in
E^{\prime}$ for some $\psi=\psi^{\prime}\in\left(  0,\psi_{n,a}\right)  .$
Then by Lemma \ref{b-in} we see that $f_{n}^{-1}$ has a univalent branch
defined on the closed Jordan domain $K_{n1}^{\prime}=\cup_{\psi\in\lbrack
\psi^{\prime},\psi_{n,a}]}\tau_{n,\psi}$ and so $\tau_{n,\psi^{\prime}}$ has
an $f_{n}$-lift in $\overline{\Delta_{n1}}$ whose endpoints are on
$\alpha_{n1}^{\circ}$ and $\alpha_{n2}^{\circ},$ but interior in $\Delta
_{n1},$ and thus $f_{n}^{-1}(E_{q})\cap\Delta_{n1}\supset f_{n}^{-1}%
(Q)\cap\Delta_{n1}\neq\emptyset,$ which with (\ref{ap2}) implies that%
\[
d_{f_{n}}\left(  f_{n}^{-1}(E_{q}\right)  \cap\Delta,\partial\Delta)\leq
d_{f_{n}}\left(  f_{n}^{-1}(Q\right)  \cap\Delta_{n1},\alpha_{n1}+\alpha
_{n2})<L(\tau_{n,\psi^{\prime}})<\delta_{0},
\]
for sufficiently large $n,$ contradicting (\ref{az1}). Then for sufficiently
large $n,$ $f_{n}$ has no any branch value in $K_{n1}\backslash\{q_{n2}\},$
and thus $f_{n}^{-1}$ has a univalent branch defined on $K_{n1}\backslash
\{q_{n2}\}.$ Therefore Claim \ref{ccl3} follows from Lemma \ref{b-in}.
\end{proof}

By Condition \ref{homI2} we may discuss Case 3 under the following condition.

\begin{condition}
\label{gd}For each $n=1,2,\dots,$ there exists $\varphi_{n}\in(\psi_{n,a}%
,\psi_{n}),$ such that for every $\psi\in(\psi_{n,a},\varphi_{n}],\tau
_{n,\psi}$ has a unique $f_{n}$-lift $I_{n,\psi}$ in $\overline{\Delta}$ from
a point $a_{n,\psi}\in\alpha_{n1}^{\circ}$ to a point $b_{n,\psi}\in
\alpha_{n2}^{\circ},$ $I_{n,\psi}^{\circ}\subset\Delta,$ and $K_{n,\psi}$ has
an $f_{n}$-lift $D_{n,\psi}$ so that $D_{n,\psi}$ is a closed Jordan domain
enclosed by%
\[
\partial D_{n,\psi}=\alpha_{n1}\left(  a_{n,\psi},a_{n2}\right)  +\alpha
_{n2}\left(  a_{n2},b_{n,\psi}\right)  -I_{n,\psi}\left(  a_{n,\psi}%
,b_{n,\psi}\right)  ,
\]
say, $f_{n}$ restricted to $D_{n,\psi}$ is a homeomorphism onto $K_{n,\psi}.$
\end{condition}

\begin{claim}
For each $n=1,2,\dots,$ let $\psi_{n}^{\ast}$ be the supremun of all
$\varphi_{n}\in(\psi_{n,a},\psi_{n})$ satisfying Condition \ref{gd}. Then%
\begin{equation}
\psi_{n}^{\ast}<\psi_{n},n=1,2,\cdots. \label{ap3}%
\end{equation}

\end{claim}

\begin{proof}
Assume $\psi_{n}^{\ast}=\psi_{n}$ for some $n.$ Then $q_{n2}^{\ast}\in
c_{n1}\cap c_{n2}$ and $f_{n}^{-1}$ has a univalent branch $g_{n}$ defined on
$K_{n,\psi_{n}^{\ast}}\backslash\{q_{n2}^{\ast}\}=K_{n,\psi_{n}}%
\backslash\{q_{n2}^{\ast}\}=K_{n}\backslash\{q_{n2}^{\ast}\}$. This, together
with Lemma \ref{b-in}, implies that $f_{n}$ restricted to $\overline{\Delta}$
is a homeomorphism onto $K_{n}$ and $\partial\Sigma_{n}=c_{n1}+c_{n2},$ and
then we have $m=2.$ This contradicts that $m>3,$ and then (\ref{ap3}) is
true.\medskip
\end{proof}

\begin{center}
\textbf{Step 4.6 Discussion of Case 3: (3) The subsurface }$F_{n0}^{\prime}$
\end{center}

By Claim \ref{ccl3} and Condition \ref{gd}, it is clear that $f_{n}^{-1}$ has
a univalent branch $g_{n}$ defined on $K_{n,\psi_{n}^{\ast}}\backslash
\tau_{n,\psi_{n}^{\ast}}$. By Lemma \ref{continue0} this branch can be
extended to a univalent branch well defined on $K_{n,\psi_{n}^{\ast}},\ $but
we still denote it by $g_{n}$. Let
\[
I_{n,\psi_{n}^{\ast}}=I_{n,\psi_{n}^{\ast}}\left(  a_{n,\psi_{n}^{\ast}%
},b_{n,\psi_{n}^{\ast}}\right)  =g_{n}(\tau_{n,\psi_{n}^{\ast}}\left(
A_{n,\psi_{n}^{\ast}},B_{n,\psi_{n}^{\ast}}\right)  )
\]
(see (\ref{ap4})), and let $\lambda_{n0}$ be the arc of $\partial\Delta$ from
$a_{n,\psi_{n}^{\ast}}$ to $b_{n,\psi_{n}^{\ast}},$ say,%
\begin{equation}
\lambda_{n0}=\alpha_{n1}\left(  a_{n,\psi_{n}^{\ast}},a_{n2}\right)
+\alpha_{n2}\left(  a_{n2},b_{n,\psi_{n}^{\ast}}\right)  . \label{gman0}%
\end{equation}
Then it is clear that $I_{n,\psi_{n}^{\ast}}\cap I_{n}=\emptyset,$ and
$-I_{n,\psi_{n}^{\ast}}+\lambda_{n0}$ encloses a domain $\Delta_{n0}$ in
$\Delta$ such that
\[
f_{n}:\overline{\Delta_{n0}}\rightarrow K_{n,\psi_{n}^{\ast}}%
\]
is a homeomorphism. Since $\left(  f_{n},I_{n,\psi_{n}^{\ast}}\right)  $ is
straight, we have
\[
L(f_{n},\partial\Delta_{n0})=L(f_{n},I_{n,\psi_{n}^{\ast}})+L(f_{n}%
,\lambda_{n0})\leq L(f_{n},\left(  \partial\Delta\right)  \backslash
\lambda_{n0})+L(f_{n},\lambda_{n0})\leq L,
\]
and then we have

\begin{claim}
\label{home}$f_{n}$ restricted to $\overline{\Delta_{n0}}$ is a homeomorphism
onto $K_{n,\psi_{n}^{\ast}},$ which is the closure of the component of
$K_{n}\backslash\tau_{n,\psi_{n}^{\ast}}$ on the right hand side of
$\tau_{n,\psi_{n}^{\ast}}$. Moreover for sufficiently large $n$
\[
F_{n0}^{\prime}=\left(  f_{n},\overline{\Delta_{n0}}\right)  \in
\mathcal{F}_{r}(L,3)\subset\mathcal{F}_{r}(L,m),
\]
and%
\begin{equation}
\lim_{n\rightarrow\infty}\inf L(\partial F_{n0}^{\prime})\geq L(c_{01}\left(
A_{n,\psi_{n}^{\ast}},q_{n2}\right)  >\delta_{a}>0. \label{L>0}%
\end{equation}
(note that $\left(  f_{n},I_{n,\psi_{n}^{\ast}}\right)  $ is straight and we
assumed $m>3$).
\end{claim}

It is clear that the closed Jordan domain $\overline{\Delta_{n0}}$ is the
union $\cup_{\psi\in\lbrack0,\psi_{n}^{\ast}]}I_{n,\psi}$ with $I_{n,\psi
}^{\circ}\subset\Delta_{n0},\partial I_{n,\psi}\subset\left(  \partial
\Delta_{n0}\right)  \cap\partial\Delta$ and $\tau_{n,\psi}^{\circ}=\left(
f_{n},I_{n,\psi}^{\circ}\right)  \subset K_{n,\psi_{n}^{\ast}}^{\circ}$ for
all $\psi\in(0,\psi_{n}^{\ast})$. Then by (\ref{az1}) and (\ref{ap2}) we have%
\begin{equation}
\Delta_{n0}\cap f_{n}^{-1}(E_{q})=\emptyset, \label{noEq}%
\end{equation}
and then by Claim \ref{home} we have
\[
R(F_{n0}^{\prime})=\left(  q-2\right)  A(F_{n0}^{\prime})<\left(  q-2\right)
A(K_{n,\psi_{n}^{\ast}})<\left(  q-2\right)  A(K_{n})\rightarrow0
\]
as $n\rightarrow\infty,$ and
\begin{equation}
H(F_{n0}^{\prime})\leq\frac{\left(  q-2\right)  A(K_{n})}{L(\partial
F_{n0}^{\prime})}\leq\frac{\left(  q-2\right)  A(K_{n})}{\delta_{a}%
}\rightarrow0\left(  n\rightarrow\infty\right)  . \label{428-3}%
\end{equation}
\medskip

\begin{center}
\textbf{Step 4.7 Discussion of Case 3: (4) Discussion of the }$d_{f_{n}}%
$\textbf{-shortest path }$I_{n,\psi_{n}^{\ast}}$
\end{center}

If there exists a subsequence $\{n_{s}\}_{s=1}^{\infty}$ of $\{n\}$ such that
$I_{n_{s},\psi_{n_{s}}^{\ast}}\subset\partial\Delta,$ then $\Delta_{n_{s}%
0}=\Delta,$ and then $F_{n_{s}0}^{\prime}=\Sigma_{n_{s}}=\left(  f_{n_{s}%
},\overline{\Delta}\right)  $ by Claim \ref{home}. But this implies
\[
A(f_{n_{s}},\Delta)=A(K_{n_{s},\psi_{n_{s}}^{\ast}})<A(K_{n_{s}})\rightarrow0
\]
as $s\rightarrow\infty.$ Then by (\ref{L>0}), we have
\[
H(\Sigma_{n_{s}})\leq\frac{\left(  q-2\right)  A(f_{n_{s}},\Delta)}%
{L(\partial\Sigma_{n_{s}})}=\frac{\left(  q-2\right)  A(f_{n_{s}},\Delta
)}{L(\partial F_{n0}^{\prime})}\rightarrow0,
\]
which contradicts that $\Sigma_{n}$ is an extremal sequence of $\mathcal{F}%
_{r}\left(  L,m\right)  $ (and of $\mathcal{F}\left(  L,m\right)  $) by Lemma
\ref{H>0}. Thus we may assume that%
\[
I_{n,\psi_{n}^{\ast}}\cap\Delta\neq\emptyset,n=1,2,\dots.
\]

By (\ref{ap2}) we have%
\begin{equation}
\lim_{n\rightarrow\infty}L(f_{n},I_{n,\psi_{n}^{\ast}})=\lim_{n\rightarrow
\infty}L(\tau_{n,\psi_{n}^{\ast}})=0, \label{ma20}%
\end{equation}
and so by (\ref{az1}) we may assume that%
\begin{equation}
\Delta\cap I_{n,\psi_{n}^{\ast}}\cap f_{n}^{-1}(E_{q})=\emptyset,n=1,2,\dots.
\label{503-1}%
\end{equation}
Then by Lemma \ref{shortest}, as (\ref{a7}) for $I_{n}$, we may assume, by
taking subsequence, that $I_{n,\psi_{n}^{\ast}}$ has a partition%
\begin{equation}
I_{n,\psi_{n}^{\ast}}=J_{n1}\left(  a_{n1}^{\prime},a_{n2}^{\prime}\right)
+J_{n2}\left(  a_{n2}^{\prime},a_{n3}^{\prime}\right)  +\cdots+J_{n,2k^{\prime
}+1}\left(  a_{n,2k^{\prime}+1}^{\prime},a_{n,2k^{\prime}+2}^{\prime}\right)
\label{In*}%
\end{equation}
satisfying all conclusions of Lemma \ref{shortest} when we regard $k^{\prime}
$ as $k$ there, where $k^{\prime}$ is independent of $n$, $a_{n1}^{\prime
}=a_{n,\psi_{n}^{\ast}}$ and $a_{n,2k^{\prime}+2}^{\prime}=b_{n,\psi_{n}%
^{\ast}}$. Then $f_{n}$ has no branch point in $\cup_{j=1}^{k^{\prime}%
}J_{n,2j}^{\circ}\subset\Delta\cap I_{n,\psi_{n}^{\ast}}$ and we have

\begin{claim}
\label{home1}$f_{n}$ restricted to a neighborhood of $J_{n,2j}^{\circ}$ in
$\overline{\Delta}$ is a homeomorphism$\ $and $\left(  f_{n},J_{n,2j}\right)
$ is straight for every $j=1,\dots,k^{\prime}.\medskip$
\end{claim}

\begin{center}
\textbf{Step 4.8 Discussion of Case 3: (5) Complete the discussion of Case 3}
\end{center}

Corresponding to (\ref{In*}), we write
\[
I_{n}^{\ast}=J_{n2}\left(  a_{n2}^{\prime},a_{n3}^{\prime}\right)
+\cdots+J_{n,2k^{\prime}}\left(  a_{n,2k^{\prime}}^{\prime},a_{n,2k^{\prime
}+1}^{\prime}\right)  .
\]
Let $\gamma_{n0}=\gamma_{n0}\left(  a_{n2}^{\prime},a_{n,2k^{\prime}%
+1}^{\prime}\right)  $ be the arc on $\partial\Delta$ from $a_{n2}^{\prime}$
to $a_{n,2k^{\prime}+1}^{\prime}$ and let $\gamma_{n0}^{\prime}=\gamma
_{n0}^{\prime}\left(  a_{ni_{0}},a_{ni_{2}}\right)  $ be the smallest arc of
$\partial\Delta$ containing $\gamma_{n0}$ such that the endpoints of
$\gamma_{n0}^{\prime}$ are contained in $\{a_{nj}\}_{j=1}^{m}.$ Then we can
write%
\begin{align*}
\gamma_{n0}^{\prime}  &  =\gamma_{n0}^{\ast}+\gamma_{n0}^{\ast\ast},\\
\gamma_{n0}^{\ast}  &  =\gamma_{n0}^{\ast}\left(  a_{ni_{0}},a_{n2}^{\prime
}\right)  +\gamma_{n0}^{\ast}\left(  a_{n2}^{\prime},a_{n2}\right)
=\gamma_{n0}^{\ast}\left(  a_{ni_{0}},a_{n2}^{\prime}\right)  +\alpha
_{n1}\left(  a_{n2}^{\prime},a_{n2}\right) \\
\gamma_{n0}^{\ast\ast}  &  =\gamma_{n0}^{\ast\ast}\left(  a_{n2}%
,a_{n,2k^{\prime}+1}^{\prime}\right)  +\gamma_{n0}^{\ast\ast}\left(
a_{n,2k^{\prime}+1},a_{ni_{2}}\right)  =\alpha_{n2}\left(  a_{n2}%
,a_{n,2k^{\prime}+1}^{\prime}\right)  +\gamma_{n0}^{\ast\ast}\left(
a_{n,2k^{\prime}+1},a_{ni_{2}}\right)  .
\end{align*}
with
\[
i_{0}\left(  \operatorname{mod}m\right)  <2<i_{2}.
\]
Then $a_{n2}\in\gamma_{n0}^{\circ}\cap\{a_{nj}\}_{j=1}^{m}\neq\emptyset.$

If $J_{n1}$ is a point, then $a_{n2}^{\prime}=a_{n1}^{\prime}\in\alpha_{n1}$
and $a_{ni_{0}}=a_{n1},$ and thus $\gamma_{n0}^{\ast}=\alpha_{n1}$ and%
\[
f_{n}(I_{n}^{\ast\circ}\cap\gamma_{n0}^{\ast})=\tau_{n,\psi_{n}^{\ast}}%
^{\circ}\cap c_{n1}\left(  f\left(  a_{n2}^{\prime}\right)  ,q_{n2}\right)
=\emptyset,
\]
thus
\begin{equation}
I_{n}^{\ast\circ}\cap\gamma_{n0}^{\ast}=\emptyset. \label{=0}%
\end{equation}
If $J_{n1}$ is not a point, then $a_{n1}^{\prime}=a_{n1},$ $a_{n2}^{\prime
}=a_{ni_{0}}$ and $\gamma_{n0}^{\ast}=-J_{n1}\left(  a_{ni_{0}},a_{n1}\right)
+\alpha_{n1}\left(  a_{n1},a_{n2}\right)  .$ Since $f_{n}$ maps $\partial
\Delta_{n0}$ homeomorphically onto $\partial K_{n,\psi_{n}^{\ast}},$
(\ref{=0}) also holds. Thus in both cases (\ref{=0}) holds. For the same
reason, we may show that $I_{n}^{\ast\circ}\cap\gamma_{n0}^{\ast\ast
}=\emptyset.$ Thus we have $I_{n}^{\ast\circ}\cap\gamma_{n0}^{\prime
}=\emptyset,$ and (e2) of Lemma \ref{m-1} holds for $\Sigma=\Sigma_{n}$.

Now we have proved that $-I_{n}^{\ast},\gamma_{n0},\gamma_{n0}^{\prime}%
,\Sigma_{n},k^{\prime}$ satisfy all hypothesis of Lemma \ref{m-1} as
$I,\gamma_{0},\gamma_{0}^{\prime},\Sigma,k$ there. Let $\left\{  \Delta
_{nj}\right\}  _{j=0}^{k^{\prime}}$ be the $k^{\prime}+1$ components of
$\Delta\backslash I_{n}^{\ast}$ such that $\Delta_{n0}$ is the the only one on
the right hand side of $I_{n}^{\ast},$ which we discussed above, and all
others are on the left hand side of $I_{n}^{\ast},$ and moreover we have
\begin{equation}
F_{nj}^{\prime}=\left(  f_{n},\overline{\Delta_{nj}}\right)  \in
\mathcal{F}(L,m-1)\mathrm{\ for\ every\ }j=1,\dots,k^{\prime}.
\label{all-Fj_FLM}%
\end{equation}

Now we deduce a contradiction. It is clear that we have
\[
\sum_{j=0}^{k^{\prime}}R(F_{nj}^{\prime})=R(\Sigma_{n}).
\]%
\[
\varepsilon_{n}=\sum_{j=0}^{k^{\prime}}L(f_{n},J_{n,2j})\rightarrow0\left(
n\rightarrow\infty\right)  ,
\]
and%
\[
\sum_{j=0}^{k^{\prime}}L(\partial F_{nj}^{\prime})\leq L(\partial\Sigma
_{n})+2\varepsilon_{n},
\]
equality holding if and only if all $J_{n,2j+1},j=0,\dots,k^{\prime},$ are
points. Thus we can conclude, by (\ref{428-3}) and Claim \ref{home}, that the
sequence $\Sigma_{n}$ is decomposable in $\mathcal{F}\left(  L,m\right)  $ (by
Definition \ref{undec-seqr} (c)). This contradicts Lemma \ref{undec-seq}. We
have completed the discussion of Case 3 and obtained a contradiction.

We have proved that each of the three cases implies a contradiction, and thus
(\ref{az4}) fails. Therefore Assertion \ref{asC} holds true, and we have
proved Assertion \ref{asA} in the case $a\in\alpha_{01}^{\circ}.\medskip
\medskip$

\begin{center}
\label{st5}\textbf{Step 5. Discussion in the case }$a\in\partial\alpha_{01}$
\end{center}

Now we prove Assertion \ref{asA} under the condition
\begin{equation}
a=a_{01}\in\partial\alpha_{01}=\{a_{01},a_{02}\}. \label{zz3}%
\end{equation}
The case $a=a_{02}$ can be discussed as the case $a=a_{01}$ in the same way.

Now that we have already proved Assertion \ref{asA} with (\ref{zz2}), we
conclude that, under the condition $a=a_{01}\in\partial\alpha_{01},$ Assertion
\ref{asA} holds for all $b\in\left(  \partial\Delta\right)  \backslash
\{a_{0j}\}_{j=1}^{n}.$ So, to prove Assertion \ref{asA} under (\ref{zz3}), it
remains to prove the following special case of (\ref{mar1}): for each
$j_{0}=2,3,\dots,m,$%
\begin{equation}
\lim_{n\rightarrow\infty}\inf d_{f_{n}}(a_{01},a_{0j_{0}})>0\mathrm{\ }%
\text{\textrm{if}}\mathrm{\ }a_{0j_{0}}\neq a_{01}. \label{az5-1}%
\end{equation}
Note that $j_{0}\neq1$ may not imply $a_{0j_{0}}\neq a_{01}:$ when
$\alpha_{0m}$ is a point, for example, $a_{0m}=a_{01}$.

Since $d_{f_{n}}(a_{nj},a_{0j})\rightarrow0$ for each $j=1,2,\dots,m$,
(\ref{az5-1}) is equivalent to%
\begin{equation}
\lim_{n\rightarrow\infty}\inf d_{f_{n}}(a_{n1},a_{nj_{0}})>0\text{
\textrm{if}}\mathrm{\ }a_{0j_{0}}\neq a_{01}. \label{az5}%
\end{equation}

Now we fix a $j_{0}$ with
\begin{equation}
a_{0j_{0}}\neq a_{01}. \label{not=}%
\end{equation}
By Lemma \ref{shortest}, there exists a $d_{f_{n}}$-shortest path $J_{n}$ from
$a_{n1}$ to $a_{nj_{0}}.$ To prove (\ref{az5}) we assume the contrary
$\lim_{n\rightarrow\infty}\inf d_{f_{n}}(a_{n1},a_{nj_{0}})=0.$ Then by taking
subsequence, we may assume%
\begin{equation}
\lim_{n\rightarrow\infty}d_{f_{n}}(a_{n1},a_{nj_{0}})=\lim_{n\rightarrow
\infty}L(f_{n},J_{n})=0. \label{az8}%
\end{equation}
If $J_{n}$ has a subsequence $J_{n_{k}}$ such that $J_{n_{k}}\subset
\partial\Delta,$ then we have by (\ref{az8}) that $J_{n_{k}}$ tends to the
point $a_{01},$ which implies $a_{j_{0}}=a_{01},$ contradicting (\ref{not=})
(note that $\Gamma_{n}=\left(  f_{n},\partial\Delta\right)  $ are
parameterized by length for all $n=0,1,2,\dots.$ So we may assume $J_{n}%
\cap\Delta_{n}\neq\emptyset\ $for all $n.$ Then by (\ref{az1}) we may assume%
\begin{equation}
J_{n}\cap f^{-1}(E_{q})\cap\Delta=\emptyset,n=1,2,\dots. \label{az7}%
\end{equation}

By Lemma \ref{shortest} and taking subsequence, we can assume that $J_{n}$ has
a partition
\[
J_{n}=J_{n1}\left(  a_{n1}^{\prime},a_{n2}^{\prime}\right)  +J_{n2}\left(
a_{n2}^{\prime},a_{n3}^{\prime}\right)  +J_{n3}\left(  a_{n3}^{\prime}%
,a_{n4}^{\prime}\right)  +\cdots+J_{n,2k^{\prime}+1}\left(  a_{n,2k^{\prime
}+1}^{\prime},a_{n,2k^{\prime}+2}^{\prime}\right)
\]
with $a_{n1}^{\prime}=a_{n1}$, $a_{n,2k^{\prime}+2}^{\prime}=a_{nj_{0}},$
$a_{nj}^{\prime}\rightarrow a_{0i_{j}}\ $for some $a_{0i_{j}}\in
\{a_{0j}\}_{j=1}^{m}$ for each $j=1,2,\dots,2k^{\prime}+2,$ and $k^{\prime
}<m.$ We show that there exists $j_{1}\in\{1,\dots,k^{\prime}\},$ such that
the two endpoints of $J_{n,2j_{1}}$ converges to distinct points, say,
\[
a_{n,2j_{1}}^{\prime}\rightarrow a_{0,i_{2j_{1}}}\neq a_{0i_{2j_{1}+1}%
}\leftarrow a_{n,2j_{1}+1}^{\prime}\mathrm{\ as\ }n\rightarrow\infty.
\]
Assume that this fails. Then each pair $a_{n,2j}^{\prime},a_{n,2j+1}^{\prime
},$ the endpoints of $J_{2j},$ converge to the same point $a_{0i_{2j}%
}=a_{0i_{2j+1}},$ for $j=1,\dots,k^{\prime};$ and, as an arc on $\partial
\Delta,$ by (\ref{az8}) each $J_{n,2j-1}\left(  a_{n,2j-1},a_{n,2j}\right)  $
converges to the same point $a_{n,i_{2j-1}}=a_{n,i_{2j}},$ for $j=1,2,\dots
,k^{\prime}+1.$ Therefore, all $a_{nj}^{\prime}$ converge to the same point,
and thus
\[
a_{01}=\lim_{n\rightarrow\infty}a_{n1}=\lim_{n\rightarrow\infty}a_{n1}%
^{\prime}=\lim_{n\rightarrow\infty}a_{n,2k+2}^{\prime}=\lim_{n\rightarrow
\infty}a_{nj_{0}}=a_{0j_{0}}%
\]
This contradicts (\ref{not=}).

Thus the segment $J_{n,2j_{1}}=J_{n,2j_{1}}\left(  a_{n,2j_{1}}^{\prime
},a_{n,2j_{1}+1}^{\prime}\right)  $ of $J_{n}$ satisfies Lemma \ref{shortest}
for the case $k=1$ there, so that the two endpoints $a_{0i_{2j_{1}}}$ and
$a_{0i_{2j_{1}+1}}$ of $J_{n,2j_{1}}$ belong to $\{a_{nj}\}_{j=1}^{m}.$ Thus
by Lemma \ref{shortest} for the case $k=1,$ $J_{n,2j_{1}}$ divides $\Sigma
_{n}$ into two subsurfaces $\Sigma_{n0}$ and $\Sigma_{n1}$ contained in
$\mathcal{F}_{r}(L,m)$. By (\ref{az7}) we have
\[
R(\Sigma_{n0})+R(\Sigma_{n1})=R\left(  \Sigma_{n}\right)  .
\]
On the other hand we also have%
\[
L(\partial\Sigma_{n0})+L(\partial\Sigma_{n1})=L(\partial\Sigma_{n}%
)+2L(f_{n},J_{n,2j_{1}})\rightarrow L(f_{0},\partial\Delta)\leq L
\]%
\[
\lim_{n\rightarrow\infty}\inf L(\partial\Sigma_{nj})\geq\min\left\{
L(f_{0},\gamma_{0}),L(f_{0},\gamma_{1})\right\}  >0,j=1,2,
\]
where $\gamma_{0}$ and $\gamma_{1}$ are the two arcs of $\partial\Delta$ with
the two distinct endpoints $a_{0i_{2j_{1}}}$ and $a_{0i_{2j_{1}+1}}.$ Then
$\Sigma_{n}$ is decomposable in $\mathcal{F}_{r}(L,m)\subset\mathcal{F}(L,m),$
by Definition \ref{decflm} (b). This contradicts the conclusion of Lemma
\ref{undec-seq}, and hence (\ref{az8}) cannot hold. We have proved Assertion
\ref{asA} completely and so (i) is completely proved.$\medskip\medskip$

\begin{center}
\label{st6}\textbf{Step 6. Theorem \ref{nobi} (i) implies Theorem \ref{nobi}
(ii)}
\end{center}

If (ii) is not true, then there exists sequences $x_{n1}\in I$ and $x_{n2}\in
J$ such that
\[
x_{n1}\rightarrow x_{01}\in I,x_{n2}\rightarrow x_{02}\in J,
\]
and
\[
d_{f_{n}}(x_{n1},x_{n2})\rightarrow0,
\]
as $n\rightarrow\infty.$ On the other hand, for the shorter arc $\alpha
_{j}=\alpha_{j}\left(  x_{nj},x_{0j}\right)  $ from $x_{nj}$ to $x_{0j}$ on
$\partial\Delta,j=1,2,$ we have by Lemma \ref{l1} (ii) $d_{f_{n}}\left(
x_{nj},x_{0j}\right)  \rightarrow0,$ as $n\rightarrow\infty,$ for $j=1,2.$
Then we have%
\[
d_{f_{n}}(x_{01},x_{02})\leq d_{f_{n}}\left(  x_{01},x_{n1}\right)  +d_{f_{n}%
}\left(  x_{n1},x_{n2}\right)  +d_{f_{n}}\left(  x_{n2},x_{02}\right)
\rightarrow0.
\]
But it is clear that $x_{01}\neq x_{02},$ and we obtain a contradiction by (i)
of the Theorem \ref{nobi}, and thus Theorem \ref{nobi} (ii) holds
true.\medskip\medskip

\section{Existence and property of extremal surfaces in $\mathcal{F}%
_{r}\left(  L,m\right)  $}

The goal of this section is to prove the following theorem.

\begin{theorem}
\label{LK}Let $L\in\mathcal{L}$ be a positive number and $m$ be a sufficiently
large positive integer. Then there exists a unique positive number $L_{1}$
with
\[
L_{1}\leq L
\]
and a precise extremal surface $\Sigma_{L_{1}}\ $in $\mathcal{F}_{r}(L,m)$
such that $L(\partial\Sigma_{L_{1}})=L_{1}.\medskip$
\end{theorem}

The proof is divided into 5 steps. The key points of the proof are Theorems
\ref{key3} and \ref{key1}, which follows from Theorem \ref{nobi}.\medskip

\begin{center}
\textbf{Step 1. Notations, \textbf{simple }discussions and the idea for proof
of Theorem \ref{LK}. \medskip}
\end{center}

When $L\leq2\delta_{E_{q}},$ let $L_{1}=L$ and let $\Sigma_{L}$ be a simple
disk on $S$ whose interior is outside $E_{q}.$ Then $L_{1}$ and $\Sigma
_{L_{1}}$ are the desired number and surface, by Theorem \ref{l<2dt}. So
through out the proof, we assume%
\[
L\geq2\delta_{E_{q}}.
\]

The proof for this case is complicated. We will first state the idea of the
proof, after some preparation.

By Corollary \ref{FF'}, there exists a precise extremal sequence $\Sigma_{n}$
of $\mathcal{F}_{r}^{\prime}(L,m),$ such that $\Sigma_{n}$ is also a precise
extremal sequence of $\mathcal{F}_{r}(L,m)$ and $\mathcal{F}(L,m),$ and,
moreover,
\begin{equation}
L_{1}=\lim_{n\rightarrow\infty}\inf L(\partial\Sigma_{n})\geq2\delta_{E_{q}}.
\label{>2d}%
\end{equation}

By definition of $\mathcal{F}_{r}^{\prime}(L,m),$ for the number $d^{\ast
}=d^{\ast}\left(  E_{q},m\right)  ,$ which depends only on $m$ and $E_{q}$ and
given by Theorem \ref{sim}, we have

\begin{conclusion}
\label{co1}$\deg_{\max}f_{n}\leq d^{\ast}\ $for all $n=1,2,\dots
\label{have delete-deg_min OK?}$
\end{conclusion}

We assume that all $\partial\Sigma_{n}=\left(  f_{n},\partial\Delta\right)  $
are $\mathcal{F}(L,m)$-partitioned, and then by definition, for each $n,$
$\partial\Delta$ has an $\mathcal{F}(L,m)$-partition
\begin{equation}
\partial\Delta=\alpha_{n1}\left(  a_{n1},a_{n2}\right)  +\alpha_{n2}\left(
a_{n2},a_{n3}\right)  +\cdots+\alpha_{nm}\left(  a_{nm},a_{n1}\right)
\label{pt1}%
\end{equation}
for $\partial\Sigma_{n},$ and $\partial\Sigma_{n}$ has the corresponding
$\mathcal{F}\left(  L,m\right)  $-partition%
\begin{equation}
\partial\Sigma_{n}=c_{n1}\left(  q_{n1},q_{n2}\right)  +c_{n2}\left(
q_{n2},q_{n3}\right)  +\cdots+c_{nm}\left(  q_{nm},q_{nm}\right)  ,
\label{pt2}%
\end{equation}
consisted of $m$ circular arcs as in Definition \ref{circu}. Then $f_{n}$
restricted to each $\alpha_{nj}$ is the SCC arc $c_{nj}$, and Remark
\ref{ap11}, we may assume

\begin{condition}
\label{co63}$a_{n1}=1$ for all $n=1,2,\dots.$
\end{condition}

Since for any homeomorphism $h$ of $\overline{\Delta}$ onto itself and
$\Sigma_{n}^{h}=\left(  f_{n}\circ h,\overline{\Delta}\right)  ,\Sigma_{n}%
^{h}$ is also a precise extremal sequence of $\mathcal{F}_{r}(L,m),$ we may
assume, by taking subsequence and applying Aazela-Ascoli Theorem to curves
parametrized by length, the following.

\begin{condition}
\label{co2}All $\Gamma_{n}=\partial\Sigma_{n}$ are parametrized by length and
uniformly converges to a closed curve $\Gamma_{0}=(f_{0},\partial\Delta)$, for
each $n\geq1$ and $j\in M=\{1,2,\dots,m\},$ $c_{nj}=\left(  f,\alpha
_{nj}\right)  $ is an SCC arc.
\end{condition}

In fact, parametrized by length, $\left(  f_{n},\partial\Delta\right)  $ is a
linear mapping in length and so does $(f_{0},\partial\Delta).$ Thus $f_{n}$
restricted to each $\alpha_{nj}$ is the simple circular arc $c_{nj}$ for all
$n\in\mathbb{N}^{0}=\{0\}\cup\mathbb{N}$, where $\mathbb{N}$ \label{nature} is
the set of positive integers, and $j\in M=\{1,2,\dots,m\}$, and we have by
Conditions \ref{co63} and \ref{co2} that:

\begin{conclusion}
\label{inver}For each $j\in M=\{1,2,\dots,m\}$ and $n\in\mathbb{N}^{0}$,
$f_{n}^{-1}$ has a unique univalent branch $\tilde{g}_{nj}$ defined on
$c_{nj}^{\circ}$ with $\tilde{g}_{nj}\left(  c_{nj}^{\circ}\right)
=\alpha_{nj}^{\circ},$ such that $\tilde{g}_{nj}$ uniformly converges to
$\tilde{g}_{0j}:c_{0j}^{\circ}\rightarrow\alpha_{0j}^{\circ},$ and thus for
any interval $I_{n}\subset c_{nj}^{\circ}$ which converges to $I_{0}\subset
c_{0j}^{\circ},$ $\tilde{g}_{nj}(I_{n})$ converges to $\tilde{g}_{0j}%
(I_{0}).\medskip$ When $c_{nj}$ is not closed, $\tilde{g}_{nj}$ can be
extended to be homeomorphic on $c_{nj}.$
\end{conclusion}

\noindent\textbf{The idea of proving Theorem \ref{LK}. }\emph{Though }%
$\Gamma_{n}=\left(  f_{n},\partial\Delta\right)  $\emph{\ converges to
}$\Gamma_{0}=(f_{0},\partial\Delta)$\emph{\ uniformly, }$f_{n}$\emph{\ may not
converge in }$\Delta.$\emph{\ The key of the proof is to construct a surface
}$\Sigma_{L_{1}}=\left(  f_{L_{1}},\overline{\Delta}\right)  \in
\mathcal{F}_{r}^{\prime}\left(  L_{1},m\right)  $\emph{\ such that}%
\[
R(\Sigma_{L_{1}})=\lim R(\Sigma_{n})=H_{L,m},
\]
\emph{\ }%
\[
f_{L_{1}}|_{\partial\Delta}=f_{0}\mathrm{\ and\ }L(f_{L_{1}},\partial
\Delta)=\Gamma_{0}=L(f_{0},\partial\Delta).
\]
\emph{\ This }$f_{L_{1}}$\emph{\ will be obtained by modifying some }$f_{n}%
$\emph{\ for sufficiently large }$n.$\emph{\ The key of this modification is
that }$\Sigma_{n}=\left(  f_{n},\Delta\right)  $\emph{\ has a subsequence,
which will be still denoted by }$\Sigma_{n},$\emph{\ such that the following
holds:}

\emph{(a) There exists a closed Jordan domain }$\Delta_{n}$\emph{\ contained
in }$\Delta\ $\emph{such that the curves }$\left(  f_{n},\partial\Delta
_{n}\right)  $\emph{\ contain no point of }$E_{q}$\emph{\ and are equivalent
each other, and }$R\left(  f_{n},\Delta_{n}\right)  $\emph{\ are equal to each
other.}

\emph{(b) For }$\mathcal{A}_{n}=\overline{\Delta}\backslash\Delta_{n}^{\circ}
$, \emph{there exists a sequence of surfaces }$B_{n}=\left(  F_{n}%
,\mathcal{A}_{n}\right)  $\emph{\ such that }$f_{n}|_{\partial\Delta_{n}%
}=F_{n}|_{\partial\Delta_{n}}$\emph{, }$f_{0}=F_{n}|_{\partial\Delta},$
\emph{and }$R(f_{n},A_{n})-R(F_{n},A_{n})\rightarrow0$\emph{\ as
}$n\rightarrow\infty.$\emph{\ }

\emph{(c) For the surface }$\Sigma_{n}^{\ast}=\left(  f_{n}^{\ast}%
,\overline{\Delta}\right)  ,$\emph{\ in which }$f_{n}^{\ast}$\emph{\ is
defined by }$f_{n}$\emph{\ on }$\Delta_{n}$\emph{\ and by }$F_{n}$\emph{\ on
}$A_{n}$\emph{, }$H\left(  \Sigma_{n}^{\ast}\right)  $\emph{\ is the constant
}$H_{L,m} $\emph{\ for all }$n.$

\emph{The key ingredient of this idea are Theorems \ref{key3} and \ref{key1},
which will be proved later in this section, and which with Theorem \ref{nobi}
deduce (a)--(c). The proof of these two theorems are applications of Theorem
\ref{nobi}.}$\medskip$

By (\ref{pt1}), (\ref{pt2}) and Conditions \ref{co63} and \ref{co2},
$\partial\Delta$ has the $\mathcal{C}\left(  L,m\right)  $-partition
\begin{equation}
\partial\Delta=\alpha_{01}\left(  a_{01},a_{02}\right)  +\alpha_{02}\left(
a_{02},a_{03}\right)  +\cdots+\alpha_{0m}\left(  a_{m},a_{01}\right)  ,
\label{pt3}%
\end{equation}
for $\Gamma_{0}$ so that $\alpha_{01}$ initiates at $a_{01}=1\in\partial
\Delta,$ and $\Gamma_{0}$ has the corresponding $\mathcal{C}\left(
L,m\right)  $-partition
\begin{equation}
\Gamma_{0}=c_{01}\left(  q_{01},q_{02}\right)  +c_{02}\left(  q_{02}%
,q_{03}\right)  +\cdots+c_{0m}\left(  q_{0m},q_{01}\right)  , \label{pt4}%
\end{equation}
so that

\begin{conclusion}
\label{co3}For each $j=1,\dots,m,f_{0}$ restricted to $\alpha_{0j}$ is the SCC
arc $c_{0j}$, $c_{nj}$ converges to $c_{0j}$ uniformly, and thus $c_{0j}$ is a
point iff $\alpha_{0j}$ is a point.
\end{conclusion}

By assumption we have%
\begin{equation}
2\delta_{E_{q}}\leq L(\Gamma_{0})=L_{1}=\lim_{n\rightarrow\infty}L(\Gamma
_{n}). \label{cc10}%
\end{equation}
Since $\partial\Sigma_{n}$ and $\Gamma_{0}$ are parametrized by length and
$a_{n1}=a_{01}=1,$ we have

\begin{conclusion}
\label{co4}For each $j=1,2,\dots,m,$ $\alpha_{nj}$ converges to $\alpha_{0j}$
as well and thus $\alpha_{nj}$ converges to a point iff $\alpha_{0j}$ and
$c_{0j}$ are both point-arcs. Moreover, for any sequence of intervals
$[\theta_{n1},\theta_{n2}]$ of real numbers with $[\theta_{n1},\theta
_{n2}]\rightarrow\left[  \theta_{01},\theta_{02}\right]  $ and for the
sequence of arcs $I_{n}=\{e^{\sqrt{-1}\theta}:\theta\in\lbrack\theta
_{n1},\theta_{n2}]\}$ of $\partial\Delta,$
\[
L(f_{n},I_{n})\rightarrow L(f_{0},I_{0}).
\]

\end{conclusion}

Recall that $M=\{1,\dots,m\}.$ Then there exists a subset $M_{0}=\{i_{1}%
,i_{2},\dots,i_{m_{0}}\}\ $of $M$ with
\[
1\leq i_{1}<i_{2}<\dots<i_{m_{0}}\leq m,
\]
such that $j\in M_{0}$ iff $\alpha_{0j}$ is not a point. Note that
$\alpha_{0j}$ is a point iff $c_{0j}$ is.

\begin{conclusion}
\label{co5}The partition (\ref{pt3}) has a simplified partition
\begin{equation}
\partial\Delta=\alpha_{0i_{1}}\left(  a_{0i_{1}},a_{0i_{1}+1}\right)
+\alpha_{0i_{2}}\left(  a_{0i_{2}},a_{0i_{2}+1}\right)  +\cdots+\alpha
_{0m_{0}}\left(  a_{i_{m_{0}}},a_{i_{m_{0}}+1}\right)  , \label{pt5}%
\end{equation}
with $m_{0}\leq m$ and $a_{01}=\cdots=a_{0i_{1}-1}=a_{0i_{1}}=1,$ such that
all point-arcs in (\ref{pt3}) are deleted, and the partition (\ref{pt4}) also
has a simplified partition:
\begin{equation}
\Gamma_{0}=c_{0i_{1}}\left(  q_{0i_{1}},q_{0i_{1}+1}\right)  +c_{0i_{2}%
}\left(  q_{0i_{2}},q_{0i_{2}+1}\right)  +\cdots+c_{0m_{0}}\left(
q_{i_{m_{0}}},q_{i_{m_{0}}+1}\right)  , \label{pt6}%
\end{equation}
such that all point-arcs in (\ref{pt4}) are also deleted and that $f_{0}$
restricted to each $a_{0j}$ for $j\in M_{0}=\{i_{1},\dots,i_{m_{0}}\}$ is the
SCC arc $c_{0j}.$
\end{conclusion}

By Lemma \ref{narrow} and Theorem \ref{nobi}, there exists a $\delta_{1}>0$
such that for all $n$,
\begin{equation}
\delta_{1}<d_{f_{n}}(f^{-1}(E_{q})\cap\Delta,\partial\Delta), \label{cc11}%
\end{equation}
and for sufficiently\footnote{Don't confuse $\alpha$ with $a$, although they
look similar from a distance, up close they are different! The three distance
in the curly brackets are of two arcs, one arc and one point, and two points.}
large $n,$
\begin{equation}
\delta_{1}<\min_{j\in M_{0}}\left\{  \min_{\substack{\left\{  i,j\right\}  \in
M_{0} \\i\neq j}}d_{f_{n}}\left(  \alpha_{0i},\alpha_{0j}\right)
,\min_{\substack{i\in M_{0} \\j\in M,a_{0j}\notin\alpha_{0i}}}d_{f_{n}}%
(\alpha_{0i},a_{0j}),\min_{\substack{\left\{  i,j\right\}  \subset M
\\a_{0i}\neq a_{0j}}}d_{f_{n}}(a_{0i},a_{0j})\right\}  . \label{co44}%
\end{equation}
It is clear that%
\begin{equation}
\delta_{2}=\frac{1}{3}\min\left\{  d(w_{1},w_{2}):\{w_{1},w_{2}\}\subset
E_{q}\cup\{q_{0j}\}_{j=1}^{m}\mathrm{\ and\ }w_{1}\neq w_{2}\right\}  >0.
\label{co51}%
\end{equation}
On the other hand we have
\begin{equation}
\delta_{3}=\min_{j\in M_{0}}L(c_{0j})>0. \label{cc7}%
\end{equation}

Let $\delta$ be a positive number with%
\begin{equation}
\delta<\frac{\min\{\delta_{1},\delta_{2},\delta_{3}\}}{12\pi\left(  d^{\ast
}+1\right)  m}. \label{co49}%
\end{equation}
Then we have
\begin{equation}
\left(  \overline{D(q_{0j},\delta)}\backslash\{q_{0j}\}\right)  \cap
E_{q}=\emptyset,j\in M, \label{cc15}%
\end{equation}
and it is clear that:

\begin{conclusion}
\label{co46}For each $j\in M_{0},$ $c_{0j}$ is divided into three arcs:%
\[
c_{0j}=c_{0j,\delta}^{1}+c_{0j,\delta}^{2}+c_{0j,\delta}^{3}=c_{0j,\delta}%
^{1}\left(  q_{0j},q_{0j,\delta}^{1}\right)  +c_{0j,\delta}^{2}\left(
q_{0j,\delta}^{1},q_{0j,\delta}^{2}\right)  +c_{0j,\delta}^{3}\left(
q_{0j,\delta}^{2},q_{0,j+1}\right)
\]
with $c_{0j,\delta}^{2}=c_{0j}\backslash D(\{q_{0j},q_{0,j+1}\},\delta),$ and
each $\alpha_{0j}$ corresponding to $c_{0j}$ is divided into three arcs%
\[
\alpha_{0j}=\alpha_{0j,\delta}^{1}+\alpha_{0j,\delta}^{2}+\alpha_{0j,\delta
}^{3}=\alpha_{0j,\delta}^{1}\left(  a_{0j},a_{0j,\delta}^{1}\right)
+\alpha_{0j,\delta}^{2}\left(  a_{0j,\delta}^{1},a_{0j,\delta}^{2}\right)
+\alpha_{0j,\delta}^{3}\left(  a_{0j,\delta}^{2},a_{0,j+1}\right)  ,
\]
with $c_{0j,\delta}^{i}=\left(  f_{0},\alpha_{0j,\delta}^{i}\right)  $ for
$i=1,2,3,$ say,
\[
f_{0}(a_{0j})=q_{0j},f_{0}(a_{0j,\delta}^{1})=q_{0j,\delta}^{1},f_{0}%
(a_{0j,\delta}^{2})=q_{0j,\delta}^{2},f_{0}(a_{0,j+1})=q_{0,j+1};
\]
and thus, for sufficiently large $n,$ $c_{nj}=c_{nj}\left(  q_{nj}%
,q_{n,j+1}\right)  $ is divided into three arcs by $\partial D(\{q_{0j}%
,q_{0,j+1}\},\delta):$%
\[
c_{nj}=c_{nj,\delta}^{1}+c_{nj,\delta}^{2}+c_{nj,\delta}^{3}=c_{nj,\delta}%
^{1}\left(  q_{nj},q_{nj,\delta}^{1}\right)  +c_{nj,\delta}^{2}\left(
q_{nj,\delta}^{1},q_{nj,\delta}^{2}\right)  +c_{nj,\delta}^{3}\left(
q_{nj,\delta}^{2},q_{n,j+1}\right)  ,
\]
with $d(q_{nj,\delta}^{1},q_{0j})=\delta,$ $d(q_{nj,\delta}^{2},q_{0,j+1}%
)=\delta$ and
\[
c_{nj,\delta}^{2}=c_{nj,\delta}^{2}\left(  q_{nj,\delta}^{1},q_{nj,\delta}%
^{2}\right)  =c_{nj}\backslash D(\left\{  q_{0j},q_{0,j+1}\right\}  ,\delta),
\]
and $\alpha_{nj}=\alpha_{nj}\left(  a_{nj},a_{nj}\right)  ,$ corresponding to
$c_{nj},$ is divided into three arcs%
\[
\alpha_{nj}=\alpha_{nj,\delta}^{1}+\alpha_{nj,\delta}^{2}+\alpha_{nj,\delta
}^{3}=\alpha_{nj,\delta}^{1}\left(  a_{nj},a_{nj,\delta}^{1}\right)
+\alpha_{nj,\delta}^{2}\left(  a_{nj,\delta}^{1},a_{nj,\delta}^{2}\right)
+\alpha_{nj,\delta}^{3}\left(  a_{nj,\delta}^{2},a_{n,j+1}\right)  ,
\]
such that $f_{n}\ $restricted to $\alpha_{nj,\delta}^{i}\ $is a homeomorphism
onto $c_{nj,\delta}^{i},$ say,
\[
c_{nj,\delta}^{i}=\left(  f_{n},\alpha_{nj,\delta}^{i}\right)  ,i=1,2,3,
\]
and
\[
f_{n}(a_{nj})=q_{nj},f_{n}(a_{nj,\delta}^{1})=q_{nj,\delta}^{1},f_{n}%
(a_{nj,\delta}^{2})=q_{nj,\delta}^{2},f_{n}(a_{nj+1})=q_{nj+1};
\]
and therefore, we have
\[
c_{nj,\delta}^{i}\rightarrow c_{0j,\delta}^{i},\alpha_{nj,\delta}%
^{i}\rightarrow\alpha_{0j,\delta}^{i}%
\]
as $n\rightarrow\infty,$ for $i=1,2,3.$
\end{conclusion}

Thus $\Sigma_{n}$ and $\Gamma_{0}$ satisfies (A)--(D) in Theorem \ref{nobi}.
Then we have

\begin{claim}
\label{lim}For two sequences $I_{nj},$ of arcs on $\partial\Delta$ such that
$I_{nj}\rightarrow I_{0j}$ as $n\rightarrow\infty$ for $j=1,2,$ we have%
\[
\lim_{n\rightarrow\infty}\inf d_{f_{n}}\left(  I_{n1},I_{n2}\right)
=\lim_{n\rightarrow\infty}\inf d_{f_{n}}\left(  I_{01},I_{02}\right)  .
\]

\end{claim}

By definition, $\alpha_{0j,\delta}^{2},j\in M_{0},$ are disjoint compact arcs
in $\partial\Delta,$ $\alpha_{nj}\rightarrow\alpha_{0j},$ $\alpha_{nj,\delta
}^{2}\rightarrow\alpha_{0j,\delta}^{2}$ and $\alpha_{0j,\delta}^{2}%
\subset\alpha_{0j}^{\circ}$. Then by Theorem \ref{nobi} and Claim \ref{lim} we
may assume by taking subsequence, that
\begin{equation}
\delta_{4}=\min\left\{  \min_{\substack{\left\{  i,j\right\}  \subset M_{0}
\\i\neq j}}\inf_{n\in\mathbb{N}}\left\{  d_{f_{n}}\left(  \alpha_{0i,\delta
}^{2},\alpha_{0j,\delta}^{2}\right)  \right\}  ,\min_{j\in M_{0}}\inf
_{n\in\mathbb{N}}d_{f_{n}}\left(  \left(  \partial\Delta\right)
\backslash\alpha_{nj},\alpha_{nj,\delta}^{2}\right)  \right\}  >0.
\label{506-1}%
\end{equation}
Then we have by the relation $\alpha_{ni,\delta}^{2}\subset\left(
\partial\Delta\right)  \backslash\alpha_{nj}$ for $\left\{  i,j\right\}
\subset M_{0},i\neq j,$ we have
\begin{equation}
\min_{\left\{  i,j\right\}  \subset M_{0},i\neq j}d_{f_{n}}\left(
\alpha_{ni,\delta}^{2},\alpha_{nj,\delta}^{2}\right)  \geq\delta_{4}>0.
\label{506-3}%
\end{equation}

It is clear that we may assume that $\delta$ is small enough at first such
that
\begin{equation}
L(c_{0j,\delta}^{1})=L(c_{0j,\delta}^{3})<2\pi\delta\mathrm{\ for\ all\ }j\in
M_{0}. \label{506-4}%
\end{equation}

By, Conclusions \ref{co4} and \ref{co46}, Claim \ref{lim}, and (\ref{506-4}),
we have for each $j\in M_{0},$%
\begin{align*}
d_{f_{n}}\left(  a_{0j},a_{nj,\delta}^{1}\right)   &  \leq d_{f_{n}}\left(
a_{0j},a_{nj}\right)  +d_{f_{n}}\left(  a_{nj},a_{nj,\delta}^{1}\right) \\
&  \leq d_{f_{n}}\left(  a_{0j},a_{nj}\right)  +L(f_{n},\alpha_{nj,\delta}%
^{1})\rightarrow0+L(c_{0j,\delta}^{1})<2\pi\delta,
\end{align*}
as $n\rightarrow\infty,$ and for the same reason%
\[
d_{f_{n}}\left(  a_{0,j+1},a_{nj,\delta}^{2}\right)  \rightarrow
L(c_{0j,\delta}^{3})<2\pi\delta,
\]
as $n\rightarrow\infty,$ for each $j\in M_{0}.$ Thus we have for sufficiently
large $n,$%
\begin{equation}
\max_{j\in M_{0}}\max\left\{  d_{f_{n}}\left(  a_{0j},a_{nj,\delta}%
^{1}\right)  ,d_{f_{n}}\left(  a_{n,j+1,\delta}^{2},a_{0,j+1}\right)
\right\}  <2\pi\delta. \label{505-1}%
\end{equation}

For each $j\in M_{0},$ we let $C_{0j}$ be the circle determined by $c_{0j}%
\ $and $C_{0j}=\partial D\left(  p_{j},r\right)  .$ Then $r_{j}\leq\pi/2$ and
$p_{j}$ is on the left hand side of $C_{0j}.$ For any positive number
$\varepsilon\ll\delta$ with
\begin{equation}
\varepsilon<\frac{\min\{\delta,\delta_{1},\delta_{2},\delta_{3},\delta_{4}%
\}}{12\pi\left(  d^{\ast}+1\right)  m}, \label{co59}%
\end{equation}
let $C_{0j,\pm\varepsilon}=\partial D(p_{j},r_{j}\pm\varepsilon)\ $and write
$\mathcal{A}_{0j,\pm\varepsilon}=D(C_{0j},\varepsilon),$ which is the
$\varepsilon$ neighborhood of $C_{0j}$ on $S:$
\[
\mathcal{A}_{0j,\pm\varepsilon}:r_{j}-\varepsilon<d(w,p_{j})<r_{j}%
+\varepsilon,
\]
and let
\begin{equation}
R_{j,\delta,\varepsilon}=\overline{D(c_{0j},\varepsilon)}\backslash
D(\{q_{0j},q_{0,j+1}\},\delta), \label{co54}%
\end{equation}
which is compact and is the component of $\overline{\mathcal{A}_{0j,\pm
\varepsilon}}\backslash D(\{q_{0j},q_{0,j+1}\},\delta)$ containing
$c_{0j}\backslash D(\{q_{0j},q_{0,j+1}\},\delta).$ Then we have%
\[
\partial\mathcal{A}_{0j,\pm\varepsilon}=C_{0j,-\varepsilon}\cup
C_{0j,+\varepsilon}.
\]
Recall that $D(X,\delta)=\cup_{x\in X}D(x,\delta).$

By (\ref{co44}), (\ref{co51}) and (\ref{co49}), $\mathcal{C}_{\delta
}=\{\partial D(q_{0j},\delta):j\in M\}$ is consisted of disjoint circles. Then
we may assume that the positive numbers $\delta$ and $\varepsilon\ll\delta$
are small enough such that the following holds.

\begin{conclusion}
\label{co53}For each circle $J_{1}$ of $\mathcal{C}_{\delta}$ and each circle
$J_{2}\ $of the $3m_{0}$ circles $\left\{  C_{0j,\pm\varepsilon}:j\in
M_{0}\right\}  \cup\left\{  C_{0j}:j\in M_{0}\right\}  ,$ either $J_{1}\cap
J_{2}=\emptyset$ or $J_{1}$ and $J_{2}$ intersect almost perpendicularly; the
closed domain $R_{j,\delta,\varepsilon},j\in M_{0},$ defined by (\ref{co54}),
is a quadrilateral enclosed by four circular arcs contained in
$C_{0j,-\varepsilon}$, $\partial D(q_{0j},\delta),$ $C_{0j,+\varepsilon}$ and
$\partial D(q_{0,j+1},\delta):$%
\begin{align}
\partial R_{j,\delta,\varepsilon}  &  =-c_{0j,\delta,-\varepsilon}^{2}\left(
q_{0j,\delta,-\varepsilon}^{1},q_{0j,\delta,-\varepsilon}^{2}\right)
+\tau_{j,\delta,\varepsilon}^{1}\left(  q_{0j,\delta,-\varepsilon}%
^{1},q_{0j,\delta,\varepsilon}^{1}\right) \label{z1}\\
&  +c_{j,\delta,\varepsilon}^{2}\left(  q_{0j,\delta,\varepsilon}%
^{1},q_{0j,\delta,\varepsilon}^{2}\right)  +\tau_{j,\delta,\varepsilon}%
^{2}\left(  q_{0j,\delta,\varepsilon}^{2},q_{0j,\delta,-\varepsilon}%
^{2}\right)  ;\nonumber
\end{align}
and%
\begin{equation}
\min_{1\leq j\leq m_{0}}L(c_{0j,\delta,-\varepsilon}^{2})\geq\frac{4}{5}%
\min_{1\leq j\leq m_{0}}L(c_{0j})>5\delta. \label{co32}%
\end{equation}

\end{conclusion}

For each $j\in M_{0},$ it is clear that for sufficiently large $n,$
$q_{nj,\delta}^{1}$ and $q_{nj,\delta}^{2}$ are contained in the interior of
$\tau_{j,\delta,\varepsilon}^{1}\ $and $\tau_{j,\delta,\varepsilon}^{2}$
respectively.\medskip

\begin{center}
\textbf{Step 2 Two lifting results: Theorems \ref{key3} and \ref{key1}%
\medskip}
\end{center}

Let $t_{0}=2\pi\left(  d^{\ast}+1\right)  \sin\delta,$ which is $\left(
d^{\ast}+1\right)  $ times of the length of any circle on $S$ with radius
$\delta,$ and for $j\in M$ let
\[
\zeta_{nj,\delta}^{t_{0}}=\zeta_{nj,\delta}^{t_{0}}(t),t\in\left[
0,t_{0}\right]  ,
\]
be the locally simple path which describes $-\partial D(q_{0j},\delta)$ by
$d^{\ast}+1$ times, parametrized by length, oriented clockwise, and from
$\zeta_{nj,\delta}^{t_{0}}(0)=q_{n,j-1,\delta}^{2}$ to itself, say,
$\zeta_{nj,\delta}^{t_{0}}(t_{0})=q_{n,j-1,\delta}^{2}.$

For each $j\in M,$ there exists two numbers $\mathfrak{i}_{1}=\varphi
_{1}\left(  j\right)  $ and $\mathfrak{i}_{2}=\varphi_{2}\left(  j\right)  $
in $M_{0},$ which are uniquely determined by $j,$ such that $\mathfrak{i}%
_{1}+1\leq j\leq\mathfrak{i}_{2},$ and that each arc $\alpha_{0i}$ with $i\in
M$ and $\mathfrak{i}_{1}+1\leq i\leq\mathfrak{i}_{2}-1$ is a point, when
$\mathfrak{i}_{2}>\mathfrak{i}_{1}+1.$ In other words, $\alpha_{0\mathfrak{i}%
_{1}}$ and $\alpha_{0\mathfrak{i}_{2}}$ are terms in (\ref{pt3}) having
positive length and joined at $a_{0j},$ and $\alpha_{0\mathfrak{i}_{1}}$ is
before $\alpha_{0\mathfrak{i}_{2}}$ in the direction of $\partial\Delta.$ For
example, when $j\in M_{0},$ we have $\varphi_{2}\left(  j\right)  =j$ and
$\varphi_{1}\left(  j+1\right)  =j,$ and it is clear that
\[
\varphi_{1}(M)=\varphi_{1}(M_{0})=\varphi_{2}(M)=\varphi_{2}(M_{0})=M_{0}.
\]

\begin{remark}
In the remain part of this section, the discussion for the case $M_{0}%
\subsetneqq M$ makes the argument more complicated, without more deeper
meaning. The reader may understand the arguments only in the special case
$M_{0}=M,$ though we discuss in general for completeness. When $M_{0}=M,$ one
has $\mathfrak{i}_{1}=\varphi_{1}(j)=j-1$ and $\mathfrak{i}_{2}=\varphi
_{2}(j)=j\ $for all $j\in M.$
\end{remark}

We first prove the following theorem.

\begin{theorem}
\label{key3}\label{CL1}Let $j\in M$. Then for sufficiently large $n,$ there
exists a number $t_{nj}\in(0,t_{0}),$ such that the for $\mathfrak{i}%
_{1}=\varphi_{1}\left(  j\right)  $ and $\mathfrak{i}_{2}=\varphi_{2}\left(
j\right)  $ the following hold.

(i) $\zeta_{nj,\delta}^{t_{nj}}=\zeta_{nj,\delta}^{t_{0}}|_{\left[
0,t_{nj}\right]  }$ has an $f_{n}$-lift $\eta_{nj,\delta}^{t_{nj}}%
=\eta_{nj,\delta}^{t_{nj}}(t),t\in\lbrack0,t_{nj}],$ from $a_{n\mathfrak{i}%
_{1},\delta}^{2}\ $to $a_{n\mathfrak{i}_{2},\delta}^{1},$ say, $\eta
_{nj,\delta}^{t_{nj}}(0)$ is the terminal point $a_{n\mathfrak{i}_{1},\delta
}^{2}$ of $\alpha_{n\mathfrak{i}_{1},\delta}^{2},$ $\eta_{nj,\delta}^{t_{nj}%
}(t_{nj})$ is the initial point$\ a_{n\mathfrak{i}_{2},\delta}^{1}$ of
$\alpha_{n\mathfrak{i}_{2},\delta}^{2},$ and
\[
f_{n}\left(  \eta_{nj,\delta}^{t_{nj}}(t)\right)  =\zeta_{nj,\delta}^{t_{nj}%
}(t),t\in\lbrack0,t_{nj}].
\]

(ii) The interior $\eta_{nj,\delta}^{t_{nj}\circ}$ of $\eta_{nj,\delta
}^{t_{nj}}\ $is contained in $\Delta,$ say, $\eta_{nj,\delta}^{t_{nj}\circ
}(t)\in\Delta$ for all $t\in(0,t_{nj}).$

(iii)
\begin{equation}
\eta_{nj,\delta}^{t_{nj}}\subset D_{f_{n}}(a_{0j},t_{nj}+2\pi\delta)\subset
D_{f_{n}}(a_{0j},\delta_{1}/3), \label{cc3}%
\end{equation}

and for each pair $\{i,j\}\subset M$ with $a_{0i}\neq a_{0j}$%
\begin{equation}
\eta_{nj,\delta}^{t_{nj}}\cap\eta_{nj,\delta}^{t_{nj}}=\emptyset.
\label{nointers}%
\end{equation}

(iv) \label{cc12}$\eta_{nj,\delta}^{t_{nj}}$ is a simple arc in $\overline
{\Delta}$ and
\begin{equation}
\eta_{nj,\delta}^{t_{nj}\circ}\cap f_{n}^{-1}(E_{q})=\emptyset, \label{cc17}%
\end{equation}
and thus $f_{n}$ has no branch point on $\eta_{nj,\delta}^{t_{nj}}%
=\eta_{nj,\delta}^{t_{nj}}\left(  a_{n\mathfrak{i}_{1},\delta}^{2}%
,a_{n\mathfrak{i}_{2},\delta}^{1}\right)  \ $(by convention, $\eta_{nj,\delta
}^{t_{nj}}\left(  a_{n\mathfrak{i}_{1},\delta}^{2},a_{n\mathfrak{i}_{2}%
,\delta}^{1}\right)  $ indicates $\eta_{nj,\delta}^{t_{nj}}$ is a path from
$a_{n\mathfrak{i}_{1},\delta}^{2}\ $to $a_{n\mathfrak{i}_{2},\delta}^{1}$).
\end{theorem}

\begin{proof}
By (\ref{cc15}) and Definition \ref{circu} (c) (iv) for $\mathcal{F}\left(
L,m\right)  \supset\mathcal{F}_{r}\left(  L,m\right)  $ we have

\begin{condition}
\label{condi2}$\zeta_{nj,\delta}^{t_{0}}$ never passes any point of $E_{q}$
and for sufficiently large $n,$ $f_{n}$ is homeomorphic in neighborhoods of
$a_{n\mathfrak{i}_{1},\delta}^{2}$ and $a_{n\mathfrak{i}_{2},\delta}^{1}$ in
$\overline{\Delta}.$ Thus $\zeta_{nj,\delta}^{t_{0}}$ never passes through any
branch point of $f_{n}.$
\end{condition}

From this condition we have: there exists a maximal number $t_{nj}\in
(0,t_{0}]\ $satisfying the following condition.

\begin{condition}
\label{cond1}(A) The part $\zeta_{nj,\delta}^{t_{nj}}$ of $\zeta_{nj,\delta
}^{t_{0}}$ has an $f_{n}$-lift $\eta_{nj,\delta}^{t_{nj}}(t)$ starting from
$a_{n\mathfrak{i}_{1},\delta}^{2}\in\alpha_{n\mathfrak{i}_{1},\delta}^{2}%
\ $(recall that $a_{n\mathfrak{i}_{1},\delta}^{2}$ is the terminal point of
$\alpha_{n\mathfrak{i}_{1},\delta}^{2}$).

(B) $\zeta_{nj,\delta}^{t_{nj}\circ}\subset\Delta.$
\end{condition}

By Condition \ref{condi2}, the lift $\eta_{nj,\delta}^{t_{nj}}$ is uniquely
determined by $t_{nj}$ and the initial point $\eta_{nj,\delta}%
(0)=a_{n\mathfrak{i}_{1},\delta}^{2}\ $and satisfies (\ref{cc17}). Since
$\deg_{\max}f_{n}\leq d^{\ast}\ $and $t_{nj}$ is maximal, we have by Lemma
\ref{Ri} that
\begin{equation}
t_{nj}\in(0,t_{0})\mathrm{\ and\ }\eta_{nj,\delta}^{t_{nj}}(t_{nj})\in
\partial\Delta, \label{49-1}%
\end{equation}
We have proved that $\eta_{nj,\delta}^{t_{nj}}$ satisfies (ii) and (\ref{cc17}).

Now, we show that $\eta_{nj,\delta}^{t_{nj}}$ satisfies (iii). Since
$\eta_{nj,\delta}^{t_{nj}}(0)=a_{n\mathfrak{i}_{1},\delta}^{2}$ and
$\eta_{nj,\delta}^{t_{nj}}$ is parametrized by $d_{f_{n}}$-length, by
(\ref{505-1}) we have for sufficiently large $n,$%
\begin{align*}
\eta_{nj,\delta}^{t_{nj}}  &  \subset\overline{D_{f_{n}}(a_{n\mathfrak{i}%
_{1},\delta}^{2},t_{nj})}\subset\overline{D_{f_{n}}(a_{0j},t_{nj}+d_{f_{n}%
}(a_{n\mathfrak{i}_{1},\delta}^{2},a_{0j}))}\\
&  \subset D_{f_{n}}(a_{0j},t_{nj}+2\pi\delta).
\end{align*}
On the other hand, by (\ref{49-1}), we have
\[
t_{nj}+2\pi\delta<2\pi\delta\left(  d^{\ast}+1\right)  +2\pi\delta<4\pi
\delta\left(  d^{\ast}+1\right)  <\delta_{1}/3.
\]
Therefore (\ref{cc3}) holds, which with (\ref{co44}) and Lemma \ref{dfset},
implies that for each pair $\{i,j\}$ in $M$ with $a_{0i}\neq a_{0j}$%
\begin{align*}
d_{f_{n}}(\eta_{ni,\delta}^{t_{ni}},\eta_{nj,\delta}^{t_{nj}})  &  \geq
d_{f_{n}}\left(  D_{f_{n}}(a_{0i},\frac{\delta_{1}}{3}),D_{f_{n}}(a_{0j}%
,\frac{\delta_{1}}{3})\right) \\
&  \geq d_{f_{n}}(a_{0i},a_{0j})-\frac{2\delta_{1}}{3}>\frac{\delta_{1}}{3}>0.
\end{align*}
That is to say (\ref{nointers}) holds, and (iii) is proved.

By Condition \ref{condi2}, the $f_{n}$-lift $\eta_{nj,\delta}^{t_{nj}}$ is
simple, and thus (iv) is true.

It remains to prove (i). For sufficiently large $n,$ it is clear that for each
$i\in M_{0},\alpha_{ni}\subset D_{f_{n}}(\alpha_{0i},\delta),$ and thus by
(\ref{cc3}) and Lemma \ref{dfset} we have, for $i\in M_{0}$ with
$i\neq\mathfrak{i}_{1},\mathfrak{i}_{2},$ that
\begin{align*}
d_{f_{n}}\left(  \eta_{nj,\delta}^{t_{nj}},\alpha_{ni}\right)   &  \geq
d_{f_{n}}\left(  D_{f_{n}}(a_{0j},\frac{\delta_{1}}{3}),D_{f_{n}}(\alpha
_{0i},\delta)\right) \\
&  \geq d_{f_{n}}\left(  a_{0j},\alpha_{0i}\right)  -\frac{\delta_{1}}%
{3}-\delta,
\end{align*}
and then by (\ref{co44}) $d_{f_{n}}\left(  \eta_{nj,\delta},\alpha
_{ni}\right)  >\delta_{1}/2.$ Then for sufficiently large $n,$ $\eta
_{nj,\delta}^{t_{nj}}\left(  t_{nj}\right)  \cap\alpha_{ni}\neq\emptyset$
holds only for $i=\mathfrak{i}_{1}\ $or $\mathfrak{i}_{2}.$ Thus, when $t$
tends to $t_{nj}$ in $[0,t_{nj}]$, $\eta_{nj,\delta}(t)$ tend to
$\alpha_{n\mathfrak{i}_{1}}\ $or $\alpha_{n\mathfrak{i}_{2}},$ and it is clear
that we only have $\eta_{nj,\delta}(t_{nj})=a_{n\mathfrak{i}_{2}}^{1}\in
\alpha_{n\mathfrak{i}_{2}}.$ (i) has been proved and Theorem \ref{CL1} is
proved completely.
\end{proof}

By taking subsequence, we may assume that for each $j\in M$ there exists
$v_{j}\in\mathbb{N}^{0},$ independent of $n,$ such that
\begin{equation}
d^{\ast}+1\geq\frac{t_{nj}}{2\pi\sin\delta}>v_{j}\geq\frac{t_{nj}}{2\pi
\sin\delta}-1. \label{za10}%
\end{equation}

For each $j\in M_{0},$ we write%
\begin{equation}
\tau_{nj,\delta,\varepsilon}^{1,L}=\tau_{j,\delta,\varepsilon}^{1}\left(
q_{0j,\delta,-\varepsilon}^{1},q_{nj,\delta}^{1}\right)  ,\tau_{nj,\delta
,\varepsilon}^{2,L}=\tau_{j,\delta,\varepsilon}^{2}\left(  q_{nj,\delta}%
^{2},q_{0j,\delta,-\varepsilon}^{2}\right)  , \label{az-8}%
\end{equation}
say, $\tau_{nj,\delta,\varepsilon}^{1,L}$ is the arc of $\tau_{j,\delta
,\varepsilon}^{1}$ from $q_{0j,\delta,-\varepsilon}^{1}$ to $q_{nj,\delta}%
^{1}$ and $\tau_{nj,\delta,\varepsilon}^{2,L}$ is the arc of $\tau
_{j,\delta,\varepsilon}^{2}$ from $q_{nj,\delta}^{2}\ $to $q_{0j,\delta
,-\varepsilon}^{2}.$ In other words, $\tau_{nj,\delta,\varepsilon}^{1,L}$ and
$\tau_{nj,\delta,\varepsilon}^{2,L}$ are the parts of $\tau_{nj,\delta
,\varepsilon}^{1}$ and $\tau_{nj,\delta,\varepsilon}^{2}$ on the left hand
side of $c_{nj},$ respectively. Then by Theorem \ref{CL1} and properties of
path lifts, we have the following.

\begin{claim}
\label{CL2}For each $j\in M,\eta_{nj,\delta}^{t_{nj}}=\eta_{nj,\delta}%
^{t_{nj}}(t),t\in\lbrack0,t_{nj}],$ is a simple path in $\overline{\Delta}$
from $a_{n\varphi_{1}\left(  j\right)  ,\delta}^{2}$ to $a_{n\varphi
_{2}(j),\delta}^{1},$
\[
\eta_{nj,\delta}^{t_{nj}\circ}=\eta_{nj,\delta}^{t_{nj}}|_{(0,t_{nj})}%
\subset\Delta,
\]
and $\eta_{nj,\delta}^{t_{nj}}$ is the $f_{n}$-lift of the path%
\begin{equation}
\zeta_{nj,\delta}^{t_{nj}}=\zeta_{nj,\delta}^{t_{0}}|_{[0,t_{nj}]}%
=\tau_{n\varphi_{1}(j),\delta,\varepsilon}^{2,L}+\zeta_{j,\delta,\varepsilon
}+\tau_{n\varphi_{2}\left(  j\right)  ,\delta,\varepsilon}^{1,L},
\label{conorL}%
\end{equation}
where $\tau_{n\varphi_{1}(j),\delta,\varepsilon}^{2,L}$ and $\tau
_{n\varphi_{2}\left(  j\right)  ,\delta,\varepsilon}^{1,L}$ are defined by
(\ref{az-8}), and $\zeta_{j,\delta,\varepsilon}$ can be expressed as
\[
\zeta_{j,\delta,\varepsilon}=\overset{v_{j}\;\mathrm{copies\;of\;}-\partial
D(q_{0j},\delta)\mathrm{\;}}{\overbrace{-\partial D(q_{0j},\delta
)-\dots-\partial D(q_{0j},\delta)}}-\kappa,
\]
in which each $-\partial D(q_{0j},\delta)$ is regarded as a closed simple path
from $q_{0\varphi_{1}(j),\delta,-\varepsilon}^{2}$ to itself, and $-\kappa$ is
the simple arc of $-\partial D(q_{0j},\delta)$ from $q_{0\varphi_{1}%
(j),\delta,-\varepsilon}^{2}$ to $q_{0\varphi_{2}(j),\delta,-\varepsilon}%
^{1},$ and $\zeta_{j,\delta,-\varepsilon}=-\kappa$ when $v_{j}=0. $
\end{claim}

Since $\{v_{j}\}_{j=1}^{M}$ is independent of $n,$ and $\zeta_{j,\delta
,\varepsilon}$ is starting from the point $q_{0\varphi_{1}(j),\delta
,-\varepsilon}^{2}\in C_{0\varphi_{1}(j),-\varepsilon}\cap\partial D\left(
q_{0j},\delta\right)  $ to the point $q_{0\varphi_{2}(j),\delta,-\varepsilon
}^{1}\in C_{0\varphi_{2}(j),-\varepsilon}\cap\partial D\left(  q_{0j}%
,\delta\right)  ,$ which are also independent of $n,$ we have:

\begin{claim}
\label{ind}For each $j\in M,\zeta_{j,\delta,\varepsilon}$ is a subarc of
$\zeta_{nj,\delta}^{t_{nj}}|_{[0,t_{nj}]}=\zeta_{nj,\delta}^{t_{0}%
}|_{[0,t_{nj}]}$ which is independent of $n,$ and we can write
\begin{equation}
\eta_{nj,\delta,\varepsilon}^{t_{nj}}=\mathfrak{t}_{n\varphi_{1}%
(j),\delta,\varepsilon}^{2}+\eta_{nj,\delta,\varepsilon}+\mathfrak{t}%
_{n\varphi_{2}(j),\delta,\varepsilon}^{1}, \label{az-5}%
\end{equation}
with
\begin{equation}
\left(  f_{n},\mathfrak{t}_{n\varphi_{1}(j),\delta,\varepsilon}^{2}\right)
=\tau_{n\varphi_{1}(j),\delta,\varepsilon}^{2,L},\;\left(  f_{n}%
,\mathfrak{t}_{n\varphi_{2}(j),\delta,\varepsilon}^{1}\right)  =\tau
_{n\varphi_{2}(j),\delta,\varepsilon}^{1,L} \label{az-6}%
\end{equation}
and
\[
\left(  f_{n},\eta_{nj,\delta,\varepsilon}\right)  =\zeta_{j,\delta
,\varepsilon}=\zeta_{j,\delta,\varepsilon}\left(  q_{0\varphi_{1}\left(
j)\right)  ,\delta,-\varepsilon}^{2},q_{0\varphi_{2}\left(  j)\right)
,\delta,-\varepsilon}^{1}\right)  .
\]

\end{claim}

It is clear that for each $j\in M_{0},$
\begin{equation}%
\begin{tabular}
[c]{ll}%
$\max\left\{  L(\tau_{nj,\delta,\varepsilon}^{1,L}),L(\tau_{nj,\delta
,\varepsilon}^{2,L})\right\}  $ & \\
$=\max\left\{  L(f_{n},\mathfrak{t}_{nj,\delta,-\varepsilon}^{1}%
),L(f_{n},\mathfrak{t}_{nj,\delta,-\varepsilon}^{2})\right\}  $ &
$=\varepsilon+o(\varepsilon)\mathrm{\ as\ }\varepsilon\rightarrow0.$%
\end{tabular}
\ \ \ \ \ \label{cc4}%
\end{equation}
We write%
\begin{equation}
\eta_{nj,\delta,\varepsilon}=\eta_{nj,\delta,\varepsilon}\left(
a_{0\varphi_{1}\left(  j)\right)  ,\delta,-\varepsilon}^{2},a_{0\varphi
_{2}\left(  j)\right)  ,\delta,-\varepsilon}^{1}\right)  \label{beg-end}%
\end{equation}

Now we prove the following Theorem.

\begin{theorem}
\label{CL4}\label{key1}For sufficiently large $n,$ and each $j\in M_{0},$ the
following hold:

(i) $c_{0j,\delta,-\varepsilon}^{2}$ has an $f_{n}$-lift $\alpha
_{nj,\delta,-\varepsilon}^{2}$ such that the initial point of $\alpha
_{nj,\delta,-\varepsilon}^{2}$ equals the terminal point $a_{0\varphi
_{2}\left(  j\right)  ,\delta,-\varepsilon}^{1}=a_{0j,\delta,-\varepsilon}%
^{1}$ of $\eta_{nj,\delta,\varepsilon},$ and the terminal point $a_{0j,\delta
,-\varepsilon}^{2}$ of $\alpha_{nj,\delta,-\varepsilon}^{2}$ equals the
initial point $a_{0\varphi_{1}(j+1),\delta,-\varepsilon}^{2}=a_{0j,\delta
,-\varepsilon}^{2}$ of $\eta_{n,j+1,\delta,\varepsilon};$ and moreover,%
\begin{equation}
\alpha_{nj,\delta,-\varepsilon}^{2}\subset D_{f_{n}}(\alpha_{0j,\delta}%
^{2},2\pi\varepsilon)\subset D_{f_{n}}(\partial\Delta,\delta_{1}). \label{m1}%
\end{equation}

(ii) For any other $i\in M_{0}$ with $i\neq j,$ $\alpha_{ni,\delta
,\varepsilon}^{2}\cap\alpha_{nj,\delta,\varepsilon}^{2}=\emptyset.$
\end{theorem}

\begin{proof}
We first show that (i) implies (ii). By (\ref{m1}) we have for large enough
$n$ and each pair of distinct $i$ and $j$ in $M_{0}$ that
\[
d_{f_{n}}\left(  \alpha_{ni,\delta,-\varepsilon}^{2},\alpha_{nj,\delta
,-\varepsilon}^{2}\right)  \geq d_{f_{n}}\left(  D_{f_{n}}(\alpha_{0i,\delta
}^{2},2\pi\varepsilon),D_{f_{n}}(\alpha_{0j,\delta}^{2},2\pi\varepsilon
)\right)  ,
\]
and then, by Lemma \ref{dfset}, (\ref{co59}) and (\ref{506-1}), we have%
\[
d_{f_{n}}\left(  \alpha_{ni,\delta,-\varepsilon}^{2},\alpha_{nj,\delta
,-\varepsilon}^{2}\right)  \geq d_{f_{n}}\left(  \alpha_{0i,\delta}^{2}%
,\alpha_{0j,\delta}^{2},\right)  -4\pi\varepsilon\geq\delta_{4}-4\pi
\varepsilon>\delta_{4}/2>0.
\]
This implies (ii).

Consider the closed quadrilateral $R_{j,\delta,\varepsilon}$ defined by
(\ref{co54}) for each $j\in M_{0}$. By Conclusion \ref{co53}, it is clear that
for sufficiently large $n,$ the interior $c_{nj,\delta}^{2\circ}$ of
$c_{nj,\delta}^{2}=c_{nj}\cap R_{j,\delta,\varepsilon}$ is contained in
$R_{j,\delta,\varepsilon}^{\circ}$ with $\partial c_{nj,\delta}^{2}%
\subset\partial R_{j,\delta,\varepsilon},$ and thus $c_{nj,\delta}^{2}$
divides $R_{j,\delta,\varepsilon}$ into two closed quadrilateral which are
closed Jordan domains$\ $and the one on the left hand side of $c_{nj,\delta
}^{2}$ is denoted by $R_{nj,\delta,\varepsilon}^{L}.$

Since $c_{nj,\delta}^{2}$ converges to $c_{0j,\delta}^{2}$ as $n\rightarrow
\infty$, $c_{nj,\delta}^{2}\cap\partial R_{j,\delta,\varepsilon}=c_{nj}%
\cap\partial R_{j,\delta,\varepsilon}\ $is consisted of the two endpoints
$q_{nj,\delta}^{1}$ and $q_{nj,\delta}^{2}$ of $c_{nj,\delta}^{2},$ which are
terminal and initial points of $\tau_{nj,\delta,\varepsilon}^{1,L}$ and
$\tau_{nj,\delta,\varepsilon}^{2,L}$ respectively. By (\ref{z1}) and
(\ref{az-8})--(\ref{az-6}), we may assume (by taking subsequence) that
$\partial R_{nj,\delta,\varepsilon}^{L}$ has the partition
\begin{equation}
\partial R_{nj,\delta,\varepsilon}^{L}=-c_{0j,\delta,-\varepsilon}^{2}%
+\tau_{nj,\delta,\varepsilon}^{1,L}+c_{nj,\delta}^{2}+\tau_{nj,\delta
,\varepsilon}^{2,L}, \label{z2}%
\end{equation}
with $c_{nj,\delta}^{2}\cap c_{0j,\delta,-\varepsilon}^{2}=\emptyset,$
corresponding to (\ref{z1}).

We will show the following:

\begin{claim}
\label{globallift}$f_{n}^{-1}$ has a univalent branch $\tilde{g}%
_{nj,\delta,\varepsilon}$ defined on closed quadrilateral $R_{nj,\delta
,\varepsilon}^{L}$ such that%
\[
\tilde{g}_{nj,\delta,\varepsilon}(c_{nj,\delta}^{2})=\tilde{g}_{nj,\delta
,\varepsilon}(c_{nj}\cap R_{j,\delta,\varepsilon})=\tilde{g}_{nj,\delta
,\varepsilon}(c_{nj}\cap R_{j,\delta,\varepsilon}^{L})=\alpha_{nj,\delta}%
^{2},
\]
and
\begin{equation}
\tilde{g}_{nj,\delta,\varepsilon}(R_{nj,\delta,\varepsilon}^{L}\backslash
c_{nj,\delta}^{2})\subset\Delta\backslash f_{n}^{-1}(E_{q}). \label{in-Delta}%
\end{equation}

\end{claim}

It is clear that there exists a family $\{\tau_{q},q\in c_{nj,\delta}^{2}\}$
of simple circular arcs which is a continuous fibration of $R_{nj,\delta
,\varepsilon}^{L}$. Precisely speaking,
\[
R_{nj,\delta,\varepsilon}^{L}=\cup_{q\in c_{nj,\delta}^{2}}\tau_{q},
\]%
\[
\tau_{q_{nj,\delta}^{1}}=-\tau_{nj,\delta,\varepsilon}^{1,L}=-\tau
_{j,\delta,\varepsilon}^{1}\left(  q_{0j,\delta,-\varepsilon}^{1}%
,q_{nj,\delta}^{1}\right)  ,
\]%
\[
\tau_{q_{nj,\delta}^{2}}=\tau_{nj,\delta,\varepsilon}^{2,L}=\tau
_{j,\delta,\varepsilon}^{2}\left(  q_{nj,\delta}^{2},q_{0j,\delta
,-\varepsilon}^{2}\right)  ,
\]%
\[
\tau_{q}\cap\tau_{q^{\prime}}=\emptyset,\mathrm{\ for\ }\left\{  q,q^{\prime
}\right\}  \subset c_{nj,\delta}^{2}\mathrm{\ with\ }q\neq q^{\prime},
\]
each $\tau_{q}$ is a simple circular path in $R_{nj,\delta,\varepsilon}^{L}$
from $q\in c_{nj,\delta}^{2}$ to a point $\psi\left(  q\right)  \in
c_{j,\delta,-\varepsilon}^{2}$ with%
\begin{equation}
L(\tau_{q})\leq\max\{L(\tau_{nj,\delta,\varepsilon}^{1,L}),L(\tau
_{nj,\delta,\varepsilon}^{2,L})\}<\pi\varepsilon\mathrm{\ for\ }q\in
c_{nj,\delta}^{2}. \label{ptR}%
\end{equation}
Moreover, $\tau_{q}\backslash\{q\}\subset T_{nj},$ where $T_{nj}$ is the
(open) disk $T_{nj}$ enclosed by the circle $C_{nj}$ determined by $c_{nj}.$

It is clear that $f_{n}$ restricted a neighborhood of $\alpha_{nj,\delta}^{2}
$ in $\overline{\Delta}$ is a homeomorphism onto a neighborhood of
$c_{nj,\delta}^{2}$ in the closed disk $\overline{T_{nj}}.$ Thus, for each
$q\in c_{nj,\delta}^{2},$ it is clear that $\tau_{q}=\tau_{q}\left(
q,\psi\left(  q\right)  \right)  $ contains a MAXIMAL arc $\tau_{q}^{\prime
}=\tau_{q}\left(  q,q^{\prime}\right)  $ such that $\tau_{q}^{\prime}$ has an
$f_{n}$-lift $\beta_{n,a}=\beta_{n,a}\left(  a,a^{\prime}\right)  $ with
$\tau_{q}^{\prime}=\left(  f_{n},\beta_{n,a}\right)  $ and $\beta_{n,a}%
^{\circ}\subset\Delta.$ We show that $a^{\prime}\in\Delta$ say $\beta
_{a}\backslash\{a\}\subset\Delta.$

Since $L(\tau_{q}^{\prime})\leq L(\tau_{q})<\pi\varepsilon$ we have
\[
d_{f_{n}}\left(  a,a^{\prime}\right)  \leq L(\tau_{q}^{\prime})<\pi
\varepsilon<\frac{\delta_{4}}{4},
\]
and then, by $a\in\alpha_{nj,\delta}^{2},$ we have%
\[
d_{f_{n}}\left(  \left(  \partial\Delta\right)  \backslash\alpha
_{nj},a^{\prime}\right)  \geq d_{f_{n}}\left(  \left(  \partial\Delta\right)
\backslash\alpha_{nj},\alpha_{nj,\delta}^{2}\right)  -d_{f_{n}}\left(
a,a^{\prime}\right)  >\delta_{4}-\pi\varepsilon>0.
\]
This implies $a^{\prime}\notin\left(  \partial\Delta\right)  \backslash
\alpha_{nj}.$ On the other hand, $q^{\prime}\in T_{nj}=T_{nj}\backslash
c_{nj}$ and thus $a^{\prime}\notin\alpha_{nj}.$ We have proved that
$a^{\prime}\in\Delta$. By Lemma \ref{Ri}, $\beta_{n,p}$ is the $f_{n}$-lift of
the whole arc $\tau_{q}$ and $f_{n}(a^{\prime})=\psi\left(  q\right)  ,$ the
endpoint of $\tau_{q}.$

Now that $\beta_{n,a}$ is the $f_{n}$-lift of $\tau_{q}$ with $\beta
_{n,a}\backslash\{a\}\subset\Delta$ we have that for each $x\in\beta
_{n,a}\backslash\{a\},$ $d_{f_{n}}\left(  x,\partial\Delta\right)  \leq
L(\tau_{q})<\varepsilon\pi<\pi\delta<\delta_{1}.$ Thus we have $\beta
_{n,a}\backslash\{a\}\cap f_{n}^{-1}\left(  E_{q}\right)  =\emptyset,$ and
thus
\begin{equation}
\left(  R_{nj,\delta,\varepsilon}^{L}\backslash c_{nj,\delta}^{2}\right)  \cap
E_{q}=\emptyset. \label{za4}%
\end{equation}
$\ $and $f_{n}$ has no branch value on $R_{nj,\delta,\varepsilon}%
^{L}\backslash c_{nj,\delta}^{2}.$ By Lemma \ref{b-in}, Claim \ref{globallift}
is proved.

We are still considering the fixed $j\in M_{0}.$ Then we have $\varphi
_{2}\left(  j\right)  =j$ and $\varphi_{1}\left(  j+1\right)  =j+1.$ Then the
edge $\tau_{q_{nj,\delta}^{1}}=\tau_{q_{nj,\delta,\varepsilon}^{1}}^{1,L}$ of
$\partial R_{nj,\delta,\varepsilon}^{L}\ $is the arc $\tau_{n\varphi
_{2}\left(  j\right)  ,\delta,\varepsilon}^{1,L}=\tau_{nj,\delta,\varepsilon
}^{1,L}$ of $\zeta_{nj,\delta}|_{[0,t_{nj}]}$ in (\ref{conorL}) and the edge
$\tau_{q_{nj,\delta}^{2}}=\tau_{q_{nj,\delta,\varepsilon}^{2}}^{2,L}$ of
$\partial R_{nj,\delta,\varepsilon}^{L}\ $is the arc $\tau_{n\varphi
_{1}\left(  j+1\right)  ,\delta,\varepsilon}^{1,L}=\tau_{n,j+1,\delta
,\varepsilon}^{1,L}$ of $\zeta_{n,j+1,\delta}|_{[0,t_{n,j+1}]}$ in
(\ref{conorL}) for $j+1.$ Thus the arc $\mathfrak{t}_{n\varphi_{2}%
(j),\delta,\varepsilon}^{1}=\mathfrak{t}_{nj,\delta,\varepsilon}^{1}$ in
(\ref{az-5}) equals $\beta_{n,a_{nj,\delta}^{1}}=\tilde{g}_{nj,\delta
,\varepsilon}\left(  \tau_{q_{nj,\delta}^{1}}^{1,L}\right)  $ and
$\mathfrak{t}_{n\varphi_{1}(j+1),\delta,\varepsilon}^{2}=\mathfrak{t}%
_{n,j+1,\delta,\varepsilon}^{2}$ in (\ref{az-5}) equals $\beta_{n,a_{nj,\delta
}^{2}}=\tilde{g}_{nj,\delta,\varepsilon}\left(  \tau_{q_{nj,\delta}^{2}}%
^{1,L}\right)  ,\ $since $\mathfrak{t}_{nj,\delta,\varepsilon}^{1}%
\backslash\left\{  a_{nj,\delta}^{1}\right\}  $ and $\mathfrak{t}%
_{n,j+1,\delta,\varepsilon}^{2}\backslash\left\{  a_{nj,\delta}^{2}\right\}  $
contains no branch point of $f_{n}\ $and $f_{n}$ is homeomorphic in a
neighborhood of $\alpha_{nj,\delta}^{2}.$ Let $\alpha_{nj,\delta,-\varepsilon
}^{2}=\tilde{g}_{nj,\delta,\varepsilon}\left(  c_{j,\delta,-\varepsilon}%
^{2}\right)  .$ Then $\alpha_{nj,\delta,-\varepsilon}^{2}$ satisfies (ii),
except (\ref{m1}).

By (\ref{ptR}) we have $R_{nj,\delta,\varepsilon}^{L}\subset D\left(
c_{nj,\delta}^{2},\varepsilon\pi\right)  ,$ which implies%
\begin{equation}
\alpha_{nj,\delta,-\varepsilon}^{2}\subset\tilde{g}_{nj,\delta.\varepsilon
}\left(  R_{nj,\delta,\varepsilon}\right)  \subset D_{f_{n}}(\alpha
_{nj,\delta}^{2},\pi\varepsilon). \label{cc13}%
\end{equation}
Since $\alpha_{nj,\delta}^{2}$ converges $\alpha_{0j}^{2}\subset\alpha
_{0j}^{\circ}$ as $n\rightarrow\infty,$ we have by Conclusion \ref{inver},
\[
D_{f_{n}}(\alpha_{nj,\delta}^{2},\pi\varepsilon)\subset D_{f_{n}}%
(\alpha_{0j,\delta}^{2},2\pi\varepsilon)
\]
for sufficiently large $n$. Then (\ref{m1}) holds and Theorem \ref{CL4} is proved.
\end{proof}

By the way, by (\ref{za4}) we have
\begin{equation}
\left(  \tilde{g}_{nj,\delta,\varepsilon}(R_{nj,\delta,\varepsilon}%
^{L})\backslash\alpha_{nj,\delta}^{2}\right)  \cap f_{n}^{-1}(E_{q}%
)=\emptyset. \label{za4-1}%
\end{equation}

It is clear that the following hold.

\begin{conclusion}
\label{co62}For each $j\in M_{0},$ the definition of $R_{nj,\delta
,\varepsilon}^{L}\ $is valid for $n=0,$ say, $R_{0j,\delta,\varepsilon}^{L}$
is well defined and we have%
\begin{equation}
\partial R_{0j,\delta,\varepsilon}^{L}=-c_{0j,\delta,-\varepsilon}^{2}%
+\tau_{0j,\delta,\varepsilon}^{1,L}+c_{0j,\delta}^{2}+\tau_{0j,\delta
,\varepsilon}^{2,L}, \label{za3}%
\end{equation}
where $\tau_{0j,\delta,\varepsilon}^{1,L}\ $and $\tau_{0j,\delta,\varepsilon
}^{2,L}\ $are the parts of $\tau_{j,\delta,\varepsilon}^{1}$ and
$\tau_{j,\delta,\varepsilon}^{2}\ $on the left hand side of $c_{0j}.$
\end{conclusion}

Since for any neighborhood $V$ of $c_{0j,\delta}^{2}$ on $S$ the closed
domains $R_{nj,\delta,\varepsilon}^{L}\ $and $R_{0j,\delta,\varepsilon}^{L}$
coincide outside $V$ when $n$ is large enough, (\ref{za4}) implies%
\begin{equation}
\left(  R_{0j,\delta,\varepsilon}^{L}\backslash c_{0j,\delta}^{2}\right)  \cap
E_{q}=\emptyset. \label{cc16}%
\end{equation}
\medskip

\begin{center}
\textbf{Step 3 The construct of the Jordan curve }$\gamma_{n,\delta
,\varepsilon}$ \textbf{in }$\Delta$\medskip
\end{center}

Let%

\begin{equation}
\gamma_{n,\delta,\varepsilon}=\sum_{j\in M_{0}}\left(  \alpha_{n,\varphi
_{1}\left(  j\right)  ,\delta,-\varepsilon}^{2}+\eta_{nj,\delta,\varepsilon
}\right)  , \label{4}%
\end{equation}
and%
\[
\tilde{\Gamma}_{\delta,\varepsilon}=\sum_{j\in M_{0}}\left(  c_{\varphi
_{1}\left(  j\right)  ,\delta,-\varepsilon}^{2}+\zeta_{j,\delta,\varepsilon
}\right)  ,
\]
then by Theorems \ref{CL1} and \ref{CL4}, Calim \ref{ind} and (\ref{beg-end}),
we see that $\gamma_{n,\delta,\varepsilon}$ is a closed curve in $\Delta$,
$\tilde{\Gamma}_{\delta,\varepsilon}$ is a closed curve on $S,$ and,%
\[
\tilde{\Gamma}_{\delta,\varepsilon}=\left(  f_{n},\gamma_{n,\delta
,\varepsilon}\right)  .
\]

We assume that $\tilde{\Gamma}_{\delta,\varepsilon}$ and $\gamma
_{n,\delta,\varepsilon}$ are parametrized by length. We will show the claim:

\begin{claim}
\label{CL5}For sufficiently large $n,\gamma_{n,\delta,\varepsilon}$ is a
simple curve in $\Delta$ which depends on $n,\delta$ and $\varepsilon,$ while
$\tilde{\Gamma}_{\delta,\varepsilon}=\left(  f_{n},\gamma_{n,\delta
,\varepsilon}\right)  $ is a closed curve on $S$ which is independent of $n.$
\end{claim}

\begin{proof}
For each $j\in M_{0},$ by Claim \ref{ind} and the fact that $c_{j-1,\delta
,-\varepsilon}^{2}$ is independent of $n,$ $\tilde{\Gamma}_{\delta
,\varepsilon}$ is also independent of $n,$ and so, to prove the claim, it
suffices to prove that $\gamma_{n,\delta,\varepsilon}$ is simple.

For each $j\in M_{0},$ by Theorem \ref{CL1} (iv), $\eta_{nj,\delta}^{t_{nj}}$
is simple, and then by Claim \ref{ind}, $\eta_{nj,\delta,\varepsilon}\ $as a
subarc of $\eta_{nj,\delta}^{t_{nj}}\ $is also simple and has distinct
endpoints; and by Theorem \ref{CL4} (i) $\alpha_{nj,\delta,-\varepsilon}^{2}$
is simple with distinct endpoints. On the other hand, it it is clear, by
Theorems \ref{key3} and \ref{key1}, that $c_{nj,\delta,-\varepsilon}^{2}%
\cap\zeta_{nj,\delta,\varepsilon}=\{q_{j,\delta,-\varepsilon}^{1}\}$ and
$c_{n\varphi_{1}\left(  j\right)  ,\delta,-\varepsilon}^{2}\cap\zeta
_{nj,\delta,\varepsilon}=\{q_{\varphi_{1}\left(  j\right)  ,\delta
,-\varepsilon}^{2}\}.$ Thus $\alpha_{n\varphi_{1}(j),\delta,-\varepsilon}%
^{2}+\eta_{nj,\delta,\varepsilon}$ and $\eta_{nj,\delta,\varepsilon}%
+\alpha_{nj,\delta,-\varepsilon}^{2}$ are simple arcs. Therefore by Theorem
\ref{key3} (iii) and Theorem \ref{key1} (ii) we have:

\begin{conclusion}
\label{co43}For each $j\in M_{0},$ the arcs $\alpha_{n\varphi_{1}%
(j),\delta,-\varepsilon}^{2}+\eta_{nj,\delta,\varepsilon}+\alpha_{n\varphi
_{2}(j),\delta,\varepsilon}^{2}$ and $\eta_{n\varphi_{1}(j),\delta
,-\varepsilon}+\alpha_{n\varphi_{1}(j),\delta,-\varepsilon}^{2}+\eta
_{nj,\delta,\varepsilon}\ $are simple for sufficiently large $n.$
\end{conclusion}

Let $j\in M.$ by (\ref{cc3}) and (\ref{m1}), we have $\eta_{nj,\delta
,\varepsilon}\subset\eta_{nj,\delta}\subset D_{f_{n}}(a_{0j},\delta_{1}%
/3)\ $and$\mathrm{\ }\alpha_{n\varphi_{l}(j),\delta,-\varepsilon}^{2}\subset
D_{f_{n}}(\alpha_{0\varphi_{l}(j),\delta}^{2},2\pi\varepsilon)$ with $l=1,2,$
for large enough $n$. Then, for every pair of $i$ and $j$ with $i\neq
\varphi_{1}(j),\varphi_{2}(j)\ $and $i\in M_{0}$ we have by Lemma \ref{dfset}
and (\ref{co49})
\begin{align*}
d_{f_{n}}(\alpha_{ni,\delta,-\varepsilon}^{2},\eta_{nj,\delta,\varepsilon})
&  \geq d_{f_{n}}\left(  D_{f_{n}}(\alpha_{0i,\delta}^{2},2\pi\varepsilon
),D_{f_{n}}(a_{0j},\delta_{1}/3)\right) \\
&  \geq d_{f_{n}}(\alpha_{0i,\delta}^{2},a_{0j})-2\pi\varepsilon-\delta
_{1}/3\\
&  \geq d_{f_{n}}(\alpha_{0i},a_{0j})-2\pi\varepsilon-\delta_{1}/3\\
&  \geq\delta_{1}-2\pi\varepsilon-\delta_{1}/3>0,
\end{align*}
for sufficiently large $n.$ Then for sufficiently large $n,$ we have for
$i\neq\varphi_{1}(j),\varphi_{2}(j),$
\begin{equation}
\alpha_{ni,\delta,-\varepsilon}^{2}\cap\eta_{nj,\delta,\varepsilon}=\emptyset.
\label{cc5}%
\end{equation}

For sufficiently large $n,$ by (\ref{cc5}), (\ref{nointers}), Conclusion
\ref{co43} and Theorem \ref{CL4} (ii) we conclude that $\gamma_{n,\delta
,\varepsilon}$ is a simple curve in $\Delta,$ and Claim \ref{CL5} is proved
completely.$\medskip$
\end{proof}

\begin{center}
\textbf{Step 4 Construction of the sequence }$\Sigma_{n,\delta,\varepsilon
}^{\ast}$ \textbf{with the same Ahlfors error terms}\medskip
\end{center}

For sufficiently large $n,$ let $\Delta_{n,\delta,\varepsilon}$ be the
\emph{closed} Jordan domain in $\Delta$ enclosed by $\gamma_{n,\delta
,\varepsilon} $ and let $\mathcal{A}_{n,\delta,\varepsilon}=\overline{\Delta
}\backslash\Delta_{n,\delta,\varepsilon}^{\circ}.$ It is clear that
$\gamma_{n,\delta,\varepsilon},\{\mathfrak{t}_{nj,\delta,\varepsilon}%
^{1}\}_{j\in M_{0}},\{\mathfrak{t}_{nj,\delta,\varepsilon}^{2}\}_{j\in M_{0}}$
divide $\Delta$ into $2m_{0}+1$ Jordan domains $\Delta_{n,\delta,\varepsilon
},\Delta_{nj,\delta},$ and $\Delta_{nj,\delta,\varepsilon}^{\prime},j\in
M_{0},$ where $\Delta_{nj,\delta}$ is the part of $\Delta$ on the right hand
side of $\eta_{nj,\delta}\label{this-is-right-no-episilon}$ and $\overline
{\Delta_{nj,\delta,\varepsilon}^{\prime}}=\tilde{g}_{nj,\delta,\varepsilon
}(R_{nj,\delta,\varepsilon}^{L}),$ with $\overline{\Delta_{nj,\delta
,\varepsilon}^{\prime}}\cap\overline{\Delta_{nj,\delta}}=\mathfrak{t}%
_{nj,\delta,\varepsilon}^{1},\ $and $\overline{\Delta_{n,j,\delta,\varepsilon
}^{\prime}}\cap\overline{\Delta_{n,j+1,\delta}}%
\label{this-is-right-no-episilon copy(1)}=\mathfrak{t}_{nj,\delta,\varepsilon
}^{2} $. Then, for each $j\in M_{0},$ we have%
\begin{equation}
\partial\Delta_{nj,\delta}=-\eta_{nj,\delta}+I_{nj,\delta}=-\mathfrak{t}%
_{nj,\delta,\varepsilon}^{1}-\eta_{nj,\delta,\varepsilon}-\mathfrak{t}%
_{n\varphi_{1}(j),\delta,\varepsilon}^{2}+I_{nj,\delta}, \label{za7-1}%
\end{equation}
where $I_{nj,\delta}$ is the arc of $\partial\Delta$ from $a_{n\varphi
_{1}(j),\delta}^{2}$ to $a_{nj,\delta}^{1}.$ In fact,
\begin{equation}
I_{nj,\delta}=\alpha_{n\varphi_{1}(j),\delta}^{3}+\alpha_{n\varphi_{1}%
(j)+1}+\dots+\alpha_{n,j-1}+\alpha_{nj,\delta}^{1}, \label{za7}%
\end{equation}
and as a limit of $I_{nj,\delta}$ we have
\begin{align*}
I_{0j,\delta}  &  =\alpha_{0\varphi_{1}(j),\delta}^{3}\left(  a_{0\varphi
_{1}\left(  j\right)  ,\delta}^{2},a_{0\varphi_{1}\left(  j\right)  +1,\delta
}\right)  +\alpha_{n\varphi_{1}(j)+1}\left(  a_{0\varphi_{1}\left(  j\right)
+1,\delta},a_{0\varphi_{1}\left(  j\right)  +2,\delta}\right) \\
&  +\dots+\alpha_{n,j-1}\left(  a_{0,j-1,\delta},a_{0j,\delta}\right)
+\alpha_{0j,\delta}^{1}\left(  a_{0j,\delta},a_{0j,\delta}^{1}\right) \\
&  =\alpha_{0\varphi_{1}(j),\delta}^{3}\left(  a_{0\varphi_{1}\left(
j\right)  ,\delta}^{2},a_{0\varphi_{1}\left(  j\right)  +1,\delta}\right)
+\alpha_{0j,\delta}^{1}\left(  a_{0j,\delta},a_{0j,\delta}^{1}\right)  ,
\end{align*}
where $\alpha_{n\varphi_{1}(j)+1},\alpha_{0\varphi_{1}\left(  j\right)
+2,\delta},\dots,\alpha_{n,j-1}\left(  a_{0,j-1,\delta},a_{0j,\delta}\right)
$ are all equal to the point-arc $a_{0j,\delta}.$

Now we can prove

\begin{claim}
\label{CL6}For sufficiently large $n,$ there exists a surface $B_{n}=\left(
F_{n,\delta,\varepsilon},\mathcal{A}_{n,\delta,\varepsilon}\right)  $ such that

(i) $F_{n,\delta,\varepsilon}|_{\partial\Delta}=f_{0}$, $F_{n,\delta
,\varepsilon}=f_{n}$ in a neighborhood of $\gamma_{n,\delta,\varepsilon}$ in
$\mathcal{A}_{n,\delta,\varepsilon},$ and $\mathcal{A}_{n,\delta,\varepsilon}
$ is contained in $D_{F_{n}}(\partial\Delta,2\pi\delta).$

(ii) $F_{n,\delta,\varepsilon}^{-1}(E_{q})\cap\left[  \mathcal{A}%
_{n,\delta,\varepsilon}\backslash\partial\Delta\right]  =\emptyset$.

(iii) For each $j\in M_{0},$ the restriction $F_{nj,\delta,\varepsilon
}^{\prime}=F_{n,\delta,\varepsilon}|_{\Delta_{nj,\delta,\varepsilon}^{\prime}%
}$ is a homeomorphism from $\Delta_{nj,\delta,\varepsilon}^{\prime}$ onto
$R_{0j,\delta,\varepsilon}$ such that
\begin{align*}
\left(  F_{nj,\delta,\varepsilon}^{\prime},\alpha_{nj,\delta,-\varepsilon}%
^{2}\right)   &  =\left(  f_{n},\alpha_{nj,\delta,-\varepsilon}^{2}\right)
=c_{0j,\delta,-\varepsilon}^{2},\\
\left(  F_{nj,\delta,\varepsilon}^{\prime},\alpha_{nj,\delta}^{2}\right)   &
=\left(  f_{0},\alpha_{0j,\delta}^{2}\right)  =c_{0j,\delta}^{2},\\
\left(  F_{nj,\delta,\varepsilon}^{\prime},\mathfrak{t}_{nj,\delta
,\varepsilon}^{1}\right)   &  =\tau_{0j,\delta,\varepsilon}^{1,L}%
,\mathrm{\ \ }\left(  F_{nj,\delta,\varepsilon}^{\prime},\mathfrak{t}%
_{nj,\delta,\varepsilon}^{2}\right)  =\tau_{0j,\delta,\varepsilon}^{2,L}.
\end{align*}

(iv) For each $j\in M_{0},$ the restriction $\left(  F_{nj,\delta,\varepsilon
},\Delta_{nj,\delta}\right)  =\left(  F_{n,\delta,\varepsilon}|_{\Delta
_{nj,\delta}},\Delta_{nj,\delta}\right)  $ is a surface contained in
$\overline{D(q_{0j},\delta)}$ such that, corresponding to (\ref{za7-1}) and
(\ref{za7}),
\begin{align*}
\left(  F_{nj,\delta,\varepsilon},\mathfrak{t}_{n\varphi_{1}\left(  j\right)
,\delta,\varepsilon}^{2}\right)   &  =\tau_{0n\varphi_{1}\left(  j\right)
,\delta,\varepsilon}^{2,L},\mathrm{\ \ }\left(  F_{nj,\delta,\varepsilon
},\mathfrak{t}_{nj,\delta,\varepsilon}^{1}\right)  =\tau_{0j,\delta
,\varepsilon}^{1,L},\\
\left(  F_{nj,\delta,\varepsilon},\eta_{nj,\delta,\varepsilon}\right)   &
=\left(  f_{n},\eta_{nj,\delta,\varepsilon}\right)  =\zeta_{j,\delta
,\varepsilon},\\
\left(  F_{nj,\delta,\varepsilon},I_{nj,\delta}\right)   &  =c_{0,j-1,\delta
}^{3}+c_{0j,\delta}^{1},\\
F_{nj,\delta,\varepsilon}^{-1}(q_{0j})  &  =\{a_{0j}\},
\end{align*}
$a_{0j}$ is the only possible branch point of $F_{nj,\delta,\varepsilon}$ and
$v_{F_{nj,\delta,\varepsilon}}(a_{0j})=v_{j}+1$ (see (\ref{za10}) for $v_{j}$).
\end{claim}

\begin{proof}
In fact there exists a Jordan domain $U$ with $\overline{U}\subset\Delta,$
which contains $\overline{\Delta_{n,\delta,\varepsilon}},$ such that for each
$j\in M_{0},\overline{U}\cap\overline{\Delta_{nj,\delta}}$ and $\overline
{U}\cap\overline{\Delta_{nj,\delta,\varepsilon}^{\prime}}$ are closed Jordan
domains, that the restriction $f_{n}|_{\overline{U}\cap\overline
{\Delta_{nj,\delta,\varepsilon}^{\prime}}}$ can be extended to a homeomorphism
$F_{nj,\delta,\varepsilon}^{\prime}$ from $\Delta_{nj,\delta,\varepsilon
}^{\prime}$ onto $R_{0j,\delta,\varepsilon}^{L}$ satisfying (iii), that the
restriction $\left(  f_{n}|_{\overline{U}\cap\overline{\Delta_{nj,\delta}}%
},\overline{U}\cap\overline{\Delta_{nj,\delta}}\right)  $ can be extended to a
surface $\left(  F_{nj,\delta,\varepsilon},\overline{\Delta_{nj,\delta}%
}\right)  $ contained in $\overline{D(q_{0j},\delta)}$ satisfying (iv), that
$F_{nj,\delta,\varepsilon}$ and $F_{nj,\delta,\varepsilon}^{\prime}$ agree on
$\mathfrak{t}_{nj,\delta,\varepsilon}^{1}$ (note that we assumed $j\in M_{0}%
$), that $F_{n,j+1,\delta,\varepsilon}$ and $F_{nj,\delta,\varepsilon}%
^{\prime}$ agree on $\mathfrak{t}_{nj,\delta,\varepsilon}^{2}.$ Then (i) holds
trivially, and, by (\ref{cc15}) and (\ref{cc16}), (ii) also holds. Then it is
clear that these $2m_{0}$ mappings agree on the intersection boundary, and so
compose the desired global mapping $F_{n,\delta,\varepsilon}$ defined on
$\mathcal{A}_{n,\delta,\varepsilon}.$ The claim is proved.
\end{proof}

Let
\[
f_{n,\delta,\varepsilon}^{\ast}(z)=\left\{
\begin{array}
[c]{c}%
f_{n}(z),z\in\Delta\backslash\mathcal{A}_{n,\delta,\varepsilon},\\
F_{n,\delta,\varepsilon},z\in\mathcal{A}_{n,\delta,\varepsilon}.
\end{array}
\right.
\]
Then we obtain a sequence of surfaces
\begin{equation}
\Sigma_{n,\delta,\varepsilon}^{\ast}=\left(  f_{n,\delta,\varepsilon}^{\ast
},\overline{\Delta}\right)  \in\mathcal{F}\left(  L,m\right)  , \label{inF}%
\end{equation}
with%
\begin{equation}
\partial\Sigma_{n,\delta,\varepsilon}^{\ast}=\left(  f_{n,\delta,\varepsilon
}^{\ast},\partial\Delta\right)  =\left(  f_{0},\partial\Sigma\right)
=\Gamma_{0}. \label{eq-boundary}%
\end{equation}
It is possible that $\Sigma_{n,\delta,\varepsilon}^{\ast}\not \in
\mathcal{F}_{r}\left(  L,m\right)  ,$ which happens only if, for some
$a_{0j},$ $a_{0j}\notin f_{n,\delta,\varepsilon}^{\ast-1}\left(  E_{q}\right)
$ and the integer $v_{j}\geq1.$ But we can show at last that this can not happen.

By Claim \ref{CL6} (ii), we have
\begin{equation}
\left(  f_{n,\delta,\varepsilon}^{\ast-1}(E_{q})\cap\mathcal{A}_{n,\delta
,\varepsilon}\right)  \backslash\partial\Delta=\emptyset, \label{za2}%
\end{equation}
which implies%
\[
\overline{n}\left(  \Sigma_{n,\delta,\varepsilon}^{\ast}\right)
=\#f_{n,\delta,\varepsilon}^{\ast-1}(E_{q})\cap\Delta_{n,\delta,\varepsilon
}=\#f_{n}^{-1}(E_{q})\cap\Delta_{n,\delta,\varepsilon}\leq\overline{n}%
(\Sigma_{n})\leq qd^{\ast},
\]
and then, taking subsequence if necessary, we have

\begin{claim}
\label{CL7}$\overline{n}\left(  \Sigma_{n,\delta,\varepsilon}^{\ast}\right)  $
is a constant for all $n=1,2,...$
\end{claim}

For each $j\leq M_{0},$ by definition of $F_{n,\delta,\varepsilon}$ and by
(\ref{za10}) we have
\[
A(F_{n,\delta,\varepsilon},\Delta_{nj,\delta})\leq\left(  v_{j}+1\right)
A(D(q_{0j},\delta))=2\pi\left(  v_{j}+1\right)  \left(  1-\cos\delta\right)
\leq\pi\left(  d^{\ast}+1\right)  \delta^{2},
\]
and
\[
A(F_{n,\delta,\varepsilon},\Delta_{nj,\delta,\varepsilon}^{\prime
})=A(R_{0j,\delta,\varepsilon}^{L})<2\pi\varepsilon L(c_{0j,\delta}%
^{2})<L(c_{0j,\delta}^{2})\delta<L\delta.
\]
Thus, by (\ref{co59}) and (\ref{za2}), we have $\delta<1$ and
\begin{equation}
0<R(f_{n,\delta,\varepsilon}^{\ast},\mathcal{A}_{n,\delta,\varepsilon
})=R(F_{n,\delta,\varepsilon},\mathcal{A}_{n,\delta,\varepsilon})=\left(
q-2\right)  A(F_{n,\delta,\varepsilon},\mathcal{A}_{n,\delta,\varepsilon
})<C\delta, \label{za5}%
\end{equation}
where $C=\left[  \pi\left(  d^{\ast}+1\right)  +L\right]  M_{0},$ which is
independent of $\delta,\varepsilon$ and $n.$ Then, by Conclusion \ref{co1}, we
may choose $\delta$ small enough such that%
\begin{equation}
A(\Sigma_{n,\delta,\varepsilon}^{\ast})=A(f_{n},\Delta_{n,\delta,\varepsilon
})+A(F_{n,\delta,\varepsilon},\mathcal{A}_{n,\delta,\varepsilon})\leq4\pi
d^{\ast}+1. \label{out*<c}%
\end{equation}
Hence by Corollary \ref{A1} and (\ref{eq-boundary}), we may assume that
$\Sigma_{n,\delta,\varepsilon}^{\ast}$ have the same area and by Claim
\ref{CL7}, we have the following.

\begin{claim}
\label{CL8}$A(\Sigma_{n,\delta,\varepsilon}^{\ast})$, $\overline{n}%
(\Sigma_{n,\delta,\varepsilon}^{\ast})$ and $R(\Sigma_{n,\delta,\varepsilon
}^{\ast})$ are constants for all $n=1,2,\dots,$ respectively.
\end{claim}

By (\ref{eq-boundary}) and Claim \ref{CL8} implies

\begin{claim}
\label{CL12}$H(\Sigma_{n,\delta,\varepsilon}^{\ast})=R(\Sigma_{n,\delta
,\varepsilon}^{\ast})/L(\partial\Sigma_{n,\delta,\varepsilon}^{\ast})$ is a
constant $H$ for all $n=1,2,\dots\medskip$
\end{claim}

\begin{center}
\textbf{Step 5 Complete the Proof of Theorem \ref{LK}}\medskip
\end{center}

We will show
\begin{equation}
H=H(\Sigma_{n,\delta,\varepsilon}^{\ast})=\lim_{n\rightarrow\infty}%
H(\Sigma_{n})=H_{L,m},\mathrm{\ for\mathrm{\ }}n=1,2,\dots\label{za9}%
\end{equation}

Since $\Sigma_{n}$ is an extremal sequence in $\mathcal{F}\left(  L,m\right)
\ $and $L(\partial\Sigma_{n})\rightarrow L(\partial\Sigma_{n,\delta
,\varepsilon}^{\ast})=L(f_{0},\partial\Delta),$ by (\ref{inF}), we have
\[
H=H(\Sigma_{n,\delta,\varepsilon}^{\ast})\leq\lim_{n\rightarrow\infty}%
H(\Sigma_{n}).
\]
So, by (\ref{eq-boundary}), to prove (\ref{za9}), it suffices to prove that
for any number $\mu>0$%
\begin{equation}
R(\Sigma_{n})<R(\Sigma_{n,\delta,\varepsilon}^{\ast})+\mu. \label{s<s*}%
\end{equation}
Since $\Sigma_{n}$ and $\Sigma_{n,\delta,\varepsilon}^{\ast}$ coincide on
$\Delta_{n,\delta,\varepsilon}\ $and by (\ref{za5}) $R(f_{n,\delta
,\varepsilon}^{\ast},\mathcal{A}_{n,\delta,\varepsilon})>0$, to prove
(\ref{s<s*}) it suffices to prove
\begin{equation}
R\left(  f_{n},\mathcal{A}_{n,\delta,\varepsilon}\right)  <\mu
,\mathrm{\ for\mathrm{\ large\ enough}\ }n. \label{out<out}%
\end{equation}

We will prove the following Claim, which implies (\ref{out<out}) by taking
$\delta$ small enough.

\begin{claim}
\label{CL9}There exists a constant $C_{1}$ independent of $n,$ $\delta,$ and
$\varepsilon,$ such that
\[
R(f_{n},\mathcal{A}_{n,\delta,\varepsilon})\leq C_{1}\delta
,\mathrm{\ for\mathrm{\ large\ enough}\ }n.
\]

\end{claim}

\begin{proof}
We first show that for each $j\in M_{0}$ and sufficiently large $n$
\begin{equation}
R(f_{n},\Delta_{nj,\delta})<C_{1}^{\prime}\delta, \label{za6}%
\end{equation}%
\begin{equation}
R(f_{n},\Delta_{nj,\delta,\varepsilon}^{\prime})<C_{2}^{\prime}\delta,
\label{R<<}%
\end{equation}
for some $C_{1}^{\prime}$ and $C_{2}^{\prime}$ independent of $n,$ $\delta$
and $\varepsilon$.

By the assumption that $\eta_{nj,\delta}$ is parameterized by length and
$t_{nj}<2\pi\left(  d^{\ast}+1\right)  \delta,$ we have $L(f_{n}%
,\eta_{nj,\delta})<2\pi\left(  d^{\ast}+1\right)  \delta.$ By (\ref{506-4})
and (\ref{za7}) we have%
\[
L(f_{n},I_{nj,\delta})\rightarrow L(c_{0,j-1,\delta}^{3})+0+L(c_{0j,\delta
}^{1})<4\pi\delta.
\]
Thus we have
\begin{equation}
L(f_{n},\partial\Delta_{nj,\delta})=L(f_{n},I_{nj,\delta}-\eta_{nj,\delta
})<4\pi\delta+2\pi\left(  d^{\ast}+1\right)  \delta=C_{1}^{\prime\prime}%
\delta\label{L<C1d}%
\end{equation}
where $C_{1}^{\prime\prime}=\left(  4\pi+2\pi\left(  d^{\ast}+1\right)
\right)  $.

By (\ref{co51}) and (\ref{co49}) we have $C_{1}^{\prime\prime}\delta
<\delta_{E_{q}},$ which with (\ref{L<C1d}), Theorem \ref{l<2dt} and Lemma
\ref{hd}, implies%
\[
R(f_{n},\Delta_{nj,\delta})=H(f_{n},\Delta_{nj,\delta})L(f_{n},\partial
\Delta_{nj,\delta})\leq\left(  q-2\right)  L(f_{n},\partial\Delta_{nj,\delta
})\leq\left(  q-2\right)  C_{1}^{\prime\prime}\delta,
\]
and then, putting $C_{1}^{\prime}=\left(  q-2\right)  C_{1}^{\prime\prime},$
we have (\ref{za6}).

By (\ref{za4}) and the fact that $R_{nj,\delta,\varepsilon}\rightarrow
R_{0j,\delta,\varepsilon}$ we have
\[
R(f_{n},\Delta_{nj,\delta,\varepsilon}^{\prime})=\left(  q-2\right)
A(R_{nj,\delta,\varepsilon})\leq\left(  q-2\right)  \pi L(c_{nj,\delta}%
^{2})\varepsilon<\left(  q-2\right)  \left(  \pi L\right)  \delta
\]
when $n$ is large enough by (\ref{co59}). This implies (\ref{R<<}).

(\ref{za6}) and (\ref{R<<}) imply Claim \ref{CL9}, for $\mathcal{A}%
_{n,\delta,\varepsilon}=\cup_{j\in M_{0}}\left[  \Delta_{nj,\delta
,\varepsilon}^{\prime}\cup\Delta_{nj,\delta,\varepsilon}\right]  $ and
$\mathcal{A}_{n,\delta,\varepsilon}\backslash\partial\Delta$ contains no point
of $f_{n,\delta,\varepsilon}^{\ast-1}(E_{q})$.
\end{proof}

Now, (\ref{za9}) is proved completely, and we have reach the position to
complete the proof of Theorem \ref{LK}.

Let $f_{L_{1}}\ $be any one of the sequence $f_{n,\delta,\varepsilon}^{\ast},$
so that $n$ is large enough, $\delta$ is small enough, and $\varepsilon
\ $satisfying (\ref{co59}) further small enough, and $\Sigma_{L_{1}}=\left(
f_{L_{1}},\overline{\Delta}\right)  .$ Then (\ref{za9}) implies $H(\Sigma
_{L_{1}})=H_{L,m},$ and then by (\ref{inF}), $\Sigma_{L_{1}}$ is an extremal
surface of $\mathcal{F}(L,m).$ Since $\Sigma_{n}$ is a precise extremal
sequence, $L_{1}=\lim_{n\rightarrow\infty}\inf L(\partial\Sigma_{n}%
)=L(\Gamma_{0})$ and $L(\partial\Sigma_{L_{1}})=L(\Gamma_{0}),$ $\Sigma
_{L_{1}}$ is a precise extremal surface in $\mathcal{F}\left(  L,m\right)  .$

If $f_{L_{1}}$ has no branch point outside $f^{-1}(E_{q}),$ then
$\Sigma_{L_{1}}\in\mathcal{F}_{r}(L,m)$ and thus $\Sigma_{L_{1}}$ is a precise
extremal surface in $\mathcal{F}_{r}(L,m).$

As it is pointed out, $f_{L_{1}}$ may be not in $\mathcal{F}_{r}\left(
L,m\right)  .$ But when this is happen, by Theorem \ref{re} there exists a
surface $\Sigma^{\prime}=\left(  f,\overline{\Delta}\right)  $ such that
$H(\Sigma^{\prime})\geq H(\Sigma_{L_{1}})$, $L(\partial\Sigma^{\prime})\leq
L(\partial\Sigma_{L_{1}})\ $and $\Sigma^{\prime}\in\mathcal{F}_{r}(L,m).$ Then
$\Sigma^{\prime}$ is again an extremal surface of $\mathcal{F}_{r}(L,m)$ and
by the definition of $L_{1}$ we have $L(\partial\Sigma^{\prime})\geq
L(\partial\Sigma_{L_{1}})$ and thus $L(\partial\Sigma^{\prime})=L(\partial
\Sigma_{L_{1}})$ and $\Sigma^{\prime}$ is also a precise extremal surface of
$\mathcal{F}_{r}(L,m).$ This completes the proof of Theorem \ref{LK}.\medskip

By Lemma \ref{Fr'FrF}, Theorem \ref{LK} and (\ref{>2d}), we have the following:

\begin{corollary}
\label{FFprime}Let $L\in\mathcal{L}$ be a positive number and $m$ be a
sufficiently large positive integer. Then there exists a precise extremal
surface $\Sigma_{L_{1}}\ $of $\mathcal{F}_{r}^{\prime}(L,m),\mathcal{F}%
_{r}(L,m)$ and $\mathcal{F}(L,m),$ such that $L(\partial\Sigma_{L_{1}}%
)=L_{1}\leq L.$ For any precise extremal surface $\Sigma$ of $\mathcal{F}%
(L,m),$ $L(\partial\Sigma)\geq2\delta_{E_{q}}$ if $L\geq2\delta_{E_{q}}.$
\end{corollary}

\section{Relation of precise extremal surfaces of $\mathcal{F}(L,m)$ and
$\mathcal{F}(L,m-1)$}

The goal of this long section is to prove the following theorem, followed by
some applications.

\begin{theorem}
\label{cat2}Let $L\in\mathcal{L}$ be a positive number with $L\geq
2\delta_{E_{q}}.$ Then for sufficiently large integer $m$ and any precise
extremal surface $\Sigma_{L_{1}}=\left(  f,\overline{\Delta}\right)  $ of
$\mathcal{F}_{r}(L,m)$ \label{Fr,r needed since we use Lemma tangent}with
$L(\partial\Sigma_{L_{1}})=L_{1}\leq L,$ $\Sigma_{L_{1}}$ is a precise
extremal surface of $\mathcal{F}_{r}\left(  L,m-1\right)  .$
\end{theorem}

To prove this theorem we fix a number $L\ $of $\mathcal{L}$ with
\begin{equation}
L\geq2\delta_{E_{q}}, \label{512}%
\end{equation}
and let $\Sigma_{L_{1}}=\left(  f,\overline{\Delta}\right)  $ be a precise
extremal surface of $\mathcal{F}_{r}(L,m)$ with $L(\partial\Sigma_{L_{1}%
})=L_{1}\leq L,$ and assume $m$ is large enough,
\begin{equation}
m>10L/\delta_{E_{q}}, \label{jj}%
\end{equation}
and $\partial\Sigma_{L_{1}}$ is parametrized by length. Then $\partial
\Sigma_{L_{1}}$ has an $\mathcal{F}(L,m)$-partition%
\begin{equation}
\partial\Sigma_{L_{1}}=c_{1}\left(  q_{1},q_{2}\right)  +c_{2}\left(
q_{2},q_{3}\right)  +\cdots+c_{m}\left(  q_{m},q_{1}\right)  , \label{zz1}%
\end{equation}
which in fact means that $\partial\Delta$ has an $\mathcal{F}(L,m)$-partition
for $\partial\Sigma_{L_{1}}:$%
\begin{equation}
\partial\Delta=\alpha_{1}\left(  a_{1},a_{2}\right)  +\alpha_{2}\left(
a_{2},a_{3}\right)  +\cdots+\alpha_{m}\left(  a_{m},a_{1}\right)  ,
\label{zza1}%
\end{equation}
such that $f$ restricted to a neighborhood of $\alpha_{j}^{\circ}$ in
$\overline{\Delta}$ is a homeomorphism onto a left hand side neighborhood of
$c_{j}^{\circ}$ and $c_{j}=\left(  f,\alpha_{j}\right)  $ is a convex circular
arc for $j=1,2,\dots,m,$ and by Remark \ref{ap11} we may assume%
\[
a_{1}=1.
\]
As in Remark \ref{parameter}, we assume $a_{j}=e^{\sqrt{-1}\psi_{j}},$
$0=\psi_{1}<\psi_{2}<\dots<\psi_{m}<2\pi$ and introduce the continuous
subscript $x$ in $a_{x}\in\partial\Delta$ and $q_{x}\in\partial\Sigma_{L_{1}}$
with $q_{x}=f(a_{x}),$ but when the letters $i,j,k$ appear as subscripts, they
are always integers.

Under the above assumptions, we will first prove Lemmas \ref{nonf},
\ref{home2}, \ref{samecur}, \ref{circular}, \ref{mykey}, \ref{Eq-end}, and
\ref{mykey2}. Then we will prove Theorem \ref{cat2} easily from Lemmas
\ref{circular}, \ref{mykey} and \ref{mykey2}.

\begin{lemma}
\label{nonf}$\partial\Sigma_{L_{1}}$ cannot folded at any $a_{i}\in
\{a_{j}\}_{j=1}^{m}\backslash f^{-1}(E_{q}),$ and thus, for each $i$ with
$a_{i}\in\{a_{j}\}_{j=1}^{m}\backslash f^{-1}(E_{q}),$ $c_{i-1}\ $and $c_{i}$
intersect only at $q_{i}$ in a neighborhood of $q_{i}\ $on $S.$
\end{lemma}

\begin{proof}
Assume $a_{i}\in\{a_{j}\}_{j=1}^{m}\backslash f^{-1}(E_{q})$ and
$\partial\Sigma_{L_{1}}$ is folded at $a_{i}.$ Then $c_{i-1}+c_{i}$ contains
an arc of the form $c_{i}^{\prime}+c_{i+1}^{\prime}=c_{i}^{\prime}%
-c_{i}^{\prime}$ such that $c_{i}^{\prime}\cap E_{q}=\emptyset.$ Therefore we
can sew\label{sew12} $\Sigma_{L_{1}}$ along $c_{i}^{\prime}$ to obtain a new
surface $\Sigma^{\prime}\in\mathcal{F}(L,m)$ so that $R(\Sigma^{\prime
})=R(\Sigma_{L_{1}})$ and $L(\partial\Sigma^{\prime})<L(\partial\Sigma_{L_{1}%
}),$ say $H(\Sigma^{\prime})>H(\Sigma_{L_{1}}),$ which contradicts the
maximality of $\Sigma_{L_{1}}.$
\end{proof}

\begin{lemma}
\label{home2}(i) The precise extremal surface $\Sigma_{L_{1}}$ of
$\mathcal{F}_{r}\left(  L,m\right)  $ is also a precise extremal surface of
$\mathcal{F}\left(  L,m\right)  $ and $L_{1}=L\left(  \partial\Sigma_{L_{1}%
}\right)  \geq2\delta_{E_{q}}.$

(ii) $f$ is locally homeomorphic in $\left[  \Delta\backslash f^{-1}%
(E_{q})\right]  \cup\left[  \left(  \partial\Delta\right)  \backslash\left[
f^{-1}(E_{q})\cap\{a_{j}\}_{j=1}^{m}\right]  \right]  .$
\end{lemma}

\begin{proof}
\label{ok20220107}(i) follows from Lemma \ref{Fr'FrF}.

By definition of $\mathcal{F}_{r}(L,m),$ each $x\in\left(  \partial
\Delta\right)  \backslash\{a_{j}\}_{j=1}^{m}$, is a simple point of $f$ (see
Definition \ref{simple}), say $f$ is a homeomorphism in a neighborhood of $x$
in $\overline{\Delta}.$ Assume for some $j_{0}\leq m$, $f(a_{j_{0}})\notin
E_{q}.$ Then the interior angle $\theta_{j_{0}}$ of $\Sigma_{L_{1}}$ at
$a_{j_{0}}\ $is strictly less than or equal to $2\pi.$ If $\theta_{j_{0}}%
<2\pi,$ then $a_{j_{0}}$ is a simple point of $f$ (Definition \ref{simple}
(a)). If $\theta_{j_{0}}=2\pi$, then $a_{j_{0}}$ is a simple point of $f$ if
$\partial\Sigma_{L_{1}}$ is simple in a neighborhood of $a_{j_{0}}$ in
$\partial\Delta.$ Thus $a_{j_{0}}$ is not a simple point of $f$ iff
$\partial\Sigma_{L_{1}}$ is folded at $a_{j_{0}},$ contradicting Lemma
\ref{nonf}. Thus every point $a_{j}\in\{a_{j}\}_{j=1}^{m}\ $outside
$f^{-1}(E_{q})\ $is a simple point of $f.$ Since all branch points of $f$ are
contained in $f^{-1}(E_{q}),$ the conclusion (ii)
holds.\label{ok20220107 copy(1)}
\end{proof}

\begin{definition}
\label{C^1}$\mathfrak{C}^{1}=\mathfrak{C}^{1}\left(  \Sigma_{L_{1}}\right)  $
is the collection of all subarcs of $\partial\Sigma_{L_{1}}$ such that for
each $c=\left(  f,\alpha\right)  \in\mathfrak{C}^{1},$ the following (a)--(c) hold:

(a) $c$ is an SCC arc and every point of $c^{\circ}$ is a simple point of
$\Sigma_{L_{1}},$ say, $f$ restricted to a neighborhood of $\alpha^{\circ}$ is
a homeomorphism.

(b) $c^{\circ}\cap E_{q}=\emptyset,$ say, $c\cap E_{q}\subset\partial c$
($\partial c$ is the set of endpoints of $c$).

(c) $L(c)<\pi.$

$\mathfrak{C}^{2}=\mathfrak{C}^{2}\left(  \Sigma_{L_{1}}\right)  $ is the
subset of $\mathfrak{C}^{1}$ such that each $c\in\mathfrak{C}^{2}$ has two
distinct endpoints.
\end{definition}

\begin{lemma}
\label{samecur}\label{1-cir}(i) All arcs in $\left\{  c_{j}\right\}
_{j=1}^{m}\cap\mathfrak{C}^{1}$ have the same curvature.

(ii) $\left\{  c_{j}\right\}  _{j=1}^{m}\cap\mathfrak{C}^{1}$ contains at most
one major circular arc (a simple circle is regarded as a major circular arc).

(iii) $\left\{  c_{j}\right\}  _{j=1}^{m}\cap\mathfrak{C}^{1}\ $contains at
most one closed arc, say, $\left\{  c_{j}\right\}  _{j=1}^{m}\cap\left[
\mathfrak{C}^{1}\backslash\mathfrak{C}^{2}\right]  $ is either empty or
contains only one element.
\end{lemma}

\begin{proof}
If $\#\left\{  c_{j}\right\}  _{j=1}^{m}\cap\mathfrak{C}^{1}\leq1,$ then there
is nothing to prove. So we assume $\#\left\{  c_{j}\right\}  _{j=1}^{m}%
\cap\mathfrak{C}^{1}\geq2.$ Assume that (i) or (ii) of the lemma fails. Then
there exist distinct arcs $c_{j_{1}}$ and $c_{j_{2}}$ in $\left\{
c_{j}\right\}  _{j=1}^{m}\cap\mathfrak{C}^{1}$ such (a) or (b) in Deformation
\ref{Deform1} holds. Then there exists a new surface $\Sigma^{\prime}%
\in\mathcal{F}\left(  L,m\right)  $ such that $H(\Sigma^{\prime}%
)>H(\Sigma_{L_{1}}).$ Thus $\Sigma_{L_{1}}$ is not extremal in $\mathcal{F}%
(L,m), $ contradicting Lemma \ref{home2} (i), and so (i) and (ii) hold, and
(iii) follows from (ii).
\end{proof}

\begin{lemma}
\label{circular}Assume that for some $j\leq m,$
\begin{equation}
c_{j}\in\mathfrak{C}^{2}. \label{829-1}%
\end{equation}
Then $c_{j}+c_{j+1}$ is circular at $q_{j+1}$ if $q_{j+1}\notin E_{q},$ and
$c_{j-1}+c_{j}$ is circular at $q_{j}$ if $q_{j}\notin E_{q}$. The term
"circular at $q_{i}$" means that $f$ restricted to a neighborhood of $a_{i}$
in $\alpha_{i-1}+\alpha_{i}$ is a homeomorphism onto a simple circular arc,
for $i=j$ or $j+1.$
\end{lemma}

\begin{remark}
\label{cir-2}By Lemmas \ref{home2} and \ref{circular}, for $i=j,j+1$, when
$q_{i}\notin E_{q},$ we have that $f$ is not only circular at $q_{i}$ but also
homeomorphic in a neighborhood of $a_{i}$ in $\overline{\Delta}.$
\end{remark}

\begin{proof}
Assume $c_{j}\in\mathfrak{C}^{2}$ and
\[
q_{j+1}\notin E_{q}.
\]
Then $c_{j}$ is not closed and by Lemma \ref{home2} $f$ is homeomorphic in a
neighborhood of $\alpha_{j}\backslash\{a_{j}\}=\alpha_{j}^{\circ}\cup\left\{
a_{j+1}\right\}  $ in $\overline{\Delta}.$ Then by Lemma \ref{int-arg1} we
have the following claim.

\begin{claim}
\label{cl1}For sufficiently small $\varepsilon_{0}>0,$ and for the arc
\[
C_{j}=c_{j}\left(  q_{j},q_{j+1}\right)  +c_{j+1}\left(  q_{j+1}%
,q_{j+1+\varepsilon_{0}}\right)  ,
\]
there exist a number $\theta\in(0,\pi/2)$ and a closed simple Jordan domain
$\left(  T_{j,\varepsilon_{0},\theta},C_{j}\right)  =\left(  f,\overline
{D_{j,\varepsilon_{0},\theta}}\right)  $ of $\Sigma_{L_{1}}$ with the old
boundary $C_{j},$ such that the following hold.

(1) $D_{j,\varepsilon_{0},\theta}\ $is a Jordan domain in $\Delta$ with
$\partial D_{j,\varepsilon_{0},\theta}=A_{j}+\beta_{j},$ in which
\[
A_{j}=\alpha_{j}\left(  a_{j},a_{j+1}\right)  +\alpha_{j+1}\left(
a_{j+1},a_{j+1+\varepsilon_{0}}\right)  ,
\]
$\beta_{j}$ is a simple arcs in $\overline{\Delta}$ with $\beta_{j}^{\circ
}\subset\Delta.$

(2) $T_{j,\varepsilon_{0},\theta}$ is a closed Jordan domain in some open
hemisphere $S_{1}$ on $S$ with
\[
\left(  T_{j,\varepsilon_{0},\theta}\backslash\{q_{j}\}\right)  \cap
E_{q}=\emptyset,
\]%
\[
\partial T_{j,\varepsilon_{0},\theta}=C_{j}+\tau_{j},
\]
in which $\tau_{j}$ is a polygonal path from $q_{j+1+\varepsilon_{0}}$ to
$q_{j}.$

(3) The interior angle of $T_{j,\varepsilon_{0},\theta}$ at $q_{j}$ and
$q_{j+1}$ are both $\theta.$
\end{claim}

To prove the lemma, we will deduce a contradiction under the assumption:

\begin{condition}
\label{cond}$c_{j}+c_{j+1}$ is not circular at $q_{j+1}.$
\end{condition}

Let $\varepsilon\ll\varepsilon_{0}\ $be a positive number which is so small
that for the arc
\[
\gamma_{j,\varepsilon}=\gamma_{j,\varepsilon}\left(  q_{j},q_{j+1+\varepsilon
}\right)  =c_{j}\left(  q_{j},q_{j+1}\right)  +c_{j+1}\left(  q_{j+1}%
,q_{j+1+\varepsilon}\right)  ,
\]%
\[
L(\gamma_{j,\varepsilon})<\pi.
\]
We first show the following.

\begin{Assertion}
\label{as1}For sufficiently small $\varepsilon<\varepsilon_{0},$ there exists
a surface $F_{\varepsilon}=\left(  f_{\varepsilon},\overline{D_{j,\varepsilon
_{0},\theta}}\right)  $ in $S_{1}$ such that $f_{\varepsilon} $ agree with $f$
in a neighborhood of $(\beta_{j}\left(  a_{j+1+\varepsilon_{0}},a_{j}\right)
)\backslash\{a_{j}\}\ $in $\overline{D_{j,\varepsilon_{0},\theta}},$
\begin{equation}
\partial F_{\varepsilon}=\gamma_{j,\varepsilon}^{\prime}+c_{j+1}\left(
q_{j+1+\varepsilon},q_{j+1+\varepsilon_{0}}\right)  +\tau_{j}, \label{1}%
\end{equation}
where $\gamma_{j,\varepsilon}^{\prime}=\gamma_{j,\varepsilon}^{\prime}\left(
q_{j},q_{j+1+\varepsilon}\right)  $ is the SCC arc from $q_{j}$ to
$q_{j+1+\varepsilon}$ with $L\left(  \gamma_{j,\varepsilon}^{\prime}\right)
=L(\gamma_{j,\varepsilon}),$ and moreover,
\begin{equation}
A(F_{\varepsilon})>A(T_{j,\varepsilon_{0},\theta}), \label{2}%
\end{equation}%
\begin{equation}
F_{\varepsilon}\backslash\{q_{j}\}\subset S_{1}\backslash E_{q}, \label{3}%
\end{equation}
and $q_{j+1+\varepsilon}$ is the only possible branch value of $f_{\varepsilon
}.$
\end{Assertion}

In fact, $F_{\varepsilon}$ is a deformation of $\overline{T_{j,\varepsilon
_{0},\theta}}=\left(  f,\overline{D_{j,\varepsilon_{0},\theta}}\right)  $ so
that the new boundary $\tau_{j}^{\circ}$ and the part $c_{j+1}\left(
q_{j+1+\varepsilon},q_{j+1+\varepsilon_{0}}\right)  $ of the old boundary of
$T_{j,\varepsilon_{0},\theta}$ remain unchanged, while the part $\gamma
_{j,\varepsilon}$ of the old boundary is changed into $\gamma_{j,\varepsilon
}^{\prime}.$

It is clear that as $\varepsilon\rightarrow0,$ $\gamma_{j,\varepsilon}%
^{\prime}$ converges to $c_{j}$ and thus the left hand side angle of $\tau
_{j}+\gamma_{j,\varepsilon}^{\prime}$ at $q_{j}$ tends to $\theta$ (this may
fail when $c_{j}\in\mathfrak{C}^{1}\backslash\mathfrak{C}^{2},$ the assumption
$c_{j}\in\mathfrak{C}^{2}$ is used here). Thus, when $\varepsilon$ is small
enough, the circular arc $\gamma_{j,\varepsilon}^{\prime}$ does not intersects
$\tau_{j}\backslash\{q_{j}\}$ as sets in $S_{1},$ and it is clear that for
sufficiently small $\varepsilon>0,$
\[
\gamma_{j,\varepsilon}^{\prime}+c_{j+1}\left(  q_{j+1+\varepsilon
},q_{j+1+\varepsilon_{0}}\right)  +\tau_{j}%
\]
encloses a surface $F_{\varepsilon}=\left(  f_{\varepsilon},\overline
{D_{j,\varepsilon_{0},\theta}}\right)  $ in $S_{1},$ which is just a simple
closed domain in $S_{1}$ when $\gamma_{j,\varepsilon}^{\prime}+c_{j+1}\left(
q_{j+1+\varepsilon},q_{j+1+\varepsilon_{0}}\right)  $ is simple. It is clear
that by Lemma \ref{for-circular} and Condition \ref{cond} we have (\ref{2}).
When $\varepsilon\rightarrow0,$ as sets in $S_{1},F_{\varepsilon}$ converges
to $T_{j,\varepsilon_{0},\theta}$ and so (\ref{3}) holds.

When $\gamma_{j,\varepsilon}^{\prime}+c_{j+1}\left(  q_{j+1+\varepsilon
},q_{j+1+\varepsilon_{0}}\right)  $ is not simple, $\gamma_{j,\varepsilon
}^{\prime}+c_{j+1}\left(  q_{j+1+\varepsilon},q_{j+1+\varepsilon_{0}}\right)
$ contains a small closed arc, which is consisted of two short subarcs of
$\gamma_{j,\varepsilon}^{\prime}$ and $c_{j+1}\left(  q_{j+1+\varepsilon
},q_{j+1+\varepsilon_{0}}\right)  $ near $q_{1+j+\varepsilon},$ and which
tends to $q_{j+1}$ as $\varepsilon\rightarrow0,$ and we may make
$q_{j+1+\varepsilon}$ to be the only branch value of $F_{\varepsilon}.$ It is
clear that $\left(  f,\beta_{j}\right)  $ is equivalent to $\left(
f_{\varepsilon},-\beta_{j}\right)  $ and since the interior angle of
$F_{\varepsilon}=\left(  f_{\varepsilon},\overline{D_{j,\varepsilon_{0}%
,\theta}}\right)  \ $at $a_{j}$ tends to that of $\left(  f,\overline
{D_{j,\varepsilon_{0},\theta}}\right)  =T_{j,\varepsilon_{0},\theta},$ we can
make $F_{\varepsilon}$ and $f$ agree in a neighborhood of $\beta
\backslash\{a_{j}\}$ in $\overline{D_{j,\varepsilon_{0},\theta}}.$ The
assertion is proved.

Now we can deform $\Sigma_{L_{1}}$ by replacing the part $T_{j,\varepsilon
_{0},\theta}$ of $\Sigma_{L_{1}}$ with $F_{\varepsilon},$ that is, we cut
$\left(  T_{j,\varepsilon_{0},\theta},C_{j}\right)  $ from $\Sigma_{L_{1}}$
along the new boundary $\tau_{j}^{\circ}$ to obtain a surface $\Sigma_{1},$
and then we sew\label{sew25} $F_{\varepsilon}$ and $\Sigma_{1}$ also along the
new boundary $\tau_{j}^{\circ}.$ Then by (\ref{zz1}) and (\ref{1}) we see that
$\Sigma_{L_{1}}$ becomes a new surface $\Sigma^{\prime}=\left(  f^{\prime
},\overline{\Delta}\right)  $ such that
\begin{equation}
\partial\Sigma^{\prime}=c_{1}+\cdots+c_{j-1}+\gamma_{j}^{\prime}%
+c_{j+1}(q_{j+1+\varepsilon},q_{j+2})+c_{j+2}+\cdots+c_{m}. \label{(4)}%
\end{equation}

We have to show $\Sigma^{\prime}\in\mathcal{F}\left(  L,m\right)  .$ Since
$f_{\varepsilon}$ agree with $f$ in a neighborhood of $\beta_{j}%
\backslash\{a_{j}\}$ in $\overline{D_{j,\varepsilon_{0},\theta}}\ $and the
partition (\ref{1}) is an $\mathcal{F}\left(  L,3\right)  $ partition of
$\partial F_{\varepsilon}$, $f^{\prime}\ $can be defined by $f$ on
$\overline{\Delta}\backslash\overline{D_{j,\varepsilon_{0},\theta}}$ and
$f_{\varepsilon}$ on $\overline{D_{j,\varepsilon_{0},\theta}},$ and so
(\ref{(4)}) is an $\mathcal{F}\left(  L,m\right)  $-partition of
$\partial\Sigma^{\prime},$ by (\ref{1}) and (\ref{(4)}), and so $\Sigma
^{\prime}\in\mathcal{F}\left(  L,m\right)  .$ By Assertion \ref{as1} we also
have
\[
L(\partial\Sigma_{L_{1}})=L(\partial\Sigma^{\prime}),\bar{n}\left(
\Sigma_{L_{1}}\right)  =\bar{n}\left(  \Sigma^{\prime}\right)  ,A(\Sigma
^{\prime})>A(\Sigma_{L_{1}}),
\]
which implies $H(\Sigma^{\prime})>H(\Sigma_{L_{1}}).$ Then $\Sigma_{L_{1}}$ is
not an extremal surface of $\mathcal{F}\left(  L,m\right)  ,$ which
contradicts Lemma \ref{home2} (i).\label{ok2002-01-08-2}
\end{proof}

\begin{lemma}
\label{mykey}Assume that%
\begin{equation}
L(c_{1})<\delta_{E_{q}}/2. \label{2022-01-03}%
\end{equation}
Then $c_{1}\ $is not closed, say, $q_{1}\neq q_{2}.$
\end{lemma}

\begin{proof}
Though the proof is complicated, the idea is quite simpler, which is implied
in the proof of Case 1 and the discussion for other cases are essentially the
same, with a little difference.

To prove the result, we assume that the opposite holds, say, $c_{1}\ $is a
whole circle. Then (\ref{2022-01-03}) implies the following.

\begin{condition}
\label{(A)}$c_{1}$ is a strictly convex circle from $q_{1}$ to $q_{2}=q_{1},$
the length and diameter of $c_{1}$ are both less than $\delta_{E_{q}}/2,$ and
$c_{1}\cap E_{q}$ is either empty or a singleton.
\end{condition}

We denote by $T_{0}$ the domain enclosed by $c_{1}.$ First of all we have by
Lemma \ref{b-in} the following.

\begin{claim}
\label{c1}$\Sigma_{L_{1}}$ contains no closed simple domain of the form
$\left(  \overline{T_{0}},c_{1}\right)  =\left(  \overline{T_{0}},\alpha
_{1}\right)  $ (as in Remark \ref{Riemann} (iii), $c_{1}$ should be understood
as $\left(  f,\alpha_{1}\right)  $). This in fact means that there is no
subdomain $D$ of $\Delta$ such that $\alpha_{1}\subset\partial D$ and $f$ is a
homeomorphism from $\overline{D}\backslash\{a_{1},a_{2}\}$ onto $\overline
{T_{0}}\backslash\{q_{1}\}.$
\end{claim}

We let $h_{\theta,q_{1}}\left(  w\right)  $ be the rotation%
\begin{equation}
h_{\theta,q_{1}}\left(  w\right)  =\varphi_{q_{1}}^{-1}\circ\varphi_{\theta
}\circ\varphi_{q_{1}}(w),w\in S, \label{ag43}%
\end{equation}
of $S,$ where $\varphi_{q_{1}}$ is a rotation of $S$ putting $q_{1}$ into $0,
$ and $\varphi_{\theta}$ is the rotation $w\mapsto e^{-i\theta}w$ of
$S,\theta\in\lbrack0,\pi].$ Recall that, by Remark \ref{parameter},
$q_{1+\frac{1}{2}}$ is the middle point of $c_{1}.$ Write
\[
c_{1,\theta}=h_{\theta,q_{1}}\left(  c_{1}\right)  ,
\]
let $q_{2,\theta}^{\prime}\in S$ be the intersection of $c_{1,\theta}^{\circ
}=c_{1,\theta}\backslash\{q_{1}\}$ and $c_{1}^{\circ}=c_{1}\backslash
\{q_{1}\}$ with $q_{2,0}^{\prime}=q_{1+\frac{1}{2}},$ let
\begin{align*}
c_{1}  &  =c_{11,\theta}+c_{12,\theta},\\
c_{1,\theta}  &  =c_{11,\theta}^{\prime}+c_{12,\theta}^{\prime}%
\end{align*}
be partitions of $c_{1}$ and $c_{1,\theta}$ with%
\[
c_{11,\theta}=c_{1}\left(  q_{1},q_{2,\theta}^{\prime}\right)  ,c_{12,\theta
}=c_{1}\left(  q_{2,\theta}^{\prime},q_{2}\right)  ,
\]%
\[
c_{11,\theta}^{\prime}=c_{1,\theta}\left(  q_{1},q_{2,\theta}^{\prime}\right)
,c_{12,\theta}^{\prime}=c_{1,\theta}\left(  q_{2,\theta}^{\prime}%
,q_{2}\right)  ,
\]
let $T_{\theta}$ be the disk enclosed by $c_{1,\theta},$ $T_{\theta}^{\prime
}=T_{0}\backslash\overline{T_{\theta}}$ and let $T_{\theta}^{\prime\prime
}=T_{\theta}\backslash\overline{T_{0}}.$ Then it is clear that $c_{12,\theta
}^{\prime}$ divides $T_{0}$ into two Jordan domains $T_{\theta}^{\prime}$ and
$T_{\theta}^{\prime\prime\prime},$ and $c_{11,\theta}$ divides $T_{\theta}$
into two Jordan domains $T_{\theta}^{\prime\prime\prime}$ and $T_{\theta
}^{\prime\prime}.$ By Lemma \ref{int-arg1} we have

\begin{claim}
\label{(h)}For sufficiently small $\theta>0,$ $\left(  \overline{T_{\theta
}^{\prime}},c_{12,\theta}\right)  $ is a simple closed Jordan domain of
$\Sigma_{L_{1}}$ (see Remark \ref{Riemann} (ii) and (iii)) with the new
boundary $c_{12,\theta}^{\prime\circ}\ $and old boundary $c_{12,\theta}.$ That
is to say, there exist two simple paths $\alpha_{12,\theta}$ and
$\alpha_{12,\theta}^{\prime}$ in $\overline{\Delta}$, both are from a point
$a_{2,\theta}^{\prime}\in\alpha_{1}^{\circ}$ to $a_{2},$ such that%
\begin{equation}
\alpha_{12,\theta}=\alpha_{1}\left(  a_{2,\theta}^{\prime},a_{2}\right)
\subset\alpha_{1},\alpha_{12,\theta}^{\prime\circ}=\alpha_{12,\theta}%
^{\prime\circ}\left(  a_{2,\theta}^{\prime},a_{2}\right)  \subset\Delta,
\label{(1}%
\end{equation}
$\alpha_{12,\theta}-\alpha_{12,\theta}^{\prime}$ encloses a Jordan domain
$D_{\theta}^{\prime}$ in $\Delta$,
\[
c_{12,\theta}=\left(  f,\alpha_{12,\theta}\right)  ,c_{12,\theta}^{\prime
}=\left(  f,\alpha_{12,\theta}^{\prime}\right)  ,
\]
and $f$ restricted to $\overline{D_{\theta}^{\prime}}$ is a homeomorphism onto
$\overline{T_{\theta}^{\prime}}.$
\end{claim}

We let
\[
\alpha_{11,\theta}=\alpha_{1}\left(  a_{1},a_{2,\theta}^{\prime}\right)  .
\]
Then we can cut $\overline{T_{\theta}^{\prime}}\backslash c_{12,\theta
}^{\prime}$ from $\Sigma_{L_{1}}$ along $c_{12,\theta}^{\prime},$ and
\label{sew24}sew $\overline{T_{\theta}^{\prime\prime}}$ to $\Sigma_{L_{1}%
}\backslash T_{\theta}^{\prime}$ along $c_{11,\theta},$ to obtain a surface
$\Sigma_{\theta}=\left(  f_{\theta},\overline{\Delta}\right)  .$ This
$\Sigma_{\theta}$ can be obtained in another way as follows.

By Lemma \ref{int-arg1}, there exists a closed path $c_{1}^{\prime}$ in
$\overline{T_{0}}$ from $q_{1}$ to $q_{1},$ oriented anticlockwise, such that
$c_{1}^{\prime\circ}\subset T_{0},$ $c_{1}^{\prime}$ is a polygonal path and
for the domain $T_{0}^{\prime}$ enclosed by $c_{1}-c_{1}^{\prime},$ the two
interior angles of $T_{0}^{\prime}$ at $q_{1}$ are both positive, and that
$\left(  \overline{T_{0}^{\prime}},c_{1}\right)  $ is a simple closed domain
of $\Sigma_{L_{1}},$ say, there exists a Jordan domain $D_{0}^{\prime}$ in
$\Delta$ such that $\partial D_{0}^{\prime}\cap\partial\Delta=\alpha_{1}$ and
$f:\overline{D_{0}^{\prime}}\backslash\{a_{1},a_{2}\}\rightarrow
\overline{T_{0}^{\prime}}\backslash\{q_{1}\}$ is a homeomorphism (see Remark
\ref{Riemann} (ii)). We let $\alpha_{1}^{\prime}=\left(  \partial
D_{0}\right)  \backslash\alpha_{1}^{\circ},$ oriented from $a_{1}$ to $a_{2}.$
Then $-c_{1}^{\prime\circ}=\left(  f,-\alpha_{1}^{\prime\circ}\right)  $ is
the new boundary of the domain $\left(  \overline{T_{0}^{\prime}}%
,c_{1}\right)  $ of $\Sigma_{L_{1}}.$ For a small enough number $\theta>0,$
$c_{1,\theta}-c_{1}^{\prime}$ encloses a domain $T_{0,\theta}^{\prime}$ and
$\left(  \overline{T_{0,\theta}^{\prime}},c_{1,\theta}\right)  $ can be
regarded as a deformation of $\left(  \overline{T_{0}^{\prime}},c_{1}\right)
,$ sharing the same new boundary $-c_{1}^{\prime\circ}.$ That is to say, when
we rotate $c_{1}$ by $h_{\theta,q_{1}}\ $to the position $c_{1,\theta},$
$c_{1,\theta}-c_{1}^{\prime}$ still enclosed a simply connected domain
$T_{0,\theta}^{\prime}$ which shares the new boundary $-c_{1}^{\prime\circ}$
of the domain $\left(  \overline{T_{0}^{\prime}},c_{1}\right)  \ $of
$\Sigma_{L_{1}}$ and when we replace $\left(  \overline{T_{0}^{\prime}}%
,c_{1}\right)  $ with $\left(  \overline{T_{0,\theta}^{\prime}},c_{1,\theta
}\right)  ,$ we obtain the new surface $\Sigma_{\theta}=\left(  f_{\theta
},\overline{\Delta}\right)  .$ It is clear that $\left(  \overline
{T_{0,\theta}^{\prime}},c_{1,\theta}\right)  $ is a simple closed domain of
the new surface $\Sigma_{\theta}$. On the other hand, we do not change
$\Sigma_{L_{1}}\backslash\left(  \overline{T_{0}^{\prime}},c_{1}\right)  $ in
the deformation. Therefore $\Sigma_{\theta}\in\mathcal{F}\left(  L,m\right)  $
and the partition (\ref{zz1}) becomes the following $\mathcal{F}\left(
L,m\right)  $-partition of $\partial\Sigma_{\theta}:$%
\begin{equation}
\partial\Sigma_{\theta}=c_{1,\theta}+c_{2}+\cdots+c_{m}. \label{zz1+1}%
\end{equation}
It is clear that
\begin{equation}
A(\Sigma_{\theta})=A(\Sigma_{L_{1}}),L(\partial\Sigma_{\theta})=L(\partial
\Sigma_{L_{1}}). \label{(L)}%
\end{equation}

By Condition \ref{(A)}, the disk $\overline{D(q_{1},\delta_{0})}$ with
$\delta_{0}=d\left(  q_{1},q_{_{1+\frac{1}{2}}}\right)  \ $is contained in
$D(q_{1},\delta_{E_{q}}/2)\subset S$ and contains at most one point of $E_{q}.
$ Then there are only three possibilities:

\noindent\textbf{Case 1. }$\overline{D(q_{1},\delta_{0})}\cap E_{q}\ $is
either empty or is a singleton $\{q^{\ast}\}$ contained in $c_{1},$ and thus
$T_{0}\cap E_{q}=\emptyset.$

\noindent\textbf{Case 2. }$\left(  \overline{D(q_{1},\delta_{0})}%
\backslash\overline{T_{0}}\right)  \cap E_{q}=\{\mathfrak{a}\}$ with
$\mathfrak{a}\in c_{1,\theta_{1}}\cap c_{1,-\theta_{2}},$ where $\theta
_{1}>0,\theta_{2}>0$ and $\pi<\theta_{1}+\theta_{2}<2\pi$. That is to say,
$\theta_{1}$ is the first number in $(0,2\pi)$ so that $\mathfrak{a}\in
c_{1,\theta_{1}}$ and $\theta_{2}$ is the first number in $(0,2\pi)$ so that
$\mathfrak{a}\in c_{1,-\theta_{2}},$ and more over, $\overline{T_{0}}\cap
E_{q}=\emptyset.$

\noindent\textbf{Case 3. }$\left(  \overline{D(q_{1},\delta_{0})}\backslash
T_{0}\right)  \cap E_{q}=\emptyset,$ and $T_{0}$ contains at most one point of
$E_{q}.$

Assume Case 1 occurs. Recall that $c_{12,0}=c_{12,0}^{\prime}=c_{12}\left(
q_{1},q_{1+\frac{1}{2}}\right)  .$ Then we may assume
\begin{equation}
\overline{D(q_{1},\delta_{0})}\cap E_{q}=\{q^{\ast}\}\in c_{12,0}. \label{dd}%
\end{equation}
When it is empty, the proof is the same, and when $\{q^{\ast}\}\in c_{11,0},$
the proof can be proceeded based on consider the rotation $\varphi
_{-\theta,q_{1}}$ in a symmetrical way.

By assumption of Case 1, we have $\overline{n}\left(  \Sigma_{\theta}\right)
=\overline{n}\left(  \Sigma_{L_{1}}\right)  ,$ and then by (\ref{(L)}) we have%
\begin{equation}
H(\Sigma_{L_{1}})=H(\Sigma_{\theta}). \label{ag29}%
\end{equation}

Let $\theta_{0}$ be the maximal positive number in $(0,\pi]$ so that
$\Sigma_{\theta}\in\mathcal{F}\left(  L,m\right)  $ is well defined for all
$\theta\in(0,\theta_{0}).$ This is equivalent to that $\theta_{0}$ is the
maximal number such that $\left(  \overline{T_{\theta}^{\prime}},c_{12,\theta
}\right)  $ is a closed simple Jordan domain of $\Sigma_{L_{1}}$ with the new
boundary $-c_{12,\theta}^{\prime\circ}$ and old boundary $c_{12,\theta}$ for
all $\theta\in(0,\theta_{0})$ (see Remark \ref{Riemann} (ii)), in other words,
for all $\theta\in(0,\theta_{0}),$ $c_{12,\theta}^{\prime\circ}$ is in the
interior of $\Sigma_{L_{1}}^{\circ},$ which just means $\alpha_{12,\theta
}^{\prime\circ}\subset\Delta,$ and $f$ has no branch point in $\alpha
_{12,\theta}^{\prime\circ}.$ Then by Calim \ref{c1} we have
\[
\theta_{0}<\pi,
\]
otherwise, $(\overline{T_{\pi}},c_{12,\pi})=\left(  \overline{T_{0}}%
,c_{1}\right)  $ is a closed simple domain of $\Sigma_{L_{1}},$ by Lemma
\ref{continue0}. Moreover, we can show the following.

\begin{claim}
\label{(j)} $\left(  \overline{T_{\theta_{0}}},c_{12,\theta_{0}}\right)  $ is
still a simple closed Jordan domain of $\Sigma_{L_{1}}$ so that the new
boundary is contained in $c_{12,\theta_{0}}^{\prime\circ}$ and the old
boundary contains $c_{12,\theta_{0}}.$ That is to say, there exist two simple
paths $\alpha_{12,\theta_{0}}$ and $\alpha_{12,\theta_{0}}^{\prime}$ in
$\overline{\Delta}$, both are from the point $a_{2,\theta_{0}}^{\prime}%
\in\alpha_{1}^{\circ}$ to $a_{2},$ such that%
\begin{equation}
\alpha_{12,\theta_{0}}=\alpha_{1}\left(  a_{2,\theta_{0}}^{\prime}%
,a_{2}\right)  \subset\alpha_{1},\alpha_{12,\theta_{0}}^{\prime\circ}%
=\alpha_{12,\theta_{0}}^{\prime\circ}\left(  a_{2,\theta_{0}}^{\prime}%
,a_{2}\right)  \subset\overline{\Delta}, \label{(2}%
\end{equation}
$\alpha_{12,\theta_{0}}-\alpha_{12,\theta_{0}}^{\prime}$ encloses a Jordan
domain $D_{\theta_{0}}^{\prime}$ in $\Delta$,
\[
c_{12,\theta_{0}}=\left(  f,\alpha_{12,\theta_{0}}\right)  ,c_{12,\theta_{0}%
}^{\prime}=\left(  f,\alpha_{12,\theta_{0}}^{\prime}\right)  ,
\]
and $f$ restricted to $\overline{D_{\theta_{0}}^{\prime}}$ is a homeomorphism
onto $\overline{T_{\theta_{0}}^{\prime}}.$
\end{claim}

The difference between Claims \ref{(h)} and \ref{(j)} is just that
$\alpha_{12,\theta_{0}}^{\prime\circ}$ may not be the new boundary, but
contains the new boundary, say, $\Delta$ in (\ref{(1}) should be replaced by
$\overline{\Delta}$ when $\theta=\theta_{0},$ as in (\ref{(2}). In fact by
definition of $\theta_{0},$ considering that $T_{\theta_{0}}^{\prime}%
=\cup_{\theta\in\lbrack0,\theta_{0})}T_{\theta}^{\prime}$ and that $T_{\theta
}^{\prime}$ is an increasing family as $\theta$ increases, $f^{-1}$ has a
univalent branch $g_{\theta_{0}}$ defined on $T_{\theta_{0}}^{\prime}.$ Then
by Lemma \ref{continue0} $g_{\theta_{0}}$ can be extended to a univalent
branch of $f^{-1}$ defined on $\overline{T_{\theta_{0}}^{\prime}}\ $such that
$g_{\theta_{0}}\left(  c_{12,\theta_{0}}\right)  =\alpha_{12,\theta_{0}}.$
Thus $\left(  \overline{T_{\theta_{0}}^{\prime}},c_{12,\theta_{0}}\right)  $
is a closed simple Jordan domain of $\Sigma_{L_{1}}.$ Then $\alpha
_{12,\theta_{0}}^{\prime}=g_{\theta_{0}}\left(  c_{12,\theta_{0}}^{\prime
}\right)  $ is a well defined arc in $\overline{\Delta}$ from $a_{2,\theta
_{0}}^{\prime}$ to $a_{2},$ and $\overline{D_{\theta_{0}}^{\prime}}%
=g_{\theta_{0}}\left(  \overline{T_{\theta_{0}}^{\prime}}\right)  $ is a
closed Jordan domain in $\overline{\Delta}.$ Therefore Claim \ref{(j)} hold.

We let $\Delta_{\theta_{0}}=\Delta\backslash\overline{D_{\theta_{0}}^{\prime}%
}.$ By condition of Case 1, $c_{12,\theta_{0}}^{\prime\circ}\cap
E_{q}=\emptyset$, and so we have

\begin{claim}
\label{(b)}$f$ has no branch value on $c_{12,\theta_{0}}^{\prime\circ}\ $and
thus each component of $\alpha_{12,\theta_{0}}^{\prime}\backslash
\partial\Delta$ has a neighborhood in $\overline{\Delta}$ on which $f$ is a homeomorphism.
\end{claim}

If $\alpha_{12,\theta_{0}}^{\prime\circ}\cap\partial\Delta=\emptyset,$ then
$\left(  f,\overline{\Delta_{\theta_{0}}}\right)  $ is a surface in
$\mathcal{F}\left(  L^{\prime},m+1\right)  ,\mathcal{\ }$with $L^{\prime
}=L-L(c_{12,\theta_{0}})+L(c_{12,\theta_{0}}^{\prime}),$ and thus for small
enough $\rho>0$ and the arc $\mathfrak{c}_{\rho}=c_{1}\left(  q_{2,\theta
_{0}+\rho}^{\prime},q_{2,\theta_{0}}^{\prime}\right)  +c_{12,\theta_{0}%
}^{\prime},$ by Lemma \ref{int-arg1} $\left(  f,\overline{\Delta_{\theta_{0}}%
}\right)  $ contains the simple and closed Jordan domain $\left(
\overline{K_{\rho}},\mathfrak{c}_{\rho}\right)  $ with $K_{\rho}=T_{\theta
_{0}+\rho}^{\prime}\backslash\overline{T_{\theta_{0}}^{\prime}}$ such that
$\mathfrak{c}_{\rho}$ is the old boundary, an arc of $\left(  f,\partial
\Delta_{\theta_{0}}\right)  ,$ and $c_{12,\theta_{0}+\rho}^{\prime\circ}$ is
the new boundary. Then we have
\[
\partial K_{\rho}=\mathfrak{c}_{\rho}-c_{12,\theta_{0}+\rho}^{\prime}%
\]
and $\left(  \overline{T_{\theta_{0}}^{\prime}},c_{12,\theta_{0}}\right)  $
can be extended to the larger simple closed Jordan domain $\left(
\overline{T_{\theta_{0}+\rho}^{\prime}},c_{12,\theta_{0}+\rho}\right)
=\left(  \overline{T_{\theta_{0}}^{\prime}},c_{12,\theta_{0}}\right)
\cup\left(  \overline{K_{\rho}},\mathfrak{c}_{\rho}\right)  $ of
$\Sigma_{L_{1}},$ and so $\Sigma_{\theta}$ is well defined for all $\theta
\in\lbrack0,\theta_{0}+\rho),$ contradicting the maximal property of
$\theta_{0}$. Thus we have

\begin{claim}
\label{(c)} $c_{12,\theta_{0}}^{\prime\circ}=\left(  f,\alpha_{12,\theta_{0}%
}^{\prime}\right)  $ has to intersect $\partial\Sigma_{L_{1}},$ say,
$\alpha_{12,\theta_{0}}^{\prime\circ}\cap\partial\Delta\neq\emptyset.$
\end{claim}

By Condition \ref{(A)}, $c_{12,\theta_{0}}^{\prime}$ is strictly convex, and
it is clear that $\alpha_{12,\theta_{0}}^{\prime}\cap\alpha_{12,\theta_{0}%
}=\{a_{2,\theta_{0}}^{\prime},a_{2}\}.$ On the other hand, by Claim \ref{(c)},
regarding $\left(  \partial\Delta\right)  \backslash\alpha_{12,\theta_{0}%
}^{\circ}$ and $-\alpha_{12,\theta_{0}}^{\prime}$ as $\alpha$ and $\beta$ in
Lemma \ref{tangent}, we conclude that $\alpha_{12,\theta_{0}}^{\prime\circ
}\cap\left(  \partial\Delta\right)  \backslash\alpha_{12,\theta_{0}}^{\circ
}\subset\left\{  a_{j}\right\}  _{j=1}^{m}$ is a finite set $\{a_{i_{1}}%
,\dots,a_{i_{k}}\}.$ Thus $\alpha_{12,\theta_{0}}^{\prime}\cap\partial
\Delta=\{a_{2},a_{i_{1}},\dots,a_{i_{k}},a_{2,\theta_{0}}^{\prime}\},$
arranged anticlockwise on $\partial\Delta.$ Therefore, by Claim \ref{(b)}, we have

\begin{claim}
\label{(d)}$\alpha_{12,\theta_{0}}^{\prime\circ}\cap\left(  \partial
\Delta\right)  \backslash\alpha_{12,\theta_{0}}^{\circ}=\{a_{i_{1}}%
,\dots,a_{i_{k}}\}$ divides $\alpha_{12,\theta_{0}}^{\prime\circ}$ into $k+1$
open arcs, each of which has a neighborhood in $\overline{\Delta}$ on which
$f$ is a homeomorphism.
\end{claim}

Then $\Sigma_{\theta_{0}}$ is no longer a surface, but is consisted of $k+1$
surfaces. We will show that these surfaces are all contained in $\mathcal{F}%
\left(  L,m-1\right)  \subset\mathcal{F}\left(  L,m\right)  .$ We only prove
this in the case that the finite set $\alpha_{12,\theta_{0}}^{\prime\circ}%
\cap\partial\Delta\ $is a singleton $\{a_{i_{1}}\},$ say, $k=1.$ When $k>1$,
the discussion is similar and more simpler. It is clear that we can define the
partition%
\begin{equation}
c_{1,\theta_{0}}=\mathfrak{c}_{1,\theta_{0}}^{\prime}+\mathfrak{c}%
_{1,\theta_{0}}^{\prime\prime}=c_{1,\theta_{0}}\left(  q_{1},q_{i_{1}}\right)
+c_{1,\theta_{0}}\left(  q_{i_{1}},q_{1}\right)  . \label{ag32}%
\end{equation}
Then $\Sigma_{\theta_{0}}$ is consisted of two surfaces $\Sigma_{\theta_{0}%
}^{1}$ and $\Sigma_{\theta_{0}}^{2}$ linked at the point $\left(  f,a_{i_{1}%
}\right)  ,$ so that
\begin{equation}
\partial\Sigma_{\theta_{0}}^{1}=\mathfrak{c}_{1,\theta_{0}}^{\prime}+c_{i_{1}%
}+\cdots+c_{m}, \label{a101}%
\end{equation}%
\begin{equation}
\partial\Sigma_{\theta_{0}}^{2}=\mathfrak{c}_{1,\theta_{0}}^{\prime\prime
}+c_{2}+\cdots+c_{i_{1}-1}. \label{a102}%
\end{equation}
It is clear that the partition (\ref{a101}) contains at least $2$ terms, and
so does (\ref{a102}). Thus by Claim \ref{(d)} we have
\begin{equation}
\Sigma_{\theta_{0}}^{j}\in\mathcal{F}\left(  L,m-1\right)  ,j=1,2, \label{(l)}%
\end{equation}
which implies
\begin{equation}
\left\{  \Sigma_{\theta,}^{1},\Sigma_{\theta_{0}}^{2}\right\}  \subset
\mathcal{F}\left(  L,m\right)  . \label{(1)-1}%
\end{equation}

We still assume $k=2.$ Then we see that (\ref{(L)}) still holds in the
following form
\begin{equation}
\sum_{j=1}^{2}A(\Sigma_{\theta_{0}}^{j})=A(\Sigma_{L_{1}}),\sum_{j=1}%
^{2}L(\partial\Sigma_{\theta_{0}}^{j})=L\left(  \partial\Sigma_{L_{1}}\right)
, \label{a105}%
\end{equation}
and by (\ref{dd}) we have $\sum_{j=1}^{2}\overline{n}(\Sigma_{\theta_{0}}%
^{j})=\overline{n}(\Sigma_{L_{1}}),$ and then $\sum_{j=1}^{2}R(\Sigma
_{\theta_{0}}^{j})=R(\Sigma_{L_{1}}).$ This contradicts Lemma \ref{undec}, and
thus Case can not occur.

Assume Case 2 occurs. If $\theta_{1}\geq\pi,$ then we can obtain a
contradiction as in Case 1.

Assume $\theta_{1}<\pi.$ Then we may find the maximum $\theta_{0}^{\prime}$ in
$(0,\theta_{1}]$ so that $\left(  \overline{T_{\theta_{0}^{\prime}}^{\prime}%
},c_{12,\theta_{0}^{\prime}}\right)  $ is a simple and closed Jordan domain of
$\Sigma_{L_{1}},$ by repeating the argument for Claim (\ref{(j)}). If
$\theta_{0}^{\prime}<\theta_{1},$ or if $\theta_{0}^{\prime}=\theta_{1}$ and
$\alpha_{12,\theta_{0}^{\prime}}^{\prime\circ}\cap\partial\Delta\neq
\emptyset,$ then we can find a contradiction again using the same method of
Case 1.

Assume $\theta_{0}^{\prime}=\theta_{1}\ $and we cannot obtain a contradiction
as in Case 1. Then we have

\begin{claim}
\label{A}$\left(  \overline{T_{\theta_{1}}^{\prime}},c_{12,\theta_{1}}\right)
=\left(  \overline{T_{\theta_{1}}^{\prime}},c_{12,\theta_{0}^{\prime}}\right)
$ is a simple closed Jordan domain of $\Sigma_{L_{1}}$ such that
$c_{12,\theta_{1}}^{\prime\circ}$ is the new boundary, say, $\alpha
_{12,\theta_{1}}^{\prime\circ}\subset\Delta.$
\end{claim}

Then we apply the above argument to obtain $\Sigma_{-\theta}$ for $\theta>0,$
by rotating $T_{0}$ into $T_{-\theta}$ in the other direction in a symmetrical
way. That is, we can construct surfaces $\Sigma_{-\theta},\theta>0,$ such that
the partition (\ref{zz1+1}) works, in which $c_{1,\theta}$ becomes the
rotation $c_{1,-\theta}$ of $c_{1}.$ Then we can either obtain a contradiction
as in the Case 1, or we must have that $\left(  \overline{T_{-\theta_{2}%
}^{\prime}},c_{11,-\theta_{2}}\right)  =\left(  \overline{T_{0}\backslash
T_{-\theta_{2}}},c_{11,-\theta_{2}}\right)  $ is a simple closed Jordan domain
of $\Sigma_{L_{1}},$ as Claim \ref{A}. Then we can see that $\left(
\overline{T_{-\theta_{2}}^{\prime}}\cup\overline{T_{\theta_{1}}^{\prime}%
},c_{1}\right)  =\left(  \overline{T_{0}},c_{1}\right)  $ is a simple Jordan
domain of $\Sigma_{L_{1}},$ contradicting Claim \ref{c1}. Thus in Case 2, we
also have a contradiction. Note that $\theta_{1}+\theta_{2}\geq\pi$ and thus
$\overline{T_{-\theta_{2}}^{\prime}}\cup\overline{T_{\theta_{1}}^{\prime}%
}=\overline{T_{0}}.$ In fact $f_{n}^{-1}$ has branches defined on
$\overline{T_{-\theta_{2}}^{\prime}}\ $and $\overline{T_{\theta_{1}}^{\prime}%
}$ which are agree on $\overline{T_{-\theta_{2}}^{\prime}}\cap\overline
{T_{\theta_{1}}^{\prime}}$ which contains an arc of $c_{1},$ and thus the two
branches are agree on $\overline{T_{-\theta_{2}}^{\prime}}\cap\overline
{T_{\theta_{1}}^{\prime}},$ and so these two branches defines a global branch
on $\overline{T_{0}},$ contradicting Claim \ref{c1}.

Assume Case 3 occurs. Then we may assume
\begin{equation}
\overline{D\left(  q_{1},\delta_{0}\right)  }\cap E_{q}=T_{0}\cap
E_{q}=\{q^{\ast}\}, \label{x13}%
\end{equation}
otherwise we have
\[
\overline{D\left(  q_{1},\delta_{0}\right)  }\cap E_{q}=T_{0}\cap
E_{q}=\emptyset,
\]
and we can obtain a contradiction based on the discussion entirely the same as
in Case 1 in this later case. Then we may assume $q^{\ast}\in c_{12,\theta
^{\ast}}^{\prime\circ}\ $for some $\theta^{\ast}\in(0,\pi).$ If $\theta^{\ast
}\geq\pi,$ we can obtain a contradiction as in Case 1, by Claim \ref{c1}.

Let $\theta_{0}$ be the maximum in $(0,\theta^{\ast}]$ such that
$\Sigma_{\theta}$ is well defined, say, $\alpha_{12,\theta}^{\prime\circ
}\subset\Delta,$ for every $\theta\in(0,\theta_{0}).$ If $\theta_{0}%
<\theta^{\ast},$ then we can obtain a contradiction as in Case 1. So we may
assume that
\begin{equation}
\theta_{0}=\theta^{\ast}. \label{ii}%
\end{equation}

The argument in Case 1 for $\Sigma_{\theta}$ with $\theta\in(0,\theta_{0})$
applies, and (\ref{zz1+1}) and (\ref{(L)}) both hold for $\theta\in
(0,\theta_{0})$, and by (\ref{ii}) we have $\overline{n}\left(  \Sigma
_{\theta}\right)  =\overline{n}\left(  \Sigma_{L_{1}}\right)  $ for all
$\theta\in\lbrack0,\theta_{0})=[0,\theta^{\ast}).$ Therefore (\ref{ag29})
still holds and thus

\begin{claim}
\label{exex}$\Sigma_{\theta}$ is a precise extremal surface of $\mathcal{F}%
\left(  L,m\right)  $ for every $\theta\in\lbrack\theta,\theta_{0}%
)=[\theta,\theta^{\ast}).$
\end{claim}

As in Case 1, we can show that $\alpha_{12,\theta_{0}}^{\prime\circ}%
\cap\partial\Delta$ is a finite set $\{a_{i_{1}},a_{i_{2}},\dots,a_{i_{k}}\}$
in $\{a_{j}\}_{j=1}^{m},$ and then $\Sigma_{\theta_{0}}$ is consisted of $k+1$
surfaces $\left\{  \Sigma_{\theta_{0}}^{j}\right\}  _{j=1}^{k+1}$ of
$\mathcal{F}\left(  L\right)  $ linked at $q_{i_{1}},\dots,q_{i_{k}},$ with
\begin{equation}
\sum_{j=1}^{k+1}A(\Sigma_{\theta_{0}}^{j})=A(\Sigma_{L_{1}}),\sum_{j=1}%
^{k+1}L(\partial\Sigma_{\theta_{0}}^{j})=L(\partial\Sigma_{L_{1}}). \label{ee}%
\end{equation}
Here $k=0$ and $\Sigma_{\theta_{0}}^{j}=\Sigma_{\theta_{0}}$ when
$\alpha_{12,\theta_{0}}^{\prime\circ}\cap\partial\Delta=\emptyset.$ By the
assumption (\ref{ii}), we have
\begin{equation}
q^{\ast}\in c_{12,\theta_{0}}^{\prime\circ}=c_{12,\theta^{\ast}}^{\prime\circ
}. \label{ag28}%
\end{equation}

We show that
\begin{equation}
k\neq0 \label{30}%
\end{equation}
and
\begin{equation}
q^{\ast}\in\{q_{i_{1}},q_{i_{2}},\dots,q_{i_{k}}\}. \label{x14}%
\end{equation}

If $k=0,$ then by (\ref{ag28}) $q^{\ast}$ is in the new boundary
$c_{12,\theta^{\ast}}^{\prime\circ}$ of $\left(  \overline{T_{\theta_{0}%
}^{\prime}},c_{12,\theta_{0}}\right)  $ and%
\begin{equation}
d_{f_{\theta}}\left(  f_{\theta}^{-1}(E_{q}\right)  ,\partial\Delta)\leq
d\left(  q^{\ast},c_{12,\theta}^{\prime}\right)  \rightarrow0 \label{gg}%
\end{equation}
as $\theta\rightarrow\theta_{0}.$ But (\ref{gg}) can't hold by Claim
\ref{exex} and Lemma \ref{nobo}. Thus $k\geq1.$

Assume (\ref{x14}) fails. Then $q^{\ast}$ is again in the new boundary
$c_{12,\theta^{\ast}}^{\prime\circ}$ of $\left(  \overline{T_{\theta_{0}%
}^{\prime}},c_{12,\theta_{0}}\right)  $ and (\ref{gg}) again holds as
$\theta\rightarrow\theta_{0}$, which again contradicts Claim \ref{exex} or
Lemma \ref{nobo}. Thus we have (\ref{x14}). Then $\left(  \alpha
_{12,\theta_{0}}^{\prime\circ}\backslash\left\{  a_{i_{j}}\right\}  _{j=1}%
^{k}\right)  \cap f^{-1}(E_{q})=\emptyset$ and each component of
$\alpha_{12,\theta_{0}}^{\prime\circ}\backslash\partial\Delta=\alpha
_{12,\theta_{0}}^{\prime\circ}\backslash\left\{  a_{i_{j}}\right\}  _{j=1}%
^{k}$ has a neighborhood in $\overline{\Delta}$ on which $f$ is homeomorphic
(by Lemma \ref{home2} (ii)). Then we can show that each $\Sigma_{\theta_{0}%
}^{i}$ is contained in $\mathcal{F}\left(  L,m\right)  $ for $i=1,2,\dots,k+1$
(this is easy to see when $k=1$ as in Case 1, and when $k>1,$ the proof is
similar). On the other hand, (\ref{x13}) and (\ref{x14}) implies that
\[
\sum_{j=1}^{k+1}\overline{n}\left(  \Sigma_{\theta_{0}}^{j}\right)
=\overline{n}\left(  \Sigma_{L_{1}}\right)  ,
\]
which with (\ref{ee}) implies that $\sum_{j=1}^{k+1}R(\Sigma_{\theta_{0}}%
^{j})=R(\Sigma_{L_{1}})\ $and thus by (\ref{ee}) $\Sigma_{L_{1}}$ is
decomposable in $\mathcal{F}\left(  L,m\right)  ,$ contradicting Lemma
\ref{undec}.

We have proved that Cases 1--3 can't occur. Thus Condition \ref{(A)} can't be
satisfied by $\Sigma_{L_{1}}$ and the lemma is proved completely.
\end{proof}

\begin{remark}
The proof of Lemma \ref{mykey} is to deduce contradictions when $c_{1}$ is a
circle with $L(c_{1})<\delta_{E_{q}}/2.$ If the condition that $\Sigma_{L_{1}%
}$ is precise extremal in $\mathcal{F}_{r}\left(  L,m\right)  $ is
strengthened to that $\Sigma_{L_{1}}$ is precise extremal in $\mathcal{F}%
_{r}\left(  L\right)  =\cup_{m=1}^{\infty}\mathcal{F}_{r}\left(  L,m\right)
$, the discussion can be greatly simplified, since we need not to discuss the
number of edges. For example, we can easily prove the following:
\end{remark}

\begin{corollary}
\label{1-cir1}If $\Sigma_{L_{1}}\in\mathcal{F}_{r}\left(  L,m\right)  $ is a
precise extremal surface of $\mathcal{F}_{r}\left(  L\right)  ,$ (\ref{zz1})
is an $\mathcal{F}\left(  L,m\right)  $-partition such that $\{q_{j}%
\}_{j=1}^{m}\subset E_{q}\ $and that $c_{1}$ is a circle with $L(c_{1})<2\pi$
and
\begin{equation}
c_{1}\cap E_{q}=\{q_{1}\}. \label{ag31}%
\end{equation}
Then there exists a precise extremal surface $\Sigma_{L_{1}}^{\prime}\ $of
$\mathcal{F}_{r}\left(  L\right)  $ such that $\partial\Sigma_{L_{1}}^{\prime
}$ has the following $\mathcal{F}\left(  L,m\right)  $-partition%
\[
\partial\Sigma_{L_{1}}^{\prime}=c_{1,\theta_{0}}+c_{2}+\cdots+c_{m}%
\]
where $c_{1,\theta_{0}}$ is the rotation $h_{\theta_{0},q_{1}}\left(
c_{1}\right)  $ of $c_{1}$ with angle $\theta_{0}\in(0,\pi)\ $and
$c_{1,\theta_{0}}^{\circ}\cap E_{q}\neq\emptyset,$ say, $c_{1,\theta_{0}}\cap
E_{q}$ contains not only $q_{1},$ but also a point other than $q_{1}.$
\end{corollary}

In Lemma \ref{mykey} we assumed $L(c_{1})<\delta_{E_{q}}/2$ and permitted
$c_{1}^{\circ}$ contains a point of $E_{q}.$ But in the Corollary we assume
$c_{1}^{\circ}\cap E_{q}=\emptyset$ but $L(c_{1})<2\pi,$ and require that the
closed arc $c_{1,\theta_{0}}$ of the new surface boundary contains not only
the point $q_{1}$ of $E_{q},$ but also another point of $E_{q}$.

\begin{proof}
Since $L(c_{1})<2\pi,$ $c_{1}$ is strictly convex. We still use the notations
in the above proof. For sufficiently small $\theta>0$, by (\ref{ag31}) we have
the following conclusion (1)--(3):

(1) $\overline{T_{\theta}^{\prime}}\cap E_{q}=\{q_{1}\}$.

(2) $\overline{T_{\theta}^{\prime\prime}}\cap E_{q}=\{q_{1}\}$.

(3) $\left(  \overline{T_{\theta}^{\prime}},c_{12,\theta}\right)  $ is a
simple closed Jordan domain of $\Sigma_{L_{1}}.$

Then $\Sigma_{\theta}$ is a well defined surface in $\mathcal{F}\left(
L,m\right)  .$ Let $\theta_{1},\theta_{2},\theta_{3}$ be the supremum of
$\theta$ in $(0,\pi)$ such that (1), (2), (3) hold on $[0,\theta]$ respectively.

We first consider the case $\theta_{1}\leq\theta_{2}.$ This means that the
moving circle $c_{\theta}$ first meets $T_{0}\cap E_{q}$ before $c_{\theta}$
first meets $\left(  S\backslash\overline{T_{0}}\right)  \cap E_{q}.$

If $\theta_{3}\leq\theta_{1},$ then we can show that Claims \ref{(c)} and
\ref{(d)} hold for $c_{12,\theta_{3}}^{\prime\circ}=\left(  f,\alpha
_{12,\theta_{3}}^{\prime}\right)  $ and $\alpha_{12,\theta_{3}}^{\prime\circ
},$ and thus $\Sigma_{\theta_{3}}$ is consisted of a finite number of
$k+1,k\geq1,$ surfaces $\Sigma_{\theta_{3}}^{j}$ in $\mathcal{F}\left(
L\right)  $ with%
\begin{equation}
\sum_{j=1}^{k+1}A(\Sigma_{\theta_{3}}^{j})=A(\Sigma_{L_{1}}),\sum_{j=1}%
^{k+1}L(\partial\Sigma_{\theta_{3}}^{j})=L(\partial\Sigma_{L_{1}}),
\label{ag35}%
\end{equation}
and%
\begin{equation}
\sum_{j=1}^{k+1}\overline{n}(\Sigma_{\theta_{3}}^{j})\leq\overline{n}%
(\Sigma_{L_{1}}), \label{ag34}%
\end{equation}
inequality holding if and only if $\theta_{3}=\theta_{1}$ and $f^{-1}%
(E_{q})\cap\alpha_{12,\theta_{3}}^{\prime\circ}\backslash\{a_{i_{1}},a_{i_{2}%
},\dots,a_{i_{k}}\}\neq\emptyset,$ which implies
\begin{equation}
\sum_{j=1}^{k+1}R(\Sigma_{\theta_{3}}^{j})\geq R(\Sigma_{L_{1}}),\sum
_{j=1}^{k+1}L(\partial\Sigma_{\theta_{3}}^{j})=L(\partial\Sigma_{L_{1}}),
\label{ag33}%
\end{equation}
contradicting Lemma \ref{undec} for $\mathcal{F}\left(  L\right)  $.

If $\theta_{3}>\theta_{1},$ then $\Sigma_{\theta_{1}}$ is a surface of
$\mathcal{F}\left(  L\right)  $ with (\ref{zz1+1}) and (\ref{(L)}) holding for
$\theta=\theta_{1}$. But $\overline{n}\left(  \Sigma_{L_{1}}\right)
>\overline{n}\left(  \Sigma_{\theta_{1}}\right)  $. Thus $H\left(
\Sigma_{\theta_{1}}\right)  >H(\Sigma_{L_{1}})$ and $\Sigma_{L_{1}}$ is not
extremal in $\mathcal{F}\left(  L\right)  ,$ contradicting the hypothesis. We
have proved the result for the case $\theta_{1}\leq\theta_{2}.$

Assume $\theta_{1}>\theta_{2}.$ This means that the moving circle $c_{\theta}
$ first meets $E_{q}$ from outside of $c_{\theta}$ before $c_{\theta}$ first
meets $E_{q}\ $from inside of $c_{\theta}$.

If $\theta_{3}>\theta_{2},$ then $\Sigma_{\theta_{2}}\in\mathcal{F}\left(
L\right)  $, (\ref{ag35}) holds for $\theta_{2}$ and $\overline{n}\left(
\Sigma_{L_{1}}\right)  =\overline{n}\left(  \Sigma_{\theta_{2}}\right)  $, and
thus $R(\Sigma_{L_{1}})=R(\Sigma_{\theta_{2}}).$ It is clear that
$c_{1,\theta_{2}}$ contains more than one points of $E_{q}$ and $q_{1}\in
c_{1,\theta_{2}}.$ Therefore $\Sigma_{\theta_{2}}$ satisfies the corollary.

Assume $\theta_{3}\leq\theta_{2}.$ Then we can obtain a contradiction as the
discussion for the case $\theta_{1}\leq\theta_{2}$ and $\theta_{3}\leq
\theta_{1}$.
\end{proof}

Recall that our goal in this section is to prove Theorem \ref{cat2}. We will
in fact deduce a contradiction from the opposite of the conclusion, say under
the condition that $\Sigma_{L_{1}}\notin\mathcal{F}_{r}\left(  L,m-1\right)
.$ This means we assume the following condition, before we prove Theorem
\ref{cat2}.

\begin{condition}
\label{nom-1}$\partial\Sigma_{L_{1}}$ has no $\mathcal{F}\left(  L,m-1\right)
$ partition, say, $\partial\Sigma_{L_{1}}\in\mathcal{F}\left(  L,m\right)
\backslash\mathcal{F}\left(  L,m-1\right)  $.
\end{condition}

We first show the following.

\begin{lemma}
\label{Eq-end}Under Condition \ref{nom-1}, every $c_{j}$ of (\ref{zz1}) which
is contained in $\mathfrak{C}^{2}$ satisfies%
\begin{equation}
c_{j}\cap E_{q}=\partial c_{j}, \label{aaa}%
\end{equation}
and%
\begin{equation}
L\left(  c_{j}\right)  \geq\delta_{E_{q}}. \label{aaa1}%
\end{equation}

\end{lemma}

\begin{proof}
\label{20220118 copy(1)}\label{20220420}\label{20220908}\label{20230220}It is
trivial that (\ref{aaa}) implies (\ref{aaa1}). Assume that (\ref{aaa}) fails
for some $c_{j_{0}}\ $contained in $\mathfrak{C}^{2}.$ Then, by definition of
$\mathfrak{C}^{2}$, $c_{j_{0}}$ is not closed and one endpoint of $c_{j_{0}}$
is not in $E_{q},$ and we may assume $q_{j_{0}}$ is not in $E_{q}.$ Thus by
Lemma \ref{circular} $c_{j_{0}-1}+c_{j_{0}}$ is circular at $q_{j_{0}}$ and
$f$ restricted to a neighborhood of $a_{j_{0}}$ in $\overline{\Delta}$ is
homeomorphic, and so $\left(  f,\alpha_{j_{0}-1}+\alpha_{j_{0}}\right)  $ is a
locally simple circular arc, and $c_{j_{0}-1}$ and $c_{j_{0}}$ are both in the
same circle. If $c_{j_{0}-1}+c_{j_{0}}$ is simple, then $c_{j_{0}-1}+c_{j_{0}%
}$ in (\ref{zz1}) can be merged into one edge so that (\ref{zz1}) becomes an
$\mathcal{F}\left(  L,m-1\right)  $-partition of $\partial\Sigma_{L_{1}},$
contradicting Condition \ref{nom-1}. Then $c_{j_{0}-1}+c_{j_{0}}$ is locally
simple but not globally simple. Thus, considering that both $c_{j_{0}-1}$ and
$c_{j_{0}} $ are simple and $c_{j_{0}}\in\mathfrak{C}^{2}$, we have
$q_{j_{0}-1}\in c_{j_{0}}\backslash\left\{  q_{j_{0}+1}\right\}  \ $and
$q_{j_{0}+1}\in c_{j_{0}-1}^{\circ},$ and thus we can write
\[
c_{j_{0}-1}+c_{j_{0}}=C_{j_{0}-1}^{\prime}+C_{j_{0}},
\]
in which $C_{j_{0}-1}^{\prime}=C_{j_{0}-1}^{\prime}\left(  q_{j_{0}%
-1,}q_{j_{0}+1}\right)  =c_{j_{0}-1}\left(  q_{j_{0}-1,}q_{j_{0}+1}\right)  $
and $C_{j_{0}}=C_{j_{0}}\left(  q_{j_{0}+1,}q_{j_{0}+1}\right)  $ is the
circle $c_{j_{0}-1}\left(  q_{j_{0}+1},q_{j_{0}}\right)  +c_{j_{0}}\left(
q_{j_{0}},q_{j_{0}+1}\right)  .$ Therefore we have the $\mathcal{F}\left(
L,m\right)  $-partition%
\begin{equation}
\partial\Sigma_{L_{1}}=c_{1}+\cdots+c_{j_{0}-2}+C_{j_{0}-1}^{\prime}+C_{j_{0}%
}+c_{j_{0}+1}+\cdots+c_{m}, \label{(n)}%
\end{equation}
and as a set on $S,$ $C_{j_{0}-1}^{\prime}\subset c_{j_{0}}$ and the initial
point $q_{j_{0}-1}$ of $C_{j_{0}-1}^{\prime}$ is not in $E_{q}$ (we have
assumed $q_{j_{0}}\notin E_{q}$ and thus $q_{j_{0}-1}\in c_{j_{0}}%
\backslash\left\{  q_{j_{0}+1}\right\}  \subset S\backslash E_{q}$). We
conclude this by

\begin{claim}
\label{mk1}$C_{j_{0}-1}^{\prime}\in\mathfrak{C}^{2}\ $and the initial point
$q_{j_{0}-1}$ of $C_{j_{0}-1}^{\prime}\ $is not in $E_{q}.$
\end{claim}

By Claim \ref{mk1}, the above discussion about $c_{j_{0}-1}+c_{j_{0}}$ applies
to $c_{j_{0}-2}+C_{j_{0}-1}^{\prime}$. Then we can repeat the same argument
$m-1$ times to obtain an $\mathcal{F}\left(  L,m\right)  $-partition%
\begin{equation}
\partial\Sigma_{L_{1}}=C_{1}\left(  q_{j_{0}+1},q_{j_{0}+1}\right)
+C_{2}\left(  q_{j_{0}+1},q_{j_{0}+1}\right)  +\cdots+C_{m}\left(  q_{j_{0}%
+1},q_{j_{0}+1}\right)  \label{kk}%
\end{equation}
such that all $C_{1},\dots,C_{m}$ are simple circles contained in the same
circle $C_{j}=C,$ and by Lemma \ref{home2} (ii) each point $p\in
C_{j}\backslash E_{q}$ is a simple point of $\Sigma_{L_{1}}\ $for
$j=1,\dots,m$ (see Definition \ref{simple} (a) and (b)). Then (\ref{jj}) and
(\ref{kk}) imply%
\[
L(C_{j})=L(C)=\frac{L(\partial\Sigma_{L_{1}})}{m}=\frac{L_{1}}{m}\leq\frac
{L}{m}\leq\frac{\delta_{E_{q}}}{10},
\]
which implies that $C$ contains at most one point of $E_{q}.$ Then we may
assume $C_{j}=C_{j}\left(  q_{j_{0}+1}^{\prime},q_{j_{0}+1}^{\prime}\right)  $
such that $C_{j}^{\circ}=C_{j}\backslash\{q_{j_{0}+1}^{\prime}\}\subset
S\backslash E_{q},$ say, $C_{j}\in\mathfrak{C}^{1}\backslash\mathfrak{C}^{2}.
$ Then by Lemma \ref{1-cir}, we have $m=1.$ This contradicts (\ref{jj}) and
Condition \ref{nom-1}.\label{20220118}\label{20220420 copy(1)}%
\label{20220908 copy(1)}
\end{proof}

\begin{lemma}
\label{mykey2}Assume that Condition \ref{nom-1} holds and
\begin{equation}
L(c_{1}+c_{2})<\delta_{E_{q}}/2. \label{2022-01-03-2}%
\end{equation}
Then the following hold.

(i) $q_{1}\neq q_{3}.$

(ii) $q_{1}\notin c_{2}^{\circ}$ and $q_{3}\notin c_{1}^{\circ}.$

(iii) $c_{1}+c_{2}$ cannot contain a closed subarc $c_{1}^{\prime}%
+c_{2}^{\prime}=c_{1}\left(  q_{1}^{\prime},q_{2}\right)  +c_{2}\left(
q_{2},q_{1}^{\prime}\right)  $ such that
\begin{equation}
q_{1}^{\prime}\in E_{q}\mathrm{\ and\ }q_{1}^{\prime}\notin\{q_{1},q_{2}%
,q_{3}\}. \label{2023-02-20}%
\end{equation}
\bigskip
\end{lemma}

\begin{proof}
The proof of (i) and (ii) is relatively easy, while the proof of (iii) is
quite complicated, but it is very similar to the proof of Case 1 in Lemma
\ref{mykey}.

By Lemma \ref{mykey} and (\ref{2022-01-03-2}) we have

\begin{claim}
\label{(a')}Neither $c_{1}$ nor $c_{2}$ is a closed arc, and $\left(
c_{1}+c_{2}\right)  \cap E_{q}$ contains at most one point.
\end{claim}

To prove (i) we assume $q_{1}=q_{3}.$ If $q_{1}\ $or $q_{2}\ $is contained in
$E_{q},$ then by Claim \ref{(a')}, we have $c_{1}^{\circ}\cap E_{q}%
=\emptyset,$ say, $c_{1}\in\mathfrak{c}^{2},$ and then by Lemma \ref{Eq-end}
$L(c_{1})\geq\delta_{E_{q}},$ contradicting (\ref{2022-01-03-2}). So we may
assume that neither $q_{1}$ nor $q_{2}$ is contained in $E_{q}.$ If
$c_{1}^{\circ}\cap c_{2}^{\circ}\neq\emptyset,$ then $c_{1}$ and $c_{2}$
contains three common points, which implies $c_{1}=-c_{2}$ and $c_{1}+c_{2}$
is folded at $q_{1}\notin E_{q}$, contradicting Lemma \ref{nonf}. Thus by
Claim \ref{(a')}, either $c_{1}$ or $c_{2}$ is contained in $\mathfrak{C}%
^{2},$ but this, together with Lemma \ref{Eq-end}, implies $L(c_{1}+c_{2}%
)\geq\delta_{E_{q}},$ contradicting (\ref{2022-01-03-2}) once more. (i) is proved.

To prove (ii) assume that it fails. Then we may assume $q_{3}\in c_{1}^{\circ
}.$ We first show that

\begin{claim}
\label{nonboth}Either $c_{1}^{\circ}\cap E_{q}$ or $c_{2}^{\circ}\cap E_{q}$
is empty.
\end{claim}

Assume neither $c_{1}^{\circ}\cap E_{q}$ nor $c_{2}^{\circ}\cap E_{q}$ is
empty. Then by Claim \ref{(a')} $c_{1}^{\circ}\cap c_{2}^{\circ}\cap E_{q}$ is
a singleton $\{q^{\ast}\}$ in $E_{q},$ $q_{2}\notin E_{q},$ and thus $c_{1}$
and $c_{2}$ contain three common points, but $c_{1}+c_{2}$ is not folded at
$q_{2}$ by Lemma \ref{nonf}. Then by Claim \ref{(a')}, neither $c_{1}\ $nor
$c_{2}$ is closed but $c_{1}+c_{2}$ is contained in the same circle and
$c_{1}+c_{2}$ is more than that circle since $q_{3}\in c_{1}^{\circ}$, and by
Lemma \ref{home2}, every point of $\left(  \alpha_{1}+\alpha_{2}\right)
^{\circ}\ $is a simple point of $f.$ Then we can write%
\[
c_{1}+c_{2}=C_{1}^{\prime}+C_{2},
\]
where $C_{1}^{\prime}=c_{1}\left(  q_{1},q_{3}\right)  \ $is the subarc of
$c_{1}$ and $C_{2}=c_{1}\left(  q_{3},q_{2}\right)  +c_{2}\left(  q_{2}%
,q_{3}\right)  \ $is a circle, and moreover, $C_{1}^{\prime\circ}$ and
$C_{2}^{\circ}$ have neighborhoods in $\Sigma_{L_{1}}$ which are simple
domains of $\Sigma_{L_{1}}$ (see Remark \ref{Riemann} (ii)). Hence the new
partition
\[
\partial\Sigma_{L_{1}}=C_{1}^{\prime}+C_{2}+c_{3}+c_{4}+\cdots+c_{m}%
\]
is still an $\mathcal{F}\left(  L,m\right)  $-partition of $\partial
\Sigma_{L_{1}}$. But this contradicts Lemma \ref{mykey}, since $L(C_{2}%
)<L(c_{1}+c_{2})<\delta_{E_{q}}/2.$ Thus Claim \ref{nonboth} holds.

By Claim \ref{nonboth} we may assume $c_{1}^{\circ}\cap E_{q}=\emptyset.$ Then
by Claim \ref{(a')} we have $c_{1}\in\mathfrak{C}^{2}$ and then by Lemma
\ref{Eq-end} $c_{1}$ contains two distinct endpoints in $E_{q}$. But this
contradicts Claim \ref{(a')} and (ii) is proved.

To prove (iii), assume it fails, say, that $c_{1}^{\prime}$ and $c_{2}%
^{\prime}$ satisfying the condition of (iii) exist. By (i) and (ii),
$c_{1}+c_{2}$ cannot be contained in a circle and can not be folded at $q_{2}.
$ Hence, by Lemma \ref{nonf}, $c_{1}^{\prime}+c_{2}^{\prime}$ is a simple
closed arc.

Let $\alpha_{1}^{\prime}=\alpha_{1}^{\prime}\left(  a_{1}^{\prime}%
,a_{2}\right)  $ and $\alpha_{2}^{\prime}=\alpha_{2}^{\prime}\left(
a_{2},a_{1}^{\prime\prime}\right)  $ be subarcs of $\alpha_{1}$ and
$\alpha_{2}$ such that $c_{1}^{\prime}=\left(  f,\alpha_{1}^{\prime}\right)  $
and $c_{2}^{\prime}=\left(  f,\alpha_{2}^{\prime}\right)  \ $and let $T_{0}$
be the domain on $S$ enclosed by $c_{1}^{\prime}+c_{2}^{\prime}.$ Then by
(\ref{2023-02-20}) we have%
\begin{equation}
D(q_{1}^{\prime},\delta_{E_{q}}/2)\cap E_{q}=\overline{T_{0}}\cap
E_{q}=\left\{  q_{1}^{\prime}\right\}  , \label{2022-01-03-1}%
\end{equation}
and that $c_{1}$ or $c_{2}$ is strictly convex. On the other hand, $c_{1}$ and
$c_{2}$ cannot externally tangent at $q_{2}$ since they both contain the two
common points $q_{1}^{\prime}$ and $q_{2}.$ Hence, by definition of
$\mathcal{F}\left(  L,m\right)  $-partition in Definition \ref{circu}, we have

\begin{claim}
\label{B}$f$ is homeomorphic in a neighborhood of $\left(  \alpha_{1}^{\prime
}+\alpha_{2}^{\prime}\right)  ^{\circ}$ in $\overline{\Delta},$ $0<\angle
\left(  \Sigma_{L_{1}},a_{2}\right)  <2\pi,$ and we may assume $c_{2}$ is
strictly convex.
\end{claim}

Similar to the discussion of Claim \ref{c1}, we have the following.

\begin{claim}
\label{c2}$\left(  \overline{T_{0}},c_{1}^{\prime}+c_{2}^{\prime}\right)  $
cannot be a simple closed domain of $\Sigma_{L_{1}},$ say, there is no
univalent branch $g$ of $f^{-1}$ defined on $\overline{T_{0}}\backslash
\{q^{\prime}\}$ such that $g\left(  c_{1}^{\prime}+c_{2}^{\prime}\right)
=\alpha_{1}^{\prime}+\alpha_{2}^{\prime}.$
\end{claim}

Now we have only two cases to discuss.

\noindent\textbf{Case A. }$c_{1}+c_{2}$ is convex at $q_{2},$ say
$0<\angle\left(  \Sigma_{L_{1}},a_{2}\right)  \leq\pi.$

\noindent\textbf{Case B. }$c_{1}+c_{2}$ is concave at $q_{2},$ say $\pi
<\angle\left(  \Sigma_{L_{1}},a_{2}\right)  <2\pi.$

The discussion of Case A is essentially a duplication of that for Case 1 (in
the proof of Lemma \ref{mykey}), with a little difference, but we will write
it down for completeness.

Assume Case A occurs and let $h_{\theta,q_{1}^{\prime}}=\varphi_{q_{1}%
^{\prime}}^{-1}\circ\varphi_{\theta}\circ\varphi_{q_{1}^{\prime}},$ where
$\varphi_{q_{1}^{\prime}}$ is a rotation of $S$ moving $q_{1}^{\prime}$ to $0$
and $\varphi_{\theta}$ is the rotation $w\mapsto e^{-i\theta}w$ of $S.$ Then
the rotation $h_{\theta,q_{1}^{\prime}}\left(  c_{1}^{\prime}+c_{2}^{\prime
}\right)  $ of $c_{1}^{\prime}+c_{2}^{\prime}$ never meets any point of
$E_{q}$ other than $q_{1}^{\prime},$ and then we can obtain a contradiction as
in Case 1 (in the proof of Lemma \ref{mykey}). But notations here may have
different meaning. For example, here $T_{0}$ is the domain on $S$ enclosed by
$c_{1}^{\prime}+c_{2}^{\prime},\ $while in the proof of Lemma \ref{mykey},
$T_{0}$ is the disk enclosed by the circle $c_{1}.$ On the other hand one of
the key points in Case 1 of Lemma \ref{mykey} is (\ref{(1)-1}), which follows
from \ref{(l)}\label{lowerL}, but here \ref{(l)} may no longer hold for
$\Sigma_{\theta_{0}}^{i}\ $(the surfaces in (\ref{a103}) and (\ref{a104})).
But we can prove (\ref{(1)-1}) still holds after (\ref{a104})).

The interior angle of $T_{0}$ at $q_{1}^{\prime}$ and $q_{2}$ are both equal
to $\angle\left(  \Sigma_{L_{1}},a_{2}\right)  .$ For $\theta\in
(0,\angle\left(  \Sigma_{L_{1}},a_{2}\right)  )$, we introduce more
notations:
\[
T_{\theta}=h_{\theta,q_{1}^{\prime}}\left(  T_{0}\right)  ,T_{\theta}^{\prime
}=T_{0}\backslash\overline{T_{\theta}},T_{\theta}^{\prime\prime}=T_{\theta
}\backslash\overline{T_{0}},
\]%
\[
c_{1,\theta}^{\prime}=h_{\theta,q_{1}^{\prime}}\left(  c_{1}^{\prime}\right)
,c_{2,\theta}^{\prime}=h_{\theta,q_{1}^{\prime}}\left(  c_{2}^{\prime}\right)
,
\]%
\[
q_{2,\theta}=h_{\theta,q_{1}^{\prime}}\left(  q_{2}\right)  ,q_{2,\theta
}^{\prime}=c_{2,\theta}^{\prime\circ}\cap c_{1}^{\prime\circ}.
\]
Then $q_{2,\theta}^{\prime}$ gives the following partitions of $c_{1}^{\prime
}$ and $c_{2,\theta}^{\prime}:$%
\[
c_{1}^{\prime}=c_{11,\theta}^{\prime}+c_{12,\theta}^{\prime}=c_{1}^{\prime
}\left(  q_{1}^{\prime},q_{2,\theta}^{\prime}\right)  +c_{1}^{\prime}\left(
q_{2,\theta}^{\prime},q_{2}\right)  ,
\]%
\[
c_{2,\theta}^{\prime}=c_{21,\theta}^{\prime}+c_{22,\theta}^{\prime
}=c_{2,\theta}^{\prime}\left(  q_{2,\theta},q_{2,\theta}^{\prime}\right)
+c_{2,\theta}^{\prime}\left(  q_{2,\theta}^{\prime},q_{1}^{\prime}\right)  ;
\]
$\alpha_{1}^{\prime}$ has a partition
\[
\alpha_{1}^{\prime}=\alpha_{11,\theta}^{\prime}\left(  a_{1}^{\prime
},a_{2,\theta}^{\prime}\right)  +\alpha_{12,\theta}^{\prime}\left(
a_{2,\theta}^{\prime},a_{2}\right)  =\alpha_{1}^{\prime}\left(  a_{1}^{\prime
},a_{2,\theta}^{\prime}\right)  +\alpha_{1}^{\prime}\left(  a_{2,\theta
}^{\prime},a_{2}\right)
\]
such that
\begin{align*}
c_{11,\theta}^{\prime}\left(  q_{1}^{\prime},q_{2,\theta}^{\prime}\right)   &
=\left(  f,\alpha_{11,\theta}^{\prime}\left(  a_{1}^{\prime},a_{2,\theta
}^{\prime}\right)  \right)  ,\\
c_{12,\theta}^{\prime}\left(  q_{2,\theta}^{\prime},q_{2}\right)   &  =\left(
f,\alpha_{12,\theta}^{\prime}\left(  a_{2,\theta}^{\prime},a_{2}\right)
\right)  ;
\end{align*}
and $\alpha_{2}=\alpha_{2}(a_{2},a_{3})$ has a partition
\[
\alpha_{2}=\alpha_{2}^{\prime}\left(  a_{2},a_{1}^{\prime\prime}\right)
+\alpha_{2}^{\prime\prime}\left(  a_{1}^{\prime\prime},a_{3}\right)
=\alpha_{2}\left(  a_{2},a_{1}^{\prime\prime}\right)  +\alpha_{2}\left(
a_{1}^{\prime\prime},a_{3}\right)
\]
such that
\[
c_{2}^{\prime}=c_{2}^{\prime}\left(  q_{2},q_{1}^{\prime}\right)  =\left(
f,\alpha_{2}^{\prime}\left(  a_{2},a_{1}^{\prime\prime}\right)  \right)  .
\]

The reader should be aware of that $c_{1j,\theta}^{\prime}$ are subarcs of
$c_{1}^{\prime}$ (not $c_{1,\theta}^{\prime}$), but $c_{2j,\theta}^{\prime}$
are subarcs of $c_{2,\theta}^{\prime}$ (not $c_{2}^{\prime}$). It is clear
that by Lemma \ref{home2} (ii) and Claim \ref{B} when $\theta>0$ and $\theta$
is small enough, $f$ restricted to a neighborhood of $\alpha_{12,\theta
}^{\prime}+\alpha_{2}^{\prime}$ in $\overline{\Delta}$ is a homeomorphism,
since $\alpha_{12,\theta}^{\prime}+\alpha_{2}^{\prime}$ is a subarc in
$\alpha_{1}^{\prime}+\alpha_{2}^{\prime}$ which tends to $\alpha_{2}^{\prime}$
as $\theta\rightarrow0.$ Thus for small enough $\theta>0$ we have the
following claim similar to Claim \ref{(h)}:

\begin{claim}
\label{(h')}$\left(  \overline{T_{\theta}^{\prime}},c_{12,\theta}^{\prime
}+c_{2}^{\prime}\right)  $ is a simple closed Jordan domain of $\Sigma_{L_{1}}
$ such that $-c_{22,\theta}^{\prime\circ}=-c_{2,\theta}^{\prime\circ}\left(
q_{2,\theta}^{\prime},q_{1}^{\prime}\right)  \ $is the new boundary and
$c_{12,\theta}^{\prime}+c_{2}^{\prime}$ is the old boundary. That is to say,
there exist a Jordan domain $D_{\theta}^{\prime}\subset\Delta\ $and an arc
$\alpha_{22,\theta}^{\prime}=\alpha_{22,\theta}^{\prime}\left(  a_{2,\theta
}^{\prime},a_{1}^{\prime\prime}\right)  $ in $\overline{\Delta}$ with
$\alpha_{22,\theta}^{\prime\circ}\subset\Delta$, such that
\[
\partial D_{\theta}^{\prime}=\alpha_{12,\theta}^{\prime}\left(  a_{2,\theta
}^{\prime},a_{2}\right)  +\alpha_{2}^{\prime}\left(  a_{2},a_{1}^{\prime
\prime}\right)  -\alpha_{22,\theta}^{\prime}\left(  a_{2,\theta}^{\prime
},a_{1}^{\prime\prime}\right)  ,
\]%
\[
c_{22,\theta}^{\prime}\left(  q_{2,\theta}^{\prime},q_{1}^{\prime}\right)
=\left(  f,\alpha_{22,\theta}^{\prime}\left(  a_{2,\theta}^{\prime}%
,a_{1}^{\prime\prime}\right)  \right)  ,
\]
and that $f$ restricted to $\overline{D_{\theta}^{\prime}}$ is a homeomorphism
onto $\overline{T_{\theta}^{\prime}}.$
\end{claim}

Let $\theta_{0}$ be the maximal number in $(0,\angle\left(  \Sigma_{L_{1}%
},a_{2}\right)  ]$ such that all $\theta\in\left(  0,\theta_{0}\right)  $
satisfy Claim \ref{(h')}. Then by Claim \ref{c2}, $\theta_{0}<\angle\left(
\Sigma_{L_{1}},a_{2}\right)  .$ Repeating the the argument for Claims
\ref{(j)}--\ref{(d)}, with a little difference, we will show the following
Claims \ref{(j')}--\ref{(d')}:

\begin{claim}
\label{(j')}Except for that $\alpha_{22,\theta_{0}}^{\prime\circ}\subset
\Delta$ may fail, all other conclusions in Claim \ref{(h')} hold for
$\theta_{0}$: $\left(  \overline{T_{\theta_{0}}^{\prime}},c_{12,\theta_{0}%
}^{\prime}+c_{2}^{\prime}\right)  $ is still a simple closed Jordan domain of
$\Sigma_{L_{1}}$ such that $-c_{22,\theta_{0}}^{\prime\circ}$ contains the the
new boundary and $c_{12,\theta_{0}}^{\prime}+c_{2}^{\prime}$ is contained in
the old boundary. That is to say, there exist a Jordan domain $D_{\theta_{0}%
}^{\prime}\subset\Delta\ $and an arc $\alpha_{22,\theta_{0}}^{\prime}%
=\alpha_{22,\theta_{0}}^{\prime}\left(  a_{2,\theta_{0}}^{\prime}%
,a_{1}^{\prime\prime}\right)  $ in $\overline{\Delta},$ such that $\partial
D_{\theta_{0}}^{\prime}=\alpha_{12,\theta_{0}}^{\prime}+\alpha_{2}^{\prime
}-\alpha_{22,\theta_{0}}^{\prime},$ $c_{22,\theta_{0}}^{\prime}=\left(
f,\alpha_{22,\theta_{0}}^{\prime}\right)  ,$ and $f$ restricted to
$\overline{D_{\theta_{0}}^{\prime}}$ is a homeomorphism onto $\overline
{T_{\theta_{0}}^{\prime}}.$
\end{claim}

\begin{claim}
\label{(b')}$f$ has no branch value on $c_{22,\theta_{0}}^{\prime\circ}\ $and
thus each component of $\alpha_{22,\theta_{0}}^{\prime}\backslash
\partial\Delta$ has a neighborhood in $\overline{\Delta}$ on which $f$ is a homeomorphism.
\end{claim}

\begin{claim}
\label{(c')}$c_{22,\theta_{0}}^{\prime\circ}=\left(  f,\alpha_{22,\theta_{0}%
}^{\prime}\right)  $ has to intersect $\partial\Sigma_{L_{1}},$ say,
$\alpha_{22,\theta_{0}}^{\prime\circ}\cap\partial\Delta\neq\emptyset.$
\end{claim}

\begin{claim}
\label{(d')}$\alpha_{22,\theta_{0}}^{\prime\circ}\cap\partial\Delta\ $is a
nonempty finite set $\left\{  a_{i_{1}},\dots,a_{i_{k}}\right\}  $ in
$\{a_{j}\}_{j=1}^{m},$ $a_{1}^{\prime\prime},a_{i_{1}},\dots,a_{i_{k}%
},a_{2,\theta_{0}}^{\prime}\ $are arranged on $\partial\Delta$ anticlockwise
and divide $\alpha_{22,\theta_{0}}^{\prime\circ}$ into $k+1$ open arcs, each
of which has a neighborhood in $\overline{\Delta}$ on which $f$ is a homeomorphism.
\end{claim}

We repeat the argument for completeness. It is obvious by Claim \ref{(h')}
that $f^{-1}$ has a univalent branch $g$ defined on $\overline{T_{\theta_{0}%
}^{\prime}}\backslash c_{22,\theta_{0}}^{\prime\circ}=\cup_{\theta\in
(0,\theta_{0})}\overline{T_{\theta}^{\prime}}$ with $\alpha_{12,\theta_{0}%
}^{\prime}\left(  a_{2,\theta_{0}}^{\prime},a_{2}\right)  =g\left(
c_{12,\theta_{0}}^{\prime}\right)  \subset\alpha_{1}\ $and $\alpha_{2}%
^{\prime}=g\left(  c_{2}^{\prime}\right)  \subset\alpha_{2}.$ By Lemma
\ref{continue0}, $g$ can be extended to be a univalent branch of $f^{-1}$
defined on $\overline{T_{\theta_{0}}^{\prime}},$ and thus Claim \ref{(j')}
holds for $\overline{D_{\theta_{0}}^{\prime}}=g\left(  \overline{T_{\theta
_{0}}^{\prime}}\right)  =\overline{g\left(  T_{\theta_{0}}^{\prime}\right)  }$.

By (\ref{2022-01-03-1}), $c_{22,\theta_{0}}^{\prime\circ}\cap E_{q}%
=\emptyset,$ which together with Lemma \ref{home2} implies Claim \ref{(b')}.

Let $\Delta_{\theta_{0}}=\Delta\backslash\overline{D_{\theta_{0}}^{\prime}}.$
If $\alpha_{22,\theta_{0}}^{\prime\circ}\cap\partial\Delta=\emptyset,$ then
$\left(  f,\overline{\Delta_{\theta_{0}}}\right)  $ is a surface in
$\mathcal{F}\left(  L^{\prime},m+2\right)  ,\mathcal{\ }$with
\[
L^{\prime}=L-L(c_{12,\theta_{0}}^{\prime}+c_{2}^{\prime})+L(c_{22,\theta_{0}%
}^{\prime}),
\]
and thus by Claim \ref{(b')} and Lemma \ref{int-arg1}, for every small enough
$\rho>0$, $\left(  f,\overline{\Delta_{\theta_{0}}}\right)  $ contains the
simple and closed Jordan domain $\left(  \overline{T_{\theta_{0}+\rho}%
^{\prime}}\backslash T_{\theta_{0}}^{\prime},\left(  c_{12,\theta_{0}+\rho
}^{\prime}\backslash c_{12,\theta_{0}}^{\prime}\right)  +c_{22,\theta_{0}%
}^{\prime}\right)  $ such that $\left(  c_{12,\theta_{0}+\rho}^{\prime
}\backslash c_{12,\theta_{0}}^{\prime}\right)  +c_{22,\theta_{0}}^{\prime}$ is
the old boundary, say, an arc of $\left(  f,\partial\Delta_{\theta_{0}%
}\right)  ,$ and $c_{22,\theta_{0}+\rho}^{\prime\circ}$ is the new boundary.
Then we have that $\left(  \overline{T_{\theta_{0}}^{\prime}},c_{12,\theta
_{0}}^{\prime}+c_{2}^{\prime}\right)  $ can be extended to a larger simple
closed Jordan domain $\left(  \overline{T_{\theta_{0}+\rho}^{\prime}%
},c_{12,\theta_{0}+\rho}^{\prime}+c_{2}^{\prime}\right)  $ of $\Sigma_{L_{1}%
},$ and that Claim \ref{(h')} holds for all $\theta\in\lbrack0,\theta_{0}%
+\rho),$ contradicting the maximal property of $\theta_{0}.$ Thus Claim
\ref{(c')} holds.

It is clear that $\alpha_{22,\theta_{0}}^{\prime}\cap\left(  \alpha
_{12,\theta_{0}}^{\prime}+\alpha_{2}^{\prime}\right)  =\{a_{2,\theta_{0}%
}^{\prime},a_{1}^{\prime\prime}\},$ which together with that $c_{22,\theta
_{0}}^{\prime}$ is strictly convex and Lemma \ref{tangent}, implies that
$\alpha_{22,\theta_{0}}^{\prime\circ}\cap\left(  \left(  \partial
\Delta\right)  \backslash\left(  \alpha_{12,\theta_{0}}^{\prime}+\alpha
_{2}^{\prime}\right)  \right)  $ is a subset of $\{a_{j}\}_{j=1}^{m},$ and so
is $\alpha_{22,\theta_{0}}^{\prime\circ}\cap\partial\Delta.$ This, together
with Claims \ref{(b')} and \ref{(c')}, implies Claim \ref{(d')}.

For simplicity, we assume that $\alpha_{22,\theta_{0}}^{\prime\circ}%
\cap\partial\Delta=\{a_{i_{1}}\}$ is a singleton. Then $q_{i_{1}}$ gives
partitions
\[
c_{22,\theta_{0}}^{\prime}=c_{221,\theta_{0}}^{\prime}\left(  q_{2,\theta_{0}%
}^{\prime},q_{i_{1}}\right)  +c_{222,\theta_{0}}^{\prime}\left(  q_{i_{1}%
},q_{1}^{\prime}\right)  ,
\]%
\[
c_{2,\theta_{0}}^{\prime}=\mathfrak{c}_{21,\theta_{0}}\left(  q_{2,\theta_{0}%
},q_{i_{1}}\right)  +c_{222,\theta_{0}}^{\prime}\left(  q_{i_{1}}%
,q_{1}^{\prime}\right)  ,
\]
where
\[
\mathfrak{c}_{21,\theta_{0}}=c_{21,\theta_{0}}^{\prime}\left(  q_{2,\theta
_{0}},q_{2,\theta_{0}}^{\prime}\right)  +c_{221,\theta_{0}}^{\prime}\left(
q_{2,\theta_{0}}^{\prime},q_{i_{1}}\right)  .
\]
We can cut $\left(  \overline{T_{\theta_{0}}^{\prime}}\backslash
c_{22,\theta_{0}}^{\prime},c_{2,\theta_{0}}^{\prime}\right)  ,$ the simple
closed Jordan domain of $\Sigma_{L_{1}}$ with new boundary $c_{2,\theta_{0}%
}^{\prime\circ},$ from $\Sigma_{L_{1}}$ and sew\label{23} $\left(
\overline{T_{\theta_{0}}^{\prime\prime}},c_{11,\theta_{0}}^{\prime}\right)  ,$
to $\Sigma_{L_{1}}\backslash\left(  \overline{T_{\theta_{0}}^{\prime}%
}\backslash c_{22,\theta_{0}}^{\prime},c_{2,\theta_{0}}^{\prime}\right)  $
along $c_{11,\theta_{0}}^{\prime}=c_{1}\cap\overline{T_{\theta_{0}}},$ to
obtain two surfaces $\Sigma_{\theta_{0}}^{1}$ and $\Sigma_{\theta_{0}}^{2},$
linked at $q_{i_{1}},$ such that%
\begin{equation}
\partial\Sigma_{\theta_{0}}^{1}=\left(  c_{1}\backslash c_{1}^{\prime}\right)
+c_{1,\theta_{0}}^{\prime}+\mathfrak{c}_{21,\theta_{0}}+c_{i_{1}}+\cdots
+c_{m}, \label{a103}%
\end{equation}%
\begin{equation}
\partial\Sigma_{\theta_{0}}^{2}=c_{222,\theta_{0}}^{\prime}+\left(
c_{2}\backslash c_{2}^{\prime}\right)  +c_{3}+\cdots+c_{i_{1}-1}. \label{a104}%
\end{equation}

It is clear that the total number of terms in the above two partitions is
$m+3,$ $q_{i_{1}}$ is contained in $T_{0}$, and $q_{1},q_{2,\theta_{0}},q_{3}$
are outside $T_{0}$ since $T_{0}$ is convex. Therefore $i_{1}\neq1,2,3,\ $and
$i_{1}\geq4$, say, the first partition contains at least four terms and the
second partition contains at least three terms, which implies that each of the
partitions has at most $m$ terms. Hence the above two partitions are both
$\mathcal{F}\left(  L,m\right)  $ partitions, by Claim \ref{(b')}. It is clear
that here (\ref{a105}) still holds and, by (\ref{2022-01-03-1}), we have%
\[
\overline{n}\left(  \Sigma_{L_{1}}\right)  =\sum_{j=1}^{2}\overline{n}\left(
\Sigma_{\theta_{0}}^{j}\right)  .
\]
Then $\Sigma_{L_{1}}$ is decomposable in $\mathcal{F}\left(  L,m\right)  ,$
contradicting Lemma \ref{undec}, and (iii) is proved in Case A.

Now we assume Case B occurs. By (\ref{2022-01-03-2}), Claim \ref{B} and the
assumption of Case B, both $c_{1}^{\prime}$ and $c_{2}^{\prime}$ are strictly
convex. Then we may further assume
\[
L(c_{1}^{\prime})\geq L(c_{2}^{\prime}).
\]
Let%
\[
l=L(c_{1}^{\prime})+L(c_{2}^{\prime}),
\]
and let $T_{0}$ be still the domain enclosed by $c_{1}^{\prime}+c_{2}^{\prime
}.$ Let $I=\overline{q_{1}^{\prime}q_{2}}$ and let $C$ be the strictly convex
circle passing through $q_{1}^{\prime}$ and $q_{2}$ whose length is $l\ $and
whose arc $c_{x_{0}}$ from $q_{1}^{\prime}$ to $q_{2}$ is longer than its
complementary, with $L(c_{x_{0}})=x_{0},$ and let $c_{l-x_{0}}^{\prime
}=C\backslash c_{x_{0}}^{\circ}.$ Then $c_{x_{0}}$ is on the right hand side
of the great circle determined by $I=\overline{q_{1}^{\prime}q_{2}}.$ Recall
Definition \ref{lune-lens} of lens and let
\[
\mathfrak{D}_{x}=\mathfrak{D}\left(  I,x,l-x\right)  =\mathfrak{D}\left(
I,c_{x},c_{l-x}^{\prime}\right)
\]
be the lens with $\partial\mathfrak{D}_{x}=c_{x}-c_{l-x}^{\prime},$ where
$c_{x}=c_{x}\left(  q_{1}^{\prime},q_{2}\right)  $ and $c_{l-x}^{\prime}%
=c_{x}^{\prime}\left(  q_{2},q_{1}^{\prime}\right)  $ are convex circular arcs
with $L(c_{x})=x$ and $L(c_{l-x}^{\prime})=l-x.$ By Corollary
\ref{2-curvature2}, we have the following Claim \ref{(c} and \ref{(d}.

\begin{claim}
\label{(c} The area $A\left(  \mathfrak{D}\left(  I,x,l-x\right)  \right)  $
strictly increases for $x\in\lbrack l/2,x_{0}].$
\end{claim}

\begin{claim}
\label{(d} The lune $\mathfrak{D}_{x}^{\prime}=\mathfrak{D}^{\prime}%
(I,c_{x})=\mathfrak{D}^{\prime}(I,x)$ strictly increases, and the lune
$\mathfrak{D}_{x}^{\prime\prime}=\mathfrak{D}^{\prime}(-I,c_{l-x}^{\prime
})=\mathfrak{D}^{\prime}(-I,l-x)$ strictly decreases, for all $x\in\lbrack
l/2,x_{0}]$ (see Definition \ref{lune-lens} for the notation $\mathfrak{D}%
^{\prime}(I,\cdot)$). That is to say, $\overline{\mathfrak{D}_{x}^{\prime}%
}\backslash I\subset\mathfrak{D}_{x^{\prime}}^{\prime}\ $and $\overline
{\mathfrak{D}_{x^{\prime}}^{\prime\prime}}\backslash I\subset\mathfrak{D}%
_{x}^{\prime\prime}$ when $l/2\leq x<x^{\prime}\leq x_{0}.$
\end{claim}

Since $c_{1}^{\prime}$ and $c_{2}^{\prime}$ are the circular arcs with the
same endpoints, we have

\begin{claim}
\label{(e}For any circular arc $\gamma$ contained in $\overline{T_{0}}$ from
$q_{1}^{\prime}$ to $q_{2},$ $q_{1}\in\gamma$ if and only if $\gamma$ is
contained in the circle determined by $c_{1}$ (three points determine a unique
circle on $S$).
\end{claim}

Assume $L(c_{1}^{\prime})=x_{0}^{\prime}.$ Since we assumed $L(c_{1}^{\prime
})\geq L(c_{2}^{\prime}),$ we have $x_{0}^{\prime}\in\lbrack l/2,x_{0}].$ For
$x\in(x_{0}^{\prime},x_{0}]$ let
\[
T_{x}^{\prime}=\mathfrak{D}_{x_{0}^{\prime}}^{\prime\prime}\backslash
\overline{\mathfrak{D}_{x}^{\prime\prime}}\mathrm{\ and\ }T_{x}^{\prime\prime
}=\mathfrak{D}_{x}^{\prime}\backslash\overline{\mathfrak{D}_{x_{0}^{\prime}%
}^{\prime}}.
\]
Then by Lemma \ref{int-arg1} we have the following result similar to Claims
\ref{(h)} and \ref{(h')}:

\begin{claim}
\label{(h'')}For every $x\in(x_{0}^{\prime},x_{0}]$ so that $x-x_{0}^{\prime}
$ is small enough, there exist a simple arc $\alpha_{l-x}^{\prime}%
=\alpha_{l-x}^{\prime}\left(  a_{2},a_{1}^{\prime\prime}\right)  $ in
$\overline{\Delta},$ with $\alpha_{l-x}^{\prime\circ}\subset\Delta$ and
$c_{l-x}^{\prime}=\left(  f,\alpha_{l-x}^{\prime}\right)  ,$ and a Jordan
domain $D_{x}^{\prime}$ in $\Delta$ with $\partial D_{x}^{\prime}=\alpha
_{2}^{\prime}-\alpha_{l-x}^{\prime}$, such that $f$ restricted to
$\overline{D_{x}^{\prime}}$ is a homeomorphism onto $\overline{T_{x}^{\prime}%
}\ $($\alpha_{2}^{\prime}$ is defined just before (\ref{2022-01-03-1})). In
other words, for each $x\in(x_{0}^{\prime},x_{0}]$ so that $x-x_{0}^{\prime}$
is small enough, $\left(  \overline{T_{x}^{\prime}},c_{2}^{\prime}\right)  $
is a simple closed Jordan domain of $\Sigma_{L_{1}}$ with new boundary
$c_{l-x}^{\prime\circ}$ and old boundary $c_{2}^{\prime}=c_{l-x_{0}^{\prime}%
}^{\prime}.$
\end{claim}

Then for every $x$ satisfying Claim \ref{(h'')}, we can cut $\left(
\overline{T_{x}^{\prime}},c_{l-x_{0}^{\prime}}^{\prime}\right)  $ from
$\Sigma_{L_{1}}$ and sew\label{sew22} $\left(  \overline{T_{x}^{\prime\prime}%
},c_{x}\right)  $ to $\Sigma_{L_{1}}\backslash\left(  \overline{T_{x}^{\prime
}},c_{l-x_{0}^{\prime}}^{\prime}\right)  $ along $c_{x_{0}^{\prime}}$ to
obtain a surface $\Sigma_{x}=\left(  f_{x},\overline{\Delta}\right)  $ in
$\mathcal{F}_{r}\left(  L,m+2\right)  .$

It is clear that there exists a maximum $x^{\ast}\in(x_{0}^{\prime},x_{0}]$
such that for every $x\in\lbrack x_{0}^{\prime},x^{\ast}),$ $\Sigma_{x}$ is a
well defined surface in $\mathcal{F}_{r}\left(  L,m+2\right)  $. Then either
$x^{\ast}=x_{0}$ or $x^{\ast}<x_{0}.$ As the argument for Claims \ref{(j')}
and \ref{(b')}, with a little difference, we can show the following Claims
\ref{(j'')} and \ref{(b'')}:

\begin{claim}
\label{(j'')}Except for that $\alpha_{l-x^{\ast}}^{\prime\circ}\subset\Delta$
may fail, all other conclusions of Claim \ref{(h'')} hold for $x^{\ast}:$
There exist a simple arc $\alpha_{l-x^{\ast}}^{\prime}=\alpha_{l-x^{\ast}%
}^{\prime}\left(  a_{2},a_{1}^{\prime\prime}\right)  $ in $\overline{\Delta},$
with $c_{l-x^{\ast}}^{\prime}=\left(  f,\alpha_{l-x^{\ast}}^{\prime}\right)
,$ and a Jordan domain $D_{x^{\ast}}^{\prime}$ in $\Delta$ with $\partial
D_{x^{\ast}}^{\prime}=\alpha_{2}^{\prime}-\alpha_{l-x^{\ast}}^{\prime},$ such
that $f$ restricted to $\overline{D_{x^{\ast}}^{\prime}}$ is a homeomorphism
onto $\overline{T_{x^{\ast}}^{\prime}}.$ In other words, $\left(
\overline{T_{x^{\ast}}^{\prime}},c_{2}^{\prime}\right)  $ is still a simple
closed Jordan domain of $\Sigma_{L_{1}}$ with the new boundary \emph{contained
in} $c_{1-x^{\ast}}^{\prime\circ}$ and the old boundary \emph{containing}
$c_{l-x_{0}^{\prime}}^{\prime}.$
\end{claim}

\begin{claim}
\label{(b'')}$f$ has no branch value on $c_{l-x^{\ast}}^{\prime\circ}\ $and
thus $\alpha_{l-x^{\ast}}^{\prime\circ}\backslash\partial\Delta$ has a
neighborhood in $\overline{\Delta}$ on which $f$ is a homeomorphism.
\end{claim}

By Claim \ref{(h'')} $f^{-1}$ has a univalent branch $g$ defined on
$\overline{T_{x^{\ast}}^{\prime}}\backslash\alpha_{l-x^{\ast}}^{\prime\circ
}=\cup_{x\in(x_{0}^{\prime},x^{\ast})}\overline{T_{x}^{\prime}}$ with
$g\left(  c_{2}^{\prime}\right)  =g\left(  c_{l-x_{0}^{\prime}}^{\prime
}\right)  =\alpha_{2}^{\prime}.$ By Lemma \ref{continue0}, $g$ can be extended
to $\overline{T_{x^{\ast}}^{\prime}},$ and then Claim \ref{(j'')} follows.
Claim \ref{(b'')} is obvious, since $c_{l-x^{\ast}}^{\prime\circ}\cap
E_{q}=\emptyset$, both $\alpha_{l-x^{\ast}}^{\prime}$ and $c_{l-x^{\ast}%
}^{\prime}$ are simple arcs and $c_{l-x^{\ast}}$ is circular.

Now, there are only three possibilities:

\noindent\textbf{Case BA. }$x^{\ast}<x_{0}.$

\noindent\textbf{Case BB. }$x^{\ast}=x_{0}\ $and $\alpha_{l-x^{\ast}}%
^{\prime\circ}\cap\partial\Delta\neq\emptyset.$

\noindent\textbf{Case BC. }$x^{\ast}=x_{0}\ $and $\alpha_{l-x^{\ast}}%
^{\prime\circ}\cap\partial\Delta=\emptyset.$

Assume Case BA occurs. Then as the discussion for Claims \ref{(c)} and
\ref{(d)}, we can show

\begin{claim}
\label{(c'')}$c_{l-x^{\ast}}^{\prime\circ}=\left(  f,\alpha_{l-x^{\ast}%
}^{\prime\circ}\right)  $ has to intersect $\partial\Sigma_{L_{1}},$ say,
$\alpha_{l-x^{\ast}}^{\prime\circ}\cap\partial\Delta\neq\emptyset.$
\end{claim}

\begin{claim}
\label{(d'')}$\alpha_{l-x^{\ast}}^{\prime\circ}\cap\partial\Delta\ $is a
nonempty finite set $\left\{  a_{i_{1}},\dots,a_{i_{k}}\right\}  $ in
$\{a_{j}\}_{j=1}^{m},$ $a_{1}^{\prime\prime},a_{i_{1}},\dots,a_{i_{k}}%
,a_{2}\ $are arranged on $\partial\Delta$ anticlockwise and divide
$\alpha_{l-x^{\ast}}^{\prime\circ}$ into $k+1$ open arcs, each of which has a
neighborhood in $\overline{\Delta}$ on which $f$ is a homeomorphism.
\end{claim}

But the proof of Claim \ref{(c'')} is simpler: Let $\Delta_{x^{\ast}}%
=\Delta\backslash\overline{D_{x^{\ast}}^{\prime}}.$ If $\alpha_{l-x^{\ast}%
}^{\prime\circ}\cap\partial\Delta=\emptyset$, then by Claim \ref{(b'')},
$\left(  f,\overline{\Delta_{x^{\ast}}}\right)  $ is a surface in
$\mathcal{F}\left(  L^{\prime},m+1\right)  \ $with $L^{\prime}=L_{1}%
-L(c_{2}^{\prime})+L(c_{l-x^{\ast}}^{\prime}),$ and thus by Lemma
\ref{int-arg1}, $f$ restricted to a neighborhood of $\alpha_{l-x^{\ast}%
}^{\prime}$ in $\overline{\Delta_{x^{\ast}}}$ is a homeomorphism, in other
words, $\left(  f,\overline{\Delta_{x^{\ast}}}\right)  $ contains a simple
closed Jordan domain $\left(  \overline{K_{\varepsilon}},c_{l-x^{\ast}%
}^{\prime}\right)  $ with old boundary $c_{l-x^{\ast}}^{\prime},$ say,
$c_{l-x^{\ast}}^{\prime}$ is an arc of $\left(  f,\partial\Delta_{x^{\ast}%
}\right)  ,$ where $K_{\varepsilon}=\mathfrak{D}_{x^{\ast}}^{\prime\prime
}\backslash\overline{\mathfrak{D}_{x^{\ast}+\varepsilon}^{\prime\prime}}$ for
every small enough $\varepsilon>0.$ Then $\left(  \overline{T_{x^{\ast
}+\varepsilon}^{\prime}},c_{2}^{\prime}\right)  =\left(  \overline{T_{x^{\ast
}}^{\prime}}\cup\overline{K_{\varepsilon}},c_{2}^{\prime}\right)  $ is a
closed and simple Jordan domain of $\Sigma_{L_{1}},$ contradicts the
maximality of $x^{\ast}.$ Thus Claim \ref{(c'')} holds.

It is clear that $\alpha_{l-x^{\ast}}^{\prime}\cap\alpha_{1}^{\prime}%
=\{a_{2},a_{1}^{\prime\prime}\}$ and, on the other hand, $c_{l-x^{\ast}%
}^{\prime}$ is strictly convex, since $x_{0}^{\prime}<x^{\ast}<x_{0}.$
Therefore, by Lemma \ref{tangent}, we have $\alpha_{l-x^{\ast}}^{\prime\circ
}\cap\partial\Delta\subset\{a_{j}\}_{j=1}^{m}.$ Thus $\alpha_{l-x^{\ast}%
}^{\prime\circ}\cap\partial\Delta$ is consisted of some points of
$\{a_{j}\}_{j=1}^{m},$ and then Claim (\ref{(d'')} holds.

When $\alpha_{2,\theta_{0}}^{\prime\circ}\cap\partial\Delta=\{a_{i_{1}}%
\}\in\{a_{j}\}_{j=1}^{m}\ $is a singleton, we can obtain a contradiction as in
Case A, and the same argument applies to Case BB. When $\alpha_{2,\theta_{0}%
}^{\prime\circ}\cap\partial\Delta$ contains more than one point, the argument
is similar. We have obtained a contradiction in Cases BA and Case BB.

Assume Case BC occurs. Then $c_{x^{\ast}}+c_{l-x^{\ast}}^{\prime}=c_{x_{0}%
}+c_{2}^{\prime}$ is the circle $C$ and it is clear that $\Sigma_{x_{0}%
}=\Sigma_{x^{\ast}}$ is a well defined surface with
\begin{equation}
L(\partial\Sigma_{x_{0}})=L(\partial\Sigma_{L_{1}}). \label{(m)}%
\end{equation}
For all $x\in\lbrack0,\theta_{0}],$ since $L(c_{x}+c_{l-x})=l<\frac
{\delta_{E_{q}}}{2},$ (\ref{2022-01-03-1}) implies that $\left(  c_{x}%
+c_{l-x}\right)  \backslash\{q_{1}^{\prime}\}$ never meets $E_{q},$ and then
we have $\overline{n}\left(  \Sigma_{x_{0}}\right)  =\overline{n}\left(
\Sigma_{L_{1}}\right)  .$ On the other hand, by Claim \ref{(c}, we have
$A(\Sigma_{x_{0}})>A(\Sigma_{L_{1}}).$ Therefore
\begin{equation}
R(\Sigma_{x_{0}})>R(\Sigma_{L_{1}}), \label{x15}%
\end{equation}
and moreover $\partial\Sigma_{x_{0}}$ has the partition%
\begin{equation}
\partial\Sigma_{x_{0}}=C_{1}+C_{2}+C_{3}+c_{3}+\cdots+c_{m}, \label{(k)}%
\end{equation}
where $C_{1}=c_{1}\backslash c_{1}^{\prime},C_{2}=c_{x_{0}}^{\prime
}+c_{l-x_{0}}^{\prime},C_{3}=c_{2}\backslash c_{2}^{\prime}.$ It is clear that
$C_{2}$ is a simple circle such that $C_{2}^{\circ}$ is the old boundary of a
simple domain of $\Sigma_{x_{0}}.$Thus (\ref{(k)}) is an $\mathcal{F}\left(
L,m+1\right)  $ partition of $\partial\Sigma_{x_{0}}.$ It is clear that
$\Sigma_{x_{0}}$ has no branch value outside $E_{q},$ thus
\begin{equation}
\Sigma_{x_{0}}\in\mathcal{F}_{r}\left(  L,m+1\right)  . \label{m+1}%
\end{equation}
It is clear that
\begin{align}
L(C_{1})+L(C_{2})  &  =L(c_{1}\backslash c_{1}^{\prime})+l=L(c_{1}\backslash
c_{1}^{\prime})+L(c_{1}^{\prime})+L(c_{2}^{\prime})\label{ll}\\
&  <L(c_{1})+L(c_{2})<\frac{\delta_{E_{q}}}{2}.\nonumber
\end{align}

By (\ref{2022-01-03-1}), (\ref{(k)}) and (\ref{m+1}), we can repeat the
argument for Case 1 to show that $\Sigma_{x_{0}},$ as a surface in
$\mathcal{F}_{r}\left(  L,m+1\right)  $ is decomposable in $\mathcal{F}\left(
L,m\right)  $ (in Case 1 we in fact proved $\Sigma_{L_{1}}\in\mathcal{F}%
\left(  L,m\right)  $ is decomposable in $\mathcal{F}\left(  L,m-1\right)  $,
by (\ref{(l)})). This implies that $\Sigma_{L_{1}}$ is also decomposable in
$\mathcal{F}\left(  L,m\right)  $ by (\ref{(m)}) and (\ref{x15}). But this
contradicts Lemma \ref{undec}, and thus Case BC can't occur. We have proved
(iii) in any case and the lemma has been proved completely.
\end{proof}

Now we can easily prove Theorem \ref{cat2}.

\begin{proof}
[\textbf{Proof of Theorem \ref{cat2}}]It is clear that there are at most
$\left[  L/\left(  \delta_{E_{q}}/4\right)  \right]  +1=\left[  4L/\delta
_{E_{q}}\right]  +1$ terms in (\ref{zz1}) which have length $\geq\delta
_{E_{q}}/4.$ Thus, we may assume, after a permutation of the subscripts like
$\left(  1,2,\dots,m\right)  \mapsto\left(  j_{0},j_{0}+1,\dots,m,1,2,\dots
,j_{0}-1\right)  ,$ that
\begin{equation}
L(c_{j})<\delta_{E_{q}}/4,j=1,2. \label{(aa)}%
\end{equation}
Then, by Lemma \ref{mykey}, we have the following

(A) Neither $c_{1}$ nor $c_{2}$ is closed.

Assume $\Sigma_{L_{1}}\notin\mathcal{F}_{r}\left(  L,m-1\right)  .$ For $j=1$
or $2,$ if $c_{j}$ is contained in $\mathfrak{C}^{2},$ then by Lemma
\ref{Eq-end} we have $\partial c_{j}=E_{q}\ $and thus $L(c_{1})\geq
\delta_{E_{q}},$ which contradicts (\ref{(aa)}). Thus neither $c_{1}$ nor
$c_{2}$ is contained in $\mathfrak{C}^{2},$ which, together with (A), implies
$c_{j}^{\circ}\cap E_{q}\neq\emptyset$ for $j=1$ and $2.$ Then by (\ref{(aa)})
we have
\[
c_{1}\cap E_{q}=c_{2}\cap E_{q}=c_{1}^{\circ}\cap E_{q}=c_{2}^{\circ}\cap
E_{q}=\{q_{1}^{\prime}\}
\]
for some $q_{1}^{\prime}\in E_{q}.$ This contradicts Lemma \ref{mykey2} (iii).
This contradiction comes from Condition \ref{nom-1} which assumes
$\Sigma_{L_{1}}\notin\mathcal{F}_{r}\left(  L,m-1\right)  $, and so Condition
\ref{nom-1} can't be satisfied. Thus we have $\Sigma_{L_{1}}\in\mathcal{F}%
_{r}\left(  L,m-1\right)  $, and Theorem \ref{cat2} is proved.
\end{proof}

\section{\label{very good}Proof of Theorem \ref{main1}}

We first prove the following result.

\begin{lemma}
\label{inhalf1}Let $\Sigma$ be a surface of $\mathcal{F}\left(  L\right)  $
and assume that $\partial\Sigma$ has a partition $\partial\Sigma=\Gamma+C$
such that $C=C\left(  p_{1},p_{2}\right)  $ is a simple circular arc with
$C\cap E_{q}\subset\{p_{1},p_{2}\}$ (it is permitted that $p_{1}=p_{2}$). If
$C\left(  p_{1},p_{2}\right)  $ cannot be contained in any open hemisphere on
$S $, then there exists a surface $\Sigma^{\prime}$ in $\mathcal{F}\left(
L\right)  $ such that $\partial\Sigma^{\prime}$ has the partition
\begin{equation}
\partial\Sigma^{\prime}=\Gamma+\gamma, \label{ag41}%
\end{equation}
where $\gamma$ is a simple polygonal path $\gamma=\overline{p_{1}%
\mathfrak{a}_{1}\mathfrak{a}_{2}\dots\mathfrak{a}_{s}p_{2}}$ with
\begin{equation}
\gamma^{\circ}\cap E_{q}=\{\mathfrak{a}_{1},\mathfrak{a}_{2},\dots
,\mathfrak{a}_{s}\}, \label{ag36}%
\end{equation}%
\begin{equation}
L(\gamma)\leq L(C), \label{ag37}%
\end{equation}%
\begin{equation}
H(\Sigma^{\prime})>H(\Sigma), \label{ag40}%
\end{equation}%
\begin{equation}
\frac{R(\Sigma^{\prime})+4\pi}{L(\partial\Sigma^{\prime})}>\frac
{R(\Sigma)+4\pi}{L(\partial\Sigma)}, \label{ag38}%
\end{equation}
and
\begin{equation}
\gamma^{\circ}\cap E_{q}\neq\emptyset\mathrm{\ if\ }d\left(  p_{1}%
,p_{2}\right)  =\pi. \label{ag39}%
\end{equation}

\end{lemma}

\begin{proof}
We first show the following claim.

\begin{claim}
\label{ag30}There exists a closed polygon $T$ on $S$ with
\[
\partial T=-C+\gamma
\]
such that $\gamma$ satisfies (\ref{ag36}), (\ref{ag37}) and (\ref{ag39}), and
moreover,%
\begin{equation}
\left(  T\backslash\gamma\right)  \cap E_{q}=\left(  T^{\circ}\cup C^{\circ
}\right)  \cap E_{q}=\emptyset. \label{ag47}%
\end{equation}

\end{claim}

First assume $C=C\left(  p_{1},p_{2}\right)  $ is not contained in any open
hemisphere on $S$. Then $C\left(  p_{1},p_{2}\right)  $ is half, or a major
arc, of a great circle $c$ on $S,$ oriented by $C$.

Assume $C\left(  p_{1},p_{2}\right)  $ is half of a great circle on $S$. Then
$p_{1}$ and $p_{2}$ are antipodal. Let $T$ be the largest closed biangular
domain on $S$ so that, $\partial T=-C+\gamma$, $-C$ is one of the two edges of
$T,$ say, $T$ is on the right hand side of $C,$ $\gamma$ is the other edge of
$T$ from $p_{1}$ to $p_{2}$, and (\ref{ag47}) holds. Since $\#E_{q}=q\geq3,$
we have $\gamma\cap C=\{p_{1},p_{2}\}$ and $\gamma^{\circ}\cap E_{q}%
\neq\emptyset.$ Then $T$ and $\gamma$ satisfies Claim \ref{ag30} with
$L(\gamma)=L(C).$

Assume that $C\left(  p_{1},p_{2}\right)  $ is a major arc of the great circle
$c$ and let $S^{\prime}$ be the closed hemisphere enclosed by $-c$. Then
$d\left(  p_{1},p_{2}\right)  <\pi.$ Since $C\cap E_{q}\subset\{p_{1}%
,p_{2}\},$ the convex hull $K$ of $\left(  S^{\prime}\cap E_{q}\right)
\cup\{p_{1},p_{2}\}\ $is a polygon in $S^{\prime}$ with
\[
\overline{p_{1}p_{2}}=\left(  \partial K\right)  \cap c\subset K\cap
S^{\prime}=K\subset S^{\prime\circ}\cup\overline{p_{1}p_{2}}%
\]
and the vertices of $K$ are all in $E_{q},$ except the two points $p_{1}$ and
$p_{2}.$ But $K$ is just the line segment $\overline{p_{1}p_{2}}$ on $S$ when
$S^{\prime\circ}\cap E_{q}=\emptyset.$ Then we have two possibilities to discuss.

First consider the case $S^{\prime\circ}\cap E_{q}=\emptyset$ and let
$T=S^{\prime}.$ Then $K=\overline{p_{1}p_{2}},$ $T$ and $\gamma=\overline
{p_{1}p_{2}}=\overline{p_{1}\mathfrak{a}_{1}\dots\mathfrak{a}_{s}p_{2}}$ with
$\{\mathfrak{a}_{1},\dots,\mathfrak{a}_{s}\}=\gamma^{\circ}\cap E_{q}$
satisfies Claim \ref{ag30}. It is possible that $\gamma^{\circ}\cap
E_{q}=\emptyset.$

Second consider the case $S^{\prime\circ}\cap E_{q}\neq\emptyset.$ In this
case $K^{\circ}$ is a Jordan domain with $\partial K=-\gamma+\overline
{p_{1}p_{2}}$, $\gamma=\overline{p_{1}\mathfrak{a}_{1}\dots\mathfrak{a}%
_{s}p_{2}}$ and
\[
\{\mathfrak{a}_{1},\dots,\mathfrak{a}_{s}\}=\gamma^{\circ}\cap E_{q}=\left[
\left(  \partial K\right)  \backslash\overline{p_{1}p_{2}}\right]  \cap
E_{q}\neq\emptyset.
\]
On the other hand, $\gamma$ is a concave polygonal path in $S^{\prime}$ whose
two endpoints are on $\partial S^{\prime}=-c,$ which implies $L(\gamma)<L(C).$
Then $T=\left(  S^{\prime}\backslash K\right)  \cup\gamma$ satisfies Claim
\ref{ag30}.

By Claim \ref{ag30}, we can sew\label{sew11} $\Sigma$ and $T$ along $C$ to
obtain a surface $\Sigma^{\prime}$ satisfying (\ref{ag41})--(\ref{ag37}) and
(\ref{ag39}). On the other hand, (\ref{ag47}) implies
\[
\overline{n}\left(  \Sigma^{\prime}\right)  =\overline{n}\left(
\Sigma\right)  +\overline{n}(T^{\circ})+\#C^{\circ}\cap E_{q}=\overline
{n}\left(  \Sigma\right)
\]
and we have
\[
A(\Sigma^{\prime})=A(\Sigma)+A(T)>A(\Sigma).
\]
Therefore we have (\ref{ag40}) and (\ref{ag38}). It is clear that
$\Sigma^{\prime}\in\mathcal{F}\left(  L\right)  $. We have proved the lemma completely.
\end{proof}

Lemma \ref{inhalf1} has a direct corollary:

\begin{corollary}
\label{inhalf}Let $\Sigma$ be an extremal surface of $\mathcal{F}\left(
L\right)  $ and assume that $\partial\Sigma$ contains an arc $C=C\left(
p_{1},p_{2}\right)  $ such that $C$ is an SCC arc with $C\cap E_{q}%
\subset\{p_{1},p_{2}\}$ (it is permitted that $p_{1}=p_{2}$), then $C$ is
contained in some open hemisphere on $S.$
\end{corollary}

Instead of proving Theorem \ref{main1}, we prove the following theorem which
implies Theorem \ref{main1} directly.

\begin{theorem}
\label{okok}Let $L\in\mathcal{L}$ be given. Then the following conclusions
(A)--(C) hold.

(A) There exists a precise extremal surface of $\mathcal{F}_{r}\left(
L\right)  ,$ and there exists a positive integer $m_{0}=m_{0}\left(
L,q\right)  ,$ depending only on $L$ and $q,$ such that every precise extremal
surface of $\mathcal{F}_{r}\left(  L\right)  $ is precise extremal in
$\mathcal{F}\left(  L\right)  ,\mathcal{F}_{r}\left(  L,m\right)  ,$ and
$\mathcal{F}\left(  L,m\right)  ,$ respectively, for every integer $m\geq
m_{0}.$

(B) For any precise extremal surface $\Sigma_{0}=\left(  f_{0},\overline
{\Delta}\right)  $ of $\mathcal{F}_{r}\left(  L\right)  ,$ there exists a
positive integer $n_{0}$ such that $\partial\Sigma_{0}$ has an $\mathcal{F}%
\left(  L,n_{0}\right)  $-partition
\begin{equation}
\partial\Sigma_{0}=C_{1}\left(  q_{1},q_{2}\right)  +C_{2}\left(  q_{2}%
,q_{3}\right)  +\cdots+C_{n_{0}}\left(  q_{n_{0}},q_{1}\right)  \label{ac5}%
\end{equation}
satisfying the following (B1)--(B4):

(B1) If $n_{0}>1,$ then, for $j=1,2,\dots,n_{0},$ $C_{j}^{\circ}\cap
E_{q}=\emptyset\ $and $\partial C_{j}=\{q_{j},q_{j+1}\}\subset E_{q}.$ If
$n_{0}=1,$ then either $C_{1}\cap E_{q}=\emptyset$ or $C_{1}\cap E_{q}$ is the
singleton $\{q_{1}\}.$

(B2) Each $C_{j}$ is contained in an open hemisphere $S_{j}$ on $S,$
$j=1,2,\dots,n_{0}.$

(B3) At most one of $C_{j},j=1,\dots,n_{0},$ is a major circular arc (a closed
circular arc is regarded major).

(B4) All $C_{j},j=1,\dots,n_{0},$ have the same curvature.

(C) There exists an integer $d^{\ast}=d_{L,q}$ depending only on $L$ and
$q\ $and there exists a precise extremal surface $\Sigma^{\ast}$ of
$\mathcal{F}\left(  L\right)  $ such that
\[
\deg_{\max}\Sigma^{\ast}\leq d^{\ast}%
\]
(see (\ref{degm}) for $\deg_{\max}$), and either $\Sigma^{\ast}$ is a simple
closed disk in $S\backslash E_{q},$ or $\partial\Sigma^{\ast}$ has a partition%
\begin{equation}
\partial\Sigma^{\ast}=C_{1}^{\prime}(q_{1}^{\prime},q_{2}^{\prime}%
)+C_{2}^{\prime}\left(  q_{2}^{\prime},q_{3}^{\prime}\right)  +\dots
+C_{n_{0}^{\prime}}^{\prime}\left(  q_{n_{0}^{\prime}}^{\prime},q_{1}^{\prime
}\right)  , \label{ag48}%
\end{equation}
with $n_{0}^{\prime}>1,$ such that
\begin{equation}
\partial C_{j}^{\prime}=\{q_{j}^{\prime},q_{j+1}^{\prime}\}\subset
E_{q},\mathrm{\ }q_{j}^{\prime}\neq q_{j+1}^{\prime},\mathrm{\ }C_{j}^{\circ
}\cap E_{q}=\emptyset, \label{ag45}%
\end{equation}
for all $j=1,2,\dots,n_{0}^{\prime}.$
\end{theorem}

\begin{proof}
Let $L\in\mathcal{L}$. For sufficiently large $m_{0}$ and each $m\geq m_{0},$
by Theorem \ref{LK} there exists a precise extremal surface $\Sigma_{m}$ of
$\mathcal{F}_{r}\left(  L,m\right)  .$ Assume $L(\partial\Sigma_{m})=L_{m}.$
Then by Theorem \ref{cat2}, we have $\{\Sigma_{m}\}_{m=1}^{\infty}%
\subset\mathcal{F}_{r}\left(  L,m_{0}\right)  ,$ and$\ $for every $m\geq
m_{0},$ since $\mathcal{F}_{r}\left(  L,m_{0}\right)  \subset\mathcal{F}%
_{r}\left(  L,m\right)  ,$ $\Sigma_{m}$ is an extremal surface of
$\mathcal{F}_{r}\left(  L,m_{0}\right)  .$ Therefore we have, for
every$\mathrm{\ }m\geq m_{0},$
\[
L_{m}\geq L_{m_{0}},\mathrm{\ }H(\Sigma_{m})=H(\Sigma_{m_{0}}),
\]
which with the relation $\Sigma_{m_{0}}\in\mathcal{F}_{r}\left(  L,m\right)
\ $implies that $\Sigma_{m_{0}}$ is an extremal surface of $\mathcal{F}%
_{r}\left(  L,m\right)  $ as well, and thus $L_{m_{0}}\geq L_{m}.$ Therefore
we have
\[
L_{m}=L_{m_{0}}\mathrm{\ and\ }H\left(  \Sigma_{m}\right)  =H\left(
\Sigma_{m_{0}}\right)  ,\mathrm{\ }m=m_{0},m_{0}+1,\dots
\]

For each $\Sigma\in\mathcal{F}_{r}(L),$ there exists an integer $m>m_{0}$ such
that $\Sigma\in\mathcal{F}_{r}(L,m).$ Then $H(\Sigma)\leq H(\Sigma
_{m})=H(\Sigma_{m_{0}}),$ and in consequence $\Sigma_{m_{0}}$ is an extremal
surface of $\mathcal{F}_{r}(L).$ Assume that $\Sigma^{\prime}$ is any other
extremal surface of $\mathcal{F}_{r}(L).$ Then for some positive integer
$m^{\prime}>m_{0}$, $\Sigma^{\prime}$ is an extremal surface of $\mathcal{F}%
_{r}(L,m^{\prime})$ and thus we have $L(\partial\Sigma^{\prime})\geq
L_{m}=L_{m_{0}},$ and therefore $\Sigma_{m_{0}}$ is precise extremal in
$\mathcal{F}_{r}\left(  L\right)  .$

We in fact proved that $\Sigma_{m}$ is precise extremal in $\mathcal{F}%
_{r}\left(  L\right)  $ for every $m\geq m_{0}.$ By Corollary \ref{FF'}, each
$\Sigma_{m}$ is precise extremal in $\mathcal{F}\left(  L,m\right)  $ as well
for each $m\geq m_{0}.$ On the other hand we have $\mathcal{F}\left(
L\right)  =\cup_{m=1}^{\infty}\mathcal{F}\left(  L,m\right)  $, $\mathcal{F}%
_{r}\left(  L\right)  =\cup_{m=1}^{\infty}\mathcal{F}_{r}\left(  L,m\right)
$, $\mathcal{F}_{r}\left(  L,m\right)  $ increases as $m$ increases, and so
does $\mathcal{F}\left(  L,m\right)  .$ Thus every precise extremal surface of
$\mathcal{F}_{r}\left(  L\right)  $ is precise extremal in $\mathcal{F}\left(
L\right)  $, $\mathcal{F}_{r}\left(  L,m\right)  $ and $\mathcal{F}\left(
L,m\right)  ,$ for each $m\geq m_{0};$ and (A) is proved.

Let $\Sigma_{0}=\left(  f_{0},\overline{\Delta}\right)  $ be any precise
extremal surface of $\mathcal{F}_{r}\left(  L\right)  $ with $L(\partial
\Sigma_{0})=L_{0}.$ Then (A) implies:

\begin{claim}
\label{ac1}$\Sigma_{0}=\left(  f_{0},\overline{\Delta}\right)  $ is a precise
extremal surface of every $\mathcal{F}_{r}\left(  L,m\right)  $ and every
$\mathcal{F}\left(  L,m\right)  $ for $m\geq m_{0}$ and $L_{0}=L\left(
\partial\Sigma_{0}\right)  =L_{m_{0}}.$
\end{claim}

Then $\partial\Delta$ and $\partial\Sigma_{0}$ have corresponding
$\mathcal{F}\left(  L,m_{0}\right)  $-partitions%
\[
\partial\Delta=\alpha_{1}\left(  a_{1},a_{2}\right)  +\alpha_{2}\left(
a_{2},a_{3}\right)  +\cdots+\alpha_{m_{0}}\left(  a_{m_{0}},a_{1}\right)  ,
\]%
\[
\partial\Sigma_{0}=c_{1}\left(  q_{1},q_{2}\right)  +c_{2}\left(  q_{2}%
,q_{3}\right)  +\cdots+c_{m_{0}}\left(  q_{m_{0}},q_{1}\right)  ,
\]
with $c_{j}=\left(  f,\alpha_{j}\right)  ,j=1,\dots,m_{0}.$ We will show that
$\partial\Sigma_{0}$ is circular at each $q_{j}\in\{q_{j}\}_{j=1}^{m_{0}%
}\backslash E_{q},$ say, $c_{j-1}+c_{j}$ is circular at $q_{j}$ if
$q_{j}\notin E_{q}.$

For sufficiently small $\varepsilon>0,$ $q_{j-\varepsilon}$ is a point in
$c_{j-1}^{\circ}$, which tends to $q_{j}$ as $\varepsilon\rightarrow0.$ Then
for sufficiently small $\varepsilon>0,$ we have an $\mathcal{F}(L,m_{0}%
+1)$-partition
\begin{equation}
\partial\Sigma_{0}=c_{1}+\cdots+c_{j-2}+c_{j-1}^{\prime}+c_{j-1}^{\prime
\prime}+c_{j}+\cdots+c_{m_{0}}, \label{ac2}%
\end{equation}
where $c_{j-1}^{\prime}=c_{j-1}(q_{j-1},q_{j-\varepsilon}),$ $c_{j-1}%
^{\prime\prime}=c_{j-1}(q_{j-\varepsilon},q_{j}).$ Then $c_{j-1}^{\prime
}+c_{j-1}^{\prime\prime}=c_{j-1}$ and for sufficiently small $\varepsilon>0,$
$c_{j-1}^{\prime\prime}\in\mathfrak{C}^{2}=\mathfrak{C}^{2}\left(  \Sigma
_{0}\right)  \ $(see Definition \ref{C^1} for the notation $\mathfrak{C}^{j}%
$). Since, by Claim \ref{ac1}, $\Sigma_{0}$ is also a precise extremal surface
of $\mathcal{F}_{r}(L,m_{0}+1)$ and (\ref{ac2}) is an $\mathcal{F}\left(
L,m_{0}+1\right)  $ partition of $\partial\Sigma_{0},$ Lemma \ref{circular}
implies that $c_{j-1}^{\prime\prime}+c_{j}$ is circular at $q_{j}$ if
$q_{j}\notin E_{q}$ and thus $c_{j-1}+c_{j}$ is circular at $q_{j}\ $if
$q_{j}\notin E_{q}.$ Therefore, we conclude that $\partial\Sigma_{0}$ is
circular everywhere outside $E_{q}.$ By Lemma \ref{home2}, $f_{0}$ is locally
homeomorphic in $\left[  \Delta\backslash f_{0}^{-1}(E_{q})\right]
\cup\left[  \left(  \partial\Delta\right)  \backslash\left[  f_{0}^{-1}%
(E_{q})\cap\{a_{j}\}_{j=1}^{m_{0}}\right]  \right]  .$ Thus we have:

\begin{claim}
\label{ac3}For each $a\in\partial\Delta$, if $a\notin f_{0}^{-1}(E_{q}),$ then
$a$ has a neighborhood $\alpha_{a}$ in $\partial\Delta$ such that $\left(
f_{0},\alpha_{a}\right)  $ is an SCC arc and $f_{0}$ restricted to a
neighborhood of $\alpha_{a}$ in $\overline{\Delta}$ is a homeomorphism.
\end{claim}

Then for sufficiently large $m,$ $\partial\Sigma_{0}$ has an $\mathcal{F}%
(L,m)$-partition $\partial\Sigma_{0}=\sum_{j=1}^{m}\mathfrak{c}_{j}$ such that
$\mathfrak{c}_{j}\in\mathfrak{C}^{1}\left(  \Sigma_{0}\right)  \ $(see
Definition \ref{C^1}). Since $\Sigma_{0}$ is also precise extremal in
$\mathcal{F}_{r}(L,m),$ by Lemma \ref{samecur} all terms $\mathfrak{c}_{j}$
have the same curvature. Thus we have

\begin{claim}
\label{ag46}The curvature of $\partial\Sigma_{0}=\left(  f_{0},\partial
\Delta\right)  $ is a constant function of $z\in\left(  \partial\Delta\right)
\backslash f_{0}^{-1}(E_{q})$
\end{claim}

Let $n_{1}=\#\left(  \partial\Delta\right)  \cap f_{0}^{-1}(E_{q}).$ Then
there are three possibility.

\noindent\textbf{Case 1. }$n_{1}=\#\left(  \partial\Delta\right)  \cap
f_{0}^{-1}(E_{q})=\emptyset.$

\noindent\textbf{Case 2. }$n_{1}=\#\left(  \partial\Delta\right)  \cap
f_{0}^{-1}(E_{q})=1.$

\noindent\textbf{Case 3. }$n_{1}=\#\left(  \partial\Delta\right)  \cap
f_{0}^{-1}(E_{q})\geq2.$

Assume Case 1 occurs. Then $C=f(\partial\Delta)$ is a circle with $C\cap
E_{q}=\emptyset$ and $f_{0}:\partial\Delta\rightarrow C$ is a CCM of degree
$k\ $for some integer $k\geq1.$ In this case, $\partial\Sigma_{0}$ has a
partition $\partial\Sigma_{0}=C_{1}+C_{2}+\dots+C_{k},$ such that every
$C_{k}$ is a simple closed path and $C_{k}=C.$ Then by Corollary \ref{inhalf}
$C$ is contained in an open hemisphere on $S.$ Thus $C$ contains a major
circular arc with length $<\pi.$

We will prove $k=1.$ We cannot use Lemma \ref{1-cir} directly, since it is for
terms of $\mathcal{F}\left(  L,m\right)  $-partitions in $\mathfrak{C}^{1},$
and every arc of $\mathfrak{C}^{1}$ has length $<\pi.$ But it applies in this
way: If $k>1,$ then $\partial\Sigma_{0}$ has an $\mathcal{F}_{r}\left(
L,k+2\right)  $ partition
\begin{equation}
\partial\Sigma_{0}=C_{1}^{\prime}+c_{1}^{\prime}+C_{2}^{\prime}+c_{2}^{\prime
}+C_{3}+\dots+C_{k} \label{mk-2major}%
\end{equation}
so that $C_{1}^{\prime}$ and $C_{2}^{\prime}$ are major circular arcs, each of
which has length $<\pi,$ thus $C_{1}^{\prime}$ and $C_{2}^{\prime}$ are both
in $\mathfrak{C}^{1}\left(  \Sigma_{0}\right)  ,$ contradicting Lemma
\ref{1-cir}. Thus (B) is proved in Case 1.

Assume Case 2 occurs and let $\left(  \partial\Delta\right)  \cap f_{0}%
^{-1}(E_{q})=\{a_{1}\}.$ Then by Claim \ref{ac3}, $\partial\Sigma
_{0}=(f,\partial\Delta)$ is a simple circle $C_{1}=C_{1}\left(  p_{1}%
,p_{1}\right)  $ with $q_{1}=f(a_{1}),$ and by Corollary \ref{inhalf} $C_{1}$
is contained in an open hemisphere on $S.$ Thus (B) also holds in Case 2.

Assume that Case 1 or 2 occurs. Then we have proved that $\partial\Sigma_{0}$
is a simple circle $C_{1}$ with $\#C_{1}\cap E_{q}\leq1.$ In this situation,
we in fact can show that (C) holds. Let $T$ be the closed disk enclosed by
$C_{1}.$ Then we can sew \label{sew10}$\Sigma_{0}$ and $S\backslash T$ along
$C_{1}$ to obtain a closed surface $F=\left(  f,S\right)  .$ Then we have
$L(\partial T)=L(\Sigma_{0})$, $A(\Sigma_{0})=A(F)-A(S\backslash
T)=A(F)+A(T)-4\pi$ and%
\[
\overline{n}\left(  \Sigma_{0}\right)  =\overline{n}\left(  F\right)
-\overline{n}\left(  S\backslash T^{\circ}\right)  =\overline{n}\left(
F\right)  -\overline{n}\left(  S\right)  +\overline{n}\left(  T\right)
=\overline{n}\left(  F\right)  -q+\overline{n}\left(  T\right)  .
\]
Note that $\overline{n}\left(  T\right)  =\#T^{\circ}\cap E_{q},$ not $\#T\cap
E_{q}.$ Therefore we have
\[
R(\Sigma_{0})=R(F)+R(T)+8\pi.
\]
By Lemma \ref{RH} we have $R(F)\leq-8\pi$, and thus $H\left(  \Sigma
_{0}\right)  \leq H(T).$ Then $H\left(  \Sigma_{0}\right)  =H(T)$, say, both
$T$ and $\Sigma_{0}$ are precise extremal surface of $\mathcal{F}\left(
L\right)  .$ If the diameter of $T$ \text{is equal to, or less than }%
$\delta_{E_{q}},$ then there is another disk $\Sigma^{\ast}$ in $S\backslash
E_{q}$ congruent to $T,$ with $H(T)\leq H(\Sigma^{\ast}).$ But $T$ is extremal
in $\mathcal{F}\left(  L\right)  .$ We have $H(T)=H(\Sigma^{\ast}),$ and thus
both $\Sigma^{\ast}$ and $T$ are precise extremal in $\mathcal{F}\left(
L\right)  .$ If the diameter of $T$ is larger than $\delta_{E_{q}},$ then by
moving $T$ continuously we can show that there is also another disk
$\Sigma^{\ast}$ on $S$ congruent to $T$ such that $\overline{n}\left(
\Sigma^{\ast}\right)  \leq\overline{n}\left(  T\right)  $ but $\partial
\Sigma^{\ast}$ contains at least two points of $E_{q},$ and then
$\partial\Sigma^{\ast}$ has a partition (\ref{ag48}) satisfying (\ref{ag45})
and $H(T)\leq H(\Sigma^{\ast}),$ which implies that $\Sigma^{\ast}$ and $T$
are both precise extremal in $\mathcal{F}\left(  L\right)  $. Therefore (C)
holds. We have proved (B) and (C) in Cases 1 and 2.

Assume Case 3 occurs. Then $f_{0}^{-1}(E_{q})$ divides $\partial\Delta$ into
$n_{1}$ arcs and thus $\partial\Sigma_{L_{0}}$ has an $\mathcal{F}\left(
L,n_{1}\right)  $-partition
\begin{equation}
\partial\Sigma_{0}=C_{1}+C_{2}+\cdots+C_{n_{1}} \label{ag44}%
\end{equation}
which satisfies (B1) and (B4) for $n_{0}=n_{1}.$ By Corollary \ref{inhalf},
(B2) holds true.

Assume that (\ref{ag44}) does satisfies (B3). Then by (B2), as the argument
for (\ref{mk-2major}), the partition (\ref{ag44}) has a refined $\mathcal{F}%
\left(  L,n_{1}+2\right)  $-partition$\ $which contains two terms of length
$<\pi$ which are major arcs and of class $\mathfrak{C}^{1}$. But this
contradicts Lemma \ref{1-cir} again. Thus (B3) holds with the partition
(\ref{ag44}), and (B) is proved completely.

Now we begin to prove (C) for Case 3. By (B3) we may assume

\begin{condition}
\label{ag50}$C_{2},C_{3},\dots,C_{n1}$ all satisfy (\ref{ag45}), $C_{1}$ may
or may not satisfy (\ref{ag45}).
\end{condition}

We first show the following.

\begin{claim}
\label{ag49}In Case 3, $\Sigma_{0}$ can be deformed to be a surface $F_{1}%
\in\mathcal{F}\left(  L,n_{1}+k\right)  $ such that $F_{1}$ has a partition of
the form (\ref{ag48}) satisfying (\ref{ag45}), with $n_{0}^{\prime}=n_{1}+k$
for some integer $k\geq0$.
\end{claim}

If $C_{1}$ also satisfies (\ref{ag45}), then there is nothing to prove. So we
assume $C_{1}$ is closed. Then by (B2) $L(C_{1})<2\pi,$ and by (B1) and
Corollary \ref{1-cir1} we can find a precise extremal surface $F_{1}$ in
$\mathcal{F}_{r}\left(  L\right)  ,$ so that $\partial F_{1}$ has an
$\mathcal{F}\left(  L,n_{1}\right)  $ partition
\[
\partial F_{1}=C_{1}^{\prime}+C_{2}+\cdots+C_{n_{1}}%
\]
which is the same as (\ref{ag44}), except that $C_{1}$ is replaced by a
rotation $C_{1}^{\prime}$ of $C_{1}\ $with $C_{1}^{\prime}\cap E_{q}%
=\{q_{1},q_{2}^{\prime},q_{3}^{\prime},\dots,q_{k}^{\prime}\}$ containing $k$
points arranged on $C_{1}^{\prime}$ anticlockwise, for some $k\geq2.$ Then
$C_{1}^{\prime}$ can be divided into $k$ arcs by $\left\{  q_{j}^{\prime
}\right\}  _{j=1}^{k}$ with $q_{1}^{\prime}=q_{1},$ and thus $\partial F_{1}$
has an $\mathcal{F}\left(  L,n_{0}^{\prime}\right)  $-partition
\[
\partial\Sigma=\partial F_{2}=C_{1}^{\prime}\left(  q_{1},q_{2}^{\prime
}\right)  +C_{1}^{\prime}\left(  q_{2}^{\prime},q_{3}^{\prime}\right)
+\cdots+C_{1}^{\prime}\left(  q_{k}^{\prime},q_{1}\right)  +C_{2}\left(
q_{1},q_{2}\right)  +\cdots+C_{n_{1}}\left(  q_{n_{1}},q_{1}\right)
\]
satisfying (\ref{ag45}) with $n_{0}^{\prime}=n_{1}+k.$

By Theorem \ref{sim}, there exists a positive integer $d^{\ast}$ depending
only on $n_{0}^{\prime}=n_{1}+k$ and $q,$ which in fact depends only on $L$
and $q$ since $n_{0}^{\prime}<L_{0}/\delta_{E_{q}},$ and there exists a
surface $\Sigma^{\ast}$ in $\mathcal{F}_{r}\left(  L,n_{0}^{\prime}\right)  $
such that
\[
H(\Sigma^{\ast})\geq H(F_{1}),L(\partial\Sigma^{\ast})\leq L(\partial F_{1}),
\]
$\label{ling-tian-run}$the second equality holding only if $\partial
\Sigma^{\ast}=\partial F_{1}.$ Then $\Sigma^{\ast}$ is a precise extremal
surface of $\mathcal{F}_{r}\left(  L,n_{0}^{\prime}\right)  $ and thus
$L(\partial\Sigma^{\ast})\geq L(\partial F_{1}),$ which implies $L(\partial
\Sigma^{\ast})=L(\partial F_{1})$ and thus $\partial\Sigma^{\ast}=\partial
F_{1},$ which has the $\mathcal{F}\left(  L,n_{0}^{\prime}\right)  $-partition
satisfying (\ref{ag45}), and (C) is proved in Case 3.
\end{proof}

\section{Proof of Theorems \ref{main2}, \ref{2,3} and \ref{q3}}

In this section we complete the proof of our second and third main theorems.

\begin{definition}
\label{S15678}Let $\mathcal{S}_{1}$ be the space that satisfies Definition
\ref{S-surface} (1)--(4). We denote by $\mathcal{S}_{1}^{\left(  5\right)
},\mathcal{S}_{1}^{(6)},\mathcal{S}_{1}^{(7)},\mathcal{S}_{1}^{(8)}$ the
subspaces of $\mathcal{S}_{1}$ satisfying (5), (6), (7) and (8) of Definition
\ref{S-surface} respectively.
\end{definition}

Then we have $\mathcal{S}_{0}=\mathcal{S}_{1}^{\left(  5\right)  }%
\cap\mathcal{S}_{1}^{(6)}\cap\mathcal{S}_{1}^{(7)}\cap\mathcal{S}_{1}^{(8)}.$

We first introduce a procedure to construct a surface for given boundary.

\begin{solution}
\label{sol}(Standard solution of surfaces with known boundary) Let $\Gamma$ be
a closed curve on $S$ with the partition
\[
\Gamma=\left(  f,\partial\Delta\right)  =C_{1}\left(  p_{1},p_{2}\right)
+\cdots+C_{q^{\prime}}\left(  p_{q^{\prime}},p_{1}\right)  ,q^{\prime}\leq q.
\]
satisfying Definition \ref{S-surface} (1)--(3). Then $d\left(  p_{j}%
,p_{j+1}\right)  <\pi$ and $\overline{p_{j}p_{j+1}}$ is well defined. We can
construct a surface $\Sigma_{\Gamma}$ satisfying the following condition.

\begin{condition}
\label{S1Fr}$\Sigma_{\Gamma}$ is contained in $\mathcal{S}_{1}\cap
\mathcal{F}_{r}$ and that $\partial\Sigma_{\Gamma}=\Gamma.$
\end{condition}

The surface $\Sigma_{\Gamma}$ can be obtained as the output when we input
$\Gamma$ and execute the following procedure.

For each $j=1,2,\dots,q^{\prime},$ let $K_{j}=\overline{\mathfrak{D}^{\prime
}\left(  \overline{p_{j}p_{j+1}},C_{j}\right)  }$ (see Definition
\ref{lune-lens}), which is the closed convex lune enclosed by the arc
$C_{j}=C_{j}\left(  p_{j},p_{j+1}\right)  $ and its chord $\overline
{p_{j+1}p_{j}},$ say $\partial K_{j}=C_{j}-\overline{p_{j}p_{j+1}}.$ If $C_{j}
$ is the line segment $\overline{p_{j}p_{j+1}}$ for some $j$, then
$K_{j}^{\circ}=\emptyset$ and we just set $K_{j}=\overline{p_{j}p_{j+1}}.$ If
$C_{j}$ is a major arc of a great circle on $S$ for some $j,$ then by
definition $K_{j}$ is the closed hemisphere enclosed by the great circle
$C_{j}-\overline{p_{j}p_{j+1}}=C_{j}+\overline{p_{j+1}p_{j}}.$ Note that
$K_{q^{\prime}}=\overline{\mathfrak{D}^{\prime}\left(  \overline{p_{q^{\prime
}}p_{1}},C_{q}\right)  },$ say $p_{q^{\prime}+1}=p_{1}.$

Let $l_{2}=\overline{p_{1}p_{2}}$ and $l_{q^{\prime}}=\overline{p_{1}%
p_{q^{\prime}}},$ which are well defined since $d\left(  p_{j},p_{j+1}\right)
<\pi$ for $j=1,\dots,q^{\prime},p_{q^{\prime}+1}=p_{1}$. For each
$j=3,\dots,q^{\prime}-1,$ let $l_{j}=\overline{p_{1}p_{j}}$ if $\overline
{p_{1}p_{j}}$ is well defined, say, $p_{j}$ is not the antipodal point
$p_{1}^{\ast}$ of $p_{1},$ or let $l_{j}\ $be the straight line segment
$\overline{p_{1}p_{j-1}p_{j}}$ if $p_{j}=p_{1}^{\ast}, $ which is half of the
great circle from $p_{1}$ to $p_{1}^{\ast}$ passing through $p_{j-1}.$ Since
$p_{1},\dots,p_{q^{\prime}}$ are distinct each other, if $p_{j}=p_{1}^{\ast}$
for some $j\ $of $3,\dots,q^{\prime}-1,$ we have that $l_{i}=\overline
{p_{1}p_{i}}$ is well defined for each $i\neq j$ with $3\leq i\leq q^{\prime
}-1$ and $\overline{p_{j-1}p_{j}}$ is also well defined.

When $q^{\prime}=2,$ let $T_{2}$ be the surface whose interior is the simple
domain $S\backslash l_{2}$ and boundary is $l_{2}-l_{2}.$

When $q^{\prime}\geq3,$ for each $j=2,\dots,q^{\prime}-1,$ we define $T_{j}$
to be the closed triangular domain enclosed by
\[
l_{j}+\overline{p_{j}p_{j+1}}-l_{j+1},j=2,\dots,q^{\prime}-1.
\]
Each $T_{j},j=2,\dots,q^{\prime}-1\left(  q^{\prime}\geq3\right)  ,$ is a
simple surface in the sense that its interior is a simple nonempty domain on
$S.$ But it is possible that for some $j,$ $l_{j}+\overline{p_{j}p_{j+1}%
}-l_{j+1}$ may be contained in a line segment, and in this case the interior
of $T_{j}$ is $S\backslash l_{j}$ when $p_{j+1}\in l_{j}^{\circ},$ or
$S\backslash l_{j+1}$ when $p_{j}\in l_{j+1}^{\circ}.$ Recall Remark
\ref{tri}, when we regard $T_{j}$ as a surface, $\partial T_{j}$ is a
\emph{simple} closed curve, and thus $l_{j}$ and $l_{j+1}$, regarded as arcs
in the surface, only intersect at $p_{1}$ in the surface $T_{j},$ even if
$\partial T_{j}$ is just a line segment as a set on $S.$

We sew\label{sew21} $T_{2},\dots,T_{q^{\prime}-1}$ along $l_{3},\dots
,l_{q^{\prime}-2}$ to obtain a surface $P$ with boundary $\overline{p_{1}%
p_{2}\dots p_{q^{\prime}}p_{1}}.$ Then we can sew \label{sew9}$P$ and $K_{j}$
along $\overline{p_{j}p_{j+1}}$ for each $j=1,2,\dots,q^{\prime}$ to obtain a
surface $\Sigma_{\Gamma}$ with $\partial\Sigma_{\Gamma}=\Gamma$ and
\[
\deg_{\max}\Sigma_{\Gamma}\leq q^{\prime}-2+q^{\prime}\leq2q^{\prime}-2.
\]
Thus we have $\Sigma_{\Gamma}\in\mathcal{S}_{1}.$

It is clear that all branch values of the surface $P$ are contained in
$\{p_{1},\dots,p_{q^{\prime}}\}\ $and thus when we patch all lunes $K_{j}$ to
$P$ along $\overline{p_{j}p_{j+1}},$ no other branch values appeared. Thus all
branch values of $\Sigma_{\Gamma}$ are contained in $\{p_{1},\dots
,p_{q^{\prime}}\}\subset E_{q},$ and so $\Sigma_{\Gamma}\in\mathcal{F}_{r}%
\cap\mathcal{S}_{1}.$

We will write $P=P(p_{1},p_{2},\dots,p_{q^{\prime}}).$ Then it is clear that
$P(p_{1},p_{2},$ $\dots,$ $p_{q^{\prime}})$ is determined by the ordered
points $p_{1},p_{2},\dots,p_{q^{\prime}}$ uniquely, and thus $\Sigma_{\Gamma}
$ is determined by $\Gamma$ uniquely, when we execute the above procedure.
Uniqueness here is in the sense described in Remark \ref{finite}.

If $q^{\prime}\geq4,$ we can regard $l_{j},j=3,\dots,q^{\prime}-1$, as simple
arcs in the surface $P$ with $l_{j}^{\circ}\subset P^{\circ}\ $and $\partial
l_{j}\subset\partial P$ so that every pair $l_{i}$ and $l_{j}$ with $3\leq
i<j\leq q^{\prime}-1$ only intersect at $p_{1}\ $in $P$ (see Remark \ref{tri}
for the convention). Then it is clear that $l_{3},\dots,l_{q^{\prime}-1}$
divide $P$ into the surfaces $T_{2},\dots,T_{q^{\prime}-1}.$ Then it is clear
that when $q^{\prime}\geq4$%
\begin{equation}
R\left(  P\right)  =\sum_{j=2}^{q^{\prime}-1}R(T_{j})-4\pi\sum_{j=2}%
^{q^{\prime}-2}\#\left(  l_{j+1}^{\circ}\cap E_{q}\right)  . \label{ag51}%
\end{equation}

The polygon $\overline{p_{1}p_{2}\dots p_{q^{\prime}}p_{1}}$ divides the
surface $\Sigma_{\Gamma}$ into the surface $P\left(  p_{1},p_{2}%
,\dots,p_{q^{\prime}}\right)  $ and the $q^{\prime}$ lunes $K_{j}%
,j=1,\dots,q^{\prime}.$ Write
\begin{equation}
J=J(\Gamma)=\{j:1\leq j\leq q^{\prime}\text{ \textrm{and}}\mathrm{\ }%
K_{j}^{\circ}\neq\emptyset\}. \label{ag12}%
\end{equation}
Note that the condition $K_{j}^{\circ}\neq\emptyset$ means that $\partial
K_{j}=C_{j}\left(  p_{j},p_{j+1}\right)  -\overline{p_{j}p_{j+1}}$ is a Jordan
curve on $S$ and $\overline{p_{j}p_{j+1}}^{\circ}$ (regarded as in the surface
$\Sigma_{\Gamma}$) is in the interior of $\Sigma_{\Gamma}.$ Then we have
\begin{equation}
R(\Sigma_{\Gamma})=R(P)+\sum_{j\in J}\left[  R(K_{j})-4\pi\#\left(
\overline{p_{j}p_{j+1}}^{\circ}\cap E_{q}\right)  \right]  . \label{ag22}%
\end{equation}
Therefore, by (\ref{ag51}) we have for $q^{\prime}\geq4$ that%
\begin{equation}
R(\Sigma_{\Gamma})=\sum_{j=2}^{q^{\prime}-1}R\left(  T_{j}\right)  +\sum_{j\in
J}\left[  R(K_{j})-4\pi\#\left(  \overline{p_{j}p_{j+1}}^{\circ}\cap
E_{q}\right)  \right]  -4\pi\sum_{j=2}^{q^{\prime}-2}\#\left(  l_{j+1}^{\circ
}\cap E_{q}\right)  . \label{bd4}%
\end{equation}

For the cases $q^{\prime}=2,3,$ by the definition of $T_{2}$ and $\left\{
K_{j}\right\}  _{j=1}^{q^{\prime}},$ we have
\begin{equation}
R(\Sigma_{\Gamma})=R(T_{2})+\sum_{j\in J}\left[  R(K_{j})-4\pi\#\overline
{p_{j}p_{j+1}}^{\circ}\cap E_{q}\right]  . \label{ag6}%
\end{equation}

\end{solution}

\begin{lemma}
\label{makes}Let $\Sigma=\left(  f,\overline{\Delta}\right)  $ be a surface in
$\mathcal{F}$ such that
\[
\partial\Sigma=\Gamma_{1}+\Gamma_{2}%
\]
and $\Gamma_{1}$ is a closed arc of $\partial\Sigma$ satisfying Definition
\ref{S-surface} (1)--(3) with the corresponding partition
\[
\Gamma_{1}=\left(  f,\partial\Delta\right)  =C_{1}\left(  p_{1},p_{2}\right)
+\cdots+C_{q^{\prime}}\left(  p_{q^{\prime}},p_{1}\right)  .
\]
Let $\Sigma_{\Gamma_{1}}\in\mathcal{S}_{1}\cap\mathcal{F}_{r}$ be the surface
given by Solution \ref{sol}. Then there exists a surface $\Sigma_{2}=\left(
f_{2},S\right)  $ without boundary or $\Sigma_{2}=\left(  f_{2},\overline
{\Delta}\right)  $ with boundary $\Gamma_{2}$ such that the following holds.

(i) If $\partial\Sigma=\Gamma_{1},$ say, $\Gamma_{2}$ reduces to the point
$p_{1}$, then $\Sigma_{2}=\left(  f_{2},S\right)  $ is a closed surface and%
\begin{equation}
R(\Sigma)=R(\Sigma_{\Gamma_{1}})+R(\Sigma_{2})+8\pi\leq R(\Sigma_{\Gamma_{1}%
}), \label{ag19}%
\end{equation}
with equality holding if and only if $CV_{f_{2}}\subset E_{q}.$

(ii) If $\Gamma_{2}$ is not a point, then $\Sigma_{2}=\left(  f_{2}%
,\overline{\Delta}\right)  \in\mathcal{F}$ and%
\begin{equation}
R(\Sigma)+4\pi=\left[  R\left(  \Sigma_{\Gamma_{1}}\right)  +4\pi\right]
+\left[  R(\Sigma_{2})+4\pi\right]  . \label{ag20}%
\end{equation}

(iii) $CV_{f_{2}}\subset E_{q}$ iff $\Sigma\in\mathcal{F}_{r}.$
\end{lemma}

\begin{proof}
In the proof the notations $K_{j},T_{j},l_{j},J$ are defined in Solution
\ref{sol}. For all $j\in J,$ we sew\label{sew20} the surface $\Sigma$ and the
surface $K_{j}^{c}=S\backslash K_{j}^{\circ}$ along $C_{j}$ to obtain a
surface $G_{1}^{\prime}$ such that
\begin{equation}
\partial G_{1}^{\prime}=\overline{p_{1}p_{2}p_{3}\cdots p_{q^{\prime}}p_{1}%
}+\Gamma_{2}. \label{bd1}%
\end{equation}
By Lemma \ref{ps-dm} (regarding $K_{j},$ $C_{j}$ and $\overline{p_{j}p_{j+1}}$
as $T,\gamma$ and $\gamma^{\prime}$ in the lemma, and applying the lemma $\#J$
times) we have
\begin{equation}
R(\Sigma)=R(G_{1}^{\prime})+\sum_{j\in J}\left[  R(K_{j})-4\pi\#\overline
{p_{j}p_{j+1}}^{\circ}\cap E_{q}\right]  , \label{ko9-1}%
\end{equation}
and $\Sigma\in\mathcal{F}_{r}$ iff $G_{1}^{\prime}\in\mathcal{F}_{r}.$

We first consider the case $q^{\prime}=2.$ Then we have by (\ref{bd1})
\[
\partial G_{1}^{\prime}=\overline{p_{1}p_{2}p_{1}}+\Gamma_{2}.
\]
Since $\partial T_{2}=\overline{p_{1}p_{2}}+\overline{p_{2}p_{1}},$ we have by
Lemma \ref{-path}
\begin{equation}
R(T_{2})=4\pi\#\overline{p_{1}p_{2}}^{\circ}\cap E_{q}. \label{ag17}%
\end{equation}

If $\Gamma_{2}$ is a point, then we can sew \label{sew8}$G_{1}^{\prime}$ along
$\overline{p_{1}p_{2}}$ to obtain a surface $\Sigma_{2}=\left(  f_{2}%
,S\right)  $ which is closed, say, a complete covering of $S$ onto itself. By
Lemma \ref{glue} (i)
\[
R(G_{1}^{\prime})=R(\Sigma_{2})+4\pi\#\left[  \overline{p_{1}p_{2}}\cap
E_{q}\right]  =R(\Sigma_{2})+4\pi\#\left[  \overline{p_{1}p_{2}}^{\circ}\cap
E_{q}\right]  +8\pi,
\]
and $G_{1}^{\prime}\in\mathcal{F}_{r}$ iff $CV_{f_{2}}\subset E_{q}.$ Then by
(\ref{ag17}) we have
\[
R(G_{1}^{\prime})=R(\Sigma_{2})+R(T_{2})+8\pi,
\]
and then by Lemma \ref{RH}, (\ref{ko9-1}) and (\ref{ag6}) we have
\begin{align*}
R(\Sigma)  &  =R(G_{1}^{\prime})+\sum_{j\in J}\left[  R(K_{j})-4\pi
\#\overline{p_{j}p_{j+1}}^{\circ}\cap E_{q}\right] \\
&  =R(T_{2})+R(\Sigma_{2})+8\pi+\sum_{j\in J}\left[  R(K_{j})-4\pi
\#\overline{p_{j}p_{j+1}}^{\circ}\cap E_{q}\right] \\
&  =R(\Sigma_{\Gamma})+R(\Sigma_{2})+8\pi\leq R(\Sigma_{\Gamma}),
\end{align*}
equality holding iff $CV_{f_{2}}\subset E_{q},$ and so (i) holds when
$q^{\prime}=2.$ Note that the three relations $CV_{f_{2}}\subset E_{q}$,
$G_{1}^{\prime}\in\mathcal{F}_{r}\ $and $\Sigma\in\mathcal{F}_{r}$ are equivalent.

If $\Gamma_{2}$ is not a point, then by Lemma \ref{glue} (ii), for the surface
$\Sigma_{2}=\left(  f_{2},\overline{\Delta}\right)  $ obtained from
$G_{1}^{\prime}$ by sewing it \label{sew19} along $\overline{p_{1}p_{2}},$ we
have%
\[
R(G_{1}^{\prime})=R(\Sigma_{2})+4\pi\#\overline{p_{1}p_{2}}\cap E_{q}%
-4\pi=R(\Sigma_{2})+4\pi\#\overline{p_{1}p_{2}}^{\circ}\cap E_{q}+4\pi,
\]
and $G_{1}^{\prime}\in\mathcal{F}_{r}$ iff $\Sigma_{2}\in\mathcal{F}_{r}.$
Thus by (\ref{ko9-1}) (\ref{ag6}), and (\ref{ag17}) we have%
\[
R(\Sigma)=R(\Sigma_{\Gamma})+R(\Sigma_{2})+4\pi,
\]
which implies (\ref{ag20}). On the other hand, the three relations $\Sigma
_{2}\in\mathcal{F}_{r}$, $G_{1}^{\prime}\in\mathcal{F}_{r}\ $and $\Sigma
\in\mathcal{F}_{r}$ are equivalent again. We have prove the lemma when
$q^{\prime}=2.$

From now on we assume $q^{\prime}\geq3$. We will construct a sequence
$G_{1}^{\prime},G_{2}^{\prime},\dots,G_{q^{\prime}-1}^{\prime}$ in
$\mathcal{F}$ such that,%
\begin{equation}
\partial G_{j}^{\prime}=l_{j+1}+\overline{p_{j+1}p_{j+2}\dots p_{q^{\prime}%
}p_{1}}+\Gamma_{2},j=1,\dots,q^{\prime}-1, \label{ag2}%
\end{equation}
and
\begin{equation}
R(G_{j}^{\prime})=R(G_{j+1}^{\prime})+R(T_{j+1})-4\pi\#\left(  l_{j+2}^{\circ
}\cap E_{q}\right)  ,j=1,\dots,q^{\prime}-2. \label{ag3}%
\end{equation}

By (\ref{bd1}), $G_{1}^{\prime}$ already satisfies (\ref{ag2}). Assume
$G_{1}^{\prime},\dots,G_{k}^{\prime},$ satisfying (\ref{ag2}) for
$j=1,\dots,k,$ and (\ref{ag3}) for $j=1,\dots,k-1,$ are already obtained for
some $k$ with $1\leq k\leq q^{\prime}-2.$

When $\partial T_{k+1}=l_{k+1}+\overline{p_{k}p_{k+1}}-l_{k+2}\ $is a Jordan
curve on $S$, we sew\label{sew18} $G_{k}^{\prime}$ and $S\backslash
T_{k+1}^{\circ}$ along $l_{k+1}+\overline{p_{k+1}p_{k+2}}$ to obtain a surface
$G_{k+1}^{\prime}\ $so that (\ref{ag2}) holds for $j=k+1$. Then by Lemma
\ref{ps-dm}%
\[
R(G_{k}^{\prime})=R(G_{k+1}^{\prime})+R(T_{k+1})-4\pi\#E_{q}\cap
l_{k+2}^{\circ},
\]
say, (\ref{ag3}) holds for $j=k,$ and moreover, $G_{k}^{\prime}\in
\mathcal{F}_{r}$ iff $G_{k+1}^{\prime}\in\mathcal{F}_{r}.$

Consider the case that $l_{k+1}=l_{k+1}\cup\overline{p_{k+1}p_{k+2}},$ say%
\begin{equation}
l_{k+2}=l_{k+1}+\overline{p_{k+1}p_{k+2}}. \label{ag13}%
\end{equation}
In this case we just put $G_{k+1}^{\prime}=G_{k}^{\prime}.$ Then (\ref{ag2})
holds for $j=k+1$, by the equality
\[
l_{k+1}+\overline{p_{k+1}p_{k+2}\dots p_{q^{\prime}}p_{1}}=l_{k+2}%
+\overline{p_{k+2}\dots p_{q^{\prime}}p_{1}}.
\]
(\ref{ag13}) implies%
\[
\partial T_{k+1}=l_{k+1}+\overline{p_{k+1}p_{k+2}}-l_{k+2}=l_{k+2}-l_{k+2},
\]
and thus we have by Lemma \ref{-path} $R(T_{k+1})=4\pi\#l_{k+2}^{\circ}\cap
E_{q},$ say, (\ref{ag3}) holds for $j=k$ as well.

Consider the case
\begin{equation}
l_{k+1}=l_{k+2}+\overline{p_{k+2}p_{k+1}}. \label{ag14}%
\end{equation}
Then%
\begin{equation}
\partial T_{k+1}=l_{k+1}+\overline{p_{k+1}p_{k+2}}-l_{k+2}=l_{k+1}-l_{k+1},
\label{ag10}%
\end{equation}
and
\begin{align*}
\partial G_{k}^{\prime}  &  =l_{k+1}+\overline{p_{k+1}p_{k+2}p_{k+3}\dots
p_{q^{\prime}}p_{1}}+\Gamma_{2}\\
&  =l_{k+2}+\overline{p_{k+2}p_{k+1}}+\overline{p_{k+1}p_{k+2}p_{k+3}\dots
p_{q^{\prime}}p_{1}}+\Gamma_{2}\\
&  =l_{k+2}+\overline{p_{k+2}p_{k+1}}+\overline{p_{k+1}p_{k+2}}+\overline
{p_{k+2}p_{k+3}\dots p_{q^{\prime}}p_{1}}+\Gamma_{2}.
\end{align*}
Thus we can sew\label{sew7} $G_{k}^{\prime}$ itself along $\overline
{p_{k+2}p_{k+1}}$ to obtain $G_{k+1}^{\prime}$ satisfying (\ref{ag2}) for
$j=k+1$. By Lemma \ref{glue} (ii)
\[
R(G_{k}^{\prime})=R(G_{k+1}^{\prime})+4\pi\#\overline{p_{k+2}p_{k+1}}\cap
E_{q}-4\pi,
\]
and moreover, $G_{k}^{\prime}\in\mathcal{F}_{r}$ iff $G_{k+1}^{\prime}%
\in\mathcal{F}_{r}.$ On the other hand%
\[
\#\overline{p_{k+2}p_{k+1}}\cap E_{q}=\#\overline{l_{k+1}}^{\circ}\cap
E_{q}-\#\overline{l_{k+2}}^{\circ}\cap E_{q}+1.
\]
Hence we have%
\begin{equation}
R(G_{k}^{\prime})=R(G_{k+1}^{\prime})+4\pi\#\overline{l_{k+1}}^{\circ}\cap
E_{q}-4\pi\#\overline{l_{k+2}}^{\circ}\cap E_{q}. \label{ag4}%
\end{equation}
By (\ref{ag10}) and Lemma \ref{-path}%
\[
R(T_{k+1})=4\pi\#l_{k+1}^{\circ}\cap E_{q}.
\]
Hence by (\ref{ag4}) we have (\ref{ag3}) for $j=k.$ We have proved the
existence of the surfaces $G_{j}^{\prime},j=1,\dots,q^{\prime}-1,$ in any
case, and moreover, $G_{j}^{\prime}\in\mathcal{F}_{r}$ iff $G_{j+1}^{\prime
}\in\mathcal{F}_{r}$ for $j=1,\dots,q-1.$

By (\ref{ag2}), we have
\begin{equation}
G_{q^{\prime}-1}^{\prime}=l_{q^{\prime}}+\overline{p_{q^{\prime}}p_{1}}%
+\Gamma_{2}=l_{q^{\prime}}-l_{q^{\prime}}+\Gamma_{2}, \label{ag5}%
\end{equation}
and, by (\ref{ko9-1}) and (\ref{ag3}), we have
\begin{equation}%
\begin{tabular}
[c]{l}%
$R(\Sigma)=R(G_{q^{\prime}-1}^{\prime})+\sum_{j=2}^{q^{\prime}-1}\left[
R(T_{j})-4\pi\#l_{j+1}^{\circ}\right]  $\\
$\;\;\;\mathrm{\ \ \ \ \ \ \ \ \ }+\sum_{j\in J}\left[  R(K_{j})-4\pi
\#\overline{p_{j}p_{j+1}}^{\circ}\cap E_{q}\right]  ,$%
\end{tabular}
\ \ \ \ \ \ \ \ \label{ag7}%
\end{equation}
which, together with (\ref{ag6}) for $q^{\prime}=3,$ or (\ref{bd4}) for
$q^{\prime}>3,$ implies that%
\begin{equation}
R(\Sigma)=R(G_{q^{\prime}-1}^{\prime})+R(\Sigma_{\Gamma})-4\pi\#l_{q^{\prime}%
}^{\circ}\cap E_{q}. \label{ag11}%
\end{equation}
Note that the term $4\pi\#l_{q^{\prime}}^{\circ}\cap E_{q}=4\pi\#\overline
{p_{1}p_{q^{\prime}}}^{\circ}\cap E_{q}$ in (\ref{ag7}) never appeared in
(\ref{bd4}).

Assume $\Gamma_{2}$ is a point. Then by (\ref{ag5}) we can sew\label{sew6} the
surface $G_{q^{\prime}-1}^{\prime}$ along $l_{q^{\prime}}=\overline
{p_{1}p_{q^{\prime}}}$ so that $G_{q^{\prime}-1}^{\prime}$ becomes a closed
surface $\Sigma_{2}=\left(  f_{2},S\right)  .$ Then by Lemma \ref{glue} (i) we
have
\begin{equation}
R\left(  G_{q^{\prime}-1}^{\prime}\right)  =R(\Sigma_{2})+4\pi\#l_{q^{\prime}%
}\cap E_{q}=R(\Sigma_{2})+4\pi\#l_{q^{\prime}}^{\circ}\cap E_{q}+8\pi,
\label{ag15}%
\end{equation}
and thus by (\ref{ag11})
\begin{equation}
R(\Sigma)=R(\Sigma_{\Gamma_{1}})+R(\Sigma_{2})+8\pi. \label{ko14}%
\end{equation}
By Lemma \ref{RH} we have $R(\Sigma_{2})\leq-8\pi,$ with equality holding if
and only if $CV_{f_{2}}\subset E_{q}.$ On the other hand it is clear that
\begin{equation}
\Sigma\in\mathcal{F}_{r}\Leftrightarrow G_{1}^{\prime}\in\mathcal{F}%
_{r}\Leftrightarrow G_{2}^{\prime}\in\mathcal{F}_{r}\Leftrightarrow
\dots\Leftrightarrow G_{q^{\prime}-1}^{\prime}\in\mathcal{F}_{r}%
\Leftrightarrow CV_{f_{2}}\subset E_{q}, \label{equok}%
\end{equation}
(i) is proved completely.

Assume that $\Gamma_{2}\ $is not a point. Then by (\ref{ag5}) we can
sew\label{sew5} $G_{q^{\prime}-1}^{\prime}$ itself along $l_{q^{\prime}%
}=\overline{p_{1}p_{q^{\prime}}}$ to obtain a surface $\Sigma_{2}=\left(
f_{2},\overline{\Delta}\right)  \in\mathcal{F}$ such that $\partial\Sigma
_{2}=\Gamma_{2}.$ By Lemma \ref{glue} (ii) we have
\[
R(G_{q^{\prime}-1}^{\prime})=R(\Sigma_{2})+4\pi\#l_{q^{\prime}}\cap E_{q}%
-4\pi=R(\Sigma_{2})+4\pi\#l_{q^{\prime}}^{\circ}\cap E_{q}+4\pi,
\]
which with (\ref{ag11}) implies
\[
R(\Sigma)=R(\Sigma_{\Gamma})+R(\Sigma_{2})+4\pi.
\]
In consequence we have (\ref{ag20}). On other hand, (\ref{equok}) also holds
and so (ii) is proved completely.
\end{proof}

For simplicity, we introduce the following condition:\medskip

\begin{definition}
A curve $\Gamma=\left(  f,\partial\Delta\right)  $ is called satisfying
$\left(  p_{1},\dots,p_{m}\right)  $-\textbf{Condition, }if\textbf{\ }%
$\Gamma=\left(  f,\partial\Delta\right)  $ has a partition%
\begin{equation}
\partial\Sigma=C_{1}\left(  p_{1},p_{2}\right)  +\cdots+C_{m}\left(
p_{m},p_{1}\right)  \label{bd9}%
\end{equation}
and for each $j\leq m,$ $C_{j}$ is an SCC arc and the endpoints $p_{j}$ and
$p_{j+1}$ are distinct and contained in $E_{q},$ and moreover $d\left(
p_{j},p_{j+1}\right)  <\pi.$
\end{definition}

\begin{remark}
If $p_{1},\dots,p_{m}$ are distinct each other, then $\Gamma$ satisfies
Definition \ref{S-surface} (1)--(3) with partition (\ref{bd9}). In general it
is clear that the following holds.
\end{remark}

\begin{claim}
\label{(a)}Assume $\Sigma\in\mathcal{F}$ such that $\partial\Sigma$ satisfies
$\left(  p_{1},\dots,p_{m}\right)  $-Condition with partition (\ref{bd9}).
Then $\partial\Sigma$ has a partition $\partial\Sigma=\Gamma_{1}+\Gamma_{2}$
satisfying Lemma \ref{makes}, say,
\[
\Gamma_{1}=C_{j}\left(  p_{j},p_{j+1}\right)  +C_{j+1}\left(  p_{j+1}%
,p_{j+2}\right)  +\dots+C_{j+k}\left(  p_{j+k},p_{j+k+1}\right)
\]
is a closed arc of $\partial\Sigma$ satisfying Definition \ref{S-surface}
(1)--(3) and $\Gamma_{2}$ is either a point when $p_{1},\dots,p_{m}$ are
distinct each other with $j=1$ and $j+k=m$, or
\begin{align*}
\Gamma_{2}  &  =C_{1}\left(  p_{1},p_{2}\right)  +C_{2}\left(  p_{2}%
,p_{3}\right)  +\dots+C_{j-1}\left(  p_{j-1},p_{j}\right) \\
&  +C_{j+k+1}\left(  p_{j+k+1},p_{j+k+2}\right)  +\dots+C_{m}\left(
p_{m},p_{1}\right)
\end{align*}
which satisfies $\left(  p_{1},p_{2},\dots,p_{j},p_{j+k+1},p_{j+k+2}%
,\dots,p_{m}\right)  $-Condition.
\end{claim}

\begin{corollary}
\label{make2}Let $\Sigma$ be a surface in $\mathcal{F}$ such that
$\partial\Sigma$ satisfying $\left(  p_{1},\dots,p_{m}\right)  $-Condition
with partition (\ref{bd9}). Then there exist surfaces $\Sigma_{j}%
\in\mathcal{S}_{1}\cap\mathcal{F}_{r}$ with partitions%
\begin{equation}
\partial\Sigma_{j}=C_{j1}\left(  p_{j1},p_{j2}\right)  +\cdots+C_{jq_{j}%
^{\prime}}\left(  p_{jq_{j}^{\prime}},p_{j1}\right)  \label{bd10}%
\end{equation}
satisfying Definition \ref{S-surface} (1)--(4) for $j=1,\dots,k,$ $k\geq1,$
such that%
\begin{equation}
L(\partial\Sigma)=\sum_{j=1}^{k}L\left(  \partial\Sigma_{j}\right)  ,
\label{bd13}%
\end{equation}
and
\begin{equation}
R(\Sigma)+4\pi\leq\sum_{j=1}^{k}\left(  R(\Sigma_{j}\right)  +4\pi),
\label{bd14}%
\end{equation}
with equality holding if and only if $\Sigma\in\mathcal{F}_{r}.$ Moreover,
each $\Sigma_{j}$ is the solution of $\partial\Sigma_{j}$ in Solution
\ref{sol} with respect to $\partial\Sigma_{j},$ say, $\Sigma_{j}%
=\Sigma_{\partial\Sigma_{j}}$,
\[
q_{1}^{\prime}+q_{2}^{\prime}+\dots+q_{k}^{\prime}=q^{\prime},
\]%
\[
\{C_{jl}:j=1,\dots,k,l=1,\dots,q_{j}^{\prime}\}=\{C_{j}:j=1,\dots,q^{\prime
}\},
\]
and%
\begin{equation}
\max_{j=1,\dots,k}\frac{R(\Sigma_{j})+4\pi}{L(\partial\Sigma_{j})}\geq
\frac{R(\Sigma)+4\pi}{L(\partial\Sigma)}. \label{bd15}%
\end{equation}

\end{corollary}

\begin{proof}
All the conclusions can be obtained by repeating the previous lemma and Claim
\ref{(a)} several times, except (\ref{bd15}). (\ref{bd15}) follows from
(\ref{bd13}), (\ref{bd14}) and Lemma \ref{ratio}.
\end{proof}

\begin{lemma}
\label{keykey} (i) There exists a surface $\Sigma\in\mathcal{F}_{r}$ which is
a $4\pi$-extremal surface $\mathcal{S}_{1}$, say,
\begin{equation}
H_{\mathcal{S}_{1}}=\sup_{\Sigma^{\prime}\in\mathcal{S}_{1}}\frac
{R(\Sigma^{\prime})+4\pi}{L(\partial\Sigma^{\prime})}=\frac{R(\Sigma)+4\pi
}{L(\partial\Sigma)}. \label{ko4}%
\end{equation}

(ii) For each surface $\Sigma$ in $\mathcal{S}_{1}$ satisfying (\ref{ko4}),
there exists a surface $\Sigma^{\prime}\ $in $\mathcal{S}_{0}\cap
\mathcal{F}_{r}$ such that%
\begin{equation}
\frac{R(\Sigma^{\prime})+4\pi}{L(\partial\Sigma^{\prime})}\geq\frac
{R(\Sigma)+4\pi}{L(\partial\Sigma)}=H_{\mathcal{S}_{1}}. \label{ko7}%
\end{equation}

(iii)
\[
H_{\mathcal{S}_{1}}=\sup_{\Sigma\in\mathcal{S}_{1}}\frac{R(\Sigma)+4\pi
}{L(\partial\Sigma)}=\sup_{\Sigma\in\mathcal{S}_{0}\cap\mathcal{F}_{r}}%
\frac{R(\Sigma)+4\pi}{L(\partial\Sigma)}=\sup_{\Sigma\in\mathcal{S}_{0}}%
\frac{R(\Sigma)+4\pi}{L(\partial\Sigma)}=H_{\mathcal{S}_{0}}.
\]

(iv) For each simplest $4\pi$-extremal surface $\Sigma$ of $\mathcal{S}_{0},$
its boundary $\partial\Sigma$ is consisted of $Q(\Sigma)$ strictly SCC arcs
(see Definitions \ref{S-surface} and \ref{4pi-extr}).
\end{lemma}

\begin{proof}
(i) Let
\[
H_{\mathcal{S}_{1}}=\sup_{\Sigma^{\prime}\in\mathcal{S}_{1}}\frac
{R(\Sigma^{\prime})+4\pi}{L(\partial\Sigma^{\prime})}.
\]
Then there exists a sequence $\Sigma_{n}$ in $\mathcal{S}_{1}$ such that
\[
\lim_{n\rightarrow\infty}\frac{R(\Sigma_{n})+4\pi}{L(\partial\Sigma_{n}%
)}=H_{\mathcal{S}_{1}}.
\]
By Definition of $\mathcal{S}_{1},$ we may assume $\Gamma_{n}=\partial
\Sigma_{n}$ has the partition
\[
\Gamma_{n}=\partial\Sigma_{n}=C_{n1}\left(  p_{1},p_{2}\right)  +C_{n2}\left(
p_{2},p_{3}\right)  +\dots+C_{nq^{\prime}}\left(  p_{q^{\prime}},p_{1}\right)
\]
such that $p_{1},\dots,p_{q^{\prime}}$ are distinct each other, the endpoints
set $\{p_{j}\}_{j=1}^{q^{\prime}}$ of all $C_{nj}$ are independent of $n\ $and
$\partial\Sigma_{n}$ with this partition satisfy Definition \ref{S-surface}
(1)--(4). Thus we may assume $C_{nj}\left(  p_{j},p_{j+1}\right)  $ converges
to an SCC arc $C_{j}\left(  p_{j},p_{j+1}\right)  $ for each $j=1,\dots
,q^{\prime}.$ It is clear that
\begin{equation}
\Gamma=C_{1}\left(  p_{1},p_{2}\right)  +\dots+C_{q^{\prime}}\left(
p_{q^{\prime}},p_{1}\right)  \label{bd6}%
\end{equation}
satisfies Definition \ref{S-surface} (1)--(3). Then for the surface
$\Sigma_{\Gamma}\in\mathcal{S}_{1}\cap\mathcal{F}_{r}$ and its partition
$\{K_{j}\}_{j=1}^{q^{\prime}}\cup\{T_{j}\}_{j=2}^{q^{\prime}-1}$ given by
Solution \ref{sol}, we have (\ref{bd4}) for $q^{\prime}\geq4,$ or (\ref{ag6})
for $q^{\prime}=2,3.$

For the surface $\Sigma_{\Gamma_{n}}$ and its partition $\{K_{nj}%
\}_{j=1}^{q^{\prime}}\cup\{T_{j}\}_{j=2}^{q^{\prime}-1}$ given by Solution
\ref{sol}, we may assume that for each $j,$ either $K_{nj}^{\circ}=\emptyset$
for all $n,$ or $K_{nj}^{\circ}\neq\emptyset$ for all $n.$ Then
\[
J_{n}=J\left(  \Sigma_{n}\right)  =\{j:K_{nj}^{\circ}\neq\emptyset
,j\in\{1,\dots,q^{\prime}\}\}=J
\]
is independent of $n,$ and Lemma \ref{makes}, (\ref{bd4}) and (\ref{ag6})
imply that\footnote{Note that $T_{j}$ is the same for each $\Sigma_{n},$ since
$p_{1},\dots,p_{q^{\prime}}$ are independent of $n.$}$\ $%
\begin{equation}%
\begin{tabular}
[c]{lll}%
$R(\Sigma_{n})$ & $\leq R(\Sigma_{\Gamma_{n}})$ & $=\sum\limits_{j=2}%
^{\max\{2,q^{\prime}-1\}}R\left(  T_{j}\right)  -4\pi\varphi\left(  q^{\prime
}\right)  $\\
&  & $+\sum_{j\in J}\left[  R(K_{nj})-4\pi\#\left(  \overline{p_{j}p_{j+1}%
}^{\circ}\cap E_{q}\right)  \right]  ,$%
\end{tabular}
\ \ \ \ \ \label{bd5}%
\end{equation}
where $\varphi(2)=\varphi(3)=0$, $\varphi(q^{\prime})=\sum_{j=2}^{q^{\prime
}-2}\#\left(  l_{j+1}^{\circ}\cap E_{q}\right)  $ for all $q^{\prime}\geq4,$ and

For each $j\in J_{n},$ we may assume $K_{nj}$ converges to the closed lune
$K_{j}$ enclosed by $C_{j}+\overline{p_{j+1}p_{j}}$, which as a limit may be
just the segment $\overline{p_{j}p_{j+1}}$ and we have
\begin{equation}
J_{n}\supset J=J(\Gamma)=\{j:K_{j}^{\circ}\neq\emptyset,j\in\{1,\dots
,q^{\prime}\}\}. \label{JnJ}%
\end{equation}
Then it is clear that $A(K_{nj})\rightarrow A(K_{j})$ and for each
sufficiently large $n$
\[
\overline{n}(K_{nj})\geq\overline{n}(K_{j}),
\]
which implies
\[
\lim_{n\rightarrow\infty}R(K_{nj})\leq R(K_{j}).
\]
Therefore by (\ref{bd5}) and (\ref{JnJ}), for the solution $\Sigma_{\Gamma}%
\in\mathcal{S}_{1}\cap\mathcal{F}_{r}$ to $\Gamma$, we have
\begin{align*}
\lim_{n\rightarrow\infty}R(\Sigma_{n})  &  \leq R(\Sigma_{\Gamma})\\
&  =\sum\limits_{j=2}^{\max\{2,q^{\prime}-1\}}R\left(  T_{j}\right)
-4\pi\varphi\left(  q^{\prime}\right)  +\sum_{j\in J}\left[  R(K_{j}%
)-4\pi\#\left(  \overline{p_{j}p_{j+1}}^{\circ}\cap E_{q}\right)  \right]  .
\end{align*}
On the other hand we have $\lim_{n\rightarrow\infty}L(\partial\Sigma
_{n})\rightarrow L(\partial\Sigma_{\Gamma}).$ Thus we have
\[
H_{\mathcal{S}_{1}}=\lim_{n\rightarrow\infty}\frac{R(\Sigma_{n})+4\pi
}{L(\partial\Sigma_{n})}\leq\frac{R(\Sigma_{\Gamma})+4\pi}{L(\partial
\Sigma_{\Gamma})}.
\]
Since $\Sigma_{\Gamma}\in\mathcal{S}_{1},$ we have
\[
H_{\mathcal{S}_{1}}=\frac{R(\Sigma_{\Gamma})+4\pi}{L(\partial\Sigma_{\Gamma}%
)}.
\]
(i) is proved.

(ii) Let $\Sigma$ be any surface in $\mathcal{S}_{1}\cap\mathcal{F}_{r}$
satisfying (i). Then $\partial\Sigma$ has a partition
\begin{equation}
\partial\Sigma=C_{1}\left(  p_{1},p_{2}\right)  +\dots+C_{q^{\prime}}\left(
p_{q^{\prime}},p_{1}\right)  \label{ko6}%
\end{equation}
satisfying Definition \ref{S-surface} (1)--(4). But (5) may fails.

$E_{q}$ divides each $C_{j}$ into subarcs each of which contains no point of
$E_{q}$ in its interior. We will show

\begin{claim}
\label{part-1}The partition (\ref{ko6}) has a refinement
\begin{equation}
\partial\Sigma=C_{1}^{\prime}\left(  p_{1}^{\prime},p_{2}^{\prime}\right)
+\cdots+C_{q^{\prime\prime}}^{\prime}\left(  p_{q^{\prime\prime}}^{\prime
},p_{1}^{\prime}\right)  , \label{bd11}%
\end{equation}
such that for each $j\leq q^{\prime\prime},$ $C_{j}^{\prime\circ}\cap
E_{q}=\emptyset,\partial C_{j}^{\prime}\subset E_{q}$ and the two endpoints of
$C_{j}^{\prime}$ are distinct for each $j=1,\dots,q^{\prime\prime}$ (note that
$q^{\prime\prime}$ may be larger than $q$, though $q^{\prime}\leq q$). In
other words, $\partial\Sigma$ with partition (\ref{bd11}) satisfies
$(p_{1}^{\prime},p_{2}^{\prime},\cdots,p_{q^{\prime\prime}}^{\prime}%
)$-Condition\footnote{It is clear that $d\left(  p_{j},p_{j+1}\right)  <\pi$
for all $j=1,\dots,q^{\prime}.$ But this cannot implies $d\left(
p_{j}^{\prime},p_{j+1}^{\prime}\right)  <\pi$ for all $j=1,\dots
,q^{\prime\prime}.$ So Claim \ref{part-1} needs a proof.}.
\end{claim}

We first show that

\begin{claim}
\label{ko5}All $C_{j}^{\prime},j=1,2,\dots,q^{\prime\prime},$ is contained in
an open hemisphere on $S$.
\end{claim}

Assume that Claim \ref{ko5} fails. Then for some $j_{0}\in\{1,\dots
,q^{\prime\prime}\},$ $C_{j_{0}}^{\prime}$ is not contained in any open
hemisphere and we may assume $j_{0}=1.$ Then by Lemma \ref{inhalf1} there
exists a surface $\Sigma^{\prime}$ such that
\[
\partial\Sigma^{\prime}=\overline{p_{1}^{\prime}\mathfrak{a}_{2}^{\prime}%
}+\overline{\mathfrak{a}_{2}^{\prime}\mathfrak{a}_{3}^{\prime}}+\cdots
+\overline{\mathfrak{a}_{s}^{\prime}p_{2}^{\prime}}+C_{2}^{\prime}%
+\dots+C_{q^{\prime\prime}}^{\prime},
\]%
\[
L(\partial\Sigma^{\prime})<L(\partial\Sigma),\frac{R(\Sigma^{\prime})+4\pi
}{L(\partial\Sigma^{\prime})}>\frac{R(\Sigma)+4\pi}{L(\partial\Sigma
)}=H_{\mathcal{S}_{1}}.
\]
We can repeat this method at most $q^{\prime\prime}-1$ times for every edge
$C_{j}^{\prime}$ which is not contained in any open hemisphere on $S,$ to
obtain a surface $\Sigma^{\prime\prime}$ such that $\partial\Sigma
^{\prime\prime}$ satisfy $(\mathfrak{a}_{1}^{\prime},\mathfrak{a}_{2}^{\prime
},\dots,\mathfrak{a}_{m}^{\prime})$-Condition and
\[
L(\partial\Sigma^{\prime\prime})\leq L(\partial\Sigma^{\prime}),\frac
{R(\Sigma^{\prime\prime})+4\pi}{L(\partial\Sigma^{\prime\prime})}\geq
\frac{R(\Sigma^{\prime})+4\pi}{L(\partial\Sigma^{\prime})}>\frac
{R(\Sigma)+4\pi}{L(\partial\Sigma)}.
\]
By Corollary \ref{make2}, there exists a surface $\Sigma^{\prime\prime\prime
}\in\mathcal{S}_{1}\cap\mathcal{F}_{r}\ $satisfying
\[
H_{\mathcal{S}_{1}}\geq\frac{R(\Sigma^{\prime\prime\prime})+4\pi}%
{L(\partial\Sigma^{\prime\prime\prime})}\geq\frac{R(\Sigma^{\prime\prime
})+4\pi}{L(\partial\Sigma^{\prime\prime})}>\frac{R(\Sigma)+4\pi}%
{L(\partial\Sigma)}=H_{\mathcal{S}_{1}}.
\]
This contradiction implies Claim \ref{ko5}.

By Claim \ref{ko5}, we have $d\left(  p_{j}^{\prime},p_{j+1}^{\prime}\right)
<\pi$ for all $j=1,\dots,q^{\prime\prime}.$ Thus Claim \ref{part-1} holds.

By Claim \ref{part-1} and Corollary \ref{make2}, there exists a surface
$\Sigma_{1}\in\mathcal{S}_{1}\cap\mathcal{F}_{r}$ such that
\begin{equation}
\frac{R(\Sigma_{1})+4\pi}{L(\partial\Sigma_{1})}\geq\frac{R(\Sigma)+4\pi
}{L(\partial\Sigma)}=H_{\mathcal{S}_{1}}, \label{ag53}%
\end{equation}
and $\partial\Sigma_{1}$ has a partition satisfying Definition \ref{S-surface}
(1)--(4), and each term of this partition is one of the arc in (\ref{bd11}),
and thus we have $\Sigma_{1}\in\mathcal{S}_{1}^{(5)}.$ Since $\Sigma_{1}%
\in\mathcal{S}_{1},$ (\ref{ag53}) in fact implies%
\begin{equation}
\frac{R(\Sigma_{1})+4\pi}{L(\partial\Sigma_{1})}=H_{\mathcal{S}_{1}}.
\label{ag54}%
\end{equation}
Let
\begin{equation}
\partial\Sigma_{1}=\mathcal{C}_{1}\left(  \mathfrak{p}_{1},\mathfrak{p}%
_{2}\right)  +\mathcal{C}_{2}\left(  \mathfrak{p}_{2},\mathfrak{p}_{3}\right)
+\cdots+\mathcal{C}_{\mathfrak{q}^{\prime}}\left(  \mathfrak{p}_{\mathfrak{q}%
^{\prime}},\mathfrak{p}_{1}\right)  \label{bd12}%
\end{equation}
be the partition of $\partial\Sigma_{1}$ such that each term is a term in
(\ref{bd11}). Then by Claim \ref{ko5} $\Sigma_{1}$ with this partition
satisfies Definition \ref{S-surface} (1)--(6), and thus $\Sigma_{1}%
\in\mathcal{S}_{1}^{(6)}$.

(\dag) Assume that $\Sigma_{1}\notin\mathcal{S}_{1}^{(7)},$ say, (\ref{bd12})
does not satisfies Definition \ref{S-surface} (7). Then there exists a pair
$\left(  i,j\right)  \ $with $i\neq j,$ such that the curvature of
$\mathcal{C}_{i}$ is not the same as that of $\mathcal{C}_{j}.$ By Corollary
\ref{2-curvature} we can change $\mathcal{C}_{i}$ and $\mathcal{C}_{j}$ a
little so that $\Sigma_{1}$ becomes another surface $\Sigma_{1}^{\prime}$ in
$\mathcal{S}_{1}$ such that $R(\Sigma_{1}^{\prime})>R(\Sigma_{1})$ and
$L(\partial\Sigma_{1}^{\prime})=L(\partial\Sigma_{1}),$ which deduce the
contradiction $H(\Sigma_{1}^{\prime})>H_{\mathcal{S}_{1}}.$ We have proved
that $\Sigma_{1}\in\mathcal{S}_{1}^{(7)}.$

Assume $\Sigma_{1}\notin\mathcal{S}_{1}^{(8)}.$ Then in the partition
(\ref{bd12}), there is an edge $\mathcal{C}_{j_{0}}$ which is a major circular
arc. Then by Lemma \ref{makes}, we may assume that $\Sigma_{1}$ is the surface
$\Sigma_{\partial\Sigma_{1}}$ given by Solution \ref{sol} and is obtained by
sewing\label{sew17} $\{\mathcal{K}_{j}\}$ and $\{\mathcal{T}_{j}\},$ where
$\mathcal{K}_{j}$ and $\mathcal{T}_{j}$ are defined by the partition
(\ref{bd12}) as in Solution \ref{sol} (as $K_{j}$ and $T_{j}$ there). Let
$\Sigma_{1}^{\prime}$ and $\Sigma_{2}^{\prime}$ be two surface in
$\mathcal{S}_{1}$ both obtained by deform $\Sigma_{1}^{\prime}$ via pushing
$\mathcal{C}_{j_{0}}$ to the right hand side a little, respectively, so that
the endpoints of $\mathcal{C}_{j_{0}}$ remain unchanged, $\overline{n}%
(\Sigma_{1}^{\prime})=\overline{n}\left(  \Sigma_{2}^{\prime}\right)
=\overline{n}\left(  \Sigma_{1}\right)  $ and $L(\partial\Sigma_{1}^{\prime
})+L(\partial\Sigma_{2}^{\prime})=2L(\partial\Sigma_{1}).$ Then we have
$A(\Sigma_{1}^{\prime})+A(\Sigma_{2}^{\prime})>2A\left(  \Sigma_{1}\right)
\ $by Corollary \ref{2-curvature} (B). Thus we have%
\[
\frac{R(\Sigma_{1}^{\prime})+4\pi+R(\Sigma_{2}^{\prime})+4\pi}{L(\partial
\Sigma_{1}^{\prime})+L\left(  \partial\Sigma_{2}^{\prime}\right)  }%
>\frac{2\left(  R(\Sigma_{1})+4\pi\right)  }{2L(\partial\Sigma_{1})}.
\]
which, with Lemma \ref{ratio} and (\ref{ag54}), implies
\[
\max\left\{  \frac{R(\Sigma_{1}^{\prime})+4\pi}{L(\partial\Sigma_{1}^{\prime
})},\frac{R(\Sigma_{2}^{\prime})+4\pi}{L\left(  \partial\Sigma_{2}^{\prime
}\right)  }\right\}  >\frac{R(\Sigma_{1})+4\pi}{L(\partial\Sigma_{1}%
)}=H_{\mathcal{S}_{1}},
\]
contradicting definition of $H_{\mathcal{S}_{1}}$ again. We have proved
$\Sigma_{1}\in\mathcal{S}_{1}^{(8)}.$

Up to now, we have proved $\Sigma_{1}$ satisfies (1)--(8) of Definition
\ref{S-surface}, say $\Sigma_{1}\in\mathcal{S}_{0}.$ Then we have by
(\ref{ag54})
\[
H_{\mathcal{S}_{1}}\geq H_{\mathcal{S}_{0}}\geq\frac{R(\Sigma_{1})+4\pi
}{L(\partial\Sigma_{1})}=H_{\mathcal{S}_{1}}.
\]
and (ii) is proved, since $\Sigma_{1}\in\mathcal{F}_{r}$. (iii) follows from
(ii) directly.

To prove (iv), let $\Sigma\in\mathcal{S}_{0}$ be a simplest $4\pi$-extremal
surface of $\mathcal{S}_{0}$ with the corresponding partition (\ref{ko12})
satisfying Definition \ref{S-surface} (1)--(8). If (iv) fails for $\Sigma$,
then all arcs $C_{j}$ in the partition (\ref{ko12}) are straight, say,
$C_{j}=\overline{p_{j}p_{j+1}},j=1,\dots,q^{\prime}.$ We may assume $\Sigma$
is the solution of Solution \ref{sol} from the partition (\ref{ko12}). Then
$\Sigma$ is obtained by sewing\label{sew16} $P=P\left(  p_{1},p_{2}%
,\dots,p_{q^{\prime}}\right)  $ and $\{K_{j}\}_{j=1}^{q^{\prime}}$ given by
Solution \ref{sol}, each $K_{j}$ is just the line segment $\overline
{p_{j}p_{j+1}}$ and we have, by Definition \ref{S-surface} (5),%
\begin{equation}
\overline{p_{j}p_{j+1}}^{\circ}\cap E_{q}=C_{j}^{\circ}\cap E_{q}=\emptyset.
\label{ag55}%
\end{equation}
Then by (\ref{ag22})%
\begin{equation}
\frac{1}{H_{\mathcal{S}_{1}}}=\frac{L(\partial\Sigma)}{R(P)+4\pi}=\frac
{2\sum_{j=1}^{q^{\prime}}L(\overline{p_{j}p_{j+1}})}{2R(P)+8\pi}. \label{ag25}%
\end{equation}

It is clear that $R(P)+4\pi>0$. Let $\theta_{1},\theta_{2},\dots
,\theta_{q^{\prime}}$ be sufficiently small positive angles such that the
closed lunes $K_{j}^{\prime}=\overline{\mathfrak{D}^{\prime}(\overline
{p_{j}p_{j+1}},\theta_{j})}$ (see Definition \ref{lune-lens}) have the same
area, say,
\begin{equation}
A(K_{j}^{\prime})=A(K_{1}^{\prime}),j=1,2,\dots,q^{\prime}, \label{ag23}%
\end{equation}
and let $C_{j}^{\prime}$ be the circular boundary of $\mathfrak{D}^{\prime
}\left(  \overline{p_{j}p_{j+1}},\theta_{j}\right)  $ from $p_{j}$ to
$p_{j+1}.$ Then all $\theta_{j}$ are determined by $\theta_{1}$ and we can sew
\label{sew4}$P$ and the lunes $K_{j}^{\prime}$ along $\partial P$ to obtain a
surface $\Sigma_{\theta_{1}}$ such that
\[
\partial\Sigma_{\theta_{1}}=C_{1}^{\prime}\left(  p_{1},p_{2}\right)
+C_{2}^{\prime}\left(  p_{2},p_{3}\right)  +\cdots+C_{q^{\prime}}^{\prime
}\left(  p_{q^{\prime}},p_{1}\right)  .
\]

By (\ref{ag55}) when $\theta_{1}$ is small enough we have that
\[
R(K_{j}^{\prime})=\left(  q-2\right)  A(K_{j}^{\prime})
\]
and that%
\[
R(\Sigma_{\theta_{1}})=R(P)+\sum_{j=1}^{q^{\prime}}R(K_{j}^{\prime}%
)=R(P)+\sum_{j=1}^{q^{\prime}}\left(  q-2\right)  A(K_{j}^{\prime}).
\]
Then we have by (\ref{ag23})%
\begin{align*}
R(\Sigma_{\theta_{1}})+4\pi &  =R(P)+4\pi+q^{\prime}\left(  q-2\right)
A(K_{1}^{\prime})\\
&  =q^{\prime}\left(  q-2\right)  \left[  \frac{R(P)+4\pi}{q^{\prime}\left(
q-2\right)  }+A\left(  K_{1}^{\prime}\right)  \right]  ,
\end{align*}
and then by (\ref{ag23})%
\begin{align*}
\frac{L\left(  \partial\Sigma_{\theta_{1}}\right)  }{R(\Sigma_{\theta_{1}%
})+4\pi}  &  =\frac{\sum_{j=1}^{q^{\prime}}L\left(  C_{j}^{\prime}\right)
}{q^{\prime}\left(  q-2\right)  \left[  \frac{R(P)+4\pi}{q^{\prime}\left(
q-2\right)  }+A\left(  K_{1}^{\prime}\right)  \right]  }\\
&  =\frac{1}{q^{\prime}\left(  q-2\right)  }\sum_{j=1}^{q^{\prime}}%
\frac{2L\left(  C_{j}^{\prime}\right)  }{\frac{2R(P)+8\pi}{q^{\prime}\left(
q-2\right)  }+2A\left(  K_{1}^{\prime}\right)  }\\
&  =\frac{1}{q^{\prime}\left(  q-2\right)  }\sum_{j=1}^{q^{\prime}}%
\frac{2L\left(  C_{j}^{\prime}\right)  }{\frac{2R(P)+8\pi}{q^{\prime}\left(
q-2\right)  }+2A\left(  K_{j}^{\prime}\right)  },
\end{align*}
which, with $2L\left(  C_{j}^{\prime}\right)  =L\left(  \partial
\mathfrak{D}\left(  \overline{p_{j}p_{j+1}},\theta_{j},\theta_{j}\right)
\right)  $ and $2A\left(  K_{j}^{\prime}\right)  =A\left(  \mathfrak{D}\left(
\overline{p_{j}p_{j+1}},\theta_{j},\theta_{j}\right)  \right)  ,$ implies%
\[
\frac{L\left(  \partial\Sigma_{\theta_{1}}\right)  }{R(\Sigma_{\theta_{1}%
})+4\pi}=\frac{1}{q^{\prime}\left(  q-2\right)  }\sum_{j=1}^{q^{\prime}}%
\frac{L\left(  \partial\mathfrak{D}\left(  \overline{p_{j}p_{j+1}},\theta
_{j},\theta_{j}\right)  \right)  }{\frac{2R(P)+8\pi}{\left(  q-2\right)
q^{\prime}}+A\left(  \mathfrak{D}\left(  \overline{p_{j}p_{j+1}},\theta
_{j},\theta_{j}\right)  \right)  }.
\]
Therefore by Lemma \ref{must-strict-con} (i) we have for small enough
$\theta_{j}=\theta_{j}\left(  \theta_{1}\right)  ,$%
\begin{align*}
\frac{L\left(  \partial\Sigma_{\theta_{1}}\right)  }{R(\Sigma_{\theta_{1}%
})+4\pi}  &  <\frac{1}{q^{\prime}\left(  q-2\right)  }\sum_{j=1}^{q^{\prime}%
}\frac{L\left(  \partial\mathfrak{D}\left(  \overline{p_{j}p_{j+1}%
},0,0\right)  \right)  }{\left[  \frac{2R(P)+8\pi}{\left(  q-2\right)
q^{\prime}}+A\left(  \mathfrak{D}\left(  \overline{p_{j}p_{j+1}},0,0\right)
\right)  \right]  }\\
&  =\frac{1}{\left(  q-2\right)  q^{\prime}}\sum_{j=1}^{q^{\prime}}%
\frac{2L\left(  \overline{p_{j}p_{j+1}}\right)  }{\left[  \frac{2R(P)+8\pi
}{\left(  q-2\right)  q^{\prime}}+0\right]  }\\
&  =\frac{\sum_{j=1}^{q^{\prime}}L\left(  \overline{p_{j}p_{j+1}}\right)
}{R(P)+4\pi}=\frac{L(\partial\Sigma)}{R(\Sigma)+4\pi}.
\end{align*}
Then by (\ref{ag25}) we have
\[
H_{\mathcal{S}_{1}}=\frac{R(\Sigma)+4\pi}{L(\partial\Sigma)}<\frac
{R(\Sigma_{\theta_{1}})+4\pi}{L\left(  \partial\Sigma_{\theta_{1}}\right)  }.
\]
But this contradicts the definition of $H_{\mathcal{S}_{1}},$ since
$\Sigma_{\theta_{1}}\in\mathcal{S}_{1}.$ Thus (iv) holds.
\end{proof}

Using the argument of the paragraph (\dag) in the above proof, we in fact have
proved Theorem \ref{main2} (iii):

\begin{corollary}
\label{same}Assume $F_{1}$ and $F_{2}$ are any two $4\pi$-extremal surfaces in
$\mathcal{S}_{0}.$ Then
\[
k\left(  F_{1},E_{q}\right)  =k\left(  F_{2},E_{q}\right)  .
\]

\end{corollary}

\begin{proof}
Since $H_{\mathcal{S}_{0}}=H_{\mathcal{S}_{1}},$ both $F_{1}$ and $F_{2}$
satisfy (\ref{ag54}) and each $\partial F_{j}$ has s partition $\partial
F_{j}=\mathcal{C}_{j1}+\mathcal{C}_{j2}+,\dots,\mathcal{C}_{jk_{j}}$
satisfying (1)--(8) of Definition \ref{S-surface}. If for some $i_{1}$ and
$i_{2},$ $\mathcal{C}_{1i_{1}}$ and $\mathcal{C}_{2i_{2}}$ have different
curvature, then applying the discussion in the paragraph (\dag) for
$\Sigma_{1}\notin\mathcal{S}_{1}^{(7)},$ we can construct two surfaces
$F_{1}^{\prime}$ and $F_{2}^{\prime}$ contained in $\mathcal{S}_{1}$ by
deforming $\mathcal{C}_{1i_{1}}$ and $\mathcal{C}_{2i_{2}},$ as we do for
$\mathcal{C}_{i}$ and $\mathcal{C}_{j}$ there, so that $R(F_{1}^{\prime
})+R(F_{2}^{\prime})>R(F_{1})+R(F_{2})$ and $L(\partial F_{1}^{\prime
})+R(\partial F_{2}^{\prime})=L(\partial F_{1})+R(\partial F_{2}).$ Then we
can use Lemma \ref{ratio} to show that $\frac{R(F_{j}^{\prime})+4\pi
}{L(\partial F_{j}^{\prime})}>H_{\mathcal{S}_{1}}$ for $j=1$ or $2.$ This is a contradiction.
\end{proof}

\begin{lemma}
\label{ag16}For any closed convex disk $T$ on $S,$
\[
H_{\mathcal{S}_{0}}=\sup_{\Sigma\in\mathcal{S}_{0}}\frac{R(\Sigma)+4\pi
}{L(\partial\Sigma)}>H(T).
\]

\end{lemma}

\begin{proof}
Since $T$ is convex, it is contained in a hemisphere $S_{1}$ on $S.$ It is
clear that there exists a disk $T_{1}$ contained in $S\backslash E_{q}$ whose
boundary is the circumcircle of a regular triangle of edge length
$\delta_{E_{q}}.$ Then we may assume that two points $p_{1}$ and $p_{2}$ of
$E_{q}$ are contained in $\partial T_{1}$ with $d\left(  p_{1},p_{2}\right)
=\delta_{E_{q}},$ and then we have $T_{1}\in\mathcal{S}_{1}.$ Thus by Lemma
\ref{keykey} (iii) we have%
\[
H_{\mathcal{S}_{0}}=H_{\mathcal{S}_{1}}=\sup_{\Sigma\in\mathcal{S}_{1}}%
\frac{R(\Sigma)+4\pi}{L(\partial\Sigma)}\geq\frac{R(T_{1})+4\pi}{L(\partial
T_{1})}=\frac{\left(  q-2\right)  A(T_{1})+4\pi}{L(\partial T_{1})}.
\]
Then in the case $L(\partial T)\leq L(\partial T_{1}),$ by Lemma \ref{hd} we
have%
\[
H_{\mathcal{S}_{0}}\geq\frac{\left(  q-2\right)  A(T_{1})+4\pi}{L(\partial
T_{1})}>\frac{\left(  q-2\right)  A(T_{1})}{L(\partial T_{1})}=H(T_{1})\geq
H(T).
\]
In the case $L(\partial T_{1})\leq L(\partial T)\leq2\pi,$ we may rotate $S$
to move $T$ to another congruent disk $T^{\prime}$ such that $\overline
{n}\left(  T^{\prime}\right)  \leq\overline{n}(T)$ and $\partial T^{\prime}%
\ $contains at least two points of $E_{q}$ with distance $<\pi,$ and thus
\[
H(T)\leq H(T^{\prime})=\frac{R(T^{\prime})}{L(\partial T^{\prime})}%
<\frac{R(T^{\prime})+4\pi}{L(\partial T^{\prime})}\leq H_{\mathcal{S}_{1}%
}=H_{\mathcal{S}_{0}}.
\]

\end{proof}

We are in the position to complete the proof of our second main theorem.

\begin{proof}
[Proof of Theorem \ref{main2}]The conclusion (iii) is obtained by Corollary
\ref{same} directly.

Let $\Sigma$ be any surface in $\mathcal{F}$. We first show
\begin{equation}
H(\Sigma)\leq H_{\mathcal{S}_{0}}. \label{ag166}%
\end{equation}

By Theorem \ref{l<2dt} and Lemma \ref{ag16}, \ref{ag166} holds when
$L\leq2\delta_{E_{q}}.$ Thus to prove (\ref{ag166}) we may assume
$L(\partial\Sigma)>2\delta_{E_{q}}.$

Let $L$ be any positive number in $\mathcal{L}$ with $L\geq L(\partial
\Sigma).$ Then by Theorem \ref{okok} (A) there exists a precise extremal
surface $\Sigma_{L_{1}}$ of $\mathcal{F}\left(  L\right)  $ such that
$L(\partial\Sigma_{L_{1}})=L_{1}.$ Then we have $H(\Sigma)\leq H(\Sigma
_{L_{1}})$ and thus by definition of $H$%
\begin{equation}
H(\Sigma)\leq H(\Sigma_{L_{1}})=\frac{R(\Sigma_{L_{1}})}{L(\partial
\Sigma_{L_{1}})}<\frac{R(\Sigma_{L_{1}})+4\pi}{L(\partial\Sigma_{L_{1}})}.
\label{ag57}%
\end{equation}
By Theorem \ref{okok} (B), for some positive integer $n_{0},$ $\partial
\Sigma_{L_{1}}$ has an $\mathcal{F}\left(  L,n_{0}\right)  $-partition%
\begin{equation}
\partial\Sigma_{L_{1}}=C_{1}\left(  p_{1},p_{2}\right)  +C_{2}\left(
p_{2},p_{3}\right)  +\cdots+C_{n_{0}}\left(  p_{n_{0}},p_{1}\right)
\label{ko1}%
\end{equation}
satisfying (B1)--(B4) of Theorem \ref{okok}, which implies that

\begin{claim}
\label{S(1)(2)} $\Sigma_{L_{1}}\ $with partition (\ref{ko1}) satisfies
$\left(  p_{1},p_{2},\cdots,p_{n_{0}}\right)  $-Condition when $n_{0}>1.$
\end{claim}

Consider the case that $\partial\Sigma_{L_{1}}$ contains at most one point of
$E_{q}.$ Then $C_{1}$ is an SCC circle and $n_{0}=1,$ and thus for the closed
disk $T$ enclosed by $C_{1}$ we have by Lemma \ref{loop2} that%
\[
H(\Sigma_{L_{1}})=\frac{R(\Sigma_{L_{1}})}{L(\partial\Sigma_{L_{1}})}\leq
H(T),
\]
and thus by Lemma \ref{ag16} and (\ref{ag57}) we have%
\[
H(\Sigma)\leq H(\Sigma_{L_{1}})\leq H(T)<H_{\mathcal{S}_{0}}.
\]
Therefore (\ref{ag166}) holds when $n_{0}=1.$

Assume $n_{0}\geq2.$ Then Claim \ref{S(1)(2)} and Corollary \ref{make2} imply
that there exists a surface $\Sigma_{1}\in\mathcal{S}_{1}$ such that
\[
\frac{R(\Sigma_{L_{1}})+4\pi}{L(\partial\Sigma_{L_{1}})}\leq\frac{R(\Sigma
_{1})+4\pi}{L(\partial\Sigma_{1})}\leq H_{\mathcal{S}_{1}},
\]
which, with (\ref{ag57}) and Lemma \ref{keykey}, implies%
\[
H\left(  \Sigma\right)  <\frac{R(\Sigma_{L_{1}})+4\pi}{L(\partial\Sigma
_{L_{1}})}\leq\frac{R(\Sigma_{1})+4\pi}{L(\partial\Sigma_{1})}\leq
H_{\mathcal{S}_{1}}=H_{\mathcal{S}_{0}},
\]
and (\ref{ag166}) follows.

By Lemma \ref{keykey}, there exists a surface $\Sigma_{0}\in\mathcal{S}_{0}$
such that
\[
\frac{R(\Sigma_{0})+4\pi}{L(\partial\Sigma_{0})}=H_{\mathcal{S}_{0}}.
\]
Then $\partial\Sigma_{0}$ has a partition
\[
\partial\Sigma_{0}=C_{1}\left(  p_{1},p_{2}\right)  +C_{2}\left(  p_{2}%
,p_{3}\right)  +\cdots+C_{q^{\prime}}\left(  p_{q^{\prime}},p_{1}\right)
\]
satisfying Definition \ref{S-surface} (1)--(8). Then by Corollary \ref{LSZ2}
there exists a sequence $\Sigma_{n}$ in $\mathcal{F}$ such that
\[
\lim_{n\rightarrow\infty}H(\Sigma_{n})=\frac{R(\Sigma_{0})+4\pi}%
{L(\partial\Sigma_{0})}=H_{\mathcal{S}_{0}}.
\]
Thus by (\ref{ag58}) we have proved $H_{0}=\sup_{\Sigma\in\mathbf{F}}%
H(\Sigma)=\sup_{\Sigma\in\mathcal{F}}H(\Sigma)=H_{\mathcal{S}_{0}},$ and
Theorem \ref{main2} (i) is proved. Theorem \ref{main2} (ii) follows from Lemma
\ref{keykey} (iv) directly.
\end{proof}

\begin{proof}
[Proof of Theorem \ref{4pi-sim}]Let $\mathfrak{S}$ be the space of all $4\pi
$-extremal surfaces of $\mathcal{S}_{0}.$ By Theorem \ref{main2} (i),
$\mathfrak{S}\neq\emptyset$ and $Q\left(  \Sigma\right)  \leq q$ for all
$\Sigma\in\mathfrak{S}.$ Then there exists $F\in\mathfrak{S}$ such that
$q^{\prime}=Q(F)=\min_{\Sigma\in\mathfrak{S}}Q(\Sigma).$ Let $\mathfrak{S}%
_{q^{\prime}}$ be the subspace of $\mathfrak{S}$ such that $Q(\Sigma
)=q^{\prime}$ for all $\Sigma\in\mathfrak{S}_{q^{\prime}}.$ Then for every
surface $\Sigma\in\mathfrak{S}_{q^{\prime}},$%
\[
q^{\prime}\delta_{E_{q}}\leq L(\partial\Sigma)<2\pi q^{\prime},
\]
and then $L=\inf_{\Sigma\in\mathfrak{S}_{q^{\prime}}}L(\partial\Sigma)\geq
q^{\prime}\delta_{E_{q}}\geq2\delta_{E_{q}}.$ Then there exists a sequence
$\Sigma_{n}$ in $\mathfrak{S}_{q^{\prime}}$ such that%
\[
L(\partial\Sigma_{n})=L,
\]
and $\Sigma_{n}$ has the partition
\begin{equation}
\partial\Sigma_{n}=C_{n1}\left(  p_{1},p_{2}\right)  +C_{n2}\left(
p_{2},p_{3}\right)  +\cdots+C_{nq^{\prime}}\left(  p_{q^{\prime}}%
,p_{1}\right)  \label{ag27}%
\end{equation}
satisfying Definition \ref{S-surface} (1)--(8) and Theorem \ref{main2} (ii),
where $p_{1},\dots,p_{q^{\prime}}$ are independent of $n\ $and distinct each other.

If $\Sigma_{n}\notin\mathcal{F}_{r}$ for some $n$, then by Lemma \ref{makes}
(i), for the surface $\Sigma_{\partial\Sigma_{n}}\in\mathcal{S}_{1}$ given by
Solution \ref{sol} from $\partial\Sigma_{n}$ and the partition (\ref{ag27}),
we have $R(\Sigma_{n})<R(\Sigma_{\partial\Sigma_{n}}),$ which implies
\[
H_{\mathcal{S}_{0}}=H(\Sigma_{n})<H(\Sigma_{\partial\Sigma_{n}})\leq
H_{\mathcal{S}_{1}}.
\]
This contradicts $H_{\mathcal{S}_{0}}=H_{\mathcal{S}_{1}}.$ Thus $\Sigma
_{n}\in\mathcal{F}_{r}$ for all $n.$

By Corollary \ref{same}, $k\left(  \Sigma_{n},E_{q}\right)  =k$ is a constant
for all $n=1,2,\dots,$ say, all $C_{nj}$ have the same curvature, and then
$C_{nj}=C_{n1}$ for all $n$ and all $j\leq q^{\prime}.$ Thus, $\partial
\Sigma_{n}=\partial\Sigma_{1}$ and $L(\partial\Sigma_{1})=L$ and so
$\Sigma_{1}$ $\in\mathfrak{S}_{q^{\prime}}$ satisfies Definition
\ref{4pi-extr} (1) and (2). Since $\deg_{\max}\Sigma^{\prime}\leq2q^{\prime
}-2$ for all $\Sigma^{\prime}\in\mathfrak{S}_{q^{\prime}},$ there exists
$\Sigma\in\mathfrak{S}_{q^{\prime}}\emph{\ }$such that $\deg_{\max}\Sigma
\leq\deg_{\max}\Sigma^{\prime}$ for all $\Sigma^{\prime}\in\mathfrak{S}%
_{q^{\prime}}.$ Therefore $\mathcal{S}^{\ast}\neq\emptyset.$
\end{proof}

\begin{proof}
[\textbf{Proof of Theorem \ref{2,3}}]We first prove (i). Let $\Sigma
\in\mathcal{S}^{\ast}$ and assume $Q(\Sigma)=2.$ Then $\partial\Sigma$ has a
partition $\partial\Sigma=C_{1}\left(  p_{1},p_{2}\right)  +C_{2}\left(
p_{2},p_{1}\right)  $ satisfying Definition \ref{S-surface} (8), and $C_{1}$
and $C_{2}$ are both strictly convex, by Theorem \ref{main2} (ii). Then
$\partial\Sigma$ encloses a closed lens $\Sigma^{\prime}=\overline
{\mathfrak{D}(\overline{p_{1}p_{2}},\theta_{0},\theta_{0})}$ with $\theta
_{0}\in(0,\frac{\pi}{2}]$, and by Corollary \ref{loop2} we have%
\[
R(\Sigma)\leq R(\Sigma^{\prime})
\]
and thus $R(\Sigma)=R(\Sigma^{\prime})$ and $\Sigma^{\prime}\in\mathcal{S}%
^{\ast}.$ Therefore $\Sigma\ $and $\Sigma^{\prime}\ $both equal the lens
$\overline{\mathfrak{D}(\overline{p_{1}p_{2}},\theta_{0},\theta_{0})},\ $by
Corollary \ref{same}, and we have
\[
H_{0}=H_{\mathcal{S}_{0}}=\frac{R(\Sigma)+4\pi}{L(\partial\Sigma)}%
=\frac{R(\Sigma^{\prime})+4\pi}{L(\partial\Sigma^{\prime})}=\frac{R\left(
\mathfrak{D}(\overline{p_{1}p_{2}},\theta_{0},\theta_{0})\right)  +4\pi
}{L\left(  \partial\mathfrak{D}\left(  \overline{p_{1}p_{2}},\theta_{0}%
,\theta_{0}\right)  \right)  }.
\]

We will show%
\begin{equation}
d\left(  p_{1},p_{2}\right)  =\delta_{E_{q}}. \label{d=d}%
\end{equation}
Assume this fails. We will deduce a contradiction.

We first show that $\theta_{0}<\pi/2$ when $d\left(  p_{1},p_{2}\right)
>\delta_{E_{q}}.$ Assume $\theta_{0}=\frac{\pi}{2}.$ Then $T_{p_{1},p_{2}%
}=\overline{\mathfrak{D}(\overline{p_{1}p_{2}},\pi/2,\pi/2)}$ is a disk. Let
$\delta_{0}=L(\overline{p_{1}p_{2}}).$ Then it is clear that for an open
neighborhood $I_{\delta_{0}}^{\circ}=\left(  \delta_{0}-\varepsilon,\delta
_{0}+\varepsilon\right)  $ of $\delta_{0}$ with $\delta_{0}-\varepsilon
>\delta_{E_{q}},$ there exists a family
\[
\mathcal{T}_{I_{\delta_{0}}^{\circ}}=\left\{  \overline{\mathfrak{D}%
(\overline{p_{1}p_{2,\delta}},\pi/2,\pi/2)}:\delta=d\left(  p_{1},p_{2,\delta
}\right)  \in I_{\delta_{0}}^{\circ},\right\}
\]
of convex disks on $S$ such that $\overline{n}\left(  T_{p_{1},p_{2,\delta}%
}\right)  =\overline{n}\left(  T_{p_{1},p_{2}}\right)  $ is a constant for
each $\delta\in I_{\delta_{0}}^{\circ}.$ In fact, $\partial T_{p_{1},p_{2}}$
contains only two points $p_{1},p_{2}$ in $E_{q}.$ Then rotating
$T_{p_{1},p_{2}}$ a little about $p_{1}$ we can obtain the desired family.
Then by Lemma \ref{disk} (ii4) there exists $\delta_{1}\in I_{\delta_{0}%
}^{\circ}$ and a disk in the family which can be write as $T_{p_{1}%
,p_{2,\delta_{1}}}=\overline{\mathfrak{D}(\overline{p_{1}p_{2,\delta_{1}}}%
,\pi/2,\pi/2)}$ with $\delta_{1}=d\left(  p_{1},p_{2,\delta_{1}}\right)  \in
I_{\delta_{0}}^{\circ}$ such that
\begin{equation}
H_{\mathcal{S}_{0}}=\frac{R\left(  T_{p_{1},p_{2}}\right)  +4\pi}{L\left(
\partial T_{p_{1},p_{2}}\right)  }<\frac{R\left(  T_{p_{1},p_{2,\delta_{1}}%
}\right)  +4\pi}{L\left(  \partial T_{p_{1},p_{2,\delta_{1}}}\right)  }.
\label{>HS0}%
\end{equation}
Since $\delta_{1}>\delta_{0}-\varepsilon>\delta_{E_{q}},$ the diameter of
$\partial T_{p_{1},p_{2,\delta_{1}}}$ is larger than $\delta_{E_{q}},$ and
then we may move the disk $T_{p_{1},p_{2,\delta_{1}}}$ congruently to another
closed disk $T_{\delta_{1}}$ so that its boundary contains at least two points
of $E_{q}\ $and
\[
\overline{n}\left(  T_{\delta_{1}}\right)  \leq\overline{n}\left(
T_{p_{1},p_{2,\delta}}\right)  .
\]
Then we have $T_{\delta_{1}}\in\mathcal{S}_{1}$ and
\[
\frac{R\left(  T_{p_{1},p_{2,\delta_{1}}}\right)  +4\pi}{L\left(  \partial
T_{p_{1},p_{2,\delta_{1}}}\right)  }\leq\frac{R\left(  T_{\delta_{1}}\right)
+4\pi}{L\left(  \partial T_{\delta_{1}}\right)  }\leq H_{\mathcal{S}_{1}},
\]
which with (\ref{>HS0}) implies $H_{\mathcal{S}_{1}}>H_{\mathcal{S}_{0}}.$
This is a contradiction by Lemma \ref{keykey} (iii). We have proved
$\theta_{0}<\pi/2\ $when $d\left(  p_{1},p_{2}\right)  >\delta_{E_{q}}.$

Now that we have proved $\theta_{0}<\pi/2,$ there exists a point
$p_{2}^{\prime}\in\overline{p_{1}p_{2}}^{\circ}$ near $p_{2}$ such that
\[
d\left(  p_{1},p_{2}^{\prime}\right)  >\delta_{E_{q}},
\]%
\[
L(\partial\mathfrak{D}\left(  \overline{p_{1}p_{2}^{\prime}},\theta
_{p_{2}^{\prime}},\theta_{p_{2}^{\prime}}\right)  )=L(\mathfrak{D}%
(\overline{p_{1}p_{2}},\theta_{0},\theta_{0})),
\]%
\[
\overline{n}\left(  \mathfrak{D}\left(  \overline{p_{1}p_{2}^{\prime}}%
,\theta_{p_{2}^{\prime}},\theta_{p_{2}^{\prime}}\right)  \right)
=\overline{n}\left(  \mathfrak{D}\left(  \overline{p_{1}p_{2}},\theta
_{0},\theta_{0}\right)  \right)  ,
\]
and%
\[
\theta_{0}<\theta_{p_{2}^{\prime}}<\frac{\pi}{2}.
\]
Then by Lemma \ref{area2}, putting $A_{0}=4\pi-4\pi\overline{n}\left(
\mathfrak{D}(\overline{p_{1}p_{2}},\theta_{0},\theta_{0})\right)  ,$ we have%
\begin{equation}
H_{\mathcal{S}_{0}}=\frac{R\left(  \mathfrak{D}(\overline{p_{1}p_{2}}%
,\theta_{0},\theta_{0})\right)  +4\pi}{L\left(  \partial\mathfrak{D}\left(
\overline{p_{1}p_{2}},\theta_{0},\theta_{0}\right)  \right)  }<\frac{R\left(
\mathfrak{D}(\overline{p_{1}p_{2}^{\prime}},\theta_{p_{2}^{\prime}}%
,\theta_{p_{2}^{\prime}})\right)  +4\pi}{L\left(  \partial\mathfrak{D}\left(
\overline{p_{1}p_{2}^{\prime}},\theta_{p_{2}^{\prime}},\theta_{p_{2}^{\prime}%
}\right)  \right)  }. \label{>HS0-1}%
\end{equation}

Now we let $F$ be the closed domain $\overline{\mathfrak{D}(\overline
{p_{1}p_{2}^{\prime}},\theta_{p_{2}^{\prime}},\theta_{p_{2}^{\prime}})}.$ Then
$F$ has diameter larger than $\delta_{E_{q}}\ $and $\partial F\cap
E_{q}=\{p_{1}\},$ but we can rotate $F$ on $S$ to another domain $F^{\prime}$
so that $H(F^{\prime})\geq H(F)$ and $\partial F^{\prime}$ contains two points
$q_{1}$ and $q_{2}$ of $E_{q}.$ Thus $\partial F^{\prime}$ satisfies (2) of
Lemma \ref{LSZ} and thus we have by (\ref{>HS0-1}) that%
\[
H_{0}\geq\frac{R(F^{\prime})+4\pi}{L(\partial F^{\prime})}\geq\frac{R(F)+4\pi
}{L(\partial F)}=\frac{R\left(  \mathfrak{D}(\overline{p_{1}p_{2}^{\prime}%
},\theta_{p_{2}^{\prime}},\theta_{p_{2}^{\prime}})\right)  +4\pi}{L\left(
\partial\mathfrak{D}\left(  \overline{p_{1}p_{2}^{\prime}},\theta
_{p_{2}^{\prime}},\theta_{p_{2}^{\prime}}\right)  \right)  }>H_{\mathcal{S}%
_{0}}.
\]
This contradicts Theorem \ref{main2} again, and thus we have (\ref{d=d}) and
Corollary \ref{2,3} (i) is proved.

Assume $q=3$ and $E_{q}=E_{3}$ is contained in a great circle on $S.$ We will
show $q^{\prime}=Q(\Sigma)=2.$ It suffices to show a contradiction when we
assume $q^{\prime}=3.$ When $q^{\prime}=3,$ $\partial\Sigma$ has a partition
\[
\partial\Sigma=C_{1}\left(  p_{1},p_{2}\right)  +C_{2}\left(  p_{2}%
,p_{3}\right)  +C_{3}\left(  p_{3},p_{1}\right)  ,
\]
satisfies Definition \ref{S-surface} (1)--(8) and Definition \ref{4pi-extr}
(1)--(3), and by Lemma \ref{makes} we have%
\[
R(\Sigma)\leq R(\Sigma_{\partial\Sigma}),\partial\Sigma=\partial
\Sigma_{\partial\Sigma}%
\]
where $\Sigma_{\partial\Sigma}\in\mathcal{S}_{1}$ is given by Solution
\ref{sol}, say, $\Sigma$ is the surface obtained by gluing the closed Jordan
domain $T_{2},K_{1},K_{2},K_{3}$ in Solution \ref{sol}. But we have proved
$H_{\mathcal{S}_{1}}=H_{\mathcal{S}_{0}}=H_{0}$, and then we have
\[
R\left(  \Sigma\right)  =R(\Sigma_{\partial\Sigma}).
\]
By Definition \ref{S-surface} (8) and by Theorem \ref{main2} (ii), each
$C_{j}$ is not a major circular arc and is strictly convex. Thus for each
$K_{j}$ we may write $K_{j}=\mathfrak{D}^{\prime}\left(  \overline
{p_{j}p_{j+1}},\theta_{j},\theta_{j}\right)  $ with $\theta_{j}\in(0,\pi/2].$

If $\overline{p_{1}p_{2}p_{3}p_{1}}$ is a Jordan curve, then $\overline
{n}\left(  \mathfrak{D}\left(  \overline{p_{j}p_{j+1}},\theta_{j},\theta
_{j}\right)  \right)  =0$ for $j=1,2,3,$ $A(T_{2})=2\pi$, since $E_{q}$ is on
a great circle, and by (\ref{ag6}), we have%
\begin{align*}
R(\Sigma)+4\pi &  =4\pi+R(\Sigma_{\partial\Sigma})=4\pi+A(T_{2})+\sum
_{j=1}^{3}\left[  R(K_{j})-4\pi\#\overline{p_{j}p_{j+1}}^{\circ}\cap
E_{q}\right] \\
&  =6\pi+\sum_{j=1}^{3}R(K_{j})=\frac{1}{2}\sum_{j=1}^{3}\left[  R\left(
\mathfrak{D}\left(  \overline{p_{j}p_{j+1}},\theta_{j},\theta_{j}\right)
\right)  +4\pi\right]  ,
\end{align*}
and%
\[
H_{\mathcal{S}_{0}}=\frac{R(\Sigma)+4\pi}{L(\partial\Sigma)}=\frac{\frac{1}%
{2}\sum_{j=1}^{3}R\left(  \mathfrak{D}\left(  \overline{p_{j}p_{j+1}}%
,\theta_{j},\theta_{j}\right)  \right)  +4\pi}{\frac{1}{2}\sum L\left(
\partial\mathfrak{D}\left(  \overline{p_{j}p_{j+1}},\theta_{j},\theta
_{j}\right)  \right)  }.
\]
We may assume $\frac{R\left(  \mathfrak{D}\left(  \overline{p_{1}p_{2}}%
,\theta_{1},\theta_{1}\right)  \right)  +4\pi}{L\left(  \partial
\mathfrak{D}\left(  \overline{p_{1}p_{2}},\theta_{1},\theta_{1}\right)
\right)  }\geq\frac{R\left(  \mathfrak{D}\left(  \overline{p_{j}p_{j+1}%
},\theta_{j},\theta_{j}\right)  \right)  +4\pi}{L\left(  \partial
\mathfrak{D}\left(  \overline{p_{j}p_{j+1}},\theta_{j},\theta_{j}\right)
\right)  },j=2,3.$ It is clear that $\overline{\mathfrak{D}\left(
\overline{p_{1}p_{2}},\theta_{1},\theta_{1}\right)  }\in\mathcal{S}_{0},$ and
then we have%
\[
H_{\mathcal{S}_{0}}=\frac{R(\Sigma)+4\pi}{L(\partial\Sigma)}\leq\frac{R\left(
\mathfrak{D}\left(  \overline{p_{1}p_{2}},\theta_{1},\theta_{1}\right)
\right)  +4\pi}{L\left(  \partial\mathfrak{D}\left(  \overline{p_{1}p_{2}%
},\theta_{1},\theta_{1}\right)  \right)  }\leq H_{\mathcal{S}_{0}}.
\]
Then $\overline{\mathfrak{D}\left(  \overline{p_{1}p_{2}},\theta_{1}%
,\theta_{1}\right)  }\in\mathcal{S}^{\ast}$ and $L\left(  \partial
\mathfrak{D}\left(  \overline{p_{1}p_{2}},\theta_{1},\theta_{1}\right)
\right)  <L\left(  \partial\Sigma\right)  .$ This is a contradiction since we
assumed $q^{\prime}=3\ $and $\Sigma\in\mathcal{S}^{\ast}.$

If $\overline{p_{1}p_{2}p_{3}p_{1}}$ is not a Jordan domain, then
$A(T_{2})=4\pi$ and we may assume $p_{2}\in\overline{p_{1}p_{3}}^{\circ}$.
Then we have $\overline{n}\left(  \mathfrak{D}\left(  \overline{p_{j}p_{j+1}%
},\theta_{j},\theta_{j}\right)  \right)  =0$ for $j=1,2,$ and $\overline
{n}\left(  \mathfrak{D}\left(  \overline{p_{3}p_{1}},\theta_{3},\theta
_{3}\right)  \right)  =1,$ in other words we have $\overline{p_{j}p_{j+1}%
}^{\circ}\cap E_{q}=\emptyset$ for $j=1,2$ and $\overline{p_{1}p_{3}}^{\circ
}\cap E_{q}=1.$ Therefore we have $R\left(  \mathfrak{D}\left(  \overline
{p_{j}p_{j+1}},\theta_{j},\theta_{j}\right)  \right)  =2R(K_{j})$ for $j=1,2$
and $R\left(  \mathfrak{D}\left(  \overline{p_{3}p_{1}},\theta_{3},\theta
_{3}\right)  \right)  =2R(K_{3})-4\pi,$ and so, by (\ref{ag6}), we have%
\begin{align*}
R(\Sigma)+4\pi &  =4\pi+R(\Sigma_{\partial\Sigma})=4\pi+A\left(  T_{2}\right)
+\sum_{j=1}^{3}\left[  R(K_{j})-4\pi\#\overline{p_{j}p_{j+1}}^{\circ}\cap
E_{q}\right] \\
&  =4\pi+\sum_{j=1}^{3}R(K_{j})=4\pi+2\pi+\frac{1}{2}\sum_{j=1}^{3}\left[
R\left(  \mathfrak{D}\left(  \overline{p_{j}p_{j+1}},\theta_{j},\theta
_{j}\right)  \right)  \right] \\
&  =\frac{1}{2}\sum_{j=1}^{3}\left[  R\left(  \mathfrak{D}\left(
\overline{p_{j}p_{j+1}},\theta_{j},\theta_{j}\right)  \right)  +4\pi\right]
\end{align*}
and%
\[
H_{\mathcal{S}_{0}}=\frac{R(\Sigma)+4\pi}{L(\partial\Sigma)}=\frac{\frac{1}%
{2}\sum_{j=1}^{3}R\left(  \mathfrak{D}\left(  \overline{p_{j}p_{j+1}}%
,\theta_{j},\theta_{j}\right)  \right)  +4\pi}{\frac{1}{2}\sum L\left(
\partial\mathfrak{D}\left(  \overline{p_{j}p_{j+1}},\theta_{j},\theta
_{j}\right)  \right)  }.
\]
We may assume $\frac{R\left(  \mathfrak{D}\left(  \overline{p_{j_{0}}%
p_{j_{0}+1}},\theta_{j_{0}},\theta_{j_{0}}\right)  \right)  +4\pi}{L\left(
\partial\mathfrak{D}\left(  \overline{p_{j_{0}}p_{j_{0}+1}},\theta_{j_{0}%
},\theta_{j_{0}}\right)  \right)  }\geq\frac{R\left(  \mathfrak{D}\left(
\overline{p_{j}p_{j+1}},\theta_{j},\theta_{j}\right)  \right)  +4\pi}{L\left(
\partial\mathfrak{D}\left(  \overline{p_{j}p_{j+1}},\theta_{j},\theta
_{j}\right)  \right)  },j=1,2,3.$ It is clear that $T=\overline{\mathfrak{D}%
\left(  \overline{p_{j_{0}}p_{j_{0}+1}},\theta_{j_{0}},\theta_{j_{0}}\right)
}\in\mathcal{S}_{0},$ and then by Lemma \ref{ratio} we have%
\[
H_{\mathcal{S}_{0}}=\frac{R(\Sigma)+4\pi}{L(\partial\Sigma)}\leq\frac{R\left(
\mathfrak{D}\left(  \overline{p_{j_{0}}p_{j_{0}+1}},\theta_{j_{0}}%
,\theta_{j_{0}}\right)  \right)  +4\pi}{L\left(  \partial\mathfrak{D}\left(
\overline{p_{j_{0}}p_{j_{0}+1}},\theta_{j_{0}},\theta_{j_{0}}\right)  \right)
}\leq H_{\mathcal{S}_{0}}.
\]
Then we have $\frac{R(T)+4\pi}{L(\partial T)}=H_{\mathcal{S}_{0}}=H_{0}$ and
$L(\partial T)<L(\partial\Sigma).$ Thus $T\in\mathcal{S}_{0}$ is a $4\pi
$-extremal surface in $\mathcal{S}_{0}$ and so $q^{\prime}=Q(\Sigma)\leq
Q(T)=2.$ This contradicts the assumption $q^{\prime}=Q(E_{q})=3.$

Then we in fact proved $q^{\prime}=Q(\Sigma)=2\ $when $q=3\ $and $E_{q}$ lies
on a great circle, and so (ii) is proved.

Now we assume that $q=q^{\prime}=Q(\Sigma)=3$ and $p_{1},p_{2},p_{3}$ are not
on a great circle. Then the triangle $\overline{p_{1}p_{2}p_{3}p_{1}}$ is in
an open hemisphere on $S.$ And it is clear that $\partial\Sigma$ has a
partition $\partial\Sigma=C_{1}\left(  p_{1},p_{2}\right)  +C_{2}\left(
p_{2},p_{3}\right)  +C_{3}\left(  p_{3},p_{1}\right)  $ satisfying Definition
\ref{S-surface} (1)--(8) and Definition \ref{4pi-extr} (1)--(3). By Theorem
\ref{main2} (ii) each $C_{j}$ is strictly convex. By Definition of
$\mathcal{S}^{\ast}$ and Corollary \ref{same}, there is no other surface
$\Sigma^{\prime}$ in $\mathcal{S}_{0}$ such that $\partial\Sigma^{\prime
}=\partial\Sigma\ $and $\deg_{\max}\Sigma^{\prime}<\deg_{\max}\Sigma.$ Thus
$\Sigma=\Sigma_{\partial\Sigma},$ which is given in Solution \ref{sol}.
Therefore we have proved Corollary \ref{2,3} (iii).

To prove (iv) of the corollary, we will show%
\begin{equation}
\overline{n}\left(  \Sigma\right)  =\#f^{-1}(E_{3})\cap\Delta=0, \label{=0,}%
\end{equation}
whether $q^{\prime}=2$ or $3.$

If $q^{\prime}=2,$ then, by Corollary \ref{2,3} (i), we may assume
$\Sigma=\overline{\mathfrak{D}(\overline{p_{1}p_{2}},\theta_{0},\theta_{0})}$
for some $\theta_{0}\in(0,\pi/2],$ with $d\left(  p_{1},p_{2}\right)
=\delta_{E_{3}},$ and then (\ref{=0,}) holds and $\partial\Sigma$ has the
partition
\begin{equation}
\partial\Sigma=C_{1}\left(  p_{1},p_{2}\right)  +C_{2}\left(  p_{2}%
,p_{1}\right)  =\partial\mathfrak{D}(\overline{p_{1}p_{2}},\theta_{0}%
,\theta_{0}), \label{lens-1}%
\end{equation}
such that $C_{j}$ are symmetric strictly convex circular arcs which are not major.

If $E_{3}$ is contained in a great circle on $S,$ then $q^{\prime}=2$ by
Corollary \ref{2,3} (ii), and thus we may again choose $\Sigma=\overline
{\mathfrak{D}(\overline{p_{1}p_{2}},\theta_{0},\theta_{0})}$ for some
$\theta_{0}\in(0,\pi/2]$ satisfying (\ref{lens-1}) and (\ref{=0,}).

Assume $q^{\prime}=3=q$. Then $\{p_{1},p_{2},p_{3}\}$ are not on a great
circle and
\begin{equation}
\partial\Sigma_{0}=C_{1}\left(  p_{1},p_{2}\right)  +C_{2}\left(  p_{2}%
,p_{3}\right)  +C_{3}\left(  p_{3},p_{1}\right)  , \label{3-edge}%
\end{equation}
with $C_{j}^{\circ}\cap E_{q}=\emptyset.$ Then $\overline{p_{1}p_{2}p_{3}%
p_{1}}$ is a triangle which is either strictly convex at all vertices or
concave at all vertices.

If $\overline{p_{1}p_{2}p_{3}p_{1}}$ is convex, then (\ref{=0,}) holds clearly.

Assume $\overline{p_{1}p_{2}p_{3}p_{1}}$ is concave and (\ref{=0,}) fails.
Then $\overline{n}\left(  \Sigma\right)  \geq1$ and for the closed domain $T$
enclosed by $\overline{p_{1}p_{3}p_{2}p_{1}}=-\overline{p_{1}p_{2}p_{3}p_{1}%
},$ $T$ is contained in some open hemisphere on $S$ containing $\overline
{p_{1}p_{3}p_{2}p_{1}},$ we have $T\in\mathcal{S}_{0}$ and thus we have the
the contradiction (note that $q-2=1$)
\begin{align}
H_{\mathcal{S}_{0}}  &  =H_{0}=\frac{R(\Sigma)+4\pi}{L(\partial\Sigma)}%
=\frac{A(\Sigma)-4\pi\overline{n}\left(  \Sigma\right)  +4\pi}{L(\partial
\Sigma)}\leq\frac{A(\Sigma)}{L(\partial\Sigma)}\label{oo1}\\
&  <\frac{4\pi}{L(\partial T)}\leq\frac{A(T)+4\pi}{L(\partial T)}.\nonumber
\end{align}
It is clear that $T\in\mathcal{S}_{0}$. Since the three edges of $T$ is not
strictly convex, by Theorem \ref{main2} (ii) we have $\frac{A(T)+4\pi
}{L(\partial T)}<H_{\mathcal{S}_{0}}.$ Then by (\ref{oo1}) we have the
contradiction $H_{\mathcal{S}_{0}}<H_{\mathcal{S}_{0}}.$ We have proved
(\ref{=0,}), whether $\overline{p_{1}p_{2}p_{3}p_{1}}$ is convex or concave,
and (iv) has been proved.
\end{proof}

Continuing the above argument, we can prove Theorem \ref{q3}.

\begin{proof}
[\textbf{Proof of Theorem \ref{q3}}]Assume $q=3,E_{3}=\{p_{1},p_{2},p_{3}\}$
and $\Sigma_{0}=\left(  f_{0},\overline{\Delta}\right)  \in\mathcal{S}^{\ast
}\left(  E_{3}\right)  .$ Then we may further assume%
\[
\delta_{E_{3}}=d\left(  p_{1},p_{2}\right)  \leq d\left(  p_{1},p_{3}\right)
\leq d\left(  p_{2},p_{3}\right)  ,
\]
and then $\overline{p_{1}p_{3}}^{\circ}$ does not contain $p_{2}.$ Therefore
$\partial\Sigma_{0}$ has the partition (\ref{lens-1}) (or (\ref{3-edge})),
(\ref{=0,}) holds, and $C_{j}^{\circ}\cap E_{q}=\emptyset$ for each term
$C_{j}$ in the partition. Then by Definition of $\mathcal{S}_{0}$ and by
Theorem \ref{main2} we have the following:%
\[
H_{0}\left(  E_{3}\right)  =H_{\mathcal{S}_{0}}\left(  E_{3}\right)
=\frac{R(\Sigma_{0})+4\pi}{L(\partial\Sigma_{0})}\geq h_{0}\left(
E_{3}\right)  .
\]

Now, inspired by the method on page 215 in \cite{Zh1}, we construct a sequence
of surfaces $\left\{  \Sigma_{n}\right\}  \subset\mathbf{F}^{\ast}%
=\mathbf{F}^{\ast}(Eq)$, such that $H(\Sigma_{n})\rightarrow H_{0}\left(
E_{3}\right)  .$ Let $S^{\prime}$ be the surface with interior $S\backslash
C_{1}$ and boundary $C_{1}-C_{1}.$ Then we can sew\label{sew1} $\Sigma_{0}$
and the surface $S^{\prime}$ along $C_{1}$ to obtain a surface $\Sigma
_{0}^{\prime}=\left(  f_{0}^{\prime},\overline{\Delta}\right)  .$ It is clear
that $A(\Sigma_{0}^{\prime})=A(\Sigma_{0})+4\pi,L(\partial\Sigma_{0}^{\prime
})=L(\partial\Sigma_{0}),$ and $f_{0}^{\prime-1}(E_{3})\cap\Delta$ contains
only one point $a$ with $f_{0}^{\prime}\left(  a\right)  =p_{3}.$ So we may
assume $a=0.$ Then the line segment $\overline{p_{1}p_{3}}$ has an
$f_{0}^{\prime}$-lift $\alpha=\alpha\left(  a_{1},0\right)  $ in
$\overline{\Delta}$ such $\alpha\cap\partial\Delta=\{a_{1}\}$, $f(a_{1}%
)=p_{1}$ and $f\left(  0\right)  =p_{3},$ and we may assume that
$-\alpha=[0,1].$

Let $\Sigma_{n}=\left(  f_{n},\overline{\Delta^{+}}\right)  $ with
$f_{n}=f_{0}^{\prime}\left(  z^{2n}\right)  ,z\in\overline{\Delta^{+}}.$ Then
we have
\[
\overline{n}\left(  \Sigma_{n},E_{3}\right)  =0,\mathrm{\ }L(\partial
\Sigma_{n})=nL\left(  \partial\Sigma_{0}\right)  +2L(\overline{p_{1}p_{3}%
}),\mathrm{\ }A\left(  \Sigma_{n}\right)  =n\left(  A\left(  \Sigma
_{0}\right)  +4\pi\right)  ,
\]
and so $\Sigma_{n}\in\mathbf{F}^{\ast}\left(  E_{3}\right)  $ (note that
$q=3$) and we have%
\[
h_{0}\left(  E_{3}\right)  \geq\frac{R(\Sigma_{n},E_{3})}{L(\partial\Sigma
_{n})}=\frac{nA\left(  \Sigma_{0}\right)  +4\pi n}{nL(\partial\Sigma
_{0})+2L(\overline{p_{1}p_{3}})}\rightarrow\frac{A\left(  \Sigma_{0}\right)
+4\pi}{L(\partial\Sigma_{0})}=H_{0}\left(  E_{3}\right)
\]
and so we have (\ref{q=3})$\ $when $q=3,$ since $H_{0}\left(  E_{3}\right)
\geq h_{0}\left(  E_{3}\right)  .$

(\ref{q=3}) may fail when $q\geq4,$ and the following is a counter example for
(\ref{q=3}).

Let $q\geq4,$ let $E_{q}$ be the set $\{p_{1},p_{2},\dots,p_{q}\}$ with
\[
0=p_{1}<p_{2}<\dots<p_{q}=1,
\]
\[
E_{3}=\{p_{2},p_{3},p_{4}\}\
\]
and
\[
\delta=d\left(  p_{1},p_{2}\right)  <\frac{1}{q^{3}}\delta_{\{p_{2}%
,\cdots,p_{q}\}},
\]
and let $D$ be the small disk with diameter $\overline{p_{1}p_{2}}$. Then
$\overline{D}\in\mathcal{S}_{0}(E_{q})$ and thus we have%
\[
H_{0}(E_{q})=H_{\mathcal{S}_{0}}(E_{q})\geq\frac{\left(  q-2\right)
A(D)+4\pi}{L(\partial D)}\geq\frac{4\pi}{2\pi\delta}=\frac{2}{\delta}.
\]

It is clear that we have $\mathbf{F}^{\ast}\left(  E_{q}\right)
\subset\mathbf{F}^{\ast}\left(  E_{3}\right)  .$ Then for each $\Sigma
\in\mathbf{F}^{\ast}\left(  E_{q}\right)  ,$ by the result $H_{0}(E_{3}%
)=h_{0}\left(  E_{3}\right)  $ just proved we have%
\[
H(\Sigma)=H(\Sigma,E_{q})=\frac{\left(  q-2\right)  A(\Sigma)}{L(\partial
\Sigma)}\leq\left(  q-2\right)  h_{0}(E_{3})=\left(  q-2\right)  H_{0}%
(E_{3}).
\]
Thus we have%
\[
h_{0}\left(  E_{q}\right)  =\sup_{\Sigma\in\mathbf{F}^{\ast}\left(
E_{q}\right)  }H(\Sigma,E_{q})\leq\left(  q-2\right)  H_{0}(E_{3}).
\]

By Corollary \ref{2,3}, there exists a surface $\Sigma_{0}\in\mathcal{S}%
^{\ast}\left(  E_{3}\right)  $ such that $\Sigma_{0}=\overline{\mathfrak{D}%
(\delta_{E_{3}},\theta_{0},\theta_{0})}$ and%
\[
H_{0}(E_{3})=\frac{R(\Sigma_{0},E_{3})+4\pi}{L(\partial\Sigma_{0})}.
\]
On the other hand we have, by $q\geq4,$ that%
\begin{align*}
\frac{R(\Sigma_{0},E_{3})+4\pi}{L(\partial\Sigma_{0})}  &  \leq\frac
{(3-2)A(\overline{\mathfrak{D}(\delta_{E_{3}},\theta_{0},\theta_{0})})+4\pi
}{2\delta_{E_{3}}}\\
&  \leq\frac{6\pi}{2\delta_{E_{3}}}\leq\frac{3\pi}{q^{3}d\left(  p_{1}%
,p_{2}\right)  }=\frac{3\pi}{q^{3}\delta}<\frac{1}{q\delta}.
\end{align*}
Therefore we have $H_{0}(E_{3})<\frac{1}{q\delta},$ and thus
\[
h_{0}(E_{q})\leq\left(  q-2\right)  H_{0}(E_{3})<\frac{q-2}{q\delta}<\frac
{1}{\delta}<H_{0}(E_{q}).
\]

\end{proof}

\section{Notations and terminology}

-----a circle $C$ determined by a circular arc $c$: $C$ is the circle
containing $c$ and oriented by $c$, p. \pageref{determine}

-----boundary radius, Remark \ref{nod-1}, \pageref{bdr}

-----CCM: Complete covering mapping, p. \pageref{CCM}

-----BCCM: branched complete covering mapping, p. \pageref{BCCM}

-----CTTD: closed topological triangular domain, Convention \ref{conv-1}, p.
\pageref{conv-1}

-----closed domain: closure of a domain, \pageref{AhlforsS}

-----convex (path, arc, curve): Definition \ref{in}, p. \pageref{in}

-----decomposable sequence, Definition \ref{undec-seqr}, p.
\pageref{undec-seqr}

-----decomposable surface, Lemma \ref{undec}, p. \pageref{undec}

-----domain: connected open set, \pageref{AhlforsS}

-----extremal surface, precise extremal surfaces, Definition \ref{HL}, p.
\pageref{HL}

-----interior angle of a surface at a boundary point, Definition
\ref{interiorangle}, p. \pageref{interiorangle}

-----OPH: orientation preserving homeomorphism, p. \pageref{OPH}

-----OPCOFO(M): orientation-preserving, continuous, open and finite-to-one
(mapping), p. \pageref{OPCOFOM}

-----SCC arc: simple, convex circular arc, Definition \ref{F}, p. \pageref{F}

-----$B_{f},B_{f}\left(  A\right)  ,B_{f}^{\ast},B_{f}^{\ast}\left(  A\right)
,\ $p. \pageref{BBC}

-----$C_{f},C_{f}\left(  A\right)  ,C_{f}^{\ast},C_{f}^{\ast}\left(  A\right)
,\ $p. \pageref{BBC}

-----$CV_{f},CV_{f}\left(  K\right)  \ $p. \pageref{BBC}

-----$\mathcal{C}\left(  L,m\right)  \supset\mathcal{C}^{\ast}\left(
L,m\right)  \supset\mathcal{F}(L,m)\supset\mathcal{F}_{r}(L,m):$ subspaces of
$\mathcal{F}\left(  L\right)  ,$ Definition \ref{circu}, p. \pageref{circu}

-----$d\left(  \cdot,\cdot\right)  $ the spherical distance on the Riemann
sphere $S$,

-----$d_{f}\left(  \cdot,\cdot\right)  $ the distance defined on the surface
$\left(  f,\overline{\Delta}\right)  :$ Definition \ref{df}, p. \pageref{df}

-----$D(p,\delta),$ the disk on $S$ with center $p$ and spherical radius
$\delta:$ p. \pageref{cov-1}

-----$\mathfrak{D}^{\prime}\left(  I,\theta\right)  ,\mathfrak{D}^{\prime
}\left(  I,k\right)  ,\mathfrak{D}^{\prime}\left(  I,c\right)  :\ $lune,
Definition \ref{lune-lens}, p. \pageref{lune-lens}

-----$\mathfrak{D}\left(  I,\theta_{1},\theta_{2}\right)  ,\mathfrak{D}\left(
I,k_{1},k_{2}\right)  ,\mathfrak{D}\left(  I,c_{1},c_{2}\right)  :$ lens
Definition \ref{lune-lens}, p. \pageref{lune-lens}

-----$\Delta:$ the open disk $\{z:\left\vert z\right\vert <1,z\in
\mathbb{C}\},$ p. \pageref{Delta}

-----$\Delta^{\pm}:$ $\Delta^{\pm}=\Delta\cap H^{\pm},$ p. \pageref{Del+-}

-----$\left(  \partial\Delta\right)  ^{\pm}=\left(  \partial\Delta\right)
\cap\overline{H^{\pm}},$ p. \pageref{+-arc}

-----$E_{q}=\{\mathfrak{a}_{1},\mathfrak{a}_{2},\cdots,\mathfrak{a}_{q}\}$: a
set of distinct $q$ points on the Riemann sphere $S,$ p. \pageref{Eq}

-----$\mathbf{F:}$ Definition \ref{family F,D}, p. \pageref{family F,D}

-----$\mathbf{F}^{\ast}=\mathbf{F}^{\ast}(E_{q}):$ subspace of $\mathbf{F},$
p. \pageref{F*}

-----$\mathcal{F},\mathcal{F}\left(  L\right)  :$ subspace of $\mathbf{F:}$
Definition \ref{F}, p. \pageref{F}

-----$\mathcal{F}_{r},\mathcal{F}_{r}\left(  L\right)  ,\mathcal{F}%
(L,m),\mathcal{F}_{r}\left(  L,m\right)  :$ subspace of $\mathcal{F}%
\mathbf{,}$ Definition \ref{circu} (c) and (d), p. \pageref{FrFLMr}

-----$\mathcal{F}_{r}^{\prime}(L,m):$ subspace of $\mathcal{F}_{r}\left(
L,m\right)  :$ Definition \ref{FR'}, p. \pageref{FR'}

-----$h_{0}=h_{0}(E_{q}),$ p. \pageref{h0}

-----$H_{L}=H_{L}(E_{q}),$ p. \pageref{HL(E)}

-----$H_{0}=H_{0}(E_{q}),$ p. \pageref{H_0}

-----$H\left(  \Sigma\right)  =H(\Sigma,E_{q}):$ $H(\Sigma)=R(\Sigma
)/L(\partial\Sigma)$, p. \pageref{DH}

-----$H^{+}:$ the upper half plane $\operatorname{Im}z>0,$ p. \pageref{H+-}

-----$H^{-}:$ the lower half plane $\operatorname{Im}z<0,$ p. \pageref{H+-}

-----$\overline{H^{\pm}}:$ the closure of $H^{\pm}\ $on $S,$ p. \pageref{H+-}

-----$\mathcal{L}:$ the continuous point set of the function $H_{L}$ of $L,$
Definition \ref{L}, p. \pageref{L}

-----$\overline{n}(\Sigma)=\overline{n}(\Sigma,E_{q}),\overline{n}%
(\Sigma,\mathfrak{a}_{v}),$ (\ref{a70}) in p. \pageref{a70}

-----$\mathbb{N}:$ the set of natural numbers $\{1,2,\dots,\},$ p.
\pageref{nature}

-----$\mathbb{N}^{0}:\mathbb{N\cup}\{0\},$ p. \pageref{nature}

-----$P:$ stereographic projection \pageref{SP}

-----$R(\Sigma)=R(\Sigma,E_{q}):$ Ahlfors error term, (\ref{Ahero}) p.
\pageref{Ahero}

-----$S:$ unit Riemann sphere with area $4\pi,$ p. \pageref{RS}

-----$\mathcal{S}_{1},\mathcal{S}_{1}^{(5)},\mathcal{S}_{1}^{(6)}%
,\mathcal{S}_{1}^{(7)},\mathcal{S}_{1}^{(8)}:$ subspace of $\mathcal{F}:$
Definition \ref{S15678}, p. \pageref{S15678}

-----$\mathcal{S}_{0}=\mathcal{S}_{1}^{(5)}\cap\mathcal{S}_{1}^{(6)}%
\cap\mathcal{S}_{1}^{(7)}\cap\mathcal{S}_{1}^{(8)}:$ Definition
\ref{S-surface}, p. \pageref{S-surface}

-----$\mathcal{S}^{\ast}:$ subspace of $\mathcal{S}_{0}$: Definition
\ref{4psim}, p. \pageref{4psim}

-----$\left(  \cdot\right)  ^{\circ}:$ interior of an arc (path), p.
\pageref{boundary}, interior of a surface, Definition \ref{interior}, p.
\pageref{interior}

-----$\partial\left(  \cdot\right)  :$ in Definition \ref{interior}: set of
end points of an arc, p. \pageref{boundaryarc} or boundary of a domain on $S$
or a surface, p. \pageref{boundary}.

-----$\overline{\left(  \cdot\right)  }:$ closure of a set, p. \pageref{cl}

-----$\#\left(  \cdot\right)  :$ the cardinality of a set, p. \pageref{card}

-----$\sim:$ in Remark \ref{finite}: equivalence of surfaces, or curves,
Remark \ref{finite}, p. \pageref{finite}

\label{aaaa}

\end{document}